\documentclass[a4paper,10pt,reqno]{amsart}
\pdfoutput=1

\usepackage{amssymb}
\usepackage{amsmath}
\usepackage{amsthm}
\usepackage{latexsym}
\usepackage{mathrsfs}
\usepackage{eucal}
\usepackage[safe]{tipa}
\usepackage{slashed}
\usepackage{textalpha}
\usepackage{upgreek}

\usepackage[inline]{enumitem}
\usepackage{subcaption}
\usepackage[numbers]{natbib}

\DeclareSymbolFontAlphabet{\mathbb}{AMSb}

\usepackage[nonewpage]{imakeidx}
\indexsetup{level=\section*,toclevel=section,noclearpage,firstpagestyle=headings,headers={\MakeUppercase\shortauthors}{\MakeUppercase\shorttitle}}

\usepackage{tikz}
\usetikzlibrary{shapes,decorations,backgrounds}
\usetikzlibrary{shapes.multipart}
\usetikzlibrary{cd}
\usetikzlibrary{arrows} 
\usetikzlibrary{graphs}
\usetikzlibrary{intersections}
\usetikzlibrary{positioning}
\usetikzlibrary{decorations.pathmorphing}
\usetikzlibrary{datavisualization} 
\usetikzlibrary{datavisualization.formats.functions}
\usetikzlibrary{intersections}
\usetikzlibrary{hobby}
\usetikzlibrary{through}
\usetikzlibrary{calc}
\usetikzlibrary{fadings}
\usetikzlibrary{patterns}

\usepackage{pgfmath}
\usepackage{pgfkeys}

\usetikzlibrary{automata,positioning}
\usetikzlibrary{decorations.markings}
\tikzset{->-/.style={decoration={
  markings,
  mark=at position #1 with {\arrow{>}}},postaction={decorate}}}
  \tikzset{-<-/.style={decoration={
  markings,
  mark=at position #1 with {\arrowreversed[red]{latex'}}},postaction={decorate}}}

\usepackage[breaklinks=true, linktocpage=true, pdfauthor={L. Andersson, T. Backdahl, P. Blue and S. Ma},
            pdftitle={Stability for linearized gravity on the Kerr spacetime}]{hyperref}

\usepackage{breakurl}
\usepackage{url}

\hypersetup{
    colorlinks=true,
    linkcolor=blue,
    filecolor=blue,      
    urlcolor=blue,
    bookmarksdepth=2
}

\usepackage{orcidlink}

\makeindex[columns=3, title=Index of symbols]
\makeatletter\def\imki@columns{3}\makeatother

\DeclareMathOperator{\tho}{\text{\rm\textthorn}}
\DeclareMathOperator{\thop}{\text{\rm\textthorn}^\prime\negthinspace}
\DeclareMathOperator{\edt}{\eth}
\DeclareMathOperator{\edtp}{\eth^\prime\negthinspace}
\newcommand{\Qop}{\text{\rm\textQoppa}}

\newcommand{\hROp}{\widehat{R}} 
\newcommand{\TMESOp}{\widehat{S}}      

\newcommand{\rem}{\text{\rm rem}}
\newcommand{\homo}{\text{\rm hom}}
\newcommand{\initial}{\text{\rm init}}
\newcommand{\interior}{\text{\rm int}} 
\newcommand{\ext}{\text{\rm ext}}
\newcommand{\near}{\text{\rm near}}
\newcommand{\early}{\text{\rm early}}

\newcommand{\AntiDeriv}{\mathbf{I}}

\newcommand{\widebar}[1]{\overline{#1}}

\newcommand{\Naturals}{\mathbb{N}}
\newcommand{\Integers}{\mathbb{Z}}
\newcommand{\Reals}{\mathbb{R}}
\newcommand{\Complexs}{\mathbb{C}}
\newcommand{\Sphere}{S^2}

\newcommand{\di}{\mathrm{d}}

\newcommand{\generalManifold}{\mathcal{N}}
\newcommand{\generalMetric}{h}

\newcommand{\ii}{i}
\newcommand{\iip}{i'}
\renewcommand{\ij}{j} 
\newcommand{\ireg}{k}

\newcommand{\diTwoVol}{\di^2\hspace{-1.5pt}\mu}
\newcommand{\diThreeVol}{\di^3\hspace{-1.5pt}\mu}
\newcommand{\diFourVol}{\di^4\hspace{-1.5pt}\mu}
\newcommand{\diThreeVolScri}{\di^3\hspace{-1.5pt}\mu_{\Scri}}

\newcommand{\bigOAnalytic}{O_\infty}

\newcommand{\unrescaledOps}{\mathbb{B}}
\newcommand{\rescaledOps}{\mathbb{D}}
\newcommand{\ScriOps}{\slashed{\mathbb{D}}}
\newcommand{\sphereOps}{\mathbb{S}}
\newcommand{\generalOps}{\mathbb{X}}
\newcommand{\generalOp}{D}

\newcommand{\setOfFields}{\Phi}

\providecommand{\abs}[1]{\lvert#1\rvert}
\providecommand{\norm}[1]{\lVert#1\rVert}

\providecommand{\absHighOrder}[3]{\abs{#1}_{#2,#3}} 
 
\providecommand{\absRescaled}[2]{\absHighOrder{#1}{#2}{\rescaledOps}} 
\providecommand{\absScri}[2]{\absHighOrder{#1}{#2}{\ScriOps}}

\makeatletter
  \def\moverlay{\mathpalette\mov@rlay}
  \def\mov@rlay#1#2{\leavevmode\vtop{%
     \baselineskip\z@skip \lineskiplimit-\maxdimen
     \ialign{\hfil$#1##$\hfil\cr#2\crcr}}}
\makeatother

\newcommand{\squareS}{\square}

\newcommand{\Scri}{\mathscr{I}}
\newcommand{\Horizon}{\mathscr{H}}

\newcommand{\GenVec}{\nu}

\def\undertilde#1{\mathord{\vtop{\ialign{##\crcr
$\hfil\displaystyle{#1}\hfil$\crcr\noalign{\kern1.5pt\nointerlineskip}
$\hfil\tilde{}\hfil$\crcr\noalign{\kern1.5pt}}}}}

\newcommand{\tHor}{{\timefunc_{\Horizon^+}}}
\newcommand{\tGen}{\tau}

\newcounter{mnotecount}[section]

\newcounter{mymnotecount}[section]
\renewcommand{\themymnotecount}{\thesection.\arabic{mymnotecount}}

\newcommand{\mymnote}[1]{\protect{\stepcounter{mymnotecount}}${\raisebox{0.5\baselineskip}[0pt]{\makebox[0pt][c]{\color{magenta}{\tiny\em$\bullet$\themymnotecount}}}}$\marginpar{\raggedright\tiny\em$\!\bullet$\themymnotecount:
#1}\ignorespaces}

\renewcommand{\mymnote}[1]{}

\newcommand{\raddegphi}[1]{\hat\varphi_{-2}^{(#1)}}
\newcommand{\radpsi}[1]{\hat\psi_{-2}^{(#1)}}
\newcommand{\radpsizero}[1]{\hat\psi_{+2}^{(#1)}}
\newcommand{\radpsizerores}[1]{{\hat\varphi_{+2}^{(#1)}}}

\def\VOp{V}
\newcommand{\hatVOp}{\hat{\VOp}}
\def\YOp{Y}

\def\LxiOp{\mathcal{L}_{\xi}}
\def\LzetaOp{\mathcal{L}_{\zeta}}
\def\LetaOp{\mathcal{L}_{\eta}}

\DeclareMathOperator{\hedt}{\mathring{\eth}{}}
\DeclareMathOperator{\hedtp}{\mathring{\eth}^\prime\negthinspace}

\newcommand{\ringtau}{\mathring{\tau}}

\newcommand{\rstar}{r_*}

\usepackage{color}

\newcommand{\timefunc}{t}

\newcommand{\rInv}{R}
\newcommand{\rCutOff}{R_0}

\newcommand{\vOne}{v_1}
\newcommand{\ingoingcone}{\mathcal{C}}
\newcommand{\tprimeingoing}{\timefunc_{\ingoingcone}(\vOne)}
\newcommand{\BEAMbulk}{B}

\newcommand{\Horgtt}[1]{\Horizon^+_{#1,\infty}}

\newcommand{\Dtau}{\Omega_{\timefunc_1,\timefunc_2}}
\newcommand{\Dtaut}{\Omega_{\timefunc,\infty}}

\newcommand{\Dtautext}{\Omega^{\ext}_{\timefunc,\infty}}
\newcommand{\Dtauitext}{\Omega^{\ext}_{\timefunc_0,\infty}}
\newcommand{\Dtautprimeext}{\Omega^{\ext}_{\timefunc',\infty}}
\newcommand{\Dtauext}{\Omega^{\ext}_{\timefunc_1,\timefunc_2}}
\newcommand{\Dtauint}{\Omega^{\interior}_{\timefunc_1,\timefunc_2}}

\newcommand{\Dtautimeint}{\Omega^{\interior}_{\timefunc,\infty}}
\newcommand{\Dtauitint}{\Omega^{\interior}_{\timefunc_0,\infty}}

\newcommand{\Dtautnear}[1]{\Omega^{\near}_{#1, \infty}}
\newcommand{\Dtautimenear}{\Omega^{\near}_{\timefunc, \infty}}
\newcommand{\Dtauitnear}{\Omega^{\near}_{\timefunc_0, \infty}}
\newcommand{\Dtauearly}{\Omega^\early_{\initial, \timefunc_0}}
\newcommand{\Dtautearly}{\Omega^\early_{\initial, \timefunc}}
\newcommand{\DtautearlyFurtherArguments}[1]{\Omega^{\early,#1}_{\initial,\timefunc}}
\newcommand{\DtautearlyR}{\DtautearlyFurtherArguments{\rCutOff}}
\newcommand{\DtautearlyRc}{\DtautearlyFurtherArguments{\rCutOff-M}}
\newcommand{\DtautearlyRcR}{\DtautearlyFurtherArguments{\rCutOff-M,\rCutOff}}
\newcommand{\DtauitearlyR}{\Omega^{\early,\rCutOff}_{\initial, \timefunc_0}}

\newcommand{\DtauitearlyRcR}{\Omega^{\early,\rCutOff-M,\rCutOff}_{\initial, \timefunc_0}}
\newcommand{\Boundext}{\Xi_{\timefunc_1,\timefunc_2}}
\newcommand{\Boundtext}{\Xi_{\timefunc,\infty}}
\newcommand{\Boundtprimeext}{\Xi_{\timefunc',\infty}}
\newcommand{\Boundtingoing}{\Xi_{\timefunc_{\ingoingcone}(\timefunc),\infty}}

\newcommand{\Staui}{\Sigma_{\timefunc_1}}
\newcommand{\Stau}{\Sigma_{\timefunc_2}}
\newcommand{\Stauiext}{\Sigma^{\ext}_{\timefunc_1}}

\newcommand{\Stauext}{\Sigma^{\ext}_{\timefunc_2}}
\newcommand{\Stauitext}{\Sigma^{\ext}_{\timefunc_0}}
\newcommand{\Stauiint}{\Sigma^{\interior}_{\timefunc_1}}
\newcommand{\Stauint}{\Sigma^{\interior}_{\timefunc_2}}

\newcommand{\Stauini}{\Sigma_{\initial}}

\newcommand{\DoDfuture}{D^+}

\newcommand{\DtauR}{\Omega^{\rCutOff}_{\timefunc_1,\timefunc_2}}
\newcommand{\DtauRc}{\Omega^{\rCutOff-M}_{\timefunc_1,\timefunc_2}}
\newcommand{\DtauRcR}{\Omega^{\rCutOff-M,\rCutOff}_{\timefunc_1,\timefunc_2}} 
\newcommand{\StauiR}{\Sigma^{\rCutOff}_{\timefunc_1}}
\newcommand{\StauiRc}{\Sigma^{\rCutOff-M}_{\timefunc_1}}
\newcommand{\StauiRcR}{\Sigma_{\timefunc_1}^{\rCutOff-M,\rCutOff}}
\newcommand{\StauR}{\Sigma^{\rCutOff}_{\timefunc_2}}
\newcommand{\StauRc}{\Sigma^{\rCutOff-M}_{\timefunc_2}}
\newcommand{\StauRcR}{\Sigma_{\timefunc_2}^{\rCutOff-M,\rCutOff}}
\newcommand{\Stauit}{\Sigma_{\timefunc_0}}
\newcommand{\StauitRcR}{\Stauit^{\rCutOff-M,\rCutOff}}
\newcommand{\Staut}{\Sigma_{\timefunc}}
\newcommand{\Stautext}{\Sigma^{\ext}_{\timefunc}}
\newcommand{\Stautimeint}{\Sigma^{\text{\rm int}}_{\timefunc}}
\newcommand{\StauitR}{\Sigma^{\rCutOff}_{\timefunc_0}}
\newcommand{\StautR}{\Sigma^{\rCutOff}_{\timefunc}}
\newcommand{\StautRc}{\Sigma_{\timefunc}^{\rCutOff-M}}
\newcommand{\StautRcR}{\Sigma_{\timefunc}^{\rCutOff-M,\rCutOff}}
\newcommand{\Scritau}{\Scri^+_{\timefunc_1,\timefunc_2}}

\newcommand{\Scritini}{\Scri^+_{-\infty,\timefunc}}
\newcommand{\veps}{\varepsilon}

\newcommand{\rCutOffInHardy}{\bar{R}_0}
\newcommand{\rCutOffInrpWave}{\bar{R}_0}
\newcommand{\rCutOffOtherBounds}{\bar{R}_0}

\newcommand{\rRethreaded}{{\tilde{r}}}
\newcommand{\thetaRethreaded}{{\tilde{\theta}}}
\newcommand{\phiRethreaded}{{\tilde{\phi}}}
\newcommand{\TparallelForRethreadedHardy}{\mathcal{X}}

\newcommand{\CInrpWave}{C}

\newcommand{\CInHyperboloids}{C_{\text{hyp}}}

\newcommand{\otherinireg}{k_1}

\newcommand{\inienergy}[2]{\mathbb{I}_{\initial}^{#1;#2}}

\newcommand{\inienergynoalpha}[1]{\mathbb{I}_{\initial}^{#1}}
\newcommand{\inienergyplustwo}[1]{\mathbb{I}_{\initial}^{+2,#1}}

\newcommand{\NEWinienergyminustwo}[1]{\mathbb{I}_{-2}^{#1}} 
\newcommand{\iniEnergyMinusTwoForDescent}[1]{\inienergy{#1}{9}(\psibase[-2])}
\newcommand{\inienergyplustwoForDescent}[1]{\inienergy{#1}{1}(\psibase[+2])}   
\newcommand{\inipointwise}[2]{\mathbb{P}_{\initial}^{#1;#2}}

\newcommand{\EspinNoArg}{s}
\newcommand{\EkNoArg}{k} 
\newcommand{\ElNoArg}{l} 
\newcommand{\EmNoArg}{m} 
\newcommand{\EalphaNoArg}{\alpha_1}
\newcommand{\EDNoArg}{D}
\newcommand{\Espin}[1]{\EspinNoArg[#1]}
\newcommand{\Ek}[1]{\EkNoArg[#1]} 
\newcommand{\El}[1]{\ElNoArg[#1]} 
\newcommand{\Em}[1]{\EmNoArg[#1]} 
\newcommand{\Ealpha}[1]{\EalphaNoArg[#1]}
\newcommand{\ED}[1]{\EDNoArg[#1]} 
\newcommand{\Lxitimes}{q}

\newcommand{\pt}{\LxiOp}

\def\G0{\widehat{G}_0}
\def\G1{\widehat{G}_1}
\def\G2{\widehat{G}_2}
\newcommand{\EnergyinrpEstimatepsizero}{\widetilde{F}}
\newcommand{\BulkinrpEstimatepsizero}{\widetilde{G}}

\usepackage{bbm}

\newcommand{\lesC}[1]{\lesssim_{#1}}
\newcommand{\gtrC}[1]{\gtrsim_{#1}}
\newcommand{\simC}[1]{\sim_{#1}}

\theoremstyle{plain}
\newtheorem{thm}{Theorem}[section]

\newtheorem{lemma}[thm]{Lemma}

\theoremstyle{definition}
\newtheorem{definition}[thm]{Definition}

\newtheorem{remark}[thm]{Remark}

\newcounter{step}
\renewcommand{\thestep}{\arabic{step}}
\newenvironment{steps}{\setcounter{step}{0}}{}
\newcommand{\step}[1]{\refstepcounter{step}\noindent\textbf{Step \thestep: #1.}}

\newcommand{\Normal}{\nu}

\newcommand{\LerayForm}[1]{\di^3\hspace{-1.5pt}\mu_{#1}}

\newcommand{\standardHypothesisOnDelta}{Let $\delta>0$ be sufficiently small}

\newcommand{\standardMinusTwoHypothesisDefinePsibasei}{Let $\psibase[-2]$ be a scalar of spin-weight $-2$ and $\{\psibase[-2][i]\}_{i=0}^{4}$ be as in definition \ref{def:psi(i)-2}}
\newcommand{\standardMinusTwoHypothesisNotBEAMOnlyTwo}{Let $\psibase[-2]$ be a scalar of spin-weight $-2$ and $\{\psibase[-2][i]\}_{i=0}^{2}$ be as in definition \ref{def:psi(i)-2}. Assume $\psibase[-2]$ satisfies the Teukolsky equation \eqref{eq:TeukolskyRegular-2}}
\newcommand{\standardMinusTwoHypothesisNotBEAM}{\standardMinusTwoHypothesisDefinePsibasei. Assume $\psibase[-2]$ satisfies the Teukolsky equation \eqref{eq:TeukolskyRegular-2}}
\newcommand{\standardMinusTwoHypothesisBEAM}{Assume the BEAM condition from definition \ref{def:BEAM} holds}

\newcommand{\standardPlusTwoHypothesisHelper}[1]{ Let $\psibase[+2]$ be a spin-weight $+2$ scalar#1.{} }
\newcommand{\standardPlusTwoHypothesisNotBEAMNotSolu}{\standardPlusTwoHypothesisHelper{{}}}
\newcommand{\standardPlusTwoHypothesisNotBEAM}
{\standardPlusTwoHypothesisHelper{ that is a solution of the Teukolsky equation \eqref{eq:TeukolskyRegular+2}}}

\newcommand{\standardPlusTwoHypothesis}{\standardPlusTwoHypothesisHelper{ that is a solution of the Teukolsky equation \eqref{eq:TeukolskyRegular+2}} Assume either of the BEAM conditions from definition \ref{ass:BEAMspin+2} holds.{} }
\newcommand{\standardPlusTwoHypothesisPointwiseCondition}
{the pointwise condition \eqref{condition:spin+2goestozero} from definition \ref{ass:BEAMspin+2} holds}

\newcommand{\PigeonTime}{\timefunc}
\newcommand{\iPigeonReg}{\iip}

\newcommand{\principal}{\textrm{principal}}
\newcommand{\error}{\textrm{error}}

\newcommand{\momentum}{P}
\newcommand{\momentumiX}{\momentum_{i,X}}
\newcommand{\momentumAV}{\momentum_{1,\VOp}}
\newcommand{\momentumBV}{\momentum_{2,\VOp}}
\newcommand{\momentumAY}{\momentum_{1,\YOp}}
\newcommand{\momentumEV}{\momentum_{5,\VOp}}
\newcommand{\momentumEY}{\momentum_{5,\YOp}}
\newcommand{\momentumDV}{\momentum_{4,\VOp}}
\newcommand{\momentumFY}{\momentum_{6,\YOp}}
\newcommand{\momentumGY}{\momentum_{7,\YOp}}
\newcommand{\momentumFV}{\momentum_{6,\VOp}}
\newcommand{\momentumGV}{\momentum_{7,\VOp}}
\newcommand{\energyComponent}{e}
\newcommand{\energyComponentiX}{\energyComponent_{i,X}}
\newcommand{\energyComponentiXPrincipal}{\energyComponent_{i,X,\principal}}
\newcommand{\energyComponentiXError}{\energyComponent_{i,X,\error}}
\newcommand{\energyComponentAVPrincipal}{\energyComponent_{1,\VOp,\principal}}
\newcommand{\energyComponentAVError}{\energyComponent_{1,\VOp,\error}}

\newcommand{\energyComponentAYPrincipal}{\energyComponent_{1,\YOp,\principal}}
\newcommand{\energyComponentAYError}{\energyComponent_{1,\YOp,\error}}

\newcommand{\energyComponentDVError}{\energyComponent_{4,\VOp,\error}}

\newcommand{\energyComponentEVError}{\energyComponent_{5,\VOp,\error}}

\newcommand{\energyComponentEYPrincipal}{\energyComponent_{5,\YOp,\principal}}
\newcommand{\energyComponentEYError}{\energyComponent_{5,\YOp,\error}}
\newcommand{\energyComponentFVError}{\energyComponent_{6,\VOp,\error}}
\newcommand{\energyComponentFYError}{\energyComponent_{6,\YOp,\error}}
\newcommand{\energyComponentFVVxi}{\energyComponent_{6,\VOp,(\VOp\xi)}}
\newcommand{\energyComponentFVxixi}{\energyComponent_{6,\VOp,(\xi\xi)}}
\newcommand{\energyComponentFYYxi}{\energyComponent_{6,\YOp,(\YOp\xi)}}
\newcommand{\energyComponentFYxixi}{\energyComponent_{6,\YOp,(\xi\xi)}}
\newcommand{\energyComponentGVError}{\energyComponent_{7,\VOp,\error}}
\newcommand{\energyComponentGYError}{\energyComponent_{7,\YOp,\error}}
\newcommand{\energyComponentGVVeta}{\energyComponent_{7,\VOp,(\VOp\eta)}}
\newcommand{\energyComponentGVxieta}{\energyComponent_{7,\VOp,(\xi\eta)}}
\newcommand{\energyComponentGYYeta}{\energyComponent_{7,\YOp,(\YOp\eta)}}
\newcommand{\energyComponentGYxieta}{\energyComponent_{7,\YOp,(\xi\eta)}}
\newcommand{\energyComponentPrincipal}{\energyComponent_{\principal}}
\newcommand{\bulkComponent}{\Pi}
\newcommand{\bulkComponentiX}{\bulkComponent_{i,X,\principal}}
\newcommand{\bulkComponentiXError}{\bulkComponent_{i,X,\error}}
\newcommand{\bulkComponentAV}{\bulkComponent_{1,\VOp,\principal}}
\newcommand{\bulkComponentAVError}{\bulkComponent_{1,\VOp,\error}}
\newcommand{\bulkComponentBV}{\bulkComponent_{2,\VOp,\principal}}
\newcommand{\bulkComponentBVError}{\bulkComponent_{2,\VOp,\error}}
\newcommand{\bulkComponentAY}{\bulkComponent_{1,\YOp,\principal}}
\newcommand{\bulkComponentAYError}{\bulkComponent_{1,\YOp,\error}}
\newcommand{\bulkComponentAYYY}{\bulkComponent_{1,\YOp,(\YOp,\YOp)}}
\newcommand{\bulkComponentAYYeta}{\bulkComponent_{1,\YOp,(\YOp,\eta)}}
\newcommand{\bulkComponentDV}{\bulkComponent_{4,\VOp,\principal}}
\newcommand{\bulkComponentDVError}{\bulkComponent_{4,\VOp,\error}}
\newcommand{\bulkComponentEV}{\bulkComponent_{5,\VOp,\principal}}
\newcommand{\bulkComponentEVError}{\bulkComponent_{5,\VOp,\error}}

\newcommand{\bulkComponentEY}{\bulkComponent_{5,\YOp,\principal}}
\newcommand{\bulkComponentEYError}{\bulkComponent_{5,\YOp,\error}}
\newcommand{\bulkComponentFVError}{\bulkComponent_{6,\VOp,\error}}
\newcommand{\bulkComponentFY}{\bulkComponent_{6,\YOp,\error}}
\newcommand{\bulkComponentGVError}{\bulkComponent_{7,\VOp,\error}}
\newcommand{\bulkComponentGY}{\bulkComponent_{7,\YOp,\error}}
\newcommand{\bulkComponentPrincipal}{\bulkComponent_{\principal}}

\newcommand{\varop}{\delta}
\newcommand{\half}{\frac{1}{2}}

\usepackage{xparse}
\usepackage{etoolbox}

\NewDocumentCommand{\psibase}{O{-2}O{}o}{\IfNoValueTF{#3}{\hat{\psi}}{\ifstrequal{#3}{bar}{\overline{\hat{\psi}}}{\GenericError{}{LaTex Warning: incorrect 3rd argument "#3" in \noexpand\psibase on input line \the\inputlineno}}}_{\ifblank{#1}{-2}{#1}}\IfValueT{#2}{\ifblank{#2}{}{^{(#2)}}}}

\usepackage{enumitem}

\newcommand{\NPl}{l}
\newcommand{\NPn}{n}
\newcommand{\NPm}{m}
\newcommand{\NPmbar}{\bar m}

\newcommand{\Mcal}{\mathcal M}

\newcommand{\met}{g} 

\newcommand{\SL}{\mathrm{SL}}

\newcommand{\alphalow}[1]{\underline{\alpha}{[}#1{]}}
\newcommand{\alphahigh}[1]{\overline{\alpha}[#1]}
\newcommand{\alphalo}{\underline{\alpha}}
\newcommand{\alphahi}{\overline{\alpha}}

\title{Stability for linearized gravity on the Kerr spacetime}
\author[L. Andersson]{Lars Andersson \orcidlink{0000-0002-6364-7384}}
\email{lars.andersson@bimsa.cn}
\address{Beijing Institute of Mathematical Sciences and Applications, Beijing 101408, China}
\author[T. B\"ackdahl]{Thomas B\"ackdahl \orcidlink{0000-0003-3240-2445}} 
\email{thomas.backdahl@chalmers.se}
\address{Mathematical Sciences, Chalmers University of Technology and University of Gothenburg, SE-412~96 Gothenburg, Sweden}
\author[P. Blue]{Pieter Blue \orcidlink{0000-0002-6195-4022}}
\email{p.blue@ed.ac.uk}
\address{School of Mathematics and Maxwell Institute for Mathematical Sciences, University of Edinburgh, Edinburgh EH9 3FD, UK}
\author[S. Ma]{Siyuan Ma \orcidlink{0000-0003-2893-7674}}
\email{siyuan.ma@amss.ac.cn}
\address{Academy of Mathematics and Systems Science, Chinese Academy of Sciences, Beijing 100190, China}

\allowdisplaybreaks[2]

\numberwithin{equation}{section}

\begin{document}
\begin{abstract}
In this paper we prove integrated energy and pointwise decay estimates for solutions of the vacuum linearized Einstein equation on the domain of outer communication of the Kerr black hole spacetime. The estimates are valid for the full subextreme range of Kerr black holes, provided integrated energy estimates for the Teukolsky equation hold. For slowly rotating Kerr backgrounds, such estimates are known to hold, due to the work of one of the authors. The results in this paper thus provide the first stability results for linearized gravity on the Kerr background, in the slowly rotating case, and reduce the linearized stability problem for the full subextreme range to proving integrated energy estimates for the Teukolsky equation. This constitutes an essential step towards a proof of the black hole stability conjecture, i.e. the statement that the Kerr family is dynamically stable, one of the central open problems in general relativity. 
\end{abstract} 
\maketitle
\tableofcontents

\section{Introduction} \label{sec:intro}
The Kerr family of asymptotically flat, stationary, and axially symmetric solutions of the vacuum Einstein equations is parametrized by mass $M$ and angular momentum per unit mass $a$. 
In ingoing Eddington-Finkelstein coordinates\footnote{See \cite[Box 33.2]{MTW}. The ingoing Eddington-Finkelstein coordinates are also known as Kerr coordinates. We work in signature $+---{}$, and use conventions and notations as in \cite{MR944085,MR917488}.}  $(v,r,\theta,\phi) \in \Reals \times (0,\infty) \times \Sphere,$ the Kerr metric takes the form 
\index{G1gab-metric@$g_{ab}$}
\begin{align}
g_{ab}={}&- 2(dr)_{(a} (dv)_{b)}
 + 2a \sin^2\theta (d\phi)_{(a} (dr)_{b)} 
 + \frac{4 M a  r \sin^2\theta}{\Sigma}(d\phi)_{(a} (dv)_{b)}
\nonumber\\
& + \frac{\Delta- a^2 \sin^2\theta}{\Sigma}(dv)_{a} (dv)_{b}
 + \frac{ a^2 \sin^2\theta \Delta - (a^2 + r^2)^2}{\Sigma}\sin^2\theta (d\phi)_{a} (d\phi)_{b}
 -  \Sigma (d\theta)_{a} (d\theta)_{b} , \label{eq:gab-intro}
\end{align} 
with volume element $\Sigma \sin\theta dv dr d\theta d\phi$.
Here 
\index{S4Sigma@$\Sigma$}
\index{D4Delta@$\Delta$}
$\Sigma=a^2 \cos^2\theta + r^2, 
\Delta=a^2 - 2 M r + r^2 .
$
The Killing vector fields of the Kerr metric are \index{X3xi@$\xi^a$} $\xi^a = (\partial_v)^a$, which has unit norm at infinity and expresses the fact that Kerr is stationary, and the axial Killing vector field \index{E3eta@$\eta^a$} $\eta^a = (\partial_\phi)^a$. 
In the subextreme case $|a|<M$, the maximally extended Kerr spacetime contains a black hole with a bifurcate event horizon whose future part 
\index{H2ScriptHorizon@$\Horizon$}
\index{R1rplus@$r_+$}
$\Horizon^+$ is located at $r=r_+$ where $r_+ = M + \sqrt{M^2-a^2}$ is the larger of the two roots of $\Delta$.  
\index{M2Mcal@$\mathcal{M}$}
The domain of outer communication of the Kerr black hole is the region $r > r_+,$ which we shall denote $\mathcal{M}$. 

In addition to being stationary and axially symmetric, the Kerr metric is algebraically special, of Petrov type D, or $\{2,2\}$. 
In particular, the Weyl curvature tensor of the Kerr spacetime has two repeated principal null vectors\footnote{Let $C_{abcd}$ be the Weyl tensor of $(\mathcal{M}, g_{ab})$. A null vector $k^a$ is a principal null vector if $k_{[a} C_{a]bc[d}k_{f]}k^b k^c = 0 $, cf. \cite[\S 4.3]{2003esef.book.....S}.}.
These may without loss of generality  be chosen to be real and future directed.
We shall here use the Znajek tetrad $(l^a, n^a, m^a, \bar m^a)$ \cite{1977MNRAS.179..457Z}, which in 
ingoing Eddington-Finkelstein coordinates takes the form 
\index{L1l@$l^a$}
\index{N1n@$n^a$}
\index{M1m@$m^a$}
\begin{subequations} \label{eq:Znajek-intro}
\begin{align}
l^{a}={}&\frac{\sqrt{2} a (\partial_{\phi})^{a}}{\Sigma}
 + \frac{\sqrt{2} (a^2 + r^2) (\partial_{v})^{a}}{\Sigma}
 + \frac{ \Delta(\partial_{r})^{a}}{\sqrt{2} \Sigma}, \label{eq:l-intro}\\
n^{a}={}&- \tfrac{1}{\sqrt{2}}(\partial_{r})^{a} , \label{eq:n-intro}\\
m^{a}={}&\frac{(\partial_{\theta})^{a}}{\sqrt{2} (r-i a \cos\theta)}
 + \frac{i  \csc\theta (\partial_{\phi})^{a}}{\sqrt{2} (r-i a \cos\theta)}
 + \frac{i a  \sin\theta (\partial_{v})^{a}}{\sqrt{2} (r-i a \cos\theta)}  \label{eq:m-intro} ,
\end{align}
\end{subequations}
with $\bar m^a$ the complex conjugate of the complex null vector $m^a$. We have $g_{ab} = 2 (l_{(a} n_{b)} - m_{(a} \bar m_{b)})$.
The vectors $l^a$ and $n^a$ are principal null vectors. The vector $n^a$ is ingoing, $n^b \nabla_b r < 0$. Furthermore, 
$n^a$ is scaled to be auto-parallel, i.e. so $n^b \nabla_b n^a = 0$ holds, and not merely $n^b \nabla_b n^a\propto n^a$, as is guaranteed by the Goldberg-Sachs theorem. The Znajek tetrad commutes with the Killing vector fields of the Kerr spacetime and extends smoothly through the future event horizon $\Horizon^+$. 

Following \cite{1974RSPSA.341...49S}, let $g_{ab}(\lambda)$ be a $1$-parameter family of metrics on $\mathcal{M}$,  with $g_{ab}(0) = g_{ab}$. The linearized metric 
$
\delta g_{ab} = \tfrac{d}{d\lambda} g_{ab}(\lambda) \big{|}_{\lambda = 0}
$
solves the linearized vacuum Einstein equations on $\mathcal{M}$ if 
\begin{align} \label{eq:linEin} 
\delta E_{ab} = 0,
\end{align} 
where $\delta E_{ab}$
is the linearization of the Einstein tensor at $g_{ab}$ in the direction of $\delta g_{ab}$. 
Due to the covariance of Einstein's equations, the space of solutions of 
the linearized Einstein equation
is invariant under gauge transformations 
\index{D3deltagtilde@$\delta \tilde g_{ab}$}
\begin{align}\label{eq:lin-gauge}
\delta g_{ab} \to \delta \tilde g_{ab} = \delta g_{ab} - 2 \nabla_{(a} \GenVec_{b)} .
\end{align} 
Upon introducing a suitable gauge condition,  e.g. harmonic gauge 
$
\nabla^a (\delta g_{ab} - \half \delta g_c{}^c g_{ab} ) = 0, 
$
the linearized Einstein equation becomes hyperbolic, and it follows from standard results  that the Cauchy problem for the linearized vacuum Einstein equation on $\mathcal{M}$ admits global solutions. A priori, these may have exponential growth.  

\index{D3deltag@$\varop g_{ab}$}
Let $\delta g_{ab}$ be a solution of the linearized vacuum Einstein equation on 
$\mathcal{M}$, and let $n^a$ be the ingoing principal null vector, cf.  \eqref{eq:n-intro}. The fact that Kerr is of Petrov type D implies there is a vector field $\GenVec^a$ such that the gauge transformed metric $\delta  \tilde g_{ab}$ satisfies \cite{2007CQGra..24.2367P}
\begin{equation}\label{eq:ORGcond} 
n^b\varop \tilde g_{ab}=0, \quad  g^{ab}\varop \tilde g_{ab}=0 . 
\end{equation} 
The resulting gauge condition is called the outgoing radiation gauge\footnote{Replacing $n^a$ by $l^a$ leads to the ingoing radiation gauge condition. The result of \cite{2007CQGra..24.2367P} is valid more generally for linearized gravity on vacuum background spacetimes of Petrov type II.}.  
For a linearized metric in outgoing radiation gauge, the only non-vanishing components are 
\index{G2Gcomp@$G_{i0'}$}
\begin{align} \label{eq:ORG-G-intro}
G_{00'} = \varop g_{ab} l^a l^b, \quad G_{10'} = \varop g_{ab} l^a m^b, \quad G_{20'} = \varop g_{ab} m^a m^b .
\end{align}

To state our main results, we define the time functions
\index{T1tHor@$\tHor$}
\index{T1timefunc@$\timefunc$}
\begin{subequations}
\begin{align} 
\tHor ={}& v - h/2, \\ 
\timefunc ={}& v - h \label{eq:timefun-intro} ,
\end{align}
\end{subequations}
where 
\index{C2CHyp@$\CInHyperboloids$}
\begin{align} \label{eq:h-def-intro} 
h(r)={}&2 (r- r_{+})
+ 4 M \log\left(\frac{r}{r_{+}}\right)
+ \frac{3 M^2 (r_{+} -  r)^2}{r_{+} r^2}
+ 2 M \arctan\left(\frac{(\CInHyperboloids-1) M}{r}\right)\nonumber\\
&-  2 M\arctan\left(\frac{(\CInHyperboloids-1) M}{r_{+}}\right)  
\end{align}
and where $\CInHyperboloids$ is sufficiently large, which for concreteness we take to be $\CInHyperboloids=10^6$.
The motivation behind these choices is explained in sections \ref{sec:foliation} and \ref{sec:rpForSpinWeightedWaveEquations}.
Let $\timefunc_0 = 10M$ and define  
\index{S4Sigmainit@$\Stauini$}
\begin{align} \label{eq:Sigmainit-def-intro} 
\Stauini = \{ \tHor = \timefunc_0 \} \cap \{ r > r_+\} . 
\end{align} 

Let $k \in \Naturals$ and $\alpha \in \Reals$. For tensors $\varpi_{a\cdots d}$ along $\Stauini$, let $H^k_\alpha(\Stauini)$ be the weighted Sobolev space with norm 
\index{H2Hkalpha@$H^k_\alpha$} 
\index{0normHkalpha@$\Vert \cdot \Vert_{H^k_\alpha(\Stauini)}$}  
\begin{align}\label{eq:Hkbeta-def}
\Vert \varpi \Vert^2_{H^k_\alpha(\Stauini)} = \int_{\Stauini} M^{-\alpha} \sum_{i=0}^k r^{\alpha +2i -1} \vert \nabla^i \varpi \vert^2_{g_E}  dr\sin\theta d\theta d\phi ,
\end{align}
where the squared modulus \index{0NormgE@$\vert\cdot\vert_{g_E}$} $\vert \varpi\vert^2_{g_E}$ of a tensor is defined in terms of the positive definite metric $g_{E\, ab} = 2T_a T_b - g_{ab}$, with $T^a$ the timelike unit normal of $\Stauini$. 

The following theorem provides the first proof of linearized stability for Kerr black holes with $|a|\ll M$, a major step towards the solution of the black hole stability problem. 
The proof has several novel features, which are discussed in section \ref{sec:strategy}.  Furthermore, our second main theorem, theorem \ref{thm:BEAMintro}, proves linear stability for the whole subextreme range $|a| < M$ provided that the basic decay condition in definition \ref{def:BEAM-intro} holds. This basic decay estimate relies on a Morawetz estimate that is expected to be valid for the full subextreme range. This thus reduces the linear stability problem for the full subextreme range of Kerr black hole spacetimes to the proof of  a Morawetz estimate. This modularity of the proof of our main result is an important and desirable feature of the present work.

\begin{thm}\label{thm:mainintro} Let  $(\mathcal{M}, g_{ab})$ be the domain of outer communication of a Kerr spacetime with $|a|/M \ll 1$. Let $\ireg \in \Naturals$ be sufficiently large and $\epsilon > 0$ be sufficiently small. 
Let  $\delta g_{ab}$ be a solution to the linearized vacuum Einstein equations on $(\mathcal{M}, g_{ab})$ in outgoing radiation gauge, with 
$
\Vert \delta g \Vert_{H^{\ireg}_{7}(\Stauini)} < \infty , 
$
and let $G_{i0'}$, $i=0,1,2$ be the components of $\delta g_{ab}$ defined by \eqref{eq:ORG-G-intro}, with respect to the Znajek tetrad.
Let $|\delta g|^2 = |G_{00'}|^2 + |G_{10'}|^2 + |G_{20'}|^2$. 
There is a constant $C = C(\ireg, |a|/M, \epsilon)$, such that the inequality 
\begin{align} 
\label{eq:mainintro}
|\delta g| \leq C M^{5/2-\epsilon} r^{-1} t^{-3/2+\epsilon} \Vert \delta g \Vert_{H^{\ireg}_{7}(\Stauini)},
\end{align}
holds for $\timefunc > 10M$. 
\end{thm}

\begin{remark}\label{rmk:ZnajekVsGHPIntro}
Within this introduction, for simplicity, the quantities in theorem \ref{thm:mainintro} are explicitly defined in terms of the Znajek tetrad. In the the main body of the paper, we use the corresponding quantities defined within the GHP formalism \cite{GHP,MR944085,MR917488}, which depends only on a choice of a pair of null directions (and of time orientation), rather than a choice of tetrad $(\NPl^a,\NPn^a,\NPm^a,\NPmbar^a)$. 
We then choose the unique pair of null directions given geometrically by the principal null vectors. This remark applies equally to theorem \ref{thm:BEAMintro} below. Hence, theorems \ref{thm:mainintro} and \ref{thm:BEAMintro} can be understood to hold in this more general setting. 
See section \ref{sec:background} for background on the GHP formalism, and section \ref{sec:notandconv} and in particular remark \ref{rmk:NPLocallyTrivialisesGHP} for details. 
\end{remark}

In order to be able to present a more refined result in theorem \ref{thm:BEAMintro}, we now discuss in more detail the Kerr geometry, curvature components, differential operators, the famous Teukolsky equations satisfied by the curvature components, energies, and a basic decay condition that was previously shown to hold when $|a|\ll M$ and that is the key assumption in theorem \ref{thm:BEAMintro}. 

\index{I2IScriPlus@$\Scri^+$}
\index{I1i0@$i_0$}
\index{I1iplus@$i^+$}
Considering the conformally rescaled metric $r^{-2} g_{ab}$ allows one to add a boundary at $r=\infty$ \cite{hawking_ellis_1973}. This contains the smooth null manifold $\Scri^+$, which is called future null infinity and which represents the limits of those future directed null geodesics that reach infinity. The conformal boundary contains a point $i^+$, which is called timelike infinity and which is the future end point of $\Scri^+$, $\Horizon^+$, and all inextendible timelike geodesics. There is also a point $i^0$, which is called spacelike infinity  and which is the limit of all spacelike geodesics that reach infinity.

As we will see in section \ref{sec:foliation}, the level sets of $\tHor$ are asymptotic to $i_0$ and induce a foliation of $\Horizon^+$. The level sets of $\timefunc$ are regular at both $\Horizon^+$ and $\Scri^+$, and they induce foliations of the future part of the event horizon $\Horizon^+$ and future null infinity $\Scri^+$, as do translated hyperboloids in Minkowski space. For these reasons, we refer to $\tHor$ as the horizon crossing time and $\timefunc$ as the hyperboloidal time function. 
We denote level sets of the hyperboloidal time function $\timefunc$ by $\Sigma_\timefunc$. For $\timefunc_1 < \timefunc_2$, let $\Omega_{\timefunc_1, \timefunc_2}$ denote the spacetime domain given by the intersection of the past of $\Sigma_{\timefunc_2}$, with the future of $\Staui$. $\Staui$, $\Sigma_{\timefunc_0}$ are illustrated in figure \ref{fig:Estimates:SigmaInit}, and $\Dtau$ in figure \ref{fig:Estimates:Teukolskyrp}. 

\begin{figure}[t]
\centering
\begin{subfigure}[t]{0.45\textwidth}
\centering 
\begin{tikzpicture}[xscale=0.80,yscale=0.80]

\coordinate (A0) at (9,1);
\coordinate (A1) at (6.5,1);
\coordinate (A2) at (5,1.1);
\coordinate (A3) at (2.04,2.54);

\coordinate (B0) at (7.5,2.5);
\coordinate (B1) at (6,2);
\coordinate (B2) at (5,1.5);
\coordinate (B3) at (2,2.5);

\coordinate (i+) at (5,5);

\draw[thick,blue,name path=horizon]  (1,1) -- node[sloped, above] {$\mathscr{H}^+$} (5,5);
\draw[name path=scri, ,thick, dash dot, red]  (5,5) -- node[sloped, above, near end] {$\Scri^+$} (B0);

\path[name path=consttau0, gray] (B0) .. controls (B1) and (B2) .. node[pos=0.6, sloped, above, black] {$\Sigma_{\timefunc_0}$} (B3); 

\path[name path=sigmainit, gray] (A0) .. controls (A1) and (A2) .. node[pos=0.6, sloped, below, black] {$\Stauini$} (A3);

\path[name intersections={of=consttau0 and horizon, by=tau0horizon}]  (1,1) -- (5,5);
\path[name intersections={of=sigmainit and horizon, by=sigmainithorizon}]  (1,1) -- (5,5);

\begin{scope}
  \clip [closed]  (1,1) -- (5,5) --  (10,5) -- (10,1) -- (5,0);
  \fill[lightgray]  (A0) .. controls (A1) and (A2) ..  (A3) -- (B3) .. controls (B2) and (B1) .. (B0) -- cycle;
\end{scope}
\begin{scope}
  \clip [closed]  (1,1) -- (5,5) --  (10,5) -- (10,1) -- (5,0);
  \draw[black, thick]  (A0) .. controls (A1) and (A2) ..  (A3) -- (B3) .. controls (B2) and (B1) .. (B0);
\end{scope}

\draw[black, thick,dashed] (A0) --  (B0);

\filldraw[fill=white,draw=black] (i+) circle[thick,radius=.7mm] node[anchor=south,yshift=.1cm]{$i_+$}; 
\filldraw[fill=white,draw=black] (A0) circle[thick,radius=.7mm] node[anchor=west,xshift=.1cm]{$i_0$}; 

\end{tikzpicture}%
\caption{Initial hypersurface $\Stauini$ and a level set of $\timefunc$.}
\label{fig:Estimates:SigmaInit}
\end{subfigure}%
\hspace{0.05\textwidth}
\begin{subfigure}[t]{0.45\textwidth}
\centering 
\begin{tikzpicture}[xscale=0.80,yscale=0.80]

\coordinate (A0) at (8.5,1.5);
\coordinate (A1) at (6.4,0.8);
\coordinate (A2) at (5,0.1);
\coordinate (A3) at (1,1.5);

\coordinate (B0) at (7.5,2.5);
\coordinate (B1) at (6,2);
\coordinate (B2) at (5,1.5);
\coordinate (B3) at (2,2.5);

\coordinate (i+) at (5,5);
\coordinate (C1) at (6,2);
\coordinate (C2) at (5,1);
\coordinate (C3) at (4.5,0);

\draw[thick,blue,name path=horizon]  (1,1) -- node[sloped, above] {$\mathscr{H}^+$} (5,5);
\draw[name path=scri, ,thick, dash dot, red]  (5,5) -- node[sloped, above] {$\Scri^+$} (9,1);

\path[name path=consttau0, gray] (A0) .. controls (A1) and (A2) .. node[pos=0.52, below, black] {$\Staui$} (A3);
\path[name path=consttau,gray] (B0) .. controls (B1) and (B2) .. node[pos=0.60, sloped, above, black] {$\Stau$} (B3);

\begin{scope}
  \clip [closed]  (1,1) -- (5,5) --  (10,5) -- (10,1) -- (5,0);
  \fill[fill=lightgray]  (A0) .. controls (A1) and (A2) ..  (A3) -- (B3) .. controls (B2) and (B1) .. (B0) -- cycle;
\end{scope}

\begin{scope}
  \clip [closed]  (1,1) -- (5,5) --  (10,5) -- (10,1) -- (5,0);
  \draw[draw=black, thick]  (A0) .. controls (A1) and (A2) ..  (A3) -- (B3) .. controls (B2) and (B1) .. (B0);
\end{scope}

\draw[black, thick, name intersections={of=consttau0 and horizon, by=htau0}, name intersections={of=consttau and horizon, by=htau}]  (htau0)  --  (htau);

\draw[black, thick,dashed] (A0) --  (B0);

\filldraw[fill=white,draw=black] (i+) circle[thick,radius=.7mm] node[anchor=south,yshift=.1cm]{$i_+$}; 
\filldraw[fill=white,draw=black] (9,1) circle[thick,radius=.7mm] node[anchor=west,xshift=.1cm]{$i_0$}; 

\node at (4.65,1.3) {$\Dtau$};

\end{tikzpicture}%
\caption{The region between two level sets of $\timefunc$, used in the $r^p$ argument applied to solutions of the Teukolsky equation.}
\label{fig:Estimates:Teukolskyrp}
\end{subfigure}
\caption{Regions in which estimates are proved.}
\label{fig:Regions}
\end{figure}

Let $\delta C_{abcd}$ be the linearized Weyl tensor. 
Due to the fact that Kerr is Petrov type D, the linearized Newman-Penrose scalars  
\index{T3thetaPsi0@$\vartheta \Psi_0$}
\index{T3thetaPsi4@$\vartheta \Psi_4$}
\begin{align} \label{eq:extreme-lin-intro}
 \vartheta \Psi_0 = - \delta C_{abcd} l^a m^b l^c m^d ,  \quad 
 \vartheta \Psi_4 = - \delta C_{abcd} n^a \bar m^b n^c \bar m^d, 
\end{align} 
are gauge invariant. 
For a solution of the linearized vacuum Einstein equation on the Kerr background spacetime, $\vartheta \Psi_0, \vartheta \Psi_4$ solve a pair of decoupled wave equations called the Teukolsky Master Equations  or simply the Teukolsky equations \cite{1972PhRvL..29.1114T}, 
and also satisfy a set of fourth-order differential relations called the Teukolsky-Starobinsky Identities \cite{1974ApJ...193..443T,1974JETP...38....1S}. 
We refer to $\vartheta\Psi_0, \vartheta\Psi_4$ (and any rescalings of them) as the Teukolsky variables. 
The Teukolsky equations and the Teukolsky-Starobinsky identities can be written in many equivalent forms, for example by rescaling the equations, rescaling the Teukolsky variables, or changing coordinates; in sections \ref{sec:TeukolskyEquations} and \ref{sec:TSI} we derive forms of these equations that are well suited to our analysis. 
Building on previous work, in section \ref{sec:transporteq}, we show that solutions of the linearized Einstein equation are uniquely determined by the Teukolsky equations for $\vartheta \Psi_0, \vartheta \Psi_4$, the Teukolsky-Starobinsky Identities, and a set of transport equations along $n^a$, for the metric components \eqref{eq:ORG-G-intro} as well as for tetrad components of the linearized connection coefficients. 

A classical approach to understanding decay is by investigating behaviour near null infinity in the rescaled geometry (see e.g.{} \cite{MR944085} for a textbook treatment). The compactified hyperboloidal coordinate system $(\timefunc,R,\theta,\phi)$, with $R = 1/r$ and $(\theta,\phi)$ as in the ingoing Eddington-Finkelstein coordinates, is regular at $\Scri^+$ considered as a null hypersurface in the conformally rescaled metric $r^{-2} g_{ab}$, as is the rescaled tetrad 
\begin{align} \label{eq:resc-Znajek}
(r^2l^a, n^a, rm^a,  r \bar m^a) ,
\end{align} 
where $(l^a, n^a, m^a, \bar m^a)$ are given by \eqref{eq:Znajek-intro}. 
The asymptotic behaviours at $\Scri^+$ of tensor fields on $\mathcal{M}$, often referred to as peeling, can be understood by passing to the conformal compactification, working with a conformally rescaled version of the field that is regular at $\Scri^+$, and using the rescaled tetrad \eqref{eq:resc-Znajek}. A peeling analysis \cite[section 9.7]{MR944085} indicates 
\begin{subequations} 
\begin{align} 
\vartheta \Psi_0 ={}& O(r^{-5}), \qquad \vartheta \Psi_4 = O(r^{-1}) , \\ 
G_{i0'} ={}& O(r^{-3+i}),  \quad i=0,1,2 .
\label{eq:indicatedPeeling}
\end{align} 
\end{subequations}
The scalars $\vartheta \Psi_0, \vartheta \Psi_4$ are properly weighted in the sense of Geroch, Held, and Penrose (GHP) \cite{GHP} and have boost- and spin-weights $+2, -2$, respectively. 

In the following, we transform properly weighted scalars and operators to boost-weight zero by rescaling with powers of a factor $\uplambda$ with boost weight $1$ and spin weight $0$, which takes the form $\uplambda=1$ in the Znajek tetrad\footnote{\label{foot:lambda}
$\uplambda = (\sqrt{2}(r-ia\cos\theta)\rho')^{-1}$ has the desired property, where $\rho' = \bar m^a m^b \nabla_b n_a$ is one of the GHP spin coefficients.}. 
\index{L3lambda@$\uplambda$}
Let
\index{P3psibase+2@$\psibase[+2]$}
\index{P3psibase-2@$\psibase[-2]$}
\begin{subequations} \label{eq:radfield-intro} 
\begin{align} 
\psibase[+2] ={}& \tfrac{1}{2} (a^2+r^2)^{1/2} (r-ia\cos\theta)^4 \uplambda^{-2}\vartheta\Psi_{0} , \label{eq:radfield(+2)-intro} \\
\psibase[-2] ={}& \tfrac{1}{2} (a^2 + r^2)^{1/2} \uplambda^{2} \vartheta \Psi_{4} .
\label{eq:radfield(-2)-intro}
\end{align} 
\end{subequations}
Then $\psibase[+2],\psibase[-2]$ have boost-weights $0$ and spin-weights $+2, -2$, respectively.  The fields $\psibase[+2], \psibase[-2]$ are the deboosted radiation fields of $\vartheta \Psi_0, \vartheta \Psi_4$, respectively, in particular they are regular, in the sense of spin-weighted fields, and non-degenerate on $\mathcal{M}$ including $\Horizon^+$ and $\Scri^+$. 
In the following, unless otherwise stated, we shall consider only fields with boost-weight $0$. 

In order to discuss our estimates for the Teukolsky Master Equations, 
we introduce operators acting on fields of spin-weight $s$, which, restricting to the Znajek tetrad and the ingoing Eddington-Finkelstein coordinate system, take the explicit form  
\index{Y2Y@$\YOp$}\index{V2V@$\VOp$}
\index{E3ethhat@$\hedt, \hedtp$}
\begin{subequations}\label{eq:ops-Znaj-IEF-intro}
\begin{align}
\VOp\varphi={}&\partial_{v} \varphi
 + \frac{\Delta \partial_{r} \varphi}{2 (a^2 + r^2)}
 + \frac{a \partial_{\phi} \varphi}{a^2 + r^2}, \label{eq:V-Znaj-IEF-intro} \\
\YOp\varphi={}&- \partial_{r} \varphi,\label{eq:Y-Znaj-IEF-intro}\\
\hedt\varphi={}&\tfrac{1}{\sqrt{2}}\partial_{\theta} \varphi
 + \tfrac{i}{\sqrt{2}}\csc\theta \partial_{\phi} \varphi
 - \tfrac{1}{\sqrt{2}}s \cot\theta \varphi,\\
\hedtp\varphi={}&\tfrac{1}{\sqrt{2}}\partial_{\theta} \varphi
 -  \tfrac{i}{\sqrt{2}}\csc\theta \partial_{\phi} \varphi
 + \tfrac{1}{\sqrt{2}}s \cot\theta \varphi
 .
, \\
\intertext{and} 
\LxiOp{}\varphi={}&\partial_{v} \varphi,
\end{align}
\end{subequations}
\index{B2BOpSet@$\unrescaledOps$}
\index{D2DOpSet@$\rescaledOps$}
\index{D2DSlashOpSet@$\ScriOps$}
We introduce the set of operators 
\begin{subequations}
\begin{align}
\unrescaledOps
={}&\{ \YOp, \VOp,r^{-1} \hedt, r^{-1} \hedtp\} ,\\
\intertext{related to the principal tetrad, and the set} 
\rescaledOps
={}&\{ M\YOp, r\VOp, \hedt, \hedtp\}, \label{eq:rescaledOps-intro} \\
\intertext{of rescaled operators. Finally, the set} 
\ScriOps={}&\{\hedt,\hedtp, M\LxiOp\} 
\end{align} 
\end{subequations}
is appropriate for controlling fields on $\Scri^+$. In stating integral estimates, we shall make use of the volume elements 
\index{D1di3@$\diThreeVol$}
\index{D1di4@$\diFourVol$}
\begin{align}
\di^4 \mu = \sin\theta d v \wedge d r \wedge  d \theta \wedge d \phi, \qquad \di^3 \mu = \sin\theta d r \wedge  d \theta \wedge d \phi .
\label{eq:referenceVolumeForms-intro}
\end{align} 

\begin{definition}
\label{def:basichighorderenergy-intro}
Let $\Sigma$ be a smooth, spacelike hypersurface, and let $\Normal_a$ be a 1-form normal to $\Sigma$. Let $\LerayForm{\Normal}$ denote a three form such that $\Normal\wedge\LerayForm{\Normal}=\di^4\mu$. Let $\varphi$ be a boost-weight zero field.
\index{E2Energy@$E^k_\Sigma$}
Let $\ireg$ be a positive integer and define
\index{D1di3Leray@$\LerayForm{\Normal}$}
\begin{subequations}
\begin{align} 
E^1_\Sigma(\varphi)
={}& M\int_{\Sigma}\left(
  (\Normal_aY^a)\abs{ \VOp\varphi}^2
  +(\Normal_a V^a)\abs{\YOp\varphi}^2
  +(\Normal_a(V^a+Y^a))r^{-2}(\abs{\hedt\varphi}^2+\abs{\hedtp\varphi}^2)
\right)\LerayForm{\Normal} , 
\label{eq:E0def-intro}\\
E^{\ireg}_\Sigma(\varphi)
={}& \sum_{i=0}^{\ireg-1} \sum_{X_1,\ldots,X_i\in \unrescaledOps} M^{2i} E^1_\Sigma(X_i\ldots X_1 \varphi),\\
\BEAMbulk^1_{\timefunc_1, \timefunc_2}(\varphi) = {}& 
\int_{\Dtau\cap\{r\geq 10M\}} M^3 r^{-3} \sum_{X \in \unrescaledOps}\abs{X \varphi}^2 \diFourVol
  +\int_{\Dtau} M r^{-3} \abs{\varphi}^2 \diFourVol,\\
\BEAMbulk^{\ireg}_{\timefunc_1, \timefunc_2}(\varphi)
={}& \sum_{i=0}^{\ireg-1} \sum_{X_1,\ldots,X_i\in \unrescaledOps} M^{2i} W^1_{\timefunc_1,\timefunc_2}(X_i\ldots X_1 \varphi).
\end{align}
\end{subequations}
\end{definition}
In order to discuss our second main result, we shall
need the fields
\begin{align} 
\psibase[-2][i]={}&\left ( \frac{a^2 + r^2}{M}  \VOp \right )^i \psibase[-2], 
\qquad 0\leq i\leq 4 , \label{eq:radpsii-intro-2} 
\end{align}
defined in terms of derivatives of $\psibase[-2]$.

\begin{definition}[Basic decay condition]
\label{def:BEAM-intro} \ \\
Let $\delta g_{ab}$ be a solution to the linearized Einstein equations on the domain of outer communication $\mathcal{M}$ of a Kerr black hole spacetime, and let $\psibase[+2]$ be as in \eqref{eq:radfield(+2)-intro}, and let $\psibase[-2][i]$, $i=0,1,2$ be as in \eqref{eq:radpsii-intro-2}. We shall say that $\varop g_{ab}$ satisfies the basic decay condition if the following holds for all sufficiently large $\ireg \in \Naturals$. 
\begin{enumerate}
\item \label{point:s=-2-BEAM} There is a positive constant $C$ such that for all $\timefunc_1 < \timefunc_2$ with $10M\leq \timefunc_1$,   
\begin{align}
\sum_{i=0}^2&\left(
  E^{\ireg}_{\Stau}(\psibase[-2][i])
  + \BEAMbulk^{\ireg}_{\timefunc_1, \timefunc_2}(\psibase[-2][i])\right )
\leq C \sum_{i=0}^2 E^{\ireg}_{\Staui}(\psibase[-2][i]) ,
\label{eq:BEAMassumptionSpin-2-intro}
\end{align}
\item \label{point:s=+2-BEAM}
\begin{align} \label{eq:s=+2-BEAM-intro}
\lim_{\timefunc\rightarrow\pm \infty} \left ( \absScri{
\psibase[+2]}{\ireg} \big{|}_{\Scri^+} \right )  = 0 .
\end{align}
\end{enumerate}
\end{definition}

\begin{remark} 
The spin-weight $-2$ case, point \ref{point:s=-2-BEAM}, of definition \ref{def:BEAM-intro} is an integrated energy estimate. The spin-weight $+2$ condition in point \ref{point:s=+2-BEAM}, on the other hand, is not in the form of an estimate, but rather a weak pointwise decay condition. 
In section \ref{sec:spin+2TeukolskyEstimates}, equation \eqref{eq:s=+2-BEAM-intro} is proved to follow from a basic integrated energy estimate analogous to the condition stated in inequality \eqref{eq:BEAMassumptionSpin-2-intro}.
\end{remark} 

We are now able to formulate the second main result of this paper. 
\begin{thm}\label{thm:BEAMintro} 
Let  $(\mathcal{M}, g_{ab})$ be the domain of outer communication of a subextreme Kerr spacetime. Let $\ireg \in \Naturals$ be sufficiently large and $\epsilon > 0$ be sufficiently small. Let  $\delta g_{ab}$ be a solution to the linearized vacuum Einstein equations on $(\mathcal{M}, g_{ab})$ in outgoing radiation gauge, with 
$\Vert \delta g \Vert_{H^{\ireg}_{7}(\Stauini)} < \infty$,
and let $G_{i0'}$, $i=0,1,2$ be the components of $\delta g_{ab}$ defined by \eqref{eq:ORG-G-intro}, with respect to the Znajek tetrad.

Assume that $\delta g_{ab}$ satisfies the basic decay conditions of definition \ref{def:BEAM-intro}. Then, there is a constant $C = C(\ireg,|a|/M,\epsilon)$, such that the following inequalities hold for $\timefunc > 10M$. 
\begin{enumerate} 
\item In the interior region $r < \timefunc$, 
\begin{subequations}
\begin{align}  
\abs{G_{20'}} \leq{}& C r^{-1} \timefunc^{-5/2 + \epsilon} \Vert \delta g \Vert_{H^{\ireg}_{7}(\Stauini)},  \\ 
\abs{G_{i0'}} \leq{}& C r^{-2} \timefunc^{-3/2 + \epsilon} \Vert \delta g \Vert_{H^{\ireg}_{7}(\Stauini)}, \quad \text{for $i\in \{0,1\}$.}  \label{eq:i=0,1-slow-intro} 
\end{align} 
\end{subequations} 

\item In the exterior region $r  \geq \timefunc$, 
\begin{align} \label{eq:peel-est-intro}
\abs{G_{i0'}} \leq C r^{i-3} \timefunc^{-i -1/2 + \epsilon} \Vert \delta g \Vert_{H^{\ireg}_{7}(\Stauini)}, \quad \text{for $i\in \{0,1,2\}$.}\end{align} 
\end{enumerate} 
\end{thm} 

\begin{remark} \label{rem:1.7}
\begin{enumerate}
\item It follows from the work in \cite{2017arXiv170807385M} together with the arguments in section~\ref{sec:spin+2TeukolskyEstimates} that the conditions stated in definition~\ref{def:BEAM-intro} hold for a solution $\delta g_{ab}$ of the linearized Einstein equation with $\Vert \delta g \Vert_{H^{\ireg}_{7}(\Stauini)} < \infty$ on a Kerr spacetime that is very slowly rotating, in the sense that $|a|\ll M$. 
\item As part of the proof of theorems \ref{thm:mainintro} and \ref{thm:BEAMintro} we prove decay estimates for $\psibase[-2]$ which are stronger than those previously available.  For example, the rate of decay in $\timefunc$ for fixed $r$ in theorem \ref{thm:ImproEstipsi01} is stronger than the previous pointwise decay for the Teukolsky variable in \cite{2017arXiv171107944D} or in the spherically symmetric case in \cite{2016arXiv160106467D}. 
\item \label{point:rem:1.7:3}
The fall-off conditions on initial data in theorems \ref{thm:mainintro} and \ref{thm:BEAMintro} imply that the linearized mass and angular momentum vanish. 
To see this, first note that an explicit calculation shows that the mass and angular momentum for the Kerr spacetime are given by the corresponding ADM expressions \cite{1974AnPhy..88..286R}. The ADM mass is given as integrals over spheres of  partial derivatives of the metric, with respect to the standard area measure $r^2\diTwoVol$, and taking the limit as the radius of the sphere goes to infinity, $r\rightarrow\infty$. The ADM angular momentum is similar with an additional factor that is bounded by $r$. 
The condition that $\Vert \delta g \Vert_{H^{\ireg}_{7}(\Stauini)}$ is finite implies that, on the initial hypersurface, $|\delta g| \leq C r^{-7/2}$ and, more importantly, the partial derivatives fall off as $r^{-9/2}$; see lemmas \ref{lem:spherica-Sobolev} and \ref{lem:Ik+1alphaenergydominatesPkalphapointwise}. 
The fall-off in the partial derivatives is sufficiently strong that when integrated against the relevant weights, and using the fact that $M>0$ by assumption, one finds that $\delta M=0=\delta a$.
This does not restrict the dynamical degrees of freedom, due to the fact that in linearized gravity, variations of the mass and angular momentum are quasi-locally conserved and be treated separately \cite{2013CQGra..30o5016A}. 
In appendix \ref{sec:linpara}, we calculate the linearized perturbations from varying $M$ and $a$ in the metric \eqref{eq:gab-intro}. 
\end{enumerate}
\end{remark}

\subsection{Background and context} \label{sec:background} 
The work in the present paper is motivated by the black hole stability conjecture, i.e. the statement that the maximal Cauchy development of Cauchy data close to data for a subextreme Kerr black hole spacetime is future asymptotic to a subextreme member of the Kerr family. The black hole stability conjecture is, together with the black hole uniqueness conjecture and the Penrose Inequality,  one of three major conjectures in general relativity related to the Kerr black hole solution that were formulated in the early 1970's, cf. \cite{2015CQGra..32l4006T} and references therein.  These conjectures are fundamental not only from the point of view of cosmology and astrophysics,
but are among the most important open problems in mathematics in the present day, and have been the subject of intense work both in the physics and mathematics communities during the last half-century and are, in spite of tremendous progress,
still open in their full generality\footnote{See section \ref{sec:recent} for developments subsequent to the first submission of the present paper.}.

The black hole stability problem has some features in common with the problem of stability of Minkowski space \cite{MR1316662,1986CMaPh.107..587F}, but exhibits several new types of difficulties that constitute major obstacles to progress, and which all have to be overcome in order to solve the black hole stability problem.  These include the fact that the Kerr spacetime has only two Killing symmetries in contrast to the Poincar\'e symmetry of Minkowski space. Superradiance caused by the rotating geometry prevents the existence of a positive, conserved energy for waves on the Kerr background. The Kerr black hole solution constitutes a 2-parameter family, and the parameters of the ``final'' Kerr black hole cannot be determined a priori. In addition one has global gauge degrees of freedom corresponding to changes of reference frame that must be controlled during the evolution.  

The importance of the fact that the Kerr geometry is algebraically special, of Petrov type D (also called type $\{2\,2\}$) can hardly be overemphasized. A Petrov type D spacetime is characterized by the existence of two repeated principal null directions. The Petrov type D nature of the Kerr geometry leads to the existence of the Carter constant and associated second order symmetry operator, called the Carter operator.  

The Petrov type can be understood in a particularly simple manner using a formalism based on complex null tetrads, and the underlying spin dyad, such as the Newman-Penrose (NP) \cite{1962JMP.....3..566N} or its later development, the Geroch-Held-Penrose (GHP) formalism \cite{GHP}. See \cite{MR917488,MR944085} for a textbook treatment. The GHP formalism has the additional significant feature that it is covariant with respect to boost- and spin-rotations. 
The NP or GHP formalism expresses tensorial quantities in terms of complex scalar components. In terms of the GHP formalism, these can be understood as sections of complex line bundles.  
The NP and GHP formalisms have been widely applied and have led to many important results, several of which are of importance in this paper, and will be discussed below. These include 
the derivation of the Teukolsky equation \cite{1972PhRvL..29.1114T}, the introduction of the ORG \cite{Chrzanowski}, and the proof that the ORG can be imposed in type D spacetimes \cite{2007CQGra..24.2367P}.
The GHP formalism is closely analogous to the null decomposition, which originates in the proof of stability of Minkowski space \cite{MR1316662} (and the work leading up to it) and continues to many recent works such as \cite{2016arXiv160106467D,2017arXiv171107597K,KlainermanSzeftel:GCMSphereEffectiveResults,KlainermanSzeftel:GCMSphereConstruction,Klainerman:2021qzy,GioriKlainermanSzeftel:WaveEstimatesForKerr}.
When applying the GHP formalism in Petrov type D spacetime, it is natural to use a principal tetrad, where two of the tetrad legs are aligned with the principal null vectors determined by the geometry.

Of crucial importance for Kerr stability is the discovery by 
Teukolsky \cite{1972PhRvL..29.1114T} that the linearized Einstein equation on a Petrov type D background implies a decoupled, separable wave equation,   
the Teukolsky Master Equation (TME), for the linearized Weyl components of extreme spin weights known as the Teukolsky scalars.  Separability of the TME is due to existence of the Carter symmetry operator. The Teukolsky scalars have the important property of being gauge-invariant.  
These facts have made it possible to develop analytical approaches to the black hole stability problem.

For perturbations of a Petrov type D spacetime such as Kerr, it is convenient to use a radiation gauge, such as the ORG \eqref{eq:ORGcond}, defined in terms of one of the principal null directions. The radiation gauge is an algebraic gauge condition that appears naturally for linear perturbations of Kerr constructed using Debye potentials \cite{Chrzanowski}. The radiation gauge is consistent on Kerr, and more generally, Petrov type II spacetimes, in spite of the fact that the linear radiation gauge involves five gauge conditions in a four-dimensional spacetime \cite{2007CQGra..24.2367P}.  

The first major analytical result on Kerr stability was Whiting's 1989 proof of mode stability for the TME \cite{whiting:1989}. Mode stability is the statement that separated solutions to the TME satisfying an outgoing radiation condition do not grow exponentially. The proof of mode stability made use of the fact that the separated form of the TME is a confluent Heun equation, and integral transformations related to this fact.  Although mode stability is a strong indication that stability holds, there remains a significant difficulty in going from a mode stability result to a decay estimate for the field equation, in particular for the full, non-linear stability problem. For this reason, much of the subsequent work on the stability problem made use of different techniques. 
Further work, which can be viewed as being in this direction, includes pointwise decay estimates for various fields, including solutions of the wave and Teukolsky equations \cite{MR2215614,MR3783838}.

A key step in proving decay estimates for fields on black hole spacetimes was to prove Morawetz
estimates in this setting. For the wave equation, this was first accomplished on Schwarzschild \cite{MR1972492,blue2006errata,blue2006uniform,MR2527808}. 
In order to prove Morawetz estimates for fields on the rotating Kerr spacetime, it is necessary to make use of the additional symmetries provided by the Carter constant or the Carter operator. A Morawetz estimate for the wave equation on Kerr spacetimes with $|a|\ll M$ has been proved using Fourier techniques \cite{MR2764864}, and  also by physical space techniques using 
the second order Carter symmetry operator \cite{AnderssonBlue:KerrWave}. 
The problem for the full extreme range $|a|<M$ was also treated \cite{MR3488738} using Fourier techniques and work extending mode stability to the real line \cite{2015AnHP...16..289S}; mode stability on the real line has also been extended to the nonzero spin case \cite{2017JMP....58g2501A}. 
Morawetz estimates can be applied to the Maxwell and Regge-Wheeler equations \cite{MR3373052,MR3450059,MR3672902,2017arXiv170806943A}; see also  \cite{MR2864787}.
Furthermore, Morawetz estimates can be used as a hypothesis to prove pointwise decay using the $r^p$ estimate \cite{DafermosRodnianski:rp} 
and its refinement \cite{MR3859608}.

There has been progress on linear and non-linear stability of Schwarzschild.
In order to prove linear stability, it is necessary to prove estimates for the linearized metric, not merely the Teukolsky variables. 
Linear stability for the Schwarzschild has been proved using a combination of wave estimates for the Teukolsky variables and transport estimates for the metric and connection coefficients \cite{2016arXiv160106467D},  as well as using wave estimates for the metric directly \cite{2017arXiv170202843H}.
The non-linear stability of the Schwarzschild spacetime with respect to polarized axially symmetric perturbations is also known \cite{2017arXiv171107597K}. In particular, the assumptions in the just cited paper imply that the spacetime geometry is asymptotic at timelike infinity, to a Schwarzschild spacetime.

For the case of of slowly rotating Kerr-de Sitter black holes, non-linear stability is known \cite{2016arXiv160604014H}. The presence of a positive cosmological constant in the Kerr-de Sitter case provides exponential fall-off in time, which plays an important role in the proof.  

Energy and Morawetz estimates for the spin-2 Teukolsky equation on Kerr are the crucial hypothesis in theorem \ref{thm:BEAMintro}. This is significantly more difficult in the case of nonzero spin, for many reasons fundamental and technical, not least of which is the absence of a divergence-free stress-energy tensor which could be used in defining energies. Our theorem \ref{thm:mainintro} rests upon the proof of these estimates in this case $|a|\ll M$, which was achieved in \cite{ma2020uniform,2017arXiv170807385M}. See also \cite{2017arXiv171107944D}. 

\subsection{Strategy of the proof} \label{sec:strategy} 
We now summarize six key elements of our proof. The first three of these outline the choice of variables considered and the equations governing them. The remaining three outline the novel aspects of the methods we use to estimate these variables. 

\begin{enumerate}
\item \textbf{Teukolsky evolution and metric reconstruction:} Schematically speaking, the fundamental approach is first to study solutions $\psibase[\pm2]$ of the Teukolsky equation and then to reconstruct the metric (and some of the connection coefficients) from the Teukolsky variables $\psibase[\pm2]$. 
This approach goes back to the classical works of \cite{teukolsky:1973,1973ApJ...185..649P,Chrzanowski} and continues to recent work \cite{2016arXiv160106467D,2017arXiv171107597K,KlainermanSzeftel:GCMSphereEffectiveResults,KlainermanSzeftel:GCMSphereConstruction,Klainerman:2021qzy,GioriKlainermanSzeftel:WaveEstimatesForKerr}. 
\item \textbf{Deboosted GHP variables:} 
The GHP formalism allows for continuity with the classical work in this problem, most notably of Teukolsky et al. \cite{teukolsky:1973,1973ApJ...185..649P}, but also more recent work such as \cite{2007CQGra..24.2367P}. 
As explained in detail in section \ref{sec:notandconv}, the normalization 
$g_{ab} = 2 (l_{(a} n_{b)} - m_{(a} \bar m_{b)})$
is left invariant by a transformation 
\begin{align}
(\NPl^a, \NPn^a, \NPm^a, \NPmbar^a) \mapsto (|\mu|^2\NPl^a,|\mu|^{-2}\NPn^a,(\mu/\bar\mu)\NPm^a,(\bar\mu/\mu)\NPmbar^a) 
\end{align} 
for $\mu \ne 0$.
The GHP formalism is covariant with respect to such transformations. 
A component $\eta$ of a tensor defined in terms of a complex null tetrad will transform  as $\eta \to |\mu|^{2b} (\mu/\bar\mu)^s \eta$ for integers $b,s$ called the boost and spin weights. 
Such scalars are called properly weighted, see definition \ref{def:propweight} below. 

In order to have well-defined norms, we introduce a process we call deboosting. 
For a GHP scalar with non-zero boost weight, the natural pointwise norm given by $|\alpha|^2=\alpha\bar\alpha$ fails to be invariant under the transformations discussed in the previous paragraph. 
As already alluded to in the discussion of the deboosted variables $\psibase[\pm2]$, defined in equation \eqref{eq:radfield(-2)-intro}, the Kerr spacetime admits a nowhere vanishing quantity $\uplambda$ with boost weight $1$. Since boost weight is additive for products, for a properly weighted GHP scalar $\alpha$ with boost weight $b$, we can define a deboosted variable $\hat\alpha=\uplambda^{-b}\alpha$ with boost weight zero.
For properly weighted GHP scalars with boost weight zero and arbitrary spin weight, $|\alpha|$ has spin and boost weight zero, i.e. depending on the position, but invariant under the transformations in the previous paragraph.

In addition to deboosting, in definition \ref{def:BoostWeightZeroQuantities}, we find it useful to introduce additional rescaling factors (with zero boost weight) in the definition of our main variables to put system \eqref{eq:TransEqsKerr} in a more tractable form. 

\item \textbf{Hierarchy of transport equations for metric reconstructions in ORG:}
We use a hierarchy of transport equations to reconstruct the metric. In lemma \ref{lem:TransportSystem}, we introduce this hierarchy of transport equations, which are of the form 
\begin{align} \label{eq:transport-intro} 
\YOp \varphi ={}& \varrho ,
\end{align}
for an ordered list of  variables. 
For each $\varphi$ after the first in this list, there is a corresponding transport equation, and the corresponding source term $\varrho$ depends only on variables earlier in the list.
This list begins with $\psibase[-2]$ and includes all non-vanishing metric coefficients. This is discussed in section \ref{sec:transporteq} and illustrated in figure~\ref{fig:EstimatesDiagram}. 

The outgoing radiation gauge (ORG), equation \eqref{eq:ORGcond}, is crucial in allowing us to derive such a hierarchy of transport equations. 
This builds on a long history of using the ORG in the study of metric perturbations of the Kerr spacetime \cite{Chrzanowski,2007CQGra..24.2367P}. 
In spherical symmetry, the double null gauge implies conditions similar to the ORG, and there is a similar hierarchy, which was used in the treatment of the linear stability of Schwarzschild \cite{2016arXiv160106467D,KlainermanSzeftel:polarized}. 

\begin{figure}[t]
\centering
\begin{subfigure}[t]{0.45\textwidth}
\centering
\begin{tikzpicture}[xscale=0.80,yscale=0.80]

\coordinate (A0) at (9,1);
\coordinate (A1) at (6.5,1);
\coordinate (A2) at (5,1.1);
\coordinate (A3) at (2.04,2.54);

\coordinate (B0) at (4,4);
\coordinate (B1) at (7.5,0.5);

\coordinate (i+) at (5,5);
\coordinate (C1) at (6,2);
\coordinate (C2) at (5,1);
\coordinate (C3) at (4.5,0);

\draw[thick,blue,name path=horizon]  (1,1) -- node[sloped, above] {$\mathscr{H}^+$} (5,5);
\draw[name path=scri, thick, dash dot, red]  (5,5) -- node[sloped, above] {$\Scri^+$} (9,1);

\path[name path=sigmainit, gray] (A0) .. controls (A1) and (A2) .. node[pos=0.7, sloped, below, black] {$\Stauini$} (A3);

\path[name path=geod] (B0) -- (B1);

\draw[name path=reqt, dashed] (i+) .. controls (C1) and (C2) .. (C3) node[sloped, pos=0.25, below] {$r=\timefunc$};

\draw[thick, name intersections={of=geod and reqt, by=reqtgeod}, name intersections={of=geod and sigmainit, by=reqtinit}] (reqtgeod) --  node[sloped, above] {$v=\vOne$} (reqtinit);

\begin{scope}
  \clip [closed]  (1,1) -- (5,5) --  (10,5) -- (10,1) -- (5,0);
  \draw[draw=black, thick]  (A0) .. controls (A1) and (A2) ..  (A3);
\end{scope}

\filldraw[fill=white,draw=black] (i+) circle[thick,radius=.7mm] node[anchor=south,yshift=.1cm]{$i_+$}; 

\filldraw[fill=black,draw=black] (reqtgeod) circle[thick,radius=.7mm] ;

\filldraw[fill=white,draw=black] (A0) circle[thick,radius=.7mm] node[anchor=west,xshift=.1cm]{$i_0$}; 

\end{tikzpicture}
\caption{Ingoing null geodesics, tangent to $\NPn$, restricted to the exterior region $r\geq\timefunc$. Along these, the power-series expansion off $\Scri^+$ is used to estimate solutions of the transport equations arising in the reconstruction of the metric.}
\label{fig:Estimates:NPnExt}
\end{subfigure}%
\hspace{0.05\textwidth}
\begin{subfigure}[t]{0.45\textwidth}
\centering
\begin{tikzpicture}[xscale=0.80,yscale=0.80]

\coordinate (A0) at (9,1);
\coordinate (A1) at (6.5,1);
\coordinate (A2) at (5,1.1);
\coordinate (A3) at (2.04,2.54);

\coordinate (B0) at (4,4);
\coordinate (B1) at (6,2);

\coordinate (i+) at (5,5);
\coordinate (C1) at (6,2);
\coordinate (C2) at (5,1);
\coordinate (C3) at (4.5,0);

\draw[thick,blue,name path=horizon]  (1,1) -- node[sloped, above] {$\mathscr{H}^+$} (5,5);
\draw[name path=scri, thick, dash dot, red]  (5,5) -- node[sloped, above] {$\Scri^+$} (9,1);

\path[name path=sigmainit, gray] (A0) .. controls (A1) and (A2) .. node[pos=0.7, sloped, below, black] {$\Stauini$} (A3);

\path[name path=geod] (B0) -- (B1);

\draw[name path=reqt, dashed] (i+) .. controls (C1) and (C2) .. (C3) node[sloped, pos=0.25, above] {$r=\timefunc$};

\draw[thick, name intersections={of=geod and reqt, by=reqtgeod}] (B0) --  node[sloped, below] {$v=\vOne$} (reqtgeod);

\begin{scope}
  \clip [closed]  (1,1) -- (5,5) --  (10,5) -- (10,1) -- (5,0);
  \draw[draw=black, thick]  (A0) .. controls (A1) and (A2) ..  (A3);
\end{scope}

\filldraw[fill=white,draw=black] (i+) circle[thick,radius=.7mm] node[anchor=south,yshift=.1cm]{$i_+$}; 

\filldraw[fill=black,draw=black] (reqtgeod) circle[thick,radius=.7mm] ;

\filldraw[fill=black,draw=black] (B0) circle[thick,radius=.7mm];

\filldraw[fill=white,draw=black] (A0) circle[thick,radius=.7mm] node[anchor=west,xshift=.1cm]{$i_0$}; 

\end{tikzpicture}
\caption{Ingoing null geodesics, tangent to $\NPn$, in the interior region $r\leq\timefunc$. }
\label{fig:Estimates:NPnInt}
\end{subfigure}
\caption{Curves for transport estimates.}
\label{fig:Estimates}
\end{figure}

\item \textbf{5-component system for improved Teukolsky decay:}
\label{point:Strategy:FiveCompSystem}
A major step in the proof of theorem \ref{thm:BEAMintro} is to convert the basic energy and Morawetz estimates of definition \ref{def:BEAM-intro} into strong energy and pointwise decay estimates for the Teukolsky scalars $\psibase[\pm2]$ in theorems \ref{thm:ImproEstipsi01} and \ref{thm:DecayEstimatesSpin+2}. 

The Teukolsky variables $\psibase[-2]$ and its derivatives $\psibase[-2][i]$ along $\VOp^a$ satisfy a coupled system of equations. The Teukolsky equation
is notoriously difficult to treat as a wave equation, but a crucial breakthrough was the observation that\footnote{Technically, the variables in the system are a further rescaling given in definition \ref{def:raddegphi}.}
$\psibase[-2][0]=\psibase[-2]$, $\psibase[-2][1]$, and $\psibase[-2][2]$ can be treated as satisfying a $3\times3$ system of coupled equations for which energy and Morawetz estimates can be proved \cite{2017arXiv170807385M,2017arXiv171107944D}. 
The third equation in this $3\times3$ system is related to the Chandrasekhar transformation \cite{Chandrasekhar}. 
In particular, the right-hand side of this system has only first order derivatives $\LetaOp$ (and $\LetaOp$ is not applied to $\psibase[-2][2]$). 

As shown in section \ref{sec:foliation}, the level sets of $\tHor$ (in particular $\Stauini$) extend from the horizon $\Horizon^+$ to spacelike infinity $i_0$, whereas the level sets of $\timefunc$ go from the horizon $\Horizon^+$ to null infinity $\Scri^+$, as illustrated in figure \ref{fig:Estimates:SigmaInit}. A fairly simple argument shows that the energy on $\Stauini$ controls the energy on the level set $\Stauit=\{\timefunc=\timefunc_0\}$, also illustrated in the figure \ref{fig:Estimates:SigmaInit}. 

The $r^p$ argument \cite{DafermosRodnianski:rp} provides an important tool for obtaining decay. From \begin{enumerate*}[label=(\roman*)]\item an a priori estimate roughly of the form of the basic decay estimate for $\psibase[-2]$ in definition \ref{def:BEAM-intro}, \item using a multiplier of the $r^\alpha\VOp$ for $0< \alpha < 2$, and \item applying the  mean-value theorem, also known as the pigeonhole principle,
one can show energies decay at a rate of $t^{-\alpha}$. This $r^p$ argument is performed between level sets of $\timefunc$, as illustrated in figure \ref{fig:Estimates:Teukolskyrp}.  
\end{enumerate*} 

One of the crucial new features of this paper is that we obtain significantly stronger decay by introducing a new $5\times5$ system for $\psibase[-2][i]$ with $i\in\{0,\ldots,4\}$. This provides decay of energies at a rate of (almost) $\timefunc^{-2}$ for the variables with $i\in\{0,\ldots,4\}$. Following the improved $r^p$ argument of \cite{MR3859608}, from this decay and the $4\times4$ system for $\psibase[-2][i]$ with $i\in\{0,\ldots,3\}$, we obtain better decay for these variables. Even stronger decay is obtained by iterating again and using the known $3\times3$ system for $\psibase[-2][i]$ with $i\in\{0,\ldots,2\}$. This gives lemma \ref{lem:ImproEnerDecaypartialt}. To obtain yet stronger estimates for $\psibase[-2][i]$ with $i\in\{0,\ldots,1\}$, we are not aware of any way to apply energy estimates for a smaller subsystem, so, instead, we apply elliptic estimates, which completes the proof of theorem \ref{thm:ImproEstipsi01}. 

In our approach, the limit on the rate of decay for $\psibase[-2][i]$ arises from the size of the $5\times5$ system. To apply the $r^p$ argument, it is important that the angular part of the spin-weighted wave equation under consideration has a nonnegative spectrum for the angular operator. This consideration constrains the size of the derived system, the length of the hierarchy of weighted estimates, and consequently the fall-off rates provided by the estimates.

We do not need such strong decay in $\timefunc$ for $\psibase[+2]$ (nor is it expected to be true, in light of peeling arguments, see e.g.{} \cite{MR944085}), so we only use a $3\times3$ system for $\psibase[+2]$ in section \ref{sec:spin+2TeukolskyEstimates}. 

\item \label{point:strat5} \textbf{Finite-order expansion off null infinity for metric reconstruction:}
To estimate solutions of the transport equations arising in the metric reconstruction, we introduce a novel, finite-order, power-series expansion in $r^{-1}$ off $\Scri^+$ in definition \ref{def:asymptoticExpansion}. 

In this expansion, a finite number, $\ElNoArg$, (in fact zero to four) of leading-order terms are estimated pointwise with the remainder estimated in a weighted, spacetime Sobolev space. The Teukolsky variable $\psibase[-2]$ is a radiation field in the sense that it has been scaled so that with $\timefunc$ fixed, we expect there is generically a non-vanishing limit as $r\rightarrow\infty$. Since $\psibase[-2]$ does not vanish as $r\rightarrow\infty$ on the level sets of $\timefunc$, it is not possible to apply standard Sobolev and Hardy estimates. For this reason, and also because the quantities we wish to estimate are typically radiation fields, we have introduced this novel power-series expansion. 

These expansions fit neatly in our scheme for estimating solutions of the transport system arising in reconstructing the metric. The expansion is valid and useful where $r$ is large, which we take to be $r>\timefunc$, illustrated in figure \ref{fig:Estimates:NPnExt}. In this region, for a transport equation of the form \eqref{eq:transport-intro}, if the source term $\varrho$ has an expansion in $r^{-1}$ off $\Scri^+$, then we can compute the leading-order terms in the solution $\varphi$ by integrating the leading-order terms in $\varrho$ along $\Scri^+$. The remainder is computed as an integral along the ingoing null geodesics tangent to $\NPn$ in the region $r\geq\timefunc$. With the solutions of the transport system estimated at $r=\timefunc$, there is then a relatively quick argument in section \ref{sec:interiorEst} to integrate the solutions along the ingoing null geodesics tangent to $\NPn$ in the region $r\leq\timefunc$, illustrated in figure \ref{fig:Estimates:NPnInt}. 

As a technical point, we note that, for reasons that remain opaque, we have found that when choosing the variables so that relevant transport equations form a hierarchy in lemma \ref{lem:TransportSystem}, we are left with some variables that are not radiation fields, in that they generically vanish like $r^{-\EmNoArg}$ approaching $\Scri$. For this reason, definition \ref{def:asymptoticExpansion} for the expansion also includes a parameter $\EmNoArg\in\Naturals$ in the definition of an expansion.

\item \label{point:strat6} \textbf{Application of the TSI in metric reconstruction:} 
To complete the argument in the previous point, we use the Teukolsky-Starobinsky Identity (TSI) in a surprising and novel way. 
In brief, the TSI can be used to express $\psibase[-2]$ in terms of a fourth time derivative, which allows us to iterate the integration in the expansion off null infinity to reconstruct the other variables. 

With $\varrho_0$ denoting the leading-order term of the source $\varrho$ for the equation $\YOp\varphi=\varrho$, then the leading-order term of the solution $\varphi_0$ is given by $\varphi_0(\timefunc)=\int_{-\infty}^\timefunc \varrho_0(s)\di s$. If $\varrho_0$ converges rapidly to zero as $\timefunc\rightarrow\pm\infty$, then $\varphi_0$ will converge rapidly to a limit, but it need not converge to zero. The vanishing of the initial data at $i_0$ is sufficient to ensure that the leading-order term $\varphi_0$ converges to $0$ as $\timefunc\rightarrow-\infty$. To complete our analysis, it is necessary to show that $\varphi_0$ converges to $0$ as $\timefunc\rightarrow\infty$.

On $\Scri^+$, the TSI can be written schematically as being that the $\hedt^4\psibase[-2]$ is equal to the fourth derivative in $\timefunc$ of $\psibase[+2]$ plus less relevant terms, which we largely ignore in this introductory explanation, which are given explicitly in equation \eqref{eq:TSIpsihat}, and which are dealt with in detail in section \ref{sec:estimatesForTheMetric}. 
In particular, $\hedt^4\psibase[-2]$ is equal to a time derivative of derivatives of $\psibase[+2]$. Since $\hedt^4$ acting on $\psibase[-2]$ is a strongly elliptic operator and $\psibase[+2]$ satisfies decay estimates, $\psibase[-2]$ is itself a time derivative of a quantity that goes to zero as $\timefunc\rightarrow\pm\infty$. Thus, $\int_{-\infty}^\infty\psibase[-2](\timefunc')\di \timefunc'=0$, and similarly for the leading-order term in the expansion. Thus, when $\psibase[-2]$ appears as a source term $\varrho$ in a transport equation, the leading-order term of the solution $\varphi$ vanishes as $\timefunc\rightarrow\infty$. 
This process can be iterated, so that one finds $\psibase[-2]$ is equal to a fourth derivative in $\timefunc$ of variables that vanish at spacelike and timelike infinity. 
This means that $\psibase[-2]$ can be integrated up to four times to give zero, e.g. 
$\lim_{t\rightarrow\infty}\int_{-\infty}^{\timefunc}\ldots\int_{-\infty}^{\timefunc'''}\psibase[-2](\timefunc'''')\di\timefunc''''\ldots \di\timefunc'=0$. 
Since up to four integrations along $\Scri^+$ can be performed, it is possible to obtain estimates for all the quantities appearing in the transport hierarchy in lemma \ref{lem:TransportSystem}, which is used to reconstruct the metric. 
\end{enumerate}

\subsection{Overview of this paper.}
In section \ref{sec:geomprel}, we collect the geometric preliminaries needed in the paper, such as the GHP formalism and our choice of time functions. In section \ref{sec:linearizedEinstein}, we consider the linearized Einstein equation and the ORG, and, in particular, we derive from these the system of equations, solutions of which we estimate in the remained of the paper. Section \ref{sec:AnalyticPreliminaries} presents various analytic preliminaries, such as definitions of various norms and basic estimates for spin-weighted operators. Section \ref{sec:weightedenergy} presents convenient forms of lemmas for proving weighted energy estimates for transport and wave equations. 

Sections \ref{sec:spin-2TeukolskyEstimates} and \ref{sec:spin+2TeukolskyEstimates} present the decay estimates for the Teukolsky scalars $\psibase[-2]$ and $\psibase[+2]$ respectively. The estimates presented here assume a basic integrated energy decay estimate, but they do not require slow rotation, i.e. smallness of $|a|/M$. 

In section \ref{sec:estimatesForTheMetric}, we use the transport system derived in section \ref{sec:transporteq}, and the decay estimates proved in sections \ref{sec:spin-2TeukolskyEstimates} and \ref{sec:spin+2TeukolskyEstimates}, to prove estimates for linearized connection coefficients and metric components. The method used here involves the analysis of Taylor expansions at $\Scri^+$. 

Three appendices at the end provide some details that were delayed from earlier sections of the paper.

\subsection{Note on recent developments}\label{sec:recent}

Since the first submission of the present paper in early 2019,  
there have been significant developments on the black hole stability problem. We shall here briefly mention a few of these.

A major breakthrough on the black hole stability problem has been achieved by Klainerman, Szeftel, and collaborators, who in a series of works starting in 2020 \cite{KlainermanSzeftel:GCMSphereConstruction,KlainermanSzeftel:GCMSphereEffectiveResults,Klainerman:2021qzy,Shen2022GCMhypersurface,GioriKlainermanSzeftel:WaveEstimatesForKerr} provide a proof of the full nonlinear stability of the Kerr family of black holes for $|a|\ll M$. 

There have been further results for the case $|a|\ll M$. 
An independent proof of stability for linearized gravity on the Kerr background using a harmonic gauge has been presented in \cite{Hafner:2019kov} shortly after the first submission of this paper. 
Sharp decay estimates for the Teukolsky scalars $\vartheta \Psi_0$ and $\vartheta \Psi_4$,  also called the Price law in the physics literature, has  been proved \cite{MZ21Kerr}.

For the Teukolsky equation in the full subextreme range $|a|<M$,  recent results include a Morawetz estimate  
\cite{2020arXiv200707211S,2023arXiv230208916S} and close to optimal decay \cite{2023arXiv230206946M}. We expect that with minor modifications, these results will lead to the basic decay condition in definition \ref{def:BEAM-intro} and therefore, by the results of the present paper, to linear stability for the full subextreme range $|a|<M$.

The authors have recently shown \cite{Andersson:2021eqc} that a nonlinear version of the outgoing radiation gauge used in the present paper leads to a reduced system that can be put in first-order symmetric hyperbolic form, and which is therefore well-posed.

\section{Geometric preliminaries} \label{sec:geomprel}
We now set aside most of the content of the introduction and reintroduce the key variables and ideas in a more systematic and detailed way.

\subsection{Notation and conventions for spinors and the GHP formalism}\label{sec:notandconv}

The GHP formalism \cite{GHP} plays a central role in this paper. In particular, we use index and sign conventions following Penrose and Rindler \cite{MR944085,MR917488}, see also \cite{2015arXiv150402069A} for background. In \cite[Appendix A]{Andersson:2021eqc}, we have presented this in the language of gauge and principal-$G$ bundle theory. In this section, we first describe briefly the value of the GHP formalism. In the bulk of this section, we summarize both the theory of spinors and the theory of GHP scalars. We then conclude by emphasizing when it is possible to define an inner product. 

The GHP formalism allows us to connect our work with a broad literature, including classical results. The classical work includes the crucial work of Teukolsky et al. \cite{1972PhRvL..29.1114T,1973ApJ...185..649P,teukolsky:1973,1974ApJ...193..443T}, the additional Teukolsky-Starobinsky Identities \cite{1974JETP...38....1S}, and the original use of the outgoing radiation gauge \cite{Chrzanowski}. A more recent work using the GHP formalism gives conditions on when the outgoing radiation gauge can be applied \cite{2007CQGra..24.2367P}, which we describe further in section \ref{sec:ORG}. 

We also make use of recent developments in the GHP formalism. The first-order system of transport equations which is used here, cf. sections \ref{sec:connectcomp}, \ref{sec:transporteq} and appendix \ref{sec:fieldeq}, has been derived using the covariant formalism for calculus of variations with spinors introduced by B\"ackdahl and Valiente-Kroon in \cite{BaeVal15} and is closely related to the first-order form of the Einstein equations as a system of scalar equations  derived by Penrose and collaborators in \cite{GHP} and \cite{1962JMP.....3..566N}. The computer algebra tools for calculations in the 2-spinor and GHP formalisms developed by Aksteiner and B\"ackdahl \cite{Bae11a,SpinframesPackage}, and related packages, have played a central role in developing the approach and the system of equations used in this paper. 

We now summarize the theory of spinors \cite{MR944085}, so that we may recall the GHP formalism. If $(\generalManifold,\generalMetric)$ is an oriented, time-oriented, $1+3$-dimensional spacetime that admits four smooth vector fields that are linearly independent at each point, then $(\generalManifold,\generalMetric)$ is a spin manifold; the domain of outer communication of the Kerr black hole spacetime has these properties. 
The spin group in this case is $\SL(2,\Complexs)$, the double cover of $O^\dagger_+(1,3)$, the group of Lorentz transformations that preserve orientation and time-orientation. The spinor space at a point is $\Complexs^2$ with the vector representation of $\SL(2,\Complexs)$, and the complex conjugate representation is denoted $\overline{\Complexs}{}^2$. Sections of the spinor bundles associated to $\Complexs^2$ and $\overline{\Complexs}{}^2$ are denoted with capital Latin indices and primed capital Latin indices respectively. The term spinor is used for sections of these bundles as well as of their tensor products, e.g. $\varphi_{A\cdots C A' \cdots D'}$.

There is an isomorphism between spaces of tensor fields and certain spinor spaces. At a point $p\in\generalManifold$, there is an isomorphism $\Complexs\otimes T_p\generalManifold$ $\eqsim \Complexs\otimes \Reals^4$ $\eqsim \Complexs^2 \otimes \overline{\Complexs}{}^2$ so that the $O^\dagger_+(1,3)$ action on $(T_p\generalManifold,\generalMetric)$ is compatible with the $SL(2,\Complexs)$ action on $\Complexs^2$. This provides a correspondence at a point between vectors and spinors with one unprimed and one primed index. This isomorphism is expressed via the soldering form $g_a{}^{AA'}$, e.g. $\GenVec_{a} = g_a{}^{AA'} \GenVec_{AA'}$. On spin manifolds, this extends to an isomorphism of vector fields to spinor fields, which further extends to an isomorphism between tensor fields and spinor fields. It is convenient to write this correspondence in the abbreviated form $\GenVec_a = \GenVec_{AA'}$. The action of $\SL(2,\Complexs)$ on $\Complexs^2$ leaves an area element $\epsilon_{AB} = \epsilon_{[AB]}$ invariant. The  normalization $g_{ab} = \epsilon_{AB}\bar\epsilon_{A'B'}$ defines the spin metric $\epsilon_{AB}$ up to a phase. This has an inverse $\epsilon^{AB}$. These are used to raise and lower spinor indices via $\xi_B=\xi^A\epsilon_{AB}$.  
Using the tensor-spinor correspondence mentioned above, it is possible to express any tensor as a sum of symmetric spinors multiplied by $\epsilon_{AB}$ factors.
As an example and of particular importance in this paper, for the Weyl tensor, we have 
\begin{align} 
C_{abcd}  = \Psi_{ABCD} \bar\epsilon_{A'B'} \bar\epsilon_{C'D'} + \epsilon_{AB} \epsilon_{CD} \bar \Psi_{A'B'C'D'},
\end{align} 
where $\Psi_{ABCD}$ is the symmetric Weyl spinor. 

Central to the GHP formalism are bases aligned with a pair of null directions and the way such bases can be rescaled. Given a pair of null directions, one can construct an aligned real tetrad by first choosing a pair of future-directed, null vectors, $\NPl$ and $\NPn$, that are parallel to these null directions and normalized so that $\NPn_a\NPl^a=1$ and then choosing an orthonormal basis $e_1,e_2$ for the plane orthogonal to $\NPn$ and $\NPl$ such that $(\NPl,\NPn,e_1,e_2)$  is an oriented basis. One can construct an aligned complex null tetrad by choosing an aligned real tetrad $(\NPl,\NPn,e_1,e_2)$ and replacing it with $(\NPl,\NPn,\NPm,\NPmbar)$ where $\NPm=(1/\sqrt{2})(e_1+ie_2)$, or equivalently choosing $\NPl$ and $\NPn$ as for an aligned real tetrad and then $\NPm$ 
such that $(\NPl,\NPn,\Re\NPm,\Im\NPm)$ is an oriented (real) basis. This implies
\begin{align} \label{eq:metric=tetrad} 
g_{ab} = 2 (\NPl_{(a} \NPn_{b)} - \NPm_{(a} \NPmbar_{b)} ) .
\end{align}
In the spinor space, one can construct a dyad by taking a basis $o^A,\iota^A$ for the spinor space such that 
\begin{equation} \label{eq:dyad-normalization} 
o_A \iota^A = 1 .
\end{equation}
Such a dyad defines a complex tetrad by
\begin{equation}\label{eq:tetrad-dyad} 
\NPl^a = o^A o^{A'}, \quad \NPn^a = \iota^A \iota^{A'}, \quad \NPm^a = o^A \iota^{A'}, \quad \NPmbar^a = \iota^A o^{A'} .
\end{equation}
The dyad is said to be aligned with a pair of null directions if the complex tetrad \eqref{eq:tetrad-dyad} is aligned with the null directions. 
The normalization \eqref{eq:dyad-normalization} remains invariant under rescalings $o_A \to \mu o_A$, $\iota_A \to \mu^{-1} \iota_A$ where $\mu\neq 0$ is a complex scalar field. The corresponding aligned complex tetrad transforms to a new aligned complex tetrad $(|\mu|^2\NPl,|\mu|^{-2}\NPn,(\mu/\bar\mu)\NPm,(\bar\mu/\mu)\NPmbar)$, and the corresponding aligned real tetrad transforms so that $\NPl$ and $\NPn$ are scaled by $|\mu|^2$ and $|\mu|^{-2}$ respectively and $e_1$ and $e_2$ are rotated through an angle equal to twice the argument of $\mu$. At each point, the set of aligned dyads and the set of aligned complex tetrads (and hence also the set of aligned real tetrads) is isomorphic to $\Complexs\backslash\{0\}$; the set of aligned dyads forms a double cover of the aligned complex tetrads. Given a globally defined set of null directions, the set of all aligned dyads or all aligned complex tetrads forms a $\Complexs\backslash\{0\}$ bundle. 

It is now possible to define GHP scalars. A GHP scalar is a map from the bundle  of aligned dyads to $\Complexs$. Commonly, this is constructed by contracting legs of the dyad against a spinor field, for example $\Psi_{ABCD} o^A o^B \iota^C \iota^D$, but in certain cases GHP scalars are constructed from differentiating a local dyad. 
The following definition singles out a particularly important class of GHP scalars. 
\begin{definition}\label{def:propweight}
A GHP scalar $\varphi$ is properly weighted if there is an ordered pair of integers $(p,q)$ such that $\varphi$ transforms as $\varphi \to \mu^p \bar\mu^q \varphi$ under a transformation of the dyad $(o^A,\iota^A)\mapsto(\mu o^A,\mu^{-1}\iota^A)$; in this case, it is said to have type $\{p,q\}$.  For properly weighted GHP scalars of type $\{p,q\}$, the boost weight is $b=(p+q)/2$ and the spin weight is $s=(p-q)/2$. 
\end{definition}
In the language of bundles, for each $\{p,q\}$, the set of properly weighted GHP scalars with type $\{p,q\}$ form an associated line bundle for the principal-$\Complexs\backslash\{0\}$ bundle of aligned dyads.
GHP scalars with integer boost and spin weight can be treated as maps from the bundle of aligned complex tetrads rather than as maps from the bundle of aligned dyads. The notions of properly weighted scalar, type, as well as boost- and spin-weight extend to tensor and spinor fields. For example, $\NPm^a$ has type $\{1,-1\}$, boost-weight $0$, and spin-weight $1$. A field of GHP type $\{0,0\}$ is well-defined, independent of rescalings of the tetrad. Examples are the metric $g_{ab}$ and the middle Weyl scalar $\Psi_2 = \Psi_{ABCD} o^A o^B \iota^C \iota^D$. The GHP type and the boost- and spin-weights are additive under multiplication. 

A further index convention is used to compactify the GHP formalism. For a spinor $\varphi_{A_1\ldots A_k A_1'\ldots A'_l}$ that is symmetric in the primed indices and symmetric in the unprimed indices, scalar components  $\varphi_{ii'}$ are defined by contracting $i$ times with $\iota^A$, $i'$ times with $\iota^{A'}$, and contracting the remaining indices with $o^A$ or $o^{A'}$. The numbers $i$ or $i'$ are omitted if the spinor is of valence $(0,l)$ or $(k,0)$ respectively. In particular, the Weyl spinor $\Psi_{ABCD}$ corresponds to the five complex Weyl scalars $\Psi_i, i = 0, \dots, 4$. 

Calculations using the GHP formalism are simplified by using the prime and complex conjugation operations.\footnote{In addition, there is the Sachs $*$ operation, see \cite{GHP}.} Complex conjugation, $\varphi \to \bar\varphi$ takes fields of type $\{p,q\}$ to type $\{q,p\}$, i.e. it changes the sign of the spin-weight, and preserves the boost-weight. The prime  operation, $\varphi \to \varphi'$, interchanges $\NPl^a \leftrightarrow \NPn^a$, $\NPm^a \leftrightarrow \NPmbar^a$, and takes fields of type $\{p,q\}$ to fields of type $\{-p,-q\}$. 
The prime operation and complex conjugation commute and are symmetries in the sense that an equation valid in the GHP formalism remains valid after applying the prime operation or complex conjugation. 

As sections of line bundles, GHP scalars cannot be differentiated with partial derivatives and must be differentiated with a relevant connection. 
Properly weighted scalars are sections of complex line bundles; more generally, properly weighted tensor and spinor fields are sections of complex vector bundles. The lift of the Levi-Civita connection $\nabla_{a}$ to these bundles gives a covariant derivative denoted $\Theta_a$. Projecting on the null tetrad $\NPl^a, \NPn^a, \NPm^a, \NPmbar^a$ gives the GHP operators \cite{GHP},
\footnote{Following \cite{GHP,MR917488,MR944085}, we represent these by the Icelandic/ old English letters thorn $\tho$ and edth $\edt$.}
\index{T3tho@$\tho$, $\thop$}
\index{E3eth@$\edt, \edtp$}
\begin{align}
\label{eq:thoThopEdtEdtp}
\tho = \NPl^a \Theta_a, \quad \thop = \NPn^a\Theta_a, \quad \edt = \NPm^a \Theta_a , \quad 
\edtp = \NPmbar^a \Theta_a.
\end{align}
See \cite{1990CQGra...7.1681H} for discussion of the geometry of properly weighted scalars and the GHP covariant derivative. 
The GHP operators are properly weighted, in the sense that they take properly weighted fields to properly weighted fields, for example if $\varphi$ has type $\{p,q\}$, then $\tho \varphi$ has type $\{p+1, q+1\}$. This can be seen from the fact that $\NPl^a = o^A \bar{o}^{A'}$ has type $\{1,1\}$. 

The connection coefficients for the Levi-Civita connection can be expressed in the GHP formalism. There are twelve connection coefficients in a null frame, up to complex conjugation. Of these, eight are properly weighted, and are given by
\begin{align}
\kappa ={}&  \NPm^b \NPl^a \nabla_a \NPl_b , \quad
\sigma =  \NPm^b \NPm^a \nabla_a \NPl_b, \quad
\rho   =  \NPm^b \NPmbar^a \nabla_a \NPl_b , \quad
\tau   = \NPm^b \NPn^a \nabla_a \NPl_b ,
\label{eq:spincoeff-def}
\end{align}
together with their primes $\kappa', \sigma', \rho', \tau'$.
These are the GHP spin coefficients. The remaining four connection coefficients, given by 
\begin{align} 
\epsilon ={}& \half ( \NPn^a \NPl^b \nabla_b \NPl_a + \NPm^a \NPl^b \nabla_b \NPmbar_a) \quad  
\beta = \half ( \NPn^a \NPm^b \nabla_b  \NPl_a + \NPm^a \NPm^b \nabla_b  \NPmbar_a) 
\label{eq:eps-beta-def}
\end{align} 
and their primes, enter in the connection 1-form for the connection $\Theta_a$. Furthermore, the GHP connection $\Theta$ acting on a GHP scalar with boost weight $b$ and spin weight $s$ can be expressed with respect to a particular choice of local complex tetrad as
\begin{align} \label{eq:Theta-def}
\Theta_a \varphi = \nabla_a \varphi - b \NPn^b\nabla_a \NPl_b \varphi + s \NPmbar^b \nabla_a \NPm_b \varphi .
\end{align} 
This also extends to properly weighted tensor and spinor fields. 

GHP scalars with boost weight zero, for which there is an inner product and hence a norm, are of particular importance in our analysis.
To avoid cumbersome terminology, we introduce the following definition. 
\begin{definition}
\label{def:spinWeightedScalar}
\index{S1s@$s$}
A spin-weighted scalar is a properly weighted GHP scalar with boost weight zero. 
For spin-weighted scalars $\varphi$ and $\varrho$ with the same spin weight, define the inner product to be
\begin{align}
\langle \varphi , \varrho \rangle ={}& \varphi \bar \varrho .
\label{eq:ipdef}
\end{align}
\end{definition}
When there is no room for confusion, we shall use $s$ to denote spin weight, otherwise we shall use $\Espin{\varphi}$ to denote the spin weight of the spin-weighted scalar $\varphi$.

The inner product has several important properties. 
The inner product has boost and spin weight zero, so is simply a complex-valued function on the manifold rather than an element of a more complicated, complex line bundle. 
The inner product defines a norm, which appears in definition \ref{def:pointwisenormNoDerivatives}, in section \ref{sec:norms}, where we introduce all the norms that we use in this paper. 
The GHP covariant derivative $\Theta_a$ is real, in the sense that 
\begin{align} 
\Theta_a \bar\varphi = \overline{\Theta_a \varphi}
\end{align}     
and hence it is also metric, with respect to the inner product given by \eqref{eq:ipdef}, in the sense that
\begin{align} 
\nabla_a \langle \varphi, \varrho \rangle = \langle \Theta_a \varphi, \varrho \rangle + \langle \varphi, \Theta_a \varrho \rangle .
\end{align} 
There is an isomorphism between spin-weighted scalars and certain geometric quantities that appear in the null decomposition, which was used in its original form in the proof of stability of Minkowski space \cite{MR1316662} and which was recently refined \cite{Klainerman:2021qzy}. For $s=0$, spin-weighted scalars are simply complex-valued functions on the manifold. Within this paragraph, let $\mathcal{H}$ denote, at each point in the manifold, the plane orthogonal to the null directions, which is called the horizontal plane in \cite{Klainerman:2021qzy}. For $s=1$, the map $\xi_a \mapsto \xi_a\NPm^a$ defines an isomorphism from real vector fields taking values in $\mathcal{H}$ to spin-weighted scalars with spin weight $s=1$. This map is invariant under rescalings $(o^A,\iota^A)\mapsto(\mu o^A,\mu^{-1}\iota^A)$. For $s=2$, there is similarly an isomorphism from symmetric, traceless $2$-tensors on $\mathcal{H}$ to spin-weighted scalars with $s=2$. 

A further refinement to the GHP formalism, which plays a central role in our analysis, is a process we refer to as deboosting. In the Kerr spacetime, introduced in the following section, there is a nowhere vanishing GHP scalar $\uplambda$ with boost weight $1$ and spin weight zero, which is defined in definition \ref{def:uplambda}. For another GHP scalar $\varphi$ with boost weight $b$ and spin weight $s$, the quantity $\uplambda^{-b}\varphi$ has boost weight zero and spin weight $s$; in particular, as a spin weight zero GHP scalar, the rescaled quantity $\uplambda^{-b}\varphi$ is a quantity for which there is a well defined norm. 

\begin{remark}
\label{rmk:NPLocallyTrivialisesGHP}
In closing, we comment on the NP formalism, which can look much like the GHP formalism and which we used in the introduction as an oversimplification. 
In the NP formalism \cite{1962JMP.....3..566N}, one begins by introducing a specific, local complex-tetrad $(\NPl^a,\NPn^a,\NPm^a,\NPmbar^a)$ (or a dyad for spinors) and then calculates complex-valued functions using formulas such as \eqref{eq:spincoeff-def} for GHP scalars as if they were formulas with respect to the specific choice of tetrad. Thus, the NP formalism uses a particular choice of local basis and provides complex-valued functions, while the GHP formalism uses the bundle of aligned tetrads and provides sections of complex line bundles. In this way, the NP formalism can be viewed as a local trivialization of the GHP formalism. 

The theorems \ref{thm:mainintro} and \ref{thm:BEAMintro} are literally correct as stated in this oversimplified, NP form; however, properly within the GHP formalism, a better statement is that the GHP scalars $G_{i0'}$ defined in equation \eqref{eq:G-compos} satisfy the decay estimates \eqref{eq:mainGHP} and \eqref{eq:metricDecayExt}-\eqref{eq:metricDecayInt} when deboosted to $\uplambda^{i-2}G_{i0'}$. 
Since we construct $\uplambda$ so that it takes the value $1$ in the local trivialization given by the Znajek tetrad in equation \eqref{eq:Znajek-intro}, the NP estimates of theorems \ref{thm:mainintro}-\ref{thm:BEAMintro} are exactly equivalent, for $\theta\not\in\{0,\pi\}$ where they are defined, to the GHP decay estimates \eqref{eq:metricDecayExt}-\eqref{eq:metricDecayInt}.
\end{remark}

\subsection{Geometry of Kerr} 
\label{sec:Kerrgeom}

We first recall a small number of the key geometric features introduced in the introduction, section \ref{sec:intro}. In the opening paragraph of section \ref{sec:intro}, the domain of outer communication of the Kerr black hole was introduced as a Lorentzian metric $(\Mcal,\met)$, where the metric $\met$ is given in equation \eqref{eq:gab-intro} with respect to the ingoing Eddington-Finkelstein coordinates $(v,r,\theta,\phi)\in\Reals\times(r_+,\infty)\times(0,\pi)\times(0,2\pi)$
where
\begin{align}
\label{eq:r+def}
r_+ ={}& M + \sqrt{M^2-a^2} 
\end{align}
 The metric extends smoothly to $\Reals\times(r_+,\infty)\times\Sphere$. The vectors 
\index{X3xi@$\xi^a$}
\index{Z3zeta@$\zeta^a$}
\index{E3eta@$\eta^a$}
\begin{align}
 \xi^a={}&(\partial_v)^a, & 
\zeta^a={}&a^2 (\partial_v)^a
 + a (\partial_\phi)^a , &
\eta^a={}&a^{-1}\zeta^a - a \xi^a  = (\partial_\phi)^a 
\label{eq:KillingVecDefs}
\end{align}
are Killing vectors. 

There are a number of further properties that follow quickly from the form of the metric in equation \eqref{eq:gab-intro}. The vector field $-\partial_r$ is null throughout $\Mcal$, so it defines a time orientation. The volume element of $g_{ab}$ is given by 
\begin{align}
\sin\theta \Sigma dv dr d\theta d\phi ,
\end{align}
which differs from the reference volume form in equation \eqref{eq:referenceVolumeForms-intro} by a factor of $\Sigma$. 

Of particular importance in our analysis is the existence of a unique pair of principal null directions. There exists a dyad $(o_A,\iota_A)$ such that the Weyl curvature spinor takes the form 
\index{P4PsiSpinor@$\Psi_{ABCD}$}
\index{P4PsiMid@$\Psi_{2}$}
\begin{align} 
\Psi_{ABCD}={}&6 \Psi_{2} o_{(A}o_{B}\iota_{C}\iota_{D)},
\label{eq:typeII-II} 
\end{align}
where $\Psi_{2}$ is a complex-valued function on $\Mcal$. Such a dyad is called a principal dyad, and the principal null directions are the pair of directions parallel to $\NPl_a=o_A\bar{o}_{A'}$ and $\NPn_a=\iota_A\bar\iota_{A'}$. The principal null directions are uniquely determined by the condition \eqref{eq:typeII-II}. Because there exists a pair of spinors that each appear twice in the factorisation of the Weyl curvature, the Kerr spacetime is said to be of type D, also called type $\{2,2\}$, in the Petrov classification \cite{MR944085}. Furthermore, on the Kerr spacetime, there is a symmetric spinor $\kappa_{AB}$ found in \cite{walker:penrose:1970CMaPh..18..265W}, satisfying
\index{K3kappaKillingSpinor@$\kappa_{AB}$} 
\begin{align} 
\nabla_{(A}{}^{A'} \kappa_{BC)} = 0 .
\label{eq:KillingSpinorCondition}
\end{align} 
Such a spinor is called a Killing spinor. In a dyad aligned with the principal null directions, the Killing spinor takes the simple form 
\index{K3kappaComponent@$\kappa_1$}
\begin{align}
\label{eq:TypeDKS}
\kappa_{AB}={}&-2 \kappa_1 o_{(A}\iota_{B)} ,
\end{align}
where $\kappa_1$ is also a complex-valued function on $\Mcal$. In the Kerr spacetime, $\kappa_{AB}$ can be normalized so that the stationary Killing field with unit norm at infinity is given by 
\index{X3xi@$\xi^a$}
\begin{align} 
\label{eq:xifromkappa} 
\xi_{AA'} ={}& \nabla^B{}_{A'} \kappa_{AB} .
\end{align}
The spinor $\kappa_{AB}$ is uniquely determined by the Killing condition \eqref{eq:KillingSpinorCondition}, that $\xi_{AA'}$ is real, and that the norm of $\xi_{AA'}$ goes to $1$ as $r\rightarrow\infty$. In ingoing Eddington-Finkelstein coordinates we get
\index{P4PsiMid@$\Psi_{2}$}
\index{K3kappaComponent@$\kappa_1$}
\begin{align} \label{eq:kappa1}
\kappa_{1}{} ={}& -\tfrac{1}{3} (r -i a \cos\theta), 
&
\Psi_{2}={}&- M (r - i a \cos\theta)^{-3} .
\end{align} 
Note that $\kappa_1$ and $\Psi_2$ can be expressed covariantly via the relations $\kappa_{AB} \kappa^{AB} = -2 \kappa_{1}{}^2$ and $\Psi_{ABCD} \Psi^{ABCD}=6 \Psi_{2}^2$. Thus, they are independent of the choice of dyad. The quantities $\kappa_1$ and $\Psi_2$ are the only non-vanishing components, in a principal dyad, of the Killing spinor and curvature in the Kerr spacetime. 

\begin{remark}
\label{rmk:NoGlobalTetrad}
In the case $a=0$, that the principal null directions are orthogonal to the level sets of constant $v,r$, which are round spheres. Thus, if $(\NPl,\NPn,e_1,e_2)$ were a global tetrad aligned with the principal null directions, then $(e_1,e_2)$ would be a global basis for the tangent space for each such sphere, and, in particular, on each such sphere, $e_1$ would be a nowhere-vanishing vector field tangent to that sphere. However, there do not exist any nowhere vanishing vector fields tangent to spheres. Thus, in the case $a=0$, it is particularly clear that there can be no global tetrad aligned with the principal null directions. While this means the NP formalism in remark \ref{rmk:NPLocallyTrivialisesGHP} cannot be applied globally in the Kerr exterior with the principal null directions, the bundle of aligned tetrads remains well defined throughout the Kerr exterior, so one can apply the GHP formalism from section \ref{sec:notandconv}. 
\end{remark}

Many of the key quantities in the Kerr metric, which have so far appeared in terms of coordinate choices, can in fact be written geometrically in terms of the Killing spinor $\kappa_{AB}$. This way of presenting the quantities can be viewed as a more natural approach when one recalls that \cite[Theorem 6]{backdahl:etal:2010:MR2753388} provides a characterisation of the Kerr spacetime in terms of the existence of a Killing spinor and certain auxiliary conditions. The Eddington-Finkelstein (or Boyer-Lindquist) coordinates $r,\theta$ can be defined covariantly via 
\begin{subequations} 
\begin{align}
r={}&- \tfrac{3}{2} (\kappa_{1}{} + \overline{\kappa}_{1'}{}), \\
a \cos\theta ={}& - \tfrac{3}{2}i (\kappa_{1}{} -  \bar{\kappa}_{1'}{}) .
\end{align}
\end{subequations}
The geometric definition of the radial coordinate $r$ remains valid in the non-rotating case, $a=0$.
Similarly, the two function $\Delta$ and $\Sigma$ appearing in the metric \eqref{eq:gab-intro} can be expressed in a principal null tetrad as
\index{D4Delta@$\Delta$}
\index{S4Sigma@$\Sigma$}
\begin{align}
\Delta={}&-162 \kappa_{1}{}^3 \overline{\kappa}_{1'}{} \rho \rho',&
\Sigma={}&9 \kappa_{1}{} \overline{\kappa}_{1'}{}.
\end{align}
The Killing vector $\xi=\partial_v$ was already given in terms of the Killing spinor in equation \eqref{eq:xifromkappa}, and \cite{2015arXiv150402069A} provides expressions for the other Killing vectors $\zeta$ and $\eta$ in equation \eqref{eq:KillingVecDefs}. 
In the Schwarzschild case $a=0$, there is no geometrically preferred axis, so the $\theta$ and $\phi$ coordinates and the vector $\eta$ cannot be constructed from the Killing spinor. 
The remaining two coordinates $v,\phi$ can be chosen to correspond to the two Killing fields of the spacetime. In general, we try to work with the geometrically defined quantity $\kappa_1$, rather than the coordinate $r$. 

The connection coefficients can be computed with respect to a choice of local tetrad. Typically within this paper, we make use of the covariant GHP formalism and properly weighted scalars, and hence our calculations are independent of the specific coordinate system and principal tetrad used. However, it is sometimes convenient to make use of the ingoing Eddington-Finkelstein coordinate system and the explicit form of the Znajek tetrad, which was given in the introduction in equation \eqref{eq:Znajek-intro} and is aligned with the principal null directions. With respect to the Znajek tetrad, the connection coefficients are \cite{1977MNRAS.179..457Z}
\index{B3beta@${\beta}, {\beta}'$}\index{E3epsilon@${\epsilon}, {\epsilon}'$}\index{K3kappa@${\kappa}, {\kappa}'$}\index{R3rho@${\rho}, {\rho}'$}\index{S3sigma@${\sigma}, {\sigma}'$}\index{T3tau@${\tau}, {\tau}'$}
\begin{subequations}\label{eq:Znajek-spincoeff} 
\begin{align}
\kappa={}&0,&
\kappa'={}&0,&
\sigma={}&0,&
\sigma'={}&0,\\
\rho={}&\frac{\Delta}{3 \sqrt{2} \kappa_{1}{} \Sigma},&
\rho'={}&- \frac{1}{3 \sqrt{2} \kappa_{1}{}},&
\tau={}&- \frac{i a \sin\theta}{9 \sqrt{2} \kappa_{1}{}^2},&
\tau'={}&- \frac{i a \sin\theta}{\sqrt{2} \Sigma},\\
\epsilon '={}&0,&
\beta '={}&- \frac{\cot\theta}{6 \sqrt{2} \overline{\kappa}_{1'}{}},&
\beta={}&- \frac{i \csc\theta (2 a - 3i \cos\theta \overline{\kappa}_{1'}{})}{18 \sqrt{2} \kappa_{1}{}^2},\hspace{-10ex}\\
\epsilon={}&\frac{2 \Delta - 6 M \kappa_{1}{} - 9 \kappa_{1}{}^2 -  \Sigma}{6 \sqrt{2} \kappa_{1}{} \Sigma}.\hspace{-20ex}
\end{align}
\end{subequations}
Since $\kappa,\kappa',\sigma,\sigma'$ are properly weighted and vanish with respect to the Znajek tetrad, they are zero with respect to any tetrad aligned to the principal null directions and hence vanish as sections of the relevant complex line bundle. By smoothness, this vanishing extends to the axis $\theta\in\{0,\pi\}$, where the Znajek tetrad is not defined. Similarly, since $\rho,\rho',\tau,\tau'$ are properly weighted and non-vanishing with respect to the Znajek tetrad, they are non-vanishing in any tetrad aligned to the principal null directions and as sections of the line bundle. Since $\epsilon'=0$, the vector field $\NPn$ in the Znajek tetrad is normalized so that $\NPn$ is auto-parallel
\begin{align} 
n^b \nabla_b n^a ={}& 0, 
\end{align} 
i.e. it generates affinely parametrized geodesics. 

To relate the results in this paper to others in the literature, it is convenient to introduce other coordinate systems that occur commonly in the literature \cite{MTW}. These appear in the following definition. 

\begin{definition} 

\begin{enumerate}
\item 
The tortoise coordinate $\rstar = \rstar(r)$ is defined by  
\begin{align} 
\label{eq:r*def}
\frac{d\rstar}{dr} = \frac{r^2+a^2}{\Delta}, \quad \rstar(3M) = 0 .
\end{align}  
Further, the angular correction $r^\sharp = r^\sharp(r)$ is defined by
\begin{align} 
\frac{dr^\sharp}{dr} = \frac{a}{\Delta}, \quad r^\sharp(3M) = 0 .
\end{align}  
\item 
The Boyer-Lindquist time $\timefunc_{BL}$  and azimuthal angle $\phi_{BL}$
\index{T1tBoyerLindquist@$\timefunc_{BL}$}
\begin{subequations}
\begin{align}\label{eq:tBL-def}
\timefunc_{BL} ={}& v - \rstar ,\\
\phi_{BL} ={}& \phi - r^\sharp. 
\end{align}
\end{subequations}
The Boyer-Lindquist coordinate system is given by $(t_{BL}, r,\theta, \phi_{BL})$. 
\item 
The retarded time $u$ and retarded angle $\phi^\sharp$ are
\begin{subequations}
\begin{align} \label{eq:u-def}
u ={}& v - 2\rstar ,\\
\phi^\sharp ={}& \phi - 2r^\sharp .
\end{align}
\end{subequations}
The outgoing Kerr, or Eddington-Finkelstein coordinates are $(u,r,\theta,\phi^\sharp)$. 
\end{enumerate}
\end{definition} 

\begin{remark}We shall sometimes refer to $v$ as the advanced time. However, neither $u$ nor $v$ is a time function, in particular their level sets are non-spacelike, in the non-static Kerr case ($a\neq 0$). 
\end{remark} 

To understand regions where $r\rightarrow\infty$, it is convenient to work with an additional coordinate system, which is given in the following definition. 

\begin{definition}
\label{def:compactifiedR}
\index{R2R@$R$}
Define the compactified radial coordinate to be 
\begin{align}
\label{eq:R=1/r}
R={}&1/r.
\end{align} 
The compactified outgoing coordinates are defined to be $(u,R,\theta,\phi^\sharp)$. 
\end{definition}

\index{I2IScriPlus@$\Scri^+$}
\index{I1i0@$i_0$}
\index{I1iplus@$i^+$}
In closing, we note some properties of the boundary, including the boundary at infinity, for the manifold $\Mcal$, particularly with the conformally rescaled metric $r^{-2}g_{ab}$ \cite{hawking_ellis_1973}. The relevant features are illustrated in figure \ref{fig:KerrDOC}. Here, for simplicity, we will work with spherical coordinates, but the standard singularities at the axes $\theta\in\{0,\pi\}$ can be removed in the standard way. As already noted, in the ingoing Eddington-Finkelstein coordinates $(v,r,\theta,\phi)$, the manifold and metric extends smoothly to the future to $\Horizon^+=\{r=r_+\}$. Working in the compactified radial coordinates $(v,R,\theta,\phi)$, the manifold and conformal metric $r^{-2}g_{ab}$ extends smoothly to the past to $\Scri^-=\{R=0\}$. Similarly, in the outgoing Eddington-Finkelstein coordinates $(u,r,\theta,\phi^\sharp)$, the manifold and metric extend smoothly to the past to $\Horizon^-=\{r=r_+\}$, and, with compactified radial parametrization $(u,R,\theta,\phi^\sharp)$, the manifold and conformal metric $r^{-2}g_{ab}$ extend smoothly to $\Scri^+=\{R=0\}$. 
In the maximally extended Kerr spacetime, the past limits of $\Horizon^+$ coincide with the future limits of $\Horizon^-$ at a sphere $\mathcal{B}$, called the bifurcation sphere. Furthermore, the Kerr exterior can be extended (as a topological space) by the addition of three additional points. There is a point $i^+$, called (future) timelike infinity, which is the future end point of $\Scri^+$, $\Horizon^+$, and all future inextendible timelike geodesics. Similarly, there is a point $i^-$,  called past timelike infinity, which is the past end point of $\Scri^+$, $\Horizon^+$, and all future inextendible timelike geodesics. There is also a point $i_0$, called spacelike infinity, which is the limit of all spacelike geodesics that reach infinity.

\begin{figure}[ht] 
\centering
\begin{tikzpicture}[scale=.75] 

\coordinate (i+) at (4,4);  		
\coordinate (i0) at (8,0);			
\coordinate (B) at  (0,0);		
\coordinate (i-) at (4,-4);  		

\draw[name path=futurehorizon,thick,blue] (B) -- (i+) node[near end,anchor=east,xshift=-.3cm,rotate=45]{$\Horizon^+$};

\draw[name path=pasthorizon,thick,blue] (B) -- (i-) node[midway,rotate=-45,yshift=-.3cm]{$\Horizon^- =  \{v=-\infty\}$};

\draw[name path=scriplus,thick, dash dot, red](i+)--(i0) node[midway,above,xshift=.1cm,rotate=-45]{$\Scri^+ = \{v=\infty\}$};

\draw[name path=scriminus,thick, dash dot, red](i-)--(i0) node[midway,anchor=west,xshift=.5cm,rotate=45,xshift=-.3cm]{$\Scri^-$};

\filldraw[fill=white] (i0) circle[thick,radius=.7mm] node[anchor=west,xshift=.1cm]{$i_0$} ;

\filldraw[fill=white] (i+) circle[thick,radius=.7mm] node[anchor=south,yshift=.1cm]{$i_+$};

\filldraw[fill=white] (i-) circle[thick,radius=.7mm] node[anchor=west,xshift=-.1cm,yshift=-.3cm]{$i_-$};

\filldraw[fill=black] (B) circle[thick,radius=.7mm] node[anchor=east]{$\mathcal{B}$} ;

\draw[name path=timeslice,thick,dashed,black] (B) to[bend right=15]  node[midway,anchor=north]{{\tiny $\timefunc_{BL} = \text{constant}$}} (i0)  ;

\draw[name path=timeslice,thick,dashed,black] (B) to[bend right=-15]  (i0)  ;

\draw[thin,red](3.75,3.75)--(7.75,-.25) node[midway,below,rotate=-45,yshift=.05cm]{{\tiny $v=\text{constant}$}};

\draw[thin,red](3.5,3.5)--(7.5,-.5) ;

\draw[thin,red](3.75,-3.75)--(7.75,.25) node[midway,above,rotate=45,yshift=-.05cm]{{\tiny $u=\text{constant}$}};

\draw[thin,red](3.5,-3.5)--(7.5,.5) ;
\end{tikzpicture} 
\caption{A conformal diagram for the domain of outer communication of the Kerr back hole, with level sets of $t_{BL}$ and (at the poles $\theta\in\{0,\pi\}$) of $v$ and $u$ indicated. } 
\label{fig:KerrDOC}
\end{figure}

\subsection{Operators on spin-weighted scalars} 
\label{sec:operators}

In this section, we introduce operators that take spin-weighted scalars to spin-weighted scalars. Recall that spin-weighted scalars were introduced in definition \ref{def:spinWeightedScalar} and are sections of a complex line bundle of properly weighted GHP scalars with boost weight zero. As sections of a bundle, they must be differentiated with a connection rather than simply with partial derivatives. Within this section, we introduce the notion a spin-weighted operator, a deboosting factor which is used to convert general properly weighted GHP scalars to spin-weighted scalars, and many important examples of spin-weighted operators. The crucial property on a spin-weighted operator, which follows from the formal definition below, is that it takes spin-weighted scalars with one spin weight $s$ to spin-weighted scalars with possibly a different spin weight. The lemmas in this section follow by direct computation. 

\begin{definition} 
A properly weighted operator of boost-weight zero is called a spin-weighted operator.
\end{definition} 

Crucial in our analysis is the process of deboosting, which we now introduce. Many of the quantities appearing in the GHP formalism are properly weighted GHP scalars but fail to be spin-weighted scalars because they have a non-vanishing boost weight. This means that, for them, the formula for the inner product in definition \ref{def:spinWeightedScalar} fails to define a norm that is independent of the choice of local tetrad. For this reason, we wish to convert general properly weighted GHP scalars to spin-weighted scalars. To do so, we multiply by an appropriate power of a spin-weighted quantity, so that the product has boost weight zero and hence is a spin-weighted scalar. In the following definition, we introduce the deboosting factor, which we use to deboost our variables and operators. 

\begin{definition} \label{def:uplambda}
\index{L3lambda@$\uplambda$}
Define the deboosting factor to be
\begin{align} \label{eq:uplambda-def} 
\uplambda ={}& (-3\sqrt{2}\kappa_1\rho')^{-1} .
\end{align} 
\end{definition} 

\begin{remark} The spin coefficient $\rho'$ is properly weighted with boost-weight $-1$ and spin-weight zero. The scalar $\uplambda$ defined in \eqref{eq:uplambda-def} has boost-weight $1$, spin-weight zero. By multiplying with powers of $\uplambda$ we may deboost operators and scalars, so that they have boost weight zero. 

Furthermore, $\uplambda$ takes the value $1$ in the Znajek tetrad \eqref{eq:Znajek-intro} (or in any tetrad in which $\NPl$ and $\NPn$ coincide with the null vectors in the Znajek tetrad). This simplifies some calculations. 
\end{remark} 

We now introduce spin-weighted operators corresponding to the derivatives along the null legs of a tetrad aligned with the principal null directions. Recall equation \eqref{eq:thoThopEdtEdtp} gave $\tho$ and $\thop$ as the GHP derivative operators along $\NPl$ and $\NPn$. There is no need to deboost $\edt$ and $\edtp$, the derivative operators along $\NPm$ and $\NPmbar$, because they are already spin-weighted scalars. 

\begin{definition} 
\label{def:VY-cov} 
Let $\uplambda$ be as in definition \ref{def:uplambda}. 
\begin{enumerate} 
\item Define the following spin-weighted operators by their action on spin-weighted scalar $\varphi$ with spin $s$
\index{Y2Y@$\YOp$}\index{V2V@$\VOp$}
\begin{subequations}
\begin{align}
\VOp\varphi={}&\frac{\Sigma}{\sqrt{2} \uplambda (a^2 + r^2)} \tho \varphi
 + \frac{27 s \kappa_{1}{}^2 (\kappa_{1}{} -  \bar{\kappa}_{1'}{}) \rho \rho ' \varphi}{a^2 + r^2},
 \label{eq:VOpGHPDef}\\
\YOp\varphi={}&\sqrt{2} \uplambda \thop .
\label{eq:YOpGHPDef}
\end{align}
\end{subequations}
\item
Define the vector fields 
\index{V2Vvec@$\VOp^a$}
\index{Y2Yvec@$\YOp^a$}
\begin{align} \label{eq:Vdef-fol} 
\VOp^a ={}& 
\frac{\Sigma}{\sqrt{2} \uplambda (a^2 + r^2)}  l^a, \qquad 
\YOp^a = 
\sqrt{2} \uplambda  n^a .
\end{align} 
\end{enumerate}
\end{definition}

\begin{remark} 
The operators $\VOp$ and $\YOp$ represent derivatives along the principal null directions, and have boost- and spin-weight zero. In fact, when acting on scalars of boost- and spin-weight zero, the operators $\VOp$ and $\YOp$ reduce to $\VOp^a\nabla_a$ and $\YOp^a\nabla_a$.
\end{remark}

It is convenient to introduce a further set of angular derivatives. For $a\not=0$, the planes orthogonal to the principal null directions are not integrable in the sense of Frobenius; in particular, the planes spanned by $\NPm$ and $\NPmbar$ are not tangent to $2$-dimensional spheres. For this reason, the following angular operators are introduced, so that they correspond to differentiation tangent to the spheres arising as level sets of $v$ and $r$. The $\mathring{}$ accent is used to denote operators tangent to these spheres. Although it is not immediately obvious that these operators correspond to differentiation along the spheres, this fact is a consequence of equations \eqref{eq:hedt-explicit}-\eqref{eq:hedtp-explicit} below. 

\begin{definition}
\label{def:edtedt'-cov}
Define the following spin-weighted operators by their action on a spin-weighted scalar $\varphi$ with spin weight $s$
\index{E3ethhat@$\hedt, \hedtp$}
\begin{subequations}
\begin{align}
\hedt\varphi={}&3 \kappa_{1}{} \edt \varphi 
- 9 \LxiOp{}\varphi \kappa_{1}{}^2 \tau
 + 3 s \kappa_{1}{} \tau \varphi,
 \label{eq:relationhedtandedt}\\
\hedtp\varphi={}&3 \bar{\kappa}_{1'}{} \edtp \varphi
-9 \LxiOp{}\varphi \bar{\kappa}_{1'}{}^2 \bar{\tau}
 - 3 s \bar{\kappa}_{1'}{} \bar{\tau} \varphi
 \label{eq:relationhedtprimeandedtprime}.
\end{align}
\end{subequations}
\end{definition}

The Lie derivative of a GHP scalar along a Killing vector field is defined in \cite{MR917488}. Note that the Lie derivative of a GHP scalar along a general vector field is not defined. The following lemma gives the derivatives along the Killing vectors $\xi$, $\zeta$, and $\eta$ given in equation \eqref{eq:KillingVecDefs}. 

\begin{lemma}
The Killing vector fields $\xi^a$, $\zeta^a$, and $\eta^a$, defined by \eqref{eq:KillingVecDefs}, yield the following spin-weighted Lie derivative operators by their action on a spin-weighted scalar $\varphi$, 
\index{L2Lxi@$\LxiOp$}
\index{L2Lzeta@$\LzetaOp$}
\index{L2Leta@$\LetaOp$}
\begin{subequations}
\begin{align}
\LxiOp{}\varphi ={}&-3 \kappa_{1}{} \rho ' \tho \varphi
 + 3 \kappa_{1}{} \rho \thop \varphi
 + 3 \kappa_{1}{} \tau ' \edt \varphi
 - 3 \kappa_{1}{} \tau \edtp \varphi
 + \tfrac{3}{2} s (\Psi_{2} \kappa_{1}{} -  \bar\Psi_{2} \bar{\kappa}_{1'}{}) \varphi ,\\
\LzetaOp{}\varphi ={}&
  \tfrac{27}{4} \kappa_{1}{} (\kappa_{1}{} -  \bar{\kappa}_{1'}{})^2 (\rho ' \tho \varphi -  \rho \thop \varphi)
 -  \tfrac{27}{4} \kappa_{1}{} (\kappa_{1}{} + \bar{\kappa}_{1'}{})^2 (\tau ' \edt \varphi -  \tau \edtp \varphi)\nonumber\\
&- \tfrac{27}{8} s \bigl((\kappa_{1}{} + \bar{\kappa}_{1'}{})^2 (\Psi_{2} \kappa_{1}{} -  \bar\Psi_{2} \bar{\kappa}_{1'}{}) + 8 \kappa_{1}{}^2 (- \kappa_{1}{} + \bar{\kappa}_{1'}{}) \rho \rho '\bigr) \varphi,\\
\LetaOp\varphi ={}&a^{-1}\LzetaOp{}\varphi - a \LxiOp{}\varphi.
\end{align}
\end{subequations}
\end{lemma}

The following relation will also turn out to be useful
\begin{align}
\LxiOp{}\varphi={}&
 \VOp\varphi
 + \frac{\Delta}{2 (a^2 + r^2)}\YOp\varphi
 - \frac{a}{a^2 + r^2}\LetaOp\varphi.
 \label{eq:LxiToVOpYOpLeta}
\end{align}

We now introduce a collection of spin-weighted operators that are useful when considering wave equations for spin-weighted scalars. 

\index{R2Rshat@$\hROp_s$}
\index{S2Ss@$\TMESOp_s$}
\index{S2Ss0@$\mathring{S}_{s}$}
\index{S5squareShats@$\widehat\squareS_s$}
\begin{definition}
\label{def:TMEOperators}
Define the following spin-weighted operators: 
\begin{subequations}
\begin{align}
\hROp_s  ={}& 2 (a^2 + r^2)\YOp\VOp
 - \frac{2 a r}{a^2 + r^2}\LetaOp  + \frac{(a^4 - 4 M a^2 r + a^2 r^2 + 2 M r^3) }{(a^2 + r^2)^2} , 
\label{eq:hatRs-sec2}\\ 
\TMESOp_s={}&
 2 (\hedt - 9 \kappa_{1}{}^2 \tau \LxiOp{})(\hedtp - 9 \bar{\kappa}_{1'}{}^2 \bar{\tau}\LxiOp{})
 -3 (2 s - 1) (\kappa_{1}{} -  \bar{\kappa}_{1'}{})\LxiOp{}, \label{eq:SOp-def} 
\\
\mathring{S}_{s}={}& 2\hedt\hedtp, \label{eq:mathringSbyhedthedt'}\\
\widehat\squareS_s ={}& \hROp_s - \TMESOp_s . 
\label{eq:rescaledWaveOperator}
\end{align}
\end{subequations} 
\end{definition} 

\begin{remark}
\begin{enumerate}
\item \label{point:Alembert} The standard d'Alembertian is related to $\widehat\squareS_s$ via
\begin{align}
\nabla^a\nabla_a\varphi ={}& \frac{1}{\Sigma \sqrt{a^2+r^2}} \widehat\squareS_{0}(\sqrt{a^2+r^2}\varphi) . 
\end{align}
\item The operator $\hROp_s$ has no explicit $s$-dependence. In particular, $\hROp_s$ coincides with the radial part of the d'Alembertian.
\item 
The operators $\hROp_s, \TMESOp_s$ are related to the Teukolsky radial and angular operators, cf. \cite{1972PhRvL..29.1114T,teukolsky:1973,1974ApJ...193..443T} and equations \eqref{eq:TeukolskyRegular-2}-\eqref{eq:TeukolskyRegular+2}.  In particular, the famous separability of the Teukolsky equation can be expressed as the commutativity of the radial and angular operators, i.e.{}
\begin{align}
[ \hROp_s, \TMESOp_s] = 0 .
\end{align}
\item Substituting $a=0$, one finds $\mathring{S}_s=\TMESOp_s$. 
\end{enumerate}
\end{remark}

The following three lemmas provide alternative expressions for some of the operators, expressions in terms of coordinates for local trivializations of the bundle via the choice of the Znajek tetrad, and expressions for commutators, respectively. Each is proved by direct computation. 

\begin{lemma} 
\label{lem:S-form} 
Let $\varphi$ be a spin-weighted scalar. 
\begin{subequations}
\begin{align}
\TMESOp_s \varphi ={}&2 \hedt\hedtp\varphi
 + 2 a \LetaOp\LxiOp{}\varphi   + \tfrac{1}{4}  \bigl(4 a^2 + 9 (\kappa_{1}{} -  \bar{\kappa}_{1'}{})^2\bigr)\LxiOp{}\LxiOp{}\varphi 
 - 3 s(\kappa_{1}{} -  \bar{\kappa}_{1'}{}) \LxiOp{}\varphi .
 \label{eq:expandedSssss} \\ 
\overline{\TMESOp_s \varphi}={}& \TMESOp_{-s}\bar{\varphi}-2 s \bar{\varphi}.
\end{align} 
\end{subequations} 
\end{lemma}

\begin{lemma} \label{lem:Op-forms} Let $\varphi$ be a spin-weighted scalar. 
In the Znajek tetrad and ingoing Eddington-Finkelstein coordinates, we have  
\begin{subequations}\label{eq:ops-Znaj-IEF}
\begin{align}
\VOp\varphi={}&\partial_{v} \varphi
 + \frac{\Delta \partial_{r} \varphi}{2 (a^2 + r^2)}
 + \frac{a \partial_{\phi} \varphi}{a^2 + r^2}, \label{eq:V-Znaj-IEF} \\
\YOp\varphi={}&- \partial_{r} \varphi,\label{eq:Y-Znaj-IEF}\\
\hedt\varphi={}&\tfrac{1}{\sqrt{2}}\partial_{\theta} \varphi
 + \tfrac{i}{\sqrt{2}}\csc\theta \partial_{\phi} \varphi
 - \tfrac{1}{\sqrt{2}}s \cot\theta \varphi, \label{eq:hedt-explicit} \\
\hedtp\varphi={}&\tfrac{1}{\sqrt{2}}\partial_{\theta} \varphi
 -  \tfrac{i}{\sqrt{2}}\csc\theta \partial_{\phi} \varphi
 + \tfrac{1}{\sqrt{2}}s \cot\theta \varphi,\label{eq:hedtp-explicit} \\
\LxiOp{}\varphi={}&\partial_{v} \varphi,\\
\LetaOp\varphi={}&\partial_{\phi} \varphi, \\
\hROp_s\varphi ={}&-2 (a^2 + r^2) \partial_{v} \partial_{r} \varphi
 -  \Delta \partial_{r} \partial_{r} \varphi
 - 2 a \partial_{r} \partial_{\phi} \varphi
 + \frac{2 M (a^2 -  r^2) \partial_{r} \varphi}{a^2 + r^2}
 + \frac{2 a r \partial_{\phi} \varphi}{a^2 + r^2}\nonumber\\
& + \frac{(a^4 - 4 M a^2 r + a^2 r^2 + 2 M r^3) \varphi}{(a^2 + r^2)^2}, \\
\TMESOp_s \varphi ={}&a^2 \sin^2\theta \partial_{v} \partial_{v} \varphi
 + 2 a \partial_{v} \partial_{\phi} \varphi
 + \partial_{\theta} \partial_{\theta} \varphi
 + \csc^2\theta \partial_{\phi} \partial_{\phi} \varphi
 - 2i a s \cos\theta \partial_{v} \varphi
\nonumber\\
&  + \cot\theta \partial_{\theta} \varphi + 2i s \cot\theta \csc\theta \partial_{\phi} \varphi
 + s (s -  s \csc^2\theta - 1) \varphi . \label{eq:SOp-coords}
\end{align}

\end{subequations}
\end{lemma}

\begin{remark} \label{rem:ops}
Restricting to the sphere, spin-weighted scalars can be viewed as sections of complex line bundles. Defining spin-weighted scalars in terms of a null tetrad corresponds to a choice of local trivialization for these bundles. The form of the operators $\hedt, \hedtp$ given in \eqref{eq:hedt-explicit} and \eqref{eq:hedtp-explicit} are expressions, in the given tetrad and coordinate system, of covariantly defined elliptic operators of order one, acting on spin-weighted scalars on the sphere,  cf. \cite{1982MPCPS..92..317E}.
\end{remark}

\begin{lemma} Let $\varphi$ be a spin-weighted scalar. 
We have the commutator relations
\begin{align}
\label{eq:CommutatorofYandMathcalV}
\YOp\VOp\varphi={}&\VOp\YOp\varphi
 + \frac{2 a r}{(a^2 + r^2)^2}\LetaOp\varphi
 + \frac{M (- a^2 + r^2)}{(a^2 + r^2)^2}\YOp\varphi,&
\mathcal{L}_\xi \YOp\varphi={}&\YOp\LxiOp\varphi ,
\end{align}
and 
\begin{subequations}
\label{eq:CommutAnguDeri}
\begin{align}
\YOp\hedt\varphi={}&\hedt\YOp\varphi,&
\YOp\hedtp\varphi={}&\hedtp\YOp\varphi,\\
\VOp\hedt\varphi={}&\hedt\VOp\varphi,&
\VOp\hedtp\varphi={}&\hedtp\VOp\varphi,\\
\LxiOp{}\hedt\varphi={}&\hedt\LxiOp{}\varphi,&
\LxiOp{}\hedtp\varphi={}&\hedtp\LxiOp{}\varphi,\\
\hedt\hedtp\varphi={}&\hedtp\hedt\varphi -  s \varphi.
\label{eq:commutatorofhedtandhedtprime}
\end{align}
\end{subequations}
\end{lemma} 

\subsection{Time functions} 
\label{sec:foliation} 

In the first part of this section, we introduce new time functions, and, in the second, we use these to define regions of the Kerr exterior in which we integrate in later sections of the paper. 

Consider first choices of time functions. Recall from the illustration in figure \ref{fig:Estimates}, that we intend to use, in one part of the argument, time functions with level sets that go from the horizon $\Horizon^+$ to spacelike infinity $i_0$ and, in another, time functions with level sets going from the horizon $\Horizon^+$ to null infinity $\Scri^+$. We refer to the former as horizon crossing and the latter as hyperboloidal. We begin by introducing definitions that more precisely specify the desired properties for these time functions. 

Definition \ref{def:maintimedef} below introduces the technical conditions on the time functions. 
The time functions are required to have smooth, spacelike level sets. 
The level sets of the hyperboloidal time function are required to reach $\Scri^+$ and to be regular and sufficiently spacelike there, cf. point \eqref{point:Vlimit}.
The hyperboloidal time is also required to march forward along $\Scri^+$ at the same rate as the retarted time $u$, cf. point  \eqref{point:Ylimit}. Point \eqref{point:kprim} of definition \ref{def:maintimedef} requires that the horizon crossing time function  behaves like the Boyer-Lindquist time function near infinity.

\begin{definition} \label{def:maintimedef}
We consider time functions $\tGen$ defined in terms of height functions $k=k(r)$, 
\begin{equation}\label{eq:tGen-def}
\tGen = v - k(r)
\end{equation}
where $v$ is the advanced time coordinate in the ingoing Eddington-Finkelstein coordinate system. 
\begin{enumerate} 
\item \label{point:hyptime}
A time function $\tGen$ of the form \eqref{eq:tGen-def} is a regular, future hyperboloidal time function, if  
\begin{enumerate}  
\item \label{point:h-smooth} $k(r)$ is smooth in an open neighbourhood of $[r_+,\infty)$.
\item \label{point:H-smooth} $K(R)= k'(1/R)  =k'(r)$, where $R = 1/r$, is smooth in an open neighbourhood of $[0, 1/r_+]$.
\item \label{point:spacelike} The level sets of $\tGen$ are strictly spacelike in $\mathcal{M}$. 
\item \label{point:Vlimit} The limit 
\begin{align}\label{eq:CInHyperboloids-def}
\lim_{r\to \infty} \frac{r^2}{M^2} \VOp^a \nabla_a \tGen 
\end{align}  
exists and is positive. 
\item \label{point:Ylimit} 
\begin{align} \label{eq:Yt-def}
\lim_{r\to\infty} \YOp^a \nabla_a \tGen = 2 .
\end{align}
\end{enumerate}

\item \label{point:tHor} A time function $\tGen$ of the form \eqref{eq:tGen-def}
with height function $k = k(r)$ is horizon crossing if 
\begin{enumerate}
\item \label{point:kHor} $k(r)$ is smooth in an open neighbourhood of $[r_+,\infty)$. 
\item \label{point:kspace} The level sets of $\tGen$ are strictly spacelike in $\mathcal{M}$. 
\item \label{point:kprim} For large $r$, $k'(r) - (a^2+r^2)/\Delta = O(r^{-2})$.
\end{enumerate}
\end{enumerate}
\end{definition} 

Having introduced the general properties we would like time functions to possess, we now introduce some specific examples that we will show have the desired properties. In our choice of height function $h$ in the construction of these time functions, the first three terms on the right of equation \eqref{eq:hdef} are those that arise from integrating the terms used to define $r_*$ in equation \eqref{eq:r*def}; if only these three terms were present, then $t=v-h(r)$ would coincide with the retarded time $u$ from equation \eqref{eq:u-def}. Unfortunately, the normal to the level sets of $u$ are not timelike and, in fact, fail to even be null for $\theta\not\in\{0,\pi\}$. For this reason, we include the final two terms in equation \eqref{eq:hdef}, which include $\CInHyperboloids$. As shown in equation \eqref{eq:uMinusTimefunc}, a level set of $\timefunc$ will approach, and in the limit towards $\Scri^+$ reach, a level set of $u$, but for a fixed value of $t$, $u$ will be larger for $1/r$ small and positive. Thus, the level sets of the hyperboloidal time function $\timefunc$ can be thought of as bending upward from $\Scri^+$ away from the level sets of $u$. This ensures that the level sets of $\timefunc$ are spacelike. The coefficient $\CInHyperboloids$ can be viewed as a measure of this curvature near $\Scri^+$; we have chosen to use the coefficient $(\CInHyperboloids-1)$ in equation \eqref{eq:hdef} so that the measure of curvature on the right of equation \eqref{eq:Vt-lem} is $\CInHyperboloids$. 

\index{C2CHyp@$\CInHyperboloids$}
\begin{lemma} 
\label{lem:h-prop} 
Let $\CInHyperboloids \geq 1$ and let $\timefunc=v-h(r)$ on $\mathcal{M}$, where 
\index{H1h@$h(r)$}
\begin{align} \label{eq:hdef}
h(r)={}&2 (r- r_{+})
+ 4 M \log\left(\frac{r}{r_{+}}\right)
+ \frac{3 M^2 (r_{+} -  r)^2}{r_{+} r^2}
+ 2 M \arctan\left(\frac{(\CInHyperboloids-1) M}{r}\right)\nonumber\\
&-  2 M\arctan\left(\frac{(\CInHyperboloids-1) M}{r_{+}}\right),
\end{align}
where $r_+$ is given by \eqref{eq:r+def}. Then $\timefunc$ is a regular, future hyperboloidal time function as in definition \ref{def:maintimedef}. Further, 
\begin{subequations}\label{eq:hprop}
\begin{align}
h(r_+) ={}& 0,  \label{eq:h=0} \\
h'(r) \geq{}& 0, \quad \text{for $r \geq r_+$}  \label{eq:h'pos} \\  
\lim_{r\to \infty} \frac{h(r)}{r} ={}& 2, \label{eq:hlim}\\ 
\lim_{r\to \infty} \frac{r^2}{M^2}\VOp^a \nabla_a \timefunc  ={}& \CInHyperboloids . \label{eq:Vt-lem}
\end{align} 
\end{subequations}
\end{lemma} 
\begin{proof}
It is straightforward to verify \eqref{eq:h=0}, \eqref{eq:hlim}, and \eqref{eq:Vt-lem}. We have 
\begin{align} \label{eq:hprim-explicit}
h'(r)={}&2  + \frac{4 M}{r} +  \frac{6 M^2 (r - r_{+})}{r^3} -  \frac{2 (\CInHyperboloids-1) M^2}{(\CInHyperboloids-1)^2 M^2 + r^2} .
\end{align}
Based on this and \eqref{eq:hdef}, it is straightforward to verify points \ref{point:h-smooth}, \ref{point:H-smooth} of definition \ref{def:maintimedef}. Next, we prove that $\timefunc$ has spacelike level sets. We have
\begin{align}
d\timefunc_{a} d\timefunc_{b} g^{ab} \frac{\Sigma}{\Delta}={}&- \frac{a^2 \sin^2\theta}{\Delta}
 + \frac{(a^2 + r^2)^2}{\Delta^2}
 -  \Bigl(h'(r) - \frac{a^2 + r^2}{\Delta}\Bigr)^2\nonumber\\
\geq{}&- \frac{a^2}{\Delta}
 + \frac{(a^2 + r^2)^2}{\Delta^2}
 -  \Bigl(h'(r) - \frac{a^2 + r^2}{\Delta}\Bigr)^2. \label{eq:t-spacelike-basic}
\end{align}
Hence, $\timefunc$ has spacelike level sets if and only if 
\begin{align}
0\leq\frac{a^2 + r^2}{\Delta}\left(1-\sqrt{1 - \frac{a^2\Delta}{(a^2 + r^2)^2}}\right)
&< h'(r) <
\frac{a^2 + r^2}{\Delta}\left(1 + \sqrt{1 - \frac{a^2\Delta}{(a^2 + r^2)^2}}\right).
\label{eq:spacelikecond1}
\end{align}
Using the inequality $x\leq \sqrt{x}$ for $0\leq x\leq 1$ one finds that a sufficient (but not necessary) condition for the level sets $\Sigma_\timefunc$ to be spacelike is given by 
\begin{align}
\frac{a^2}{a^2+r^2}
&< h'(r) <
\frac{2(a^2 + r^2)}{\Delta}  -\frac{a^2}{a^2+r^2}. \label{eq:t-spacelike-useful}
\end{align}
Since $\CInHyperboloids \geq 1$ by assumption, we have using \eqref{eq:hprim-explicit}
\begin{align} 
\frac{2(a^2 + r^2)}{\Delta}  -\frac{a^2}{a^2+r^2} - h'(r) >{}& 
\frac{6M^2r_+}{r^3} + J  
\end{align} 
where 
\begin{align} 
J = \frac{2(a^2 + r^2)}{\Delta}  -\frac{a^2}{a^2+r^2}
- 2  - \frac{4 M}{r} -  \frac{6 M^2}{r^2} .
\end{align} 
Collecting powers of $r$ in $\Delta(a^2+r^2)r^2 J$, and using $r > r_+ > M> |a|$, one finds $J > 0$ on $\mathcal{M}$ and the right inequality in \eqref{eq:t-spacelike-useful} follows. To see that the left inequality in \eqref{eq:t-spacelike-useful} holds, note that  
\begin{align} \label{eq:left-ineq}
h'(r) - \frac{a^2}{a^2+r^2} >{}& 2  + \frac{4 M}{r}  -  \frac{2 (\CInHyperboloids-1) M^2}{(\CInHyperboloids-1)^2 M^2 + r^2} - \frac{a^2}{a^2+r^2} .
\end{align} 
To bound the second term of the right-hand side from below, we note that it is of the form 
\begin{align} \label{eq:somefun}
-2Mx/(x^2+r^2), 
\end{align} 
with $x=(\CInHyperboloids-1)M$. For $x > 0$, \eqref{eq:somefun} is bounded from below by $-M/r$. 
Further,  $a^2/(a^2+r^2) < 1$ is monotone decreasing for $r> r_+$. This gives 
\begin{align}
h'(r) - \frac{a^2}{a^2+r^2} >{}& 2 + \frac{3M}{r} - \frac{a^2}{a^2+r_+^2} \nonumber \\
\geq{}& 1+\frac{3M}{r} > 0 .\label{eq:2.50}
\end{align} 
Hence, the level sets of $\timefunc$ are strictly spacelike in $\mathcal{M}$. The inequality \eqref{eq:2.50} yields \eqref{eq:h'pos}. 
The remaining points \ref{point:Vlimit}, \ref{point:Ylimit} of definition \ref{def:maintimedef} can be verified by straightforward calculations. 
\end{proof}

The time function $\tHor$ is constructed similarly but in such a way that its level sets approach $i_0$. 

\begin{lemma}
\label{lem:k-prop}
Let $k = h/2$ with $h$ given by \eqref{eq:hdef}. Then $\tHor = v - k$ is a horizon crossing time function and 
\begin{align}
k(r_+) = 0  . \label{eq:k=0}  
\end{align}
\end{lemma} 
\begin{proof} It is straightforward to verify points \ref{point:kHor}, \ref{point:kprim} of definition \ref{def:maintimedef}. For point \ref{point:kspace} we proceed as in the proof of lemma \ref{lem:h-prop}, and note that a sufficient condition for $\tHor$ to have spacelike level sets is given by \eqref{eq:t-spacelike-useful} with $h'$ replaced by $k'$, 
\begin{align}
\frac{a^2}{a^2+r^2}
&< k'(r) <
\frac{2(a^2 + r^2)}{\Delta}  -\frac{a^2}{a^2+r^2}. \label{eq:t-spacelike-useful-k}
\end{align}
It follows from the proof of lemma \ref{lem:h-prop} that $h' > 0$, and the second inequality in \eqref{eq:t-spacelike-useful-k} holds since from $k=h/2$ we have that $k' < h'$.  For the first inequality in \eqref{eq:t-spacelike-useful-k}, we have, following the proof of lemma \ref{lem:h-prop}, 
\begin{align} 
k' - \frac{a^2}{a^2+r^2} >{}& 1  + \frac{2 M}{r}  -  \frac{ (\CInHyperboloids-1) M^2}{(\CInHyperboloids-1)^2 M^2 + r^2} - \frac{a^2}{a^2+r^2} \nonumber \\ 
>{}& 1 + \frac{3}{2} \frac{M}{r} - \frac{a^2}{a^2+r_+^2} \nonumber \\
>{}& 0 .
\end{align} 
This completes the proof. 
\end{proof}

\begin{definition}[Time functions] \label{def:timefuncs}
Define the horizon-crossing time $\tHor$ and the hyperboloidal time $\timefunc$,
\begin{subequations} 
\begin{align}
\tHor ={}& v - h/2, \\
\timefunc ={}& v -h
\end{align}
\end{subequations}
with $h$ as in \eqref{eq:hdef}.
\end{definition}

\begin{remark}
\begin{enumerate} 
\item There is a constant $c_h$ such that the retarded time  $u$ and the hyperboloidal time $\timefunc$ satisfy, for large $r$, 
\begin{align} 
u-\timefunc ={}& c_h + 2\CInHyperboloids M^2/r + O(1/r^2).
\label{eq:uMinusTimefunc}
\end{align}
Thus, the level sets of $\timefunc$ are asymptotic to level sets of $u$ and intersect at $\Scri^+$. This is consistent with the fact that $\lim_{r\to \infty} \YOp^a \nabla_a u  = 2$ is the same as the limit $\lim_{r\to \infty} \YOp^a \nabla_a \timefunc$ given in \eqref{eq:Yt-def}. 
\item Similarly, there is a constant $c_k$ such that the Boyer-Lindquist time and the horizon crossing time $\tHor$ satisfy, for large $r$,  
\begin{align} 
t_{BL} - \tHor ={}& c_k + \CInHyperboloids M^2/r + O(1/r^2) .
\end{align} 
\item From $h(r_+) = 0$ and \eqref{eq:h'pos} we have that $h(r) \geq 0$ for $r \geq r_+$. It follows that $\Sigma_{\timefunc_2}$ is contained in the future of $\{\tHor = \timefunc_1\} \cap \{ r > r_+\}$ precisely when $\timefunc_2 \geq \timefunc_1$.    
\item Although the class of hyperboloidal time functions introduced in point \ref{point:hyptime} of definition \ref{def:maintimedef} could be employed in this paper, for simplicity we only make use of the explicit hyperboloidal time $\timefunc$. 
\end{enumerate}  
\end{remark} 

We now define several hypersurfaces and regions of the Kerr exterior in terms of the time functions. 
These are illustrated in figures  \ref{fig:FoliationHyperbolic}. 
We will pose initial data on an initial hypersurface $\tHor=10M$, with $10M$ taken fairly arbitrarily so that it is sufficiently large that it is positive, satisfies $10M>r_+$, and so forth. We also find it useful to use different arguments in the regions separated by the transition hypersurface $\Xi=\{\timefunc=r\}$. The interior of $\Xi$, where $r<t$ is denoted with a superscript $\interior$, and the exterior, where $r>t$ by $\ext$. 

\begin{definition}
\begin{enumerate}
\item 
The future domain of dependence of a hypersurface $\Sigma \subset \mathcal{M}$ is denoted $\DoDfuture(\Sigma)$. 
\index{D2DoD@$\DoDfuture$}
\item
For a subset $\Omega \in \mathcal{M}$, let $I^+(\Omega), I^-(\Omega)$ denote the time-like future and past of $\Omega$, respectively. 
\end{enumerate}
\end{definition}

\begin{figure}[t!]
\centering
\begin{subfigure}[t]{0.5\textwidth}
\centering 
\begin{tikzpicture}[xscale=0.80,yscale=0.80]

\coordinate (A0) at (8.5,1.5);
\coordinate (A1) at (6.4,0.8);
\coordinate (A2) at (5,0.1);
\coordinate (A3) at (1,1.5);

\coordinate (B0) at (7.5,2.5);
\coordinate (B1) at (6,2);
\coordinate (B2) at (5,1.5);
\coordinate (B3) at (2,2.5);

\coordinate (i+) at (5,5);
\coordinate (C1) at (6,2);
\coordinate (C2) at (5,1);
\coordinate (C3) at (4.5,0);

\draw[thick,blue,name path=horizon]  (1,1) -- node[sloped, above] {$\mathscr{H}^+$} (5,5);
\draw[name path=scri, ,thick, dash dot, red]  (5,5) -- node[sloped, above] {$\Scri^+$} (9,1);

\path[name path=consttau0, gray] (A0) .. controls (A1) and (A2) .. node[pos=0.52, below, black] {$\Staui$} (A3);
\path[name path=consttau,gray] (B0) .. controls (B1) and (B2) .. node[pos=0.60, sloped, above, black] {$\Stau$} (B3);

\begin{scope}
  \clip [closed]  (1,1) -- (5,5) --  (10,5) -- (10,1) -- (5,0);
  \fill[fill=lightgray!50]  (A0) .. controls (A1) and (A2) ..  (A3) -- (B3) .. controls (B2) and (B1) .. (B0) -- cycle;
\end{scope}

\begin{scope}
  \clip [closed]   (C3) .. controls (C2) and (C1) .. (i+) --  (10,5) -- (10,1) -- (5,0);
  \fill[fill=lightgray]  (A0) .. controls (A1) and (A2) ..  (A3) -- (B3) .. controls (B2) and (B1) .. (B0) -- cycle;
\end{scope}

\begin{scope}
  \clip [closed]  (1,1) -- (5,5) --  (10,5) -- (10,1) -- (5,0);
  \draw[draw=black, thick]  (A0) .. controls (A1) and (A2) ..  (A3) -- (B3) .. controls (B2) and (B1) .. (B0);
\end{scope}

\draw[black, thick, name intersections={of=consttau0 and horizon, by=htau0}, name intersections={of=consttau and horizon, by=htau}]  (htau0)  --  (htau);

\draw[black, thick,dashed] (A0) --  (B0);

\draw [dashed] (i+) .. controls (C1) and (C2) .. (C3) node[sloped, near start, above] {$r=\timefunc$};

\filldraw[fill=white,draw=black] (i+) circle[thick,radius=.7mm] node[anchor=south,yshift=.1cm]{$i_+$};

\node at (5.9,1.5) {$\Boundext$};

\node at (7,1.6) {$\Dtauext$};

\node at (3.5,1.6) {$\Dtauint$};

\end{tikzpicture}%
\end{subfigure}%
\begin{subfigure}[t]{0.5\textwidth}
\centering
\begin{tikzpicture}[xscale=0.80,yscale=0.80]

\coordinate (A0) at (8.5,1.5);
\coordinate (A1) at (6.4,0.8);
\coordinate (A2) at (5,0.1);
\coordinate (A3) at (1,1.5);

\coordinate (B0) at (3.5,3.5);
\coordinate (B1) at (5.5,1.5);

\coordinate (i+) at (5,5);
\coordinate (C1) at (6,2);
\coordinate (C2) at (5,1);
\coordinate (C3) at (4.5,0);

\begin{scope}
  \clip [closed]   (C3) .. controls (C2) and (C1) .. (i+) --  (B0) -- cycle;
  \fill[fill=lightgray]  (B0) -- (B1) -- (6,3.5) -- (i+) -- cycle;
\end{scope}

\draw[thick,blue,name path=horizon]  (1,1) -- node[sloped, above, pos=0.4] {$\mathscr{H}^+$} (5,5);
\draw[name path=scri, thick, dash dot, red]  (5,5) -- node[sloped, above] {$\Scri^+$} (9,1);

\path[name path=consttau0, gray] (A0) .. controls (A1) and (A2) .. node[pos=0.52, below, black] {$\Sigma_{\timefunc_0}$} (A3);

\path[name path=geod] (B0) -- (B1);

\draw[name path=reqt, dashed] (i+) .. controls (C1) and (C2) .. (C3) node[sloped, pos=0.25, above] {$r=\timefunc$};

\draw[thick, name intersections={of=geod and reqt, by=reqtgeod}] (B0) --  node[sloped, below] {$v=\vOne$} (reqtgeod);

\begin{scope}
  \clip [closed]  (1,1) -- (5,5) --  (10,5) -- (10,1) -- (5,0);
  \draw[draw=black, thick]  (A0) .. controls (A1) and (A2) ..  (A3);
\end{scope}

\filldraw[fill=white,draw=black] (i+) circle[thick,radius=.7mm] node[anchor=south,yshift=.1cm]{$i_+$}; 

\filldraw[fill=black,draw=black] (reqtgeod) circle[thick,radius=.7mm] node[anchor=west,xshift=.1cm]{$\tprimeingoing$}; 

\filldraw[fill=black,draw=black] (B0) circle[thick,radius=.7mm] node[anchor=south east]{$\vOne$}; 

\node at (4.5,3.5) {$\Dtautnear{\vOne}$};

\end{tikzpicture}
\end{subfigure} 
\caption{Hyperboloidal regions, cf. definition \ref{def:regions}, and surfaces used for interior estimates.}
\label{fig:FoliationHyperbolic}
\end{figure}

\index{O4Omegat@$\Dtau$}
\index{S4Sigmat@$\Staut$}
\begin{definition} \label{def:regions}
\begin{enumerate} 
\item
Define $\timefunc_0 = 10M$, and define the initial hypersurface $\Stauini$ by 
\begin{align} \label{eq:Sigmainit-def} 
\Stauini = \{ \tHor = \timefunc_0 \} \cap \{ r > r_+\} . 
\end{align}
\item
Given $\timefunc_1\in\Reals$, $\Staui$ denotes the corresponding level set of the hyperboloidal time function $\timefunc$, restricted to $\DoDfuture(\Stauini)$, 
\begin{align} 
\Staui = \{ \timefunc = \timefunc_1 \} \cap \DoDfuture(\Stauini). 
\end{align} 
\item 
Given $- \infty \leq \timefunc_1<\timefunc_2 \leq \infty$ and $r_+\leq r_1< r_2$, define 
\begin{subequations}
\begin{align}
\Staut^{r_1} ={}& \Staut \cap \{r_1\leq r\} ,\\
\Staut^{r_1,r_2} ={}& \Staut \cap \{r_1\leq r\leq r_2\} ,\\
\Dtau ={}& \bigcup_{\timefunc_1\leq\timefunc\leq\timefunc_2} \Staut ,\\
\Dtau^{r_1} ={}& \Dtau \cap \{r_1\leq r\} ,\\
\Dtau^{r_1,r_2} ={}& \Dtau \cap \{r_1\leq r\leq r_2\}.
\end{align}
\end{subequations}
\index{O4Omegatext@$\Dtauext$}
\index{O4Omegatint@$\Dtauint$}
\index{X4XiBoundext@$\Boundext$}
\index{S4Sigmatiext@$\Stauiext$}
\index{S4Sigmatiint@$\Stauiint$}
\index{H4Horgtt@$\Horgtt{\timefunc}$}

\item Given $- \infty \leq \timefunc_1<\timefunc_2 \leq \infty$, define 
the transition surface $\Xi$ and a subset thereof to be  
\begin{subequations}
\begin{align} 
\Xi ={}& \{r = \timefunc\} \cap \DoDfuture(\Stauini),\\
\Boundext={}&\Dtau \cap \Xi.
\end{align}
\end{subequations}
\item Given $- \infty \leq \timefunc_1<\timefunc_2 \leq \infty$, define 
\begin{subequations}
\begin{align}
\Stauiext={}& \Staui \cap \{r\geq \timefunc\},\\
\Stauiint={}& \Staui \cap \{r\leq \timefunc\},\\
\Dtauext={}&\Dtau \cap \{r\geq \timefunc\},\\
\Dtauint={}&\Dtau \cap \{r\leq \timefunc\}.
\end{align}
\end{subequations}
\item  
Given $\timefunc_1 < \infty$, define 
\begin{align}
\Horgtt{\timefunc_1} ={}& \Horizon^+ \cap \{\timefunc \geq \timefunc_1 \}. 
\end{align} 
\end{enumerate} 
\end{definition}

\begin{figure}[t!]
\centering
\begin{subfigure}[t]{0.5\textwidth}
\centering 
\begin{tikzpicture}[xscale=0.80,yscale=0.80]

\coordinate (A0) at (9,1);
\coordinate (A1) at (6.5,1);
\coordinate (A2) at (5,1.1);
\coordinate (A3) at (2.04,2.54);

\coordinate (B0) at (7.5,2.5);
\coordinate (B1) at (6,2);
\coordinate (B2) at (5,1.5);
\coordinate (B3) at (2,2.5);

\coordinate (i+) at (5,5);

\draw[thick,blue,name path=horizon]  (1,1) -- node[sloped, above] {$\mathscr{H}^+$} (5,5);
\draw[name path=scri, ,thick, dash dot, red]  (5,5) -- node[sloped, above, near end] {$\Scri^+$} (B0);

\path[name path=consttau0, gray] (B0) .. controls (B1) and (B2) .. node[pos=0.6, sloped, above, black] {$\Sigma_{\timefunc_0}$} (B3); 

\path[name path=sigmainit, gray] (A0) .. controls (A1) and (A2) .. node[pos=0.6, sloped, below, black] {$\Stauini$} (A3);

\path[name intersections={of=consttau0 and horizon, by=tau0horizon}]  (1,1) -- (5,5);
\path[name intersections={of=sigmainit and horizon, by=sigmainithorizon}]  (1,1) -- (5,5);

\begin{scope}
  \clip [closed]  (1,1) -- (5,5) --  (10,5) -- (10,1) -- (5,0);
  \fill[lightgray]  (A0) .. controls (A1) and (A2) ..  (A3) -- (B3) .. controls (B2) and (B1) .. (B0) -- cycle;
\end{scope}
\begin{scope}
  \clip [closed]  (1,1) -- (5,5) --  (10,5) -- (10,1) -- (5,0);
  \draw[black, thick]  (A0) .. controls (A1) and (A2) ..  (A3) -- (B3) .. controls (B2) and (B1) .. (B0);
\end{scope}

\draw[black, thick,dashed] (A0) --  (B0);

\filldraw[fill=white,draw=black] (i+) circle[thick,radius=.7mm] node[anchor=south,yshift=.1cm]{$i_+$}; 
\filldraw[fill=white,draw=black] (A0) circle[thick,radius=.7mm] node[anchor=west,xshift=.1cm]{$i_0$}; 

\node at (7,1.6) {$\Dtauearly$};

\end{tikzpicture}%
\end{subfigure}%
\begin{subfigure}[t]{0.5\textwidth}
\centering
\begin{tikzpicture}[xscale=0.80,yscale=0.80]

\coordinate (A0) at (9,1);
\coordinate (A1) at (6.5,1);
\coordinate (A2) at (5,1.1);
\coordinate (A3) at (2.04,2.54);

\coordinate (B0) at (7.5,2.5);
\coordinate (B1) at (6,2);
\coordinate (B2) at (5,1.5);
\coordinate (B3) at (2,2.5);

\coordinate (C0) at (8,2);
\coordinate (C1) at (6.2,1.4);
\coordinate (C2) at (5,0.8);
\coordinate (C3) at (1.4,2);

\coordinate (i+) at (5,5);

\draw[thick,blue,name path=horizon]  (1,1) -- node[sloped, above] {$\mathscr{H}^+$} (5,5);
\draw[name path=scri, ,thick, dash dot, red]  (5,5) -- node[sloped, above, near end] {$\Scri^+$} (C0);

\path[name path=consttau, black] (C0) .. controls (C1) and (C2) .. node[pos=0.2, sloped, above] {$\Sigma_{\timefunc}$} (C3); 

\path[name path=consttau0, gray] (B0) .. controls (B1) and (B2) .. node[pos=0.5, sloped, above, black] {$\Sigma_{\timefunc_0}$} (B3); 

\path[name path=sigmainit, gray] (A0) .. controls (A1) and (A2) .. node[pos=0.5, sloped, below,black] {$\Stauini$} (A3);

\path[name intersections={of=consttau0 and horizon, by=tau0horizon}]  (1,1) -- (5,5);
\path[name intersections={of=sigmainit and horizon, by=sigmainithorizon}]  (1,1) -- (5,5);

\begin{scope}
  \clip [closed]  (1,1) -- (5,5) --  (10,5) -- (10,1) -- (5,0);
  \clip [closed]  (A0) .. controls (A1) and (A2) ..  (A3) -- (B3) .. controls (B2) and (B1) .. (B0);
  \fill[lightgray]  (C0) .. controls (C1) and (C2) ..  (C3) -- (A3) .. controls (A2) and (A1) .. (A0) -- cycle;
\end{scope}
\begin{scope}
  \clip [closed]  (1,1) -- (5,5) --  (10,5) -- (10,1) -- (5,0);
  \clip [closed]  (A0) .. controls (A1) and (A2) ..  (A3) -- (B3) .. controls (B2) and (B1) .. (B0);
  \draw[black, thick]  (C0) .. controls (C1) and (C2) ..  (C3) -- (A3) .. controls (A2) and (A1) .. (A0);
\end{scope}
\begin{scope}
  \clip [closed]  (1,1) -- (5,5) --  (10,5) -- (10,1) -- (5,0);
  \draw[black, thick] (B0) .. controls (B1) and (B2) .. (B3); 
  \draw[black,thick] (A0) .. controls (A1) and (A2) .. (A3);
\end{scope}

\draw[black, thick,dashed] (A0) --  (C0);

\filldraw[fill=white,draw=black] (i+) circle[thick,radius=.7mm] node[anchor=south,yshift=.1cm]{$i_+$}; 
\filldraw[fill=white,draw=black] (A0) circle[thick,radius=.7mm] node[anchor=west,xshift=.1cm]{$i_0$}; 

\node at (7.5,1.4) {$\Dtautearly$};

\end{tikzpicture}
\end{subfigure} 
\caption{Early regions ($\timefunc<\timefunc_0$), cf. definition \ref{def:Sigmainit}}
\label{fig:EarlyEstSurfaces}
\end{figure}

\begin{remark} \label{rem:2.27}
\begin{enumerate} 
\item For $\timefunc_1 \geq \timefunc_0$, the level set $\{\timefunc = \timefunc_1\} \cap \{ r > r_+\}$ is contained in $\DoDfuture(\Stauini)$, i.e. $\Staui = \{\timefunc = \timefunc_1\} \cap \{r > r_+\}$.  This follows from the fact that on $\Stauini$, $\tHor = \timefunc_0$, hence at each point on $\Stauini$, $\timefunc= \tHor - h/2 =t_0-h/2\leq t_0$, which demonstrates that $\Stauini$ is always in the past of $\Staui$ for any $\timefunc_1\geq \timefunc_0$. See figures \ref{fig:FoliationHyperbolic} and \ref{fig:EarlyEstSurfaces} for illustration.

\item 
From the definition of the hyperboloidal time function, we have that on $\Xi$, $r+h(r) = v$. Due to \eqref{eq:h'pos}, we have that $r \mapsto r+h(r)$ defines a diffeomorphism $[\timefunc_0,\infty) \to  [\timefunc_0 + h(\timefunc_0), \infty)$. 
\end{enumerate}
\end{remark} 

At a certain point in our argument, we need to consider the division of the interior and exterior regions not in the hyperboloidal coordinates but in the outgoing Eddington-Finkelstein coordinates. In this situation, we use the superscript $\near$ and introduce the function $\tprimeingoing$ to denote the value of $\timefunc(\vOne)$ to denote the value of $\timefunc$ at the intersection of the level set $v=\vOne$ with the transition hypersurface $\Xi$. 

\begin{definition} \label{def:tprime} 
Let $\vOne\geq \timefunc_0 + h(\timefunc_0)$.
\begin{enumerate} 
\item 
Define 
 \index{O4Omegav1near@$\Dtautnear{\vOne}$}
\begin{align}
\Dtautnear{\vOne}={}&\Dtauitint\cap \{v\geq \vOne\},
\end{align}
where $v$ is the advanced time.
\item 
\index{T1tCv1@$\tprimeingoing$}
Let $\tprimeingoing$ be the solution to the equation 
\begin{align} 
\tprimeingoing + h(\tprimeingoing) = \vOne  .
\end{align} 
\end{enumerate} 
\end{definition}
\begin{remark} \label{rem:tprime}
For $\vOne$ as in definition \ref{def:tprime}, $\tprimeingoing$ is well defined, and the point with hyperboloidal coordinate $(\tprimeingoing, \tprimeingoing,\omega)$  lies on $\Xi$, and is the point in $\Xi$ with ingoing Eddington-Finkelstein coordinate $(v_1, \tprimeingoing,\omega)$. For $v_1 = \timefunc_0 + h(\timefunc_0)$, this point lies on $\Sigma_{\timefunc_0}$. Further, for $\vOne \geq \timefunc_0 + h(\timefunc_0)$, $\vOne \sim \tprimeingoing$. 
\end{remark}  

We refer to the regions between the initial hypersurface $\Stauini$ and the hyperboloids $\Staut$ as early and denote these with the superscript $\early$. These regions are illustrated in figure \ref{fig:EarlyEstSurfaces}. 

\index{O4Omegat@$\Dtautearly$} 
\index{S4Sigmainit@$\Stauini$}
\begin{definition} \label{def:Sigmainit}
For $\timefunc \in \Reals$, define $\Dtautearly$ to be the intersection of the future of $\Stauini$ and the past of $\Staut$,
\begin{align} 
\Dtautearly = \DoDfuture(\Stauini) \cap I^-(\Staut).
\end{align}  
Furthermore, for $r_2> r_1\geq r_+$, define
\begin{subequations}
\begin{align}
\DtautearlyFurtherArguments{r_1}={}&\Dtautearly\cap\{r_1\leq r\} ,\\
\DtautearlyFurtherArguments{r_1,r_2}={}&\Dtautearly\cap\{r_1\leq r\leq r_2\} .
\end{align}
\end{subequations}
\end{definition}

\subsection{Compactified hyperboloidal coordinates}

In the final part of this section, we introduce a final coordinate system which is well adapted to working near infinity $\Scri^+$. This consists of the hyperboloidal time function $\timefunc$, the inverted radial coordinate $R=1/r$, and angular position. 

\begin{definition} 
\label{def:hPrimeInR}
Let $(v,r,\theta,\phi)$ be the ingoing Eddington-Finkelstein coordinates, let $\timefunc$ be the hyperboloidal time function given by \eqref{eq:tGen-def} with $h(r)$ given by \eqref{eq:hdef}, and let $R=1/r$ be the compactified radial coordinate from \eqref{eq:R=1/r}. 

The compactified hyperboloidal coordinate system is $(\timefunc, R,\theta,\phi)$. We shall write 
\index{H2hPrimeInR@$H$}
\begin{align} \label{eq:H(R)}
H(R) = h'(r) .
\end{align} 
\end{definition} 
The domain of outer communication is parametrized by $(\timefunc,\rInv,\omega)\in\Reals\times(0,r_+^{-1})\times\Sphere$. 
For $\epsilon>0$, these coordinates can be extended to $\Reals\times(-\epsilon,r_+^{-1})\times\Sphere$.
From equation \eqref{eq:uMinusTimefunc}, the hyperboloidal time function $\timefunc$ differs from the retarded time $u$ by $c_h +O(R)$. 
Thus, in compactified hyperboloidal coordinates, 
\begin{align}
\Scri^+={}&\Reals\times \{0\}\times \Sphere 
\end{align}
coincides with the standard notion of future null infinity used in section \ref{sec:Kerrgeom}. In fact, the conformal metric $R^2 g_{ab}$ extends analytically not merely to $\Scri^+$ but beyond to $R\in(-\epsilon,r_+^{-1})$.

\begin{definition}
For $\timefunc_1<\timefunc_2$, in compactified hyperboloidal coordinates, let 
\index{I2IScriPlus@$\Scri^+$}
\index{I2IScritau@$\Scritau$}
\begin{align}
\Scritau={}& [\timefunc_1,\timefunc_2]\times\{0\}\times \Sphere .  
\end{align}
\end{definition}

\begin{remark} 
The angular coordinates in the compactified hyperboloidal coordinate system are those of the ingoing Eddington-Finkelstein coordinates.
\end{remark}

\begin{lemma} 
In the Znajek tetrad and the compactified hyperboloidal coordinates $(\timefunc,R,\theta,\phi)$, we have 
\begin{subequations}
\label{eq:tRthetaphi}
\begin{align}
\YOp\varphi={}&H \partial_{\timefunc} \varphi
 + R^2 \partial_{R} \varphi,
 \label{eq:YOptRthetaphi}\\
\VOp\varphi={}&\bigl(1 -  \frac{H R^2 \Delta }{2(1 + a^2 R^2)}\bigr) \partial_{\timefunc} \varphi
 -  \frac{R^4 \Delta \partial_{R} \varphi}{2(1 + a^2 R^2)}
 + \frac{a R^2 \partial_{\phi} \varphi}{1 + a^2 R^2}, \\
\partial_{R} \varphi ={}&\frac{2 a \LetaOp\varphi}{R^2 \Delta}
 -  \frac{2 (1 + a^2 R^2) \VOp\varphi}{R^4 \Delta}
 + \frac{\LxiOp{}\varphi \bigl(2 + 2 a^2 R^2 -  H R^2 \Delta \bigr)}{R^4 \Delta}.
 \label{eq:ddRasVOp}
\end{align}
\end{subequations}
The operators $\hedt, \hedtp, \LxiOp, 
\LetaOp, \TMESOp_s$ take the form given in \eqref{eq:ops-Znaj-IEF}. 
\end{lemma} 
\begin{lemma} In the Znajek tetrad and the compactified hyperboloidal coordinate system $(\timefunc, R, \theta, \phi)$, the operator $\hROp_s$ from definition~\ref{def:TMEOperators} takes the form
\begin{align}\label{eq:ROperatorsThroughScri}
\hROp_s(\varphi)={}&\frac{H \bigl(2 + 2 a^2 R^2 -  H R^2 \Delta \bigr) \partial_{\timefunc} \partial_{\timefunc} \varphi}{R^2}
 + 2 \bigl(1 + a^2 R^2 -  H R^2 \Delta \bigr) \partial_{\timefunc} \partial_{R} \varphi
  + 2 a H \partial_{\timefunc} \partial_{\phi} \varphi\nonumber\\
& -  R^4 \Delta \partial_{R} \partial_{R} \varphi
 + 2 a R^2 \partial_{R} \partial_{\phi} \varphi 
  -  \frac{2 R \bigl((1 + a^2 R^2)^2 -  M R (3 + a^2 R^2)\bigr) \partial_{R} \varphi}{1 + a^2 R^2}
 + \frac{2 a R \partial_{\phi} \varphi}{1 + a^2 R^2}\nonumber\\
& -  \bigl(2 M H (1 - \frac{2}{1 + a^2 R^2}) +  R^2 \Delta \partial_R H\bigr)\partial_{\timefunc} \varphi
 + \frac{R \bigl(a^2 R + a^4 R^3 + M (2 - 4 a^2 R^2)\bigr) \varphi}{(1 + a^2 R^2)^2} .
\end{align}
\end{lemma}

\section{The linearized Einstein equation}
\label{sec:linearizedEinstein}
In this section, we introduce the key variables that we use to study the linearized Einstein equation, review the outgoing radiation gauge  for the linearized Einstein equation, derive a hierarchy of equations used to reconstruct the linearized metric components from the Teukolsky variable, and present convenient forms of the Teukolsky equation and the Teukolsky-Starobinsky identities.

\subsection{First-order form of the linearized Einstein equations} 
\label{sec:connectcomp}

In this subsection, we first recall spinor fields arising in the study of the linearized Einstein equation, then present the spinorial equations for these quantities, and conclude by introducing the GHP scalars components of these spinor fields. This is a precursor to the later subsections, in which we will introduce equations for the GHP scalars. 

\index{D3deltag@$\varop g_{ab}$}
\index{G2GFourIndex@$G_{ABA'B'}$}
\index{G2Gslash@$\slashed{G}$}
We begin by introducing the spinor fields arising in the study of the linearized Einstein equation following \cite{BaeVal15}. 
Let $\delta g_{ab}$ be a solution of the linearized Einstein equations on $(\mathcal{M}, g_{ab})$. Define $\slashed{G}$ and $G_{ABA'B'}$ to be the trace and trace-free parts of $\delta g_{ab}$ respectively. Define the trace and trace-free parts of the linearized connection, $\slashed{\Qop}{}_{CA'}$ and $\Qop{}_{ABCA'}$, to be given by covariant derivatives of $G_{ABA'B'}$ and $\slashed{G}$ by\footnote{We follow the convention \cite{BaeVal15} of representing these by the early Greek letter qoppa $\Qop$.}
\index{Q4Qopslash@$\slashed{\Qop}{}_{AA'}$}
\index{Q4Qop@$\Qop{}_{ABCA'}$}
\begin{align}
\slashed{\Qop}{}_{CA'}={}&\tfrac{1}{4} \nabla^{AB'}G_{CAA'B'}
 -  \tfrac{3}{16} \nabla_{CA'}\slashed{G},&
\Qop{}_{ABCA'}={}&- \tfrac{1}{2} \nabla_{(A}{}^{B'}G_{BC)A'B'}.
\label{eq:QoppaToGCov}
\end{align}
The quantity $\Qop{}_{ABCA'}$ introduced in \eqref{eq:QoppaToGCov} is the symmetrized part of the spinor $\Qop{}_{AA'BC}$ used in \cite{BaeVal15}. 
\index{T3thetaPsiABCD@$\vartheta\Psi_{ABCD}$}
Define $\vartheta\Psi_{ABCD}$ to be the covariant linearized Weyl spinor in the sense of \cite{BaeVal15}. Recall that $\Psi_{ABCD}$ denotes the unperturbed Weyl curvature of the background Kerr metric. 

We now present the equations for these spinor variables. Following \cite{BaeVal15}, one finds that 
\begin{subequations}
\label{eq:spinorEinsteinEquationEtc}
\begin{align}
\nabla^{CA'}\Qop{}_{ABCA'}={}&- \tfrac{4}{3} \nabla_{(A}{}^{A'}\slashed{\Qop}{}_{B)A'},
\label{eq:DivQop31Cov}\\
\nabla^{AA'}\slashed{\Qop}{}_{AA'}={}&0,
\label{eq:DivQop11Cov}\\
\nabla^{C(A'}\Qop{}_{ABC}{}^{B')}={}&\tfrac{1}{2} G^{CDA'B'} \Psi_{ABCD}
 + \tfrac{2}{3} \nabla_{(A}{}^{(A'}\slashed{\Qop}{}_{B)}{}^{B')},
\label{eq:CurlDgQop31Cov}\\
\nabla_{(A}{}^{A'}\Qop{}_{BCD)A'}={}&- \tfrac{1}{4} \Psi_{ABCD}\slashed{G}
 -  \vartheta \Psi_{ABCD},
\label{eq:CurlQop31Cov}\\
\nabla^{D}{}_{A'}\vartheta \Psi_{ABCD}={}&2 \Psi_{ABCD} \slashed{\Qop}{}^{D}{}_{A'}
 + \tfrac{1}{2}  (\nabla_{FB'}\Psi_{ABCD})G^{DF}{}_{A'}{}^{B'}
 + 3 \Psi_{(AB}{}^{DF}\Qop{}_{C)DFA'} 
\label{eq:VacuumLinBianchiCov}
\end{align}
\end{subequations}
follow respectively from a commutator relation, the trace and tracefree parts of the linearized vacuum Einstein equation, the vacuum Ricci relations, and the vacuum Bianchi identity. The system \eqref{eq:VacuumLinBianchiCov} is clearly a first-order system, but, since no gauge has been imposed at this stage, it should not be expected to be a well-posed system. 

We now present the GHP scalar components of these spinor fields. For the linearized metric, we use the compactified index notation\footnote{Recall that $\varphi_{ii'}$ denotes the dyad component of a symmetric spinor $\varphi_{AB\cdots D A'B'\cdots D'}$ defined by contracting $i$ times with $\iota^A$ and $i'$ times with $\iota^{A'}$ as explained in section~\ref{sec:notandconv}. } for the trace-free part $G_{ab}$ of the linearized metric, i.e.
\begin{subequations} 
\label{eq:G-compos}  
\begin{align}  
G_{00'} ={}& G_{ab} \NPl^{a} \NPl^{b}, \quad G_{10'} = G_{ab} \NPl^{a} \NPmbar^{b}, \quad G_{11'} = G_{ab} \NPl^{a} \NPn^{b}, \\ 
G_{20'} ={}& G_{ab} \NPmbar^a \NPmbar^b, \quad 
G_{21'} = G_{ab} \NPn^a \NPmbar^b, \quad 
G_{22'} = G_{ab} \NPn^a \NPn^b
\end{align} 
\end{subequations} 
and their complex conjugates. We have that $G_{00'}, G_{11'}, G_{22'}$ are real, while the remaining components are complex. Define the following linear combinations of the components of the linearized connection, 
\index{B3betatilde@$\tilde{\beta}, \tilde{\beta}'$}\index{E3epsilontilde@$\tilde{\epsilon}, \tilde{\epsilon}'$}\index{K3kappatilde@$\tilde{\kappa}, \tilde{\kappa}'$}\index{R3rhotilde@$\tilde{\rho}, \tilde{\rho}'$}\index{S3sigmatilde@$\tilde{\sigma}, \tilde{\sigma}'$}\index{T3tautilde@$\tilde{\tau}, \tilde{\tau}'$}
\begin{subequations} \label{eq:tildespincoeff}
\begin{align}
\tilde{\beta}={}&- \tfrac{1}{3} \slashed{\Qop}{}_{01'}
 + \Qop{}_{11'},&
\tilde{\beta}'={}&- \tfrac{1}{3} \slashed{\Qop}{}_{10'}
 -  \Qop{}_{20'},&
\tilde{\epsilon}={}&- \tfrac{1}{3} \slashed{\Qop}{}_{00'}
 + \Qop{}_{10'},&
\tilde{\epsilon}'={}&- \tfrac{1}{3} \slashed{\Qop}{}_{11'}
 -  \Qop{}_{21'},\\
\tilde{\kappa}={}&\Qop{}_{00'},&
\tilde{\kappa}'={}&- \Qop{}_{31'},&
\tilde{\rho}={}&\tfrac{2}{3} \slashed{\Qop}{}_{00'}
 + \Qop{}_{10'},&
\tilde{\rho}'={}&\tfrac{2}{3} \slashed{\Qop}{}_{11'}
 -  \Qop{}_{21'},\\
\tilde{\sigma}={}&\Qop{}_{01'},&
\tilde{\sigma}'={}&- \Qop{}_{30'},&
\tilde{\tau}={}&\tfrac{2}{3} \slashed{\Qop}{}_{01'}
 + \Qop{}_{11'},&
\tilde{\tau}'={}&\tfrac{2}{3} \slashed{\Qop}{}_{10'}
 -  \Qop{}_{20'}.
\end{align}
\end{subequations}
The notation used here is inspired by the notation for the GHP spin coefficients in \eqref{eq:spincoeff-def}-\eqref{eq:eps-beta-def}.
We use the tilde $\tilde{}$ accent here to denote these components of the linearized connection. 
Note that in contrast to the linearized spin coefficients often used in applications of the NP  formalism \cite{1974RSPSA.341...49S}, the scalars defined in \eqref{eq:tildespincoeff} are components of the linearized connection with respect to a background tetrad. In particular, this avoids introducing a linearly perturbed tetrad and the associated additional degrees of freedom. 
In the terminology of \cite{BaeVal15}, this is refered to as invariance under linearized frame rotations. 
The compactified index notation can be applied to the linearized Weyl curvature, but the only components that are relevant in our analysis are \index{T3thetaPsi0@$\vartheta \Psi_0$}
\index{T3thetaPsi4@$\vartheta \Psi_4$}
\begin{align} 
\label{eq:extreme-lin}
 \vartheta \Psi_0 = - \delta C_{abcd} l^a m^b l^c m^d ,  \quad 
 \vartheta \Psi_4 = - \delta C_{abcd} n^a \bar m^b n^c \bar m^d . 
\end{align} 
Appendix \ref{sec:fieldeq} presents the first-order system \eqref{eq:spinorEinsteinEquationEtc} in terms of these GHP scalars.

\subsection{Outgoing radiation gauge}
\label{sec:ORG}

Here, we present the outgoing radiation gauge, which is the linearized gauge that we use throughout the rest of this paper. We begin by recalling the definition of this gauge from \cite{Chrzanowski}. We then recall a result \cite{2007CQGra..24.2367P} showing that this linearized gauge condition can be imposed. We conclude with a preliminary result that, as a consequence of our choice of gauge, several of the GHP scalars from section \ref{sec:connectcomp} vanish. The non-vanishing linearized metric, connection, and curvature components in the outgoing radiation gauge are illustrated with their $\{p,q\}$ type in figure \ref{fig:GHP}. In terms of these non-vanishing components, there are statements of the Einstein equation and various other relations, which we relegate to appendix~\ref{sec:LinGraSysteminORG}, because of their length and because they follow by direct computation. 

\begin{definition}
\label{def:ORG}
Let $\delta g_{ab}$ be a linearized metric on $(\mathcal{M}, g_{ab})$. 
We say that $\delta g_{ab}$ satisfies the $\delta g \cdot n$ condition if 
\begin{align}
\varop g_{ab} \NPn^b={}&0 , 
\label{eq:lhGauge}
\end{align}
and the trace-free condition if 
\begin{align}
g^{ab}\varop g_{ab}={}& 0 . 
\label{eq:traceFree}
\end{align}
If both \eqref{eq:lhGauge} and \eqref{eq:traceFree} hold, then $\delta g_{ab}$ is said to be in outgoing radiation gauge (ORG). Replacing $\NPn^a$ by $\NPl^a$ yields the ingoing radiation gauge (IRG) condition. 
\end{definition}

\begin{lemma}[Price, Shankar and Whiting \cite{2007CQGra..24.2367P}] 
\label{lem:ORG} 
Let $\delta g_{ab}$ be a solution of the linearized vacuum Einstein equation on $(\mathcal{M}, g_{ab})$. 
There is a vector field $\GenVec^a$ such that the gauge transformed metric  
\begin{align} \label{eq:gaugetrafo}
\varop g_{ab} - 2\nabla_{(a} \GenVec_{b)} 
\end{align} 
is in ORG. 
\end{lemma}

\begin{remark} 
\begin{enumerate} 
\item The $\delta g \cdot n$ gauge condition \eqref{eq:lhGauge}, which consists of four conditions, can be imposed for a general linearized metric on a background vacuum spacetime with repeated principal null direction $n^a$, by sequentially solving a system of four scalar equations, cf. \cite[Eq. (15)]{2007CQGra..24.2367P}. 
The analogous statement is valid for the $\delta g \cdot l$ condition. 
This is in contrast to the ORG or IRG conditions, which contain five conditions, and which can be imposed only for linearized metrics on algebraically special background spacetimes, provided that the linearized Einstein tensor satisfies additional conditions. In \cite{2007CQGra..24.2367P}, it is shown to be possible to impose IRG for solutions of the linearized Einstein equations $\delta E_{ab} = 8\pi \delta T_{ab}$ on a Petrov type II or type D background with repeated principal vector $\NPl^a$, provided $\delta T_{ab} \NPl^a \NPl^b = 0$. Analogously the ORG condition can be imposed provided $\delta T_{ab} \NPn^a \NPn^b = 0$. Here we shall be interested only in the case of  solutions of the linearized vacuum Einstein equations $\delta E_{ab}=0$ on the Kerr spacetime, which is Petrov type D.
\item Imposing the gauge condition does not determine the vector field $\GenVec^a$ uniquely. In particular, there remains residual gauge degrees of freedom in $\GenVec^a$, subject to constraint equations. The vector field $\GenVec^a$ can determined uniquely along the flow lines of $\NPn^a$ by specifying its initial values on a hypersurface.
\item The gauge vector field $\GenVec^a$ plays no explicit role in this paper.   
\end{enumerate} 
\end{remark}

\begin{lemma}
Let $\varop g_{ab}$ be a solution to the vacuum linearized Einstein equations on $(\mathcal{M},  g_{ab})$, in ORG. Then, in the notation introduced in section \ref{sec:connectcomp}, the following holds. 
\begin{enumerate} 
\item 
\begin{subequations} 
\label{eq:ORGcondGHP}
\begin{align}
\slashed{G}={}&0,&
G_{11'}={}&0,&
G_{12'}={}&0,& \\ 
G_{21'}={}&0,&
G_{22'}={}&0 ,
\end{align}
\end{subequations}
and
\begin{align}
\tilde{\epsilon}'={}&0,&
\tilde{\kappa}'={}&0,&
\tilde{\rho}'={}&0.
\label{eq:ORGlinconnectioncoeffConseq}
\end{align}
\item 
The only non-vanishing components of the metric are 
\begin{equation}\label{eq:dg-G} 
G_{00'} = \delta g_{ab} l^a l^b , \quad  G_{10'} = \delta g_{ab}l^a \bar m^b  , \quad G_{20'} = \delta g_{ab}  \bar m^a \bar m^b .
\end{equation} 
\end{enumerate} 
\end{lemma}
\begin{proof}
Splitting $\delta g_{ab}$ into the trace and trace-free parts and expanding into components yields
\begin{align}
\delta g_{ab}={}&G_{22'} l_{a} l_{b}
 + G_{20'} m_{a} m_{b}
 + G_{02'} \bar{m}_{a} \bar{m}_{b}
 + G_{00'} n_{a} n_{b}
 - 2 G_{21'} l_{(a}m_{b)}
 - 2 G_{12'} l_{(a}\bar{m}_{b)}\nonumber\\
& + (\tfrac{1}{2} \slashed{G}_{} + 2 G_{11'}) l_{(a}n_{b)}
 + (- \tfrac{1}{2} \slashed{G}_{} + 2 G_{11'}) m_{(a}\bar{m}_{b)}
 - 2 G_{10'} m_{(a}n_{b)}
 - 2 G_{01'} \bar{m}_{(a}n_{b)}.
\end{align}
Contraction with $\NPn^a$ and $g^{ab}$ yield
\begin{align}
n^{b} \delta g_{ab}={}&G_{22'} l_{a}
 + (\tfrac{1}{4} \slashed{G}_{}
 + G_{11'}) n_{a}
 -  G_{21'} m_{a}
 -  G_{12'} \bar{m}_{a},
&
g^{ab} \delta g_{ab}={}&\slashed{G}_{}.
\end{align}
Both vanish due to the gauge condition, so \eqref{eq:ORGcondGHP} follows. The relations \eqref{eq:ORGlinconnectioncoeffConseq} then follow from \eqref{eq:Aepsilonp}, \eqref{eq:Akappap} and \eqref{eq:Arhop}. Noting that $G_{01'}$ and $G_{02'}$ are complex conjugates of $G_{10'}$ and $G_{20'}$ we see that $\delta g_{ab}$ is completely specified by $G_{00'}$, $G_{10'}$ and $G_{20'}$.
\end{proof}

\begin{figure}[t]
\centering
\begin{tikzpicture}[xscale=1,yscale=1]
\filldraw[draw=white,fill=yellow!10!white]
(-4.9,-3.9) rectangle (4.9,3.9);

\draw[step=1cm,gray, dotted, thin] (-4.9,-3.9) grid (4.9,3.9);
\draw[very thin, gray, ->] (-4.9,0) -- (4.9,0);
\draw[very thin,gray, ->] (0,-3.9) -- (0,3.9);
\draw (4.9,0) node[anchor=west]{$p$};
\draw (0,3.9) node[anchor=south]{$q$};\draw[white,fill] (1,-1) ellipse(10pt and 6pt) ;
\draw[blue] (1,-1)  node{$\tilde{\beta},\tilde{\tau}$};
\draw[white,fill] (-1,1) ellipse(13pt and 6pt) ;
\draw[blue] (-1,1)  node{$\tilde{\beta}',\tilde{\tau}'$};
\draw[white,fill] (1,1) ellipse(9pt and 6pt) ;
\draw[blue] (1,1)  node{$\tilde{\epsilon},\tilde{\rho}$};
\draw[white,fill] (3,1) ellipse(5pt and 6pt) ;
\draw[blue] (3,1)  node{$\tilde{\kappa}$};
\draw[white,fill] (3,-1) ellipse(5pt and 6pt) ;
\draw[blue] (3,-1)  node{$\tilde{\sigma}$};
\draw[white,fill] (-3,1) ellipse(6pt and 6pt) ;
\draw[blue] (-3,1)  node{$\tilde{\sigma}'$};
\draw[white,fill] (2,2) ellipse(12pt and 6pt) ;
\draw[blue] (2,2)  node{$G_{00'}$};
\draw[white,fill] (2,0) ellipse(23pt and 6pt) ;
\draw[blue] (2,0)  node{$G_{01'},\vartheta \Psi_{1}$};
\draw[white,fill] (2,-2) ellipse(12pt and 6pt) ;
\draw[blue] (2,-2)  node{$G_{02'}$};
\draw[white,fill] (0,2) ellipse(12pt and 6pt) ;
\draw[blue] (0,2)  node{$G_{10'}$};
\draw[white,fill] (-2,2) ellipse(12pt and 6pt) ;
\draw[blue] (-2,2)  node{$G_{20'}$};
\draw[white,fill] (4,0) ellipse(11pt and 6pt) ;
\draw[blue] (4,0)  node{$\vartheta \Psi_{0}$};
\draw[white,fill] (0,0) ellipse(11pt and 6pt) ;
\draw[blue] (0,0)  node{$\vartheta \Psi_{2}$};
\draw[white,fill] (-2,0) ellipse(11pt and 6pt) ;
\draw[blue] (-2,0)  node{$\vartheta \Psi_{3}$};
\draw[white,fill] (-4,0) ellipse(11pt and 6pt) ;
\draw[blue] (-4,0)  node{$\vartheta \Psi_{4}$};\filldraw[draw=white,fill=blue!4!white]
(2.8,1.8) rectangle (5.2,4.2);

\draw[very thin,gray, ->,shorten >=7pt,shorten <=2pt] (4,3) -- (5,4) node {$\tho$};
\draw[very thin,gray, ->,shorten >=7pt,shorten <=2pt] (4,3) -- (3,2) node{$\thop$};
\draw[very thin,gray, ->,shorten >=7pt,shorten <=2pt] (4,3) -- (5,2) node {$\edt$} ;
\draw[very thin,gray, ->,shorten >=7pt,shorten <=2pt] (4,3) -- (3,4) node {$\edtp$} ;

\end{tikzpicture}
\caption{GHP weights of the non-vanishing components in ORG.}
\label{fig:GHP}
\end{figure}

\subsection{Equations of linearized gravity in the boost-weight zero formalism} 
\label{sec:transporteq}
Here, we derive the system \eqref{eq:TransEqsKerr} that we use to obtain decay of the metric coefficients $G_{i0'}$ from the decay of the extreme curvature components $\vartheta\Psi_0$ and especially $\vartheta\Psi_4$. 

Our approach is driven by several key ideas. 
\begin{enumerate*}[label=(\roman*)]
\item It is possible to reconstruct the metric from the extreme curvature components in the ORG. However, 
we have a large amount of freedom in choosing a system of evolution equations to construct the metric. 
\item Since our goal is to construct the metric coefficients, it is not necessary to construct all the connection coefficients. 
\item We wish to have an ordered hierarchy in this system, so that we can construct each variable from either the curvature or other variables that have appeared previously in the ordering; the choice of variables and their ordering is illustrated in figure \ref{fig:EstimatesDiagram}. 
\item We wish to construct the variables through evolution equations. In particular, we wish for these evolution equations to be transport equations along the null direction $\NPn$, although in some cases the right-hand side contains derivatives of variables that appeared previously in the ordering. 
\item \label{point:v} We wish to work not merely with GHP scalars but with spin-weighted scalars; for this reason, we use deboosted variables. 
\end{enumerate*}

Our approach is also influenced by a number of other, more technical ideas. 
\begin{enumerate*}[label=(\roman*), resume]
\item \label{point:vi} The choice of rescaling of the extreme curvature variables, also called the Teukolsky variables, is not driven by our goal of deriving system \eqref{eq:TransEqsKerr}, but instead by the goal of achieving a convenient form of the Teukolsky equation, as explained in section \ref{sec:TeukolskyEquations}. 
\item \label{point:vii} We wish to further rescale the remaining variables so that each is governed by a linear equation in which the transport operator in the homogeneous part is the deboosted operator $\YOp$. Equation \eqref{eq:linearizedEinsteinEquationInORG} gives transport equations for the relevant variables, but typically these also include lower-order coefficients, such as $\rho'$. In definition \ref{def:BoostWeightZeroQuantities}, we have rescaled the variables to eliminate these lower-order terms in the system \eqref{eq:TransEqsKerr}. 
\item \label{point:viii} We need to be able to derive decay estimates for our variables. For most of our variables, it is sufficient to deboost and rescale to eliminate the lower-order terms, as in points \ref{point:v} and \ref{point:vii}, but, we have found that for $\tilde{\tau}'$ and $\tilde{\beta}'$, we needed to introduce certain linear combinations to cancel terms we were not otherwise able to control. 
\end{enumerate*}

In passing, we briefly comment on the notion of radiation field. A radiation field is a variable that has been rescaled so that it is neither divergent nor (generically) vanishing at $\Scri^+$ (see e.g.{} \cite{MR944085}). While our rescaling  for the extreme curvature (Teukolsky) variables, $\psibase[\pm2]$, was chosen to obtain a convenient form for the resulting Teukolsky equations in section \ref{sec:TeukolskyEquations}, it so happens that $\psibase[\pm2]$ are radiation fields. Our choice of rescaling for the remaining variables was driven by points \ref{point:vii}-\ref{point:viii} above, and it so happens that $\hat\sigma'$, $\widehat{G}_{2}$, and $\widehat{G}_{1}$ in definition \ref{def:BoostWeightZeroQuantities} are radiation fields, while $\hat{\tau}'$, $\hat{\beta}'$, and $\widehat{G}_0$ are not. Within our analysis, the notion of radiation field only arises as the vanishing of the index $\Em{\varphi}$ in definition \ref{def:alpha1Values} and lemma \ref{lem:allquantitiesextestimates}. 

We now define the deboosted and rescaled variables used to derive system \eqref{eq:TransEqsKerr}. These all are denoted with a hat accent. 
Following \cite{Chandrasekhar}, at this stage, we switch from indexing the curvature components in the compactified GHP index convention, and instead we index by spin weight. For the deboosted and rescaled metric components, we drop the redundant $0'$ index. 

\begin{definition}
\label{def:psibase}
Let $\varop g_{ab}$ be a solution to the linearized vacuum Einstein equation on the Kerr exterior $(\mathcal{M}, g_{ab})$ and let $\vartheta \Psi_0, \vartheta \Psi_4$ be the components of the linearized Weyl spinor $\vartheta \Psi_{ABCD}$ of boost- and spin-weights $(2,2), (-2,-2)$. Define
\index{P3psibase-2@$\psibase[-2]$}
\index{P3psibase+2@$\psibase[+2]$}
\begin{subequations}
\begin{align}
\psibase[-2]={}& \tfrac{1}{2}\sqrt{a^2 + r^2}\uplambda^{2} \vartheta \Psi_{4},\\
\psibase[+2]={}& \tfrac{1}{2}\sqrt{a^2 + r^2}(3\kappa_{1}{})^4\uplambda^{-2}\vartheta \Psi_{0} ,
\end{align}
where $\uplambda$ is given by definition~\ref{def:uplambda}.
\end{subequations}
\end{definition}

\begin{remark} 
The Weyl scalars $\vartheta \Psi_0, \vartheta \Psi_4$ are given in terms of the linearized Weyl tensor by equation \eqref{eq:extreme-lin}. The fields $\psibase[-2]$ and $\psibase[+2]$ have boost-weight zero and spin-weights $-2$ and $+2$, respectively. 
\end{remark} 

\index{S3sigmaprimhat@$\hat{\sigma}'$}\index{G2G0hatG1hatG2hat@$\widehat{G}_0, \widehat{G}_1, \widehat{G}_2$}\index{T3tauprimhat@$\hat{\tau}'$}\index{B3betaprimhat@$\hat{\beta}'$}
\begin{definition}
\label{def:BoostWeightZeroQuantities}
Define the spin-weighted scalars 
\begin{subequations}
\label{eq:BoostWeightZeroQuantities}
\begin{align}
\hat{\sigma}'={}&\frac{\tilde{\sigma}'}{\bar{\rho}'}, & 
\widehat{G}_2={}&G_{20'} \overline{\kappa}_{1'}{}, \\
\hat{\tau}'={}&
 \bigl (1 + \frac{\kappa_{1}{}}{2 \overline{\kappa}_{1'}{}}\bigr) \tilde{\tau}'
 - \tilde{\beta}',&
\widehat{G}_1={}&\frac{G_{10'} \kappa_{1}{}^3 \overline{\kappa}_{1'}{} \rho'}{r}=- \frac{G_{10'} \kappa_{1}{} \Sigma}{27 \sqrt{2} \uplambda r},
\label{eq:BoostWeightZeroQuantitiesG1}\\
\hat{\beta}'={}&\overline{\kappa}_{1'}{}(\tilde{\beta}'
 -  \tfrac{1}{2} G_{10'} \bar{\rho}'
 + \tfrac{1}{2} G_{20'} \bar{\tau}'
 -  \tilde{\tau}'),& 
\widehat{G}_0={}&\frac{G_{00'} \kappa_{1}{}^3 \overline{\kappa}_{1'}{} \rho'^2}{r}=\frac{G_{00'} \Sigma}{162 \uplambda^2 r}.
\end{align}
\end{subequations}
\end{definition}

The quantities $\hat{\sigma}', \widehat{G}_2, \hat{\tau}',\widehat{G}_1, \hat{\beta}', \widehat{G}_0$ have spin-weights $-2, -2, -1, -1, -1, 0$, respectively. 
The definition of the quantities $\widehat{G}_0$ and $\widehat{G}_1$ has the consequence that the linearized mass $\delta M$ and angular momentum per unit mass $\delta a$ appear as constants of integration in equations \eqref{eq:TransEqG00} and \eqref{eq:TransEqG10}, respectively. In section~\ref{sec:exteriorest} we show that our assumptions imply that these constant vanish. 
The choice of $\hat{\tau}'$ happens to be such that it vanishes for a linearized mass or angular momentum perturbation in ORG. See appendix~\ref{sec:linpara}.

\begin{lemma}
\label{lem:TransportSystem} 
Given a solution to the linearized vacuum Einstein equation in ORG on $(\mathcal{M}, g_{ab})$, let the quantities $\hat{\sigma}', \widehat{G}_2, \hat{\tau}',\widehat{G}_1, \hat{\beta}', \widehat{G}_0$ be as in definition \ref{def:BoostWeightZeroQuantities}, and let $\psibase[-2]$ be as in definition \ref{def:psibase}. Then we have 
\begin{subequations}
\label{eq:TransEqsKerr}
\begin{align}
\YOp(\hat{\sigma}')={}&- \frac{12 \bar{\kappa}_{1'}{} \psibase[-2]}{\sqrt{r^2 + a^2}},\label{eq:TransEqTildeSigma}\\
\YOp(\widehat{G}_2)={}&- \tfrac{2}{3} \hat{\sigma}',\label{eq:TransEqG20}\\
\YOp(\hat{\tau}')={}&- \frac{\kappa_{1}{} (\edt - 2 \tau + 2 \bar{\tau}')\hat{\sigma}'}{6 \overline{\kappa}_{1'}{}^2},\label{eq:TransEqHatTau}\\
\YOp(\widehat{G}_1)={}&\frac{2 \kappa_{1}{}^2 \overline{\kappa}_{1'}{}^2 \hat{\tau}'}{r^2}
 + \frac{\kappa_{1}{}^2 \overline{\kappa}_{1'}{} (\edt -  \tau + \bar{\tau}')\widehat{G}_2}{2 r^2},\label{eq:TransEqG10}\\
\YOp(\hat{\beta}')={}&\frac{r \widehat{G}_1}{6 \kappa_{1}{}^2 \overline{\kappa}_{1'}{}^2}
 + \frac{\kappa_{1}{} \tau \widehat{G}_2}{6 \overline{\kappa}_{1'}{}^2},\label{eq:TransEqHatBeta'}\\
\YOp(\widehat{G}_0)={}&-  \frac{(\edt -  \tau)\widehat{G}_1}{3 \kappa_{1}{}}
 - \frac{\tau \widehat{G}_1}{r}
 -  \frac{\bar{\tau}\overline{\widehat{G}_1}}{r}
 + \frac{2 \kappa_{1}{}^2 \overline{\kappa}_{1'}{} (\edt -  \bar{\tau}')\hat{\beta}'}{r^2}
 -  \frac{(\edtp -  \bar{\tau})\overline{\widehat{G}_1}}{3 \overline{\kappa}_{1'}{}}
 + \frac{2 \kappa_{1}{} \overline{\kappa}_{1'}{}^2 (\edtp -  \tau')\overline{\hat{\beta}'}}{r^2}.\label{eq:TransEqG00}
\end{align}
\end{subequations}
\end{lemma}

\begin{figure}[t]
\begin{tikzpicture}[node distance=10mm,>=latex,
StrongEst/.style={
circle,minimum size=10mm, very thick,draw=black!50,
top color=white,bottom color=black!20,
align=center}]]
\node (NodeHypothesis) [StrongEst]{$\psibase[-2]$};
\node (NodeSigmahatprim) [StrongEst, right=of NodeHypothesis]{$ \hat{\sigma}'$};
\node (NodeG2hat) [StrongEst, right=of NodeSigmahatprim]{$\widehat{G}_2$};
\node (NodeTauhatprim) [StrongEst, below=of NodeG2hat]{$\hat{\tau}'$};
\node (NodeG1hat) [StrongEst,  right=of NodeG2hat]{$\widehat{G}_1$};
\node (NodeBetahatprim) [StrongEst, right=of NodeTauhatprim]{$\hat{\beta}'$};
\node (NodeG0hat) [StrongEst, right=of NodeG1hat]{$\widehat{G}_0$};
\path(NodeHypothesis) edge[->]  node[above] {\eqref{eq:TransEqTildeSigma}} (NodeSigmahatprim);
\path(NodeSigmahatprim) edge[->] node[above] {\eqref{eq:TransEqG20}} (NodeG2hat);
\path(NodeSigmahatprim) edge[->] node[above right=-2ex and -2ex] {\rotatebox{-45}{\eqref{eq:TransEqHatTau}}} (NodeTauhatprim);
\path(NodeG2hat) edge[->] node[above] {\eqref{eq:TransEqG10}} (NodeG1hat);
\path(NodeTauhatprim) edge[->, out=45,in=180] (NodeG1hat);
\path(NodeG1hat) edge[->] node[right] {\rotatebox{-90}{\eqref{eq:TransEqHatBeta'}}} (NodeBetahatprim);
\path(NodeG2hat) edge[->,out=-45,in=90] (NodeBetahatprim);
\path(NodeG1hat) edge[->] node[above] {\eqref{eq:TransEqG00}} (NodeG0hat);
\path(NodeBetahatprim) edge[->,out=45,in=180] (NodeG0hat);
\end{tikzpicture}
\caption{Structure of transport equations.}
\label{fig:EstimatesDiagram}
\end{figure}

\begin{proof}
Throughout the proof 
we will use the relations
\begin{subequations}\label{eq:GHPkappa}
\begin{align}
\thop \kappa_{1}{}={}&- \kappa_{1}{} \rho',&
\thop \overline{\kappa}_{1'}{}={}&- \overline{\kappa}_{1'}{} \bar{\rho}',&
\thop \rho'={}&\rho'^2,&
\thop \bar{\rho}'={}&\bar{\rho}'^2,\\
\edt \kappa_{1}{}={}&- \kappa_{1}{} \tau,&
\edt \overline{\kappa}_{1'}{}={}&\kappa_{1}{} \tau,&
\edt \rho'={}&2 \rho' \tau,&
\edt \bar{\rho}'={}&\bar{\rho}' \tau
 + \bar{\rho}' \bar{\tau}',\\
\overline{\kappa}_{1'}{} \bar{\rho}={}&\kappa_{1}{} \rho,&
\overline{\kappa}_{1'}{} \bar{\rho}'={}&\kappa_{1}{} \rho',&
\overline{\kappa}_{1'}{} \bar{\tau}'={}&- \kappa_{1}{} \tau,&
\overline{\kappa}_{1'}{} \bar{\tau}={}&- \kappa_{1}{} \tau'.
\end{align}
\end{subequations}
For some calculations it might also be worth to notice
\begin{subequations}
\begin{align}
\edt \tau ={}&\tau^2,&
\edtp \tau ={}&\tfrac{1}{2} \Psi_{2}
 -  \frac{\bar\Psi_{2} \bar{\kappa}_{1'}{}}{2 \kappa_{1}{}}
 + \rho \rho '
 -  \frac{\kappa_{1}{} \rho \rho '}{\bar{\kappa}_{1'}{}}
 + \tau \tau '
 =\frac{(a^2 + r^2) (\kappa_{1}{} -  \bar{\kappa}_{1'}{})}{162 \kappa_{1}{}^3 \bar{\kappa}_{1'}{}^2}
 + \tau \tau ',\\
\edtp \tau '={}&\tau '^2,&
\edt \tau '={}&\tfrac{1}{2} \Psi_{2}
 -  \frac{\bar\Psi_{2} \bar{\kappa}_{1'}{}}{2 \kappa_{1}{}}
 + \rho \rho '
 -  \frac{\kappa_{1}{} \rho \rho '}{\bar{\kappa}_{1'}{}}
 + \tau \tau '
 =\frac{(a^2 + r^2) (\kappa_{1}{} -  \bar{\kappa}_{1'}{})}{162 \kappa_{1}{}^3 \bar{\kappa}_{1'}{}^2}
 + \tau \tau '.
\end{align}
\end{subequations}
The ORG condition reduces equations \eqref{eq:APsi4} and \eqref{eq:Asigmap} to the transport equations 
\begin{align}
(\thop -  \bar{\rho}')\tilde{\sigma}'={}&\vartheta \Psi_{4},& 
(\thop -  \bar{\rho}')G_{20'}={}&2 \tilde{\sigma}'.
\end{align}
The choices made in definition~\ref{def:BoostWeightZeroQuantities} are explained below.
First we note that $\tilde{\sigma}'$ has boost weight, which is eliminated by defining $\hat{\sigma}'=\frac{\tilde{\sigma}'}{\bar{\rho}'}$. Similarly the definition $\widehat{G}_2= \overline{\kappa}_{1'}{}G_{20'}$ compensates for the lower-order term in the left-hand side. We then re-express the transport equations in terms of the spin-weighted operator $Y$ defined in \eqref{eq:YOpGHPDef}, from which we get \eqref{eq:TransEqTildeSigma} and \eqref{eq:TransEqG20}.

Under the ORG conditions, the equations \eqref{eq:AEinsteinJ} and \eqref{eq:ACommutatorC} will yield expressions \eqref{eq:linearizedEinsteinEquationInORGtauprime} and \eqref{eq:linearizedEinsteinEquationInORGbetaprime} for $\thop \tilde{\tau}'$ and $\thop \tilde{\beta}'$. 
However, these transport equations are coupled, so we need to change variables. We found that the variable $\hat{\tau}'$ in definition \ref{def:BoostWeightZeroQuantities} satisfies a good transport equation 
\begin{align}
\thop \hat{\tau}'={}&\frac{\kappa_{1}{} (\edt - 3 \tau + \bar{\tau}')\tilde{\sigma}'}{2 \overline{\kappa}_{1'}{}},
\end{align}
which can be written as \eqref{eq:TransEqHatTau}. This was derived just from the definition of $\hat{\tau}'$, \eqref{eq:linearizedEinsteinEquationInORGtauprime} and \eqref{eq:linearizedEinsteinEquationInORGbetaprime}.
Using equations \eqref{eq:Abetap} and \eqref{eq:Ataup} to express $\hat{\tau}'$ in terms of $G_{10'}$ and $G_{20'}$ yields
\begin{align}
\hat{\tau}'={}& -\frac{r \thop G_{10'}}{6 \overline{\kappa}_{1'}{}} - \frac{(\kappa_{1}{}^2 + \kappa_{1}{} \overline{\kappa}_{1'}{} + 2 \overline{\kappa}_{1'}{}^2) \rho' G_{10'} }{4 \overline{\kappa}_{1'}{}^2}
 -  \tfrac{1}{4} (\edt -  \tau)G_{20'},
\end{align}
which can be rewritten as a transport equation for $G_{10'}$.
Furthermore, one can rescale $G_{10'}$ to produce a boost-weight zero quantity $\widehat{G}_1$ in \eqref{eq:BoostWeightZeroQuantitiesG1} such that the contribution from the linearized angular momentum in ORG gauge is $r$ independent, cf. \eqref{eq:LinAngMomentumG} and \eqref{eq:LinAngMomentumGhat}. This also eliminates the lower-order terms to yield the transport equation \eqref{eq:TransEqG10}.

The transport equation for $\tilde{\beta}'$ is complicated. $\hat{\beta}' = \bar\kappa_{1'}\overline{\tilde{\beta}}$ satisfies a much simpler equation arising from the complex conjugate of \eqref{eq:AEinsteinC} subject to the ORG conditions \eqref{eq:ORGcondGHP} and \eqref{eq:ORGlinconnectioncoeffConseq}. The rescaling eliminates the lower-order term.
However, $\hat{\beta}'$ can be reexpressed in terms of $\tilde{\beta}'$, and the already controlled quantities $\hat{\tau}'$, $\hat{G}_1$, and $\hat{G}_0$ using \eqref{eq:Areality} and
\begin{align}
\tilde{\tau}={}&- \tfrac{1}{2}  \rho' G_{01'} + \tfrac{1}{2}  \tau' G_{02'},\label{eq:tauORG}
\end{align} 
which follows from \eqref{eq:Atau} and the ORG conditions.
 
Taking a derivative of equation \eqref{eq:tauORG} and using the relations \eqref{eq:Aspincoeff}, \eqref{eq:Areality} and the definitions of $\hat{\tau}'$ and $\hat{\beta}'$ yield 
\begin{align}
\edtp \tilde{\tau}={}&\tfrac{1}{2}  \rho' \bar{\rho}' G_{00'}
 + \rho' \overline{\tilde{\rho}}
 + \tfrac{1}{2}  \bar{\rho}' \tau G_{10'}
 + \frac{4 r  \bar{\tau} \overline{\hat{\beta}'}}{3 \kappa_{1}{}^2}
 + \tfrac{1}{2}  \bar{\rho}' \tau' G_{01'}
 +  \bar{\tau} \tau' G_{02'}
 - 2 \bar{\tau} \overline{\hat{\tau}'}. \label{eq:ethptau}
\end{align}
The relations \eqref{eq:AEinsteinA}, \eqref{eq:AEinsteinH} and \eqref{eq:ACommutatorB} together give
\begin{align}
0={}&- \rho'\tilde{\epsilon}
 -  \bar{\rho}' \tilde{\epsilon}
 + \rho' \tilde{\rho}
 + 2  \tau'\tilde{\beta}
 -  \bar{\tau}' \tilde{\tau}'
 -  \edt \tilde{\beta}'
 + \edt \tilde{\tau}'
 -  \edtp \tilde{\beta}.
\end{align}
This together with the definitions of $\hat{\tau}'$ and $\hat{\beta}'$ and \eqref{eq:ethptau} yields
\begin{align}
0={}&-  \rho'\tilde{\epsilon}
 -  \bar{\rho}'\tilde{\epsilon}
 + \tfrac{1}{2}  \rho' \bar{\rho}' G_{00'}
 + \rho' \tilde{\rho}
 + \bar{\rho}' \tilde{\rho}
 -  \frac{ \kappa_{1}{} \tau \hat{\beta}'}{\overline{\kappa}_{1'}{}^2}
 + \tfrac{1}{2}  \bar{\rho}' \tau G_{10'}
 + \frac{ \tau'\overline{\hat{\beta}'}}{\kappa_{1}{}}
 -  \tfrac{1}{2}  \bar{\rho}' \tau'G_{01'}
 -  \frac{\edt \hat{\beta}'}{\overline{\kappa}_{1'}{}}
 -  \frac{\edtp \overline{\hat{\beta}'}}{\kappa_{1}{}}.
\end{align}
With the help of equations \eqref{eq:Aepsilon} and \eqref{eq:Arho}, this can be rewritten as a transport equation for $G_{00'}$
\begin{align}
\thop G_{00'}={}&- \frac{3 (\kappa_{1}{}^2 + \overline{\kappa}_{1'}{}^2) \rho'G_{00'}}{2 r \overline{\kappa}_{1'}{}}
 -  \frac{3 (\kappa_{1}{}^2 - 3 \kappa_{1}{} \overline{\kappa}_{1'}{} - 2 \overline{\kappa}_{1'}{}^2) \tau G_{10'}}{2 r \overline{\kappa}_{1'}{}}
 -  \frac{6  \kappa_{1}{} \tau \hat{\beta}'}{r \overline{\kappa}_{1'}{} \rho'}\nonumber\\
& -  \frac{3 (2 \kappa_{1}{}^2 + 3 \kappa_{1}{} \overline{\kappa}_{1'}{} -  \overline{\kappa}_{1'}{}^2) \tau' G_{01'}}{2 r \overline{\kappa}_{1'}{}}
 + \frac{6  \overline{\kappa}_{1'}{} \tau' \overline{\hat{\beta}'}}{r \kappa_{1}{} \rho'}
 + \edt G_{10'}
 -  \frac{6 \edt \hat{\beta}'}{r \rho'}
 + \edtp G_{01'}
 -  \frac{6 \overline{\kappa}_{1'}{} \edtp \overline{\hat{\beta}'}}{r \kappa_{1}{} \rho'}.
\end{align}
This can then be expressed as \eqref{eq:TransEqG00} in terms of spin-weighted quantities, where the scaling of $\widehat{G}_0$ was chosen to eliminate the lower-order terms. This also has the effect that a variation of the linearized mass in the ORG corresponds to adding a constant to $\widehat{G}_0$, cf. \eqref{eq:G0hatlinmass}. 
\end{proof}

\subsection{The Teukolsky equations}
\label{sec:TeukolskyEquations}
We now present the Teukolsky equations in a form that is convenient as a precursor for proving decay estimates in sections \ref{sec:spin-2TeukolskyEstimates}-\ref{sec:spin+2TeukolskyEstimates} for the extreme curvature components $\psibase[\pm2]$. The original work of Teukolsky \cite{1972PhRvL..29.1114T} on the existence of decoupled equations for the extreme components of the linearized curvature is one of the key breakthroughs from what is called the ``golden age'' of black hole physics. The original equations of Teukolsky were written in the NP formalism with respect to particular choices of coordinates, tetrad, and scaling, and many others forms are possible. In particular, the source-free Teukolsky equations can be written in the GHP formalism relative to a principal tetrad as \cite[Eqs. (A.2)]{2016arXiv160106084A}
\begin{subequations}\label{eq:TME:spin-2}
\begin{align}
 \big((\tho - 3 \rho -\bar{\rho})\thop   - (\edt - 3 \tau  -  \bar{\tau}')\edtp  - 3 \Psi_{2} \big) (\kappa_1 \vartheta\Psi_0) ={}&0,\\
 \big((\thop - 3 \rho'-  \bar{\rho}')\tho  - (\edtp -3 \tau'- \bar{\tau}) \edt  - 3 \Psi_{2} \big) (\kappa_1\vartheta\Psi_4)={}&0 .
\end{align}
\end{subequations}

The lemma below gives a further form of the Teukolsky equation, which is obtained by transforming to the rescaled and deboosted variables $\psibase[\pm2]$ from definition \ref{def:psibase}. The rescaling is chosen, following \cite{2017arXiv170807385M}, so that in the limit as $r\rightarrow\infty$, the $\psibase[\pm2]$ are radiation fields that have non-vanishing limits on $\Scri^+$. 
In particular, the equations are such that it is possible to apply the $r^p$ estimate (for example in lemma \ref{lem:GeneralrpSpinWave}) as $r\rightarrow\infty$, although, in sections \ref{sec:spin-2TeukolskyEstimates}-\ref{sec:spin+2TeukolskyEstimates}, we introduce a further rescaling factor that converges to $1$ as $r\rightarrow\infty$. 
In fact, due to different tetrad choices and different further scalings, the scalars $\psibase[\pm 2]$ differ from the Teukolsky scalars $\psi_{\text{Teu}, \pm2}$ solving the classic Teukolsky equations derived in \cite{1972PhRvL..29.1114T}  via the formulas  $\psibase[+2]=\tfrac{1}{4}\sqrt{r^2+a^2} \Delta^2 \psi_{\text{Teu}, +2}$ and $\psibase[-2]= \sqrt{r^2+a^2} \Delta^{-2} \psi_{\text{Teu}, -2}$.
The proof of the following lemma appears in appendix \ref{sec:DeBoostingTME}.

\begin{lemma}[Teukolsky equation]\label{lem:TeukolskyRegular}
Let $\psibase[-2], \psibase[+2]$ be as in definition \ref{def:psibase}. 
\begin{subequations} 
\begin{align}
\widehat\squareS_{-2}(\psibase[-2])={}&- \frac{8 a r}{a^2 + r^2}\LetaOp\psibase[-2]
 + 8 r\VOp\psibase[-2]
 + \frac{4 M (a^2 -  r^2)}{a^2 + r^2}\YOp\psibase[-2]
 - \frac{4 r (r - M) \psibase[-2]}{a^2 + r^2},
\label{eq:TeukolskyRegular-2} \\
\widehat\squareS_{2}(\psibase[+2])={}&\frac{8 a r}{a^2 + r^2}\LetaOp\psibase[+2]
 - 8 r\VOp\psibase[+2]
 -  \frac{4 M (a^2 -  r^2)}{a^2 + r^2}\YOp\psibase[+2]
 + \frac{4 r (r - M) \psibase[+2]}{a^2 + r^2}.
\label{eq:TeukolskyRegular+2} 
\end{align}
\end{subequations} 
\end{lemma}

\subsection{The Teukolsky-Starobinsky Identities} 
\label{sec:TSI} 
Another classical equation in this field is the Teukolsky-Starobinsky Identities (TSI) \cite{1974JETP...38....1S,1974ApJ...193..443T}. The GHP form of the TSI \cite[Eqs. (A.5a), (A.5e)]{2016arXiv160106084A} is 
\begin{subequations}\label{eq:A.5-TSI}
\begin{align}
0={}&\thop \thop \thop \thop (\kappa_1^4\vartheta\Psi_{0})
 -  \edt \edt \edt \edt (\kappa_1^4\vartheta\Psi_{4})
 -  \tfrac{M}{27} \LxiOp\overline{\vartheta\Psi_4},
 \label{eq:A.5a-TSI} \\
0={}&\edtp \edtp \edtp \edtp (\kappa_1^4\vartheta\Psi_{0}) -  \tho \tho \tho \tho (\kappa_1^4\vartheta\Psi_{4})
 -  \tfrac{M}{27} \mathcal{L}_{\xi}\overline{\vartheta\Psi_0} \label{eq:A.5e-TSI} .
\end{align}
\end{subequations}
See \cite{2016arXiv160106084A} for the complete set of $5$ TSI for linearized gravity on Petrov type D spacetimes. 

The following lemma expresses the TSI \eqref{eq:A.5a-TSI} in terms of the $\psibase[\pm2]$. It is proved by a calculation, which is simplified by noting that, if we define the spin-weight $1$ quantity $\ringtau$  by  
\index{T3tauring@$\ringtau$}
\begin{align}
\ringtau={}&-9 \kappa_{1}{}^2 \tau , 
\end{align}
where $\tau$ is the GHP spin-coefficient, then $\ringtau$  
satisfies  
\begin{align} \label{eq:tau0ids}
\hedt(\ringtau)={}&0,&
\LxiOp(\ringtau)={}&0,&
\YOp(\ringtau)={}&0. 
\end{align}

\begin{lemma}[Teukolsky-Starobinsky Identity]\label{lem:TSIRegular} In terms of the variables $\psibase[-2]$ and $\psibase[+2]$ introduced in definition \ref{def:psibase} and the spin-weighted operators introduced in definitions \ref{def:VY-cov} and \ref{def:edtedt'-cov}, we have 
\begin{align} \label{eq:TSIpsihat}
\hedt^{4}\psibase[-2]={}&-3 M \LxiOp(\overline{\psibase[-2]})
  - \sum_{k=1}^4\binom{4}{k}  \ringtau^k \hedt^{4-k} \LxiOp{}^k\psibase[-2]
  +  \tfrac{1}{4} \Bigl(Y + \frac{r}{a^2 + r^2}\Bigr)^4\psibase[+2].
\end{align}
\end{lemma} 
The details of the proof can be found in appendix \ref{sec:DeBoostingTME}.

\section{Analytic preliminaries}
\label{sec:AnalyticPreliminaries}
\subsection{Conventions and notation}
In this subsection, we primarily state our conventions and notation to treat common techniques in analysis.

While most of this section is standard, our choice of the constant $\CInHyperboloids$ is motivated by technical considerations. 
Recall that $\CInHyperboloids$ appeared as a constant in lemma \ref{lem:h-prop} and definition \ref{def:timefuncs} for the time functions $\timefunc$. From equation \eqref{eq:uMinusTimefunc}, $2\CInHyperboloids$ is the coefficient of the leading-order, $M^2/r$ term by which the level sets of $\timefunc$ curve above the level sets of the retarded time $u$. As such, $\CInHyperboloids$ can be thought of as a measure of the curvature towards future null infinity $\Scri^+$ of the hyperboloids of constant $\timefunc$. As explained in the paragraph before the statement of lemma \ref{lem:GeneralrpSpinWave}, we require $\CInHyperboloids$ to be sufficiently large, that is that the hyperboloids be sufficiently curved near $\Scri^+$. The results of this paper are valid for any sufficiently large choice of $\CInHyperboloids$, but we state a specific value here, since the choice of $\CInHyperboloids$ affects the choice of some of the constants appearing in, for example, the basic estimates of section \ref{sec:BasicEstimates}. 

\index{C2CHyp@$\CInHyperboloids$}
\begin{definition}
\label{def:timefunc0CInHyperboloids}
Throughout the rest of the paper, let $\CInHyperboloids= 10^6$ be fixed. 
\end{definition}

The set of natural numbers $\{0,1,\dots\}$ is denoted $\Naturals$, the integers $\Integers$, and the positive integers $\Integers^+$. Recall that $\timefunc_0= 10M$ was set in definition \ref{def:Sigmainit}.

\begin{definition}
\label{def:volumeForms}
\index{D1di2@$\diTwoVol$}
\index{D1di3@$\diThreeVol$}
\index{D1di4@$\diFourVol$}
\index{D1di3Scri@$\diThreeVolScri$}
The reference volume forms are
\begin{subequations}
\begin{align}
\diTwoVol ={}& \sin\theta\di\theta \wedge \di\phi,\\
\diThreeVol ={}& \di r \wedge \diTwoVol ,\\
\diFourVol ={}& \di \timefunc \wedge \diThreeVol,\\
\diThreeVolScri={}& \di \timefunc \wedge \diTwoVol  . \label{eq:diThreeVolScri}
\end{align}
\end{subequations}
Given a 1-form $\nu$, let \index{D1di3Leray@$\LerayForm{\Normal}$}
 $\LerayForm{\Normal}$ denote a Leray 3-form such that $\Normal\wedge\LerayForm{\Normal}=\di^4\mu$, see \cite{MR0460898}. 
\end{definition}

\begin{remark} 
\label{rem:rescaleMToOne}
The family of Kerr metrics, when written for example in ingoing Eddington-Finkelstein coordinates, are such that, for any $\Lambda>0$, the rescaling 
\begin{align} \label{eq:Mrescale} 
(M,a,v,r,\theta,\phi) \mapsto (\Lambda M,\Lambda a,\Lambda v, \Lambda r,\theta,\phi)
\end{align} 
takes a Kerr solution to a Kerr solution. Thus, if an estimate can be proved for a given value of $M=M_1$, then the same estimate can be proved for another value $M=M_2$ by rescaling with $\Lambda =M_2/M_1$.  Furthermore, any statement in this paper involving $(a,v,r)$ can be restated for any given $M$ as a statement in terms of $(a/M,v/M,r/M)$. It follows from the definition of the  hyperboloidal time function that it scales as 
\begin{align} 
\timefunc \to \Lambda \timefunc
\end{align} 
with respect to  the rescaling \eqref{eq:Mrescale}.   
\end{remark}

\begin{definition}\label{def:massnormalization}
\begin{enumerate} 
\item We say that a quantity $Q$ has dimension $M^\mu$ if $Q \to \Lambda^\mu Q$ under a rescaling of the type \eqref{eq:Mrescale}. In particular, $Q$ is said to be dimensionless if $\mu = 0$. 
\item In view of remark \ref{rem:rescaleMToOne} it is sufficient to consider $M=1$. This procedure will be referred to as mass normalization. 
\end{enumerate}
\end{definition} 
All results in the paper are stated in terms of a general mass parameter. However, mass normalization will be used in some proofs  for simplicity. 
\begin{definition} \label{def:universal} 
\begin{enumerate} 
\item Let $\delta$ denote a sufficiently small, positive constant.  
\item We shall use regularity parameters, generally denoted $\ireg$, and sufficiently large regularity constants $K$, independent of  $\ireg,|a|/M, \delta$.
\item Unless otherwise specified,  we shall in estimates use constants $C = C(\ireg,|a|/M, \delta)$.
\item Let $P$ be a set of parameters. A constant $C(P)$ is a constant of the form 
\begin{align} 
C(P) = C(P;\ireg,|a|/M, \delta). 
\end{align} 
\end{enumerate}
\end{definition}

\begin{remark}
\begin{enumerate} 
\item  
Throughout this paper, it is necessary to have many small parameters. It is sufficient to replace all of these small parameters by the smallest of them and, hence, to treat them all as a single parameter. This small parameter is denoted $\delta>0$ as stated in the previous definition. 
\item Unless otherwise stated, constants such as $C, K$ can change value from line to line, as needed, and the allowed range of values for $\delta$ may decrease as needed. 
\end{enumerate} 
\end{remark} 

\begin{definition} 
\label{def:lesssim} 
\index{1Lesssim@$\lesssim$}
\begin{enumerate}
\item 
Let $F_1, F_2$ be dimensionless quantities, and let $\delta$ be a positive dimensionless constant. We say that  $F_1\lesssim F_2$ if there exists a constant $C$ such that $F_1\leq C F_2$.
\item Let $F_1, F_2$ be such that $F_1/F_2$ has dimension $M^\gamma$.  We say that $F_1\lesssim F_2$ if $F_1 \leq M^\gamma C F_2$. 
\item  Let $P$ be a set of parameters. We say that $F_1\lesC{P} F_2$ if there is a constant $C(P)$ such that $F_1 \lesssim C(P) F_2$. 
\item We say that $F_1 \gtrsim F_2$ and $F_1 \gtrC{P} F_2$ if $F_2 \lesssim F_1$ and $F_2 \lesC{P} F_1$, respectively, and further that $F_1 \sim F_2$ if it holds that $F_1 \lesssim F_2$ and $F_2 \lesssim F_1$. For a set of parameters $P$,  $F_1 \simC{P} F_2$ is defined analogously.  
\end{enumerate}  
 \end{definition} 

\begin{definition} Let $m\in \Naturals$. 
\index{O2ObigAnalytic@$\bigOAnalytic(\cdot)$}
\begin{enumerate}
\item 
Let $R$ be the compactified radial coordinate. We say that $f(R,\omega) = \bigOAnalytic(R^m)$ if $\forall  j \in \Naturals$,     
\begin{align} 
| \partial_R^j f(R)| \leq C(j) R^{\max\{m-j,0\}} \quad \text{for $R \in (0,1/10M]$} .
\end{align} 
\item We say that $f(r,\omega) = \bigOAnalytic(r^{-m})$ if $f(R) = \bigOAnalytic(R^m)$.

\end{enumerate}
\end{definition}

\index{0zdotplus@$\cdot{+}$}
\index{0zdotminus@$\cdot{-}$}
\begin{definition}
For any $\gamma\in\Reals$, a bound involving the expression $\gamma-$ means that there is a constant $C>0$, not depending on $\ireg,|a|/M, \delta$, such that the bound holds with $\gamma-$ replaced by $\gamma-C\delta$. Similarly, a bound involving the expression $\gamma+$ means that there is a constant $C>0$ such that the bound holds with $\gamma+$ replaced by $\gamma+C\delta$. 
\end{definition}

\begin{definition}
Let $\timefunc$ be the hyperboloidal time function from definition \ref{def:timefuncs}. Define 
\index{0tbracket@$\langle \cdot \rangle$}
\begin{align}\label{eq:jpbrack}
\langle \timefunc \rangle = (M^2 + \timefunc^2)^{1/2}.
\end{align} 
\end{definition}

\subsection{Conformal regularity}
\label{sec:ConformalRegularity} 

Here we state a definition and basic properties of conformal regularity.  Lemma \ref{lem:ddRasVOpstructure} is used in the proof of lemma \ref{lem:ConformalExtension} and elsewhere in this paper.  

\begin{definition}
\label{def:conformallyRegular}
A spin-weighted scalar $\varphi$ is said to be conformally regular if it is smooth in the future domain of dependence of $\Stauini$ and extends smoothly to  $\Reals\times [-\epsilon,r_+^{-1}) \times \Sphere$ in the compactified hyperboloidal coordinates $(\timefunc,\rInv,\omega)$, for some $\epsilon > 0$. A differential operator is conformally regular if it has an extension that maps conformally regular scalars to conformally regular scalars.
\end{definition}

\index{C2CHyp@$\CInHyperboloids$}
\begin{lemma}
\label{lem:ddRasVOpstructure}
The coefficient $H$ from definition \ref{def:hPrimeInR} which arises in considering $\Staut$ satisfies
\begin{align} 
(2 + 2a^2 R^2 - H) R^2\Delta = 2\CInHyperboloids M^2 R^2 + M^3\bigOAnalytic(R^3) .
\end{align} 
In the Znajek tetrad and the compactified hyperboloidal coordinate system $(\timefunc, R, \theta, \phi)$, we have for a spin-weighted scalar $\varphi$, which is smooth at $R=0$, 
\begin{align}
\partial_{R} \varphi ={}&-2R^{-2}\VOp\varphi+MR^{-1}\bigOAnalytic(1)\VOp\varphi+M\bigOAnalytic(1) \LetaOp\varphi
 + M^2\bigOAnalytic(1)\LxiOp{}\varphi .
\label{eq:ddRasVOpstructure} 
\end{align}
\end{lemma}
\begin{proof}These follow by direct computation.
\end{proof}

\begin{lemma}
\label{lem:ConformalExtension}
Let  $b_\phi,b_0$ be conformally regular functions,  let $b_{\VOp}$ be such that $R b_{\VOp}$ is conformally regular, and let $\vartheta$ be a conformally regular spin-weighted scalar. 
If $\varphi$ is a solution of 
\begin{align}
\widehat\squareS_{s} \varphi
 +b_\VOp  \VOp\varphi
 +b_\phi \LetaOp \varphi
 +b_0 \varphi
={}&
 \vartheta  ,
\label{eq:eqnForConformalRegularity}
\end{align}
and if the initial data for $\varphi$ on $\Stauini$ is smooth and compactly supported, then $\varphi$ is conformally regular. 
\end{lemma}
\begin{proof}
The essence of this proof is to apply standard local well-posedness results for linear wave equations in both the hyperboloidal coordinates $(\timefunc,r,\omega)$ and the compactified hyperboloidal coordinates $(x^a) = (\timefunc,\rInv,\omega)$. 
Working in the compactified coordinate system, one  finds
\begin{align}
\widehat\squareS_s(\varphi)
={}&(4 \CInHyperboloids  + M\bigOAnalytic(R))M^2\partial_{\timefunc} \partial_{\timefunc} \varphi 
+ (-2 + M^2\bigOAnalytic(R^2))\partial_{\timefunc} \partial_{R} \varphi \nonumber\\
& + (4 a + M^2\bigOAnalytic(R))\partial_{\timefunc} \partial_{\phi} \varphi 
+ \bigOAnalytic(R^2)\partial_{R} \partial_{R} \varphi 
+ M\bigOAnalytic(R^2)\partial_{R} \partial_{\phi} \varphi \nonumber\\
&
+ (-2 R + M\bigOAnalytic(R^2))\partial_{R} \varphi 
 + (2 a R + M^2\bigOAnalytic(R^2))\partial_{\phi} \varphi \nonumber \\
& + (4 \CInHyperboloids  R + M\bigOAnalytic(R^2))M^2\partial_{\timefunc} \varphi 
+ (2 M R + M^2\bigOAnalytic(R^2))\varphi 
- \TMESOp_s(\varphi), 
\end{align}
where $\TMESOp_s$ is given by \eqref{eq:SOp-def}. 
The principal part of $\widehat\squareS_s$ can be written 
\begin{align} 
\generalMetric^{ab} \partial_{x^a}\partial_{x^b} = \generalMetric^{ab}_0 \partial_{x^a}\partial_{x^b} + R\generalMetric_1^{ab} \partial_{x^a} \partial_{x^b}
\end{align} 
where 
\begin{align} 
\generalMetric^{ab}_0 \partial_{x^a} \partial_{x^b} ={}& (4\CInHyperboloids M^2 - a^2 \sin^2 \theta) \partial_\timefunc\partial_\timefunc -2\partial_\timefunc\partial_R+4a\partial_\timefunc\partial_\phi - \partial_\theta\partial_\theta -\sin^{-2}\theta\partial_\phi\partial_\phi 
\end{align}
and $\generalMetric_1^{ab}$ has conformally regular components. One finds that $\generalMetric_{ab} $ extends as a  Lorentzian metric across $\Scri^+$ and that the level sets of $\timefunc$ are spacelike with respect to $h_{ab}$. 

The lower-order terms $b_\phi \LetaOp + b_0$ in \eqref{eq:eqnForConformalRegularity} are conformally regular. Further, in view of \eqref{eq:ddRasVOpstructure}, we have that $b_{\VOp} V$ is conformally regular. Thus, the operator on the left-hand side of \eqref{eq:eqnForConformalRegularity} is conformally regular and has principal part with symbol given by the inverse conformal metric $h^{ab}$. Thus, equation \eqref{eq:eqnForConformalRegularity} is a spin-weighted wave equation in the extended spacetime. 

Since the initial data for $\varphi$ is assumed to be compactly supported, there is some $\timefunc$ and a smooth, spacelike surface $\Sigma$ in the extended spacetime, which agrees with $\Sigma_\timefunc$ for large $r$, such that $\varphi$ is smooth and compactly supported on $\Sigma \cap \{R > 0\}$, and such that the future domain of dependence of $\Sigma$ includes $\Scri^+_{\timefunc,\infty} = \{R=0\} \cap (\timefunc,\infty)$. 
It follows that $\varphi$ is smooth in the domain of dependence of $\Sigma$ in $ \Reals\times (-\epsilon,r_+^{-1})\times \Sphere$ with inverse metric $\generalMetric^{ab}$, and in particular conformally regular. 
\end{proof}

\subsection{Norms} 
\label{sec:norms}  

This subsection introduces various norms at a point, on hypersurface, and in space-time regions. To define various Sobolev norms, families of differential operators are also defined in definition \ref{def:diffOps}. 

\index{0Norm@$\abs{{{}\cdot{}}}$}
\begin{definition}
\label{def:pointwisenormNoDerivatives}
Let $\varphi$ be a spin-weighted scalar. Its norm is defined to be 
\begin{align}
\abs{\varphi}^2 ={}& \bar\varphi \varphi .
\end{align}
\end{definition}

Recall from section \ref{sec:notandconv} that, if $\varphi$ is a spin-weighted scalar, then $\abs{\varphi}^2 = \bar\varphi \varphi$ has GHP type $\{0,0\}$ and hence is a real-valued function on the manifold. It follows that $\abs{\varphi}$ and expressions like $\nabla_a \abs{\varphi}^2$ have an invariant sense, and we may use this fact to define Sobolev type norms on spaces of spin-weighted scalars. 

\begin{definition}
Let $n \geq 1$ be an integer, and let $\generalOps=\{X_1,\ldots,X_n\}$ be spin-weighted operators. 
Define a multi-index to be either the empty set or an ordered set $\mathbf{a} = (a_1,\ldots,a_m)$ with $m\in \Integers^+$ and $a_i \in \{1,\dots, n\}$ for $i\in\{1,\dots,m\}$. 
If $\mathbf{a} =\emptyset$, define $|\mathbf{a}|=0$ and define $\generalOps^{\mathbf{a}}$ to be the identity operator. If $\mathbf{a}=(a_1,\ldots,a_m)$, define $|\mathbf{a}|=m$ and define the operator 
\begin{align}
\generalOps^{\mathbf{a}}={}& X_{a_1} X_{a_2}\dots X_{a_m}.
\end{align}
\end{definition} 

\begin{definition}
\label{def:pointwisenorm}
Let $\generalOps$ be a set of spin-weighted first-order operators, and let $\varphi$ be a spin-weighted scalar. For $k \in \Naturals$, we define the order $k$ pointwise norm 
\index{0NormXk@$\absHighOrder{{}\cdot{}}{k}{\generalOps}$}
\begin{align}
\absHighOrder{\varphi}{k}{\generalOps}^2
={}& \sum_{\abs{\mathbf{a}}\leq k}\abs{\generalOps^{\mathbf{a}} \varphi}^2 .
\end{align}
\end{definition}

We now introduce  sets of operators to be used in the norms from the previous definition.
The operators in $\unrescaledOps$ have dimensions $M^{-1}$ as is standard for derivative operators. The operators in the remaining sets have been scaled so that they are dimensionless. The operators in $\unrescaledOps$ have been scaled so that, for large $r$, the corresponding vector fields have bounded components with respect to an orthonormal basis for which one vector is parallel to $\xi$. The operators $\rescaledOps$ are such that the three operators $r\VOp$, $\hedt$, $\hedtp$ are rescaled by a factor of $r$, which is useful for obtaining additional decay in $r$; this is the set of derivatives operators that we typically use in our Sobolev norms. The operators $\sphereOps$ are tangent to spheres of constant $\timefunc,r$. The operators $\ScriOps$  are tangent to $\Scri^+$. Mnemonically, these are $\unrescaledOps$ for bounded, $\rescaledOps$ for derivative, $\sphereOps$ for spherical derivatives, and $\ScriOps$ for derivatives on null infinity. 

\begin{definition}
\label{def:diffOps}
Define
\index{B2BOpSet@$\unrescaledOps$}
\index{D2DOpSet@$\rescaledOps$}
\index{S2SOpSet@$\sphereOps$}
\index{D2DSlashOpSet@$\ScriOps$}
\begin{subequations}
\begin{align}
\unrescaledOps
={}&\{ \YOp,  \VOp,  r^{-1}\hedt,  r^{-1}\hedtp \} , \label{eq:unrescop} \\
\rescaledOps
={}&\{ M\YOp, r\VOp, \hedt, \hedtp\},\\
\sphereOps={}& \{\hedt,\hedtp\}, \label{eq:SphereOps}\\
\ScriOps={}&\{\hedt,\hedtp, M\LxiOp\} .
\end{align}
\end{subequations}
\end{definition}

The following definition introduces weighted Sobolev spaces. Because the mass $M$ provides a natural length scale, we are able to ensure that the integrands in the weighted Sobolev norms are dimensionless. 

\begin{definition} 
Let $\varphi$ be a spin-weighted scalar. Let $\Omega$ denote a four-dimensional subset of the domain of outer communication, and let $\Sigma$ denote a hypersurface in the domain of outer communication that can be parametrized by $(r,\omega)$. For an $k \in \Naturals$ and $\gamma \in \Reals$, define 
\index{0WnormBulk@$\norm{{}\cdot{}}_{W^{k}_\gamma(\Omega)}$}
\index{0WnormSurface@$\norm{{}\cdot{}}_{W^{k}_\gamma(\Sigma)}$}
\index{0WnormSphere@$\norm{{}\cdot{}}_{W^{k}(\Sphere)}$}
\index{W2WnormBulk@$W^{k}_\gamma(\Omega)$}
\index{W2WnormSurface@$W^{k}_\gamma(\Sigma)$}
\index{W2WnormSphere@$W^{k}_\gamma(\Sphere_{t,r})$}
\begin{subequations}
\begin{align}
\norm{\varphi}_{W^{k}_\gamma(\Omega)}^2 ={}&  
\int_\Omega M^{-\gamma-2} r^\gamma \absHighOrder{\varphi}{k}{\rescaledOps}^2  \diFourVol ,\\
\norm{\varphi}_{W^{k}_\gamma(\Sigma)}^2 ={}&  
\int_\Sigma M^{-\gamma-1} r^\gamma \absHighOrder{\varphi}{k}{\rescaledOps}^2  \diThreeVol,\\
\norm{\varphi}_{W^{k}(\Sphere)}^2 ={}&  
\int_{\Sphere} \absHighOrder{\varphi}{k}{\sphereOps}^2  \diTwoVol . \label{eq:WkS2}
\end{align}
\end{subequations}
We shall refer to norms $\norm{\varphi}_{W^k_\gamma(\Dtau)}$ and $\norm{\varphi}_{W^{k}_\gamma(\Sigma_{\timefunc})}$ as weighted Morawetz and energy norms, respectively.  
We say that $\varphi \in W_{\gamma}^{k}(\Dtau)$ if $\norm{\varphi}_{W_{\gamma}^{k}(\Dtau)}<\infty$ and similarly for $W_{\gamma}^{k}(\Staut)$, $W_{\gamma}^{k}(\Boundtext)$, $W^{k}(\Sphere)$, and so on.  
\end{definition}

\begin{remark} \label{rem:secsob}
Since definition \ref{def:pointwisenormNoDerivatives} (following definition \ref{def:spinWeightedScalar}) introduces a pointwise norm on spin-weighted scalars, the spaces $W^k_\gamma(\Omega), W^k_\gamma(\Sigma), W^k(S^2)$ etc., are Sobolev spaces of sections of Riemannian vector bundles, and by remark \ref{rem:ops}, when restricting to the sphere $S^2$, the operators $\hedt, \hedtp$ are elliptic operators of order one, acting on sections of these bundles. In the following we shall freely make use of these facts.  
\end{remark}

\begin{definition} 
\label{def:generalSpinorInitialDataNorm}
Let $\varphi$ be a spin-weighted scalar, and let $\ireg\in\Naturals$ and $\alpha\in \Reals$.
\begin{enumerate}
\item 
\index{I2Ikinitial@$\inienergy{\ireg}{\alpha}$}
\index{0normHkalpha@$\Vert \cdot \Vert_{H^\ireg_\alpha(\Sigma)}$}
Let $\Sigma$ denote a hypersurface in the domain of outer communication that can be parametrized by $(r,\omega)$.
Define 
\begin{align} 
\norm{\varphi}_{H^\ireg_\alpha(\Sigma)}^2 = \sum_{|\mathbf{a}|\leq \ireg} \int_{\Sigma} M^{-\alpha} r^{\alpha+2|\mathbf{a}|-1}
\abs{\unrescaledOps^\mathbf{a} \varphi}^2\diThreeVol
\label{eq:ininormdef}
\end{align}
and introduce the quantity 
\begin{align}
\inienergy{\ireg}{\alpha}(\varphi)
=  \norm{\varphi}_{H^\ireg_\alpha(\Stauini)}^2 .
\end{align}
\item \label{point:Pnorm}
Define the following norm on the surface $\Stauini$ \index{P2Pkinitial@$\inipointwise{\ireg}{\alpha}$}
\begin{align}
\inipointwise{\ireg}{\alpha}(\varphi)
=  \sup_{r\in [r_+,\infty)} \sum_{|\mathbf{a}| \leq \ireg} M^{-\alpha} r^{\alpha+2|\mathbf{a}|} \int_{\Sphere}\abs{\unrescaledOps^{\mathbf{a}}\varphi(\timefunc_0-h(r)/2,r,\omega)}^2\diTwoVol. 
\end{align}
\end{enumerate} 
\end{definition}

\begin{remark}
\label{rem:HWrel} 
We have 
\begin{align} \label{eq:HWrel}
\Vert \varphi \Vert_{W^k_{\alpha}(\Stauini)} \lesssim \Vert \varphi \Vert_{H^k_{\alpha+1}(\Stauini)} .
\end{align}
\end{remark}

\begin{definition}\label{def:generalSpinorScriNorm} 
Let $\varphi$ be a spin-weighted scalar, and let  $\ireg\in\Naturals$. Define 
\index{0normF@$\Vert \cdot \Vert_{F^\ireg(\Scritini)}$}
\begin{align} \label{eq:flux-norm}
\norm{\varphi}_{F^\ireg(\Scritini)}^2 ={}& \int_{\Scritini} M\absHighOrder{\YOp\varphi}{\ireg}{\ScriOps}^2 \diThreeVolScri .
\end{align} 
\end{definition}

\subsection{Basic estimates}
\label{sec:BasicEstimates}

In this subsection, we state several classical estimates. These require some slight adaption from their most common formulations, since they are applied to spin-weighted scalars rather than real- or complex-valued functions. These include results for integration by parts, the equivalence of different definitions of the Sobolev norm, eigenvalues for derivatives on spheres, Hardy estimates, Sobolev estimates, and Taylor expansions. 

The operators $\hedt, \hedtp$ are the spherical edth operators, see 
\cite{1982MPCPS..92..317E} for background. In particular, they are elliptic first order operators acting on properly weighted functions on the sphere. For completeness, we recall some useful facts about $\hedt, \hedtp$. Lemma \ref{lem:Op-forms} gives coordinate expressions.
\begin{lemma}\label{lem:hedt'hedtnormrelation}
Let $\varphi,\psi$ be scalars with spin-weight $s$ and $s+1$ respectively. 
Then,
\begin{enumerate}
\item \label{point:IBPhedtGeneral}
\begin{align}
\int_{\Sphere}   \psi (\overline{\hedt \varphi}) \diTwoVol 
={}& - \int_{\Sphere} (\hedtp \psi) \overline{\varphi} \diTwoVol .
\label{eq:IBPhedtGeneral} 
\end{align}
\item \label{point:IBPhedt}
if $s=-1$ it holds that
\begin{align}
\int_{\Sphere} \hedt \varphi \diTwoVol
={}& 0 ;
\label{eq:IBPhedt} 
\end{align}
\item \label{point:IBPhedt'}
if $s=1$, it holds that
\begin{align}
\int_{\Sphere}  \hedtp\varphi \diTwoVol
={}& 0 ;
\label{eq:IBPhedt'}
\end{align}
\item \label{point:hedt'hedtnormrelation}
we have the following relation between $\Vert \hedt \varphi \Vert_{L^2(\Sphere)}$ and $\Vert\hedtp \varphi\Vert_{L^2(\Sphere)}$:
\begin{align}
\int_{\Sphere}|\hedt\varphi|^2 \diTwoVol = \int_{\Sphere}|\hedtp\varphi|^2 \diTwoVol - s\int_{\Sphere}|\varphi|^2 \diTwoVol.
\label{eq:hedt'hedtnormrelation}
\end{align}
\end{enumerate}
\end{lemma}
\begin{proof}
The first point follows from integration by parts, see \cite[(A13)]{leonard:poisson:1997}. The second follows from taking $\psi=1$, and the third follows from complex conjugation. For the fourth point, we multiply both sides of the commutator relation \eqref{eq:commutatorofhedtandhedtprime} by $\bar{\varphi}$ and use the Leibniz rule to obtain
\begin{align}
\hedt (\hedtp \varphi \bar{\varphi})-\hedt\bar{\varphi} \hedtp\varphi={}& \hedtp(\hedt \varphi \bar{\varphi})-\hedtp\bar\varphi\hedt \varphi-s|\varphi|^2.
\end{align}
By integrating over $\Sphere$ and noting the facts that $\hedtp \varphi \bar{\varphi}$ has  boost- and spin-weight $0,-1$ and $\hedt \varphi \bar{\varphi}$ has  boost- and spin-weight $0,1$, the integrals over $\Sphere$ of the first term on the left and the first term on the right are both vanishing, hence the relation \eqref{eq:hedt'hedtnormrelation} follows.
\end{proof}

\begin{lemma}
\label{lem:Ellipticforsecondordersphericalangularop}
Let $\varphi$ be a scalar of spin-weight $s$. For any $\ireg \geq 0$, it holds
\begin{align}
\int_{\Sphere}\absHighOrder{\varphi}{2\ireg}{\sphereOps}^2 \diTwoVol
\simC{s}{} \sum_{i=0}^\ireg \int_{\Sphere}\abs{\mathring{S}^{i}_{s}\varphi}^2 \diTwoVol .
\label{eq:ellipestispinweightedsphericalLaplacian}
\end{align}
\end{lemma}
\begin{proof} 
Since $\hedt$ and $\hedtp$ are both in $\sphereOps$, this follows from the relation \eqref{eq:mathringSbyhedthedt'} and the fact that $\hedt, \hedtp$ are elliptic operators of order one \cite[Theorem III.5.2]{MR1031992}. 
\end{proof} 

\begin{lemma}[Eigenvalue estimates for $\hedt, \hedtp$]
\label{lem:ellip}
If $\varphi$ is a scalar of spin-weight $s$, then 
\begin{subequations}
\label{eq:SphereHardyBoth}
\begin{align} 
\frac{|s|-s}{2} \int_{\Sphere} |\varphi |^2 \diTwoVol
\leq{}&
\int_{\Sphere} |\hedt \varphi |^2 \diTwoVol, 
\label{eq:SphereHardyNegaSpin}\\
\frac{|s|+s}{2} \int_{\Sphere} |\varphi |^2 \diTwoVol
\leq{}&
 \int_{\Sphere} |\hedtp \varphi |^2 \diTwoVol , 
 \label{eq:SphereHardyPosSpin}
\end{align}
\end{subequations}
and for a four dimensional spacetime region $\Omega$,
\begin{subequations}
\begin{align} 
\frac{|s|-s}{2} \Vert \varphi \Vert_{W^k_\gamma(\Omega)}^2 
\leq{}&
\Vert \hedt \varphi \Vert_{W^k_\gamma(\Omega)}^2, \\
\frac{|s|+s}{2} \Vert \varphi \Vert_{W^k_\gamma(\Omega)}^2 
\leq{}&
\Vert \hedtp \varphi \Vert_{W^k_\gamma(\Omega)}^2 .
\end{align}
\end{subequations}
The first (second) case gives an estimate if $\varphi$ has negative (positive) spin-weight.
\end{lemma}
\begin{proof}
We will prove the statement for $\hedt$. The statement for $\hedtp$ follows by complex conjugation. 
Expand $\varphi$ in terms of spin-weighted spherical harmonics (see \cite[section 4.15]{MR917488})
\begin{align}
\varphi(\theta,\phi)={}&\sum_{l=|s|}^\infty\sum_{m=-l}^{l} a_{l,m}\thinspace {}_{s} Y_{lm}(\theta,\phi). 
\end{align}
From \cite[Eq. (4.15.106)]{MR917488} we have
\begin{align}
\hedt \varphi (\theta,\phi)={}&- \sum_{l=|s|}^\infty\sum_{m=-l}^{l} a_{l,m} \frac{\sqrt{(l+s+1)(l-s)}}{\sqrt{2}}{}_{s+1} Y_{lm}(\theta,\phi).
\end{align}
Through the orthogonality conditions \cite[Eq. (4.15.99)]{MR917488} we get 
\begin{subequations}
\begin{align}
\int_{\Sphere} |\varphi|^2 \diTwoVol={}&4\pi \sum_{l=|s|}^\infty\sum_{m=-l}^{l} |a_{l,m}|^2,
\label{eq:Plancherel}\\
\int_{\Sphere} |\hedt \varphi|^2 \diTwoVol={}&4\pi \sum_{l=|s|}^\infty\sum_{m=-l}^{l} |a_{l,m}|^2 \frac{(l+s+1)(l-s)}{2}.
\label{eq:eigenvalueOfhedt}
\end{align}
\end{subequations}
As $(l+s+1)(l-s)\geq|s|-s$, this proves \eqref{eq:SphereHardyNegaSpin}. 
Integrating in $\timefunc,r$ gives the remaining results. 
\end{proof}

\begin{lemma}[Control of $\LetaOp$ in $L^2(\Sphere)$]
\label{lem:controlOfLetaOpInL2Sphere}
If $\varphi$ is a scalar of spin weight $s$, then
\begin{align*}
\frac12 \int_{\Sphere} \abs{\LetaOp\varphi}^2 \diTwoVol
\leq{}& \int_{\Sphere} \left(
  \abs{\hedt\varphi}^2
  +\frac{s^2}{2}\abs{\varphi}^2
\right) \diTwoVol .
\end{align*}
\end{lemma}
\begin{proof}
This follows from decomposing into spin-weighted spherical harmonics ${}_s{}Y_{l,m}$, the relations $|m|\leq l$ and $|s|\leq l$, from equations \eqref{eq:Plancherel}-\eqref{eq:eigenvalueOfhedt}, and the fact that
$(l+s+1)(l-s) +s^2
=l^2 +l-s
\geq l^2
\geq m^2 $.
\end{proof}

\begin{lemma}[Spherical Sobolev estimate]  
\label{lem:spherica-Sobolev} 
If $\varphi$ is a scalar of spin-weight $s$, then 
\begin{align} \label{eq:sobineq}
|\varphi|^2 \lesC{s} \int_{\Sphere} \absHighOrder{\varphi}{2}{\sphereOps}^2 \diTwoVol .
\end{align}
\end{lemma}  
\begin{proof} 
The right-hand side of \eqref{eq:sobineq} is the norm on the space $W^2(S^2)$. The standard Sobolev estimate for sections of vector bundles applies. See \cite[Theorems III.2.15 and II.5.2]{MR1031992}. 
\end{proof}

\begin{lemma}[Integration by parts] 
\label{lem:IntegrationByParts}
If $f$ is a smooth scalar with spin- and boost-weight zero and if $f$ vanishes at $\rCutOff$, then 
\begin{subequations}
\begin{align}
\int_{\DtauR} \YOp f \diFourVol
={}& \left[ \int_{\StautR} (\di\timefunc_a \YOp^a) f \diThreeVol \right]_{\timefunc=\timefunc_1}^{\timefunc_2} ,
\label{eq:IBPY}\\
\int_{\DtauR} \VOp f \diFourVol
={}& \left[ \int_{\StautR} f (\di\timefunc_a \VOp^a) \diThreeVol \right]_{\timefunc=\timefunc_1}^{\timefunc_2} 
-\int_{\DtauR} M \frac{r^2-a^2}{(r^2+a^2)^2} f \diFourVol 
+\frac12\int_{\Scritau} f \diThreeVolScri 
\label{eq:IBPV} .
\end{align}
\end{subequations}
\end{lemma}
\begin{proof}
In ingoing Eddington-Finkelstein coordinates, $\diFourVol =\sin\theta\di\phi\di\theta\di r\di v$. The first claim follows from the fact that $\YOp $ is $-\partial_r$ in ingoing Eddington-Finkelstein coordinates. The second claim follows from equation \eqref{eq:V-Znaj-IEF} and that 
\begin{equation}
\partial_r\left(\frac{\Delta}{2(r^2+a^2)}\right)=M\frac{r^2-a^2}{(r^2+a^2)^{2}}.
\end{equation}
\end{proof}

\begin{lemma}[Weighted integration by parts]
Let $f$ be a smooth, real-valued function of $r$ and $\theta$ that vanishes at $\rCutOff$ and $\varphi$ be a spin-weighted scalar. 
\begin{subequations}
\begin{align}
\int_{\DtauR} \Re\left(f \bar\varphi \YOp\varphi\right)  \diFourVol
={}& \left[ \int_{\StautR} (\di\timefunc_a \YOp^a) \frac12 f|\varphi|^2 \diThreeVol \right]_{\timefunc=\timefunc_1}^{\timefunc_2} 
+\int_{\DtauR} \frac12 (\partial_r f) |\varphi|^2 \diFourVol , 
\label{eq:IBPYInnerProduct}\\
\int_{\DtauR} \Re\left(f\bar\varphi \VOp\varphi\right) \diFourVol
={}& \left[ \int_{\StautR} (\di\timefunc_a \VOp^a) \frac12 f|\varphi|^2 \diThreeVol \right]_{\timefunc=\timefunc_1}^{\timefunc_2} 
-\int_{\DtauR} \partial_r \left(f\frac{\Delta}{4(r^2+a^2)}\right) |\varphi|^2 \diFourVol \nonumber\\
&+\frac14\int_{\Scritau} f \abs{\varphi}^2 \diThreeVolScri .
\label{eq:IBPVInnerProduct} 
\end{align}
\end{subequations}
\end{lemma}
\begin{proof}
This follows from the previous lemma and the fact that $\Re(f\bar{\varphi}\VOp\varphi)$ $=\VOp\left(\frac12f\abs{\varphi}^2\right) -(\VOp f)\abs{\varphi}^2/2$, and similarly for $\YOp$. 
\end{proof}

The following lemma gives a standard one-dimensional Hardy inequality on bounded intervals. The subsequent lemma applies this to obtain a similar estimate on each $\StautRc$ with an estimate in terms of the operators $\VOp$ and $\YOp$. Since we consider a bounded interval, we must include terms arising from the end points. When the exponent $\gamma$ is negative, the (nonnegative) contribution from the right endpoint $r_1$ is one of the terms that is bounded above, and the contribution from the left endpoint $r_0$ appears as a term in the upper bound; in contrast, when $\gamma$ is positive, the contribution from the left endpoint $r_0$ is bounded above, and the contribution from the right endpoint $r_1$ appears as part of the upper bound. 

\begin{lemma}[One-dimensional Hardy estimates]
\label{lem:HardyIneq}
Let $\gamma \in \mathbb{R}\setminus \{0\}$  and $h: [r_0,r_1] \rightarrow \mathbb{R}$ be a $C^1$ function. 
\begin{enumerate}
\item \label{point:lem:HardyIneqLHS} If $r_0^{\gamma}\vert h(r_0)\vert^2 \leq D_0$ and $\gamma<0$, then
\begin{subequations}
\begin{align}\label{eq:HardyIneqLHS}
-2\gamma^{-1}r_1^{\gamma}\vert h(r_1)\vert^2+\int_{r_0}^{r_1}r^{\gamma -1} \vert h(r)\vert ^2 \di r \leq \frac{4}{\gamma^2}\int_{r_0}^{r_1}r^{\gamma +1} \vert \partial_r h(r)\vert ^2 \di r-2\gamma^{-1}D_0.
\end{align}
\item \label{point:lem:HardyIneqRHS} If $r_1^{\gamma}\vert h(r_1)\vert^2 \leq D_0$ and $\gamma>0$, then
\begin{align}\label{eq:HardyIneqRHS}
2\gamma^{-1}r_0^{\gamma}\vert h(r_0)\vert^2+\int_{r_0}^{r_1}r^{\gamma -1} \vert h(r)\vert ^2 \di r \leq \frac{4}{\gamma^2}\int_{r_0}^{r_1}r^{\gamma +1} \vert \partial_r h(r)\vert ^2 \di r +2\gamma^{-1}D_0.
\end{align}
\end{subequations}
\end{enumerate}
\end{lemma}
\begin{proof}
We integrate $\partial_r (r^{\gamma} \vert h\vert^2)$ over $[r_0,r_1]$ to obtain:
\begin{align}
r_1^{\gamma}\vert h(r_1)\vert^2-r_0^{\gamma}\vert h(r_0)\vert^2={}&\gamma\int_{r_0}^{r_1} r^{\gamma-1}\vert h(r)\vert^2 \di r +2\int_{r_0}^{r_1}r^{\gamma}\Re\{\bar{h}\partial_r h\} \di r.
\end{align}
In the first case where $\gamma <0$, we apply a Cauchy-Schwarz inequality to estimate the last integral term 
\begin{align}
\left\vert 2\int_{r_0}^{r_1}r^{\gamma}\Re\{\bar{h}\partial_r h\} \di r \right\vert \leq \frac{-\gamma}{2}\int_{r_0}^{r_1} r^{\gamma-1}\vert h(r)\vert^2 \di r +\frac{2}{-\gamma}\int_{r_0}^{r_1} r^{\gamma+1}\vert \partial_r h(r)\vert^2 \di r
\end{align}
Collecting the above two estimates implies \eqref{eq:HardyIneqLHS}. The estimate \eqref{eq:HardyIneqRHS} follows in the same way.
\end{proof}

\begin{lemma}[Hardy estimate on hypersurfaces]
\label{lem:HardyOnStautRc} 
Let $\veps>0$. There is an $\rCutOffInHardy\geq 10M$ such that for $\rCutOff\geq\rCutOffInHardy$ and all spin-weighted scalars $\varphi$, 
\begin{align}
\norm{\varphi}_{W^{0}_{-2}(\StautRc)}^2
\leq{}& (16+\veps) \norm{r\VOp \varphi}_{W^{0}_{-2}(\StautRc)}^2
+\veps\norm{M\YOp \varphi}_{W^{0}_{-2}(\StautRc)}^2
+\norm{\varphi}_{W^{0}_{0}(\StautRcR)}^2 .
\end{align}

Similarly for $\delta>0$ and $\alpha\in[\delta,2-\delta]$, there is a constant $\rCutOffInHardy=\rCutOffInHardy(\delta)\geq 10M$ such that for $\rCutOff\geq\rCutOffInHardy$ and all spin-weighted scalars $\varphi$, 
\begin{align}
\label{eq:HardyOnHypersurfaceBulkWeights}
\norm{\varphi}_{W^{0}_{\alpha-3}(\StautRc)}^2
\lesssim{}& \norm{r\VOp \varphi}_{W^{0}_{\alpha-3}(\StautRc)}^2
+\norm{M\YOp \varphi}_{W^{0}_{-\delta-1}(\StautRc)}^2
+\norm{\varphi}_{W^{0}_{0}(\StautRcR)}^2 .
\end{align}
\end{lemma}
\begin{proof}
Let 
\begin{align} \label{eq:Xcal-def}
\TparallelForRethreadedHardy=h'\VOp +(1-\Delta h'/(2(r^2+a^2)))\YOp .
\end{align} 
Then the vector field $\TparallelForRethreadedHardy^a$ corresponding to $\TparallelForRethreadedHardy$ is tangent to $\Staut$. 
We may  introduce new coordinates $(\rRethreaded,\thetaRethreaded,\phiRethreaded)$ on $\Staut$ by taking $\rRethreaded=r$, $\thetaRethreaded=\theta$, and $\phiRethreaded$ is constant along the flow lines of $\TparallelForRethreadedHardy^a$ such $\phiRethreaded$ agrees with $\phi$ on $r=\rCutOff$. In such coordinates and the Znajek tetrad, one finds that $\TparallelForRethreadedHardy$ is $\partial_\rRethreaded$. 

From the one-dimensional Hardy estimate \eqref{eq:HardyIneqLHS} with $\gamma=-1$, one finds for sufficiently large~$r$ 
\begin{align}
\int_{r}^\infty (r')^{-2} \abs{\varphi(r',\omega)}^2 \di r'
\leq{}& 4 \int_{r}^\infty \abs{\TparallelForRethreadedHardy\varphi(r',\omega)}^2 \di r' 
+2 r^{-1}\abs{\varphi(r,\omega)}^2 .
\end{align}
Integrating this over $r\in(\rCutOff-M,\rCutOff)$, and since $\rCutOff \geq 10M$, one finds
\begin{align}
M\int_{\rCutOff}^\infty r^{-2} \abs{\varphi}^2 \diThreeVol
\leq{}& 4M \int_{\rCutOff-M}^\infty \abs{\TparallelForRethreadedHardy\varphi}^2 \diThreeVol
+4 M\rCutOff^{-1} \int_{\StautRcR} \abs{\varphi}^2 \diThreeVol .
\end{align}
From the definition of $\TparallelForRethreadedHardy$ in equation \eqref{eq:Xcal-def}, the expansion for $h'$ in equation \eqref{eq:hprim-explicit}, and the observation that the $\YOp$ coefficient in $\TparallelForRethreadedHardy$ satisfies
\begin{align}
1-\frac{\Delta h'}{2(r^2+a^2)}
=M^2\bigOAnalytic(r^{-2}) ,
\end{align}
it follows that for sufficiently large $r$, there is the bound $4\abs{\TparallelForRethreadedHardy\varphi}^2 \leq (16+\veps)\abs{\VOp\varphi}^2 +\veps M^2r^{-2}\abs{\YOp\varphi}^2$, which completes the proof. 

For $\alpha\in[\delta,2-\delta]$, a similar argument applies, except the bound $\alpha-3\leq-\delta-1$ is used. The constant in the one-dimensional Hardy estimate \eqref{eq:HardyIneqLHS} diverges as $\gamma=\alpha-2$ goes to zero, but, if $\alpha$ is restricted to an interval $[\delta,2-\delta]$ the constant is uniform in $\alpha$, but depends upon $\delta$. 
\end{proof}

\begin{lemma}[Sobolev estimate on hypersurfaces]
\label{lem:SobolevOnStaut}
Assume $\varphi$ is a scalar of spin-weight $s$, and let  $\TparallelForRethreadedHardy$ be the operator from the proof of the Hardy lemma \ref{lem:HardyOnStautRc}.
For $\gamma \in \Reals$, 
we have, for $\timefunc\geq\timefunc_0$, 
\begin{align}
\sup_{\Staut} \abs{\varphi}^2
\lesC{s}{}& \left(\int_{\Staut} r^{-1-\gamma}
  \abs{\varphi}_{2,\sphereOps}^2 
  \diThreeVol
  \int_{\Staut} r^{-1+\gamma}\abs{r\TparallelForRethreadedHardy\varphi}_{2,\sphereOps}^2
\diThreeVol\right)^{1/2} 
+\int_{\Staut^{r_+,10M}} M^{-1}
  \abs{\varphi}_{2,\sphereOps}^2 
\diThreeVol  \nonumber\\
\lesC{s} {}& \norm{\varphi}_{W^{2}_{-1-\gamma}(\Staut)}\norm{r\TparallelForRethreadedHardy\varphi}_{W^{2}_{-1+\gamma}(\Staut)}+\norm{\varphi}_{W^{2}_{0}(\Staut^{r_+,10M})}.
\label{eq:SobolevOnStautWithgenneralgamma}
\end{align}
In the case that $\gamma=0$, we have
\begin{align}
\sup_{\Staut} \abs{\varphi}^2
\lesC{s}{}& \norm{\varphi}_{W^{3}_{-1}(\Staut)}^2 .
\label{eq:SobolevOnStautWithgammazero}
\end{align}
If $0<\gamma\leq 1$, we also have 
\begin{align}
\sup_{\Staut} \abs{\varphi}^2
\lesC{\gamma,s} {}&(\norm{\varphi}_{W^{3}_{-2}(\Staut)}^2 +\norm{r\VOp\varphi}_{W^{2}_{-1+\gamma}(\Staut)}^2)^{1/2}(\norm{\varphi}_{W^{3}_{-2}(\Staut)}^2 +\norm{r\VOp\varphi}_{W^{2}_{-1-\gamma}(\Staut)}^2)^{1/2}.
\label{eq:SobolevOnStautHolder}
\end{align}
\end{lemma}
\begin{proof}
Let $\TparallelForRethreadedHardy$ be as in the proof of lemma \ref{lem:HardyOnStautRc}. For $r_1,r_2\in [r_+,\infty)$, one has
\begin{align}
\int_{\Sphere}|\varphi(r_2)|^2 \diTwoVol
={}& \Big\vert\int_{r_1}^{r_2} \int_{\Sphere} \partial_r |\varphi(r)|^2 \diThreeVol\Big\vert +\int_{\Sphere}|\varphi(r_1)|^2 \diTwoVol, \nonumber\\
={}& \Big\vert\int_{r_1}^{r_2} \int_{\Sphere}\TparallelForRethreadedHardy |\varphi(r)|^2 \diThreeVol\Big\vert +\int_{\Sphere}|\varphi(r_1)|^2 \diTwoVol, \nonumber\\
\leq {}& \left( \int_{r_1}^{\infty}\int_{\Sphere} r^{-1-\gamma} |\varphi(r)|^2 \diThreeVol \int_{r_1}^{\infty}\int_{\Sphere} r^{-1+\gamma}|r\TparallelForRethreadedHardy \varphi(r)|^2 \diThreeVol\right)^{1/2}
+\int_{\Sphere}|\varphi(r_1)|^2\diTwoVol ,
\label{eq:SobolevOnStaut:intermediatestep}
\end{align}
where in the last step we have used H\"{o}lder inequality. 
We integrate over $r_1$ from $r_+$ to $10M$ and the first line of  \eqref{eq:SobolevOnStautWithgenneralgamma} holds from the spherical Sobolev lemma \ref{lem:spherica-Sobolev} where the integral is taken to be over the sphere with given $\timefunc$ and $r$. The second line of \eqref{eq:SobolevOnStautWithgenneralgamma} holds since $\sphereOps\subset\rescaledOps$.

The estimate \eqref{eq:SobolevOnStautWithgammazero} when $\gamma=0$ follows from applying Cauchy-Schwarz inequality to the right of \eqref{eq:SobolevOnStautWithgenneralgamma} and the fact that $r\TparallelForRethreadedHardy$ is in the span of $r\VOp$ and $M^2r^{-1}\YOp$ with $\bigOAnalytic(1)$ coefficients.

We now prove the estimate \eqref{eq:SobolevOnStautHolder}.
Since $\gamma>0$, one can use the Hardy inequality \eqref{eq:HardyIneqLHS} to arrive at
\begin{align}
\int_{r_+}^{\infty}\int_{\Sphere} r^{-1-\gamma} |\varphi(r)|^2 \diThreeVol
\lesC{\gamma} {}&\int_{r_+}^{\infty}\int_{\Sphere}r^{-1-\gamma}|r\TparallelForRethreadedHardy \varphi(r)|^2 \diThreeVol +\int_{r_+}^{10M}\int_{\Sphere} r^{-1-\gamma} |\varphi(r)|^2\diThreeVol.
\label{eq:Sobolev:OneDimHardy}
\end{align}
Hence, by integrating \eqref{eq:SobolevOnStaut:intermediatestep} over $r_1$ from $r_+$ to $10M$, one finds for any $r\in [r_+,\infty)$
\begin{align}
\int_{\Sphere}|\varphi(r)|^2 \diTwoVol
\lesC{\gamma} {}& \left( \int_{r_+}^{\infty}\int_{\Sphere} 
 r^{-1-\gamma}|r\TparallelForRethreadedHardy \varphi(r)|^2
 \diThreeVol \int_{r_+}^{\infty}\int_{\Sphere} r^{-1+\gamma}|r\TparallelForRethreadedHardy \varphi(r)|^2 \diThreeVol\right)^{1/2} \nonumber\\
&+\int_{r_+}^{10M}\int_{\Sphere}M^{-1}|\varphi(r)|^2\diThreeVol.
\end{align}
Since $r\TparallelForRethreadedHardy$ is in the span of $r\VOp$ and $M^2r^{-1}\YOp$ with $\bigOAnalytic(1)$ coefficients, and since $\sphereOps\subset\rescaledOps$ and the assumption $\gamma\leq 1$, the estimate \eqref{eq:SobolevOnStautHolder} then follows. 
\end{proof}

\begin{lemma}[Anisotropic, spacetime Sobolev inequality]
\label{lem:anisotropicSpaceTimeSobolev}
Let $\varphi$ be a scalar of spin-weight $s$. 

If $\lim_{\timefunc\rightarrow\infty}\abs{r^{-1}\varphi}=0$ pointwise in $(r,\omega)$, then
\begin{align}
\abs{r^{-1}\varphi}^2
\lesC{s}{}& 
\norm{\varphi}_{W^{3}_{-3}(\Dtaut)}
\norm{\LxiOp\varphi}_{W^{3}_{-3}(\Dtaut)} .
\label{eq:anisotropicSpaceTimeSobolev}
\end{align}
\end{lemma}
\begin{proof}
Using the fundamental theorem of calculus and the Cauchy-Schwarz inequality, we have 
\begin{align}
\abs{r^{-1}\varphi}^2
={}& -\int_{\timefunc}^\infty \LxiOp \abs{r^{-1}\varphi}^2 \di\timefunc' \nonumber\\
\leq{}& 2\int_{\timefunc}^\infty \abs{\LxiOp r^{-1}\varphi} \abs{r^{-1}\varphi}\di\timefunc' \nonumber\\
\leq{}& 2
\left(\int_{\timefunc}^\infty \abs{r^{-1}\LxiOp \varphi}^2 \di\timefunc'\right)^{1/2}
\left(\int_{\timefunc}^\infty \abs{r^{-1}\varphi}^2 \di\timefunc'\right)^{1/2} .
\end{align}
Now, from applying the Sobolev inequality \eqref{eq:SobolevOnStautWithgammazero} on each $\Staut$, the result holds. 
\end{proof}

\begin{lemma}[Transition flux is controlled by bulk]
\label{lem:Transitionfluxcontrolledbybulkexterior}
Let $f(\timefunc, r)$ be a spin-weighted scalar. For any real value $\gamma$ and $\timefunc\geq \timefunc_0\geq 1$, it holds true that
\begin{align}
\int_{\timefunc}^{\infty} (\timefunc')^{\gamma}\abs{f(\timefunc', \timefunc')}^2 \di \timefunc'
\lesC{\gamma} \int_{\timefunc}^{\infty}\int_{\timefunc'}^{\infty} r^{\gamma-1}(\abs{f(\timefunc', r)}^2 +\abs{r\TparallelForRethreadedHardy f (\timefunc', r)}^2 )  \di r \di \timefunc'.
\label{eq:Transitionfluxcontrolledbybulkexterior}
\end{align}
\end{lemma}

\begin{proof}
We make a change of coordinate $r=\timefunc' +\zeta$. Fix any $\timefunc''\geq \timefunc$. From the mean-value principle, we can find a $\zeta'\in [\timefunc'', 2\timefunc'']$ such that
\begin{align}
\int_{\timefunc}^{\timefunc+\timefunc''}(\timefunc'+\zeta')^{\gamma-1}\abs{f(\timefunc', \timefunc'+\zeta')}^2 \di \timefunc' \leq {}& (\timefunc'')^{-1}\int_{\timefunc''}^{2\timefunc''}\int_{\timefunc}^{\timefunc+\timefunc''} (\timefunc'+\zeta)^{\gamma-1}\abs{f(\timefunc', \timefunc'+\zeta)}^2 \di \timefunc' \di \zeta \nonumber\\
\leq {}& (\timefunc'')^{-1} \int_{\timefunc}^{\infty}\int_{0}^{\infty} (\timefunc'+\zeta)^{\gamma-1}\abs{f(\timefunc', \timefunc'+\zeta)}^2\di \zeta \di \timefunc' .
\end{align}
Therefore, for the $\zeta'$ chosen above,
\begin{align}
\int_{\timefunc}^{\timefunc+\timefunc''}(\timefunc'+\zeta')^{\gamma}\abs{f(\timefunc', \timefunc'+\zeta')}^2 \di \timefunc' \leq {}& 4\timefunc''\int_{\timefunc}^{\timefunc+\timefunc''}(\timefunc'+\zeta')^{\gamma-1}\abs{f(\timefunc', \timefunc'+\zeta')}^2 \di \timefunc'\nonumber\\
\leq{}&4 \int_{\timefunc}^{\infty}\int_{0}^{\infty} (\timefunc'+\zeta)^{\gamma-1}\abs{f(\timefunc', \timefunc'+\zeta)}^2\di \zeta \di \timefunc '\nonumber\\
\leq{}&4 \int_{\timefunc}^{\infty}\int_{t'}^{\infty} r^{\gamma-1}\abs{f(\timefunc',r)}^2\di r \di \timefunc '.
\end{align}
Since $\timefunc''\geq \timefunc$, $\timefunc'\in [\timefunc,\timefunc+\timefunc'']$ and $\zeta'\in [\timefunc'', 2\timefunc'']$, we have $\timefunc'+\zeta'\in [\timefunc',4\timefunc'']$. It then follows from the fundamental theorem of calculus that
\begin{align}
\hspace{4ex}&\hspace{-4ex}\int_{\timefunc}^{\timefunc+\timefunc''} (\timefunc')^{\gamma}\abs{f(\timefunc', \timefunc')}^2 \di \timefunc'\nonumber\\
\leq {}&\int_{\timefunc}^{\timefunc+\timefunc''}(\timefunc'+\zeta')^{\gamma}\abs{f(\timefunc', \timefunc'+\zeta')}^2 \di \timefunc'
+ \int_{\timefunc}^{\timefunc+\timefunc''}
\int_{\timefunc'}^{4\timefunc''}\abs{\TparallelForRethreadedHardy (r^{\gamma}\abs{f(\timefunc', r)}^2) }\di r \di \timefunc' \nonumber\\
\leq{}& C(\gamma) \int_{\timefunc}^{\infty}\int_{\timefunc'}^{\infty} r^{\gamma-1}(\abs{f(\timefunc', r)}^2+\abs{r\TparallelForRethreadedHardy(f(\timefunc', r))}^2)\di r \di \timefunc' .
\end{align}
Letting $\timefunc''$ go to infinity proves the estimate \eqref{eq:Transitionfluxcontrolledbybulkexterior}.
\end{proof}

\begin{lemma}[Taylor expansion in $L^2$]
\label{thm:Taylor-est}
Let $A>0$, $n\in \Naturals$, $f\in C^{n+1}([0,A])$, and
\begin{align}
P_n(x)={}& \sum_{k=0}^n \frac{x^k}{k!}f^{(k)}(0).
\end{align}
Then
for any $-1 < \alpha < 1$, there exists a constant $C=C(n,\alpha)$ such that
\begin{align}
\Bigl\Vert \frac{f(x)-P_n(x)}{x^{n+1+\alpha/2}} \Bigr\Vert_{L^2((0,A))}
\leq{}&C\Vert x^{-\alpha/2} f^{(n+1)}\Vert_{L^2((0,A))}.
\label{eq:generalremainderestimate}
\end{align}
\end{lemma}

\begin{proof}
From the assumptions on the function $f(x)$, we have for any integer $0\leq i\leq n$, there exist constants $C(n,i)$ such that
\begin{align}
\lim_{x\rightarrow 0^+} \frac{\partial_x^i (f(x)-P_n(x))}{x^{n+1-i}}=C(n,i)f^{(n+1)}(0).
\label{eq:RemainderAssupconsequence}
\end{align}
Given any integer $0\leq i\leq n$, we do the replacements
\begin{align}
(r,r_0,r_1,h(r),\gamma)\mapsto (x,0,A,\partial_x^{i}(f(x)-P_n(x)),-2n+2i-1-\alpha)
\end{align}
in point \eqref{point:lem:HardyIneqLHS} of lemma \ref{lem:HardyIneq}, and note from the assumption $\alpha \in (-1,1)$ and the fact \eqref{eq:RemainderAssupconsequence} that
\begin{subequations}
\begin{align}
\gamma=-2n+2i-1-\alpha<{}&0,\\
\lim_{x\rightarrow 0^+}  x^{-2n+2i-1-\alpha}( \partial_x^{i}(f(x)-P_n(x)))^2 ={}&0.
\end{align}
\end{subequations}
Therefore, it follows from point \eqref{point:lem:HardyIneqLHS} of lemma \ref{lem:HardyIneq} that for any integer $0\leq i\leq n$ and any $\alpha\in (-1,1)$,
\begin{align}
\int_{0}^A \frac{( \partial_x^{i}(f(x)-P_n(x)))^2}{x^{2n-2i+2+\alpha}}\di x \leq C(n,i,\alpha)\int_{0}^A \frac{( \partial_x^{i+1}(f(x)-P_n(x)))^2}{x^{2n-2i+\alpha}}\di x.
\label{eq:remaindergeneralj}
\end{align}
Thus, by induction, one finds, for $i\in\{0,\ldots,n\}$, that 
\begin{align}
\| x^{-n-1-\alpha/2}(f-P_n)\|_{L^2((0,A))} 
\lesssim{}& \|x^{-n+i-\alpha/2} \partial_x^{i+1}(f-P_n)\|_{L^2((0,A))}.
\end{align} 
The case $i=n$ gives the desired result. 
\end{proof}

\begin{lemma}
\label{lem:Ik+1alphaenergydominatesPkalphapointwise}
For a spin-weighted scalar $\varphi$ and for any $\ireg\in\Naturals$ and $\alpha\in \Reals$, there is the bound
\begin{align}
\inipointwise{\ireg}{\alpha}(\varphi) 
\lesC{\alpha}{}&  \inienergy{\ireg+1}{\alpha}(\varphi).
\label{eq:Ik+1alphaenergydominatesPkalphapointwise}
\end{align}
\end{lemma}

\begin{proof} 
From the definition of $\inienergy{\ireg+1}{\alpha}$ and commuting $r$ through the $\unrescaledOps$ derivatives, it follows that
\begin{align}
\inienergy{\ireg+1}{\alpha}(\varphi) 
={}& \sum_{|\mathbf{a}|\leq \ireg+1} \int_{\Stauini} M^{-\alpha}r^{\alpha+2|\mathbf{a}|-1}\abs{\unrescaledOps^{\mathbf{a}} \varphi}^2 \diThreeVol \nonumber\\
\gtrC{\alpha}{}& \sum_{|\mathbf{a}| \leq \ireg+1} \int_{\Stauini} M^{-\alpha}r^2 
\abs{\unrescaledOps^{\mathbf{a}}(r^{\alpha/2+|\mathbf{a}|-3/2}\varphi)}^2
 \diThreeVol . 
\end{align}

There are two important consequences of this. First, one finds, from ignoring the case $|\mathbf{a}|=0$  and the divergence of $\int r^{-1}\di r$, that 
\begin{align}
r  \sum_{|\mathbf{a}| \leq \ireg} \int_{\Sphere} M^{-\alpha}r^2 \abs{\unrescaledOps^{\mathbf{a}} (r^{\alpha/2+|\mathbf{a}|-3/2}\varphi)}^2 \diThreeVol\rightarrow 0 
\label{eq:Ik+1ControlPk:sequentialStep}
\end{align} 
as $r\rightarrow\infty$, at least along some sequence. Before considering the second, observe that there is a vector field $\TparallelForRethreadedHardy^a$ that is parallel to $\Stauini$ and the corresponding operator $\TparallelForRethreadedHardy$ has an expansion solely in terms of $\VOp$ and $\YOp$ with $\bigOAnalytic(1)$ coefficients. As in lemma \ref{lem:HardyOnStautRc}, this can be used to define a radial coordinate $\tilde{r}$ such that, in the Znajek tetrad, $\TparallelForRethreadedHardy = \partial_{\tilde{r}}$ on $\Stauini$. Thus, there is the second observation that
\begin{align}
\inienergy{\ireg+1}{\alpha}(\varphi) 
\gtrC{\alpha}{}&  \sum_{|\mathbf{a}|\leq \ireg} \int_{\Stauini} M^{-\alpha}r^2 \abs{\TparallelForRethreadedHardy \unrescaledOps^{\mathbf{a}} (r^{\alpha/2+|\mathbf{a}|-1/2}\varphi)}^2 \diThreeVol . 
\end{align}
where we have taken into account the shift in $i$. Now applying the pointwise control in point \ref{point:lem:HardyIneqRHS} of Lemma \ref{lem:HardyIneq} with $\gamma=1$, and using the limit \eqref{eq:Ik+1ControlPk:sequentialStep} to drop the right endpoint, one concludes for any $(\timefunc,r,\omega)\in \Stauini$, 
\begin{align}
\inienergy{\ireg+1}{\alpha}(\varphi) 
\gtrC{\alpha}{}& r\sum_{|\mathbf{a}| \leq \ireg} M^{-\alpha} \int_{\Sphere}  \abs{\unrescaledOps^{\mathbf{a}} r^{\alpha/2+|\mathbf{a}|-1/2}\varphi(\timefunc,r,\omega)} ^2
\diThreeVol \nonumber\\
\gtrC{\alpha}{}& \sum_{|\mathbf{a}| \leq \ireg} M^{-\alpha} r^{\alpha+2|\mathbf{a}|}\int_{\Sphere}  \abs{\unrescaledOps^{\mathbf{a}} \varphi(\timefunc,r,\omega)}^2 \diThreeVol . 
\end{align}
By taking the supremum in $r\in [r_+,\infty)$ and $\timefunc=\timefunc_0-h(r)/2$, this completes the proof. 
\end{proof}

\section{Weighted energy estimates} \label{sec:weightedenergy}
\subsection{A hierarchy of pointwise and integral estimates implies decay}
\label{sec:HierarchyImpliesDecay}
This subsection provides some simple lemmas for treating hierarchies of decay estimates. Such hierarchies arise both in the analysis of the Teukolsky equation and in the analysis of transport equations. The proof of these results relies on the (continuous) pigeonhole principle. 

For transport equations, the hierarchy of estimates is generally fairly straightforward, with a weighted integral of a solution being controlled by a weighted integral of a source. However, for wave-like equations, such as the Teukolsky equation, one finds that the weighted integral of a function at one level of regularity is estimated in terms of another weighted integral at a different level of regularity. For this reason, lemma \ref{lem:hierarchyImpliesDecay} involves a function $f(\iPigeonReg,\alpha,\PigeonTime)$, which should be thought of as being an integral involving a regularity $\iPigeonReg$, a weight $\alpha$, and a time $\PigeonTime$. 

The following lemma uses a single application of the pigeonhole principle and is used in the proof of lemma \ref{lem:hierarchyImpliesDecay}. 

\begin{lemma}[Single step]
\label{lem:HierarchyImpliesDecay:OneStep} Let $f:\{-1,0,1\}\times[\PigeonTime_0,\infty)\rightarrow[0,\infty)$ be such that $f(\iPigeonReg,\PigeonTime)$ is Lebesgue measurable in $\PigeonTime$ for each $\iPigeonReg$. If there is a $D\geq0$ and $\alpha\in \Reals$  such that, for all $\iPigeonReg\in\{0,1\}$ and $\PigeonTime_2\geq\PigeonTime_1\geq \PigeonTime_0$, 
\begin{align}
f(\iPigeonReg,\PigeonTime_2)
+\int_{\PigeonTime_1}^{\PigeonTime_2} f(\iPigeonReg-1,\PigeonTime) \di \PigeonTime
\lesssim f(\iPigeonReg,\PigeonTime_1) +\PigeonTime_1^{\alpha+\iPigeonReg}D , 
\label{eq:Rev:HierarchyToDecay:Hypothesis}
\end{align}
then, for all $\PigeonTime\geq 2 \PigeonTime_0$, 
\begin{align}
f(0,\PigeonTime)
\lesC{\alpha}{}& \PigeonTime^{-1} f(1,\PigeonTime/2)+\PigeonTime^{\alpha}D .
\label{eq:Rev:HierarchyToDecay}
\end{align}
\end{lemma}
\begin{proof}
From the mean-value principle, for any $\timefunc\geq2\timefunc_0$, there is a $\tilde{\PigeonTime}\in[\PigeonTime/2,\PigeonTime]$ such that
\begin{align}
f(0,\tilde{\PigeonTime})
\leq{}&\frac{2}{\PigeonTime}\int_{\PigeonTime/2}^{\PigeonTime} f(0,\PigeonTime') \di\PigeonTime' .
\end{align}
Combining this with the integral estimate for $\iPigeonReg=1$ in hypothesis \eqref{eq:Rev:HierarchyToDecay:Hypothesis}, one can control $f$ at $\tilde\PigeonTime$ by
\begin{align}
f(0,\tilde{\PigeonTime})
\lesssim{}& 2\PigeonTime^{-1} 
\left(f(1,\PigeonTime/2) +(\PigeonTime/2)^{\alpha+1}D\right)
\lesC{\alpha} \PigeonTime^{-1}f(1,\PigeonTime/2) +\PigeonTime^\alpha D .
\label{eq:SinglePigeonMiddle}
\end{align}
From the pointwise estimate for $\iPigeonReg=0$ in hypothesis \eqref{eq:Rev:HierarchyToDecay:Hypothesis}, one can control $f$ at $\PigeonTime$ by
\begin{align}
f(0,\PigeonTime)
\lesssim{}&f(0,\tilde{\PigeonTime}) +\tilde{\PigeonTime}^\alpha D
\lesC{\alpha} f(0,\tilde{\PigeonTime}) +\PigeonTime^\alpha D.
\label{eq:SinglePigeonFinal}
\end{align}
The lemma follows from combining estimates \eqref{eq:SinglePigeonMiddle} and \eqref{eq:SinglePigeonFinal}. 
\end{proof}

The following lemma proves that a hierarchy of decay estimates implies a decay rate for the terms in the hierarchy. In applications, $\iPigeonReg$ represents a level of regularity, $\alpha$ represents a weight, and $\PigeonTime$ represents a time coordinate. The weights take values in an interval, whereas the levels of regularity are discrete. 

\begin{lemma}[A hierarchy of estimates implies decay rates]
\label{lem:hierarchyImpliesDecay}
Let $D\geq 0$. Let $\alpha_1,\alpha_2\in\Reals$ and $i\in\Integers^+$ be such that $\alpha_1\leq\alpha_2-1$, and $\alpha_2-\alpha_1\leq i$. Let $F:\{-1,\ldots,i\}\times[\alpha_1-1,\alpha_2]\times[\timefunc_0,\infty)\rightarrow[0,\infty)$ be such that $F(\iPigeonReg,\alpha,\PigeonTime)$ is Lebesgue measurable in $\PigeonTime$ for each $\alpha$ and $\iPigeonReg$. Let $\gamma\geq 0$. 

If 
\begin{subequations}
\begin{enumerate}
\item{} [monotonicity] \label{assump:HierarchyToDecay(1)}for all $\iPigeonReg,\iPigeonReg_1,\iPigeonReg_2\in\{-1,\ldots,i\}$ with $\iPigeonReg_1\leq \iPigeonReg_2$, all $\beta, \beta_1,\beta_2\in[\alpha_1,\alpha_2]$ with $\beta_1\leq\beta_2$, and all $\PigeonTime\geq \timefunc_0$, 
\begin{align}
F(\iPigeonReg_1,\beta,\PigeonTime)\lesssim{}& F(\iPigeonReg_2,\beta,\PigeonTime) ,
\label{eq:Rev:HierarchyToDecayReal:MonotonicityHypothesis:j}\\
F(\iPigeonReg,\beta_1,\PigeonTime)\lesssim{}& F(\iPigeonReg,\beta_2,\PigeonTime) ,
\label{eq:Rev:HierarchyToDecayReal:MonotonicityHypothesis:beta}
\end{align}
\item{} [interpolation] \label{assump:HierarchyToDecay(2)}for all $\iPigeonReg\in\{-1,\ldots,i\}$, all $\alpha,\beta_1,\beta_2\in[\alpha_1,\alpha_2]$ such that $\beta_1\leq\alpha\leq\beta_2$, and all $\PigeonTime\geq \timefunc_0$,
\begin{align}
F(\iPigeonReg,\alpha,\PigeonTime)
\lesssim{}& 
F(\iPigeonReg,\beta_1,\PigeonTime)^{\frac{\beta_2-\alpha}{\beta_2-\beta_1}}
F(\iPigeonReg,\beta_2,\PigeonTime)^{\frac{\alpha-\beta_1}{\beta_2-\beta_1}} ,
\label{eq:Rev:HierarchyToDecayReal:InterpolationHypothesis}
\end{align}
\item{} [energy and Morawetz estimate] for all $\iPigeonReg\in\{0,\ldots,i\}$, $\alpha\in[\alpha_1,\alpha_2]$, and $\PigeonTime_2\geq \PigeonTime_1\geq \timefunc_0$,
\begin{align}
F(\iPigeonReg,\alpha,\PigeonTime_2)
+\int_{\PigeonTime_1}^{\PigeonTime_2} F(\iPigeonReg-1,\alpha-1,t) \di \PigeonTime
\lesssim F(\iPigeonReg,\alpha,\PigeonTime_1) +D \PigeonTime_1^{\alpha-\alpha_2-\gamma} , 
\label{eq:Rev:HierarchyToDecayReal:EvolutionHypothesis}
\end{align}
and 
\item{} [initial decay rate] if $\gamma>0$, then for any  $\PigeonTime\geq \timefunc_0$, 
\begin{align}
F(i,\alpha_2,\PigeonTime)
\lesssim \PigeonTime^{-\gamma} \left(F(i,\alpha_2,\timefunc_0) +D\right) , 
\label{eq:Rev:HierarchyToDecayReal:InitialDecay}
\end{align}
\end{enumerate}
\end{subequations}
then, for all $\iPigeonReg\in\{0,\ldots,i\}$, all $\alpha\in[\max\{\alpha_1,\alpha_2-\iPigeonReg\},\alpha_2]$, and all $\PigeonTime\geq \timefunc_0$, 
\begin{align}
F(i-\iPigeonReg,\alpha,\PigeonTime) \lesssim{}& \PigeonTime^{\alpha-\alpha_2-\gamma}  (F(i,\alpha_2,\timefunc_0) +D),
\label{eq:Rev:HierarchyToDecayReal}
\end{align}
where the implicit constant in $\lesssim$ can depend on $\alpha_2$ and $\alpha_1$. 
\end{lemma}

\begin{proof}
Let $I=\lfloor\alpha_2-\alpha_1\rfloor\geq 1$. If $\gamma=0$, then from the energy hypothesis \eqref{eq:Rev:HierarchyToDecayReal:EvolutionHypothesis}, one finds that the initial decay hypothesis estimate \eqref{eq:Rev:HierarchyToDecayReal:InitialDecay} holds. Thus, in all cases, one finds for $\PigeonTime\geq \timefunc_0$, 
\begin{align}
F(i,\alpha_2,\PigeonTime)
\lesssim{}&\PigeonTime^{-\gamma}\left(F(i,\alpha_2,\timefunc_0)+D\right) .
\label{eq:Rev:HierarchyToDecayReal:AtTop}
\end{align}

First, consider $\alpha_2-\alpha\in\Naturals$. For $\iPigeonReg\in\{1,\ldots,I\}$ and $k\in\{0,1\}$, observe that $F(i-\iPigeonReg+k,\alpha_2-\iPigeonReg+k,\PigeonTime)$ satisfy 
\begin{align}
F(i-\iPigeonReg+k,\alpha_2-\iPigeonReg+k,\PigeonTime_2)
&+\int_{\PigeonTime_1}^{\PigeonTime_2}F(i-\iPigeonReg+k-1,\alpha_2-\iPigeonReg+k-1,\PigeonTime')\di\PigeonTime' \nonumber\\
\lesssim{}&F(i-\iPigeonReg+k,\alpha_2-\iPigeonReg+k,\PigeonTime_1)
+\PigeonTime_1^{-\gamma-\iPigeonReg+k} D.
\end{align} 
This combined with lemma \ref{lem:HierarchyImpliesDecay:OneStep} implies, for $\timefunc>2\timefunc_0$, 
\begin{align}
F(i-\iPigeonReg,\alpha_2-\iPigeonReg,\PigeonTime)
\lesssim{}& \PigeonTime^{-1}F(i-\iPigeonReg+1,\alpha_2-\iPigeonReg+1,\PigeonTime/2) +\PigeonTime^{-1-\gamma-\iPigeonReg}D .
\label{eq:Rev:HierarchyToDecayReal:OneStep}
\end{align} 
By induction, taking equation \eqref{eq:Rev:HierarchyToDecayReal:AtTop} as the base case and estimate \eqref{eq:Rev:HierarchyToDecayReal:OneStep} to justify the inductive step, we that, for all $\iPigeonReg\in\{0,\ldots,I\}$ and $\PigeonTime\geq \timefunc_0$, there is the bound
\begin{align}\label{eq:induction}
F(i-\iPigeonReg,\alpha_2-\iPigeonReg,\PigeonTime)\lesssim \PigeonTime^{-\iPigeonReg-\gamma}(F(i,\alpha_2,\timefunc_0)+D) .
\end{align}
The same bound holds for $\timefunc\in[\timefunc_0,2\timefunc_0]$ from \eqref{eq:Rev:HierarchyToDecayReal:OneStep} for $\PigeonTime\geq 2 \timefunc_0$ and from the basic energy hypothesis \eqref{eq:Rev:HierarchyToDecayReal:EvolutionHypothesis} and the monotonicity hypotheses \eqref{eq:Rev:HierarchyToDecayReal:MonotonicityHypothesis:j}-\eqref{eq:Rev:HierarchyToDecayReal:MonotonicityHypothesis:beta}. 

Now, consider the case $\alpha\geq \alpha_2-\lfloor\alpha_2-\alpha_1\rfloor$. Consider $\iPigeonReg\in\{0,\ldots,I\}$ and $\zeta\in[0,\iPigeonReg]$. From the interpolation hypothesis \eqref{eq:Rev:HierarchyToDecayReal:InterpolationHypothesis} with $\alpha \rightarrow \alpha_2-\zeta$, $\beta_1 \rightarrow \alpha_2-\iPigeonReg$, $\beta_2 \rightarrow \alpha$, we get that for all $\PigeonTime\geq \timefunc_0$, 
\begin{align}\label{eq:apply-interpolation}
F(i-\iPigeonReg,\alpha_2-\zeta,\PigeonTime)
\lesssim{}&F(i-\iPigeonReg,\alpha_2-\iPigeonReg,\PigeonTime)^{\frac{\zeta}{\iPigeonReg}}F(i-\iPigeonReg,\alpha_2,\PigeonTime)^{\frac{\iPigeonReg-\zeta}{\iPigeonReg}}\nonumber\\
\lesssim{}& \PigeonTime^{-\zeta-\gamma}(F(i,\alpha_2,\timefunc_0)+D)^{\frac{\zeta}{\iPigeonReg}}F(i,\alpha_2,\PigeonTime)^{\frac{\iPigeonReg-\zeta}{\iPigeonReg}}\nonumber\\
\lesssim{}& \PigeonTime^{-\zeta-\gamma}(F(i,\alpha_2,\timefunc_0)+D) .
\end{align}
Making the substitution $k=i-\iPigeonReg$ and $\alpha=\alpha_2-\zeta\geq\alpha_2-\iPigeonReg\geq \alpha_2+k-i$ and using the monotonicity hypothesis \eqref{eq:Rev:HierarchyToDecayReal:MonotonicityHypothesis:beta}, one finds, for $k\in\{i-I,\ldots,i\}$, $\alpha\in[\alpha_2+k-i,\alpha_2]$, and $\PigeonTime\geq\timefunc_0$, there is the bound
\begin{align}
F(k,\alpha,\PigeonTime)\lesssim{}&\PigeonTime^{\alpha-\alpha_2-\gamma}(F(i,\alpha_2,\timefunc_0)+D) .
\label{eq:Hierarchyimpliesdecay:largeklargealpha}
\end{align}
This gives the desired estimate for the cases $k\geq i-I$. 

Finally, consider $\alpha<\alpha_2-\lfloor\alpha_2-\alpha_1\rfloor$. Since $\alpha_2-\alpha_1<i$, one finds $I< i$, and since $\alpha_1+1\geq \alpha_2+(i-I)-i$, we have from the conclusion of the previous paragraph
\begin{align}
F(i-I,\alpha_1+1,\PigeonTime)
\lesssim{}& \PigeonTime^{\alpha_1+1-\alpha_2-\gamma}(F(i,\alpha_2,\timefunc_0)+D). 
\end{align}
Combining this with the energy and Morawetz hypothesis \eqref{eq:Rev:HierarchyToDecayReal:EvolutionHypothesis} and lemma \ref{lem:HierarchyImpliesDecay:OneStep}, one finds
\begin{align}
F(i-I-1,\alpha_1,\PigeonTime)
\lesssim{}& \PigeonTime^{\alpha_1-\alpha_2-\gamma}(F(i,\alpha_2,\timefunc_0)+D) .
\end{align}
Interpolation now gives for all $\alpha\in[\alpha_1,\alpha_1+1]$ and $\PigeonTime\geq \timefunc_0$
\begin{align}
F(i-I-1,\alpha,\PigeonTime)
\lesssim{}&\PigeonTime^{\alpha-\alpha_2-\gamma}(F(i,\alpha_2,\timefunc_0)+D) .
\end{align}
This combined with \eqref{eq:Hierarchyimpliesdecay:largeklargealpha} implies for all $\iPigeonReg\in\{0,\ldots,I+1\}$, all $\alpha\in[\max\{\alpha_1,\alpha_2-\iPigeonReg\},\alpha_2]$, and all $\PigeonTime\geq \timefunc_0$, 
\begin{align}
F(i-\iPigeonReg,\alpha,\PigeonTime) \lesssim{}& \PigeonTime^{\alpha-\alpha_2-\gamma}  (F(i,\alpha_2,\timefunc_0) +D).
\end{align} 
The other monotonicity hypothesis \eqref{eq:Rev:HierarchyToDecayReal:MonotonicityHypothesis:j} then gives 
the desired estimate in the remaining cases. 
\end{proof}

The following lemma states the $W^{\iPigeonReg}_\alpha$ norms squared satisfy the monotonicity and interpolation conditions for $f(\iPigeonReg,\alpha,\PigeonTime)$ in lemma \ref{lem:hierarchyImpliesDecay}. 

\begin{lemma}
\label{lem:Walphaknormsproperties}
Let $\iPigeonReg,\iPigeonReg_1,\iPigeonReg_2\in\Naturals$ and $\alpha,\beta,\beta_1,\beta_2\in\Reals$. Let $\varphi$ be a spin-weighted scalar. Let $\timefunc,\timefunc_1,\timefunc_2\in[\timefunc_0,\infty)$. 

\begin{enumerate}
\item{} [monotonicity] If $\iPigeonReg_1\leq \iPigeonReg_2$ and $\beta_1\leq \beta_2$, then
\begin{subequations}
\begin{align}
\| \varphi\|_{W^{\iPigeonReg_1}_{\beta}(\Sigma_\timefunc)}^2 \lesssim{}& \| \varphi\|_{W^{\iPigeonReg_2}_{\beta}(\Sigma_\timefunc)}^2 , \\
\| \varphi\|_{W^{\iPigeonReg}_{\beta_1}(\Sigma_\timefunc)}^2 \lesssim{}& \| \varphi\|_{W^{\iPigeonReg}_{\beta_2} (\Sigma_\timefunc)}^2.
\end{align}
\end{subequations}
\item{} [interpolation] If $\beta_1\leq\alpha\leq\beta_2$, then
\begin{align}
\| \varphi\|_{W^{\iPigeonReg}_{\alpha}(\Sigma_\timefunc)}
\lesssim \| \varphi\|_{W^{\iPigeonReg}_{\beta_1}(\Sigma_\timefunc)}^{\frac{\alpha-\beta_1}{\beta_2-\beta_1}}
\| \varphi\|_{W^{\iPigeonReg}_{\beta_2}(\Sigma_\timefunc)}^{\frac{\beta_2-\alpha}{\beta_2-\beta_1}} .
\end{align}
\item{} [relation of spatial and spacetime norms] 
\begin{align}
\|\varphi\|_{W^{\iPigeonReg}_{\beta}(\Dtau)}^2
={}&M^{-1} \int_{\timefunc_1}^{\timefunc_2} \|\varphi\|_{W^{\iPigeonReg}_{\beta}(\Sigma_\timefunc)}^2 \di\timefunc .
\end{align}
\end{enumerate}
\end{lemma}
\begin{proof}
The first monotonicity result follows from summing fewer non-negative terms. The second monotonicity result follows from the fact that $\beta_1\leq \beta_2$ implies $r^{\beta_1}\lesssim r^{\beta_2}$. The interpolation result follows from H\"o{}lder's inequality. The relation between the spatial and spacetime norms follows from the definition of $\diThreeVol$ and $\diFourVol$. 
\end{proof}

\subsection{Spin-weighted transport equations}
\label{sec:EnergyMoraTwoTransEqs}

Now we state a general lemma which provides energy and Morawetz estimates for the ingoing transport equation with source term satisfying energy and Morawetz estimates. 

\begin{lemma}[\texorpdfstring{$\YOp$}{Y} estimate]
\label{lem:TransEqGeneralRegionMora}
Let $\gamma\in(0,\infty)$ and $k\in\Naturals$. Let $b_0(r)$ be a non-negative, smooth function defined in $\mathcal{M}$ such that $b_0(r)=M\bigOAnalytic(r^{-1})$. 

If $\varphi$ and $\varrho$ are scalars with spin weight $s$ and $\varphi$ satisfies
\begin{equation}\label{eq:AssuTransIngoing}
M\YOp\varphi+b_0(r)\varphi=\varrho,
\end{equation}
then for all $\timefunc_2>\timefunc_1\geq \timefunc_0$,
\begin{subequations}
\label{eq:TransIngoingMoraall}
\begin{align}
\Vert \varphi \Vert^2_{W_{\gamma}^k(\Sigma_{\timefunc_2})} +\Vert \varphi \Vert^2_{W_{\gamma-1}^k(\Dtau)}
\lesC{s}{}&\Vert \varphi \Vert^2_{W_{\gamma}^k(\Staui)} + \Vert \varrho \Vert^2_{W_{\gamma+1}^k(\Dtau)},
\label{eq:TransIngoingMora}\\
\Vert \varphi \Vert^2_{W^k_\gamma (\Stauint)}+\Vert \varphi \Vert^2_{W^k_{\gamma-1} (\Dtauint)}
\lesC{s}{}&
\Vert \varphi \Vert^2_{W^k_\gamma (\Stauiint)}+\Vert \varrho \Vert^2_{W^k_{\gamma+1} (\Dtauint)}+\Vert \varphi \Vert^2_{W^k_\gamma (\Boundext)},
\label{eq:TransIngoingMoraint}\\
\Vert \varphi \Vert^2_{W^k_\gamma (\Boundext)}+\Vert \varphi \Vert^2_{W^k_\gamma (\Stauext)}+\Vert \varphi \Vert^2_{W^k_{\gamma-1} (\Dtauext)}
\lesC{s}{}&
\Vert \varphi \Vert^2_{W^k_\gamma (\Stauiext)}+\Vert \varrho \Vert^2_{W^k_{\gamma+1} (\Dtauext)},
\label{eq:TransIngoingMoraext} \\
\Vert \varphi \Vert^2_{W^k_\gamma (\Stauitext)}+\Vert \varphi \Vert^2_{W^k_{\gamma-1} (\Dtauearly)}
\lesC{s}{}&
\Vert \varphi \Vert^2_{W^k_\gamma (\Stauini)}+\Vert \varrho \Vert^2_{W^k_{\gamma+1} (\Dtauearly)},
\label{eq:TransIngoingMorafar}
\end{align}
\end{subequations}
and, for $\timefunc \geq \timefunc_0 + h(\timefunc_0)$ as in definition \ref{def:tprime},
\begin{align}
\norm{\varphi}_{W^{k}_{\gamma} (\Horgtt{\timefunc-h(r_+)})}^2
+
\norm{\varphi}_{W^{k}_{\gamma}(\Stautimeint)}^2
+ \norm{\varphi}_{W^{k}_{\gamma-1}(\Dtautnear{\timefunc})}^2 
\lesC{s}{}&\norm{\varphi}_{W^{k}_{\gamma} (\Boundtingoing)}^2
+ \norm{\varrho}_{W^{k}_{\gamma+1}(\Dtautnear{\timefunc})}^2.
\label{eq:TransIngoingMoranearregion}
\end{align}
The implicit constants in the above estimates depend on only $\gamma$ and $k$.
\end{lemma}
\begin{proof} 
Consider the case $k=0$ first. Multiplying \eqref{eq:AssuTransIngoing} by $M^{-1}r^{\gamma}\bar{\varphi}$, taking the real part, and applying the Cauchy-Schwarz inequality, one obtains
\begin{align}
\YOp(\tfrac{1}{2}r^{\gamma}|\varphi|^2)+(\tfrac{\gamma}{2}+M^{-1}rb_0(r))r^{\gamma-1}|\varphi|^2
={}&r^{\gamma}\Re\{M^{-1}\varrho\bar{\varphi}\} \nonumber\\
\leq{}& \tfrac{\gamma}{4}r^{\gamma-1}|\varphi|^2+\gamma^{-1}M^{-2}r^{\gamma+1}|\varrho|^2.
\end{align}
Absorbing the $|\varphi|^2$ on the right in to the left and multiplying with $M^{-\gamma-1}$ gives
\begin{align}
\YOp(\tfrac{1}{2}M^{-\gamma-1}r^{\gamma}|\varphi|^2)+(\tfrac{\gamma}{4}+M^{-1}rb_0(r))M^{-(\gamma-1)-2}r^{\gamma-1}|\varphi|^2
\leq{}&\gamma^{-1}M^{-(\gamma+1)-2}r^{\gamma+1}|\varrho|^2.
\label{eq:YTransportk=0}
\end{align}
The energy and Morawetz estimate \eqref{eq:TransIngoingMora} for $k=0$ then follows from  integrating over $\Dtau$ with the measure $\diFourVol$ and the fact that $\di \timefunc_a \YOp^a=h'(r)$ is uniformly equivalent to $1$. Here, we have dropped the positive flux at future null infinity.
 In an analogous way, the $k=0$ case of the remaining estimates in \eqref{eq:TransIngoingMoraall} follows by integrating \eqref{eq:YTransportk=0} with the measure $\diFourVol$ over $\Dtauint$, $\Dtauext$, and $\Dtauearly$, respectively, and the $k=0$ case of \eqref{eq:TransIngoingMoranearregion} follows from integrating with the measure $\diFourVol$ over $\Dtautnear{\timefunc}$ and 
$\Dtautnear{\timefunc} \cap\{\timefunc'\leq \timefunc\}$ such that the first integration gives the first and third term on the left of \eqref{eq:TransIngoingMoranearregion} and the second integration gives the second term on the left. Here, we made use of the facts that $(\di\timefunc_a-\di r_a) \YOp^a=1+h'(r)$ and $\di \timefunc_a \YOp^a=h'(r)$ are both uniformly equivalent to $1$ and $\di v_a \YOp^a=0$.

Now assume the result holds for some $k\geq 0$. From the fact that the operators $M\LxiOp$, $\hedt$, $\hedtp$ commute with $Y$, it follows that the estimates \eqref{eq:TransIngoingMoraall} and \eqref{eq:TransIngoingMoranearregion} hold for $k$ but with $(\varphi,\varrho)$ replaced by these derivatives operated on $(\varphi,\varrho)$. Hence, the estimates \eqref{eq:TransIngoingMoraall} and \eqref{eq:TransIngoingMoranearregion} hold for $k$ but with $(\varphi,\varrho)$ replaced by any of $\{(\varphi,\varrho), (M\LxiOp \varphi,M\LxiOp \varrho), (\hedt \varphi,\hedt \varrho), (\hedtp\varphi, \hedtp \varrho)\}$.
 
 If we commute \eqref{eq:AssuTransIngoing} with $\VOp(r\cdot)$, then, because of the first relation in
 \eqref{eq:CommutatorofYandMathcalV}, we have 
\begin{align}\label{eq:TransIngoingCommuteOneDeri}
\hspace{4ex}&\hspace{-4ex}M\YOp\VOp(r\varphi)+(\tfrac{M}{r}+b_0(r))\VOp(r\varphi)\nonumber\\
={}&\VOp(r\varrho)+\tfrac{Mr(r^2-a^2)}{(r^2+a^2)^2}\varrho
+\left(\tfrac{\Delta}{2(r^2+a^2)}(M-r^2\partial_r(b_0(r)))-\tfrac{rM(r^2-a^2)}{(r^2+a^2)^2}(M+r b_0(r))\right)\tfrac{\varphi}{r}
+\tfrac{2a M r^2 \LetaOp\varphi}{(r^2+a^2)^2}\nonumber\\
={}&\VOp(r\varrho) +\tfrac{Mr(r^2-a^2)}{(r^2+a^2)^2}\varrho + M\bigOAnalytic(r^{-1}) \varphi +M^2\bigOAnalytic(r^{-2})\LetaOp\varphi .
\end{align}
This equation is in the form of equation of \eqref{eq:AssuTransIngoing}, so it remains to control the $W^k_{\gamma+1}(\Omega)$ norm squared of the right-hand side of \eqref{eq:TransIngoingCommuteOneDeri}, $\Omega$ being the region $\Dtau$, $\Dtauint$, $\Dtauext$, $\Dtauearly$, or $\Dtautnear{\timefunc}$, which one integrates over. The $W^k_{\gamma+1}(\Omega)$ norm squared of the first two terms is clearly bounded by the $W^{k+1}_{\gamma+1}(\Omega)$ norm squared of $\varrho$ itself. 
The last two terms are bounded by 
\begin{align}
\hspace{4ex}&\hspace{-4ex} \norm{Mr^{-1}\varphi}^2_{W_{\gamma+1}^k(\Omega)} 
+\norm{{M^2r^{-2}\LetaOp\varphi}}^2_{W_{\gamma+1}^k(\Omega)}\nonumber\\
&{}\lesssim\norm{{\varphi}}^2_{W_{\gamma-1}^k(\Omega)}
+\norm{\LetaOp\varphi}^2_{W_{\gamma-3}^k(\Omega)}\nonumber\\
&{}\lesC{s} \norm{{\varphi}}^2_{W_{\gamma-1}^k(\Omega)}
+\norm{\hedt\varphi}^2_{W_{\gamma-1}^k(\Omega)}
+\norm{\hedtp\varphi}^2_{W_{\gamma-1}^k(\Omega)}.
\end{align}
Therefore, the right-hand side is bounded by
\begin{align}
\norm{f}^2_{W_{\gamma+1}^{k+1}(\Omega)}
+\norm{{\varphi}}^2_{W_{\gamma-1}^k(\Omega)}
+\norm{\hedt\varphi}^2_{W_{\gamma-1}^k(\Omega)}
+\norm{\hedtp\varphi}^2_{W_{\gamma-1}^k(\Omega)}.
\end{align}
We have estimates for the last four terms from the previous paragraph, and by adding those estimates, the desired estimates \eqref{eq:TransIngoingMoraall} and \eqref{eq:TransIngoingMoranearregion} hold for $k+1$. By induction, the estimates \eqref{eq:TransIngoingMoraall} and \eqref{eq:TransIngoingMoranearregion} hold for all $k\in\Naturals$. 
\end{proof}

\subsection{Spin-weighted wave equations}
\label{sec:rpForSpinWeightedWaveEquations}

The following is a standard $r^p$ argument following the ideas originally given in \cite{DafermosRodnianski:rp}. Essentially one uses the vector-field method with the vector $M(1+M^\delta r^{-\delta})\YOp  +M^{-\alpha+1}r^\alpha \VOp $ with $\delta>0$ small and $\alpha\in [\delta,2-\delta]$. Since we use $p$ for a spinoral weight, we use $\alpha$ for the exponent traditionally denoted by $p$ in the $r^p$ argument. 

There are some technical differences between the statement here and similar results appearing elsewhere, such as \cite{DafermosRodnianski:rp}. Largely, these differences arise because we are not working in spherical symmetry. First, we work with a hyperboloidal foliation rather than a null foliation. Most obviously, this means that the coefficient of the ingoing derivatives $\abs{\YOp\varphi}^2$ in the energy (i.e. the flux integral on hypersurfaces) is strictly positive in equation \eqref{eq:energyComponentAYPrincipal}. Second, we restrict the exponent $\alpha$ to lie in the interval $[\delta,2-\delta]$ rather than $[0,2]$. It is well known that, in the space-time bulk integral, the coefficient of the outgoing derivatives $\abs{\VOp \varphi}^2$ (appearing in equation \eqref{eq:rpEstimate:VSquaredTerm}) and the coefficient of the angular derivatives (appearing in \eqref{eq:rpEstimate:AngularSquaredTerm}) have factors of $\alpha$ and $2-\alpha$ respectively. We restrict the range of $\alpha$ so that these coefficients have a lower bound in terms of $\delta$. Third, there are additional terms that must be estimated away and that arise because the background is not spherically symmetric. In any $r^p$ argument, there are several indefinite terms that must be estimated in terms of quantities with positive coefficients. We refer to the indefinite terms as ``error'' terms and those with positive coefficients as ``principal'' terms. Some of the error terms vanish in spherical symmetry, when $a=0$. Of the terms that vanish for $a=0$, the ones that we found hardest to estimate were those in equation \eqref{eq:rpEstimate:Energy:Harder}. As shown in equation \eqref{eq:reasonForChoiceOfCInHyperboloids}, we found that by taking $\CInHyperboloids$ sufficiently large, the contribution from these terms could be made small relative to the principal terms in the energy. This dictated our choice of $\CInHyperboloids$ in definition \ref{def:timefunc0CInHyperboloids}, although we expect that there are other methods that could be used to estimate these terms. In particular, for theorem \ref{thm:mainintro}, which is stated for $|a|/M\ll 1$, the smallness of $|a|/M\ll 1$ is sufficient to control the terms in equation \eqref{eq:reasonForChoiceOfCInHyperboloids}, but, for our more general theorem \ref{thm:BEAMintro}, which is valid for all $|a|/M<1$ for which BEAM estimates are known, we use the largeness of $\CInHyperboloids$ instead of the smallness of $|a|/M$. 

\begin{lemma}[$r^p$ estimates for spin-weighted waves in weighted energy spaces]
\label{lem:GeneralrpSpinWave}
\standardHypothesisOnDelta. 
Let\footnote{The range of $s$ is essentially arbitrary, but a larger range of $s$ requires redefining $\timefunc$ with larger values of $\CInHyperboloids$.} $|s| \leq 3$. 
Let $b_\VOp $, $b_\phi$, and $b_0$ be real, smooth functions of $r$ such that
\begin{enumerate}
\item $\exists b_{\VOp ,-1}\in\Reals$ such that $b_\VOp =b_{\VOp ,-1}r +M\bigOAnalytic(1)$
and $b_{\VOp,-1}\geq 0$, 
\item $b_\phi =M\bigOAnalytic(r^{-1})$, and 
\item \label{Assumption:b0:Generalrp} $\exists b_{0,0}\in\Reals$ such that $b_0 = b_{0,0} +M\bigOAnalytic(r^{-1})$
and $b_{0,0}+|s|+s \geq 0$. 
\end{enumerate}

Given these, there are constants $\rCutOffInrpWave=\rCutOffInrpWave(b_0,b_{\phi},b_{\VOp})$ and $\CInrpWave=\CInrpWave(b_0,b_{\phi},b_{\VOp})$ such that for all scalars $\varphi$ and $\vartheta$ with spin weight $s$, and if 
\begin{align}
\widehat\squareS_s \varphi
 +b_\VOp  \VOp\varphi
 +b_\phi \LetaOp \varphi
 +b_0 \varphi
={}&
 \vartheta  , 
\label{eq:GeneralrpSpinWave:Assumption}
\end{align}
then for all $\rCutOff\geq\rCutOffInrpWave$, $\timefunc_2\geq \timefunc_1 \geq\timefunc_0$, and $\alpha \in [\delta,2-\delta]$,
\begin{align}
&\norm{r\VOp\varphi}_{W^{0}_{\alpha-2}(\StauR)}^2
+\norm{\varphi}_{W^1_{-2}(\StauR)}^2\nonumber\\
&+\norm{\varphi}_{W^{1}_{\alpha-3}(\DtauR)}^2
+\norm{M\YOp\varphi}_{W^{0}_{-1-\delta}(\DtauR)}^2\nonumber\\
&+\norm{\varphi}_{F^{0}(\Scritau)}^2 \nonumber\\
&{}\leq \CInrpWave \bigg(
\norm{r\VOp\varphi}_{W^{0}_{\alpha-2}(\StauiR)}^2
+\norm{\varphi}_{W^1_{-2}(\StauiR)}^2\nonumber\\
&\qquad\quad+\norm{\varphi}_{W^{1}_{0}(\DtauRcR)}^2 
  +\sum_{\timefunc\in\{\timefunc_1,\timefunc_2\}} \norm{\varphi}_{W^{1}_{\alpha}(\StautRcR)}^2 
  +\|\vartheta\|_{W^0_{\alpha-3}(\DtauRc)}^2 \bigg).
\label{eq:GeneralrpSpinWave:Conclusion}
\end{align}
\end{lemma}
\begin{proof}
The proof uses the method of multipliers with a multiplier that is a cut-off version of $M(1+M^\delta r^{-\delta})\YOp  +M^{1-\alpha}r^\alpha \VOp $ with $\alpha\in[2\delta,2-2\delta]$ and a rescaling $\delta\mapsto\delta/2$ will be made at the end of the proof. Within this proof, the relation $\lesssim$ is used to denote $\lesC{b_0,b_\phi,b_{\VOp},\rCutOffInrpWave}$, and we use mass normalization as in definition \ref{def:massnormalization}. 

Because the conformally regular functions are dense in the $W^{k}_{\alpha}$ spaces, by applying a density argument, it is sufficient to assume that $\vartheta$ and the initial data for $\varphi$ are conformally regular. In particular, it is sufficient to assume that $\varphi$ is conformally regular. This simplifies the treatment of certain terms on $\Scri^+$. 

\begin{steps}
\step{Set up the method of multipliers}
From equations \eqref{eq:rescaledWaveOperator} and \eqref{eq:expandedSssss}, the spin-weighted wave equation \eqref{eq:GeneralrpSpinWave:Assumption} can be expanded out as 
\begin{align}
\left(
  2 (r^2+a^2) \YOp\VOp
  +b_\VOp \VOp
  + (b_\phi+c_\phi)\LetaOp
  + (b_0+c_0) 
\right)\varphi &\nonumber\\
+\left(
  - 2 \hedt\hedtp
  - f_1(\theta)\LxiOp\LxiOp
  - f_2(\theta) \LetaOp\LxiOp
  - f_3(\theta) \LxiOp
\right)\varphi &
- \vartheta 
= 0 ,
\label{eq:ExpandSWWE}
\end{align}
where
\begin{subequations}
\begin{align}
c_\phi ={}&- \frac{2 a r }{a^2 + r^2} , &
c_0 ={}& \frac{(a^4 - 4 M a^2 r + a^2 r^2 + 2 M r^3)}{(a^2 + r^2)^2},\\
f_1(\theta)={}&
a^2 \sin^2\theta,&
f_2(\theta)={}& 2 a, \qquad \ \qquad 
f_3(\theta)={}
-2i a s \cos\theta . 
\end{align}
\label{eq:ExpandSWWE:coefficients}
\end{subequations}
Observe that, for each $s$, the $f_i$ are smooth functions on the sphere such that $\mathcal{L}_\eta f_i=0$. Thus, the spin-weighted wave equation \eqref{eq:GeneralrpSpinWave:Assumption} can be rewritten as 
\begin{align}
\sum_{i=1}^9 I_i ={} 0 , 
\label{eq:intermediateWaveOperator} 
\end{align}
where 
\begin{align}
I_1={}&2 (r^2+a^2) \YOp\VOp\varphi ,&
I_2={}& b_\VOp \VOp\varphi,&
I_3={}& (b_\phi+c_\phi)\LetaOp\varphi,&
I_4={}& (b_0+c_0) \varphi,\nonumber\\
I_5={}& -2 \hedt\hedtp\varphi ,&
I_6={}& - f_1(\theta)\LxiOp\LxiOp\varphi,&
I_7={}& - f_2(\theta) \LetaOp\LxiOp\varphi,&
I_8={}& - f_3(\theta) \LxiOp\varphi ,\nonumber\\
I_9={}& - \vartheta .
\label{eq:SWWEIdefs}
\end{align}

Let $\chi_1$ be decreasing, smooth, equal to $1$ on $(-\infty,0)$, and equal to $0$ on $(1,\infty)$, and let $\chi=\chi_1((\rCutOff-r)/M)$. This implies that $\chi$ vanishes for $r\leq \rCutOff-M$ and is identically $1$ for $r\geq\rCutOff$. 

Following the standard method-of-multipliers procedure, one can multiply the spin-weighted wave equation \eqref{eq:intermediateWaveOperator}  by $\chi^2 M^{1-\alpha} r^{\alpha} (\VOp \bar\varphi) +\chi^2 M(1+ M^\delta r^{-\delta}) (\YOp \bar\varphi)$, multiply by a further factor of $M^2/(r^2+a^2)$, take the real part, integrate the resulting equation over $\DtauRc$, and then estimate the various terms. To do so, it is convenient to introduce, for $i\in\{1,\ldots,9\}$, 
\begin{align}
I_{i,\VOp } ={}& \Re\left(\chi^2 M^{1-\alpha} r^{\alpha} (\VOp \bar\varphi) \frac{M^2}{r^2+a^2} I_i \right), &
I_{i,\YOp } ={}& \Re\left(\chi^2 M(1+ M^\delta r^{-\delta}) (\YOp \bar\varphi) \frac{M^2}{r^2+a^2} I_i \right).
\end{align}

For $i\in\{1,\ldots,9\}$ and $X\in\{\VOp,\YOp\}$, the term $I_{i,X}$ is said to be put in standard form when there are $\momentumiX$, $\bulkComponentiX$, and $\bulkComponentiXError$ such that, for any region $\Omega=\Omega_{\timefunc,r}\times\Sphere$ with $\Omega_{\timefunc,r}\subset\Reals\times(2M,\infty)$ and with boundary $\partial\Omega$,
\begin{align}
\int_{\Omega} I_{i,X} \diFourVol
={}& \int_{\partial\Omega} \Normal_ a \momentumiX^a  \LerayForm{\Normal}
+\int_{\Omega} (\bulkComponentiX +\bulkComponentiXError) \diFourVol . 
\end{align}

After the method of multipliers presented in the first step of this proof, the purpose of step \ref{step:rp:favourableTerms} is to isolate the principal terms, both in the bulk $\DtauRc$ and in energies on $\StautRc$. The $I_{1,\VOp }$ and $I_{2,\VOp }$ terms contribute the dominant $\abs{\VOp \varphi}^2$ terms both on $\StautRc$ and in $\DtauRc$, the $I_{1,\YOp }$ term contributes the dominant $\abs{\YOp \varphi}$ term on $\StautRc$ and in $\DtauRc$, the $I_{5,\YOp }$ term contributes the dominant $\abs{\hedtp\varphi}^2$ term on $\StautRc$, but the $I_{4,\VOp}$ and $I_{5,\VOp }$ terms together contribute the dominant $\abs{\varphi}^2$ and $\abs{\hedtp\varphi}^2$ term in $\DtauRc$. Step \ref{step:rp:remainingTerms} is to define the remaining, nonprincipal terms.  

The $I_6$ and $I_7$ terms are particularly difficult to treat. The $I_{6,\YOp}$ and $I_{7,\YOp}$ contribute terms that do not decay in $r$ faster than those that arise in the principal terms. To handle these, it is necessary to exploit the largeness of $\CInHyperboloids$, 
which is set in definition \ref{def:timefunc0CInHyperboloids}. 
Step \ref{step:rp:treatPrincipalBulk} treats the principal part in $\DtauRc$. Step \ref{step:rp:TreatTheEnergyOnHyperboloids} treats the energy on each $\Staut$, and in particular the $I_{6,\VOp}$ and $I_{7,\VOp}$ terms. Step \ref{step:rp:ScriFlux} treats the flux through $\Scri^+$. Step \ref{step:rp:treatRemainingBulk} treats the remaining bulk terms, which completes the proof. 

The remainder of this proof uses mass normalization, as in definition \ref{def:massnormalization}. 

\step{Definition of the principal terms}
\label{step:rp:favourableTerms}
Within this proof, the principal terms are those that contribute a nonnegative, leading-order term, either in the bulk or on hypersurfaces. To isolate pure powers of $r$ in the principal bulk terms, instead of powers of $r^2+a^2$, it is useful to observe
\begin{align}
\frac{1}{r^2} -\frac{1}{r^2+a^2}
={}& \frac{a^2}{r^2(r^2+a^2)} . 
\end{align}

Integrating $I_{1,\VOp}$ $=\chi^2 r^\alpha (r^2+a^2)^{-1} \Re((\VOp\bar\varphi)(2(r^2+a^2)\YOp\VOp\varphi))$ and applying $\YOp $ in\-te\-gra\-tion-by-parts formula \eqref{eq:IBPYInnerProduct}, one finds $I_{1,\VOp}$ is in standard form with
\begin{subequations}
\begin{align}
\momentumAV^a
={}&\chi^2  r^\alpha |\VOp\varphi|^2 \YOp^a,\\
\bulkComponentAV
={}& \chi^2 r^{\alpha-1} \alpha  \abs{\VOp\varphi}^2 ,
\label{eq:rpEstimate:VSquaredTerm}\\
\bulkComponentAVError
={}& \partial_r (\chi^2) r^{\alpha} |\VOp\varphi|^2 
.
\end{align}
\end{subequations}

The term $I_{2,\VOp}$ $=\chi^2 r^\alpha (r^2+a^2)^{-1} \Re((\VOp\bar\varphi)(b_{\VOp}\VOp\varphi))$ can immediately be put in standard form with 
\begin{subequations}
\begin{align}
\momentumBV^a 
={}& 0,\\
\bulkComponentBV
={}& \chi^2 r^{\alpha-1} b_{\VOp,-1}  \abs{\VOp\varphi}^2 ,\\
\bulkComponentBVError
={}& \chi^2 r^\alpha\left(
  \frac{-b_{\VOp,-1}r a^2}{r^2(r^2+a^2)}
  +\frac{b_{\VOp}-rb_{\VOp,-1}}{r^2+a^2}
\right)|\VOp\varphi|^2 
.
\end{align}
\end{subequations}

Integrating $I_{1,\YOp }$ and applying commutator formula \eqref{eq:CommutatorofYandMathcalV}, one finds
\begin{align}
\int_{\DtauRc} I_{1,\YOp } \diFourVol 
={}&\int_{\DtauRc} 2\chi^2(1+ r^{-\delta})\Re\left((\YOp \bar\varphi)(\YOp \VOp \varphi)\right) \diFourVol \nonumber\\
={}& \int_{\DtauRc} 2\chi^2(1+ r^{-\delta})\Re\left((\YOp \bar\varphi)(\VOp \YOp \varphi)\right) \diFourVol \nonumber\\
&+\int_{\DtauRc} 2\chi^2(1+ r^{-\delta})(\YOp \bar\varphi)\frac{r^2-a^2}{(r^2+a^2)^2}(\YOp \varphi) \diFourVol\nonumber\\
&+\int_{\DtauRc} 2\chi^2(1+ r^{-\delta})\Re\left((\YOp \bar\varphi)\frac{2ar}{(r^2+a^2)^2}(\mathcal{L}_\eta\varphi)\right) \diFourVol .
\end{align}
Now, applying $\VOp $ integration-by-parts formula \eqref{eq:IBPVInnerProduct} to the first term on the right, one finds $I_{1,\YOp}$ in standard form 
\begin{subequations}
\begin{align}
\momentumAY^a
={}& \chi^2  (1+ r^{-\delta}) \abs{\YOp \varphi}^2 \VOp^a ,\\
\bulkComponentAY
={}&   \frac12 \delta \chi^2  r^{-\delta-1} \abs{\YOp \varphi}^2 ,\\
\bulkComponentAYError
={}& \bulkComponentAYYY +\bulkComponentAYYeta ,\\
\bulkComponentAYYY
={}& \bigg(
  -\frac12\delta \chi^2  r^{-\delta-1}
  -\partial_r\left(\chi^2 r^{-\delta}\frac{\Delta}{2(r^2+a^2)}\right)\nonumber\\
&\qquad  +2\chi^2(1+ r^{-\delta})\frac{r^2-a^2}{(r^2+a^2)^2}
\bigg)\abs{\YOp \varphi}^2 ,\\
\bulkComponentAYYeta
={}& 2\chi^2(1+ r^{-\delta})\frac{2ar}{(r^2+a^2)^2}\Re\left((\YOp \bar\varphi)(\mathcal{L}_\eta\varphi)\right) . 
\label{eq:rp:DefBAYError} 
\end{align}
\end{subequations}

The term $I_{5,\VOp }$ can be rewritten, using $\hedt\bar\varphi=\widebar{\hedtp\varphi}$, as
\begin{align}
I_{5,\VOp }
={}&-2\Re\left(\hedt\left(\chi^2 r^{\alpha}\frac{1}{r^2+a^2}(\VOp\bar\varphi)(\hedtp\varphi)\right)\right)
+2\Re\left(\chi^2 r^{\alpha}\frac{1}{r^2+a^2}(\VOp\widebar{\hedtp\varphi})(\hedtp\varphi)\right) .
\end{align}
Thus, applying the $\hedt$ and $\VOp $ integration-by-parts formulas \eqref{eq:IBPhedt} and \eqref{eq:IBPVInnerProduct}, one finds $I_{5,\VOp}$ is in standard form with
\begin{subequations}
\begin{align}
\momentumEV^a
={}& \chi^2 r^{\alpha}\frac{1}{r^2+a^2} |\hedtp\varphi|^2 \VOp^a ,\\
\bulkComponentEV
={}&\frac{2-\alpha}{2}\chi^2 r^{\alpha-3}\abs{\hedtp\varphi}^2 ,
\label{eq:rpEstimate:AngularSquaredTerm}\\
\bulkComponentEVError
={}&\left(
    \frac{2-\alpha}{2} \chi^2 r^{\alpha-3}
    -\partial_r\left(\chi^2 r^{\alpha}\frac{\Delta}{2(r^2+a^2)^2}\right)
  \right)\abs{\hedtp\varphi}^2 .
\end{align}
\end{subequations}

Similarly, for the $I_{5,\YOp}$ term, 
\begin{align}
I_{5,\YOp}
={}& -2\Re\left(\hedt\left(\chi^2 \frac{1+ r^{-\delta}}{r^2+a^2} (\YOp \bar\varphi)(\hedtp\varphi)\right)\right) 
+2\Re\left(\chi^2 \frac{1+ r^{-\delta}}{r^2+a^2} (\YOp \hedt\bar\varphi)(\hedtp\varphi)\right) ,
\label{eq:rp:I5YIntegral}
\end{align}
so that $\YOp$ integration-by-parts formula \eqref{eq:IBPYInnerProduct} gives the standard form with
\begin{subequations}
\begin{align}
\momentumEY^a
={}& \chi^2 \left(1+\frac{1}{r^{\delta}}\right)\frac{1}{r^2+a^2} \abs{\hedtp\varphi}^2 \YOp^a ,\\
\bulkComponentEY
={}&0, \\
\bulkComponentEYError
={}& \left(\partial_r\left( \chi^2(1+ r^{-\delta})\frac{1}{r^2+a^2}\right)\right)\abs{\hedtp\varphi}^2. 
\label{eq:rp:DefBEYError}
\end{align}
\end{subequations}

Integrating $I_{4,\VOp}$ $=\chi^2r^\alpha (r^2+a^2)^{-1}\Re((\VOp\bar\varphi)(b_0+c_0)\varphi)$ and applying $\VOp$ integration-by-parts formula \eqref{eq:IBPVInnerProduct}, one obtains the standard form with
\begin{subequations}
\begin{align}
\momentumDV^a 
={}& \frac12 \chi^2\frac{r^\alpha}{r^2+a^2} (b_0+c_0)\abs{\varphi}^2 \VOp^a ,\\
\bulkComponentDV
={}& \frac{2-\alpha}{4}  b_{0,0}\chi^2 r^{\alpha-3} \abs{\varphi}^2 ,\\
\bulkComponentDVError
={}&  \left(
  -\frac{2-\alpha}{4} b_{0,0}\chi^2 r^{\alpha-3}
  +\partial_r\left(r^\alpha\frac{\Delta}{4(r^2+a^2)^2} (b_0+c_0)\right)
\right)\abs{\varphi}^2 .
\end{align}
\end{subequations}

For $(i,X)$ $\in \{(2,\YOp),(3,\VOp),(3,\YOp),(4,\YOp),(6,\VOp),(6,\YOp),(7,\VOp),(7,\YOp),(8,\VOp),(8,\YOp),(9,\VOp),(9,\YOp)\}$ -that is, for all $(i,X)$ for which $\bulkComponentiX$ has not yet been defined- define
\begin{align}
\bulkComponentiX ={}& 0 .
\end{align}

\step{Define the remaining terms}
\label{step:rp:remainingTerms}
Considering $I_{6,\YOp}$ and isolating a total $\xi$ derivative, one finds
\begin{align}
I_{6,\YOp} 
={}&- \chi^2 (r^2+a^2)^{-1}(1+ r^{-\delta})\Re\left((\YOp\bar\varphi)f_1(\theta)
\LxiOp\mathcal{L}_\xi\varphi\right)\nonumber\\
={}&-\mathcal{L}_\xi\left( \chi^2 \frac{1+ r^{-\delta}}{r^2+a^2}\Re\left((\YOp\bar\varphi)f_1(\theta)\mathcal{L}_\xi\varphi \right) \right) \nonumber\\
&+\chi^2 \frac{1+ r^{-\delta}}{r^2+a^2}\Re\left((\YOp\mathcal{L}_\xi\bar\varphi)f_1(\theta)\mathcal{L}_\xi\varphi \right)  .
\end{align}
Now integrating and applying $\YOp$ integration-by-parts formula \eqref{eq:IBPYInnerProduct}, one obtains the standard form for $I_{6,\YOp}$ with
\begin{subequations}
\begin{align}
\momentumFY^a
={}& -\chi^2 \frac{1+ r^{-\delta}}{r^2+a^2} f_1(\theta) \Re\left((\YOp\bar\varphi)\mathcal{L}_\xi\varphi\right) \xi^a \nonumber\\
&+\frac12\chi^2 \frac{1+ r^{-\delta}}{r^2+a^2} f_1(\theta) \abs{\mathcal{L}_\xi\varphi}^2 \YOp^a ,\\
\bulkComponentFY
={}&\left(\partial_r\left(\frac12\chi^2 \frac{1+ r^{-\delta}}{r^2+a^2}\right)\right)\abs{\mathcal{L}_\xi\varphi}^2f_1(\theta) 
\label{eq:rp:DefBFY}.
\end{align}
\end{subequations}
(Recall all the principal terms were defined in the previous step.) 

The term $I_{7,\YOp}$ can be rewritten, using the Leibniz rule in $\LxiOp$, $\YOp$, and $\LetaOp$, as
\begin{align}
&-\chi^2\frac{1+ r^{-\delta}}{r^2+a^2}\Re\left((\YOp\bar\varphi)f_2(\theta)
\mathcal{L}_\eta\mathcal{L}_\xi\varphi\right) \nonumber\\
={}& -\mathcal{L}_\xi\left(\chi^2\frac{1+ r^{-\delta}}{r^2+a^2}\Re\left((\YOp\bar\varphi)f_2(\theta)\mathcal{L}_\eta\varphi\right)\right)
+\chi^2\frac{1+ r^{-\delta}}{r^2+a^2}\Re\left((\mathcal{L}_\xi\YOp\bar\varphi)f_2(\theta)\mathcal{L}_\eta\varphi\right) \nonumber\\
={}& -\mathcal{L}_\xi\left(\chi^2\frac{1+ r^{-\delta}}{r^2+a^2}\Re\left((\YOp\bar\varphi)f_2(\theta)\mathcal{L}_\eta\varphi\right)\right)
+\YOp\left(\chi^2\frac{1+ r^{-\delta}}{r^2+a^2}\Re\left((\mathcal{L}_\xi\bar\varphi)f_2(\theta)\mathcal{L}_\eta\varphi\right)\right) \nonumber\\
&+\partial_r\left(\chi^2 \frac{1+ r^{-\delta}}{r^2+a^2}\right)f_2(\theta)\Re\left((\mathcal{L}_\xi\bar\varphi)\mathcal{L}_\eta\varphi\right)
-\chi^2\frac{1+ r^{-\delta}}{r^2+a^2}\Re\left((\mathcal{L}_\xi\bar\varphi)f_2(\theta)\YOp\mathcal{L}_\eta\varphi\right) \nonumber\\
={}& -\mathcal{L}_\xi\left(\chi^2\frac{1+ r^{-\delta}}{r^2+a^2}\Re\left((\YOp\bar\varphi)f_2(\theta)\mathcal{L}_\eta\varphi\right)\right)
+\YOp\left(\chi^2\frac{1+ r^{-\delta}}{r^2+a^2}\Re\left((\mathcal{L}_\xi\bar\varphi)f_2(\theta)\mathcal{L}_\eta\varphi\right)\right) \nonumber\\
&+\partial_r\left(\chi^2\frac{1+ r^{-\delta}}{r^2+a^2}\right)f_2(\theta)\Re\left((\mathcal{L}_\xi\bar\varphi)\mathcal{L}_\eta\varphi\right)
-\mathcal{L}_\eta\left(\chi^2\frac{1+ r^{-\delta}}{r^2+a^2}\Re\left((\mathcal{L}_\xi\bar\varphi)f_2(\theta)\YOp\varphi\right)\right)
\nonumber\\
&+\chi^2\frac{1+ r^{-\delta}}{r^2+a^2}\Re\left((\mathcal{L}_\xi\mathcal{L}_\eta\bar\varphi)f_2(\theta)\YOp\varphi\right) .
\end{align}
Now, identifying the final term as the opposite of the term on the first line, one can integrate to obtain the standard form for $I_{7,\YOp}$ with 
\begin{subequations}
\begin{align}
\momentumGY^a
={}& -\frac12 \chi^2\frac{1+ r^{-\delta}}{r^2+a^2}\Re\left((\YOp\bar\varphi)f_2(\theta)\mathcal{L}_\eta\varphi\right) \xi^a\nonumber\\
&+\frac12\chi^2\frac{1+ r^{-\delta}}{r^2+a^2}\Re\left((\mathcal{L}_\xi\bar\varphi)f_2(\theta)\mathcal{L}_\eta\varphi\right)\YOp^a ,\\
\bulkComponentGY
={}& \frac12\partial_r\left(\chi^2\frac{1+ r^{-\delta}}{r^2+a^2}\right) \Re\left((\mathcal{L}_\xi\bar\varphi)f_2(\theta)\LetaOp\varphi \right) .
\label{eq:rp:DefBGY} 
\end{align}
\end{subequations}

$I_{6,\VOp}$ can be rewritten, by isolating a total $\xi$ derivative, as
\begin{align}
I_{6,\VOp}
={}& \LxiOp\left(-\chi^2 f_1(\theta)\frac{r^\alpha}{r^2+a^2} \Re\left((\VOp\overline{\varphi})\LxiOp\varphi\right)\right)
+\chi^2 f_1(\theta)\frac{r^\alpha}{r^2+a^2}\Re((\VOp\widebar{\mathcal{L}_\xi\varphi})(\mathcal{L}_\xi\varphi)) .
\end{align}
Since $\mathcal{L}_\xi$ acting on a scalar and in the $(\timefunc,r,\omega)$ parametrization is just $\partial_\timefunc$, if one integrates the first term in $\timefunc$ and applies $V$ integration-by-parts formula \eqref{eq:IBPVInnerProduct} on the second, then one obtains $I_{6,\VOp}$ in standard form with
\begin{subequations}
\begin{align}
\momentumFV^a
={}& -\chi^2 f_1(\theta)\frac{r^\alpha}{r^2+a^2} \Re\left((\VOp\overline{\varphi})\LxiOp\varphi\right)\xi^a
  +\frac12\chi^2f_1(\theta)\frac{r^\alpha}{r^2+a^2} |\mathcal{L}_\xi\varphi|^2 \VOp^a,\\
\bulkComponentFVError
={}& -\frac14 \partial_r\left(\chi^2f_1(\theta)r^{\alpha}\frac{\Delta}{(r^2+a^2)^2}\right)|\mathcal{L}_\xi\varphi|^2 .
\end{align}
\end{subequations}

This type of analysis can be applied to $I_{7,\VOp}$. Term $I_{7,\VOp}$ can be rewritten using the Leibniz rule in $\mathcal{L}_\xi$, $V$, and $\LetaOp$ as
\begin{align}
I_{7,\VOp}
={}& -\chi^2 f_2(\theta)\frac{r^\alpha}{r^2+a^2} \Re\left((\VOp\overline{\varphi})\LetaOp\LxiOp\varphi\right) \notag\\
={}& \LxiOp\left(-\chi^2 f_2(\theta)\frac{r^\alpha}{r^2+a^2} \Re\left((\VOp\overline{\varphi})\LetaOp\varphi\right)\right)
+\chi^2 f_2(\theta)\frac{r^\alpha}{r^2+a^2}\Re\left((\LxiOp\VOp \overline\varphi)\LetaOp\varphi\right)\notag\\
={}& \LxiOp\left(-\chi^2 f_2(\theta)\frac{r^\alpha}{r^2+a^2}\ \Re\left((\VOp\overline{\varphi})\LetaOp\varphi\right)\right)
+\VOp\left(\chi^2 f_2(\theta)\frac{r^\alpha}{r^2+a^2}\Re\left((\LxiOp\overline\varphi)\LetaOp\varphi\right)\right)\notag\\
&-\VOp\left( \chi^2 f_2(\theta)\frac{r^\alpha}{r^2+a^2}\right)\Re\left((\LxiOp\overline\varphi)\LetaOp\varphi\right)
- \chi^2 f_2(\theta)\frac{r^\alpha}{r^2+a^2}\Re\left((\LxiOp\overline\varphi)\LetaOp\VOp\varphi\right)\notag\\
={}& \LxiOp\left(-\chi^2 f_2(\theta)\frac{r^\alpha}{r^2+a^2}\ \Re\left((\VOp\overline{\varphi})\LetaOp\varphi\right)\right)
+\VOp\left(\chi^2 f_2(\theta)\frac{r^\alpha}{r^2+a^2}\Re\left((\LxiOp\overline\varphi)\LetaOp\varphi\right)\right)\notag\\
&-\VOp\left( \chi^2 f_2(\theta)\frac{r^\alpha}{r^2+a^2}\right)\Re\left((\LxiOp\overline\varphi)\LetaOp\varphi\right)
-\LetaOp\left(\chi^2 f_2(\theta)\frac{r^\alpha}{r^2+a^2}\Re\left((\LxiOp\overline\varphi)\VOp\varphi\right)\right)\notag\\
&+ \chi^2 f_2(\theta)\frac{r^\alpha}{r^2+a^2}\Re\left((\LetaOp\LxiOp\overline\varphi)\VOp\varphi\right).
\end{align}
Identifying the last term on the right with the opposite of the term on the left, one obtains
\begin{align}
2I_{7,\VOp}
={}& -2\chi^2 f_2(\theta)\frac{r^\alpha}{r^2+a^2} \Re\left((\VOp\overline{\varphi})\LetaOp\LxiOp\varphi\right) \notag\\
={}& \LxiOp\left(-\chi^2 f_2(\theta)\frac{r^\alpha}{r^2+a^2}\Re\left((\VOp\overline{\varphi})\LetaOp\varphi\right)\right)
+\VOp\left(\chi^2 f_2(\theta)\frac{r^\alpha}{r^2+a^2}\Re\left((\LxiOp\overline\varphi)\LetaOp\varphi\right)\right)\notag\\
&
-\LetaOp\left(\chi^2 f_2(\theta)\frac{r^\alpha}{r^2+a^2}\Re\left((\LxiOp\overline\varphi)\VOp\varphi\right)\right)
-\VOp\left( \chi^2 f_2(\theta)\frac{r^\alpha}{r^2+a^2}\right)\Re\left((\LxiOp\overline\varphi)\LetaOp\varphi\right).
\end{align}
Since $\mathcal{L}_\xi$ and $\LetaOp$ are $\partial_\timefunc$ and $\partial_\phi$ in the $(\timefunc,r,\theta,\phi)$ coordinates, from the $V$ integration by parts formula \eqref{eq:IBPV}, one obtains the standard form for $I_{7,\VOp}$ with
\begin{subequations}
\begin{align}
\momentumGV^a
={}&-\frac12\chi^2 f_2(\theta) \frac{r^\alpha}{r^2+a^2} \Re((\VOp\overline\varphi)\LetaOp\varphi) \xi^a
+\frac12\chi^2 f_2(\theta) \frac{r^\alpha}{r^2+a^2}\Re((\LxiOp\overline\varphi) \LetaOp\varphi ) \VOp^a,\\
\bulkComponentGVError
={}&-\left(\partial_r\left(\frac12\chi^2 f_2(\theta) \frac{r^\alpha}{r^2+a^2} \frac{\Delta}{2(r^2+a^2)}\right)\right) \Re((\LxiOp\overline\varphi) \LetaOp\varphi) .
\end{align}
\end{subequations}

For $(i,X)$ $\in \{(2,\YOp),(3,\VOp),(3,\YOp),(4,\YOp),(8,\VOp),(8,\YOp),(9,\VOp),(9,\YOp)\}$ -that is, for all $(i,X)$ for which $\momentumiX^a$ has not yet been defined- define
\begin{align}
\momentumiX^a ={}& 0 .
\end{align}

\step{Treat the principal bulk term}
\label{step:rp:treatPrincipalBulk}
Let
\begin{align}
\bulkComponentPrincipal
={}& \sum_{i\in\{1,\ldots,9\},X\in\{\VOp,\YOp\}} \bulkComponentiX \nonumber\\
={}& \bulkComponentAV
+ \bulkComponentBV
+ \bulkComponentAY
+ \bulkComponentEV 
+ \bulkComponentDV \nonumber\\
={}& (\alpha+b_{\VOp,-1}) \chi^2  r^{\alpha-1}\abs{\VOp\varphi}^2 
+\frac12\delta \chi^2  r^{-\delta-1} \abs{\YOp\varphi}^2 \nonumber\\
&+\frac{2-\alpha}{2}\chi^2r^{\alpha-3} \abs{\hedtp\varphi}^2 
+\frac{2-\alpha}{4}b_{0,0} \chi^2r^{\alpha-3} \abs{\varphi}^2 
. 
\end{align}

Since $b_{\VOp,-1}\geq0$, by assumption, that term can be dropped. It is convenient to rewrite $\abs{\hedtp\varphi}^2$ as $\left(\abs{\hedtp\varphi}^2-\frac{|s|+s}{2}\abs{\varphi}^2\right) +\frac{|s|+s}{2}\abs{\varphi}^2$ and to observe that, when integrated over spheres, both these summands are nonnegative from the lower bound on $\hedtp$ in lemma \ref{lem:ellip}. 
Thus, 
\begin{align}
\int_{\DtauRc} \bulkComponentPrincipal \diFourVol 
\geq{} \int_{\DtauRc}\chi^2\bigg(
& \alpha  r^{\alpha-1}\abs{\VOp\varphi}^2 
+\frac12\delta   r^{-\delta-1} \abs{\YOp\varphi}^2 \nonumber\\
&+\frac{2-\alpha}{2}r^{\alpha-3} \left(\abs{\hedtp\varphi}^2-\frac{|s|+s}{2}\abs{\varphi}^2\right) \nonumber\\
&+\frac{2-\alpha}{4}\left(b_{0,0}+|s|+s\right) r^{\alpha-3} \abs{\varphi}^2 
\bigg) \diFourVol .
\label{eq:rp:bulkPrincipalIntemediate}
\end{align}
Thus, using the Hardy inequality \eqref{eq:HardyOnHypersurfaceBulkWeights}, one finds, for some positive constants $C_1$, $C_2$, $C_3$, $C_4$, 
\begin{align}
\int_{\DtauRc} \bulkComponentPrincipal \diFourVol &
+C_1 \norm{\varphi}_{W^{1}_0(\DtauRcR)}^2  \nonumber\\
\geq{} \int_{\DtauRc}\bigg(
& C_2  r^{\alpha-1}\abs{\VOp\varphi}^2 
+C_3   r^{-\delta-1} \abs{\YOp\varphi}^2 \nonumber\\
&+\frac{2-\alpha}{2}r^{\alpha-3} \left(\abs{\hedtp\varphi}^2-\frac{|s|+s}{2}\abs{\varphi}^2\right) \nonumber\\
&+\frac{2-\alpha}{4}\left(b_{0,0}+|s|+s+C_4\right) r^{\alpha-3} \abs{\varphi}^2 
\bigg) \diFourVol  . 
\label{eq:rp:bulkPrincipalIntemediateB}
\end{align}
Since the hypothesis of the theorem assumes that $b_{0,0}+|s|+s\geq0$ and $C_4>0$, all the coefficients are strictly positive. Furthermore, the terms that they multiply are all nonnegative. Given that there are positive multiples of $\abs{\hedtp\varphi}^2-\frac{|s|+s}{2}\abs{\varphi}^2$ and $\abs{\varphi}^2$, both with coefficients of $r^{\alpha-3}$, these can be lower bounded by positive multiples of $\abs{\hedtp\varphi}^2$ and $\abs{\varphi}^2$. Hence, there are constants $C_1,C_2$ such that
\begin{align}
&\int_{\DtauRc} \bulkComponentPrincipal \diFourVol 
+C_1 \norm{\varphi}_{W^{1}_0(\DtauRcR)}^2  \nonumber\\
&{}\geq C_2 \int_{\DtauRc}\bigg(
   r^{\alpha-1}\abs{\VOp\varphi}^2 
+  r^{-\delta-1} \abs{\YOp\varphi}^2 
+ r^{\alpha-3} \abs{\hedtp\varphi}^2
+ r^{\alpha-3} \abs{\varphi}^2 
\bigg) \diFourVol .
\label{eq:rp:bulkPrincipal}
\end{align}

The relation \eqref{eq:LxiToVOpYOpLeta} gives $\mathcal{L}_\xi$ $=\VOp +\xi^\YOp\YOp +\xi^\eta \mathcal{L}_\eta$ with 
 coefficients satisfying the bounds $\abs{\xi^{\YOp}}\lesssim 1$, and $\abs{\xi^\eta}\lesssim Mr^{-2}$, from which it follows that
\begin{align}
\int_{\DtauRc} r^{-1-\delta} \abs{\mathcal{L}_\xi\varphi}^2 \diFourVol
\lesssim{}& \int_{\Dtau} \bulkComponentPrincipal \diFourVol 
+\norm{\varphi}_{W^{1}_0(\DtauRcR)}^2  .
\end{align}

\step{Treat the energy on hyperboloids}
\label{step:rp:TreatTheEnergyOnHyperboloids}
On hyperboloids, the energies can be decomposed into the principal and error terms. Unfortunately, the error energies for the $I_{6,\YOp}$ and $I_{7,\YOp}$ terms have coefficients that are of the same order in $r$ as those in the principal part. Fortunately, we can use $\CInHyperboloids$ as a large parameter to dominate these error terms by the principal parts. All the remaining terms are strictly lower order and, hence, easily dominated. 

First, define the principal terms. To do so, it is also useful to recall, cf. \eqref{eq:Vt-lem}, \eqref{eq:Yt-def}, that  
\begin{subequations}\label{eq:YVt-rp}
\begin{align} 
\lim_{r\to \infty} \di\timefunc_a \YOp^a = 2, \label{eq:Yt-rp} \\
\lim_{r\to \infty} \frac{r^2}{M^{2}}\di\timefunc_a \VOp^a  = \CInHyperboloids \label{eq:CInHyperboloids-rp}. 
\end{align} 
\end{subequations} 
On the hyperboloids, let
\begin{subequations}
\begin{align}
\energyComponentAVPrincipal
={}& 2 \chi^2  r^\alpha |\VOp\varphi|^2,\\
\energyComponentAYPrincipal
={}& \CInHyperboloids \chi^2 r^{-2} \abs{\YOp \varphi}^2 ,
\label{eq:energyComponentAYPrincipal}\\
\energyComponentEYPrincipal
={}&  2 \chi^2 r^{-2} \abs{\hedtp\varphi}^2.
\end{align}
\end{subequations}
For $(i,X)\in$ $\{(2,\VOp),$ $(2,\YOp),$ 
$(3,\VOp),$ $(3,\YOp),$ 
$(4,\VOp),$ $(4,\YOp),$ 
$(5,\VOp),$ 
$(6,\VOp),$ $(6,\YOp),$ 
$(7,\VOp),$ $(7,\YOp),$ 
$(8,\VOp),$ $(8,\YOp),$ 
$(9,\VOp),$ $(9,\YOp)\}$ -that is for all $(i,X)$ for which $\energyComponentiXPrincipal$ has not been defined- define
\begin{align}
\energyComponentiXPrincipal ={}& 0 .
\end{align}
Define
\begin{align}
\energyComponentPrincipal
={}& \sum_{i\in\{1,\ldots,9\},X\in\{\VOp,\YOp\}} \energyComponentiX \nonumber\\
={}& \energyComponentAVPrincipal
+ \energyComponentAYPrincipal
+ \energyComponentEYPrincipal \nonumber\\
={}&2\chi^2  r^{\alpha}\abs{\VOp\varphi}^2 
+\CInHyperboloids \chi^2 r^{-2} \abs{\YOp\varphi}^2 
+2\chi^2r^{-2} \abs{\hedtp\varphi}^2 
.
\end{align}

There are some useful lower bounds to observe. First, note that since $r^\alpha$ can be taken to be larger than any given constant by taking $r$ sufficiently large, and since $\CInHyperboloids\geq \frac12 $, it follows from the Hardy lemma \ref{lem:HardyOnStautRc} that for any sufficiently small $\veps$ if $\rCutOff$ is sufficiently large, then 
\begin{align}
\int_{\StautRc} \chi^2 r^{-2} \abs{\varphi}^2 \diThreeVol 
\leq{}& (16+\veps)\int_{\StautRc} \chi^2 \abs{\VOp\varphi}^2 \diThreeVol
+\veps \int_{\StautRc} \chi^2r^{-2} \abs{\YOp\varphi}^2 \diThreeVol \nonumber\\
&+\int_{\StautRcR} \abs{\varphi}^2 \diThreeVol \nonumber\\
\leq{}& \int_{\StautRc} \energyComponentPrincipal \diThreeVol 
+\int_{\StautRcR} \abs{\varphi}^2 \diThreeVol . 
\label{eq:L2ByPrincipalEnergy}
\end{align}
Lemma \ref{lem:controlOfLetaOpInL2Sphere} controls the integral of $\abs{\LetaOp\varphi}^2$ on the sphere. Integrating in $r$ and then applying equation \eqref{eq:L2ByPrincipalEnergy}, one finds
\begin{align}
\int_{\StautRc} \chi^2  r^{-2}\abs{\mathcal{L}_\eta\varphi}^2 \diThreeVol 
\leq{}& \int_{\StautRc} \left(
  2\chi^2 r^{-2}\abs{\hedtp\varphi}^2
  +\chi^2  r^{-2} s^2\abs{\varphi}^2 
\right)\diThreeVol \nonumber\\
\leq{}&  (1+s^2) \int_{\StautRc} \energyComponentPrincipal \diThreeVol 
+s^2 \int_{\StautRcR} \abs{\varphi}^2 \diThreeVol .
\label{eq:LetaL2BoundInrp}
\end{align}
Furthermore, due to equation \eqref{eq:LxiToVOpYOpLeta} and since $|a|\leq M$, for $\rCutOff/M$ sufficiently large relative to $\CInHyperboloids$ and $s$, one has, for $r\geq \rCutOff$,
\begin{align}
\abs{\mathcal{L}_\xi\varphi}^2 
\leq{}& 3\abs{\VOp\varphi}^2 +\frac{3}{4}\abs{\YOp\varphi}^2 +\frac{1}{\CInHyperboloids (1+s^2)} \abs{\mathcal{L}_\eta\varphi}^2 .
\end{align}
Multiplying by $\chi^2r^{-2}$, integrating in $r$, using the bound \eqref{eq:LetaL2BoundInrp} to control the final term, the definitions of $\energyComponentAVPrincipal$ and $\energyComponentAYPrincipal$ to control the first two terms, using the largeness of $r^\alpha$ in $\energyComponentAVPrincipal$, and the factor of $\CInHyperboloids$ in $\energyComponentAYPrincipal$, one finds 
\begin{align}
\int_{\StautRc} \chi^2 r^{-2} \abs{\mathcal{L}_\xi\varphi}^2 \diThreeVol
\leq{}& \frac{4}{\CInHyperboloids} \int_{\StautRc} \energyComponentPrincipal \diThreeVol 
+\frac{1}{\CInHyperboloids} \int_{\StautRcR} \abs{\varphi}^2 \diThreeVol .
\label{eq:LxiL2BoundInrp}
\end{align}

Let
\begin{subequations}
\begin{align}
\energyComponentAVError
={}& (-2+(\di\timefunc_a \YOp^a)) \chi^2  r^\alpha |\VOp\varphi|^2,\\
\energyComponentAYError
={}& \chi^2 \left(
  (-\CInHyperboloids r^{-2}+(\di\timefunc_a \VOp^a)) 
  +(\di\timefunc_a \VOp^a)  r^{-\delta} 
\right)\abs{\YOp \varphi}^2 ,
\label{eq:rp:DefEAYError}\\
\energyComponentDVError
={}& (\di\timefunc_a \VOp^a) \chi^2 r^\alpha \frac12 \frac{1}{r^2+a^2} (b_0+c_0)\abs{\varphi}^2 ,\\
\energyComponentEVError
={}&(\di\timefunc_a \VOp^a)\chi^2 r^{\alpha}\frac{1}{r^2+a^2} |\hedtp\varphi|^2,
\label{eq:rp:DefEEV}\\
\energyComponentEYError
={}&  \chi^2\left(
  -\frac{2}{r^2}
  +\frac{\di\timefunc_a \YOp^a}{r^2+a^2}
  +(\di\timefunc_a \YOp^a)\frac{1}{r^{\delta}}\frac{1}{r^2+a^2}
\right)\abs{\hedtp\varphi}^2 ,
\label{eq:rp:DefEEYError} \\
\energyComponentFVError
={}& \energyComponentFVVxi +\energyComponentFVxixi,\\
\energyComponentFVVxi
={}& -(\di\timefunc_a\xi^a)\chi^2 f_1(\theta)\frac{r^\alpha}{r^2+a^2} \Re\left((\VOp\overline{\varphi})\LxiOp\varphi\right) ,\\
\energyComponentFVxixi
={}& \frac12(\di\timefunc_a\VOp^a)\chi^2f_1(\theta)\frac{r^\alpha}{r^2+a^2} |\LxiOp\varphi|^2 ,\\
\energyComponentGVError
={}& \energyComponentGVVeta +\energyComponentGVxieta ,\\
\energyComponentGVVeta
={}&-(\di\timefunc_a\xi^a)\frac12 \chi^2 f_2(\theta) \frac{r^\alpha}{r^2+a^2} \Re((\VOp\overline\varphi)\LetaOp\varphi) ,\\
\energyComponentGVxieta
={}& (\di\timefunc_a\VOp^a)\frac12 \chi^2 f_2(\theta) \frac{r^\alpha}{r^2+a^2}\Re((\LxiOp\overline\varphi) \LetaOp\varphi ) ,
\end{align}
\end{subequations}
and, turning to the terms that are harder to estimate, let
\begin{subequations}
\label{eq:rpEstimate:Energy:Harder}
\begin{align}
\energyComponentFYError
={}& \energyComponentFYYxi +\energyComponentFYxixi ,\\
\energyComponentFYYxi
={}& -(\di \timefunc_a \xi^a) \chi^2 \frac{1+ r^{-\delta}}{r^2+a^2} f_1(\theta) \Re((\YOp\bar\varphi)\mathcal{L}_\xi\varphi),\\
\energyComponentFYxixi
={}&(\di\timefunc_a\YOp^a)\frac12\chi^2 \frac{1+ r^{-\delta}}{r^2+a^2} f_1(\theta) \abs{\mathcal{L}_\xi\varphi}^2 ,\\
\energyComponentGYError
={}&\energyComponentGYYeta +\energyComponentGYxieta ,\\
\energyComponentGYYeta
={}& -(\di\timefunc_a\xi^a)\frac12\chi^2\frac{1+ r^{-\delta}}{r^2+a^2}f_2(\theta)\Re\left((\YOp\bar\varphi)\mathcal{L}_\eta\varphi \right),\\
\energyComponentGYxieta
={}&(\di\timefunc_a\YOp^a)\frac12\chi^2\frac{1+ r^{-\delta}}{r^2+a^2}f_2(\theta)\Re\left((\mathcal{L}_\xi\bar\varphi)\mathcal{L}_\eta\varphi\right) .
\end{align}
\end{subequations}
For $(i,X)\in$ $\{(2,\VOp),$ $(2,\YOp),$ 
$(3,\VOp),$ $(3,\YOp),$ 
$(4,\YOp),$ 
$(6,\VOp),$ $(6,\YOp),$ 
$(7,\VOp),$ $(7,\YOp),$ 
$(8,\VOp),$ $(8,\YOp),$ 
$(9,\VOp),$ $(9,\YOp)\}$ -that is for all $(i,X)$ for which $\energyComponentiXError$ has not been defined- let
\begin{align}
\energyComponentiXError ={}& 0 .
\end{align}

For $r$ sufficiently large, one has $|\di\timefunc_a\YOp^a|\leq 4$,  $|\di\timefunc_a\xi^a|\leq 2$, and $1+ r^{-\delta} \leq 2$. Independently of $r$, one has $|f_1(\theta)|\leq M^2$ and $|f_2(\theta)|\leq 2M$. Thus, from the previous bounds 
\begin{align}
\bigg\lvert\int_{\StautRc} &\left(
  \energyComponentFYYxi
  +\energyComponentFYxixi
  +\energyComponentGYYeta
  +\energyComponentGYxieta
\right) \diThreeVol \bigg\rvert\nonumber\\
\leq{}& \int_{\StautRc}\left(
  4r^{-2} \abs{\mathcal{L}_\xi\varphi}\abs{\YOp\varphi}
  +4r^{-2}\abs{\mathcal{L}_\xi\varphi}^2 
\right)\diThreeVol\nonumber\\
&+ \int_{\StautRc}\left(
  4r^{-2}\abs{\YOp\varphi}\abs{\mathcal{L}_\eta\varphi}
  +8r^{-2}\abs{\mathcal{L}_\xi\varphi}\abs{\mathcal{L}_\eta\varphi}
\right)\diThreeVol .
\end{align}
Every one of these terms has a factor of either $\abs{\LxiOp\varphi}$ or $\abs{\YOp\varphi}$, so that one obtains a factor of $\CInHyperboloids^{-1/2}$ either from the coefficient of $\abs{\YOp\varphi}^2$ in the definition of $\energyComponentAYPrincipal$ or from the bound on $\abs{\LxiOp\varphi}^2$ in inequality \eqref{eq:LxiL2BoundInrp}. Thus, from the Cauchy-Schwarz inequality, from introducing a factor of $\CInHyperboloids^{-1/2}$ on the $\LetaOp$ derivatives when applying the Cauchy-Schwarz inequality, and from equations \eqref{eq:energyComponentAYPrincipal}, \eqref{eq:LetaL2BoundInrp}, \eqref{eq:LxiL2BoundInrp}, one finds
\begin{align}
\bigg\lvert\int_{\StautRc} &\left(
  \energyComponentFYYxi
  +\energyComponentFYxixi
  +\energyComponentGYYeta
  +\energyComponentGYxieta
\right) \diThreeVol \bigg\rvert\nonumber\\
\leq{}&\int_{\StautRc}\left(
  \left(2+2\CInHyperboloids^{1/2}\right)\abs{\YOp\varphi}^2
  +\left(2+4+4\CInHyperboloids^{1/2}\right)\abs{\LxiOp\varphi}^2\right)r^{-2}\diThreeVol \nonumber\\
&+\int_{\StautRc}
  \left(\frac{2}{\CInHyperboloids^{1/2}}+\frac{4}{\CInHyperboloids^{1/2}}\right)\abs{\LetaOp\varphi}^2
r^{-2}\diThreeVol \nonumber\\
\leq{}& 
\left(
  \frac{2+2\CInHyperboloids^{1/2}}{\CInHyperboloids}
  +\frac{4(6+4\CInHyperboloids^{1/2})}{\CInHyperboloids}
  +\frac{6(1+s^2)}{\CInHyperboloids^{1/2}}
\right)
\int_{\StautRc} \energyComponentPrincipal \diThreeVol \nonumber\\
&+\left(
  \frac{6+4\CInHyperboloids^{1/2}}{\CInHyperboloids}
  +\frac{6s^2}{\CInHyperboloids^{1/2}}
\right)
\int_{\StautRcR}\abs{\varphi}^2\diThreeVol .
\label{eq:reasonForChoiceOfCInHyperboloids}
\end{align}
\index{C2CHyp@$\CInHyperboloids$}
Since $s^2$ is bounded by $9$, and since $\CInHyperboloids$ is chosen to be $10^{6}$ in definition \ref{def:timefunc0CInHyperboloids}, it follows that, for some constant $C$, on any hyperboloid there is the bound
\begin{align}
\bigg\lvert\int_{\StautRc} &\left(
  \energyComponentFYYxi
  +\energyComponentFYxixi
  +\energyComponentGYYeta
  +\energyComponentGYxieta
\right) \diThreeVol \bigg\rvert\nonumber\\
\leq{}& \frac12 \int_{\StautRc} \energyComponentPrincipal \diThreeVol 
+C\int_{\StautRcR}\abs{\varphi}^2\diThreeVol .
\end{align}

It can now be shown that the remaining error terms can be made arbitrarily small relative to $\energyComponentPrincipal$ by taking $r$ sufficiently large. One way to show this is to show that the term consists of a norm squared appearing in $\energyComponentPrincipal$ but with a lower exponent. For example, in $\energyComponentAVError$, there is a factor of $\abs{\VOp\varphi}^2$ with an exponent that vanishes at a rate of $r^{\alpha-1}$ (since $-2+\di\timefunc_a\YOp^a$ vanishes as $r^{-1}$), which decays faster than the $r^\alpha$ coefficient of $\abs{\VOp\varphi}^2$ in $\energyComponentPrincipal$. Another, similar, method is to show that the term involves the (real part of) the inner product of two terms involving $\varphi$, each of which appear in $\energyComponentPrincipal$, and that the coefficient of this inner product vanishes faster the geometric mean of the corresponding coefficients for the terms in $\energyComponentPrincipal$. For example, the term $\energyComponentFVVxi$ has a factor of $\Re((\VOp\overline\varphi)\LxiOp\varphi)$ multiplied by a coefficient that vanishes as $r^{\alpha-2}$. The geometric mean of two terms that decay with a particular exponent decays with an exponent that is given by the arithmetic mean. The energy $\energyComponentPrincipal$, dominates $r^\alpha\abs{\varphi}^2$ and $r^{-2}\abs{\mathcal{L}_\xi\varphi}^2$, and the exponents satisfy $\alpha-2<((\alpha)+(-2))/2$, so, by taking $r$ sufficiently large, one can ensure that $\energyComponentFVVxi$ is arbitrarily small relative to $\energyComponentPrincipal$. Thus, for all the error terms, it is simply a matter of checking the relevant exponents, which are given in the following table. 

\begin{center}
\begin{tabular}{lll}
Term&Exponent&Exponent from $\energyComponentPrincipal$\\
\hline
$(1,\VOp)$ &$\alpha-1$  &$\alpha$\\
$(1,\YOp)$ &$-\delta-2 $  &$-2 $\\
$(4,\VOp)$ &$(-2)+(\alpha-2)$  &$-2$\\
$(5,\VOp)$ &$(-2)+(\alpha-2)$  &$-2 $\\
$(5,\YOp)$ &$-\delta-2 $  &$-2 $\\
$(6,\VOp,(\VOp\xi))$ &$\alpha-2$ &$((\alpha)+(-2))/2$\\
$(6,\VOp,(\xi\xi))$ &$(-2)+(\alpha-2)$ &$-2$\\
$(7,\VOp,(\VOp\eta))$ &$\alpha-2$ &$((\alpha)+(-2))/2$\\
$(7,\VOp,(\xi\eta))$ &$(-2)+(\alpha-2)$ &$-2$
\end{tabular}
\end{center}

On the level sets of $\timefunc$, one has that $\Normal$ can be chosen to be $\di\timefunc$. Furthermore, one has $\LerayForm{\Normal}=\diThreeVol$. From this and the definitions in the previous paragraph, one finds for all $(i,X)$ that
\begin{align}
\int_{\StautRc} \Normal_a \momentumiX^a \LerayForm{\Normal}
={}& \int_{\StautRc} \left(
  \energyComponentiXPrincipal
  +\energyComponentiXError
\right)\diThreeVol .
\end{align} 
Thus, one can conclude 
\begin{subequations}
\begin{align}
\int_{\StauR} \energyComponentPrincipal \diThreeVol
\lesssim{}&\int_{\StauRc} \sum_{i\in\{1,\ldots,9\},X\in\{\VOp,\YOp\}} \Normal_a\momentumiX^a \LerayForm{\Normal}
+\norm{\varphi}_{W^{1}_{0}(\StauRcR)}^2 ,
\label{eq:rp:finalTimeEstimate}\\
\int_{\StauiRc} \sum_{i\in\{1,\ldots,9\},X\in\{\VOp,\YOp\}}\Normal_a\momentumiX^a \LerayForm{\Normal}
\lesssim{}& \int_{\StauiR} \energyComponentPrincipal \diThreeVol
+\norm{\varphi}_{W^{1}_{0}(\StauiRcR)}^2 .
\label{eq:rp:startingTimeEstimate}
\end{align}
\end{subequations}

\step{Treat the flux through $\Scritau$}
\label{step:rp:ScriFlux}
In this step, it is useful to treat $\Scritau$ as the limit as $r\rightarrow\infty$ of a sequence of surfaces given in hyperboloidal coordinates by $[\timefunc_1,\timefunc_2]\times\{r\}\times \Sphere$ but to think of this in the conformal geometry. 

The only non-vanishing $\momentumiX^a$ arise from $(i,X)\in$ 
$\{(1,\VOp),$ $(1,\YOp),$
$(4,\VOp),$
$(5,\VOp),$ $(5,\YOp),$
$(6,\VOp),$ $(6,\YOp),$
$(7,\VOp),$ $(7,\YOp)\}$. 
The normal to the surfaces of constant $r$ is $\Normal=\di r$, so $\Normal_a\YOp^a\sim -1$, $\Normal_a\VOp^a\sim1$, and $\Normal_a\xi^a=0$. 
From conformal regularity, one finds that $r^{\alpha-2}\abs{\varphi}_k^2\rightarrow 0$. 
Thus, 
\begin{align}
0={}& \int_{\Scritau} \Normal_a\momentumAV^a\LerayForm{\Normal} 
= \int_{\Scritau} \Normal_a\momentumEV^a\LerayForm{\Normal} 
= \int_{\Scritau} \Normal_a\momentumEY^a\LerayForm{\Normal} 
= \int_{\Scritau} \Normal_a\momentumDV^a\LerayForm{\Normal} \nonumber\\
={}& \int_{\Scritau} \Normal_a\momentumFV^a\LerayForm{\Normal} 
= \int_{\Scritau} \Normal_a\momentumGV^a\LerayForm{\Normal} 
= \int_{\Scritau} \Normal_a\momentumFY^a\LerayForm{\Normal} 
= \int_{\Scritau} \Normal_a\momentumGY^a\LerayForm{\Normal} ,
\end{align}
and the only non-vanishing term is
\begin{align}
\int_{\Scritau} \Normal_a\momentumAY^a \LerayForm{\Normal}
={}& \int_{\Scritau}  \abs{\YOp\varphi}^2 \diThreeVolScri \geq 0.
\end{align}

\step{Treat the remaining terms in the bulk via the Cauchy-Schwarz inequality}
\label{step:rp:treatRemainingBulk}
The same type of analysis as in step \ref{step:rp:TreatTheEnergyOnHyperboloids} can be used to show that the bulk error terms are all small relative to $\int_{\DtauRc}\bulkComponentPrincipal\diFourVol +\norm{\varphi}_{W^{1}_{0}(\DtauRcR)}^2$. The following table shows that the exponents satisfy the relevant bound, with $-\infty$ standing in when the error term decays faster than polynomially or is compactly supported. Note that many of the relevant exponents arise from the cancellation of leading-order terms. Note also that $(1,\YOp,(\YOp,\YOp))$ and $(1,\YOp,(\YOp,\eta))$ are used to denote $\bulkComponentAYYY$ and $\bulkComponentAYYeta$ respectively. 

\begin{center}
\begin{tabular}{lll}
Term&Exponent&Exponent from $\energyComponentPrincipal$\\
\hline
$(1,\VOp)$             & $-\infty $  & $\alpha-1 $ \\
$(2,\VOp)$             & $\alpha-2 $ & $\alpha-1 $ \\
$(1,\YOp,(\YOp,\YOp))$ & $-2$        & $-\delta-1 $ \\
$(1,\YOp,(\YOp,\eta))$ & $-3 $       & $((-\delta-1)+(\alpha-3))/2 $ \\
$(5,\VOp)$             & $\alpha-4$  & $\alpha-3 $ \\
$(5,\YOp)$             & $-3$        & $\alpha-3 $ \\
$(4,\VOp)$             & $\alpha-4$  & $\alpha-3$ \\
$(6,\YOp)$             & $-3$        & $-\delta-1$ \\
$(7,\YOp)$             & $-3$        & $-\delta-1$ \\
$(6,\VOp)$             & $\alpha-3$  & $-\delta-1$ \\
$(7,\VOp)$             & $\alpha-3$  & $((-\delta-1)+(\alpha-3))/2$ \\
$(3,\VOp)$             & $\alpha-3$  & $((\alpha-1)+(\alpha-3))/2$ \\
$(8,\VOp)$             & $\alpha-2$  & $((\alpha-1)+(-\delta-1))/2$ \\
$(2,\YOp)$             & $-1$        & $((-\delta-1)+(\alpha-1))/2$ \\
$(3,\YOp)$             & $-3$        & $((-\delta-1)+(\alpha-3))/2$ \\
$(8,\YOp)$             & $-2$        & $((-\delta-1)+(-\delta-1))/2$ 
\end{tabular}
\end{center}
For the $(2,\YOp)$ term to be controlled, it is necessary that $(-\delta+\alpha-2)/2>1$, which is why the proof has so far considered $\alpha>2\delta$. 

It remains to treat the $I_9$ terms. For any $\veps>0$, 
\begin{subequations}
\begin{align}
\lvert I_{9,\VOp } \rvert
\lesssim{}&  \veps \chi^2 r^{\alpha-1} |\VOp\varphi|^2 
+ \veps^{-1} \chi^2 r^{\alpha-3} |\vartheta| ,\\
\abs{I_{9,\YOp}}
\lesssim{}& \veps  \chi^2 r^{-1-\delta}\abs{\YOp\varphi}^2
+ \veps^{-1} \chi^2 r^{\delta-3}\abs{\vartheta}^2 . 
\end{align}
\end{subequations}
For $\veps$ sufficiently small, the first term on the right of each of these bounds is dominated by $\bulkComponentPrincipal$. Thus from the fact that all the error terms can be made small relative to the principal terms (plus some additional term for $r\in[\rCutOff-M,\rCutOff]$), one finds 
\begin{align}
\int_{\StauRc} &
  \sum_{i\in\{1,\ldots,9\},X\in\{\VOp,\YOp\}} \Normal_a\momentumiX^a \LerayForm{\Normal}
+\int_{\DtauR} \bulkComponentPrincipal \diFourVol
+\int_{\Scritau} \abs{\YOp\varphi}^2 \diThreeVolScri
\nonumber\\
\lesssim{}& \int_{\StauiRc} 
  \sum_{i\in\{1,\ldots,9\},X\in\{\VOp,\YOp\}} \Normal_a\momentumiX^a \LerayForm{\Normal}
+ \norm{\varphi}_{W^{1}_{0}(\DtauRcR)}^2 \nonumber\\
&+\int_{\DtauRc} r^{\alpha-3} |\vartheta| \diThreeVol .
\end{align}
From this, from, the estimates \eqref{eq:rp:startingTimeEstimate} and \eqref{eq:rp:finalTimeEstimate} and, from the fact that we can add an extra term $\norm{\hedt\varphi}_{W^{0}_{\alpha-3}(\DtauR)}^2$ to the left because of the relation \eqref{eq:hedt'hedtnormrelation}, it now follows that
\begin{align}
&\norm{r\VOp\varphi}_{W^{0}_{\alpha-2}(\StauR)}^2
+\norm{\varphi}_{W^1_{-2}(\StauR)}^2\nonumber\\
&+\norm{\varphi}_{W^{1}_{\alpha-3}(\DtauR)}^2
+\norm{M\YOp\varphi}_{W^{0}_{-1-\delta}(\DtauR)}^2\nonumber\\
&+\norm{\varphi}_{F^{0}(\Scritau)}^2 \nonumber\\
&{}\leq \CInrpWave \bigg(
\norm{r\VOp\varphi}_{W^{0}_{\alpha-2}(\StauiR)}^2
+\norm{\varphi}_{W^1_{-2}(\StauiR)}^2\nonumber\\
&\qquad\quad+\norm{\varphi}_{W^{1}_{0}(\DtauRcR)}^2 
  +\sum_{\timefunc\in\{\timefunc_1,\timefunc_2\}} \norm{\varphi}_{W^{1}_{\alpha}(\StautRcR)}^2 
  +\|\vartheta\|_{W^0_{\alpha-3}(\DtauRc)}^2 \bigg).
\end{align}
The term $\norm{M\YOp\varphi}_{W^{0}_{-1-\delta}(\DtauR)}^2$ can trivially be replaced by $\norm{M\YOp\varphi}_{W^{0}_{-1-2\delta}(\DtauR)}^2$. Doing so and making the rescaling $\delta\mapsto\delta/2$, one obtains the desired result \eqref{eq:GeneralrpSpinWave:Conclusion} for all $\alpha\in[\delta,2-\delta]$. 
\end{steps}
\end{proof}

\subsection{Spin-weighted wave equations in higher regularity}
\label{sec:EnergyMoraWaveHigherReg}

This section proves the analogue of the $r^p$-estimate for spin-weighted wave equations, from lemma \ref{lem:GeneralrpSpinWave}, but in higher regularity. 

\begin{lemma}[Higher-regularity $r^p$-estimates for waves in weighted energy spaces]
\label{lem:GeneralrpSpinWaveHighReg}
Under the same assumptions as in lemma \ref{lem:GeneralrpSpinWave} except that we now assume $\varphi$ has spin weight 
$ |s| \leq 2$, 
for  any $\ireg \in \Naturals$, there are constants $\rCutOffInrpWave=\rCutOffInrpWave(b_0,b_\phi,b_{\VOp})$ and $\CInrpWave=\CInrpWave(b_0,b_\phi,b_{\VOp})$ such that for all spin-weight $s$ scalars $\varphi$ and $\vartheta$, and if \eqref{eq:GeneralrpSpinWave:Assumption} is satisfied,
then for all $\rCutOff\geq\rCutOffInrpWave$,  $\timefunc_2\leq \timefunc_1 \leq \timefunc_0$  and $\alpha\in[\delta,2-\delta]$, there is 
\begin{align}
&\norm{ r\VOp\varphi}_{W^{\ireg}_{\alpha-2}(\StauR)}^2
+\norm{\varphi}_{W^{\ireg+1}_{-2}(\StauR)}^2\nonumber\\
&+\norm{\varphi}_{W^{\ireg+1}_{\alpha-3}(\DtauR)}^2
+\norm{M\YOp\varphi}_{W^{\ireg}_{-1-\delta}(\DtauR)}^2\nonumber\\
&+\norm{\varphi}_{F^{\ireg}(\Scritau)}^2 \nonumber\\
&{}\leq \CInrpWave \bigg(
\norm{ r\VOp\varphi}_{W^{\ireg}_{\alpha-2}(\StauiR)}^2
+\norm{\varphi}_{W^{\ireg+1}_{-2}(\StauiR)}^2\nonumber\\
&\qquad\quad+\norm{\varphi}_{W^{\ireg+1}_{0}(\DtauRcR)}^2 
  +\sum_{\timefunc\in\{\timefunc_1,\timefunc_2\}} \norm{\varphi}_{W^{\ireg+1}_{\alpha}(\StautRcR)}^2 
  +\|\vartheta\|_{W^{\ireg}_{\alpha-3}(\DtauRc)}^2 \bigg).
\label{eq:GeneralrpSpinWave:ConclusionHighReg}
\end{align}
\end{lemma}

\begin{proof}
For a given set of operators $\generalOps$, consider the estimate
\begin{align}
\sum_{|\mathbf{a}|\leq \ireg}\bigg(&\norm{\generalOps^{\mathbf{a}} r\VOp\varphi}_{W^{0}_{\alpha-2}(\StauR)}^2
+\norm{\generalOps^{\mathbf{a}}\varphi}_{W^{1}_{-2}(\StauR)}^2\nonumber\\
&+\norm{\generalOps^{\mathbf{a}}\varphi}_{W^{1}_{\alpha-3}(\DtauR)}^2
+\norm{\generalOps^{\mathbf{a}} M\YOp \varphi}_{W^{0}_{-1-\delta}(\DtauR)}^2\nonumber\\
&+\int_{\Scritau} M\abs{\generalOps^{\mathbf{a}}\LxiOp\varphi}^2 \diThreeVolScri \bigg)\nonumber\\
&{}\leq \CInrpWave \sum_{|\mathbf{a}|\leq \ireg}\bigg(
\norm{ r\VOp\generalOps^{\mathbf{a}}\varphi}_{W^{0}_{\alpha-2}(\StauiR)}^2
+\norm{\generalOps^{\mathbf{a}}\varphi}_{W^{1}_{-2}(\StauiR)}^2\nonumber\\
&\qquad\quad+\norm{\generalOps^{\mathbf{a}}\varphi}_{W^{1}_{0}(\DtauRcR)}^2 
  +\sum_{\timefunc\in\{\timefunc_1,\timefunc_2\}} \norm{\generalOps^{\mathbf{a}}\varphi}_{W^{1}_{\alpha}(\StautRcR)}^2 
  +\|\generalOps^{\mathbf{a}}\vartheta\|_{W^{0}_{\alpha-3}(\DtauRc)}^2 \bigg).
\label{eq:GeneralrpSpinWave:ConclusionHighRegGeneralOps}
\end{align}
In the following steps, the bound \eqref{eq:GeneralrpSpinWave:ConclusionHighRegGeneralOps} is proved for an increasingly large sequence of sets of operators until the estimate is proved for $\generalOps=\rescaledOps$, which completes the proof. 

\begin{steps}
\step{$\generalOps=\{M\LxiOp\}$} 
\label{step:rpHigherReg:xiCommutations} 
Since $M\LxiOp$ commutes through the spin-weighted wave equation \eqref{eq:GeneralrpSpinWave:Assumption}, any number of compositions of $M\LxiOp$ can be applied, and the original $r^p$ bound \eqref{eq:GeneralrpSpinWave:Conclusion} will hold with $\varphi$ and $\vartheta$ replaced by $(M\LxiOp)^\ii\varphi$ and $(M\LxiOp)^\ii\vartheta$, which proves the higher-regularity $r^p$ bound \eqref{eq:GeneralrpSpinWave:ConclusionHighRegGeneralOps} with $\generalOps=\{M\LxiOp\}$. 
\end{steps}

\step{$\generalOps=\ScriOps$ with at most one angular derivative}
\label{step:rpHigherReg:oneAngularCommutation} 
If the spin weight is negative, $s<0$, then, commuting the original spin-weighted wave equation in its expanded form \eqref{eq:ExpandSWWE} with $\hedtp$ and using the commutation relation \eqref{eq:commutatorofhedtandhedtprime}, one finds
\begin{align}
\left(
  2 (r^2+a^2) \YOp\VOp
  +{b}_{\VOp,\hedtp} \VOp
  + ({b}_{\phi,\hedtp}+c_{\phi})\LetaOp
  + ({b}_{0,\hedtp}+c_{0})
\right)\hedtp\varphi &\nonumber\\
+\left(
  - 2 \hedt\hedtp
  - f_1(\theta)\LxiOp{}\LxiOp{}
  - f_2(\theta) \LetaOp\LxiOp{}
  - f_3(\theta) \LxiOp{}
\right)\hedtp\varphi &
- \vartheta_{\hedtp} 
= 0 ,
\label{eq:ExpandSWWEcommuterhedt}
\end{align}
where
\begin{align}
{b}_{\VOp,\hedtp}={}& b_{\VOp}, 
\quad {b}_{\phi,\hedtp}={} b_{\phi}, 
\quad {b}_{0,\hedtp}={} b_0- 2(s-1) , \notag\\
\vartheta_{\hedtp}={}&\hedtp\vartheta-\frac{1}{\sqrt{2}}(\hedtp f_1(\theta))\LxiOp{}\LxiOp\varphi -\frac{1}{\sqrt{2}} (\hedtp f_3(\theta))\LxiOp\varphi ,
\label{eq:varthetahedt'}
\end{align}
and $c_\phi$, $c_0$, and the $f_i$ are given in equation \eqref{eq:ExpandSWWE:coefficients}. While in the case of $s\geq 0$, one can commute \eqref{eq:ExpandSWWE} with $\hedt$ and apply the commutation relation \eqref{eq:commutatorofhedtandhedtprime} to find that $\hedt \varphi$ satisfies an equation of the form \eqref{eq:ExpandSWWEcommuterhedt} with  $b_{\VOp,\hedtp}$, $b_{\phi,\hedtp}$, $b_0$ and $\vartheta_{\hedtp}$ replaced by
\begin{align}
{b}_{\VOp,\hedt}={}& b_{\VOp}, 
\quad {b}_{\phi,\hedt}={} b_{\phi}, 
\quad {b}_{0,\hedt}={} b_0 +2s, \notag\\
\vartheta_{\hedt}={}&\hedt\vartheta-\frac{1}{\sqrt{2}}(\hedt f_1(\theta))\LxiOp{}\LxiOp\varphi -\frac{1}{\sqrt{2}} (\hedt f_3(\theta))\LxiOp\varphi,
\label{eq:varthetahedt}
\end{align}
respectively. 

These are in the form of the spin-weighted wave equation \eqref{eq:GeneralrpSpinWave:Assumption} from the $r^p$ lemma \ref{lem:GeneralrpSpinWave}. It is clear that if the first two assumptions in lemma \ref{lem:GeneralrpSpinWave}, on the asymptotics of $b_{\VOp}$ and $b_{\phi}$, held for the original wave equation, then they hold for $b_{\VOp,\hedtp}$ and $b_{\phi,\hedtp}$ or for $b_{\VOp,\hedt}$ and $b_{\phi,\hedt}$ respectively. The scalars $\hedtp \varphi$ and $ \hedt\varphi$ have spin weight $s-1$ and $s+1$ respectively, and their spin weights lie in $\{-3,\cdots,3\}$. Furthermore, the leading-order parts of $b_{0,\hedtp}$ and $b_{0,\hedt}$ satisfy
\begin{align}
b_{0,\hedtp,0}+|s-1| +(s-1)
={}&b_{0,0} -2s+2 =(b_{0,0} +|s| + s) +2|s|+2 ,& \text{for}\ s<0,\notag\\
b_{0,\hedt,0}+\abs{s+1} +(s+1)
={}&b_{0,0}+4s+2 =(b_{0,0} +|s| +s) +2s+2 ,& \text{for}\ s\geq 0 ,
\end{align}
which means that if $b_{0,0} +|s|+s>0$, then $b_{0,\hedtp,0}+|s-1| +(s-1)>0$ and $b_{0,\hedt,0}+\abs{s+1} +(s+1)>0$. In particular, if $b_{0}$ from the original equation \eqref {eq:GeneralrpSpinWave:Assumption} satisfies assumption \eqref{Assumption:b0:Generalrp} from the $r^p$ lemma \ref{lem:GeneralrpSpinWave}, then so do $b_{0,\hedtp}$ from the commuted equation \eqref{eq:ExpandSWWEcommuterhedt} and and $b_{0,\hedt}$ from the analogue for $\hedt\varphi$. Thus, if the original spin-weighted wave equation \eqref {eq:GeneralrpSpinWave:Assumption} satisfies the hypotheses of the $r^p$ lemma \ref{lem:GeneralrpSpinWave}, then so do the $\hedtp$ or $\hedt$ commuted equations. 

Hence by applying the $r^p$ lemma \ref{lem:GeneralrpSpinWave}, the bound \eqref{eq:GeneralrpSpinWave:Conclusion} holds if we replace $\varphi$ and $\vartheta$ by $\hedtp \varphi$ and the sum of $\hedtp\vartheta$ and a $\bigOAnalytic(1)$ weight times at most two compositions of $M\LxiOp$ acting on $\varphi$. The terms involving compositions of $M\LxiOp$ acting on $\varphi$ can be estimated by the higher-regularity $r^p$ bound \eqref{eq:GeneralrpSpinWave:ConclusionHighRegGeneralOps} with $\generalOps=\{M\LxiOp\}$, which proves the higher regularity $r^p$ bound \eqref{eq:GeneralrpSpinWave:ConclusionHighRegGeneralOps} with $\generalOps=\sphereOps$ in the special case where the multiindex $\mathbf{a}$ is restricted so that there is at most one angular derivative and it is either $\hedtp$ if $s<0$ and $\hedt$ if $s\geq0$. 

\step{$\generalOps=\ScriOps$ without restriction on the number of angular derivatives} 
\label{step:rpHigherReg:AngularCommutations} 
Since any $\generalOp\in \{ M^2\LxiOp\LxiOp, M\LxiOp\LetaOp,\LetaOp\LetaOp,\TMESOp_{s}\}$ commutes with the homogeneous part of the wave equation \eqref{eq:GeneralrpSpinWave:Assumption}, the $r^p$ estimate \eqref{eq:GeneralrpSpinWave:ConclusionHighRegGeneralOps} follows trivially if we replace $\varphi$ and $\vartheta$ by $\generalOp\varphi$ and $\generalOp\vartheta$, respectively. In view of the relation \eqref{eq:expandedSssss} between $\TMESOp_{s}$ and $\mathring{S}_s$, the estimate \eqref{eq:GeneralrpSpinWave:ConclusionHighRegGeneralOps} holds if $\generalOps= \{ M^2\LxiOp\LxiOp, M\LxiOp\LetaOp,\LetaOp\LetaOp,\mathring{S}_{s}\}$. 

Consider now the higher-regularity $r^p$ bound \eqref{eq:GeneralrpSpinWave:ConclusionHighRegGeneralOps} with $\generalOps=\ScriOps$. First, consider the case where there is a sum up to an even order $2\ii$ of angular derivatives. By lemma \ref{lem:Ellipticforsecondordersphericalangularop}, the corresponding norms can be replaced by norms involving $\mathring{S}_s^\ii$, and such norms were already controlled in the previous paragraph. Now, consider the case where there is a sum up to an odd order $2\ii+1$ of angular derivatives. By the previous argument, all the terms of order up to $2\ii$ can be replaced by norms defined in terms of $\mathring{S}_s$. Since the lower-order terms are controlled, by equation \eqref{eq:hedt'hedtnormrelation} and the previous argument, the terms involving $2\ii+1$ derivatives can be controlled by terms involving lower-order terms and terms involving $\mathring{S}_s^\ii$ and either $\hedtp$ or $\hedt$ depending on whether $s<0$ or $s\geq0$. Such terms can be controlled by combining the arguments of the previous paragraph and step \ref{step:rpHigherReg:oneAngularCommutation}. 

Note that in step \ref{step:rpHigherReg:oneAngularCommutation} in equations \eqref{eq:varthetahedt'}-\eqref{eq:varthetahedt}, $\vartheta$ was replaced by the sum of one angular derivative acting on $\vartheta$ and a $\bigOAnalytic(1)$ coefficient of at most two compositions of $M\LxiOp$ acting on $\varphi$. Applying compositions of $M\LxiOp$, $\LetaOp$, or $\mathring{S}_s$ of total order $\ireg-1$ to an angular derivative of $\vartheta$ will give terms bounded by $\absHighOrder{\vartheta}{\ireg}{\ScriOps}^2$.  Similarly, applying compositions of $M\LxiOp$, $\LetaOp$, or $\mathring{S}_s$ of total order $\ireg-1$ to at most two compositions of $M\LxiOp$ acting on $\varphi$ will give terms bounded by $\abs{\ScriOps^{\mathbf{a}}\varphi}^2$, in which either $|\mathbf{a}|\leq\ireg-1$ or such that at least one term in $\ScriOps^{\mathbf{a}}$ is a $M\LxiOp$ derivative. In either case, by first proving the $r^p$ bound \eqref{eq:GeneralrpSpinWave:ConclusionHighRegGeneralOps} to order $\ireg$ with $\generalOps=\{M\LxiOp\}$ and then proving the bound with $\generalOps=\ScriOps$ with increasing orders $\ii\leq\ireg$, one finds that all the terms arising of the form $\abs{\ScriOps^{\mathbf{a}}\varphi}^2$ are controlled by earlier bounds.

\step{$\generalOps=\{M\LxiOp,\hedt,\hedtp,r\VOp\}$} 
Commuting the original wave equation \eqref{eq:ExpandSWWE} with $r\VOp$ and using the commutator relation \eqref{eq:CommutatorofYandMathcalV} for $\YOp$ and $\VOp$, one finds that $r\VOp \varphi$ satisfies
\begin{align}
\left(
  2 (r^2+a^2) \YOp\VOp
  +{b}_{\VOp,r\VOp} \VOp
  + ({b}_{\phi,r\VOp}+c_{\phi})\LetaOp
  + ({b}_{0,r\VOp}+c_0)
\right)(r\VOp \varphi) &\nonumber\\
+\left(
  - 2 \hedt\hedtp
  - f_1(\theta)\LxiOp{}\LxiOp{}
  - f_2(\theta) \LetaOp\LxiOp{}
  - f_3(\theta) \LxiOp{}
\right)(r\VOp \varphi) &
- {\vartheta}_{ r\VOp}
= 0 ,
\label{eq:ExpandSWWEcommuterV}
\end{align}
where
\begin{subequations}
\begin{align}
{b}_{\VOp,r\VOp}={}& b_{\VOp}+\frac{2(r^2+a^2)}{r}, 
\quad  {b}_{\phi,r\VOp}={} b_{\phi}-\frac{4ar}{r^2+a^2},
\quad {b}_{0,r\VOp}={} b_0 +1, \\
{\vartheta}_{ r\VOp}
={}&r\VOp\vartheta
-\frac{(r^2-a^2)(\Delta-2Mr)}{r^2+a^2}\YOp \VOp \varphi- \frac{r\Delta}{2(r^2+a^2)}\partial_r (b_{\phi}+c_{\phi})\LetaOp\varphi \notag\\
&-\left(\frac{r\Delta}{2(r^2+a^2)}\partial_r(r^{-1}b_{\VOp})+\frac{4Mr}{r^2+a^2}-\frac{2r^2+a^2}{r^2}\right)(r\VOp\varphi) 
- \frac{r\Delta}{2(r^2+a^2)}\partial_r (b_{0}+c_{0})\varphi
\label{eq:varthetarV}
\end{align}
\end{subequations}
and the $c_\phi$, $c_0$, and $f_i$ are again given in equation \eqref{eq:ExpandSWWE:coefficients}. 
The commuted wave equation \eqref{eq:ExpandSWWEcommuterV} can be rewritten as
\begin{align}
\widehat\squareS_s (r\VOp \varphi)
 +{b}_{\VOp,r\VOp}  \VOp(r\VOp \varphi)
 +{b}_{\phi,r\VOp} \LetaOp (r\VOp \varphi)
 +{b}_{0,r\VOp} (r\VOp \varphi)
={}&
 {\vartheta}_{ r\VOp} .
\label{eq:GeneralrpSpinWave:AssumptionCommuteV}
\end{align}

The $\YOp\VOp$ term in $\vartheta_{r\VOp}$ can be expanded using the spin-weighted wave equation that $\varphi$ is assumed to satisfy. Doing so, one finds that $\vartheta_{r\VOp}$ is the sum of $r\VOp$ applied to $\vartheta$ and a sum of terms given by $\bigOAnalytic(1)$ coefficients multiplied by terms of the form either $\sphereOps^{\mathbf{a}}\varphi$ with $|\mathbf{a}|\leq 2$ or $r\VOp\varphi$. 

Again, this is in the form of equation \eqref{eq:GeneralrpSpinWave:Assumption} from the $r^p$ lemma \ref{lem:GeneralrpSpinWave}, and again, it is clear that the first two assumptions in lemma \ref{lem:GeneralrpSpinWave}, on the asymptotics of $b_{\VOp}$ and $b_{\phi}$, hold for the commuted equation \eqref{eq:GeneralrpSpinWave:AssumptionCommuteV} if they held for the original equation \eqref{eq:GeneralrpSpinWave:Assumption}. The scalar $r\VOp\varphi$ has the same spin as $\varphi$, and the condition on the leading-order coefficient in $b_{0,r\VOp}$ is $0< b_{0,r\VOp,0}+\abs{s}+s$ $=b_{0,0}+\abs{s}+s+1$, so that assumption \eqref{Assumption:b0:Generalrp} from lemma \ref{sec:rpForSpinWeightedWaveEquations} holds for the commuted equation \eqref{eq:GeneralrpSpinWave:AssumptionCommuteV} if $0\leq b_{0,0}+\abs{s}+s$ holds, which was assumption \eqref{Assumption:b0:Generalrp} for the original equation. In particular, if one starts with a spin-weighted wave equation of the form \eqref{eq:GeneralrpSpinWave:Assumption} that satisfies the three hypotheses of the $r^p$ lemma \ref{lem:GeneralrpSpinWave}, then commuting with $r\VOp$ will give a new equation of the same form that also satisfies the three hypotheses. 

Thus, for any multiindex $\mathbf{a}$, when considering $\{M\LxiOp,\hedt,\hedtp,r\VOp\}^{\mathbf{a}}$, there will be some number of operators from $\ScriOps$ and some number of compositions of $r\VOp$. Since if $\varphi$ satisfies the hypotheses of the $r^p$ lemma \ref{lem:GeneralrpSpinWave}, then so does $r\VOp\varphi$, it follows by induction on the order of the composition of $r\VOp$ that each $\generalOps^\mathbf{a}\varphi$ (where $\generalOps=\{M\LxiOp,\hedt,\hedtp,r\VOp\}$) satisfies a spin-weighted wave equation satisfying the three hypotheses of the $r^p$ lemma \ref{lem:GeneralrpSpinWave}. 

It remains to treat the corresponding $\vartheta$ terms. From applying $r\VOp$, there is one term involving $r\VOp\vartheta$ and additional terms of the form either $\ScriOps^\mathbf{a}\varphi$ with $|\mathbf{a}|\leq 2$ or $r\VOp\varphi$. Recall from step \ref{step:rpHigherReg:oneAngularCommutation}, the terms arising from commutation with $\hedtp$ were either the $\hedtp\vartheta$ or $\ScriOps^{\mathbf{a}}\varphi$ with $|\mathbf{a}|\leq2$, and similarly for $\hedt$. Thus, from commuting $\{M\LxiOp,\hedt,\hedtp,r\VOp\}^{\mathbf{a}}$ through the spin-weighted wave equation \eqref{eq:GeneralrpSpinWave:Assumption}, the terms that arise are of either of the form $\{M\LxiOp,\hedt,\hedtp,r\VOp\}^{\mathbf{a}}\vartheta$ or of the form $\{M\LxiOp,\hedt,\hedtp,r\VOp\}^{\mathbf{b}}\varphi$. All such $\{M\LxiOp,\hedt,\hedtp,r\VOp\}^{\mathbf{b}}\varphi$ arise from the additional terms in equation \eqref{eq:varthetarV} from commuting with $r\VOp$, from the additional terms in equation \eqref{eq:varthetahedt} from commuting with $\hedt$, or from the additional terms in equation \eqref{eq:varthetahedt'} from commuting with $\hedtp$. In commuting $\{M\LxiOp,\hedt,\hedtp,r\VOp\}^{\mathbf{a}}$ through the spin-weighted wave equations, the operators can at most once be applied so that they generate terms arising in one of the three equations \eqref{eq:varthetahedt'}, \eqref{eq:varthetahedt}, or \eqref{eq:varthetarV}, with all other factors either being applied to $\varphi$ or to one of the coefficients. If the $\vartheta_{r\VOp}$ equation \eqref{eq:varthetarV} is applied, then either the number of $r\VOp$ terms is reduced or the total order is reduced. If the $\vartheta_{\hedtp}$ or $\vartheta_{\hedt}$ equation \eqref{eq:varthetahedt'} or \eqref{eq:varthetahedt} is applied, then the number of $r\VOp$ terms is unchanged, and either the number of angular derivatives is reduced or the total order is reduced. Thus, by applying a triple induction on total order, within that order of $\sphereOps$ derivatives, and within that order of $M\LxiOp$ derivatives, one obtains that the the $r^p$ estimate \eqref{eq:GeneralrpSpinWave:ConclusionHighRegGeneralOps} holds with $\generalOps=\{M\LxiOp,\hedt,\hedtp,r\VOp\}$. 

\step{$\generalOps=\rescaledOps$}
In the domain of consideration $r\geq\rCutOff-M$, the operator $M\YOp$ can be expanded in terms of $M\LxiOp$, $\hedt$, $\hedtp$, $r\VOp$ and, conversely, the operator $M\LxiOp$ can be expanded in terms of $M\YOp$, $\hedt$, $\hedtp$, $r\VOp$. The coefficients appearing in these expansions are all at most $\bigOAnalytic(1)$, which implies the equivalence of the norms generated by these two sets of operators. To complete the proof, note that, on $\Scri^+$, $r\VOp$ vanishes on conformally regular functions and that $M\YOp=2M\LxiOp$. 
\end{proof}

\subsection{Spin-weighted wave equations in the early region}

The following lemma allows norms on the hyperboloid $\Staut$ to be estimated in terms of norms on the hypersurface $\Stauini$, which extends to spacelike infinity. 

\begin{lemma}[Controlling the early region]
\label{lem:GeneralrpSpinWaveFar}
Under the same assumptions of Lemma \ref{lem:GeneralrpSpinWaveHighReg}, for any $\ireg \in \Naturals$, there are constants $\rCutOffInrpWave=\rCutOffInrpWave(b_0,b_\phi,b_{\VOp})$ and $\CInrpWave=\CInrpWave(b_0,b_\phi,b_{\VOp})$ such that, if $\varphi$ and $\vartheta$ are spin-weighted scalars satisfying \eqref{eq:GeneralrpSpinWave:Assumption}, 
then, for all $\rCutOff\geq\rCutOffInrpWave$, $\alpha\in[\delta,2-\delta]$, and $\timefunc\leq\timefunc_0$, 
\begin{align}
&\norm{ r\VOp\varphi}_{W^{\ireg}_{\alpha-2}(\StautR)}^2
+\norm{\varphi}_{W^{\ireg+1}_{-2}(\StautR)}^2\nonumber\\
&+\norm{\varphi}_{W^{\ireg+1}_{\alpha-3}(\DtautearlyR)}^2
+\norm{M\YOp\varphi}_{W^{\ireg}_{-1-\delta}(\DtautearlyR)}^2\nonumber\\
&+\norm{\varphi}_{F^{\ireg}(\Scritini)}^2 \nonumber\\
&\leq \CInrpWave \bigg(
\norm{\varphi}_{H^{\ireg+1}_{\alpha-1}(\Stauini)}^2
+\norm{\varphi}_{W^{\ireg+1}_{0}(\DtautearlyRcR)}^2
+\norm{\varphi}_{W^{\ireg+1}_{\alpha}(\StautRcR)}^2  
+\norm{\vartheta}_{W^{\ireg}_{\alpha-3}(\DtautearlyRc)}^2 \bigg).
\label{eq:GeneralrpSpinWave:ConclusionIntegratedPast}
\end{align}
\end{lemma}
\begin{proof}
Throughout this proof, $\lesssim$ is used to mean $\lesC{b_0,b_\phi,b_{\VOp}}$, and we use mass normalization as in definition \ref{def:massnormalization}. The method for increasing the regularity that appeared in the proof of the higher regularity $r^p$ lemma \ref{lem:GeneralrpSpinWaveHighReg} applies in exactly the same way. Thus, it is sufficient to modify the proof of the original $r^p$ lemma \ref{lem:GeneralrpSpinWave}. The only change that must be made is in step \ref{step:rp:TreatTheEnergyOnHyperboloids}, where the energy on the $\StauiRc$ must be replaced by an energy on $\Stauini$. The energy densities $\energyComponentiX$ can be estimated following the same ideas appearing in the step \ref{step:rp:TreatTheEnergyOnHyperboloids} of the proof of the $r^p$ lemma \ref{lem:GeneralrpSpinWave}. The major change is that on the Cauchy slice $\Normal_a \VOp^a\sim 1$ instead of $M^2r^{-2}$. It remains the case that $\Normal_a\xi^a$ $\sim1$ $\sim\di\Normal_a\YOp^a$. Thus, one finds
\begin{align}
&\hspace{-4ex}\left|\int_{\Stauini} \sum_{i=1}^{9}\sum_{X\in\{\VOp,\YOp\}} \Normal_a \momentumiX^a \LerayForm{\Normal} \right|\nonumber\\
\lesssim{}& \int_{\Stauini} \bigg(
  M^{-\alpha}r^\alpha \abs{\VOp\varphi}^2 
  +\abs{\YOp\varphi}^2 
  +M^{-\alpha+2}r^{\alpha-2}\abs{\hedtp\varphi}^2 
  +M^{-\alpha+2}r^{\alpha-2} \abs{\varphi}^2 
\bigg)\diThreeVol \nonumber\\
\lesssim{}& \int_{\Stauini} 
\sum_{|\mathbf{a}|\leq 1} 
M^{-\alpha+2}r^{\alpha-2+2|\mathbf{a}|}\abs{\unrescaledOps^{\mathbf{a}}\varphi}^2
\diThreeVol .
\end{align}
The stated result now follows from the fact that, for any $\ireg$, 
\begin{align}
\int_{\Stauini} 
\sum_{|\mathbf{a}|\leq 1} 
\sum_{|\mathbf{b}|\leq\ireg}
M^{-\alpha+2}r^{\alpha-2+2|\mathbf{a}|}\abs{\unrescaledOps^{\mathbf{a}}\rescaledOps^{\mathbf{b}}\varphi}^2
\diThreeVol
\lesssim
\int_{\Stauini} 
\sum_{|\mathbf{a}|\leq \ireg+1} 
M^{-\alpha+2}r^{\alpha-2+2|\mathbf{a}|}\abs{\unrescaledOps^{\mathbf{a}}\varphi}^2
\diThreeVol ,
\end{align}
which completes the proof. 
\end{proof}

\section{The spin-weight \texorpdfstring{$-2$}{-2} Teukolsky equation}
\label{sec:spin-2TeukolskyEstimates}
In this section, we consider the field $\psibase[-2]$ of spin-weight $-2$ that solves the Teukolsky equation \eqref{eq:TeukolskyRegular-2}.

\subsection{Extended system}
\label{sec:ExtendedSystem}
This section introduces a collection $\{\psibase[-2]\}_{i=0}^4$ of conformally regular derivatives of $\psibase[-2]$, a collection of rescalings $\{\raddegphi{i}\}_{i=0}^4$ that are (depending on the index) divergent or vanishing at the horizon, shows that the $\raddegphi{i}$ satisfy a system of wave equations, and finally shows that the $W^{\ireg}_{\alpha}(\StauR)$ norms of the $\psibase[-2][i]$ and $\raddegphi{i}$ are equivalent for sufficiently large $\rCutOff$. 

\begin{definition}
\label{def:psi(i)-2}
Let $\psibase[-2]$ be a scalar of spin-weight $-2$. Define
\index{P3psibase-2i@$\psibase[-2][i]$}
\begin{align}
\psibase[-2][i]={}&\left ( \frac{a^2 + r^2}{M}\VOp \right )^i \psibase[-2], 
\quad \quad \quad 0\leq i\leq 4 . 
\label{eq:psi(i)-2-def}
\end{align}
\end{definition}

\begin{definition}
\label{def:raddegphi}
\index{P3phiCurlySpinMinus2Rad@$\raddegphi{i}$}
\standardMinusTwoHypothesisDefinePsibasei. 
Define
\begin{subequations}
\label{eq:def:raddegphi}
\begin{align}
\raddegphi{0}
={}& \Big(\frac{\Delta}{r^2 + a^2 }\Big)^{2} \psibase[-2][0] ,\\
\raddegphi{i}
={}& \frac{2}{M}\frac{(r^2+a^2)^2}{\Delta} \VOp\raddegphi{i-1}
&1\leq i\leq 4.
\end{align}
\end{subequations}
\end{definition}

\begin{remark}
Compared to the quantity introduced by Ma in \cite[Appendix A]{SiyuanMathesis}, which we denote here by $\hat\phi^{0,\textrm{Ma}}_{-2}$, we have $\raddegphi{0}=\sqrt{r^2+a^2}\bar{\kappa}_{1'}{}^2\kappa_{1}{}^{-2}\hat\phi^{0,\textrm{Ma}}_{-2}$  where the first factor $\sqrt{r^2+a^2}$ is to make the quantity nondegenerate at future null infinity and the other factor $\bar{\kappa}_{1'}{}^2\kappa_{1}{}^{-2}$ corresponds to a spin rotation of the frame.
\end{remark}

\begin{lemma} 
If $\psibase[-2]$ is a solution to \eqref{eq:TeukolskyRegular-2}, then the variables $\raddegphi{0},\dots,\raddegphi{4}$ satisfy the system
\begin{align}
\label{eq:rescaledwavesys5eqs}
\widehat\squareS_{-2}\left(\begin{array}{c}
\raddegphi{0}\\
\raddegphi{1}\\
\raddegphi{2}\\
\raddegphi{3}\\
\raddegphi{4}
\end{array}\right)={}&\mathbf{A}\left(\begin{array}{c}
\raddegphi{0}\\
\raddegphi{1}\\
\raddegphi{2}\\
\raddegphi{3}\\
\raddegphi{4}
\end{array}\right)+\mathbf{B}\LetaOp{}\left(\begin{array}{c}
\raddegphi{0}\\
\raddegphi{1}\\
\raddegphi{2}\\
\raddegphi{3}\\
\raddegphi{4}
\end{array}\right)
+\mathbf{C} \VOp\left(\begin{array}{c}
\raddegphi{0}\\
\raddegphi{1}\\
\raddegphi{2}\\
\raddegphi{3}\\
\raddegphi{4}
\end{array}\right),
\end{align}
with
\begin{subequations}
\begin{align}
\mathbf{A}={}&\left(\begin{array}{cc}
- \frac{4 r (M + r)}{a^2 + r^2} & \frac{4 M (M a^2 + a^2 r - 3 M r^2 + r^3)}{(a^2 + r^2)^2}\\
- \frac{6 r (a^4 + 3 M a^2 r + a^2 r^2 -  M r^3)}{M (a^2 + r^2)^2} & \frac{2 (a^4 - 12 M a^2 r - 2 a^2 r^2 + 4 M r^3 - 3 r^4)}{(a^2 + r^2)^2}\\
- \frac{6 a^2 (a^4 + 6 M a^2 r - 10 M r^3 -  r^4)}{M^2 (a^2 + r^2)^2} & - \frac{20 a^2 (M a^2 + a^2 r - 3 M r^2 + r^3)}{M (a^2 + r^2)^2}\\
- \frac{12 a^2 (3 M a^4 - 2 a^4 r - 24 M a^2 r^2 - 2 a^2 r^3 + 5 M r^4)}{M^3 (a^2 + r^2)^2} & \frac{2 a^2 (-13 a^4 + 82 M a^2 r - 30 M r^3 + 13 r^4)}{M^2 (a^2 + r^2)^2}\\
\frac{24 a^4 (a^4 + 30 M a^2 r - 34 M r^3 -  r^4)}{M^4 (a^2 + r^2)^2} & \frac{128 a^4 (M a^2 + a^2 r - 3 M r^2 + r^3)}{M^3 (a^2 + r^2)^2}
\end{array}\right.\nonumber\\
{}&\hspace{-6ex}\left.\begin{array}{ccc}
0 & 0 & 0\\
\frac{2 M (M a^2 + a^2 r - 3 M r^2 + r^3)}{(a^2 + r^2)^2} & 0 & 0\\
\frac{2 (a^4 - 12 M a^2 r - 2 a^2 r^2 + 4 M r^3 - 3 r^4)}{(a^2 + r^2)^2} & 0 & 0\\
- \frac{2 (20 M a^4 + 17 a^4 r - 69 M a^2 r^2 + 17 a^2 r^3 + 3 M r^4)}{M (a^2 + r^2)^2} & - \frac{4 r (M + r)}{a^2 + r^2} & 0\\
\frac{60 a^2 (- a^4 + 10 M a^2 r - 6 M r^3 + r^4)}{M^2 (a^2 + r^2)^2} & - \frac{40 a^2 (M a^2 + a^2 r - 3 M r^2 + r^3)}{M (a^2 + r^2)^2} & - \frac{4 (a^4 - 9 M a^2 r + a^2 r^2 + 7 M r^3)}{(a^2 + r^2)^2}
\end{array}\right),\\
\mathbf{B}={}&- \frac{2 a}{M^3 (a^2 + r^2)}\left(\begin{array}{ccccc}
4 M^3 r & 0 & 0 & 0 & 0\\
3 M^2 (a^2 -  r^2) & 2 M^3 r & 0 & 0 & 0\\
-12 M a^2 r & 4 M^2 (a^2 -  r^2) & 0 & 0 & 0\\
-12 a^2 (a^2 -  r^2) & -28 M a^2 r & 3 M^2 (a^2 -  r^2) & -2 M^3 r & 0\\
\frac{48 a^4 r}{M} & -40 a^2 (a^2 -  r^2) & -40 M a^2 r & 0 & -4 M^3 r
\end{array}\right),\\
\mathbf{C}={}&- \frac{4 (M a^2 + a^2 r - 3 M r^2 + r^3)}{\Delta}\left(\begin{array}{ccccc}
0 & 0 & 0 & 0 & 0\\
0 & 0 & 0 & 0 & 0\\
0 & 0 & 0 & 0 & 0\\
0 & 0 & 0 & 1 & 0\\
0 & 0 & 0 & 0 & 2
\end{array}\right).
\end{align}
\end{subequations}
\end{lemma}
\begin{proof}
The rescaling in the variable $\raddegphi{0}$ eliminates the $\YOp$ terms of \eqref{eq:TeukolskyRegular-2} to yield the first row of the system.
Repeated application of the commutator
\begin{align}
\widehat\squareS_s\Bigl(\frac{(a^2 + r^2)^2}{\Delta}\VOp\varphi\Bigr)={}&\frac{(a^2 + r^2)^2}{\Delta}\VOp\widehat\squareS_s(\varphi)
 + \frac{4 a r}{a^2 + r^2}\LetaOp\bigl(\frac{(a^2 + r^2)^2}{\Delta}\VOp\varphi\bigr)\nonumber\\
& -  \frac{4 (M a^2 + a^2 r - 3 M r^2 + r^3)}{\Delta}\VOp\bigl(\frac{(a^2 + r^2)^2}{\Delta}\VOp\varphi\bigr)\nonumber\\
& + \frac{a (a -  r) (a + r)}{a^2 + r^2}\LetaOp\varphi
 -  \frac{2 (a^4 - 10 M a^2 r + 6 M r^3 -  r^4)}{\Delta}\VOp\varphi\nonumber\\
& + \frac{(2 M a^4 + a^4 r - 9 M a^2 r^2 + a^2 r^3 + M r^4) \varphi}{(a^2 + r^2)^2}
\end{align}
gives the remaining rows.
\end{proof}

\begin{lemma}
\label{lem:equiavelenceOfradphiAndraddegphi}
\standardMinusTwoHypothesisDefinePsibasei. 
Let $\{\raddegphi{i}\}_{i=0}^4$ be as in definition \ref{def:raddegphi}. Let $\ireg\in\Naturals$, $\beta\in\Reals$, $\rCutOff \geq 10M$ and $0\leq i\leq 4$. We have \begin{align}
\sum_{\iPigeonReg=0}^i \norm{\psibase[-2][\iPigeonReg]}_{W^{\ireg}_{\beta}(\StautR)}
\simC{}{}& \sum_{\iPigeonReg=0}^i \norm{\raddegphi{\iPigeonReg}}_{W^{\ireg}_{\beta}(\StautR)} .
\label{eq:radpsiAndraddegphiEquivalence}
\end{align}
Furthermore, for $\alpha\in[0,2]$, 
\begin{align}
\sum_{\iPigeonReg=0}^i &\left(
  \norm{r\VOp\radpsi{\iPigeonReg}}_{W^{\ireg-1}_{\alpha-2}(\StautR)}
  +\norm{\radpsi{\iPigeonReg}}_{W^{\ireg}_{-2}(\StautR)}
\right)\nonumber\\
&{}\simC{} \sum_{\iPigeonReg=0}^i \left(
  \norm{r\VOp\raddegphi{\iPigeonReg}}_{W^{\ireg-1}_{\alpha-2}(\StautR)}
  +\norm{\raddegphi{\iPigeonReg}}_{W^{\ireg}_{-2}(\StautR)}
\right) .
\label{eq:radpsiAndraddegphiEquivalenceWithV}
\end{align}
\end{lemma}

\begin{proof}
The first step is to prove that, given $\rCutOff\geq10M$ sufficiently large, each $\raddegphi{i}$ is a linear combination of the $\radpsi{\iPigeonReg}$ with $0\leq \iPigeonReg\leq i$ and coefficients that are analytic in $\rInv$ and vice versa. Let $\hatVOp=M^{-1}(r^2+a^2)\VOp$ and extend this analytically through $\rInv=0$. First, observe that $\Delta^2(r^2+a^2)^{-1}$ and its inverse are analytic in $\rInv$ on intervals corresponding to $r\geq\rCutOff$ and $\rInv$ not excessively negative. Second, observe that $\raddegphi{0}=\Delta^2(r^2+a^2)^{-2}\psibase[-2]$. Third, observe that $\psibase[-2][i]=\hatVOp\radpsi{i-1}$ and $\raddegphi{i}=2(\Delta(r^2+a^2)^{-1})^{-1}\hatVOp\raddegphi{i-1}$ for $1\leq i\leq 4$. Fourth, observe that if the operator $\hatVOp$ is applied to any function that is analytic in $\rInv$ on an interval extending through $\rInv=0$, then the result is also analytic in $\rInv$ on the same interval. The claim holds for $i=0$ from the first two observations. From the third and fourth observations and induction, the claim follows for $1\leq i\leq 4$.

Since the $\psibase[-2][i]$ and $\raddegphi{i}$ are linear combinations of each other with bounded coefficients (for $r\geq\rCutOff$, which prevents the divergence or vanishing of powers of $\Delta(r^2+a^2)^{-1}$), it follows that, for any $\alpha$,
\begin{align}
\sum_{\iPigeonReg=0}^i \norm{\radpsi{\iPigeonReg}}_{W^{0}_{\alpha}(\StautR)}
\sim{}& \sum_{\iPigeonReg=0}^i \norm{\raddegphi{\iPigeonReg}}_{W^{0}_{\alpha}(\StautR)} .
\end{align}
Since, for any $\alpha\in\Reals$, the operators $r\VOp$, $\YOp$, $\hedt$, and $\hedtp$ take $r^\alpha\bigOAnalytic(1)$ functions to $r^\alpha\bigOAnalytic(1)$ functions, the same estimate remains true when increasing the level of regularity from $0$ to $\ireg$. This proves estimate \eqref{eq:radpsiAndraddegphiEquivalence}.

Now consider estimate \eqref{eq:radpsiAndraddegphiEquivalenceWithV}. From estimate \eqref{eq:radpsiAndraddegphiEquivalence} with $\beta=-2$, there is the equivalence of $\sum_{\iip=0}^\ii\norm{\radpsi{\iip}}_{W^{\ireg}_{-2}(\StautR)}^2$ and $ \sum_{\iip=0}^\ii\norm{\raddegphi{\iip}}_{W^{\ireg}_{-2}(\StautR)}^2$. Observe that if a prefactor $f$ is analytic in $\rInv=r^{-1}$, then its $\VOp$ derivative is $\bigOAnalytic(r^{-2})$ and $r\VOp f$ is $\bigOAnalytic(r^{-1})$. Thus, when considering $r\VOp\sum_{\iip=0}^\ii\radpsi{\iip}$ and $r\VOp\sum_{\iip=0}^\ii\raddegphi{\iip}$ the difference is bounded by a linear combination of the $\radpsi{\iip}$ or of the $\raddegphi{\iip}$ each with coefficients decaying like $r^{-1}$. Since $(\alpha-2)-2\leq -2$, the lower-order terms arising from comparing $\sum_{\iip=0}^\ii\norm{r\VOp\radpsi{\iip}}_{W^{0}_{\alpha-2}(\StautR)}$ and $\sum_{\iip=0}^\ii\norm{r\VOp\raddegphi{\iip}}_{W^{0}_{\alpha-2}(\StautR)}$ are dominated by $\sum_{\iip=0}^\ii \norm{\raddegphi{\iip}}_{W^{0}_{-2}(\StautR)}\sim\sum_{\iip=0}^\ii \norm{\radpsi{\iip}}_{W^{0}_{-2}(\StautR)}$. The same holds after commuting with $r\VOp$, $\YOp$, $\hedt$, and $\hedtp$, which completes the proof.
\end{proof}

\subsection{Basic energy and Morawetz (BEAM) condition}
\label{sec:Assumptions}

\begin{definition}
\label{def:basichighorderenergy}
Let $\Sigma$ be a smooth, achronal hypersurface. Let $\Normal$ be a local map from $\Sigma$ to $T\mathcal{M}$ such that $\Normal$ is always normal to $\Sigma$. Let $\varphi$ be a spin-weighted scalar field. Define 
\index{E2Energy@$E^k_\Sigma$}
\index{D1di3Leray@$\LerayForm{\Normal}$}
\begin{align} \label{eq:E0def}
E^1_\Sigma(\varphi)
={}& M \int_{\Sigma}\left(
  (\Normal_aY^a)\abs{ \VOp\varphi}^2
  +(\Normal_a V^a)\abs{\YOp\varphi}^2
  +(\Normal_a(V^a+Y^a))r^{-2}(\abs{\hedt\varphi}^2+\abs{\hedtp\varphi}^2)
\right)\LerayForm{\Normal} ,
\end{align}
where $\LerayForm{\Normal}$ denotes a Leray form as in definition \ref{def:volumeForms}. 
Further, for $\ireg \in \Integers^+$, let
\begin{align}
E^{\ireg}_\Sigma(\varphi)
={}& \sum_{|\mathbf{a}| \leq \ireg-1} M^{2|\mathbf{a}|} E^1_\Sigma(\unrescaledOps^{\mathbf{a}} \varphi) .
\end{align}
\end{definition}

\begin{definition}
\label{def:Wknorm} 
Let $\timefunc_1, \timefunc_2 \in \Reals$, $\timefunc_1 < \timefunc_2$. Let $\varphi$ be a spin-weighted scalar. Define 
\begin{align} 
\BEAMbulk^1_{\timefunc_1, \timefunc_2}(\varphi) =  
\int_{\Dtau^{10M}} M^3 r^{-3} \sum_{|\mathbf{a}| = 1}\abs{\unrescaledOps^{\mathbf{a}} \varphi}^2 \diFourVol
  +\int_{\Dtau} M r^{-3} \abs{\varphi}^2 \diFourVol .
\end{align}
Further, for $\ireg \in \Integers^+$, let 
\begin{align}
\BEAMbulk^{\ireg}_{\timefunc_1, \timefunc_2}(\varphi)
={}& \sum_{|\mathbf{a}|\leq \ireg-1} M^{2|\mathbf{a}|} \BEAMbulk^1_{\timefunc_1,\timefunc_2}(\unrescaledOps^{\mathbf{a}} \varphi) .
\end{align}
\end{definition} 

We now introduce our main basic energy and Morawetz (BEAM) assumption. This is given in terms of the $\psibase[-2][i]$ for $i\in\{0,1,2\}$. Estimates of this form have been proved in the very slowly rotating case $|a|\ll M$ using a $3\times 3$ version of the system \eqref{eq:rescaledwavesys5eqs} \cite{2017arXiv170807385M}. Similar estimates have been proved using an equation for a variable similar to $\psibase[-2][2]$ and the equations relating $\psibase[-2][0]$ and $\psibase[-2][1]$ \cite{2017arXiv171107944D}.

\begin{definition}[BEAM condition for $\protect{\psibase[-2]}$]
\label{def:BEAM}
\standardMinusTwoHypothesisNotBEAMOnlyTwo. 
We shall say that the BEAM condition holds if for all sufficiently large $\ireg\in\Naturals$ and all $\timefunc_2\geq\timefunc_1\geq\timefunc_0$, 
\begin{align}
\sum_{i=0}^2 \left ( 
E^{\ireg}_{\Sigma_{\timefunc_2}}(\psibase[-2][i]) + 
\BEAMbulk^{\ireg}_{\timefunc_1, \timefunc_2}(\psibase[-2][i]) \right ) 
\lesssim{}& \sum_{i=0}^2 E^{\ireg}_{\Staui}(\psibase[-2][i])   .  
\end{align}
\end{definition}

\begin{definition}[Spin-weight $-2$ data norm on $\Stauit$]
\label{def:initialDataNormMinusTwoTest}
\standardHypothesisOnDelta. 
\standardMinusTwoHypothesisDefinePsibasei. 
For $\ireg \in \Integers^+$, the initial data norm for $\psibase[-2]$ on $\Stauini$, with regularity $\ireg$ is 
\index{I2initialDataMinusTwo@$\NEWinienergyminustwo{\ireg}$}
\begin{align}
\NEWinienergyminustwo{\ireg}
={}&   
\sum_{\ii=0}^{4}\left(
  \norm{\psibase[-2][\ii]}_{W^{\ireg}_{-2}(\Stauit)}^2
  +\norm{r\VOp\psibase[-2][\ii]}_{W^{\ireg -1}_{-\delta}(\Stauit)}^2
\right) .
\end{align}

\end{definition}

\subsection{Decay estimates}
\label{sec:-2:BasicDecay}

This section proves three results. The first is the boundedness of various weighted norms. These bounds are proved using the $r^p$ lemma \ref{lem:GeneralrpSpinWaveHighReg}. The second is a series of pointwise-in-$\timefunc$ decay-estimates for various energies. The third gives improved rates of decay when $\LxiOp^{\ij}$ is applied. Because of the form of the BEAM assumption, the components $\radpsi{\ii}$ for $\ii\in\{0,1,2\}$ are treated together. Further estimates are proved when $\ii=3$ and then $\ii=4$ are also included.

\begin{lemma}[\texorpdfstring{$r^p$}{rp} estimate for \texorpdfstring{$\radpsi{j}$}{psi j}]
\label{lem:rpestimatesforallradpsii}
\standardHypothesisOnDelta. 
\standardMinusTwoHypothesisNotBEAM. 
For $\iip\in\{0,\ldots,4\}$, define $\ell(\iip)=\max(0,\iip-2)$. 
Let $\ii\in\{2,3,4\}$ and $\alpha\in[\delta,2-\delta]$. 
\standardMinusTwoHypothesisBEAM.
If $\ireg\in\Naturals$ is sufficiently large, then 
for $\timefunc_2\geq\timefunc_1\geq\timefunc_0$, there is the bound
\begin{align}
\hspace{4ex}&\hspace{-4ex}\sum_{\iip=0}^{\ii} \Big(
  \norm{\radpsi{\iip}}_{W^{\ireg-\ell(\iip)}_{-2}(\Stau)}^2
  +\norm{ r\VOp\radpsi{\iip} }_{W^{\ireg -1-\ell(\iip)}_{\alpha-2}(\Stau)}^2
  +\norm{\radpsi{\iip}}_{W^{\ireg-1-\ell(\iip)}_{\alpha-3}(\Dtau)}^2
\Big)\notag\\
\lesssim {}& \sum_{\iip=0}^{\ii}\left(
  \norm{\radpsi{\iip}}_{W^{\ireg-\ell(\iip)}_{-2}(\Staui)}^2
  +\norm{r\VOp\radpsi{\iip}}_{W^{\ireg -1-\ell(\iip)}_{\alpha-2}(\Staui)}^2
\right) .
\label{eq:rpPsi234}
\end{align}
\end{lemma}
\begin{proof}
Consider the $\{\raddegphi{i}\}_{i=0}^4$ which are defined in definition \ref{def:raddegphi} and satisfy the $5$-component system \eqref{eq:rescaledwavesys5eqs}.
The central idea in this proof is to apply the (higher-regularity) $r^p$ lemma \ref{lem:GeneralrpSpinWaveHighReg} to each component of the $5$-component system \eqref{eq:rescaledwavesys5eqs}. To do so, it is necessary to relate the components of the matrices of coefficients $\mathbf{A}$, $\mathbf{B}$, and $\mathbf{C}$ in \eqref{eq:rescaledwavesys5eqs} to the coefficients $b_{\VOp}$, $b_{\eta}$, $b_0$ in the hypotheses of the $r^p$ lemma \ref{lem:GeneralrpSpinWaveHighReg}. The diagonal components of $\mathbf{A}$ all converge to nonpositive limits, so (when the corresponding $\raddegphi{\ii}$ terms are moved from the right of the equation to the left) the condition $b_{0,0}+|s|+s\geq0$ is always satisfied. The diagonal components of $\mathbf{B}$ are all $\bigOAnalytic(r^{-1})$. The diagonal components of $\mathbf{C}/r$ all converge to nonpositive limits, so the condition $b_{\VOp,-1}\geq 0$ always holds. The off-diagonal components of the $\mathbf{A}$, $\mathbf{B}$, $\mathbf{C}$ couple each $\raddegphi{\iip}$ to the other $\raddegphi{\ii}$, which can be treated as inhomogeneities $\vartheta$. There are no off-diagonal terms in $\mathbf{C}$, so these do not need to be treated. All the subdiagonal terms in $\mathbf{A}$ and $\mathbf{B}$ are $\bigOAnalytic(1)$. The only superdiagonal terms are the $(0,1)$ and $(1,2)$ components of $\mathbf{A}$, and these are both $\bigOAnalytic(r^{-1})$. Treating these  as inhomogeneities, $\vartheta$, will contribute $\norm{Mr^{-1}\raddegphi{i}}_{W^{\ireg}_{\alpha-3}(\Dtau)}^2$ terms on the right with $i\in\{1,2\}$; it is convenient to also add an $i=0$ term to the right. In the $r^p$ lemma \ref{lem:GeneralrpSpinWaveHighReg}, the $\Scri^+$ flux, and the spacetime integrals of $\YOp\raddegphi{i}$ are not needed to achieve the statement of the lemma and can be simply dropped. 
From all of this, the $r^p$ lemma \ref{lem:GeneralrpSpinWaveHighReg} implies, for each $\ii\in \{0,1,2,3,4\}$ and $\ireg$, there are constants $\rCutOff\geq 10M$ 
 and  $C$ 
 \footnote{In the applications of the $r^p$ lemma \ref{lem:GeneralrpSpinWaveHighReg} to each subequation of the system \eqref{eq:rescaledwavesys5eqs}, the $\rCutOffOtherBounds$ and $C$ for each $\ii$ is different, but we can take $\rCutOff$ and $C$ stated here to be the maximum value among the sets of different $\rCutOffOtherBounds$ and  different $C$, respectively, such that the estimate \eqref{eq:rpPsij:intermediateA} holds for all $\ii \in \{0,1,2,3,4\}$.}
 such that, with $\alpha\in[\delta,2-\delta]$, 
\begin{align}
&\norm{ r\VOp\raddegphi{\ii}}_{W^{\ireg}_{\alpha-2}(\StauR)}^2
+\norm{\raddegphi{\ii}}_{W^{\ireg+1}_{-2}(\StauR)}^2
+\norm{\raddegphi{\ii}}_{W^{\ireg+1}_{\alpha-3}(\DtauR)}^2\nonumber\\
&{}\leq \CInrpWave \bigg(
\norm{ r\VOp\raddegphi{\ii}}_{W^{\ireg}_{\alpha-2}(\StauiR)}^2
+\norm{\raddegphi{\ii}}_{W^{\ireg+1}_{-2}(\StauiR)}^2\nonumber\\
&\qquad\quad +\norm{\raddegphi{\ii}}_{W^{\ireg+1}_{0}(\DtauRcR)}^2
  + \sum_{\timefunc\in\{\timefunc_1,\timefunc_2\}} \norm{\raddegphi{\ii}}_{W^{\ireg+1}_{\alpha}(\StautRcR)}^2   \nonumber\\
&\qquad\quad   +\sum_{\iip=0}^{\ii-1} \|\raddegphi{\iip}\|_{W^{\ireg}_{\alpha-3}(\DtauRc)}^2
  +\sum_{\iip=0}^{\ii-1} \|\LetaOp\raddegphi{\iip}\|_{W^{\ireg}_{\alpha-3}(\DtauRc)}^2
  +\sum_{\iip=0}^2 \|Mr^{-1}\raddegphi{\iip}\|_{W^{\ireg}_{\alpha-3}(\DtauRc)}^2
 \bigg).
\label{eq:rpPsij:intermediateA}
\end{align}
From lemma \ref{lem:equiavelenceOfradphiAndraddegphi}, the $\raddegphi{\ii}$ may be replaced by $\radpsi{\ii}$. Furthermore, given an estimate of the form \eqref{eq:rpPsij:intermediateA} for some $\ii$ up to $n$, then, for $\ii=n+1$, one can control the terms involving the sum $\sum_{\iip=0}^{i-1}$
by the previous estimates,
at the expense of a further implicit constant. Furthermore, when making such a sum, for $\ii\geq2$ and $\delta\leq\alpha\leq2-\delta$, the integral over $\DtauRc$ in the final term can be divided into $\DtauRcR$ and $\DtauR$, with the integral over $\DtauRcR$ absorbed into the other integral over $\DtauRcR$, so that the final integral over $\DtauRc$ can be treated as an integral merely over $\DtauR$. Thus, using the trivial bound $\norm{\radpsi{\iip}}_{W^{\ireg}_{\alpha-3}(\DtauR)}^2\leq\norm{\radpsi{\iip}}_{W^{\ireg+1}_{\alpha-3}(\DtauR)}^2$, one finds, for $\ii\in\{2,3,4\}$, there is a constant $C$ such that
\begin{align}
&\sum_{\iip=0}^\ii \bigg(\norm{ r\VOp\radpsi{\iip}}_{W^{\ireg}_{\alpha-2}(\StauR)}^2
+\norm{\radpsi{\iip}}_{W^{\ireg+1}_{-2}(\StauR)}^2
+\norm{\radpsi{\iip}}_{W^{\ireg}_{\alpha-3}(\DtauR)}^2\bigg)\nonumber\\
&{}\leq C\bigg( \sum_{\iip=0}^{\ii} \bigg(
\norm{r\VOp\radpsi{\iip}}_{W^{\ireg}_{\alpha-2}(\StauiR)}^2
+\norm{\radpsi{\iip}}_{W^{\ireg+1}_{-2}(\StauiR)}^2\nonumber\\
&\qquad\quad +\norm{\radpsi{\iip}}_{W^{\ireg+1}_{\alpha}(\DtauRcR)}^2
  +\sum_{\timefunc\in\{\timefunc_1,\timefunc_2\}} \norm{\radpsi{\iip}}_{W^{\ireg+1}_{\alpha}(\StautRcR)}^2 \bigg) \nonumber\\
&\qquad\quad
  +\sum_{\iip=0}^2 \|Mr^{-1}\radpsi{\iip}\|_{W^{\ireg}_{\alpha-3}(\DtauR)}^2 \bigg).
\label{eq:rpPsij:intermediateB}
\end{align}

Consider $\ii=2$. Recall that the implicit constant in the bound \eqref{eq:rpPsij:intermediateB} is independent of $\rCutOff$. Thus, for $\ii=2$, by taking $\rCutOff$ sufficiently large, $Mr^{-1}$ can be taken sufficiently small relative to the implicit constant, and the $\norm{Mr^{-1}\radpsi{\ii}}_{W^{\ireg}_{\alpha-3}(\DtauR)}^2$ terms on the right can be absorbed into the $\norm{\radpsi{\ii}}_{W^{\ireg}_{\alpha-3}(\DtauR)}^2$ terms on the left. Because the energy $\sum_{\iip=0}^2 E^{\ireg}_{\Staut}(\radpsi{\iip})$ controls all derivatives, and because for $r\leq\rCutOff$, there is a constant $C(\rCutOff,p)$ such that $1\leq r^p\leq C(\rCutOff,p)$, one finds that, for any $\beta$, there is the bound $\norm{\radpsi{\ii}}_{W^{\ireg}_{\beta}(\Staut^{r_+,\rCutOff})}^2\lesssim C(\rCutOff,\beta) E^\ireg_{\Staut^{r_+,\rCutOff}}(\radpsi{\ii})$. Similarly, for $\ii=2$, the integrals over $\Dtau^{r_+,\rCutOff}$ in the bound \eqref{eq:rpPsij:intermediateB} can be controlled by $\sum_{i=0}^2E^{\ireg}_{\Staui}(\radpsi{\ii})$ if the BEAM condition from definition \ref{def:BEAM} holds. Thus, under these conditions, the claim of the lemma, inequality \eqref{eq:rpPsi234}, holds for $\ii=2$.

A similar argument holds for $\ii\in\{3,4\}$; however, it is no longer true that the energy appearing in the BEAM condition, $\sum_{\iip=0}^2 E^{\ireg}_{\Staut^{r_+,\rCutOff}}(\radpsi{\iip})$, controls $\norm{\radpsi{\ii}}_{W^{\ireg}_{\beta}(\Staut^{r_+,\rCutOff})}^2$. To overcome this, one can apply $r\VOp\in\rescaledOps$, so that, with $\ii\in\{3,4\}$, $\ell=\ii-2$, for any $\ireg,\beta$,
\begin{align}
\norm{\radpsi{\ii}}_{W^{\ireg-\ell}_{\beta}(\Staut^{r_+,\rCutOff})}^2
\lesC{\rCutOff,\beta}{} \sum_{\iip=0}^2 E^{\ireg}_{\Staut^{r_+,\rCutOff}}(\radpsi{\iip})
,
\end{align}
and similarly for the spacetime integral over $\Dtau^{r_+,\rCutOff}$ if the BEAM condition from definition \ref{def:BEAM} holds. Furthermore, 
\begin{align}
\sum_{\iip=0}^\ii \norm{\radpsi{\iip}}_{W^{\ireg-\ell(\iip)}_{\beta}(\Staut^{r_+,\rCutOff})}^2
\simC{\rCutOff,\beta}{} \sum_{\iip=0}^2 E^{\ireg}_{\Staut^{r_+,\rCutOff}}(\radpsi{\iip})
,
\end{align}
which is needed at $\timefunc=\timefunc_1$. From these and the
previous arguments, inequality \eqref{eq:rpPsi234} holds for $\ii=\{3,4\}$.
\end{proof}

\begin{lemma}[Decay estimates for $\radpsi{\ii}$]
\label{lem:EnerDecPsi234}
\standardHypothesisOnDelta. 
\standardMinusTwoHypothesisNotBEAM. 
For $\iip\in\{0,\ldots,4\}$, define $\ell(\iip)=\max(0,\iip-2)$. 
Let $\ii\in\{2,3,4\}$ and $\alpha\in[\delta,2-\delta]$. 
\standardMinusTwoHypothesisBEAM.
If $\ireg\in\Naturals$ is sufficiently large, then for $\timefunc\geq \timefunc_0$, there is the bound
\begin{align}
\sum_{\iip=0}^{\ii}&\Big(
  \norm{\radpsi{\iip}}_{W_{-2}^{\ireg+2(i-5)-\ell(\iip)}(\Staut)}^2
  +\norm{r\VOp\radpsi{\iip}}_{W_{\alpha-2}^{\ireg+2(i-5)-\ell(\iip)-1}(\Staut)}^2
  +\norm{\radpsi{\iip}}_{W_{\alpha-3}^{\ireg+2(i-5)-\ell(\iip)-1}(\Dtaut)}^2
\Big) \nonumber\\
\lesssim{}& \timefunc^{\alpha-10+9\delta+(2-2\delta)\ii} \NEWinienergyminustwo{\ireg} . 
\label{eq:rpPsij:BasictDecay}
\end{align}
\end{lemma}
\begin{proof}
The strategy of the proof is to apply the pigeonhole lemma \ref{lem:hierarchyImpliesDecay} to the $r^p$ bound \eqref{eq:rpPsi234}.
Let 
\begin{align}
F^\ii(\ireg,\alpha,\timefunc)
={}& \sum_{\iip=0}^{\ii}\Big(\Vert \radpsi{\iip} \Vert^2_{W_{-2}^{\ireg-\ell(\iip)}(\Staut)}
+\Vert r\VOp\radpsi{\iip} \Vert^2_{W_{\alpha-2}^{\ireg-\ell(\iip)-1}(\Staut)} \Big)
\end{align}
for $\alpha\geq \delta$ and $F(\ireg,\alpha,\timefunc)=0$ for $\alpha<\delta$. Here the $\ii$ denotes how many of the $\radpsi{\iip}$ are to be treated, $\ireg$ the level of regularity, and $\alpha$ the weight.

Observe that, since $r\VOp$ is in the set of operators used to define regularity $\rescaledOps$, and since $(\alpha+1)-3\geq -2$, one has that
\begin{align}
\Vert \radpsi{\ii} \Vert_{W^{(\ireg+1)-1}_{(\alpha+1)-3}(\Dtau)}^2
\gtrsim{}&\int_{\timefunc_1}^{\timefunc_2}\left(\Vert \radpsi{\ii} \Vert_{W_{-2}^{\ireg}(\Staut)}^2
+\Vert r\VOp\radpsi{\ii} \Vert_{W_{\alpha-2}^{\ireg-1}(\Staut)}^2 \right)\di\timefunc ,
\end{align}
Thus, the $r^p$ bound \eqref{eq:rpPsi234} can be written in the form
\begin{align}
F^\ii(\ireg,\alpha,\timefunc_2) +M^{-1} \int_{\timefunc_1}^{\timefunc_2} F^\ii(\ireg-1,\alpha-1,\timefunc)\di\timefunc
\lesssim{}& F^\ii(\ireg,\alpha,\timefunc_1)
\label{eq:rpPsi234:alphaHierarchy}
\end{align}
for $\alpha\in[\delta,2-\delta]$. This hierarchy of estimates is in the form treated by the pigeonhole lemma \ref{lem:hierarchyImpliesDecay}, and the assumptions \eqref{assump:HierarchyToDecay(1)} and \eqref{assump:HierarchyToDecay(2)} in the pigeonhole lemma \ref{lem:hierarchyImpliesDecay} are easily seen to be satisfied from lemma \ref{lem:Walphaknormsproperties}. 

Consider first the case $\ii=4$. From applying the pigeonhole lemma
 \ref{lem:hierarchyImpliesDecay} to the hierarchy \eqref{eq:rpPsi234:alphaHierarchy},
 one finds $F^4(\ireg-2,\alpha,\timefunc)\lesssim \timefunc^{\alpha-2+\delta} F^4(\ireg,2-\delta,\timefunc_0)$. Applying this decay estimate and the $r^p$ bound \eqref{eq:rpPsi234} a second time, one obtains the bound 
\begin{align} 
\sum_{\iip=0}^{4}\Vert \radpsi{\iip} \Vert^2_{W_{\alpha-3}^{\ireg-2-\ell(\iip)-1}(\Dtaut)}\lesssim \timefunc^{\alpha-2+\delta} F^4(\ireg,2-\delta,\timefunc_0). 
\end{align} 
A third application shows that $F^4(\ireg,2-\delta,\timefunc_0)$ is bounded by $\NEWinienergyminustwo{\ireg}$. This proves the desired inequality \eqref{eq:rpPsij:BasictDecay} in the case $\ii=4$.

Consider now lower $\ii$. Observing that $\radpsi{\ii+1}=M^{-1}(r^2+a^2)\VOp\radpsi{\ii}$, $r^2+a^2=r^2\bigOAnalytic(1)$, one finds 
\begin{align} 
\norm{r\VOp\radpsi{\ii}}_{W^{k-l(\ii)-1}_{-\delta}(\Staut)}^2 \leq \norm{\radpsi{\ii+1}}_{W^{k-l(\ii +1)}_{-2-\delta}(\Staut)}^2. 
\end{align} 
Additionally, one also has the trivial estimate 
\begin{align} 
\sum_{\iip=0}^{\ii}\norm{\radpsi{\iip}}_{W^{k-l(\iip)}_{-2}(\Staut)}^2 \leq \sum_{\iip=0}^{\ii+1}\norm{\radpsi{\iip}}_{W^{k-l(\iip)}_{-2}(\Staut)}^2. 
\end{align} 
Thus, one finds $F^{\ii}(\ireg,2-\delta,\timefunc)$ $\lesssim F^{\ii+1}(\ireg,\delta,\timefunc)$. In particular, $F^{3}(\ireg-2,2-\delta,\timefunc)$ $\lesssim F^{4}(\ireg-2,\delta,\timefunc)$ $\lesssim\timefunc^{-2+2\delta}\NEWinienergyminustwo{\ireg}$. Applying the pigeonhole lemma \ref{lem:hierarchyImpliesDecay} that treats hierarchies where the top energy is known to decay a priori, one finds $F^3(\ireg-4,\alpha,\timefunc)$ $\leq \timefunc^{\alpha-(2-\delta)-(2-2\delta)} \NEWinienergyminustwo{\ireg}$. The spacetime integral and estimate by the energy of the initial data are estimated in the same way as in the $\ii=4$ case, which proves inequality \eqref{eq:rpPsij:BasictDecay} in the case $\ii=3$. Observing $F^2(\ireg-4,2-\delta,\timefunc)$ $\lesssim F^3(\ireg-4,\delta,\timefunc)\lesssim \timefunc^{\alpha-4+3\delta}\NEWinienergyminustwo{\ireg}$ and iterating the same argument once more proves inequality \eqref{eq:rpPsij:BasictDecay} in the case $\ii=2$.
\end{proof}

We now prove decay estimates for time derivatives of $\psibase[-2][i]$. For $i\in\{2,3,4\}$, these are the strongest estimates that we obtain in this paper, but for $i\in\{0,1\}$, there is a further refinement in theorem \ref{thm:ImproEstipsi01}.

\begin{lemma}[Decay estimates for $\LxiOp^\ij\radpsi{\ii}$] 
\label{lem:ImproEnerDecaypartialt} 
\standardHypothesisOnDelta. 
\standardMinusTwoHypothesisNotBEAM.
For $\ii,\iip\in\{0,\ldots,4\}$, define $\ell(\ii,\iip)=2(i-5) -\max(0,\iip-2)$.
Let $\ii\in\{2,3,4\}$ and $\delta \leq \alpha\leq 2-\delta$. 
\standardMinusTwoHypothesisBEAM.
\begin{enumerate} 
\item 
If $\ireg,j\in\Naturals$ are such that $\ireg-3j$ is sufficiently large, then there are the energy and Morawetz estimates for $\timefunc\geq \timefunc_0$,
\begin{align}
\hspace{6ex}&\hspace{-6ex}\sum_{\iip=0}^{\ii}\Big(
  \norm{\LxiOp^\ij \radpsi{\iip}}^2_{W^{\ireg-3\ij-\ell(\ii,\iip)}_{-2}(\Staut)}
  +\norm{r\LxiOp^\ij \VOp\radpsi{\iip}}^2_{W^{\ireg-3\ij-\ell(\ii,\iip)-1}_{\alpha-2}(\Staut)}
  +\norm{\LxiOp^\ij \radpsi{\iip}}^2_{W^{\ireg-3\ij-\ell(\ii,\iip)-1}_{\alpha-3}(\Dtaut)}
\Big)\notag\\
\lesssim{}& \timefunc^{\alpha-10+9\delta+(2-2\delta)\ii-(2-2\delta)\ij}\NEWinienergyminustwo{\ireg} .
\label{eq:ImproEnerDecPsi234}
\end{align}
\item 
If $\ireg,j\in\Naturals$ are such that $\ireg -3\ij$ is sufficiently large, then 
there are the pointwise decay estimates for $\timefunc\geq \timefunc_0$
\begin{align}
\sum_{\iip=0}^{\ii}
  \absHighOrder{\LxiOp^\ij \radpsi{\iip}}{\ireg-3\ij-\ell(\ii,\iip)-7}{\rescaledOps}
\lesssim{}& r v^{-1} \timefunc^{-(1-\delta)(\frac{9}{2}+\ij-\ii)+\delta} (\NEWinienergyminustwo{\ireg})^{1/2} .
\label{eq:pointwisePsi234}
\end{align}
\end{enumerate} 
\end{lemma}

\begin{proof}
Observe that $\LxiOp$ is a symmetry of the Teukolsky equation \eqref{eq:TeukolskyRegular-2}. Furthermore, if $\psibase[-2]$ is replaced by $\LxiOp^\ij\psibase[-2]$, then the $\{\radpsi{\ii}\}$ in definition \ref{def:psi(i)-2} are replaced by $\LxiOp^\ij\psibase[-2][i]$. From the $r^p$ estimate \eqref{eq:rpPsi234} for $\LxiOp^\ij\psibase[-2]$, one has for $\delta\leq\alpha\leq2-\delta$, $\ij,\ireg\in\Naturals$, and $\ii\in\{2,3,4\}$,
\begin{align}
\hspace{4ex}&\hspace{-4ex}\sum_{\iip=0}^{\ii} \Big(
  \norm{\LxiOp^\ij\radpsi{\iip}}_{W^{\ireg-\ell(\iip)}_{-2}(\Stau)}^2
  +\norm{ r\VOp\LxiOp^\ij\radpsi{\iip} }_{W^{\ireg -1-\ell(\iip)}_{\alpha-2}(\Stau)}^2
  +\norm{\LxiOp^\ij\radpsi{\iip}}_{W^{\ireg-1-\ell(\iip)}_{\alpha-3}(\Dtau)}^2
\Big)\notag\\
\lesssim {}& \sum_{\iip=0}^{\ii}\left(
  \norm{\LxiOp^\ij\radpsi{\iip}}_{W^{\ireg-\ell(\iip)}_{-2}(\Staui)}^2
  +\norm{r\VOp\LxiOp^\ij\radpsi{\iip}}_{W^{\ireg -1-\ell(\iip)}_{\alpha-2}(\Staui)}^2
\right) .
\label{eq:ImproEnerDecPsi234:Hierarchy}
\end{align}
Similarly, the basic decay lemma \ref{lem:EnerDecPsi234} gives
\begin{align}
\sum_{\iip=0}^{\ii}&\Big(
  \norm{\LxiOp^\ij\radpsi{\iip}}_{W_{-2}^{\ireg+2(i-5)-\ell(\iip)}(\Staut)}^2
  +\norm{r\VOp\LxiOp^\ij\radpsi{\iip}}_{W_{\alpha-2}^{\ireg+2(i-5)-\ell(\iip)-1}(\Staut)}^2
  +\norm{\LxiOp^\ij\radpsi{\iip}}_{W_{\alpha-3}^{\ireg+2(i-5)-\ell(\iip)-1}(\Dtaut)}^2
\Big) \nonumber\\
\lesssim{}& \timefunc^{\alpha-10+9\delta+(2-2\delta)\ii} \NEWinienergyminustwo{\ireg+\ij} .
\end{align}

Rearranging the expansions \eqref{eq:V-Znaj-IEF}-\eqref{eq:Y-Znaj-IEF} for $\VOp$ and $\YOp$, for $r$ sufficiently large, one can write $\YOp$ as a weighted sum of $\LxiOp$, $\VOp$, and $r^{-2}\LetaOp$ all with $\bigOAnalytic(1)$ coefficients. Using this to eliminate $\YOp$ from the Teukolsky equation \eqref{eq:TeukolskyRegular-2}, rescaling equation \eqref{eq:rescaledWaveOperator} for $\widehat\squareS_{-2}$, and isolating the term $r^2\VOp\LxiOp\psibase[-2]$, one can write $r^2\VOp\LxiOp\psibase[-2]$ as a weighted sum of $(r\VOp)^2\psibase[-2]$, $r\VOp\psibase[-2]$, $r^{-1}\LetaOp (r\VOp)\psibase[-2]$, $r^{-1}\LetaOp\psibase[-2]$, $\TMESOp_{-2}\psibase[-2]$, $\LxiOp\psibase[-2]$, and $\psibase[-2]$ all with $\bigOAnalytic(1)$ coefficients. Rewriting $\LxiOp$ again as a weighted sum of $\YOp$, $\VOp$, and $r^{-2}\LetaOp$ all with $\bigOAnalytic(1)$ coefficients, one finds that $r^2\VOp\LxiOp\psibase[-2]$ can be written as a linear combination with $\bigOAnalytic(1)$ coefficients of $r^{-1}\LetaOp (r\VOp \psibase[-2])$ and terms of the form $X_2X_1\psibase[-2]$ with $X_1,X_2\in\rescaledOps\cup\{1\}$.  The commutator of the operator $M^{-1}(r^2+a^2)\VOp$ used to construct the $\radpsi{\ii}$ with any of the operators $r\VOp$, $\LxiOp$, $\LetaOp$, $\TMESOp_{-2}$, $1$, where $1$ is the identity operator, appearing in the expansion of the $r^2\VOp\LxiOp\psibase[-2]$ is in the span of $M^{-1}(r^2+a^2)\VOp$, $r\VOp$, and $1$. Thus, induction implies that a similar expansion exists for each of the $r^2\VOp\radpsi{\ii}$, but also involving the previous $\radpsi{\iip}$ with $\iip<\ii$. Thus,
\begin{align}
\sum_{\iip=0}^{\ii} \norm{r\VOp\LxiOp^{\ij+1}\radpsi{\iip}}_{W^{\ireg -2-\ell(\iip)}_{-\delta}(\Staut)}^2
\lesssim{}& \sum_{\iip=0}^{\ii} \norm{r^2\VOp\LxiOp^{\ij+1}\radpsi{\iip}}_{W^{\ireg-2-\ell(\iip)}_{-\delta-2}(\Staut)}^2 \nonumber\\
\lesssim{}&\sum_{\iip=0}^{\ii} \norm{\LxiOp^{\ij}\radpsi{\iip}}_{W^{\ireg-\ell(\iip)}_{-\delta-2}(\Staut)}^2  \nonumber\\
\lesssim{}&\sum_{\iip=0}^{\ii} \norm{\LxiOp^{\ij}\radpsi{\iip}}_{W^{\ireg-\ell(\iip)}_{-2}(\Staut)}^2 .
\end{align}
Since $\LxiOp$ is a linear combination of $\YOp$, $\VOp$ and $r^{-2}\LetaOp$ with $\bigOAnalytic(1)$ coefficients, one also finds
\begin{align}
 \sum_{\iip=0}^{\ii} \norm{\LxiOp^{\ij+1}\radpsi{\iip}}_{W^{\ireg-1-\ell(\iip)}_{-2}(\Staut)}^2
\lesssim{}& \sum_{\iip=0}^{\ii} \norm{\LxiOp^{\ij}\radpsi{\iip}}_{W^{\ireg-\ell(\iip)}_{-2}(\Staut)}^2 .
\end{align}
Combining these results, one finds
\begin{align}
\sum_{\iip=0}^{\ii} \norm{\LxiOp^{\ij+1}\radpsi{\iip}}_{W^{\ireg-1-\ell(\iip)}_{-2}(\Staut)}^2
+\sum_{\iip=0}^{\ii} \norm{r\VOp\LxiOp^{\ij+1}\radpsi{\iip}}_{W^{\ireg-2-\ell(\iip)}_{-\delta}(\Staut)}^2
\lesssim{}& \sum_{\iip=0}^{\ii} \norm{\LxiOp^{\ij}\radpsi{\iip}}_{W^{\ireg-\ell(\iip)}_{-2}(\Staut)}^2 .
\label{eq:ImproEnerDecPsi01234:Passjtoj+1}
\end{align}

With these preliminaries proved, one can now consider the proof of the energy and Morawetz estimate \eqref{eq:ImproEnerDecPsi234}. The $j=0$ case is proved in lemma \ref{lem:EnerDecPsi234}. If inequality \eqref{eq:ImproEnerDecPsi234} is known to hold for $\ij$, then inequality \eqref{eq:ImproEnerDecPsi01234:Passjtoj+1} implies for $\ii\in\{2,3,4\}$
\begin{align}
\sum_{\iip=0}^{\ii} \norm{\LxiOp^{\ij+1}\radpsi{\iip}}_{W^{\ireg+2(\ii-5)-3\ij-1-\ell(\iip)}_{-2}(\Staut)}^2
&+\sum_{\iip=0}^{\ii} \norm{r\VOp\LxiOp^{\ij+1}\radpsi{\iip}}_{W^{\ireg+2(\ii-5)-3\ij-2-\ell(\iip)}_{-\delta}(\Staut)}^2 \nonumber\\
\lesssim{}& \timefunc^{-10+10\delta+(2-2\delta)\ii -(2-2\delta)\ij} \NEWinienergyminustwo{\ireg} .
\label{eq:ImproEnerDecPsi01234:BoundAtTopOfHierarchy}
\end{align}
The hierarchy \eqref{eq:ImproEnerDecPsi234:Hierarchy} and the bound at the top of the hierarchy \eqref{eq:ImproEnerDecPsi01234:BoundAtTopOfHierarchy} provide the hypotheses necessary to apply the pigeonhole lemma \ref{lem:hierarchyImpliesDecay}, an application of which implies
\begin{align}
\sum_{\iip=0}^{\ii} \norm{\LxiOp^{\ij+1}\radpsi{\iip}}_{W^{\ireg+2(i-5)-3j-3-\ell(\iip)}_{-2}(\Staut)}^2
&+\sum_{\iip=0}^{\ii} \norm{r\VOp\LxiOp^{\ij+1}\radpsi{\iip}}_{W^{\ireg+2(i-5)-3j-4-\ell(\iip)}_{\alpha-2}(\Staut)}^2 \nonumber\\
\lesssim{}& \timefunc^{\alpha-12+11\delta+(2-2\delta)\ii -(2-2\delta)\ij} \NEWinienergyminustwo{\ireg} .
\end{align}
Writing $-12+11\delta-(2-2\delta)\ij=-10+9\delta-(2-2\delta)(\ij+1)$, one obtains inequality \eqref{eq:ImproEnerDecPsi234} for $\ij+1$, so inequality \eqref{eq:ImproEnerDecPsi234} holds for all $\ij\in\Naturals$ by induction.

From the Sobolev inequality \eqref{eq:SobolevOnStautHolder} with $\gamma=\delta$ and the energy estimate \eqref{eq:ImproEnerDecPsi234} with $\alpha=1+\delta$ and $\alpha=1-\delta$, one finds
\begin{align}
\sum_{\iip=0}^\ii \absHighOrder{\LxiOp^\ij\radpsi{\iip}}{\ireg'-3}{\rescaledOps}^2
\lesssim{}&  \timefunc^{-(1-\delta)(9+2\ij-2\ii)} \NEWinienergyminustwo{\ireg}.
\end{align}
Alternatively, having already established the limit as $\timefunc\rightarrow\infty$ is zero, one can now apply the anisotropic spacetime Sobolev inequality \eqref{eq:anisotropicSpaceTimeSobolev}. Applying this, the trivial bound $-3<-3+\delta$, and the Morawetz estimate \eqref{eq:ImproEnerDecPsi234} with $\alpha=\delta$, one finds
\begin{align}
\sum_{\iip=0}^{\ii} \absHighOrder{\LxiOp^\ij r^{-1}\radpsi{\iip}}{\ireg'-7}{\rescaledOps}^2
\lesssim{}& \sum_{\iip=0}^{\ii}
  \norm{\LxiOp^\ij r^{-1}\radpsi{\iip}}_{W^{\ireg'-4}_{-1}(\Dtaut)}^{1/2}
  \norm{\LxiOp^{\ij+1} r^{-1}\radpsi{\iip}}_{W^{\ireg'-4}_{-1}(\Dtaut)}^{1/2}\nonumber\\
\lesssim{}& \sum_{\iip=0}^{\ii}
  \norm{\LxiOp^\ij \radpsi{\iip}}_{W^{\ireg'-4}_{-3}(\Dtaut)}^{1/2}
  \norm{\LxiOp^{\ij+1} \radpsi{\iip}}_{W^{\ireg'-4}_{-3}(\Dtaut)}^{1/2}\nonumber\\
\lesssim{}& \sum_{\iip=0}^{\ii}
  \norm{\LxiOp^\ij \radpsi{\iip}}_{W^{\ireg'-4}_{-3+\delta}(\Dtaut)}^{1/2}
  \norm{\LxiOp^{\ij+1} \radpsi{\iip}}_{W^{\ireg'-4}_{-3+\delta}(\Dtaut)}^{1/2}\nonumber\\
\lesssim{}& \timefunc^{-(1-\delta)(11+2\ij-2\ii)}
  \NEWinienergyminustwo{\ireg} .
\end{align}
Combining the two pointwise estimates and observing $v^{-1}\lesssim\min(r^{-1},\timefunc^{-1})$ gives the desired estimate \eqref{eq:pointwisePsi234}.
\end{proof}

\subsection{Improved decay estimates} 
\label{sec:ImprPsi4}

In this section, we conclude with a statement of the best decay estimates that we derive for $\psibase[-2][i]$ and its time derivatives. For $i\in\{2,3,4\}$ this simply restates the results from lemma \ref{lem:ImproEnerDecaypartialt}. 
In the exterior region (where $r\geq \timefunc$) and for $\ii\in\{0,1\}$, we improve the $t$ decay for $\LxiOp^\ij \radpsi{\ii}$, and,
in the interior region (where $r < \timefunc$), we improve the $r$ decay for the same quantities.
This is done by rewriting the first two lines of \eqref{eq:rescaledwavesys5eqs} as an elliptic equation of $\raddegphi{i}$ with source terms each of which either contains at least one $\LxiOp$ derivative (which has extra $\timefunc^{-1+\delta}$ decay from lemma \ref{lem:ImproEnerDecaypartialt}) or have an extra $r^{-1}$ prefactor. We exploit this extra $\timefunc^{-1+\delta}$ decay and $r^{-1}$ prefactor in the source terms, and an elliptic estimate yields improved pointwise-in-$\timefunc$ decay estimates for $\LxiOp^\ij \radpsi{\ii}$ $(\ii=0,1)$ and their spacetime norms in different regions.

The decay estimates for all $\LxiOp^\ij \radpsi{\ii}$ $(i=0,\ldots,4)$ are as follows.

\begin{thm}[Decay estimates with improvements for $\radpsi{\ii}$] 
\label{thm:ImproEstipsi01}
\standardHypothesisOnDelta. 
\standardMinusTwoHypothesisNotBEAM. 
\standardMinusTwoHypothesisBEAM. 
There is a regularity constant $K$ such that the following holds. 
If $\ireg,\ij\in\Naturals$ are such that $\ireg-3\ij-K\geq0$, 
then with $k''=k-3j-K$,
\begin{enumerate}
\item \label{point:psiiExt}

In the exterior region where $r\geq \timefunc$, we have for $\ii\in\{0,\ldots,4\}$ and $\delta \leq \alpha\leq 2-\delta$ the energy and Morawetz estimates for
$\timefunc\geq \timefunc_0$
\begin{align}
\hspace{6ex}&\hspace{-6ex}\sum_{\iip=0}^{\ii}\Big(
  \norm{\LxiOp^\ij \radpsi{\iip}}^2_{W^{\ireg''}_{-2}(\Stautext)}
  +\norm{r\LxiOp^\ij \VOp\radpsi{\iip}}^2_{W^{\ireg''-1}_{\alpha-2}(\Stautext)}
  +\norm{\LxiOp^\ij \radpsi{\iip}}^2_{W^{\ireg''-1}_{\alpha-3}(\Dtautext)}
\Big)\notag\\
\lesssim{}& \timefunc^{\alpha-10+9\delta+(2-2\delta)\ii-(2-2\delta)\ij}\NEWinienergyminustwo{\ireg},
\label{eq:ImproEnerDecPsiiExt}
\end{align}
and pointwise decay estimates for 
$\timefunc\geq \timefunc_0$
\begin{align}
\sum_{\iip=0}^{\ii}
 \absHighOrder{\LxiOp^\ij \radpsi{\iip}}{\ireg''}{\rescaledOps}
\lesssim{}& r v^{-1} \timefunc^{-(1-\delta)(\frac{9}{2}+\ij-\ii)+\delta} (\NEWinienergyminustwo{\ireg})^{1/2} .
\label{eq:pointwisePsiiExt}
\end{align}

\item
\label{point:psiiInt}
In the interior region where $ r<\timefunc$, for $\ii\in\{2,3,4\}$, and $\delta \leq \alpha\leq 2-\delta$, there are the energy and Morawetz estimates for 
$\timefunc\geq \timefunc_0$ 
\begin{align}
\hspace{6ex}&\hspace{-6ex}\sum_{\iip=0}^{\ii}\Big(
  \norm{\LxiOp^\ij \radpsi{\iip}}^2_{W^{\ireg''}_{-2}(\Stautimeint)}
  +\norm{r\LxiOp^\ij \VOp\radpsi{\iip}}^2_{W^{\ireg''-1}_{\alpha-2}(\Stautimeint)}
  +\norm{\LxiOp^\ij \radpsi{\iip}}^2_{W^{\ireg''-1}_{\alpha-3}(\Dtautimeint)}
\Big)\notag\\
\lesssim{}& \timefunc^{\alpha-10+9\delta+(2-2\delta)\ii-(2-2\delta)\ij}\NEWinienergyminustwo{\ireg} ,
\label{eq:ImproEnerDecPsi234copy}
\end{align}
and pointwise decay estimates for 
$\timefunc\geq \timefunc_0$
\begin{align}
\sum_{\iip=0}^{\ii}
  \absHighOrder{\LxiOp^\ij \radpsi{\iip}}{\ireg''}{\rescaledOps}
\lesssim{}& r v^{-1} \timefunc^{-(1-\delta)(\frac{9}{2}+\ij-\ii)+\delta} (\NEWinienergyminustwo{\ireg})^{1/2}.
\label{eq:pointwisePsi234copy}
\end{align}
Moreover, we have for 
$\timefunc\geq \timefunc_0$ that
\begin{subequations}\label{eq:psi01ImproInt}
\begin{align}
\absHighOrder{\LxiOp^\ij \psibase[-2]}{\ireg''}{\rescaledOps} \lesssim{}&  r^{-1+2\delta} v^{-1}\timefunc^{-(1-\delta)(\frac{5}{2}+\ij)+\delta}(\NEWinienergyminustwo{\ireg})^{1/2},
\label{eq:psi0PWint2}\\
\norm{\pt^j \psibase[-2]}^2_{W_{\alpha+1-3\delta}^{ \ireg''}(\Dtautimeint)}
\lesssim {}&\timefunc^{-(6+2j)(1-\delta)+\alpha}\NEWinienergyminustwo{\ireg},
\label{eq:Psi0InteriorMoraDecay}\\
\absHighOrder{\LxiOp^\ij \radpsi{1}}{\ireg''}{\rescaledOps}
 \lesssim{}&  r^{\delta} v^{-1}\timefunc^{-(1-\delta)(\frac{5}{2}+\ij)+\delta}(\NEWinienergyminustwo{\ireg})^{1/2},
\label{eq:psi1PWint1}\\
\norm{\pt^j  \radpsi{1}}^2_{W_{\alpha-1-\delta}^{\ireg''}(\Dtautimeint)}
\lesssim {}&\timefunc^{-(6+2j)(1-\delta)+\alpha}\NEWinienergyminustwo{\ireg}.
\label{eq:Psi1InteriorMoraDecay}
\end{align}
\end{subequations}
\end{enumerate}
\end{thm}
\begin{proof}
We prove point \eqref{point:psiiExt} first. Note that the estimates \eqref{eq:ImproEnerDecPsiiExt} and \eqref{eq:pointwisePsiiExt} have been proven for $i=2,3,4$ in lemma \ref{lem:ImproEnerDecaypartialt}. In the $i=0,1$ cases, these estimates improve the pointwise-in-$\timefunc$ decay compared to the pointwise estimate \eqref{eq:pointwisePsi234} and Morawetz estimate \eqref{eq:ImproEnerDecPsi234}, hence they hold trivially if $r\leq 10M$ since then $\timefunc \leq 10M$ is finite. Therefore, we shall only consider the exterior region intersected with $r\geq 10M$. Starting from the first two lines of \eqref{eq:rescaledwavesys5eqs} and making use of \eqref{eq:rescaledWaveOperator}, we get the following elliptic equations with source terms on the right 
\begin{subequations}\label{eq:Psi01EllipExt-explicit}
\begin{align}
\hspace{6ex}&\hspace{-6ex}2\hedt\hedtp\raddegphi{0}
 - 4 \raddegphi{0}\nonumber\\
={}&-2 a\LetaOp\LxiOp\raddegphi{0}
 + 2 M\LxiOp\raddegphi{1}
 + \frac{2 M a}{a^2 + r^2}\LetaOp\raddegphi{1}
 + \frac{6 a r}{a^2 + r^2}\LetaOp\raddegphi{0}
 - 2 M\VOp\raddegphi{1}\nonumber\\
& - \tfrac{1}{4} \bigl(4 a^2 + 9 (\kappa_{1}{} -  \bar{\kappa}_{1'}{})^2\bigr)\LxiOp\LxiOp\raddegphi{0}
 - 6 (\kappa_{1}{} -  \bar{\kappa}_{1'}{})\LxiOp\raddegphi{0}
 -  \frac{3 (a^4 + a^2 r^2 - 2 M r^3) \raddegphi{0}}{(a^2 + r^2)^2}\nonumber\\
& -  \frac{2 M (M a^2 + a^2 r - 3 M r^2 + r^3) \raddegphi{1}}{(a^2 + r^2)^2},\\
\hspace{6ex}&\hspace{-6ex}2\hedt\hedtp\raddegphi{1}
 - 6 \raddegphi{1}\nonumber\\
={}&-2 a\LetaOp\LxiOp\raddegphi{1}
 + 2 M\LxiOp\raddegphi{2}
 + \frac{2 M a}{a^2 + r^2}\LetaOp\raddegphi{2}
 + \frac{2 a r}{a^2 + r^2}\LetaOp\raddegphi{1}
 + \frac{6 a (a^2 -  r^2)}{M (a^2 + r^2)}\LetaOp\raddegphi{0}\nonumber\\
& - 2 M\VOp\raddegphi{2}
 - \tfrac{1}{4} \bigl(4 a^2 + 9 (\kappa_{1}{} -  \bar{\kappa}_{1'}{})^2\bigr)\LxiOp\LxiOp\raddegphi{1}
 - 6 (\kappa_{1}{} - \bar{\kappa}_{1'}{})\LxiOp\raddegphi{1}\nonumber\\
& -  \frac{6 r (- a^4 - 3 M a^2 r -  a^2 r^2 + M r^3) \raddegphi{0}}{M (a^2 + r^2)^2}
 -  \frac{(7 a^4 - 20 M a^2 r + 7 a^2 r^2 + 6 M r^3) \raddegphi{1}}{(a^2 + r^2)^2}.
\end{align}
\end{subequations}
It is then manifest that
\begin{subequations}\label{eq:Psi01EllipExt}
\begin{align}\label{eq:Psi0EllipExt}
\hspace{4ex}&\hspace{-4ex}(2 \hedt \hedtp-4) \raddegphi{0}\notag\\
={}&
\bigOAnalytic(r^{-2})M^2\LetaOp\raddegphi{1}+\bigOAnalytic(r^{-1}) Mr\VOp\raddegphi{1}+\bigOAnalytic(r^{-1})M \raddegphi{1} +\bigOAnalytic(r^{-1})M\LetaOp\raddegphi{0}\notag\\
 &
+\bigOAnalytic(r^{-1})M\raddegphi{0}
+\bigOAnalytic(1) M\LetaOp\LxiOp\raddegphi{0}+\bigOAnalytic(1) M^2\LxiOp\LxiOp\raddegphi{0}+\bigOAnalytic(1) M\LxiOp\raddegphi{0}\notag\\
 &+\bigOAnalytic(1)M\LxiOp\raddegphi{1},\\
 \label{eq:Psi1EllipExt}
\hspace{4ex}&\hspace{-4ex}(2 \hedt \hedtp-6) \raddegphi{1}\notag\\
={}&
\bigOAnalytic(r^{-2}) M^2\LetaOp\raddegphi{2}+\bigOAnalytic(r^{-1}) Mr\VOp\raddegphi{2}+\bigOAnalytic(r^{-1})M \LetaOp\raddegphi{1}+\bigOAnalytic(r^{-1})M\raddegphi{1}\notag\\
 &+\bigOAnalytic(1)M \LetaOp\LxiOp\raddegphi{1}+\bigOAnalytic(1) M^2\LxiOp\LxiOp\raddegphi{1}+\bigOAnalytic(1) M\LxiOp\raddegphi{1}+\bigOAnalytic(1)M\LxiOp\raddegphi{2}\notag\\
 &+\bigOAnalytic(1)\raddegphi{0}+\bigOAnalytic(1)\LetaOp\raddegphi{0},
\end{align}
\end{subequations}
and commuting with $r\VOp$ gives 
\begin{subequations}\label{eq:rVPsi01EllipExt}
\begin{align}\label{eq:rVPsi0EllipExt}
\hspace{4ex}&\hspace{-4ex}(2 \hedt \hedtp-4) r\VOp\raddegphi{0}\notag\\
={}&
\bigOAnalytic(r^{-2})M^2\LetaOp r\VOp\raddegphi{1}
+\bigOAnalytic(r^{-2})M^2\LetaOp \raddegphi{1}
+\bigOAnalytic(r^{-1}) Mr\VOp(r\VOp\raddegphi{1})
+\bigOAnalytic(r^{-1}) Mr\VOp\raddegphi{1}
\notag\\
&
+\bigOAnalytic(r^{-1})M r\VOp\raddegphi{1}+\bigOAnalytic(r^{-1})M \raddegphi{1} +\bigOAnalytic(r^{-1})M\LetaOp r\VOp\raddegphi{0}\notag\\
&
+\bigOAnalytic(r^{-1})M\LetaOp \raddegphi{0}
+\bigOAnalytic(r^{-1})Mr\VOp\raddegphi{0}
+\bigOAnalytic(r^{-1})M\raddegphi{0}\notag\\
 &+\bigOAnalytic(1) M\LetaOp\LxiOp r\VOp\raddegphi{0}
 +\bigOAnalytic(1) M^2\LxiOp\LxiOp r\VOp\raddegphi{0}
 +\bigOAnalytic(1) M\LxiOp r\VOp\raddegphi{0}
 +\bigOAnalytic(1)M \LxiOp r\VOp\raddegphi{1},\\
 \label{eq:rVPsi1EllipExt}
\hspace{4ex}&\hspace{-4ex}(2 \hedt \hedtp -6) r\VOp\raddegphi{1}\notag\\
={}&
\bigOAnalytic(r^{-1}) Mr\VOp (r\VOp\raddegphi{2})
+\bigOAnalytic(r^{-1}) Mr\VOp \raddegphi{2}
+\bigOAnalytic(r^{-2}) M^2\LetaOp r\VOp\raddegphi{2}
+\bigOAnalytic(r^{-2}) M^2\LetaOp\raddegphi{2}\notag\\
&
+\bigOAnalytic(r^{-1})M \LetaOp r\VOp\raddegphi{1}
+\bigOAnalytic(r^{-1})M \LetaOp \raddegphi{1}
+\bigOAnalytic(r^{-1})Mr\VOp\raddegphi{1}
+\bigOAnalytic(r^{-1})M\raddegphi{1}\notag\\
 &+\bigOAnalytic(1)M \LetaOp\LxiOp r\VOp\raddegphi{1}
 +\bigOAnalytic(1) M^2\LxiOp\LxiOp r\VOp\raddegphi{1}
 +\bigOAnalytic(1) M\LxiOp r\VOp\raddegphi{1}
 +\bigOAnalytic(1)M\LxiOp r\VOp\raddegphi{2}\notag\\
 &+\bigOAnalytic(1)r\VOp\raddegphi{0}
 +\bigOAnalytic(r^{-1})M\raddegphi{0}
 +\bigOAnalytic(1)\LetaOp r\VOp\raddegphi{0}
+\bigOAnalytic(r^{-1})M\LetaOp \raddegphi{0}.
\end{align}
\end{subequations}
The operators that appear on the left-hand side of \eqref{eq:Psi01EllipExt} and \eqref{eq:rVPsi01EllipExt} are $2\hedt\hedtp-4$ and $2\hedt\hedtp-6$ which are second-order, self-adjoint elliptic operators. By an argument similar to the derivation of \eqref{eq:SphereHardyPosSpin}, one finds that $\hedt\hedtp\,$ has spectrum bounded above by $-(|s|+s)/2$. Thus, $-(2\hedt\hedtp-4)$ and $-(2\hedt\hedtp-6)$ have spectrum with a positive lower bound. 
On the right-hand sides of  both \eqref{eq:Psi0EllipExt} and \eqref{eq:rVPsi0EllipExt}, the source terms involving $\LxiOp$ derivatives have better $t^{-1+\delta}$ pointwise decay, and when obtaining pointwise, energy, and Morawetz estimates for the terms on the right-hand side, $r$ inverse coefficients will give $\timefunc$ inverse decay since $r\geq \timefunc$ in the exterior region. 
The properties of $2\hedt\hedtp-4$ allows us to apply an elliptic estimate to \eqref{eq:Psi0EllipExt}, and this together with the pointwise estimate \eqref{eq:pointwisePsi234} yields
\begin{align}\label{eq:psi0PWext1}
\absHighOrder{\LxiOp^\ij \psibase[-2]}{\ireg-17-3j}{\rescaledOps} \lesssim r v^{-1} \timefunc^{-(1-\delta)(7/2+\ij)+\delta}(\NEWinienergyminustwo{\ireg})^{1/2}.
\end{align}
Here, the nonzero $j$ cases come from the fact that $\pt$ is a symmetry of the systems \eqref{eq:Psi01EllipExt} and \eqref{eq:rVPsi01EllipExt}.
We can also obtain an energy and Morawetz estimate for $\delta\leq \alpha\leq 2-\delta$ from the energy and Morawetz estimate \eqref{eq:ImproEnerDecPsi234}  that
\begin{align}\label{eq:psi0Moraext1}
\hspace{4ex}&\hspace{-4ex}\norm{\LxiOp^\ij \psibase[-2]}^2_{W^{\ireg-10-3j}_{-2}(\Stautext)}
  +\norm{r\LxiOp^\ij \VOp\psibase[-2]}^2_{W^{\ireg-11-3j}_{\alpha-2}(\Stautext)}
  + \norm{\LxiOp^\ij   \psibase[-2] }^2_{W_{\alpha-3}^{ \ireg-11-3j}(\Dtautext)}\notag\\
\lesssim {}&\timefunc^{-(1-\delta)(8+2j)+\alpha-\delta}\NEWinienergyminustwo{\ireg}.
\end{align}
Substituting these two estimates into \eqref{eq:Psi1EllipExt} and \eqref{eq:rVPsi1EllipExt}, and arguing as above, this yields improved exterior estimates for $\pt^j \radpsi{1}$ for $\delta\leq \alpha\leq 2-\delta$:
\begin{subequations}
\begin{align}\label{eq:psi1PWext1}
\hspace{6ex}&\hspace{-6ex}\absHighOrder{\LxiOp^\ij \radpsi{1}}{\ireg-18-3j}{\rescaledOps} \lesssim{} r v^{-1} \timefunc^{-(1-\delta)(7/2+\ij)+\delta}(\NEWinienergyminustwo{\ireg})^{1/2},\\
\label{eq:psi1Moraext1}
\hspace{6ex}&\hspace{-6ex}\norm{\LxiOp^\ij \radpsi{1}}^2_{W^{\ireg-11-3j}_{-2}(\Stautext)}
  +\norm{r\LxiOp^\ij \VOp\radpsi{1}}^2_{W^{\ireg-12-3j}_{\alpha-2}(\Stautext)}
  +\norm {\LxiOp^\ij  \radpsi{1} }^2_{W_{\alpha-3}^{ \ireg-12-3j}(\Dtautext)}\notag\\
\lesssim {}&\timefunc^{-(1-\delta)(8+2j)+\alpha-\delta}\NEWinienergyminustwo{\ireg}.
\end{align}
\end{subequations}
The above two estimates together prove the $i=1$ case of the estimates \eqref{eq:ImproEnerDecPsiiExt} and \eqref{eq:pointwisePsiiExt}. 
From the preliminary estimates \eqref{eq:psi0PWext1} and \eqref{eq:psi0Moraext1} for $\LxiOp^j\psibase[-2]$, from estimates \eqref{eq:psi1PWext1} and \eqref{eq:psi1Moraext1} for $\LxiOp^j \radpsi{1}$, from equations \eqref{eq:Psi0EllipExt} and \eqref{eq:rVPsi0EllipExt}, and from elliptic estimates, there are the following improved estimates for $\LxiOp^j\psibase[-2]$
\begin{subequations}
\begin{align}
\label{eq:psi0PWext2}
\hspace{6ex}&\hspace{-6ex}\absHighOrder{\LxiOp^\ij \psibase[-2]}{\ireg-21-3j}{\rescaledOps} \lesssim{} r v^{-1} \timefunc^{-(1-\delta)(9/2+\ij)+\delta}(\NEWinienergyminustwo{\ireg})^{1/2},\\
\label{eq:psi1Moraext2}
\hspace{6ex}&\hspace{-6ex}\norm{\LxiOp^\ij \psibase[-2]}^2_{W^{\ireg-14-3j}_{-2}(\Stautext)}
  +\norm{r\LxiOp^\ij \VOp\psibase[-2]}^2_{W^{\ireg-15-3j}_{\alpha-2}(\Stautext)}
  +\norm{\LxiOp^\ij  \psibase[-2] }^2_{W_{\alpha-3}^{ \ireg-15-3j}(\Dtautext)}\notag\\
\lesssim {}&\timefunc^{-(1-\delta)(10+2j)+\alpha-\delta}\NEWinienergyminustwo{\ireg},
\end{align}
\end{subequations}
which is the $i=0$ case of \eqref{eq:ImproEnerDecPsiiExt} and \eqref{eq:pointwisePsiiExt}.

Let us turn to point \eqref{point:psiiInt} now. The estimates \eqref{eq:ImproEnerDecPsi234copy}  and
\eqref{eq:pointwisePsi234copy}  are proved in lemma \ref{lem:ImproEnerDecaypartialt}, so we consider only the estimates \eqref{eq:psi01ImproInt}. 
We note that these estimates only improve the $r$ decay compared to the pointwise estimate \eqref{eq:pointwisePsi234} and Morawetz estimate \eqref{eq:ImproEnerDecPsi234}, hence in the following proof we will restrict to $r\geq 10M$ region where the left-hand sides of \eqref{eq:Psi01EllipExt} are both strongly elliptic operators acting on the field.

From the pointwise estimate \eqref{eq:pointwisePsi234} and Morawetz estimate \eqref{eq:ImproEnerDecPsi234}, an elliptic estimate applied to \eqref{eq:Psi0EllipExt} gives that 
\begin{subequations}
\begin{align}
\label{eq:psi0PWint1}
\absHighOrder{\LxiOp^\ij \psibase[-2]}{\ireg-17-3j}{\rescaledOps} \lesssim{}& r^{\delta} v^{-1}\timefunc^{-(1-\delta)(\frac{5}{2}+\ij)+\delta}(\NEWinienergyminustwo{\ireg})^{1/2},\\
\label{eq:psi0Moraint12}
\Vert\pt^j  \psibase[-2] \Vert^2_{W_{-1-\delta}^{ \ireg-11-3j}(\Dtautimeint)}
\lesssim {}&\timefunc^{-(1-\delta)(6+2j)}\NEWinienergyminustwo{\ireg}.
\end{align}
\end{subequations}
Turning to \eqref{eq:Psi1EllipExt}, we make use of these estimates of $\pt^j \psibase[-2]$, the pointwise estimate \eqref{eq:pointwisePsi234}, and Morawetz estimate \eqref{eq:ImproEnerDecPsi234}, and obtain from elliptic estimates that
\begin{subequations}
\begin{align}
\absHighOrder{\LxiOp^\ij \radpsi{1}}{\ireg-18-3j}{\rescaledOps} \lesssim{}&  r^{\delta} v^{-1}\timefunc^{-(1-\delta)(\frac{5}{2}+\ij)+\delta}(\NEWinienergyminustwo{\ireg})^{1/2},\\
\Vert\pt^j  \radpsi{1} \Vert^2_{W_{-1-\delta}^{ \ireg-12-3j}(\Dtautimeint)}
\lesssim {}&\timefunc^{-(1-\delta)(6+2j)}\NEWinienergyminustwo{\ireg}.
\label{eq:psi1Moraint12}
\end{align}
\end{subequations}
Notice that the first estimate is exactly the estimate \eqref{eq:psi1PWint1}.
From the estimate \eqref{eq:psi1Moraint12}, it follows that for any $l\in {\Naturals}$ and $0\leq \alpha< (6+2j)(1-\delta)$,
\begin{align}
\Vert\pt^j  \radpsi{1} \Vert^2_{W_{\alpha-1-\delta}^{\ireg-12-3j}(\Omega^{\interior}_{2^l\timefunc,2^{l+1}\timefunc})}
\lesssim {}&(2^l\timefunc)^{-(6+2j)(1-\delta)+\alpha}\NEWinienergyminustwo{\ireg}.
\end{align}
Summing over these estimates from $l=0$ to $\infty$, this proves \eqref{eq:Psi1InteriorMoraDecay}.

In the same manner, we obtain the preliminary estimate for $\psibase[-2]$ that, for $0\leq \alpha< (6+2j)(1-\delta)$,
\begin{align}
\Vert\pt^j  \psibase[-2] \Vert^2_{W_{\alpha-1-\delta}^{\ireg-11-3j}(\Dtautimeint)}
\lesssim {}&\timefunc^{-(6+2j)(1-\delta)+\alpha}\NEWinienergyminustwo{\ireg}.
\label{eq:Psi0InteriorMoraDecayv1}
\end{align}
Substituting the pointwise estimates \eqref{eq:psi0PWint1} and \eqref{eq:psi1PWint1} and Morawetz estimates \eqref{eq:Psi0InteriorMoraDecayv1} and \eqref{eq:Psi1InteriorMoraDecay} back in to \eqref{eq:Psi0EllipExt}, we conclude from elliptic estimates that estimates \eqref{eq:psi0PWint2} and \eqref{eq:Psi0InteriorMoraDecay} hold.
\end{proof}

\section{The spin-weight \texorpdfstring{$+2$}{+2} Teukolsky equation} 
\label{sec:spin+2TeukolskyEstimates}
In this section, we consider the field $\psibase[+2]$ that satisfies the Teukolsky equation \eqref{eq:TeukolskyRegular+2}. 
In section \ref{sec:estimatesForTheMetric}, the condition that we need is that $\psibase[2]$ satisfies a pointwise decay condition. In definition \ref{ass:BEAMspin+2}, we introduce two BEAM conditions and the necessary pointwise decay condition. The goal of this section is to show that the first BEAM condition implies the second and that either of the BEAM conditions imply the pointwise decay condition. The second BEAM condition is proved in \cite{2017arXiv170807385M}. See also \cite{2017arXiv171107944D} for a related result. We expect that the stronger first BEAM condition should hold.

\subsection{Basic assumptions}
\label{sec:spin+2TeukolskyAssumptions}
Let us first introduce two different basic energy and Morawetz (BEAM) conditions and one pointwise condition.
\begin{definition}[BEAM conditions and pointwise condition for  $\protect{\psibase[+2]}$ ]
\label{ass:BEAMspin+2}
\standardPlusTwoHypothesisNotBEAM 
For a spin-weighted scalar $\varphi$ and $\ireg\in\Integers^{+}$, let the energies $E^{\ireg}_{\Staut}(\varphi)$ and $E^{\ireg}_{\Stauini}(\varphi)$  be defined as in definition \ref{def:basichighorderenergy}, and the spacetime integral $\BEAMbulk^{\ireg}_{\timefunc_1, \timefunc_2}[\varphi]$ be defined as in definition \ref{def:Wknorm}. Two BEAM conditions and one pointwise condition are defined to be that
\begin{enumerate}
\item (First BEAM condition) \label{point:BEAMspin+2withoutDelta}
for all sufficiently large $\ireg\in \Naturals$ and any $\timefunc_2\geq\timefunc_1\geq\timefunc_0$, 
\begin{align}
&\sum_{i=0}^2
 \left( E^{\ireg}_{\Stau}(M^{4-i}(r^2\YOp)^i (r^{-4}\psibase[+2]))+\BEAMbulk^{\ireg}_{\timefunc_1, \timefunc_2}(M^{4-i}(r^2\YOp)^i (r^{-4}\psibase[+2])) \right)\nonumber\\
&\quad \lesssim{} \sum_{i=0}^2 E^{\ireg}_{\Staui}(M^{4-i}(r^2\YOp)^i (r^{-4}\psibase[+2])).
\label{eq:BEAMassumptionSpin+2withoutDelta}
\end{align}

\item (Second BEAM condition) \label{point:BEAMspin+2withDelta}
there is a $\delta_0\in (0,1/2)$ such that for all sufficiently large $\ireg\in \Naturals$  and any $\timefunc_2\geq\timefunc_1\geq\timefunc_0$, 
\begin{align}
&\sum_{i=0}^1\left(
  E^{\ireg}_{\Stau}(M^{i+\frac{\delta_0}{2}}r^{-\frac{\delta_0}{2}}\YOp^i \psibase[+2])
  +\BEAMbulk^{\ireg}_{\timefunc_1, \timefunc_2}(M^{i+\frac{\delta_0}{2}}r^{-\frac{\delta_0}{2}}\YOp^i \psibase[+2]) 
\right)\nonumber\\
&+E^{\ireg}_{\Stau}(M^2\YOp^2 \psibase[+2])
  +\BEAMbulk^{\ireg}_{\timefunc_1, \timefunc_2}(M^2\YOp^2 \psibase[+2])  \nonumber\\
&\quad \lesssim{} \sum_{i=0}^1 E^{\ireg}_{\Staui}(M^{i+\frac{\delta_0}{2}}r^{-\frac{\delta_0}{2}}\YOp^i\psibase[+2]) + E^{\ireg}_{\Staui}(M^2\YOp^2\psibase[+2]) .
\label{eq:BEAMassumptionSpin+2}
\end{align}

\item (Pointwise condition)\label{condition:spin+2goestozero}
for all sufficiently large $\ireg\in \Naturals$, 
\begin{align}
\lim_{\timefunc\rightarrow\pm \infty}\big(\absScri{
\psibase[+2] }{\ireg} \big{|}_{\Scri^+}\big)\rightarrow 0 .
\end{align}
\end{enumerate}
\end{definition}

The pointwise condition \eqref{condition:spin+2goestozero} in definition \ref{ass:BEAMspin+2} is one of the basic assumptions used in section \ref{sec:estimatesForTheMetric}, and either of the two BEAM conditions in the above definition together with the assumption that $\inienergyplustwo{\ireg}$ is bounded are shown in theorem \ref{thm:DecayEstimatesSpin+2} to imply this pointwise condition.

\begin{remark} 
Compared to the quantities introduced by Ma in \cite[Appendix A]{SiyuanMathesis}, which are denoted here by $\hat\phi^{i,\textrm{Ma}}_{+2}$,
we have $\psibase[+2]=\tfrac{1}{4} (a^2+r^2)^{5/2} \kappa_{1}{}^2 \bar{\kappa}_{1'}{}^{-2}\hat\phi^{0,\textrm{Ma}}_{+2}$ where the first factor $\tfrac{1}{4} (a^2+r^2)^{5/2}$ is to make the quantity nondegenerate at future null infinity, and the other factor $ \kappa_{1}{}^2 \bar{\kappa}_{1'}{}^{-2}$ corresponds to a spin rotation of the frame.
The quantities $\hat\phi_{+2}^{i,\textrm{Ma}}$ $(i=0,1,2)$ and the quantity $\psibase[+2]$ are related by
\begin{align}
r\hat\phi_{+2}^{i,\textrm{Ma}}={}\sum_{j=0}^i\bigOAnalytic(1)(M^{-1}r^2 \YOp )^j (M^4r^{-4}\psibase[+2]).
\label{eq:relationofSpin+2andMaNotation}
\end{align}
\end{remark} 

As a preliminary, the following relations between the two BEAM conditions are useful.
\begin{lemma}
\label{lem:BEAMwithoutDeltaimpliesBEAMwithDeltaspin+2}
Let $0<\delta_0<1/2$ be fixed. The BEAM condition \eqref{point:BEAMspin+2withoutDelta} in definition \ref{ass:BEAMspin+2} implies BEAM condition \eqref{point:BEAMspin+2withDelta} in definition \ref{ass:BEAMspin+2}.
\end{lemma}

\begin{proof}
For ease of presentation we will here use mass normalization as in definition~\ref{def:massnormalization}.
The lemma follows from adapting the proof of \cite[Proposition 3.1.2]{SiyuanMathesis} to our hyperboloidal foliation. By arguing in the same way as in the proof of \cite[Proposition 3.1.2]{SiyuanMathesis} except that the integration is over $\DtauRc$, and using the relation \eqref{eq:relationofSpin+2andMaNotation}, one finds that there exists a constant  $\rCutOff\geq 10M$ such that for any $\ireg\geq 1$,
\begin{align}
\hspace{6ex}&\hspace{-6ex}\sum_{i=0}^1\left(
  E^{\ireg}_{\StauR}(r^{4-2i-\frac{\delta_0}{2}}(r^2\YOp)^i (r^{-4}\psibase[+2]))
  +\int_{\DtauR} r^{-3} \sum_{|\mathbf{a}|\leq \ireg}\abs{\unrescaledOps^{\mathbf{a}}(r^{4-2i}(r^2\YOp)^i (r^{-4}\psibase[+2]))}^2 \diFourVol
\right)\nonumber\\
\lesssim{}&
  \sum_{i=0}^1\left(
  E^{\ireg}_{\StauiRc}(r^{4-2i-\frac{\delta_0}{2}}(r^2\YOp)^i (r^{-4}\psibase[+2]))
  + E^{\ireg}_{\StauRcR}(r^{4-2i-\frac{\delta_0}{2}}(r^2\YOp)^i (r^{-4}\psibase[+2]))\right)\nonumber\\
  &+\int_{\DtauRcR} r^{-1} \sum_{|\mathbf{a}|\leq \ireg}\abs{\unrescaledOps^{\mathbf{a}} (r^{4-2i}(r^2\YOp)^i (r^{-4}\psibase[+2]))}^2 \diFourVol .
  \label{eq:MoraLargerwithDeltaSpin+2}
\end{align}
The $\ireg >1$ case here follows from commuting the Killing symmetry $\LxiOp$ (which is timelike for $r\geq \rCutOff-M\geq 9M$) and elliptic estimates. Combining the BEAM condition \eqref{point:BEAMspin+2withoutDelta} with the above estimate \eqref{eq:MoraLargerwithDeltaSpin+2}, and from the following facts
\begin{subequations}
\begin{align}
r^{4-\frac{\delta_0}{2}} (r^{-4}\psibase[+2])={}&\bigOAnalytic(1)r^{-\frac{\delta_0}{2}}\psibase[+2],\\
r^{2-\frac{\delta_0}{2}}(r^2\YOp) (r^{-4}\psibase[+2])={}&\bigOAnalytic(1)r^{-\frac{\delta_0}{2}}\YOp\psibase[+2]+\bigOAnalytic(r^{-1})r^{-\frac{\delta_0}{2}}\psibase[+2],\\
(r^2\YOp)^2 (r^{-4}\psibase[+2])={}&\bigOAnalytic(1)\YOp^2\radpsizero{0
}+\bigOAnalytic(r^{-1+\frac{\delta_0}{2}})r^{-\frac{\delta_0}{2}}\YOp\psibase[+2]+\bigOAnalytic(r^{-2+\frac{\delta_0}{2}})r^{-\frac{\delta_0}{2}}\psibase[+2], 
\end{align}
\end{subequations}
the estimate \eqref{eq:BEAMassumptionSpin+2} is valid.
\end{proof}

\subsection{The estimates}
\label{sec:spin2TeukolskyEstimates}

This section uses the $r^p$ lemma \ref{lem:GeneralrpSpinWaveHighReg} to obtain decay estimates for $\psibase[+2]$. One can perform a rescaling to $\psibase[+2]$ as follows such that the governing equation of the new scalar can be put into the form of \eqref{eq:GeneralrpSpinWave:Assumption} with $\vartheta=0$, to which the $r^p$ lemma \ref{lem:GeneralrpSpinWaveHighReg} can be applied.

\begin{lemma}
Given a spin-weight $+2$ scalar $\psibase[+2]$ that satisfies equation \eqref{eq:TeukolskyRegular+2}, the quantity $\radpsizerores{0}$ defined by
\index{P3phihatSpinPlus2@$\radpsizerores{0}$}
\begin{align}
\radpsizerores{0}={}&\frac{(a^2 + r^2)^{2} \psibase[+2]}{\Delta^2}
\label{eq:DefSingularRadPsizero}
\end{align}
then satisfies
\begin{align}
\widehat\squareS_{+2}(\radpsizerores{0})={}&\frac{8 a r}{a^2 + r^2}\LetaOp\radpsizerores{0}
 -  \frac{8 (M a^2 + a^2 r - 3 M r^2 + r^3)}{\Delta}\VOp\radpsizerores{0}\nonumber\\
& + \frac{4 r (9 M a^2 + a^2 r - 7 M r^2 + r^3) \radpsizerores{0}}{(a^2 + r^2)^2}.
\label{eq:WaveEqofSingularRadPsizero}
\end{align}
\end{lemma}

Before proving weighted $r^p$ estimates for \eqref{eq:WaveEqofSingularRadPsizero}, we state some equivalent relations between the energy norms  of $\psibase[+2]$ and $\radpsizerores{0}$, which turn out to be useful in translating $r^p$ estimates of $\radpsizerores{0}$ to $r^p$ estimates of $\psibase[+2]$.

\begin{lemma}
\label{lem:radpsizeroAndradpsizeroresEquivalence}
\standardPlusTwoHypothesisNotBEAMNotSolu
Let $\radpsizerores{0}$ be as in equation \eqref{eq:DefSingularRadPsizero}. 
Let $k\in\Naturals$, $\beta\in\Reals$ and $\rCutOff\geq 10M$. There is the bound
\begin{align}
 \norm{\radpsizerores{0}}_{W^{k}_{\beta}(\StautR)}
\simC{}{}&  \norm{\psibase[+2]}_{W^{k}_{\beta}(\StautR)} .
\label{eq:radpsizeroAndradpsizeroresEquivalence}
\end{align}
Furthermore, for $\alpha\in[0,2]$ and $k\geq 1$,
\begin{align}
 \norm{r\VOp\radpsizerores{0}}_{W^{k-1}_{\alpha-2}(\StautR)}
  +\norm{\radpsizerores{0}}_{W^{k}_{-2}(\StautR)}
\simC{}{}&
  \norm{r\VOp\psibase[+2]}_{W^{k-1}_{\alpha-2}(\StautR)}
  +\norm{\psibase[+2]}_{W^{k}_{-2}(\StautR)}.
\label{eq:radpsizeroAndradpsizeroresEquivalenceWithV}
\end{align}
\end{lemma}

\begin{proof}
These estimates follow easily by arguing in the same way as in lemma \ref{lem:equiavelenceOfradphiAndraddegphi} and taking into account the relation \eqref{eq:DefSingularRadPsizero}.
\end{proof}

Now we are ready to apply the $r^p$ lemma \ref{lem:GeneralrpSpinWaveHighReg} to equation \eqref{eq:WaveEqofSingularRadPsizero} and to state the $\alpha$-weighted estimate, which is a combination of the $r^p$ estimate for $\radpsizerores{0}$ and the BEAM estimate \eqref{point:BEAMspin+2withDelta} in definition \ref{ass:BEAMspin+2}.

\begin{lemma}[$r^p$ estimate for $\protect{\psibase[+2]}$]
\label{lem:rplemmaspin+2}
\standardPlusTwoHypothesis  
Then, for all sufficiently large $\ireg\in\Naturals$, any $0<\delta\leq \delta_0$, $\alpha\in[\delta,2-\delta]$ and $\timefunc_2\geq\timefunc_1\geq\timefunc_0$,
\begin{align}
\hspace{6ex}&\hspace{-6ex}\norm{\psibase[+2]}_{W^{\ireg +1}_{-2}(\Stau)}^2+\norm{ r\VOp\psibase[+2]}_{W^{\ireg}_{\alpha-2}(\Stau)}^2
 +E^{\ireg +1}_{\Stau}(M^{1+\frac{\delta_0}{2}}r^{-\frac{\delta_0}{2}}\YOp \psibase[+2])\nonumber\\
&+E^{\ireg +1}_{\Stau}(M^2\YOp^2 \psibase[+2])
+\norm{\psibase[+2]}_{W^{\ireg}_{\alpha-3}(\Dtau)}^2\nonumber\\
 \lesssim {}&\norm{\psibase[+2]}_{W^{\ireg +1}_{-2}(\Staui)}^2+\norm{ r\VOp\psibase[+2]}_{W^{\ireg}_{\alpha-2}(\Staui)}^2
 +E^{\ireg +1}_{\Staui}(M^{1+\frac{\delta_0}{2}}r^{-\frac{\delta_0}{2}}\YOp \psibase[+2])
+E^{\ireg +1}_{\Staui}(M^2\YOp^2 \psibase[+2]).
\label{eq:rpSpin+2}
\end{align}
\end{lemma}

\begin{proof}
 From lemma \ref{lem:BEAMwithoutDeltaimpliesBEAMwithDeltaspin+2}, we only need to prove this lemma under the assumption that BEAM condition \eqref{point:BEAMspin+2withDelta} from definition \ref{ass:BEAMspin+2} is satisfied. In the following, we assume that such an assumption holds. 

By putting equation \eqref{eq:WaveEqofSingularRadPsizero} into the form of \eqref{eq:GeneralrpSpinWave:Assumption}, we see that $\vartheta =0$ and the assumptions in lemma \ref{lem:GeneralrpSpinWaveHighReg} are satisfied with
\begin{align}
b_{\VOp, -1}={}&8>0, &b_{\phi}={}&M\bigOAnalytic(r^{-1}), & b_{0,0}+ 2+2 ={}&0, 
\end{align}
and the spin weight is $+2$.
Thus, we apply the $r^p$ lemma \ref{lem:GeneralrpSpinWaveHighReg} and obtain that for any $\ireg \in \Naturals$, $\timefunc_0\leq \timefunc_1 \leq\timefunc_2$, $0<\delta\leq \delta_0$ and $\alpha\in[\delta,2-\delta]$, there are constants $\rCutOff=\rCutOff(\ireg)\geq 10M$ and $\CInrpWave=\CInrpWave(\ireg) $ such that
\begin{align}
&\norm{ r\VOp\radpsizerores{0}}_{W^{\ireg}_{\alpha-2}(\StauR)}^2
+\norm{\radpsizerores{0}}_{W^{k+1}_{-2}(\StauR)}^2\nonumber\\
&+\norm{\hedtp\radpsizerores{0}}_{W^{\ireg}_{\alpha-3}(\DtauR)}^2
+\norm{\radpsizerores{0}}_{W^{\ireg}_{\alpha-3}(\DtauR)}^2\nonumber\\
&{}\leq \CInrpWave \bigg(
\norm{ r\VOp\radpsizerores{0}}_{W^{\ireg}_{\alpha-2}(\StauiR)}^2
+\norm{\radpsizerores{0}}_{W^{\ireg+1}_{-2}(\StauiR)}^2\nonumber\\
&\qquad\quad +\norm{\radpsizerores{0}}_{W^{\ireg+1}_{0}(\DtauRcR)}^2
  +\sum_{\timefunc\in\{\timefunc_1,\timefunc_2\}} \norm{\radpsizerores{0}}_{W^{1}_{\alpha}(\StautRcR)}^2 \bigg).
\label{eq:rpDegPsizeroSpin+2}
\end{align}
This is an $r^p$ estimate for $\radpsizerores{0}$.
From lemma \ref{lem:radpsizeroAndradpsizeroresEquivalence}, $\radpsizerores{0}$ can be replaced by $\psibase[+2]$ in this estimate. By adding this $r^p$ estimate of $\psibase[+2]$ to the assumed BEAM estimate \eqref{eq:BEAMassumptionSpin+2}, the estimate \eqref{eq:rpSpin+2} follows. 
\end{proof}

\begin{lemma}
\label{lem:EarlyRegionSpinPlus2}
Under the same assumptions of lemma \ref{lem:rplemmaspin+2},
the estimate \eqref{eq:rpSpin+2} holds as well if we replace the right-hand side by $\inienergyplustwoForDescent{\ireg+3}$ as in definition \ref{def:generalSpinorInitialDataNorm}. 
\end{lemma}

\begin{proof}
For ease of presentation we will here use mass normalization as in definition~\ref{def:massnormalization}.
To prove this result, we just need to show the following estimate which bounds the norms on $\Stauit$ by those on $\Stauini$:
\begin{align}
&\norm{\psibase[+2]}_{W^{\ireg +1}_{-2}(\Stauit)}^2+\norm{ r\VOp\psibase[+2]}_{W^{\ireg}_{\alpha-2}(\Stauit)}^2
\nonumber\\
& +E^{\ireg +1}_{\Stauit}(r^{-\frac{\delta_0}{2}}\YOp \psibase[+2])
+E^{\ireg +1}_{\Stauit}(\YOp^2 \psibase[+2]) \lesssim \inienergyplustwoForDescent{\ireg+3}.
\label{eq:plus2EarlyRegion}
\end{align}

Applying lemma \ref{lem:GeneralrpSpinWaveFar} to the spin-weighted wave equation \eqref{eq:WaveEqofSingularRadPsizero} in the early region, and from the relation between $\radpsizerores{0}$ and $\psibase[+2]$ norms in lemma \ref{lem:radpsizeroAndradpsizeroresEquivalence}, it follows that for $\alpha\in [\delta,2-\delta]$, there is a constant $\rCutOff=\rCutOff(\ireg)\geq 10M$ such that 
\begin{align}
\hspace{4ex}&\hspace{-4ex}
  \norm{r\VOp\psibase[+2]}_{W^{\ireg}_{\alpha-2}(\StauitR)}^2
  +\norm{\psibase[+2]}_{W^{\ireg+1}_{-2}(\StauitR)}^2
  +\norm{\psibase[+2]}_{W^{\ireg+1}_{\alpha-3}(\DtauitearlyR)}^2
\nonumber\\
\lesssim{}&
\norm{\psibase[+2]}_{H^{\ireg+1}_{\alpha-1}(\Stauini)}
+ \norm{\psibase[+2]}_{W^{\ireg+1}_{0}(\DtauitearlyRcR)}^2
+  \norm{\psibase[+2]}_{W^{\ireg+1}_{-\delta}(\StauitRcR)}^2.
\label{eq:rpInTheEarlyRegionplus2}
\end{align}
Since $\rCutOff$ is bounded, $\norm{\psibase[+2]}_{W^{\ireg+1}_{0}(\DtauitearlyRcR)}^2$
and   $\norm{\psibase[+2]}_{W^{\ireg+1}_{-\delta}(\StauitRcR)}^2$ are both bounded by a multiple of an initial norm $\inienergyplustwoForDescent{\ireg+1}$, by standard exponential growth estimates for wave-like equations. For the same reason, the sum of 
 $\norm{r\VOp\psibase[+2]}_{W^{\ireg}_{\alpha-2}(\Stauit^{r_+,\rCutOff})}^2$
 and $\norm{\psibase[+2]}_{W^{\ireg+1}_{-2}(\Stauit^{r_+,\rCutOff})}^2$ is bounded by $\inienergyplustwoForDescent{\ireg+1}$ as well. For the first term on the right of \eqref{eq:rpInTheEarlyRegionplus2}, since $\alpha\leq 2-\delta$, it holds that
\begin{align}
\int_{\Stauini} \sum_{|\mathbf{a}|\leq \ireg+1}r^{\alpha+2|\mathbf{a}|-2}\abs{\unrescaledOps^{\mathbf{a}}\psibase[+2]}^2 \diThreeVol 
\leq {}& \int_{\Stauini} \sum_{|\mathbf{a}|\leq \ireg+1}r^{-\delta+2|\mathbf{a}|}\abs{\unrescaledOps^{\mathbf{a}}\psibase[+2]}^2 \diThreeVol \nonumber\\
\lesssim {}& \inienergyplustwoForDescent{\ireg+1}.
\end{align} 
Thus, for any $\alpha\in [\delta,2-\delta]$,
\begin{align}
\norm{r\VOp\psibase[+2]}_{W^{\ireg}_{\alpha-2}
(\Stauit)}^2
+\norm{\psibase[+2]}_{W^{\ireg+1}_{-2}(\Stauit)}^2
\lesssim{}&
\inienergyplustwoForDescent{\ireg+1}.
\label{eq:rpInTheEarlyRegionplus2v2}
\end{align}
In addition, since $M\YOp$ belongs to the operator set $\rescaledOps$,
\begin{align}
E^{\ireg +1}_{\Stauit}(r^{-\frac{\delta_0}{2}}\YOp \psibase[+2])
+E^{\ireg +1}_{\Stauit}(\YOp^2 \psibase[+2]) 
\lesssim {}& \norm{r\VOp\psibase[+2]}_{W^{\ireg+2}_{-\delta}
(\Stauit)}^2
+\norm{\psibase[+2]}_{W^{\ireg+3}_{-2}(\Stauit)}^2\nonumber\\
\lesssim {}&\inienergyplustwoForDescent{\ireg+3},
\end{align}
where the second step follows from \eqref{eq:rpInTheEarlyRegionplus2v2}.
The above two estimates together imply the inequality \eqref{eq:plus2EarlyRegion}, which then completes the proof.
\end{proof}

\begin{thm}[Decay estimates for $\LxiOp^j\protect{\psibase[+2]}$]
\label{thm:DecayEstimatesSpin+2}
\standardPlusTwoHypothesis 
 Assume furthermore that $\inienergyplustwoForDescent{\ireg}$ is finite for all $\ireg \in \Naturals$. 
Under these conditions:  
\begin{enumerate}
\item
 the pointwise condition \eqref{condition:spin+2goestozero} in definition \ref{ass:BEAMspin+2} holds; 
\item
  furthermore, there is a regularity constant $K$ such that for all $\ij\in\Naturals$, sufficiently large $\ireg-K-3\ij$, $0<\delta\leq \delta_0$, $\delta \leq \alpha\leq 2-\delta$, and $\timefunc\geq \timefunc_0$, there are the energy and Morawetz estimates  
\begin{align}
\hspace{6ex}&\hspace{-6ex}\norm{\LxiOp^\ij\psibase[+2]}_{W^{\ireg -K-7\ij}_{-2}(\Staut)}^2+\norm{ r\VOp\LxiOp^\ij\psibase[+2]}_{W^{\ireg-K-1-7\ij}_{\alpha-2}(\Staut)}^2
+\norm{\LxiOp^\ij\psibase[+2]}_{W^{\ireg-K-1-7\ij}_{\alpha}(\Dtaut)}^2 \notag\\
\lesssim {}&\timefunc^{\alpha -2+\delta-(2-2\delta)\ij}\inienergyplustwoForDescent{\ireg} ,
\label{eq:EnergyDecaySpin+2Explicit}
\end{align}
and pointwise decay estimates
\begin{align}
  \absHighOrder{\LxiOp^\ij \psibase[+2]}{\ireg-K-7\ij}{\rescaledOps}
\lesssim{}& r v^{-1} \timefunc^{-(1-\delta)(\frac{1}{2}+\ij)+\delta} (\inienergyplustwoForDescent{\ireg})^{1/2}.
\label{eq:pointwiseSpin+2}
\end{align}
\end{enumerate}
\end{thm}

\begin{proof}
For ease of presentation we will here use mass normalization as in definition~\ref{def:massnormalization}. 

First, consider the limits along $\Scri^+$ as $\timefunc\rightarrow\pm\infty$. Let $r(\timefunc)$ denote the value of $r$ corresponding to the intersection of $\Stauini$ and $\Staut$. For $\rCutOff$ fixed and $\timefunc$ sufficiently negative, we have that $r(\timefunc) > \rCutOff$ and $r(\timefunc) \sim -\timefunc$. Recall that the proof of the $r^p$ lemma \ref{lem:GeneralrpSpinWaveFar} is based on an application of Stokes' theorem, so we may replace $\Stauini$ by $\Stauini\cap\{r>r(\timefunc)\}$. The region under consideration is $r>\rCutOff$, so we may drop all the terms supported on $r\in[\rCutOff-M,\rCutOff]$, which gives
\begin{align}
\norm{r\VOp\psibase[+2]}_{W^{\ireg}_{-1}(\Staut)}^2
+\norm{\psibase[+2]}_{W^{\ireg+1}_{-2}(\Staut)}^2
\lesssim{}&\norm{\psibase[+2]}_{H^{\ireg+1}_{0}(\Stauini\cap\{r>r(\timefunc)\})}^2 . 
\end{align}
From adapting the proof of the Sobolev lemma \ref{lem:SobolevOnStaut} on $\Staut$, in particular from estimate \eqref{eq:SobolevOnStaut:intermediatestep}, one finds
\begin{align}
\lim_{r\rightarrow\infty}\int_{\Sphere} \abs{\psibase[+2](\timefunc,r,\omega)}_{\ireg}^2 \diTwoVol
\lesssim{}& \left(\norm{r\VOp\psibase[+2]}_{W^{\ireg}_{-1}(\Stau)}^2
+\norm{\psibase[+2]}_{W^{\ireg+1}_{-2}(\Stau)}^2\right) 
+\int_{\Sphere} \abs{\psibase[+2](\timefunc,r(\timefunc),\omega)}_{\ireg}^2 \diTwoVol .
\end{align}
Adapting the bound on $\psibase[+2]$ on $\Stauini$ in lemma \ref{lem:Ik+1alphaenergydominatesPkalphapointwise} and applying the previous estimate on the energy on $\Staut$, one finds
\begin{align}
\absScri{\psibase[+2] }{\ireg}^2 \big{|}_{\Scri^+} 
={}&\lim_{r\rightarrow\infty}\int_{\Sphere} \abs{\psibase[+2](\timefunc,r,\omega)}_{\ireg}^2 \diTwoVol
\lesssim \norm{\psibase[+2]}_{H^{\ireg+1}_{0}(\Stauini\cap\{r>r(\timefunc)\})}^2 ,
\end{align}
which goes to zero as $\timefunc\rightarrow-\infty$. (In fact, this argument gives a rate, but we do not need to calculate the rate for  the pointwise condition \eqref{condition:spin+2goestozero}.) 
As $\timefunc\rightarrow\infty$, the pointwise decay estimates \eqref{eq:pointwiseSpin+2} implies that $\lim_{\timefunc\rightarrow  \infty}\big(\absScri{\psibase[+2] }{\ireg} \big{|}_{\Scri^+}\big)\rightarrow 0$ holds for any $\ireg\in\Naturals$, and hence the first claim holds. 

Based on the above discussion and 
from lemma \ref{lem:rplemmaspin+2}, to prove this theorem, we only need to show the second claim under the assumption that the conclusions of lemma \ref{lem:rplemmaspin+2} are valid. For a general spin-weighted scalar $\varphi$, define
\begin{subequations}
\label{eq:EnergyandBulkinrppsizero}
\begin{align}
\EnergyinrpEstimatepsizero(\varphi,\ireg,\alpha,\timefunc)
={}&\begin{cases}
  \norm{\varphi}_{W^{\ireg -2}_{-2}(\Staut)}^2 +\norm{ r\VOp\varphi}_{W^{\ireg-3}_{\alpha-2}(\Staut)}^2
+E^{\ireg -2}_{\Staut}(r^{-\frac{\delta_0}{2}}\YOp \varphi)
+E^{\ireg -2}_{\Staut}(\YOp^2 \varphi)&\text{if $\alpha\in[\delta,2-\delta]$}\\
0&\text{if $\alpha<\delta$} 
\end{cases}\\
\BulkinrpEstimatepsizero(\varphi,\ireg,\alpha,\timefunc)
={}&\norm{\varphi}_{W^{\ireg}_{\alpha-2}(\Staut)}^2.
\end{align}
\end{subequations}
To prove the energy and Morawetz estimate \eqref{eq:EnergyDecaySpin+2Explicit}, we shall prove
\begin{align}
\hspace{4ex}&\hspace{-4ex}\EnergyinrpEstimatepsizero(\LxiOp^{\ij}\psibase[+2],\ireg-6-7\ij,\alpha,\timefunc)+\int_{\timefunc}^{\infty} \BulkinrpEstimatepsizero(\LxiOp^{\ij}\psibase[+2],\ireg-9-7\ij,\alpha-1,\timefunc')\di \timefunc' \nonumber\\
\lesssim {}&\timefunc^{\alpha -2+\delta-(2-2\delta)\ij}\inienergyplustwoForDescent{\ireg} .
\label{eq:EnergyDecaySpin+2}
\end{align}
Estimate \eqref{eq:rpSpin+2} in lemma \ref{lem:rplemmaspin+2} can be stated as, for $\alpha\in[\delta,2-\delta]$, 
\begin{align}
\EnergyinrpEstimatepsizero(\psibase[+2],\ireg+3,\alpha,\timefunc_2) +\int_{\timefunc_1}^{\timefunc_2} \BulkinrpEstimatepsizero(\psibase[+2],\ireg,\alpha-1,\timefunc)\di \timefunc \lesssim \EnergyinrpEstimatepsizero(\psibase[+2],\ireg+3,\alpha,\timefunc_1),
\label{eq:rpSpin+2shortform}
\end{align}
and 
note from \eqref{eq:EnergyandBulkinrppsizero} that
\begin{align}
\BulkinrpEstimatepsizero(\psibase[+2],\ireg,\alpha-1,\timefunc) \gtrsim \EnergyinrpEstimatepsizero(\psibase[+2],\ireg,\alpha-1,\timefunc),
\end{align}
hence  for any $\otherinireg\leq \ireg \in \Naturals$, $\delta \leq \alpha \leq 2-\delta$, and $\timefunc_0\leq \timefunc_1\leq \timefunc_2$,
\begin{align}
\EnergyinrpEstimatepsizero(\psibase[+2],\ireg+3,\alpha,\timefunc_2) +\int_{\timefunc_1}^{\timefunc_2} \EnergyinrpEstimatepsizero(\psibase[+2],\ireg,\alpha-1,\timefunc)\di \timefunc
\lesssim \EnergyinrpEstimatepsizero(\psibase[+2],\ireg+3,\alpha,\timefunc_1).
\label{eq:HierarchyofrpSpin+2}
\end{align}
This can be put into the form of \eqref{eq:Rev:HierarchyToDecayReal:EvolutionHypothesis} by taking $D=\gamma=0$ and performing the following replacement
\begin{align}
\delta \mapsto {}& \alpha_1, & 2-\delta \mapsto {}&\alpha_2,&  \EnergyinrpEstimatepsizero(\psibase[+2],\ireg+3,\alpha,\timefunc) \mapsto {} &F(\lfloor \tfrac{\ireg+3}{3}\rfloor,\alpha,\PigeonTime) .
\end{align}
An application of lemma \ref{lem:hierarchyImpliesDecay} then yields for $\alpha\in[\delta,2-\delta]$, 
\begin{align}
\EnergyinrpEstimatepsizero(\psibase[+2],\ireg-6,\alpha,\timefunc)
+\int_{\timefunc}^{\infty} \BulkinrpEstimatepsizero(\psibase[+2],\ireg-9,\alpha-1,\timefunc')\di \timefunc'
\lesssim {}&\timefunc^{\alpha -2+\delta}\EnergyinrpEstimatepsizero(\psibase[+2],\ireg,2-\delta,\timefunc_0).
\label{eq:EnergyDecaySpin+2:j=0}
\end{align}
From lemma \ref{lem:EarlyRegionSpinPlus2} (or estimate \eqref{eq:plus2EarlyRegion}), it holds that
\begin{align}
\EnergyinrpEstimatepsizero(\psibase[+2],\ireg,2-\delta,\timefunc_0) \lesssim {}&\inienergyplustwoForDescent{\ireg} ,
\end{align}
hence this proves the $j=0$ case of \eqref{eq:EnergyDecaySpin+2}.

We prove the general $\ij$ case of  \eqref{eq:EnergyDecaySpin+2} by induction. Assume that estimate \eqref{eq:EnergyDecaySpin+2} holds for $\ij=\ij'$, so that
\begin{align}
\EnergyinrpEstimatepsizero(\LxiOp^{\ij'}\psibase[+2],\ireg-6-7\ij',\delta,\timefunc)
\lesssim {}&\timefunc^{ -2+2\delta-(2-2\delta)\ij'}\inienergyplustwoForDescent{\ireg} .
\label{eq:EnergyDecaySpin+2jcase}
\end{align}
Since $\LxiOp$ is a symmetry of \eqref{eq:TeukolskyRegular+2}, it holds that for any $\ij \in \Naturals$, $\ireg$ sufficiently large, $\delta$ sufficiently small, $\alpha\in[\delta,2-\delta]$, and $\timefunc_2\geq\timefunc_1\geq\timefunc_0$,
\begin{align}
\EnergyinrpEstimatepsizero(\LxiOp^{\ij}\psibase[+2],\ireg+3,\alpha,\timefunc_2) 
+\int_{\timefunc_1}^{\timefunc_2} \BulkinrpEstimatepsizero(\LxiOp^{\ij}\psibase[+2],\ireg,\alpha-1,\timefunc)\di \timefunc 
\lesssim \EnergyinrpEstimatepsizero(\LxiOp^{\ij}\psibase[+2],\ireg+3,\alpha,\timefunc_1).
\label{eq:rpLxiSpin+2}
\end{align}
One can argue similarly to the proof of lemma \ref{lem:ImproEnerDecaypartialt} to obtain better decay estimates for $\LxiOp^{\ij}\psibase[+2]$ as follows. Rescaling equation \eqref{eq:rescaledWaveOperator} for $\widehat\squareS_{+2}$, we can isolate the term $r^2\VOp\LxiOp\psibase[+2]$ from  \eqref{eq:TeukolskyRegular+2} and write $r^2\VOp\LxiOp\psibase[+2]$ as a weighted sum of $(r\VOp)^2\psibase[+2]$, $r\VOp\psibase[+2]$, $r^{-1}\LetaOp (r\VOp\psibase[+2])$, $r^{-1}\LetaOp\psibase[+2]$, $\TMESOp_s\psibase[+2]$, $\LxiOp\psibase[+2]$, and $\psibase[+2]$ all with $\bigOAnalytic(1)$ coefficients.
Therefore,
\begin{align}
\norm{r\VOp\LxiOp^{\ij'+1}\psibase[+2]}_{W^{\ireg -2}_{-\delta}(\Staut)}^2
\lesssim{}& \norm{r^2\VOp\LxiOp^{\ij'+1}\psibase[+2]}_{W^{\ireg-2}_{-\delta-2}(\Staut)}^2
\lesssim \norm{\LxiOp^{\ij'}\psibase[+2]}_{W^{\ireg}_{-2}(\Staut)}^2,
\end{align}
which furthermore implies
\begin{align}
\EnergyinrpEstimatepsizero(\LxiOp^{\ij'+1}\psibase[+2],\ireg-7-7\ij',2-\delta,\timefunc) \lesssim \EnergyinrpEstimatepsizero(\LxiOp^{\ij'}\psibase[+2],\ireg-6-7\ij',\delta,\timefunc)
\lesssim \timefunc^{ -2+2\delta-(2-2\delta)\ij'}\inienergyplustwoForDescent{\ireg} .
\end{align}
A repeated application of lemma \ref{lem:hierarchyImpliesDecay} as above  to
\eqref{eq:rpLxiSpin+2} but with $\ij\mapsto \ij'+1$ and $\ireg\mapsto \ireg-10-7\ij'$ then yields
\begin{align}
\hspace{6ex}&\hspace{-6ex}\EnergyinrpEstimatepsizero(\LxiOp^{\ij'+1}\psibase[+2],\ireg-7-7\ij'-6,\alpha,\timefunc)
+\int_{\timefunc}^{\infty} \BulkinrpEstimatepsizero(\LxiOp^{\ij'+1}\psibase[+2],\ireg-7-7\ij'-9,\alpha-1,\timefunc')\di \timefunc'\nonumber\\
\lesssim {}&\timefunc^{\alpha -2+\delta-2+2\delta-(2-2\delta)\ij'}\inienergyplustwoForDescent{\ireg}.
\end{align}
This proves the $\ij=\ij'+1$ case of \eqref{eq:EnergyDecaySpin+2}, which completes the induction and justifies the estimate \eqref{eq:EnergyDecaySpin+2} for general $\ij \in \Naturals$ cases and hence the estimate \eqref{eq:EnergyDecaySpin+2Explicit}.

As to the pointwise decay estimates, the proof is the same as the one for lemma \ref{lem:ImproEnerDecaypartialt}.
From the Sobolev inequality \eqref{eq:SobolevOnStautHolder} with $\gamma=\delta$ and the energy estimate \eqref{eq:EnergyDecaySpin+2} with $\alpha=1+\delta$ and $\alpha=1-\delta$, one finds
\begin{align}
\absHighOrder{\LxiOp^\ij\psibase[+2]}{\ireg-11-7\ij}{\rescaledOps}^2
\lesssim{}& \left(\EnergyinrpEstimatepsizero(\LxiOp^{\ij}\psibase[+2],\ireg-6-7\ij,1+\delta,\timefunc)\EnergyinrpEstimatepsizero(\LxiOp^{\ij}\psibase[+2],\ireg-6-7\ij,1-\delta,\timefunc)\right)^{\frac12}\nonumber\\
\lesssim {}&\timefunc^{-(1-\delta)(1+2\ij)} \inienergyplustwoForDescent{\ireg}.
\end{align}

Alternatively, having already established the limit as $\timefunc\rightarrow\infty$ is zero, one can now apply the anisotropic spacetime Sobolev inequality \eqref{eq:anisotropicSpaceTimeSobolev} and the Morawetz estimate \eqref{eq:EnergyDecaySpin+2} with $\alpha=\delta$ to obtain
\begin{align}
\absHighOrder{\LxiOp^\ij r^{-1}\psibase[+2]}{\ireg-19-7\ij}{\rescaledOps}^2
\lesssim{}&
  \norm{\LxiOp^\ij r^{-1}\psibase[+2]}_{W^{\ireg-16-7\ij}_{-1}(\Dtaut)}^{1/2}
  \norm{\LxiOp^{\ij+1} r^{-1}\psibase[+2]}_{W^{\ireg-16-7\ij}_{-1}(\Dtaut)}^{1/2}\nonumber\\
\lesssim{}&
  \norm{\LxiOp^\ij \psibase[+2]}_{W^{\ireg-16-7\ij}_{-3+\delta}(\Dtaut)}^{1/2}
  \norm{\LxiOp^{\ij+1} \psibase[+2]}_{W^{\ireg-9-7(\ij+1)}_{-3+\delta}(\Dtaut)}^{1/2}\nonumber\\
\lesssim{}& \timefunc^{-(1-\delta)(3+2\ij)}
  \inienergyplustwoForDescent{\ireg} .
\end{align}
Combining the two pointwise estimates and observing $v^{-1}\lesssim\min(r^{-1},\timefunc^{-1})$ give the desired pointwise decay estimate \eqref{eq:pointwiseSpin+2}.
\end{proof}

\section{The metric and core connection coefficients}
\label{sec:estimatesForTheMetric}
We shall now use the results presented in Sections \ref{sec:spin-2TeukolskyEstimates} and \ref{sec:spin+2TeukolskyEstimates} to prove pointwise, energy, and Morawetz estimates for linearized gravity from the transport form of the equations of linearized gravity in ORG gauge, derived in section \ref{sec:transporteq}. We shall work in terms of the compactified hyperboloidal coordinate system  $(\timefunc,R,\theta,\phi)$ where $\timefunc$ is the hyperboloidal time introduced in section \ref{sec:foliation}, $R = 1/r$, and $\theta,\phi$ are the angular coordinates in the ingoing Eddington-Finkelstein coordinate system. We shall sometimes use the notation $\omega = (\theta,\phi)$. In terms of this coordinate system, future null infinity $\Scri^+$ is located at $R=0$. For our considerations here, we may without loss of generality consider compactly supported initial data, in which case the solution of the Teukolsky equation is smooth at $\Scri^+$ in the compactified hyperboloidal coordinate system, cf. section \ref{sec:ConformalRegularity}. 

\begin{definition}
\label{def:linearizedEinsteinFields}
A set of linearized Einstein fields is defined to consist of the following: 
\begin{enumerate}
\item a linearized metric $\delta g_{ab}$,  
\item linearized metric components $G_{0i'}$ from section \ref{sec:connectcomp}, 
\item linearized connection and connection coefficients from section \ref{sec:connectcomp}, 
\item linearized curvature components from \eqref{eq:extreme-lin},  
\item rescaled linearized curvature components $\psibase[-2]$ and $\psibase[+2]$ from definition \ref{def:psibase}, and
\item the core quantities $\hat{\sigma}'$, $\widehat{G}_2$, $\hat{\tau}'$, $\widehat{G}_1$, $\hat{\beta}'$, and $\widehat{G}_0$ from definition \ref{def:BoostWeightZeroQuantities}. 
\end{enumerate}
\end{definition}

\begin{definition}
\label{def:BEAMLinearisedEinsteinSolution}
An outgoing BEAM solution of the linearized Einstein equation is defined to be be a set of linearized Einstein fields as in definition \ref{def:linearizedEinsteinFields} such that
\begin{enumerate}
\item $\delta g_{ab}$ satisfies the linearized Einstein equation \eqref{eq:linEin} in the outgoing radiation gauge \eqref{eq:ORGcond}, 
\item $\psibase[-2]$ satisfies the BEAM condition from definition \ref{def:BEAM}, 
\item $\psibase[+2]$ satisfies the pointwise decay condition, point \ref{condition:spin+2goestozero} of definition \ref{ass:BEAMspin+2}. 
\end{enumerate}
\end{definition}

\subsection{Expansions at infinity and transport equations}
\label{sec:expansionAtInfinity} 

In this section, we introduce expansions at null infinity. These are Taylor expansions in $r^{-1}$, with the coefficients being functions on null infinity. One notable novel feature of our approach is that the functions on null infinity not only decay in time but also integrate to zero along null infinity and their iterated integrals also decay and integrate to zero. As shown in later sections, the Teukolsky variable $\psibase[-2]$ has such an expansion as a consequence of the Teukolsky-Starobinsky Identity. 

\begin{definition}
\label{def:AntiDerivOper} Let $f$ be a spin-weighted scalar on $\Scri^+$ which decays sufficiently rapidly at $i_0$, and define 
\begin{align}
(\AntiDeriv f)(\timefunc,\omega)
={}&\int_{-\infty}^\timefunc f(\timefunc',\omega)\di \timefunc' .
\end{align}
For a non-negative integer $\ii$,  define $\AntiDeriv^\ii$ by 
\begin{align}
\AntiDeriv^\ii={}&\AntiDeriv\circ\AntiDeriv^{\ii-1} ,
\end{align}
with $\AntiDeriv^0$ the identity operator. 
\end{definition}

It is now possible to define an expansion at null infinity. This depends on a level of regularity $\EkNoArg$, an order of the expansion $\ElNoArg$, an order $\EmNoArg$ up to which the expansion terms vanish, a weight parameter $\EalphaNoArg$, and a positive constant $D^2$. In the case that $\EmNoArg=\ElNoArg+1$, then all the terms in the expansion vanish, and the scalar is estimated solely by the remainder term. 
Conditions \eqref{eq:expansion:remainderBound}-\eqref{eq:expansion:expansionTermsDecay} are boundedness and decay conditions, but, under condition \eqref{eq:expansion:integrability},  the expansion terms  and their iterated time integrals  integrate to zero. 

\begin{definition}[$(\EkNoArg,\ElNoArg,\EmNoArg,\EalphaNoArg,\EDNoArg^2)$ expansion]
\label{def:asymptoticExpansion}
Let $\EkNoArg,\ElNoArg,\EmNoArg\in\Naturals$ be such that $0\leq \EmNoArg\leq \ElNoArg+1$. Let $\EalphaNoArg\in\Reals$. Let $\EDNoArg>0$. 

In the exterior region where $r\geq \timefunc$, a spin-weighted scalar $\varphi$ is defined to have a $(\EkNoArg,\ElNoArg,0,\EalphaNoArg,\EDNoArg^2)$ expansion if, for $\ii\in\{0,\ldots,\ElNoArg\}$, there are functions $\varphi_\ii$ on $\Scri^+$ and there is a function $\varphi_{\rem;\ElNoArg}$ in the exterior such that 
\begin{subequations}\label{eq:8.3}
\begin{align}
&\forall (t,r,\omega):& 
\varphi(t,r,\omega) 
={} \sum_{i=0}^\ElNoArg \frac{\rInv^i}{i!} \varphi_i(t,\omega)
&+\varphi_{\rem;\ElNoArg}(t,r,\omega) , 
\label{eq:expansion:expansionExists}\\
&&\norm{\varphi_{\rem;\ElNoArg}}_{W^{k}_{\EalphaNoArg-3} (\Dtauitext)}^2\lesssim{}& D^2 ,
\label{eq:expansion:remainderBound}\\
&&\norm{\varphi_{\rem;\ElNoArg}}_{W^k_{\EalphaNoArg-3} (\Dtauearly)}^2\lesssim{}& D^2  ,
\label{eq:expansion:remainderBoundInThePast}\\ 
&\forall \timefunc\in\Reals, \forall i\in\{0,\ldots,\ElNoArg\}:&
\int_{\Sphere}\absScri{\varphi_i(\timefunc,\omega) }{k}^2 \diTwoVol
\lesssim{}& D^2 \langle\timefunc\rangle^{2i-\EalphaNoArg+1},
\label{eq:expansion:expansionTermsDecay}\\
&\forall \omega\in\Sphere, 0\leq i < j\leq \ElNoArg+1, |\mathbf{a}|\leq \EkNoArg :&
\lim_{\timefunc \to \infty} (\mathbf{I}^{j-i} \ScriOps^{\mathbf{a}}\varphi_i)(\timefunc,\omega) ={}& 0 .
 \label{eq:expansion:integrability}
\end{align}
\end{subequations}
If, furthermore, for $\EmNoArg\in\Integers^+$, the expansion terms up to order $\EmNoArg-1\geq 0$ vanish, i.e.
\begin{align}
&\forall \timefunc\in\Reals, \forall i\in\{0,\ldots,m-1\}:&
\varphi_i(\timefunc,\omega)=0,
\label{eq:vanishingorderassumption}
\end{align}
then we say $\varphi$ has a $(\EkNoArg,\ElNoArg,\EmNoArg,\EalphaNoArg,\EDNoArg^2)$ expansion.
\end{definition}

Because $\YOp\timefunc = h'(r)$,  when trying to solve $\YOp\varphi=\varrho$ in terms of expansions from null infinity, one finds that the expansion coefficients for $\varphi$ are coupled through the expansion coefficients in $h'(r)$. The following lemma handles this coupling. 

\begin{lemma}
\label{lem:algebraicTransportExpansion}
Given any $l\in \Naturals$, for $k\in\{0,\ldots,l\}$, define $a_k$ and $b_k(R)$ to be such that 
\begin{align}
\frac{1}{h'(r)}={}& \sum_{k=0}^l a_k R^k + b_l(R)R^{l+1},
\label{eq:algebraicTransport:h'Expansion}
\end{align}
and define $b_{-1}(R)=1/h'(r)$. 

Let $\varrho$ and $\varphi$ be spin-weighted scalars.  
 Let $\varphi_{\initial}$ be a spin-weighted scalar on $\Sigma_{\initial}$. 

If $\varphi$ solves
\begin{subequations}
\begin{align}
\YOp\varphi
={}&\varrho ,\\
\varphi |_{\Sigma_{\initial}}
={}& \varphi_{\rm {init}} 
\end{align} 
\end{subequations}
with $\varrho$ having the expansion 
\begin{equation}
\label{eq:varrhoexpansion}
\varrho = \sum_{i=0}^j \frac{R^i}{i!} \varrho_i(\timefunc,\omega) + \varrho_{\rem;j},
\end{equation} 
then $\varphi$ is given by
\begin{equation}
\label{eq:varphiexpansion}
\varphi = \sum_{i=0}^{j-1} \frac{R^i}{i!}\varphi_i(\timefunc,\omega) + \varphi_{\rem;j-1},
\end{equation} 
where \index{P3phiCurlyiphiCurlyremj@$\varphi_i, \varphi_{\rem;j}$}\index{P3phiCurlyTildeiphiCurlyTilderemj@$\tilde\varphi_i, \tilde\varphi_{\rem;j}$}
\begin{subequations}
\begin{align}
\varphi_i(\timefunc,\omega) 
={}&  \sum_{k=0}^{i} \frac{i! a_{i-k}}{k!} \tilde\varphi_k(\timefunc,\omega), \qquad 0\leq i\leq j-1, 
\label{eq:varphii}\\
\varphi_{\rem;j-1}
={}&\sum_{i=0}^j \frac{b_{j-i-1}(R)}{i!} R^{j}\tilde\varphi_i(\timefunc,\omega) 
+ \tilde\varphi_{\rem;j}
+ \tilde\varphi_{\homo;j} ,
\label{eq:varphirem}\\
\tilde\varphi_0(\timefunc,\omega)
 ={}&   \AntiDeriv \varrho_0(\timefunc,\omega), \label{eq:rhotildephi0integral}\\
\tilde\varphi_i(\timefunc,\omega)
 ={}& \AntiDeriv \bigl (\varrho_i(\timefunc,\omega)- \sum_{k=0}^{i-1} \frac{ (i-1)i!a_{i-k-1}}{k!}\tilde\varphi_k(\timefunc,\omega)\bigr),\qquad 1\leq i\leq j,
 \label{eq:rhotildephiiintegral}
\end{align}
\end{subequations}
$\tilde\varphi_{\rem;j}$ is the solution of 
\begin{align}
\label{eq:Ytildevarphirem}
\YOp\tilde\varphi_{\rem;j} 
={}& \varrho_{\rem;j}-\sum_{i=0}^j \frac{1}{i!}(b'_{j-i-1}(R)R+jb_{j-i-1}(R))R^{j+1}\tilde\varphi_i(\timefunc,\omega) ,\\
\tilde\varphi_{\rem;j}|_{\Sigma_{\initial}}
={}& 0 ,
\label{eq:tildevarphiremainderinitialcondition}
\end{align}
and $\tilde{\varphi}_{\homo;j}$ is the solution of 
\begin{align}
Y\tilde{\varphi}_{\homo;j} ={}& 0 ,
\label{eq:Ytildevarphihomopart}\\
\tilde{\varphi}_{\homo;j} (\timefunc_{\initial}(r),r,\omega)
={}& \varphi_{\initial} (r,\omega) 
-\sum_{i=0}^j \frac{R^i}{i! h'(r)}\tilde{\varphi}_i(\timefunc_{\initial}(r),\omega) ,
\label{eq:algebriacTransport:tildeInitialData}
\end{align}
where $\timefunc_{\initial}(r)=\timefunc_0-h(r)/2$ is the value of $\timefunc$ on $\Sigma_{\initial}$ at $r$. 
\end{lemma}
\begin{proof}
Make an ansatz
\begin{equation}
\varphi = \sum_{i=0}^j \frac{R^i}{i!h'(r)}\tilde\varphi_i(\timefunc,\omega) + \tilde\varphi_{\rem;j}+\tilde{\varphi}_{\homo;j} .
\label{eq:varphiexpansiontildequantities}
\end{equation} 
This gives
\begin{equation}
\YOp\varphi = \sum_{i=0}^j\biggl( \YOp\Bigl(\frac{R^i}{i!h'(r)}\Bigr)\tilde\varphi_i(\timefunc,\omega)+\frac{R^i}{i!}\partial_\timefunc\tilde\varphi_i(\timefunc,\omega)\biggr) + \YOp\tilde\varphi_{\rem;j} +\YOp\tilde{\varphi}_{\homo;j}.
\label{eq:Yvarphiexpand}
\end{equation} 
We set $l=j-i-1$ in \eqref{eq:algebraicTransport:h'Expansion} and calculate
\begin{align}
\YOp\left(\frac{R^i}{h'(r)}\right)={}& 
\sum_{k=0}^l a_k (i+k)R^{i+k+1}
+(b'_l(R)R+(i+l+1)b_l(R))R^{i+l+2}.
\end{align}
Substituting this into \eqref{eq:Yvarphiexpand} gives
\begin{align}
\YOp\varphi 
={}& \sum_{i=0}^j\frac{R^i}{i!}\partial_\timefunc\tilde\varphi_i(\timefunc,\omega) 
+ \sum_{i=0}^j \sum_{k=0}^{j-i-1} \frac{a_k (i+k)}{i!}R^{i+k+1}\tilde\varphi_i(\timefunc,\omega)\nonumber\\
&+\sum_{i=0}^j \frac{1}{i!}(b'_{j-i-1}(R)R+jb_{j-i-1}(R))R^{j+1}\tilde\varphi_i(\timefunc,\omega)
 + \YOp\tilde\varphi_{\rem;j}+\YOp\tilde{\varphi}_{\homo;j}\nonumber\\
={}& \sum_{i=0}^j\frac{R^i}{i!}\partial_\timefunc\tilde\varphi_i(\timefunc,\omega) 
+ \sum_{i=1}^j \sum_{k=0}^{i-1} \frac{a_{i-k-1} (i-1)}{k!}R^{i}\tilde\varphi_k(\timefunc,\omega)\nonumber\\
&+\sum_{i=0}^j \frac{1}{i!}(b'_{j-i-1}(R)R+jb_{j-i-1}(R))R^{j+1}\tilde\varphi_i(\timefunc,\omega)
 + \YOp\tilde\varphi_{\rem;j}+\YOp\tilde{\varphi}_{\homo;j}.
\label{eq:algebraicTransport:intermediate}
\end{align}

If one now imposes conditions \eqref{eq:rhotildephi0integral} and \eqref{eq:rhotildephiiintegral} on the $\tilde{\varphi}_i$, then 
\begin{subequations}
\begin{align}
 \varrho_0(\timefunc,\omega)
 ={}&  \partial_\timefunc\tilde\varphi_0(\timefunc,\omega), \\
 \varrho_i(\timefunc,\omega)
 ={}&  \partial_\timefunc\tilde\varphi_i(\timefunc,\omega) + \sum_{k=0}^{i-1} \frac{a_{i-k-1} (i-1)i!}{k!}\tilde\varphi_k(\timefunc,\omega),\qquad 1\leq i\leq j .
\end{align}
\end{subequations}
If one further imposes that $\tilde{\varphi}_{\rem;l}$ and $\tilde{\varphi}_{\homo;j}$ satisfy the differential equations \eqref{eq:Ytildevarphirem} and \eqref{eq:Ytildevarphihomopart} respectively, then equation \eqref{eq:algebraicTransport:intermediate} becomes $Y\varphi=\varrho$. If one imposes the initial conditions \eqref{eq:tildevarphiremainderinitialcondition} on $\tilde\varphi_{\rem;j}$ and \eqref{eq:algebriacTransport:tildeInitialData}  on $\tilde{\varphi}_{\homo;j}$, then one finds that $\varphi$ satisfies the initial condition $\varphi|_{\Sigma_{\initial}}=\varphi_{\initial}$. 

Now applying the expansion \eqref{eq:algebraicTransport:h'Expansion} for $(h')^{-1}$ with $l=j-i-1$ in equation \eqref{eq:varphiexpansiontildequantities}, gathering like powers of $R$, and putting the $R^l$ term with the remainder term, one finds that
\begin{align}
\varphi 
={}& \sum_{i=0}^j  \sum_{k=0}^{j-i-1} \frac{a_k}{i!}  R^{i+k}\tilde\varphi_i(\timefunc,\omega) + \sum_{i=0}^j \frac{b_{j-i-1}(R)}{i!} R^{j}\tilde\varphi_i(\timefunc,\omega) + \tilde\varphi_{\rem;j} +\tilde{\varphi}_{\homo;j}\nonumber\\
={}& \sum_{i=0}^{j-1}  \sum_{k=0}^{i} \frac{a_{i-k}}{k!}  R^{i}\tilde\varphi_k(\timefunc,\omega) + \sum_{i=0}^j \frac{b_{j-i-1}(R)}{i!} R^{j}\tilde\varphi_i(\timefunc,\omega) + \tilde\varphi_{\rem;j} +\tilde{\varphi}_{\homo;j} .
\end{align} 
By comparing this expansion with the expansion \eqref{eq:varphiexpansion}, we finally get \eqref{eq:varphii} and \eqref{eq:varphirem}. 
\end{proof}

\begin{lemma}[Propagation of expansions]
\label{lem:YintExpansionEst}
\standardHypothesisOnDelta.  
Let $\varphi$ and $\varrho$ be spin-weighted scalars, 
and let $\varphi_{\initial}$ be a spin-weighted scalar on $\Sigma_{\initial}$. 
Let $\Ek{\varrho},\El{\varrho},\Em{\varrho}\in\Naturals$, $\Ealpha{\varrho}>0$, and $\ED{\varrho}>0$ be such that $\El{\varrho}\geq1$ and $2\El{\varrho}+3+\delta\leq \Ealpha{\varrho} \leq 2\El{\varrho}+4-\delta$. 
If  $\varphi$ solves
\begin{subequations}
\begin{align}
\YOp\varphi
={}&\varrho ,\\
\varphi |_{\Sigma_{\initial}}
={}& \varphi_{\rm {init}} 
\end{align} 
\end{subequations}
and $\varrho$ has a $(\Ek{\varrho},\El{\varrho},\Em{\varrho},\Ealpha{\varrho},\ED{\varrho}^2)$ expansion, 
then the following hold:
\begin{enumerate}
\item With
\begin{subequations}
\begin{align}
\Ek{\varphi} ={}& \Ek{\varrho} ,\\
\El{\varphi} ={}& \El{\varrho} -1 ,\\
\Em{\varphi} ={}& \min(\Em{\varrho},\El{\varphi}+1) ,\\
\Ealpha{\varphi} ={}& \Ealpha{\varrho} -2 -\delta , \\
\ED{\varphi}^2 ={}& \ED{\varrho}^2 + \inienergy{\Ek{\varphi}+1}{2\El{\varphi}+3}(\varphi) ,
\end{align}
\end{subequations}
$\varphi$ has a $(\Ek{\varphi},\El{\varphi},\Em{\varphi},\Ealpha{\varphi},\ED{\varphi}^2)$ expansion.

\item For any $\Lxitimes\in \{0,1\}$, and $\timefunc\geq \timefunc_0$, $\varphi$ satisfies 
\begin{align}
\norm{\LxiOp^{\Lxitimes}\varphi}_{W^{\Ek{\varphi}-\Lxitimes}_{\Ealpha{\varphi}-2} (\Boundtext)}^2
\lesC{\El{\varphi}}{}& \ED{\varphi}^2 \timefunc^{-2\Lxitimes}
\label{eq:LxiOpvarphitransitionregionesti}
\end{align}
and in the exterior region where $r\geq \timefunc$ that 
\begin{subequations}
\label{eq:variouspointwiseforvarphigeneral}
\begin{align}
\text{for }& \Em{\varrho}\leq \El{\varrho}, &\absRescaled{\varphi}{\Ek{\varphi}-3}^2 \lesC{\El{\varphi}}{}& \ED{\varphi}^2r^{-2\Em{\varrho}}\timefunc^{-\Ealpha{\varphi}+1+2\Em{\varrho}},
\label{eq:pointwiseDescentGeneralCase}
\\
\text{for }& \Em{\varrho}= \El{\varrho}+1, &\absRescaled{\varphi}{\Ek{\varphi}-3}^2 \lesC{\El{\varphi}}{}& \ED{\varphi}^2r^{-\Ealpha{\varphi}+1} .
\label{eq:pointwiseDescentSpecialCase}
\end{align}
\end{subequations}
\end{enumerate}
\end{lemma}

\begin{proof}
For ease of presentation, throughout this proof, we use mass normalization as in definition~\ref{def:massnormalization} and use $\lesssim$ to mean $\lesC{\El{\varphi}}$. 
Since, by assumption, $\varrho$ has an expansion, one can apply lemma \ref{lem:algebraicTransportExpansion} to obtain an expansion for $\varphi$. In the following, for simplicity, we use $\EkNoArg$ to denote $\Ek{\varphi}=\Ek{\varrho}$.

\begin{steps}
\step{Treat the $\tilde{\varphi}_i$} We first show in this step that
\begin{subequations}
\label{eq:tildephiinductionassump}
\begin{align}
&\forall \timefunc\in \Reals, \Lxitimes\in\{0,1\}, i\in\{0,\ldots, \El{\varrho}\},& \int_{\Sphere}\absScri{\LxiOp^{\Lxitimes}\tilde\varphi_{i}(\timefunc,\omega)}{\EkNoArg-\Lxitimes}^2\diTwoVol
\lesssim{}&\ED{\varrho}^2\langle\timefunc\rangle^{2i-\Ealpha{\varrho}+3-2\Lxitimes},
\label{eq:tildephiinductionassumpEstimate}\\
&\forall \omega\in \Sphere, 0\leq i<j\leq \El{\varrho},|\mathbf{a}|\leq k,&\lim_{\timefunc \to \infty} (\mathbf{I}^{j-i} \ScriOps^{\mathbf{a}} \tilde\varphi_i)(\timefunc,\omega)
={}&0,   
\end{align}
\end{subequations}
and
\begin{align}
&\forall \timefunc \in \Reals, \forall i\in\{0,\ldots,\Em{\varrho}-1\},& \tilde\varphi_i={}&0.&\hspace{20ex}
\label{eq:Firstexpansiontermsvanish}
\end{align}

From \eqref{eq:rhotildephi0integral} and \eqref{eq:rhotildephiiintegral}, it is clear that \eqref{eq:Firstexpansiontermsvanish} holds, and hence \eqref{eq:tildephiinductionassump} holds true for $i\in\{0,\ldots,\Em{\varrho}-1\}$. Furthermore, if $\Em{\varrho}=\El{\varrho}+1$, all the $\{\tilde\varphi_i\}_{i=0}^{\El{\varrho}}$ vanish, and \eqref{eq:tildephiinductionassump} is manifestly valid.
Hence, we only need to prove \eqref{eq:tildephiinductionassump} below for $\Em{\varrho}\leq i\leq \El{\varrho}$.
 
The remaining $\Em{\varrho}\leq i\leq \El{\varrho}$ cases are treated by induction. First, consider the $i=\Em{\varrho}$ case. Since $\Ealpha{\varrho}>2\El{\varrho}+3\geq 2\El{\varphi}+3$, the expression \eqref{eq:rhotildephiiintegral} for $\tilde{\varphi}_i$ and the integrability and decay conditions \eqref{eq:8.3} for $\varrho_{\Em{\varrho}}$ give, for any $\timefunc\geq \timefunc_0$, 
\begin{subequations}
\begin{align}
\ScriOps^{\mathbf{a}}\tilde\varphi_{\Em{\varrho}}(\timefunc,\omega)={}&\int_{-\infty}^\timefunc \ScriOps^{\mathbf{a}}\varrho_{\Em{\varrho}}(\timefunc',\omega)\di \timefunc'=-\int_\timefunc^\infty \ScriOps^{\mathbf{a}}\varrho_{\Em{\varrho}}(\timefunc',\omega)\di \timefunc',\\
\ScriOps^{\mathbf{a}}\LxiOp\tilde\varphi_{\Em{\varrho}}(\timefunc,\omega)={}& \ScriOps^{\mathbf{a}}\varrho_{\Em{\varrho}}(\timefunc,\omega),\\
\lim_{\timefunc \to \infty} (\mathbf{I}^{j} \ScriOps^{\mathbf{a}} \tilde\varphi_{\Em{\varrho}})(\timefunc,\omega)={}&\lim_{\timefunc \to \infty} (\mathbf{I}^{j+1} \ScriOps^{\mathbf{a}}\varrho_{\Em{\varrho}})(\timefunc,\omega)=0,  
\qquad 0 \leq j \leq \El{\varrho}-\Em{\varrho},  |\mathbf{a}|\leq \EkNoArg ,
\end{align}
\end{subequations}
and, for any $\timefunc\geq \timefunc_0$, 
\begin{align}
\int_{\Sphere}\absScri{\LxiOp\tilde\varphi_{\Em{\varrho}}(\timefunc,\omega)}{\EkNoArg}^2\diTwoVol
\leq{}&
 \int_{\Sphere} \absScri{\varrho_{\Em{\varrho}}(\timefunc,\omega)}{\EkNoArg}^2 \diTwoVol \nonumber\\
\lesssim {}& \ED{\varrho}^2\timefunc^{-\Ealpha{\varrho} +1+2\Em{\varrho}},\\
\int_{\Sphere}\absScri{\tilde\varphi_{\Em{\varrho}}(\timefunc,\omega)}{\EkNoArg}^2\diTwoVol
\leq{}&
 \int_{\Sphere}\left(\int_\timefunc^\infty \absScri{\varrho_{\Em{\varrho}}(\timefunc',\omega)}{\EkNoArg}\di\timefunc'\right)^2 \diTwoVol \nonumber\\
\leq{}&
 \left(\int_\timefunc^\infty \left(\int_{\Sphere}\absScri{\varrho_{\Em{\varrho}}(\timefunc',\omega)}{\EkNoArg}^2\diTwoVol\right)^{1/2} \di\timefunc'\right)^2 \nonumber\\
\lesssim {}& \left(\ED{\varrho}\int_\timefunc^\infty (\timefunc')^{-\Ealpha{\varrho}/2+1/2+\Em{\varrho}}\di \timefunc'\right)^2 \nonumber\\
\lesssim {}& \ED{\varrho}^2\timefunc^{-\Ealpha{\varrho}+3+2\Em{\varrho}}
\label{eq:tildevarphim+1esti}
\end{align}
where the second step of \eqref{eq:tildevarphim+1esti} follows from Minkowski's integral inequality.
Similarly, for $\timefunc\leq -\timefunc_0$ and $\Lxitimes\in\{0,1\}$, one has $\int_{\Sphere}\absScri{\LxiOp^{\Lxitimes}
\tilde\varphi_{\Em{\varrho}}(\timefunc,\omega)}{\EkNoArg-\Lxitimes}^2\diTwoVol \lesssim \ED{\varrho}^2 \abs{\timefunc}^{-\Ealpha{\varrho} +3+2\Em{\varrho}-2\Lxitimes}$, and, for $\timefunc\in[-\timefunc_0,\timefunc_0]$, one has that $\int_{\Sphere}\absScri{\LxiOp^{\Lxitimes}\tilde\varphi_{\Em{\varrho}}(\timefunc,\omega)}{\EkNoArg-\Lxitimes}\diTwoVol$ is bounded. These prove the $i=\Em{\varrho}$ case of \eqref{eq:tildephiinductionassump}.

For induction, let $l'\leq \El{\varrho}$, and suppose that the estimates \eqref{eq:tildephiinductionassump} hold for $\Em{\varrho}\leq i\leq l'-1$.
From the expression \eqref{eq:rhotildephiiintegral} for the $\tilde\varphi_i$, the decay and integrability conditions for $\varrho_i$, the assumption that $\Ealpha{\varrho} >2\El{\varrho}+3$, and the inductive hypothesis, one finds that, for any $\Em{\varrho} \leq i \leq l'\leq \El{\varrho}$, $\Lxitimes\in \{0,1\}$, and $\timefunc\geq \timefunc_0$, 
\begin{align}
\ScriOps^{\mathbf{a}}\tilde\varphi_i(\timefunc,\omega)
 ={}&\int_{-\infty}^\timefunc \biggl (\ScriOps^{\mathbf{a}}\varrho_i(\timefunc',\omega)
 - \sum_{j=0}^{i-1} \frac{a_{i-j-1} (i-1)i!}{j!}\ScriOps^{\mathbf{a}}\tilde\varphi_j(\timefunc',\omega)
 \biggr) \di \timefunc'\nonumber\\
 ={}&\int_\timefunc^{\infty} \biggl (\ScriOps^{\mathbf{a}}\varrho_i(\timefunc',\omega)- \sum_{j=0}^{i-1} \frac{a_{i-j-1} (i-1)i!}{j!}\ScriOps^{\mathbf{a}}\tilde\varphi_j(\timefunc',\omega)\biggr) \di \timefunc',\\
 \hspace{13ex}&\hspace{-13ex}
\int_{\Sphere}\absScri{\LxiOp^{\Lxitimes}\tilde\varphi_{i}(\timefunc,\omega)}{\EkNoArg-\Lxitimes}^2\diTwoVol \nonumber\\
\leq{}&
 \int_{\Sphere}\Biggl(\int_\timefunc^\infty \biggl(\absScri{\LxiOp^{\Lxitimes}\varrho_i(\timefunc',\omega)}{\EkNoArg-\Lxitimes} +\sum_{j=0}^{i-1} \frac{a_{i-j-1} (i-1)i!}{j!}\absScri{\LxiOp^{\Lxitimes}\tilde\varphi_j(\timefunc',\omega)}{\EkNoArg-\Lxitimes}\biggr)
 \di\timefunc'\Biggr)^2 \diTwoVol \nonumber\\
\lesssim{}&
 \Biggl(\int_\timefunc^\infty \Biggl(\int_{\Sphere}
 \biggl(\absScri{\LxiOp^{\Lxitimes}\varrho_i(\timefunc',\omega)}{\EkNoArg-\Lxitimes}^2+\sum_{j=0}^{i-1}\absScri{\LxiOp^{\Lxitimes}\tilde\varphi_j(\timefunc',\omega)}{\EkNoArg-\Lxitimes}^2 \biggr)
 \diTwoVol\Biggr)^{1/2} \di\timefunc'\Biggr)^2 \nonumber\\
\lesssim{}&\Biggl(\ED{\varrho}\int_\timefunc^{\infty} \biggl ((\timefunc')^{i-\Ealpha{\varrho}/2+1/2}+ \sum_{j=0}^{i-1} (\timefunc')^{j-\Ealpha{\varrho}/2+3/2-\Lxitimes}\biggr) \di \timefunc'\Biggr)^2\nonumber\\
\lesssim{}& \ED{\varrho}^2\timefunc^{2i-\Ealpha{\varrho}+3-2\Lxitimes}. 
\end{align}
Similarly, for $\timefunc\leq -\timefunc_0$, one finds $\int_{\Sphere}\absScri{\LxiOp^{\Lxitimes}\tilde\varphi_{i}(\timefunc,\omega)}{\EkNoArg-\Lxitimes}^2\diTwoVol$ $\lesssim \ED{\varrho}^2\abs{\timefunc}^{2i-\Ealpha{\varrho}+3-2\Lxitimes}$, and, for $\timefunc\in[-\timefunc_0,\timefunc_0]$, one has that $\int_{\Sphere}\absScri{\LxiOp^{\Lxitimes}\tilde\varphi_{i}(\timefunc,\omega)}{\EkNoArg-\Lxitimes}^2\diTwoVol$ is bounded. These together imply
\begin{align}
&\forall \timefunc\in\Reals, \forall i\in\{\Em{\varrho},\ldots,l'\}, \forall \Lxitimes\in \{0,1\}:&
\int_{\Sphere}\absScri{\LxiOp^{\Lxitimes}\tilde\varphi_{i}(\timefunc,\omega)}{\EkNoArg-\Lxitimes}^2\diTwoVol \lesssim{}& \ED{\varrho}^2\langle\timefunc\rangle^{2i-\Ealpha{\varrho}+3-2\Lxitimes}.
\label{eq:tildevarphiiesti}
\end{align}
Furthermore, for $i$ satisfying $\Em{\varrho}\leq i\leq j\leq l'\leq \El{\varrho}$, one finds
\begin{align}
\lim_{\timefunc \to \infty} (\mathbf{I}^{j-i} \ScriOps^{\mathbf{a}}\tilde\varphi_i)(\timefunc,\omega)
={}&\lim_{\timefunc \to \infty} (\mathbf{I}^{j-i+1} \ScriOps^{\mathbf{a}}\tilde\varrho_i)(\timefunc,\omega)
- \sum_{\iip=0}^{i-1} \frac{a_{i-\iip-1} (i-1)i!}{\iip!}\lim_{\timefunc \to \infty} (\mathbf{I}^{j-i+1}\ScriOps^{\mathbf{a}}\tilde\varphi_{\iip})(\timefunc,\omega)=0.
\end{align}
Thus, by induction, the $\tilde{\varphi}_i$ satisfy \eqref{eq:tildephiinductionassump} for $\Em{\varrho}\leq i\leq \El{\varrho}$. This then completes the proofs of
\eqref{eq:tildephiinductionassump} and 
\eqref{eq:Firstexpansiontermsvanish}.

Next, we consider the estimates of the flux and bulk integrals of $\tilde\varphi_i$. Since $\Ealpha{\varphi}<\Ealpha{\varrho}-2<2\El{\varrho}+2$, the operators in $\rescaledOps$ are linear combinations of operators in $\ScriOps$ and $R\partial_R$ with coefficients $\bigOAnalytic(1)$, and $R\partial_R$ commutes with the operators in $\ScriOps$, the above implies, for any $0\leq i\leq \El{\varrho}$, $\Lxitimes\in \{0,1\}$, and  $\timefunc'\geq\timefunc_0$,
\begin{align}
\norm{ r^{-\El{\varrho}}\LxiOp^{\Lxitimes}\tilde\varphi_i}_{W^{\EkNoArg-\Lxitimes}_{\Ealpha{\varphi}-3}(\Dtautprimeext)}^2
\lesssim{}& \sum_{j=0}^{\EkNoArg-\Lxitimes}\int_{\timefunc'}^\infty \int_{\timefunc}^\infty \int_{\Sphere} \left(r^{\Ealpha{\varphi}-3} \absScri{(R\partial_R)^j (r^{-\El{\varrho}}\LxiOp^{\Lxitimes}\tilde\varphi_i) }{{\EkNoArg-\Lxitimes}-j}^2\right)\diTwoVol\di r \di\timefunc
\nonumber\\
\lesssim{}& \int_{\Dtautprimeext} r^{\Ealpha{\varphi}-3-2\El{\varrho}} \absRescaled{\LxiOp^{\Lxitimes}\tilde\varphi_i}{\EkNoArg-\Lxitimes}^2 \diFourVol
\nonumber\\
\lesssim{} & \ED{\varrho}^2 \int_{\timefunc'}^\infty \int_{\timefunc}^\infty r^{\Ealpha{\varphi}-3-2\El{\varrho}} \timefunc^{2i-\Ealpha{\varrho} +3-2\Lxitimes}  \di r\di\timefunc \nonumber\\
\lesssim{}& \ED{\varrho}^2 (\timefunc')^{-\delta-2\Lxitimes}.
\end{align}
By the same argument, it follows that
\begin{align}
\norm{ r^{-\El{\varrho}}\LxiOp^{\Lxitimes}\tilde\varphi_i }_{W^{\EkNoArg-\Lxitimes}_{\Ealpha{\varphi}-2}(\Boundtprimeext)}^2
\lesssim{}&\sum_{j=0}^{\EkNoArg-\Lxitimes}\int_{\timefunc'}^{\infty} \int_{\Sphere} \left(r^{\Ealpha{\varphi}-2} \absScri{(R\partial_R)^j (r^{-\El{\varrho}}\LxiOp^{\Lxitimes}\tilde\varphi_i) }{{\EkNoArg-\Lxitimes}-j}^2\right)\Big\vert_{r=\timefunc} \diTwoVol\di\timefunc\nonumber\\
\lesssim{}& \int_{\timefunc'}^{\infty} \int_{\Sphere} \timefunc^{\Ealpha{\varphi}-2-2\El{\varrho}} \absScri{ \LxiOp^{\Lxitimes}\tilde\varphi_i }{\EkNoArg-\Lxitimes}^2 \diTwoVol\di\timefunc,
\end{align}
where in the last step we used the fact that $\tilde\varphi_i$ is independent of $r$.
Hence, for any $0\leq i\leq \El{\varrho}$, $\Lxitimes\in \{0,1\}$,  and $\timefunc'\geq\timefunc_0$ 
\begin{align}
\norm{ r^{-\El{\varrho}}\LxiOp^{\Lxitimes}\tilde\varphi_i }_{W^{\EkNoArg-\Lxitimes}_{\Ealpha{\varphi}-2}(\Boundtprimeext)}^2
\lesssim{}& \int_{\timefunc'}^{\infty} \ED{\varrho}^2 \timefunc^{\Ealpha{\varphi}-2-2\El{\varrho}}  \timefunc^{2i-\Ealpha{\varrho}+3-2\Lxitimes} \di\timefunc \nonumber\\
\lesssim{} &\ED{\varrho}^2 (\timefunc')^{-\delta-2\Lxitimes} . 
\end{align} 
Gathering together these estimates for $\tilde\varphi_i$,  we obtain for any $0\leq i\leq \El{\varrho}$, $\Lxitimes\in \{0,1\}$, and $\timefunc\geq\timefunc_0$, 
\begin{align}
\norm{ r^{-\El{\varrho}}\LxiOp^{\Lxitimes}\tilde\varphi_i }_{W^{\EkNoArg-\Lxitimes}_{\Ealpha{\varphi}-2}(\Boundtext)}^2
+\norm{ r^{-\El{\varrho}}\LxiOp^{\Lxitimes}\tilde\varphi_i}_{W^{\EkNoArg-\Lxitimes}_{\Ealpha{\varphi}-3}(\Dtautext)}^2
\lesssim{} &\ED{\varrho}^2 \timefunc^{-\delta-2\Lxitimes} . 
\label{eq:phiiTildeTransitionAndMorawetzExterior}
\end{align} 
Similarly, for any $0\leq i\leq \El{\varrho}$, $\Lxitimes\in \{0,1\}$, and $\timefunc\geq\timefunc_0$, 
\begin{align}
\Vert r^{-\El{\varrho}-1}\LxiOp^{\Lxitimes}\tilde\varphi_i \Vert_{W^{\EkNoArg-\Lxitimes}_{\Ealpha{\varphi}+2-3}(\Dtautext)}^2
\lesssim{}& \ED{\varrho}^2 \timefunc^{-\delta-2\Lxitimes} .  
\label{eq:phiiTildeMorawetzExteriorRescaled}
\end{align} 

\step{Treat the $\varphi_i$} If $\Em{\varrho}\geq \El{\varrho}$, it follows from  \eqref{eq:Firstexpansiontermsvanish} that $\tilde{\varphi}_i=0$ for any $0\leq i\leq \Em{\varrho}-1$ and hence formula \eqref{eq:varphii} implies $\varphi_i=0$ for all $i\in \{0,\ldots,\El{\varrho}-1\}$. Instead, if $\Em{\varrho}\leq \El{\varrho}-1$, it follows from equations \eqref{eq:Firstexpansiontermsvanish} and \eqref{eq:varphii} that $\varphi_i=0$ for any $i\in \{0,\ldots,\Em{\varrho}-1\}$. Therefore, in either case, $\varphi_i=0$ for any $i\in \{0,\ldots,\Em{\varphi}-1\}$ and any $(\timefunc, \omega)$. This proves condition \eqref{eq:vanishingorderassumption}.

  For any $\Em{\varrho}\leq i \leq \El{\varrho}-1=\El{\varphi}$, $\timefunc\in\Reals$, and $\Lxitimes\in \{0,1\}$, since $\Ealpha{\varrho} >2\El{\varrho}+3$, equations \eqref{eq:varphii}, \eqref{eq:tildephiinductionassump}, and \eqref{eq:Firstexpansiontermsvanish} can be used to obtain 
\begin{subequations}
\begin{align}
\int_{\Sphere}\absScri{\LxiOp^{\Lxitimes}\varphi_{i}(\timefunc,\omega)}{\EkNoArg-\Lxitimes}^2\diTwoVol 
\lesssim{}& \sum_{j=0}^{i}  \int_{\Sphere}\absScri{\LxiOp^{\Lxitimes}\tilde\varphi_{j}(\timefunc,\omega)}{\EkNoArg-\Lxitimes}^2\diTwoVol
\lesssim{} \ED{\varrho}^2\sum_{j=0}^{i}\langle\timefunc\rangle^{2j-\Ealpha{\varphi}+1-\delta-2\Lxitimes}\nonumber\\
\lesssim{}&\ED{\varrho}^2 \langle\timefunc\rangle^{2i-\Ealpha{\varrho}+3-2\Lxitimes},
\label{eq:Pointwisevarphiiexpansion} \\
\lim_{\timefunc \to \infty} (\mathbf{I}^{j-i} \ScriOps^{\mathbf{a}}\varphi_i)(\timefunc,\omega) 
={}&  \sum_{i'=0}^{i} \frac{i! a_{i-i'}}{(i')!} \lim_{\timefunc \to \infty} (\mathbf{I}^{j-i} \ScriOps^{\mathbf{a}}\tilde\varphi_{i'})(\timefunc,\omega)=0 , \qquad  \Em{\varrho} \leq i < j\leq \El{\varrho}.
\label{eq:Integrabilityvarphiiexpansion}
\end{align}
\end{subequations}
In particular, the estimate \eqref{eq:Pointwisevarphiiexpansion} holds for any $0\leq i \leq \El{\varphi}$. These together verify the conditions \eqref{eq:expansion:expansionTermsDecay} and \eqref{eq:expansion:integrability}.

The estimates for $\tilde\varphi_i$ in the above step, together with the uniform boundedness of the coefficients $\frac{i! a_{i-k}}{k!}$ in the expression \eqref{eq:varphii} of $\varphi_i$, imply that for any  $0\leq i \leq \El{\varphi}$, $\Lxitimes\in \{0,1\}$ and $\timefunc\geq \timefunc_0$,
\begin{align}
\label{eq:phiiTransitionAndMorawetz}
\norm{ r^{-i}\LxiOp^{\Lxitimes}\varphi_i }_{W^{\EkNoArg-\Lxitimes}_{\Ealpha{\varphi}-2}(\Boundtext)}^2
+\norm{ r^{-i}\LxiOp^{\Lxitimes}\varphi_i}_{W^{\EkNoArg-\Lxitimes}_{\Ealpha{\varphi}-3}(\Dtautext)}^2
\lesssim{} &\ED{\varrho}^2 \timefunc^{-\delta-2\Lxitimes} . 
\end{align}

\step{Treat $\tilde{\varphi}_{\rem;\El{\varrho}}$} 
Since each $b'_{j-i-1}(R)R+jb_{j-i-1}(R)$ is uniformly bounded, from estimates \eqref{eq:TransIngoingMoraext} and \eqref{eq:TransIngoingMorafar} in lemma \ref{lem:TransEqGeneralRegionMora} about transport equations, one finds that, for any  $\Lxitimes\in \{0,1\}$ and $\timefunc_2\geq \timefunc_1\geq \timefunc_0$,
\begin{align}
\hspace{4ex}&\hspace{-4ex}
\Vert \tilde\varphi_{\rem;\El{\varrho}} \Vert_{W^{\EkNoArg}_{\Ealpha{\varphi}-2} (\Boundext)}^2
+ \norm{\tilde{\varphi}_{\rem;\El{\varrho}}}_{W^{\EkNoArg}_{\Ealpha{\varphi}-2}(\Stauext)}^2
+ \norm{\tilde{\varphi}_{\rem;\El{\varrho}}}_{W^{\EkNoArg}_{\Ealpha{\varphi}-3}(\Dtauext)}^2 \nonumber\\
\lesssim{}&
\norm{\tilde{\varphi}_{\rem;\El{\varrho}}}_{W^{\EkNoArg}_{\Ealpha{\varphi}-2}(\Stauiext)}^2
+ \norm{{\varrho}_{\rem;\El{\varrho}}}_{W^{\EkNoArg}_{(\Ealpha{\varphi}+2)-3}(\Dtauext)}^2
+ \sum_{i=0}^{\El{\varrho}} \norm{r^{-{\El{\varrho}}-1} \tilde{\varphi}_i }_{W^{\EkNoArg}_{(\Ealpha{\varphi}+2)-3}(\Dtauext)}^2,
\label{eq:tildevarphiremainderinDtauv1}\\
\hspace{4ex}&\hspace{-4ex}
 \norm{\tilde{\varphi}_{\rem;\El{\varrho}}}_{W^{\EkNoArg}_{\Ealpha{\varphi}-2}(\Stauitext)}^2
+ \norm{\tilde{\varphi}_{\rem;\El{\varrho}}}_{W^{\EkNoArg}_{\Ealpha{\varphi}-3}(\Dtauearly)}^2 \nonumber\\
\lesssim{}&
\norm{\tilde{\varphi}_{\rem;\El{\varrho}}}_{W^{\EkNoArg}_{\Ealpha{\varphi}-2}(\Sigma_{\initial})}^2
+ \norm{{\varrho}_{\rem;\El{\varrho}}}_{W^{\EkNoArg}_{(\Ealpha{\varphi}+2)-3}(\Dtauearly)}^2
+ \sum_{i=0}^{\El{\varrho}} \norm{r^{-{\El{\varrho}}-1} \tilde{\varphi}_i }_{W^{\EkNoArg}_{(\Ealpha{\varphi}+2)-3}(\Dtauearly)}^2 .
\label{eq:tildevarphiremainderinOmegainitialt0}
\end{align}
From the assumption that $\Ealpha{\varphi}+2<\Ealpha{\varrho} $, there is the bound $\norm{\varrho_{\rem;\El{\varrho}}}^2_{W^{\EkNoArg}_{(\Ealpha{\varphi}+2)-3}(\Dtauext)} \lesssim \ED{\varrho}^2$ for the second term on the right of \eqref{eq:tildevarphiremainderinDtauv1}. The third term on the right of \eqref{eq:tildevarphiremainderinDtauv1} are bounded by $\ED{\varrho}^2$ in estimate \eqref{eq:phiiTildeMorawetzExteriorRescaled}. Thus, one finds, for any  $\timefunc\geq \timefunc_0$,
\begin{align}
\hspace{4ex}&\hspace{-4ex}
\Vert\tilde\varphi_{\rem;\El{\varrho}} \Vert_{W^{\EkNoArg}_{\Ealpha{\varphi}-2} (\Boundtext)}^2
+ \norm{\tilde{\varphi}_{\rem;\El{\varrho}}}_{W^{\EkNoArg}_{\Ealpha{\varphi}-2}(\Stautext)}^2
+ \norm{\tilde{\varphi}_{\rem;\El{\varrho}}}_{W^{\EkNoArg}_{\Ealpha{\varphi}-3}(\Dtautext)}^2 \nonumber\\
\lesssim{}&
\ED{\varrho}^2 + \norm{\tilde{\varphi}_{\rem;\El{\varrho}}}_{W^{\EkNoArg}_{\Ealpha{\varphi}-2}(\Stauitext)}^2 .
\label{eq:tildevarphiremainderinOmegatit2}
\end{align}

From the assumption that $\varrho$ has a $(\Ek{\varrho},\El{\varrho},\Em{\varrho},\Ealpha{\varrho},\ED{\varrho}^2)$ expansion and estimates \eqref{eq:tildevarphiiesti} and \eqref{eq:phiiTildeTransitionAndMorawetzExterior} for $\tilde{\varphi}_i$, it follows that
\begin{subequations}
\label{eq:varrhoremainderandtildevarphiiinfarregion}
\begin{align}
 \norm{{\varrho}_{\rem;\El{\varrho}}}_{W^{\EkNoArg}_{(\Ealpha{\varphi}+2)-3}(\Dtauearly)}^2
 \lesssim{}&\ED{\varrho}^2,\\
 \Vert
{\varrho}_{\rem;\El{\varrho}} \Vert^2_{W^{\EkNoArg}_{(\Ealpha{\varphi}+2)-3} (\Dtautext)}\lesssim{}& \ED{\varrho}^2 ,
\label{eq:varrhoremainderbulkestiexterior}\\
 \sum_{i=0}^{\El{\varrho}} \norm{r^{-{\El{\varrho}}-1} \tilde{\varphi}_i }_{W^{\EkNoArg}_{(\Ealpha{\varphi}+2)-3}(\Dtauearly)}^2 \lesssim{}&  \ED{\varrho}^2 ,\\
 \sum_{i=0}^{\El{\varrho}}\left(\Vert
r^{-{\El{\varrho}}-1} \tilde{\varphi}_i \Vert^2_{W^{\EkNoArg}_{(\Ealpha{\varphi}+2)-2} (\Boundtext)}
+\Vert
r^{-{\El{\varrho}}-1} \tilde{\varphi}_i  \Vert^2_{W^{\EkNoArg}_{(\Ealpha{\varphi}+2)-3} (\Dtautext)}\right) \lesssim{}& \ED{\varrho}^2 .
\label{eq:phiiTildeTransitionAndMorawetzExteriorSummed}
\end{align}
\end{subequations}
Moreover, it holds that $r\sim -\timefunc$ on $\Stauini$, and, since $\Ealpha{\varphi}<\Ealpha{\varrho}-2<2\El{\varrho}+2$, it follows that 
\begin{subequations}
\label{eq:varrhoremainder:initialenergy}
\begin{align}
\sum_{i=0}^{\El{\varrho}} \norm{r^{-i} \varrho_i }_{W^{\EkNoArg-1}_{\Ealpha{\varphi}}(\Stauini)}^2 \lesssim{}&  \ED{\varrho}^2\int_{r_+}^{\infty}  r^{\Ealpha{\varphi}}r^{-2i}r^{2i-\Ealpha{\varrho} +1} \di r\notag\\
\lesssim {}& \ED{\varrho}^2,\\
\sum_{i=0}^{\El{\varrho}}\norm{ R^{{\El{\varrho}}+1}\tilde\varphi_i(\timefunc,\omega)}_{W^{\EkNoArg-1}_{\Ealpha{\varphi}}(\Sigma_{\initial})}^2 \lesssim{}&
\ED{\varrho}^2\sum_{i=0}^{\El{\varrho}}\int_{r_+}^{\infty}  r^{\Ealpha{\varphi}}r^{-2{\El{\varrho}}-2}r^{2i-\Ealpha{\varrho} +3} \di r\notag\\
\lesssim {}& \ED{\varrho}^2.
\end{align}
\end{subequations}
Since $\tilde{\varphi}_{\rem;\El{\varrho}}$ vanishes on $\Sigma_{\initial}$ by assumption, all the derivatives tangent to  $\Sigma_{\initial}$  of $\tilde{\varphi}_{\rem;\El{\varrho}}$ also vanish. Each of the operators in $\rescaledOps$ on $\Sigma_{\initial}$  can be written as a sum of the tangential derivatives and $\bigOAnalytic(1) r\YOp$. Therefore, we have from the expression \eqref{eq:tildevarphiremainderinitialcondition} and estimates \eqref{eq:varrhoremainder:initialenergy} that
\begin{align}
\norm{\tilde{\varphi}_{\rem;\El{\varrho}}}_{W^{\EkNoArg}_{\Ealpha{\varphi}-2}(\Sigma_{\initial})}^2 \lesssim {}& \norm{\YOp\tilde{\varphi}_{\rem;\El{\varrho}}}_{W^{\EkNoArg-1}_{\Ealpha{\varphi}}(\Sigma_{\initial})}^2\notag\\
\lesssim {}&\norm{\varrho_{\rem;{\El{\varrho}}}}_{W^{\EkNoArg-1}_{\Ealpha{\varphi}}(\Sigma_{\initial})}^2
+\sum_{i=0}^{\El{\varrho}}\norm{ R^{{\El{\varrho}}+1}\tilde\varphi_i(\timefunc,\omega)}_{W^{\EkNoArg-1}_{\Ealpha{\varphi}}(\Sigma_{\initial})}^2\notag\\
\lesssim {}& \ED{\varrho}^2.
\label{eq:tildevarphiremainderinitialenergy}
\end{align}
Combining estimates \eqref{eq:tildevarphiremainderinDtauv1},  \eqref{eq:tildevarphiremainderinOmegainitialt0},   \eqref{eq:tildevarphiremainderinOmegatit2}, \eqref{eq:varrhoremainderandtildevarphiiinfarregion}, and \eqref{eq:tildevarphiremainderinitialenergy} gives that,
for any  $\timefunc\geq \timefunc_0$,
\begin{subequations}
\begin{align}
\Vert \tilde\varphi_{\rem;\El{\varrho}} \Vert_{W^{\EkNoArg}_{\Ealpha{\varphi}-2} (\Boundtext)}^2
+ \norm{\tilde{\varphi}_{\rem;\El{\varrho}}}_{W^{\EkNoArg}_{\Ealpha{\varphi}-2}(\Stautext)}^2
+ \norm{\tilde{\varphi}_{\rem;\El{\varrho}}}_{W^{\EkNoArg}_{\Ealpha{\varphi}-3}(\Dtautext)}^2
\lesssim{}&
\ED{\varrho}^2,
\label{eq:tildevarphiremainderinDtauzeroorder}\\
 \norm{\tilde{\varphi}_{\rem;\El{\varrho}}}_{W^{\EkNoArg}_{\Ealpha{\varphi}-2}(\Stauitext)}^2
+ \norm{\tilde{\varphi}_{\rem;\El{\varrho}}}_{W^{\EkNoArg}_{\Ealpha{\varphi}-3}(\Dtauearly)}^2
\lesssim{} &
\ED{\varrho}^2 .
\end{align}
An application of lemma \ref{lem:Transitionfluxcontrolledbybulkexterior} together with estimate \eqref{eq:varrhoremainderbulkestiexterior} implies
\begin{align}
\Vert \varrho_{\rem;\El{\varrho}} \Vert^2_{W^{\EkNoArg-1}_{(\Ealpha{\varphi}+2)-2} (\Boundtext)}
\lesssim\Vert
\varrho_{\rem;\El{\varrho}} \Vert^2_{W^{\EkNoArg}_{(\Ealpha{\varphi}+2)-3} (\Dtautext)}\lesssim \ED{\varrho}^2. \label{eq:transitionfluxcontrolledbybulk}
\end{align}
Hence, from the assumption, equation \eqref{eq:Ytildevarphirem}, and estimates \eqref{eq:varrhoremainderandtildevarphiiinfarregion}, we have, for any  $\timefunc\geq \timefunc_0$,
\begin{align}
\hspace{4ex}&\hspace{-4ex}\Vert
\YOp\tilde\varphi_{\rem;\El{\varrho}} \Vert^2_{W^{\EkNoArg-1}_{\Ealpha{\varphi}-2} (\Boundtext)}
+\Vert
\YOp\tilde\varphi_{\rem;\El{\varrho}} \Vert^2_{W^{\EkNoArg}_{\Ealpha{\varphi}-3} (\Dtautext)}\nonumber\\
\lesssim{}&\timefunc^{-2}\left(\Vert
\YOp\tilde\varphi_{\rem;\El{\varrho}} \Vert^2_{W^{\EkNoArg-1}_{(\Ealpha{\varphi}+2)-2} (\Boundtext)}
+\Vert
\YOp\tilde\varphi_{\rem;\El{\varrho}} \Vert^2_{W^{\EkNoArg}_{(\Ealpha{\varphi}+2)-3} (\Dtautext)}\right)\nonumber\\
\lesssim{}& \ED{\varrho}^2\timefunc^{-2},
\label{eq:tildevarphiremainderinDtauYOp}
\end{align}
\end{subequations}
which follows from \eqref{eq:phiiTildeTransitionAndMorawetzExteriorSummed}, \eqref{eq:transitionfluxcontrolledbybulk},
and the fact that $r\geq \timefunc$ in the exterior region.
It then holds that, for any  $\timefunc\geq \timefunc_0$,
\begin{align}
\hspace{6ex}&\hspace{-6ex}\Vert
\LxiOp\tilde\varphi_{\rem;\El{\varrho}} \Vert^2_{W^{\EkNoArg-1}_{\Ealpha{\varphi}-2} (\Boundtext)}
+\Vert
\LxiOp\tilde\varphi_{\rem;\El{\varrho}} \Vert^2_{W^{\EkNoArg-1}_{\Ealpha{\varphi}-3} (\Dtautext)}\nonumber\\
\lesssim {}&\Vert
\YOp\tilde\varphi_{\rem;\El{\varrho}} \Vert^2_{W^{\EkNoArg-1}_{\Ealpha{\varphi}-2} (\Boundtext)}
+\Vert
\YOp\tilde\varphi_{\rem;\El{\varrho}} \Vert^2_{W^{\EkNoArg-1}_{\Ealpha{\varphi}-3} (\Dtautext)}\nonumber\\
&+\Vert
\tilde\varphi_{\rem;\El{\varrho}} \Vert^2_{W^{\EkNoArg}_{\Ealpha{\varphi}-2-2} (\Boundtext)}
+\Vert
\tilde\varphi_{\rem;\El{\varrho}} \Vert^2_{W^{\EkNoArg}_{\Ealpha{\varphi}-2-3} (\Dtautext)}\nonumber\\
\lesssim{}& \ED{\varrho}^2 \timefunc^{-2}.
\end{align}
Hence, together with \eqref{eq:tildevarphiremainderinDtauzeroorder}, this implies,
for any  $\Lxitimes\in \{0,1\}$ and $\timefunc\geq \timefunc_0$,
\begin{align}
\Vert \LxiOp^{\Lxitimes}\tilde\varphi_{\rem;\El{\varrho}} \Vert_{W^{\EkNoArg-\Lxitimes}_{\Ealpha{\varphi}-2} (\Boundtext)}^2
+ \norm{\LxiOp^{\Lxitimes}\tilde{\varphi}_{\rem;\El{\varrho}}}_{W^{\EkNoArg-\Lxitimes}_{\Ealpha{\varphi}-3}(\Dtautext)}^2
\lesssim{}&
\ED{\varrho}^2 \timefunc^{-2\Lxitimes}.
\label{eq:tildevarphiremainderinDtau}
\end{align}
For any $\timefunc\geq \timefunc_0 +1$, there exists an $i\in \Naturals$ such that $\timefunc \in [\timefunc_0 + 2^i, \timefunc_0+ 2^{i+1}]$. We apply the mean-value principle to the first term of \eqref{eq:tildevarphiremainderinDtau}, with the time interval replaced by $ [\timefunc_0 + 2^i, \timefunc_0+ 2^{i+1}]$,  to conclude there exists a $\timefunc_{(i)}\in [\timefunc_0 + 2^i, \timefunc_0+ 2^{i+1}]$ such that
\begin{align}
\int_{\Sphere}\absRescaled{\timefunc_{(i)}^{\frac{\Ealpha{\varphi}-2}{2}}\tilde\varphi_{\rem;\El{\varrho}} (\timefunc_{(i)},\timefunc_{(i)},\omega)}{\EkNoArg}^2 \diTwoVol \lesssim \ED{\varrho}^2 (\timefunc_0 + 2^i)^{-1}\lesssim \ED{\varrho}^2\timefunc^{-1}.
\end{align}
From fundamental theorem of calculus,
\begin{align}
\int_{\Sphere}\absRescaled{\timefunc^{\frac{\Ealpha{\varphi}-2}{2}}\tilde\varphi_{\rem;\El{\varrho}} (\timefunc,\timefunc,\omega)}{\EkNoArg-1}^2 \diTwoVol
\lesssim {}&\int_{\Sphere}\absRescaled{\timefunc_{(i)}^{\frac{\Ealpha{\varphi}-2}{2}}\tilde\varphi_{\rem;\El{\varrho}} (\timefunc_{(i)},\timefunc_{(i)},\omega)}{\EkNoArg-1}^2 \diTwoVol \nonumber\\
&
+\Vert \LxiOp\tilde\varphi_{\rem;\El{\varrho}} \Vert_{W^{\EkNoArg-1}_{\Ealpha{\varphi}-2} (\Boundtext)}^2
+\Vert \tilde\varphi_{\rem;\El{\varrho}} \Vert_{W^{\EkNoArg}_{\Ealpha{\varphi}-2-2} (\Boundtext)}^2\nonumber\\
\lesssim{}& \ED{\varrho}^2\timefunc^{-1}.
\end{align}
Similarly  we have for $\timefunc\in[\timefunc_0,\timefunc_0+1]$ that
$\int_{\Sphere}\absRescaled{\tilde\varphi_{\rem;\El{\varrho}} (\timefunc,\timefunc,\omega)}{\EkNoArg-1}^2 \diTwoVol
\lesssim  \ED{\varrho}^2 $. Therefore, for any $\timefunc \geq \timefunc_0$,
\begin{align}
\int_{\Sphere}\absRescaled{\timefunc^{\frac{\Ealpha{\varphi}-1}{2}}\tilde\varphi_{\rem;\El{\varrho}} (\timefunc,\timefunc,\omega)}{\EkNoArg-1}^2 \diTwoVol
\lesssim {}& \ED{\varrho}^2.
\end{align}
Notice from \eqref{eq:SobolevOnStaut:intermediatestep} and \eqref{eq:tildevarphiremainderinDtauzeroorder}, we have in the exterior region that, for any $\timefunc \geq \timefunc_0$,
\begin{align}
\hspace{4ex}&\hspace{-4ex}\int_{\Sphere}\absRescaled{r^{\frac{\Ealpha{\varphi}-1}{2}}\tilde\varphi_{\rem;\El{\varrho}}(\timefunc, r,\omega)}{\EkNoArg-1}^2 \diTwoVol \nonumber\\
\lesssim{}&\int_{\Sphere}\absRescaled{r^{\frac{\Ealpha{\varphi}-1}{2}}\tilde\varphi_{\rem;\El{\varrho}}(\timefunc, \timefunc,\omega)}{\EkNoArg-1}^2 \diTwoVol
 +\norm{\tilde{\varphi}_{\rem;\El{\varrho}}}_{W^{\EkNoArg}_{\Ealpha{\varphi}-2}(\Stautext)}^2\nonumber\\
 \lesssim{}& \ED{\varrho}^2 .
\end{align}
From lemma \ref{lem:spherica-Sobolev}, the following pointwise estimates then hold for any $\timefunc \geq \timefunc_0$ in the exterior region
\begin{align}
\absRescaled{\tilde\varphi_{\rem;\El{\varrho}}}{\EkNoArg-3}^2 \lesssim{}& \ED{\varrho}^2 r^{-(\Ealpha{\varphi}-1)}.
\end{align}

\step{Treat $\tilde{\varphi}_{\homo;{\El{\varrho}}}$}
 Given a point $p\in \Dtauext$  with coordinates $(\timefunc,r,\omega)$, let $\gamma$ denote the integral curve along $Y$ through the point. The value of $\LxiOp^{\Lxitimes}\tilde{\varphi}_{\homo;{\El{\varrho}}}$ is constant along $\gamma$, so its value at $p$ is equal to its value at the intersection of $\gamma$ and $\Sigma_{\initial}$. Since the rates of change of $-\timefunc$ and $r$ are comparable along $\gamma$, it follows that the coordinates $(\tilde\timefunc,\tilde{r},\tilde\omega)$ of the intersection of $\gamma$ and $\Sigma_{\initial}$ satisfy $-\tilde\timefunc \sim \tilde{r} \sim \timefunc+2r$. From the decay rates for $\LxiOp^{\Lxitimes}\varphi_{\initial}$ and  $\LxiOp^{\Lxitimes}\tilde\varphi_i$, and since $\Ealpha{\varrho}-3 <2\El{\varrho}+1=2\El{\varphi}+3$ and $\Ealpha{\varphi}=\Ealpha{\varrho}-2-\delta$, one finds for any $\Lxitimes\in \{0,1\}$, 
\begin{align}
\int_{\Sphere}\abs{\LxiOp^{\Lxitimes}\tilde{\varphi}_{\homo;{\El{\varrho}}}}^2\diTwoVol
\lesssim{}& (\timefunc +2r)^{-\Ealpha{\varrho}+3-2\Lxitimes}
\inipointwise{\Ek{\varphi}}{\Ealpha{\varrho}-3}(\varphi)
+\ED{\varrho}^2 (\timefunc +2r)^{-\Ealpha{\varrho}+3-2\Lxitimes}\nonumber\\
\lesssim{}& (\timefunc +2r)^{-\Ealpha{\varphi}+1-\delta-2\Lxitimes}
\left(\inipointwise{\Ek{\varphi}}{2\El{\varphi}+3}(\varphi)
+\ED{\varrho}^2 \right) .
\end{align}
 The quantity $\inipointwise{\Ek{\varphi}}{2\El{\varphi}+3}(\varphi)$  in the above estimates can be replaced by $\inienergy{\Ek{\varphi}+1}{2\El{\varphi}+3}(\varphi)$ from lemma \ref{lem:Ik+1alphaenergydominatesPkalphapointwise}, implying that
\begin{align}
\int_{\Sphere}\abs{\LxiOp^{\Lxitimes}\tilde{\varphi}_{\homo;{\El{\varrho}}}}^2\diTwoVol
\lesssim{}&(\timefunc +2r)^{-\Ealpha{\varphi}+1-\delta-2\Lxitimes}
\ED{\varphi}^2 .
\end{align}
Applying a $Y$ derivative to $\LxiOp^{\Lxitimes}\tilde{\varphi}_{\homo;{\El{\varrho}}}$ gives zero. Differentiating along $rV$ or applying $\hedt$ or $\hedtp$ corresponds to differentiating along a vector of length $r$ on the initial data. Since derivatives decay one power faster, this means that $\int_{\Sphere}\absRescaled{\LxiOp^{\Lxitimes}\tilde{\varphi}_{\homo;{\El{\varrho}}}}{\EkNoArg-\Lxitimes}^2\diTwoVol$ decays at the same rate, although the constant depends on the $\EkNoArg$ norm, i.e.  for any $\Lxitimes\in \{0,1\}$,
\begin{align}
\label{eq:tildephiHomoDecay}
\int_{\Sphere}\absRescaled{\LxiOp^{\Lxitimes}\tilde{\varphi}_{\homo;{\El{\varrho}}}}{\EkNoArg-\Lxitimes}^2\diTwoVol 
\lesssim{}& (\timefunc +2r)^{-\Ealpha{\varphi}+1-\delta-2\Lxitimes}
\ED{\varphi}^2  .
\end{align}
 As with the $\LxiOp^{\Lxitimes}\tilde\varphi_i$, since $\Ealpha{\varphi}<\Ealpha{\varrho}-2$, one finds that, for any  $\Lxitimes\in \{0,1\}$ and  $\timefunc'\geq\timefunc_0$,
\begin{subequations}
\begin{align}
\norm{\LxiOp^{\Lxitimes}\tilde\varphi_{\homo;{\El{\varrho}}}}_{W^{\EkNoArg-\Lxitimes}_{\Ealpha{\varphi}-3}(\Dtautext)}^2
\lesssim{}& \int_{\timefunc}^\infty \int_{\timefunc'}^\infty \ED{\varphi}^2  r^{\Ealpha{\varphi}-3} (\timefunc' +2r)^{-\Ealpha{\varphi}+1-\delta-2\Lxitimes}   \di r\di\timefunc' \nonumber\\
\lesssim{}& \ED{\varphi}^2  \timefunc^{-\delta-2\Lxitimes} , \\
\norm{\LxiOp^{\Lxitimes}\tilde\varphi_{\homo;{\El{\varrho}}}}_{W^{\EkNoArg-\Lxitimes}_{\Ealpha{\varphi}-2}(\Boundtext)}^2
\lesssim{}& \int_{\timefunc}^\infty \ED{\varphi}^2  r^{\Ealpha{\varphi}-2} (3r)^{-\Ealpha{\varphi}+1-\delta-2\Lxitimes}\di\timefunc' \nonumber\\
\lesssim{}& \ED{\varphi}^2  \timefunc^{-\delta-2\Lxitimes},\\
\norm{\LxiOp^{\Lxitimes}\tilde\varphi_{\homo;{\El{\varrho}}}}_{W^{\EkNoArg-\Lxitimes}_{\Ealpha{\varphi}-2}(\Stautext)}^2
\lesssim{}& \int_{\timefunc}^\infty \ED{\varphi}^2  r^{\Ealpha{\varphi}-2} (\timefunc +2r)^{-\Ealpha{\varphi}+1-\delta-2\Lxitimes}\di r \nonumber\\
\lesssim{}& \ED{\varphi}^2  \timefunc^{-\delta-2\Lxitimes},\\
\norm{\LxiOp^{\Lxitimes}\tilde\varphi_{\homo;{\El{\varrho}}}}_{W^{\EkNoArg-\Lxitimes}_{\Ealpha{\varphi}-3}(\Dtauearly)}^2
\lesssim{}& \int_{-\infty}^{\timefunc_0} \int_{\abs{\timefunc'}}^\infty \ED{\varphi}^2  r^{\Ealpha{\varphi}-3} (\timefunc' +2r)^{-\Ealpha{\varphi}+1-\delta-2\Lxitimes}\di r \di \timefunc' \nonumber\\
\lesssim{}& \ED{\varphi}^2  .
\end{align}
\end{subequations}

\step{Treat $\varphi_{\rem;\El{\varphi}}$}
One can combine the results for the $\{\LxiOp^{\Lxitimes}\tilde{\varphi}_i\}_{i=0}^{\El{\varrho}}$, for $\LxiOp^{\Lxitimes}\tilde{\varphi}_{\rem;\El{\varrho}}$, and for $\LxiOp^{\Lxitimes}\tilde{\varphi}_{\homo;{\El{\varrho}}}$. Combining these bounds with uniform bounds on $b_{j-i-1}(R)$, and noticing $\El{\varphi}=\El{\varrho}-1$, one finds, for any $\Lxitimes\in \{0,1\}$ and $\timefunc\geq\timefunc_0$,
\begin{subequations}
\begin{align}
\norm{\LxiOp^{\Lxitimes}\varphi_{\rem;\El{\varphi}}}_{W^{\EkNoArg-\Lxitimes}_{\Ealpha{\varphi}-2}(\Boundtext)}^2
+\norm{\LxiOp^{\Lxitimes}\varphi_{\rem;\El{\varphi}}}_{W^{\EkNoArg-\Lxitimes}_{\Ealpha{\varphi}-3}(\Dtautext)}^2 
\lesssim{}& \ED{\varphi}^2  \timefunc^{-2\Lxitimes} ,
\label{eq:phiRemainderTransitionAndMorawetz}\\
\norm{\varphi_{\rem;\El{\varphi}}}_{W^{\EkNoArg}_{\Ealpha{\varphi}-3}(\Dtauearly)}^2
\lesssim {}&  \ED{\varphi}^2  .
\end{align}
\end{subequations}
From lemma \ref{lem:spherica-Sobolev} and rewriting $r\VOp$ using equation \eqref{eq:ddRasVOpstructure}, the estimates of the $L^2(\Sphere)$ norm of $\{\LxiOp^{\Lxitimes}\tilde{\varphi}_i\}_{i=0}^{\El{\varrho}}$ and  $\LxiOp^{\Lxitimes}\tilde{\varphi}_{\homo;{\El{\varrho}}}$ in inequalities \eqref{eq:tildephiinductionassumpEstimate} and \eqref{eq:tildephiHomoDecay} imply that in the exterior region, for any $\timefunc\geq \timefunc_0$, $i\in\{\Em{\varrho}+1,\ldots,{\El{\varrho}}\}$, and $n\in\Naturals$,
\begin{align}
\absRescaled{  R^{n}\tilde\varphi_i}{\EkNoArg-2}^2
\lesssim{}&\sum_{j=0}^{\EkNoArg-2}\absScri{  (r\VOp)^j (R^{n}\tilde\varphi_i)}{\EkNoArg-2-j}^2 \nonumber\\
\lesssim{}&\sum_{j=0}^{\EkNoArg-2}\absScri{  (R\partial_R)^j (R^{n})\tilde\varphi_i}{\EkNoArg-2-j}^2 \nonumber\\
\lesssim{}&\int_{\Sphere}R^{2n}\absScri{\tilde\varphi_{i}(\timefunc,\omega)}{\EkNoArg}^2\diTwoVol \nonumber\\
\lesssim{}&\ED{\varrho}^2 R^{2n}\timefunc^{2i-\Ealpha{\varrho}+3},
\label{eq:pointwisetildevarphii}\\
\absRescaled{\tilde{\varphi}_{\homo;{\El{\varrho}}}}{\EkNoArg-2}^2
\lesssim{}& \ED{\varphi}^2  R^{\Ealpha{\varrho} -3} .
\label{eq:pointwisetildevarphihomogeneous}
\end{align}
Together with the pointwise estimates of $\tilde{\varphi}_{\rem;\El{\varrho}}$ and the uniform boundedness of $b_{\El{\varrho}-i-1}(R)$, it follows that in the exterior region, for any  $\timefunc\geq \timefunc_0$,
\begin{subequations}
\begin{align}
\text{for}\  \Em{\varrho}\leq{}& \El{\varrho}, &\absRescaled{\varphi_{\rem;\El{\varphi}}}{\EkNoArg-3}^2 \lesssim{}& \ED{\varphi}^2 r^{-2\El{\varrho}}\timefunc^{-{\Ealpha{\varphi}+1}+2\El{\varrho}},\\
\text{for}\  \Em{\varrho}={}& \El{\varrho}+1, &\absRescaled{\varphi_{\rem;\El{\varphi}}}{\EkNoArg-3}^2 \lesssim{}& \ED{\varphi}^2 r^{-{\Ealpha{\varphi}+1}}.
\end{align}
\end{subequations}

\step{Treat $\varphi$}
Combining inequalities \eqref{eq:phiiTransitionAndMorawetz} and \eqref{eq:phiRemainderTransitionAndMorawetz} for the $\{\LxiOp^{\Lxitimes}{\varphi}_i\}_{i=0}^{\El{\varphi}}$ and $\LxiOp^{\Lxitimes}{\varphi}_{\rem;\El{\varphi}}$ gives, for any $\Lxitimes\in \{0,1\}$ and $\timefunc\geq \timefunc_0$,
\begin{align}
\norm{\LxiOp^{\Lxitimes}\varphi}_{W^{\EkNoArg-\Lxitimes}_{\Ealpha{\varphi}-2} (\Boundtext)}^2
\lesssim{}&\sum_{i=0}^{\El{\varphi}}\norm{r^{-i}\LxiOp^{\Lxitimes}\varphi_i}_{W^{\EkNoArg-\Lxitimes}_{\Ealpha{\varphi}-2} (\Boundtext)}^2
+\norm{\varphi_{\rem;\El{\varphi}}}_{W^{\EkNoArg-\Lxitimes}_{\Ealpha{\varphi}-2} (\Boundtext)}^2\nonumber\\
\lesssim{}& \ED{\varphi}^2  \timefunc^{-2\Lxitimes} .
\end{align}
For any $i\leq \Em{\varrho}-1$, $\varphi_i=0$, and for any $\Em{\varrho}\leq i\leq \El{\varrho}-1=\El{\varphi}$, we have from \eqref{eq:pointwisetildevarphii} with $n=i$ and the uniform boundedness of $\{a_i\}_{i=0}^{\El{\varphi}}$ that in the exterior region, for any $\timefunc\geq \timefunc_0$,
\begin{align}
\absRescaled{  R^i\varphi_i}{\EkNoArg-2}^2
\lesssim{}\sum_{j=0}^{i}\absRescaled{  R^i\tilde\varphi_j}{\EkNoArg-2}^2
\lesssim\sum_{j=\Em{\varphi}+1}^{i} r^{-2i}\timefunc^{2j-\Ealpha{\varrho} +3}
\lesssim \ED{\varrho}^2 R^{2i}\timefunc^{2i-\Ealpha{\varrho} +3}.
\end{align}
We have then from the pointwise estimate for $\varphi_{\rem;\El{\varphi}}$, the fact that the $\{\varphi_i\}_{i=0}^{\Em{\varrho}-1}$ vanish and the above pointwise estimates for $\{\varphi_i\}_{i=\Em{\varrho}}^{\El{\varphi}}$ that in the exterior region, for any $\timefunc\geq \timefunc_0$,
\begin{subequations}
\begin{align}
\text{for}\  \Em{\varrho}\leq{}& \El{\varrho}, &\absRescaled{\varphi}{\EkNoArg-3}^2 \lesssim{}& \ED{\varphi}^2 r^{-2\Em{\varrho} }\timefunc^{-\Em{\varphi}+1+2\Ealpha{\varrho} },\\
\text{for}\ \Em{\varrho}={}& \El{\varrho}+1, &\absRescaled{\varphi}{\EkNoArg-3}^2 \lesssim{}& \ED{\varphi}^2 r^{-\Ealpha{\varphi}+1} .
\end{align}
\end{subequations}
Therefore, we conclude that  $\varphi$ has a $(\Ek{\varrho},\El{\varrho},\Em{\varphi},\Ealpha{\varphi},\ED{\varphi}^2)$ expansion, and, for any  $\Lxitimes\in \{0,1\}$ and $\timefunc\geq \timefunc_0$, the estimates \eqref{eq:LxiOpvarphitransitionregionesti} and \eqref{eq:variouspointwiseforvarphigeneral} hold true. 
\end{steps}
\end{proof}

\begin{lemma}[Transformations of expansions]
\label{lem:transformationofexpansion}
Let $\EkNoArg,\ElNoArg,\EmNoArg\in\Naturals$ be such that $0\leq \EmNoArg\leq \ElNoArg+1$. Let $\EalphaNoArg$ be such that $2\ElNoArg+3<\EalphaNoArg<2\ElNoArg+4$. Let $\EDNoArg>0$. Let $\varrho$ be a spin-weighted scalar.

\begin{enumerate}
\item If $0\leq \EkNoArg'\leq \EkNoArg$, $0\leq \EmNoArg ' \leq \EmNoArg $ and $\varrho$ has a $(\EkNoArg,\ElNoArg,\EmNoArg,\EalphaNoArg,\EDNoArg^2)$ expansion, then $\varrho$ has a $(\EkNoArg',\ElNoArg,\EmNoArg',\EalphaNoArg,\EDNoArg^2)$ expansion.
\label{pt:trivialEmbedding}
\item If $\varrho_1$ and $\varrho_2$ both have $(\EkNoArg,\ElNoArg,\EmNoArg,\EalphaNoArg,\EDNoArg^2)$ expansions, then $\varrho_1+\varrho_2$ has a $(\EkNoArg,\ElNoArg,\EmNoArg,\EalphaNoArg,\EDNoArg^2)$ expansion.
\label{pt:sumsOfExpansions}
\item Let $n\in\Integers$, $n+ \EmNoArg\geq 0$, and $n+\ElNoArg\geq 0$. Let $f$ be a homogeneous rational function of $r$, $\sqrt{r^2+a^2}$, $\kappa_1$, and $\bar{\kappa}_{1'}$ of degree $-n$ that has no singularities for $\rInv\in[0,\rCutOff^{-1}]$.
Then there is a constant $C_f>0$ such that if $\varrho$ has a $(\EkNoArg,\ElNoArg,\EmNoArg,\EalphaNoArg,\EDNoArg^2)$ expansion, then $f\varrho$ has a $(\EkNoArg,\ElNoArg+n,\EmNoArg+n,\EalphaNoArg+2n,C_f\EDNoArg^2)$ expansion.
\label{pt:reweightingExpansion}
\item  If $\varrho$ has a $(\EkNoArg,\ElNoArg,\EmNoArg,\EalphaNoArg,\EDNoArg^2)$ expansion and has spin-weight $s$, then $\tau\varrho$ and $\bar\tau'\varrho$ have $(\EkNoArg,\ElNoArg+2,\EmNoArg+2,\EalphaNoArg+4,\EDNoArg^2)$ expansions and have spin-weight $s+1$, and $\bar\tau\varrho$ and $\tau'\varrho$ have $(\EkNoArg,\ElNoArg+2,\EmNoArg+2,\EalphaNoArg+4,\EDNoArg^2)$ expansions and have spin-weight $s-1$.
\label{pt:respinExpansion}
\item If $\varrho$ has a $(\EkNoArg,\ElNoArg,\EmNoArg,\EalphaNoArg,\EDNoArg^2)$ expansion and has spin-weight $s$, then $\kappa_1\edt\varrho$ has a $(\EkNoArg-1,\ElNoArg,\EmNoArg,\EalphaNoArg,\EDNoArg^2)$ expansion and has spin-weight $s+1$, and $\bar\kappa_{1'}\edtp\varrho$ has a $(\EkNoArg-1,\ElNoArg,\EmNoArg,\EalphaNoArg,\EDNoArg^2)$ expansion and has spin-weight $s-1$.
\label{pt:hedtExpansion}
\end{enumerate}
\end{lemma}
\begin{proof}
If $\EkNoArg'\leq \EkNoArg$ and $\EmNoArg'\leq \EmNoArg$, then the condition to have a $(\EkNoArg,\ElNoArg,\EmNoArg,\EalphaNoArg,\EDNoArg^2)$ expansion is strictly stronger than the condition to have a $(\EkNoArg',\ElNoArg,\EmNoArg',\EalphaNoArg,\EDNoArg^2)$ expansion, so the former implies the latter, which implies point \ref{pt:trivialEmbedding}.

Point \ref{pt:sumsOfExpansions} follows directly from summing the expansions, summing the bounds, and noting the linearity in both the integrability condition  \eqref{eq:expansion:integrability} and the vanishing condition \eqref{eq:vanishingorderassumption}.

Now consider point \ref{pt:reweightingExpansion}. Observe that if $\varrho$ has a $(\EkNoArg,\ElNoArg,\EmNoArg,\EalphaNoArg,\EDNoArg^2)$ expansion, then $\vartheta=r^{-n}\varrho$ has a $(\EkNoArg,\ElNoArg+n,\EmNoArg+n,\EalphaNoArg+2n,\EDNoArg^2)$ expansion, where $\vartheta_i=0$ for $i\leq n+\EmNoArg$, $\vartheta_i=\varrho_{i-\EmNoArg}$ for $i>n+\EmNoArg$, and $\vartheta_{\rem;\ElNoArg+n}=r^{-n}\varrho_{\rem;\ElNoArg}$. Thus, it is sufficient to show that if $f$ is a homogeneous rational function of degree $0$ and $\varrho$ has a $(\EkNoArg,\ElNoArg,\EmNoArg,\EalphaNoArg,\EDNoArg^2)$ expansion, then $f\varrho$ has a $(\EkNoArg,\ElNoArg,\EmNoArg,\EalphaNoArg,\EDNoArg^2)$ expansion. Expanding $f$ as an order $\ElNoArg$ power series in $\rInv$ and multiplying the expansions for $f$ and $\varrho$ together, one obtains an order $\ElNoArg$ expansion for $f\varrho$. Because $f$ is rational with no singularities on $\rInv=0$, each of the expansion terms in $f$ are smooth functions of the spherical coordinates alone. Thus, the expansion terms for $f\varrho$ have the same decay and $\timefunc$-integrability conditions as $\varrho$. The remainder term for $f$ decays as $r^{-\ElNoArg-1}$. The remainder term for $f\varrho$ consists of products of expansion terms of $f$ and of $\varrho$, of expansion terms of $f$ and the remainder for $\varrho$, of the remainder for $f$ and the expansion terms of $\varrho$, and of the remainder term for $f$ and the remainder term for $\varrho$. The expansion terms for $f$ and the remainder are all homogeneous rational functions without singularities in the region under consideration and with a characteristic rate of decay. Since $f$ is $\timefunc$ independent, $\LxiOp(f\varrho)=f\LxiOp\varrho$ and similarly for higher derivatives. All four types of products in the expansion of $f\varrho$ will have bounded integrals for $\timefunc\leq\timefunc_0$ when integrated over $\Dtautext$. Thus, all the conditions for a $(\EkNoArg,\ElNoArg,\EmNoArg,\EalphaNoArg,\EDNoArg^2)$ expansion are satisfied.

In point \ref{pt:respinExpansion}, the claim about the spin weight follows from properties of products of spin-weighted quantities. The bounds can be calculated in the Znajek tetrad using the argument from the previous paragraph and the fact that, in the Znajek tetrad, $\tau$ and $\tau'$ is $a\sin\theta$ times a homogeneous rational function in $\kappa_1$ and $\bar\kappa_{1'}$ of degree $-2$.

Similarly, in point \ref{pt:hedtExpansion}, the claim about spin follows from the fact that $\kappa_1$ and $\bar\kappa_{1'}$ are spin-weight zero quantities, and $\edt$ and $\edtp$ are spin $+1$ and $-1$ operators. The bounds follow from the relations \eqref{eq:relationhedtprimeandedtprime} and \eqref{eq:relationhedtprimeandedtprime} that $\kappa_1\edt\varrho$ is a linear combination of $\hedt\varrho$, $\kappa_1^2\tau\LxiOp\varrho$, and $\kappa_1\tau\varrho$ and $\bar\kappa_{1'}\edtp\varrho$ is a linear combination of $\hedtp\varrho$, $\bar\kappa_{1'}\bar{\tau}\LxiOp\varrho$, and $\bar\kappa_{1'}\bar{\tau}\varrho$, and the fact that the operators $\hedt$, $\hedtp$, and $\LxiOp$ are in $\rescaledOps$, the number of which is measured by $\EkNoArg$.
\end{proof}

\subsection{Integration on \texorpdfstring{$\Scri^+$}{Scri+} and the Teukolsky-Starobinsky identity}
\label{sec:Scri-int}
In this subsection and the following one, we show that $\psibase[-2]$ has an expansion at infinity. This subsection focuses on showing that the leading-order terms in the expansion of $\psibase[-2]$ satisfies the integrability condition \eqref{eq:expansion:integrability} on null infinity. The following subsection treats the remaining decay conditions and bounds on the remainder terms. 
\begin{definition}[Taylor expansion at $\Scri^+$]
\label{def:scri-tayl}
Let $\psibase[-2], \psibase[+2]$ be as in definition \ref{def:linearizedEinsteinFields}. Working in the compactified hyperboloidal coordinate system $(\timefunc, R,\theta, \phi)$ and restricting to the Znajek tetrad,
let the spin-weighted scalars $A_i$, $i=0,\dots,3$, $B_0$ on $\Scri^+$ be the Taylor coefficients of $\psibase[-2], \psibase[+2]$ defined by
\begin{subequations}
\begin{align}
A_i ={}& \partial_R^i \psibase[-2] \big{|}_{\Scri^+}, \quad i=0,\dots,3,  \\
B_0 ={}& \psibase[+2] \big{|}_{\Scri^+}  ,
\end{align}
\end{subequations}
and let $A_{\rem;3}, B_{\rem;0}$ be the corresponding remainder terms such that
\begin{subequations} \label{eq:scriexp-radpsi}
\begin{align}
\psibase[-2]={}&\sum_{i=0}^3 \frac{R^i}{i!} A_i(\timefunc,\omega) + A_{\rem;3},\\
\psibase[+2]={}&B_0+ B_{\rem;0}.
\end{align}
\end{subequations}
\end{definition}

\begin{lemma}
\label{lem:intA=0}
Let $\mathbf{a}$ be a multiindex. With $A_i$, $i=0,\dots,3$, $B_0$ as in definition \ref{def:scri-tayl}, assume that
\begin{subequations}\label{eq:AB-limits}
\begin{align}\label{eq:B-limits}
\lim_{\timefunc\to -\infty}  \pt^j \ScriOps^{\mathbf{a}} B_{0} = \lim_{\timefunc\to \infty} \pt^j \ScriOps^{\mathbf{a}} B_{0} = 0, \quad j=0,\dots,4 \\
\intertext{and}
\label{eq:A-limits}
\lim_{\timefunc\to -\infty} \pt^j \ScriOps^{\mathbf{a}} A_{0} = \lim_{\timefunc \to \infty} \pt^j \ScriOps^{\mathbf{a}} A_{0} = 0, \quad j=0,\dots, 4.
\end{align}
\end{subequations}
Then with $\AntiDeriv$ defined as in definition \ref{def:AntiDerivOper},
\begin{equation}\label{eq:intA=0}
\lim_{\timefunc\to \infty}  \AntiDeriv^j  \ScriOps^{\mathbf{a}} A_{0} = 0, \quad j=1,\dots, 4.
\end{equation}
\end{lemma}
\begin{proof}
We first prove the statement for $\mathbf{a} = 0$. Passing to the Znajek tetrad,  we may replace $\LxiOp$ by $\partial_\timefunc$ for spin-weighted scalars on $\Scri^+$. Equation \eqref{eq:TSIpsihat} yields, after using the expression \eqref{eq:YOptRthetaphi} for $Y$ and taking the limit $R\rightarrow 0$, that on $\Scri^+$,
\begin{align} \label{eq:new-ABScri}
\hedt^{4}A_0={}&-3 M \partial_\timefunc (\bar{A}_0)
  - \sum_{k=1}^4\binom{4}{k}  \ringtau^k \hedt^{4-k} \partial_\timefunc{}^k A_0
  +  4  \partial_\timefunc^4 B_0.
\end{align}

Integrating \eqref{eq:new-ABScri} $j$ times from $\timefunc=-\infty$, we have by \eqref{eq:AB-limits}
\begin{align}
\hedt^{4}\AntiDeriv^j A_0={}&-3 M \AntiDeriv^j \partial_\timefunc (\bar{A}_0)
  - \sum_{k=1}^4\binom{4}{k}  \ringtau^k \hedt^{4-k} \AntiDeriv^j \partial_\timefunc^k A_0
  +  4  \AntiDeriv^j \partial_\timefunc^4 B_0, \quad j=1,\dots,4.
\end{align}
From definition \ref{def:AntiDerivOper} we have that for a function $f$ satisfying \eqref{eq:AB-limits},
\begin{align}
\partial_\timefunc \AntiDeriv f ={}& \AntiDeriv \partial_\timefunc  f = f .
\end{align}
For $j=1$ we have
\begin{align}
\label{eq:OneIntegrationTSI}
\hedt^{4}\AntiDeriv A_0={}&-3 M  \bar{A}_0
  - \sum_{k=1}^4\binom{4}{k}  \ringtau^k \hedt^{4-k} \partial_\timefunc^{k-1} A_0
  +  4   \partial_\timefunc^3 B_0 .
\end{align}
Recall that $A$, and hence $A_0$ has spin-weight $-2$. 
Acting on a spin-weighted spherical harmonic ${}_{-2} Y_{l m}$, we have
\begin{equation}
\hedt^4 {}_{-2} Y_{l m} 
=\frac{(l+2)!}{4(l-2)!} {}_{+2} Y_{l m} ,
\end{equation}
and hence, since we may restrict to considering $l \geq 2$, we find that the operator $\hedt^4$ has trivial kernel when acting on fields of spin-weight $-2$.
Taking the limit $\timefunc\to \infty$ on both sides of \eqref{eq:OneIntegrationTSI}, and after using \eqref{eq:AB-limits} and the fact that $\hedt^4$ has trivial kernel on spin-weighted functions on $\Sphere$ with spin-weight $-2$, this gives the statement for $j=1$.
For $j=2,\dots, 4$, the statement can be proven in a similar manner, using induction with $j=1$ as base. This proves the lemma for $\mathbf{a} = 0$.

We prove the lemma for $\mathbf{a} \ne 0$ by induction on $\abs{\mathbf{a}}$, with $\mathbf{a} = 0$ as base.   Thus, let $k \geq 1$ be an integer, and assume the lemma is proved for $\abs{\mathbf{a}} = k-1$. Applying $\ScriOps^{\mathbf{a}}$ to both sides of \eqref{eq:new-ABScri} yields
\begin{align} \label{eq:scriOpAB}
\hedt^4 \ScriOps^{\mathbf{a}} A_0 = {}& [ \hedt^4, \ScriOps^{\mathbf{a}} ] A_0 - 3M \partial_\timefunc (\ScriOps^{\mathbf{a}} \bar{A}_0 )
- \sum_{k=1}^4\binom{4}{k}  \ringtau^k \hedt^{4-k} \partial_\timefunc{}^k \ScriOps^{\mathbf{a}}A_0 \nonumber\\
& - \sum_{k=1}^4 \binom{4}{k} \partial_\timefunc^k [ \ScriOps^{\mathbf{a}}, \ringtau^k \hedt^{4-k} ] A_0 + 4 \partial_\timefunc^4 \ScriOps^{\mathbf{a}} B_0.
\end{align}
The commutators on the right-hand side of \eqref{eq:scriOpAB} can be evaluated by noting that $\partial_\timefunc$ commutes with $\hedt$ and making use of the identities \eqref{eq:tau0ids} and the commutation formula \eqref{eq:commutatorofhedtandhedtprime}. By the induction hypothesis, we have that each term on the right-hand side satisfies the assumptions of the lemma. Therefore, we can proceed as above and inductively prove \eqref{eq:intA=0} for $j=1,\dots,4$. This completes the proof of the lemma.
\end{proof}

\begin{lemma}
\label{lem:Ailim+infty}
Let $\mathbf{a}$ be a multiindex, and let the assumptions in lemma \ref{lem:intA=0} hold.
Assume that
\begin{equation}\label{eq:Ailim-infty}
\lim_{\timefunc \to -\infty} \ScriOps^{\mathbf{a}} A_i (\timefunc,\omega) = 0, \quad i=0,\dots,3.
\end{equation}
Then
\begin{equation}\label{eq:Ailim+infty}
\lim_{\timefunc\to \infty} \AntiDeriv^j \ScriOps^{\mathbf{a}} A_i(\timefunc,\omega) = 0, \quad i=0,\dots,3,  \quad j=0,\dots,4-i .
\end{equation}
\end{lemma}
\begin{proof}
We first consider the case $\mathbf{a} = 0$. Taylor expanding the Teukolsky equation \eqref{eq:TeukolskyRegular-2} at $\Scri^+$ and using \eqref{eq:scriexp-radpsi} gives a recursive set of equations for $\pt A_i$, $i=0,\dots,3$, which after passing to the Znajek tetrad takes the form
\begin{subequations} \label{eq:bigAsys}
\begin{align}
\partial_{\timefunc} A_{1}{}={}&2 A_{0}{}
 + 4 M \partial_{\timefunc} A_{0}{}
 + 2 \CInHyperboloids M^2 \partial_{\timefunc} \partial_{\timefunc} A_{0}{}
 + 2 a \partial_{\timefunc} \partial_{\phi} A_{0}{}
 -  \tfrac{1}{2} \TMESOp_{-2}(A_{0}{}),\\
\partial_{\timefunc} A_{2}{}={}&- M A_{0}{}
 + 3 A_{1}{}
 + 2 (4 - \CInHyperboloids) M^2 \partial_{\timefunc} A_{0}{}
 + 4 M \partial_{\timefunc} A_{1}{}
 + 2 \CInHyperboloids M^2 \partial_{\timefunc} \partial_{\timefunc} A_{1}{}
 + 4 M a \partial_{\timefunc} \partial_{\phi} A_{0}{}\nonumber\\
& + 2 a \partial_{\timefunc} \partial_{\phi} A_{1}{}
 + a \partial_{\phi} A_{0}{}
 -  \tfrac{1}{2} \TMESOp_{-2}(A_{1}{})
 + (16 M^3 - 4 M a^2 -  \tfrac{1}{6} H^{(3)}(0))\partial_{\timefunc} \partial_{\timefunc} A_{0}{},\\
\partial_{\timefunc} A_{3}{}={}&-3 a^2 A_{0}{}
 + 3 A_{2}{}
 + (16 M^2 - 2 a^2) \partial_{\timefunc} A_{1}{}
 + 4 M \partial_{\timefunc} A_{2}{}
 + 2 \CInHyperboloids M^2 \partial_{\timefunc} \partial_{\timefunc} A_{2}{}\nonumber\\
& + 4 M^2 a(4  - \CInHyperboloids) \partial_{\timefunc} \partial_{\phi} A_{0}{}
 + 8 M a \partial_{\timefunc} \partial_{\phi} A_{1}{}
 + 2 a \partial_{\timefunc} \partial_{\phi} A_{2}{}
 + 4 a \partial_{\phi} A_{1}{}
 -  \tfrac{1}{2} \TMESOp_{-2}(A_{2}{})\nonumber\\
& + (32 M^3 - 8 M a^2 -  \tfrac{1}{3} H^{(3)}(0))\partial_{\timefunc} \partial_{\timefunc} A_{1}{}
 + (16 M^3 - 4 \CInHyperboloids M^3 - 12 M a^2 + \tfrac{1}{6} H^{(3)}(0))\partial_{\timefunc} A_{0}{}\nonumber\\
& + (64 M^4 - 4 \CInHyperboloids^2 M^4 - 32 M^2 a^2 + 4 \CInHyperboloids M^2 a^2 -  \tfrac{1}{12} H^{(4)}(0))\partial_{\timefunc} \partial_{\timefunc} A_{0}{}.
\end{align}
\end{subequations}
The system \eqref{eq:bigAsys} is of the form
\begin{align} \label{eq:A=LA}
\partial_t A_i = \mathbf{L}_i{}^k A_k
\end{align}
where, by inspection, $\mathbf{L}_i{}^j$ is a strictly lower-triangular matrix of operators on $\Scri^+$ with entries which are linear combinations of symmetry operators of order up to two of the Teukolsky equation, i.e. $\TMESOp_{-2}$, $\partial_t^2, \partial_t \partial_\phi, \partial_t, \partial_\phi$ and constants. The coefficients are bounded constants and depend only on $M,a,\CInHyperboloids$, and the Taylor terms $H^{3}(0)$ and $H^{4}(0)$,  where $H$ is given by \eqref{eq:H(R)} and \eqref{eq:hprim-explicit}.

From \eqref{eq:A=LA} we get the recursion relation
\begin{equation}\label{eq:Arecur}
A_i(\timefunc,\omega) = \lim_{\timefunc\to -\infty} A_{i} (\timefunc,\omega) + \int_{-\infty}^\timefunc \sum_{k=0}^{i-1} \mathbf{L}_i{}^k A_k (\timefunc',\omega) \di \timefunc', \quad i=1,2,3.
\end{equation}
Lemma \ref{lem:intA=0} shows that the $i=0$ case of \eqref{eq:Ailim+infty} is valid.
We consider the case $i=1$. From \eqref{eq:Arecur} and \eqref{eq:Ailim-infty} we have
\begin{align}\label{eq:A1inTMEscri}
\lim_{\timefunc\to \infty} \AntiDeriv^j A_1 ={}& \int_{-\infty}^{\infty} \mathbf{L}_1{}^0 \AntiDeriv^j A_0 (\timefunc',\omega) \di \timefunc' .
\end{align}
From lemma \ref{lem:intA=0}, the right of \eqref{eq:A1inTMEscri} vanishes for $j=0,1,2,3$, yielding the $i=1$ case of \eqref{eq:Ailim+infty} is valid.
Repeating this argument proves the statement for $i=2,3$ in the case $\mathbf{a} = 0$.

For the general case, we use induction on $\abs{\mathbf{a}}$. Let $m$ be a positive integer, and assume the statement has been proved for multiindices of length $\abs{\mathbf{a}} \leq m-1$. Apply $\ScriOps^{\mathbf{a}}$ to both sides of the Teukolsky equation, and Taylor expand the result at $\Scri^+$. This yields a version of system \eqref{eq:bigAsys} for $\ScriOps^{\mathbf{a}} A_i$, $i=0,\dots,3$, which again has the form
\begin{align} \label{eq:lowertriang}
\partial_t \ScriOps^{\mathbf{a}} A_i = \mathbf{L}_{i}{}^k \ScriOps^{\mathbf{a}}A_k + [\ScriOps^{\mathbf{a}} , \mathbf{L}_{i}{}^k ] A_k .
\end{align}
The last term in \eqref{eq:lowertriang} can be expressed in terms of $\ScriOps^{\mathbf{b}} A_k$ with $\abs{\mathbf{b}} \leq m-1$ and $k<i$.  This means that we can argue as above and use the fact that the system $\mathbf{L}_{i}{}^k$ is strictly lower triangular, to get the statement for $\abs{\mathbf{a}} = m$. This completes the proof of the lemma.
\end{proof}

\subsection{Expansion for the spin-weight \texorpdfstring{$-2$}{-2} Teukolsky scalar}
\label{sec:-2Expansion}
Here, we complete the proof that $\psibase[-2]$ has an expansion, which we do in lemma \ref{lem:TMEHasAnExpansion}. The first lemma treats the early region as a preliminary case, since the results in section \ref{sec:spin-2TeukolskyEstimates} focused on late times. 
\begin{lemma}[Control of the Teukolsky scalar at and prior to $\Stauit$] \label{lem:idata}
\standardMinusTwoHypothesisDefinePsibasei. 
There is a regularity constant $K$ such that the following holds. Let $j,k\in\Naturals$ such that $k-j- K$ is sufficiently large.
\standardMinusTwoHypothesisBEAM.
Let $\NEWinienergyminustwo{k}$ be as in definition \ref{def:initialDataNormMinusTwoTest}, and
let $\iniEnergyMinusTwoForDescent{\ireg}$ be as in definition
\ref{def:generalSpinorInitialDataNorm}.
\begin{enumerate}
\item
\begin{align}
\NEWinienergyminustwo{\ireg-4}
+\sum_{i=0}^4\norm{\radpsi{\ii}}_{W^{\ireg-4}_{-1-}(\Dtauearly)}^2
\lesssim{}&
\iniEnergyMinusTwoForDescent{\ireg} .
\label{eq:psi-2:pastIntegrabilityNew}
\end{align}
\item For $\ii\in\{0,\ldots,4\}$, 
and for any $\timefunc\leq\timefunc_0$,
\begin{align}
\int_{\Sphere_{\timefunc,\infty}}\abs{\LxiOp^{\ij}\radpsi{\ii}}_{\ireg-\ij-K,\ScriOps}^2\diTwoVol
\lesssim{}& \langle\timefunc\rangle^{-9+2\ii-2\ij+}
\norm{\psibase[-2]}_{H^{\ireg}_{9}(\Stauini)}^2 .
\label{eq:psi-2:pastSphereDecay-NEW}
\end{align}
\end{enumerate}
\end{lemma}
\begin{proof}
Consider estimating norms on $\Staut$ by those on $\Stauini$ for $\timefunc\leq\timefunc_0$. The basic estimate on $\radpsi{\ii}$ in lemma \ref{lem:rpestimatesforallradpsii} can essentially be repeated. In particular, from lemma \ref{lem:GeneralrpSpinWaveFar} on spin-weighted wave equations in the early region, from the $5$-component system \eqref{eq:rescaledwavesys5eqs}, and from the relation between $\raddegphi{\ii}$ and $\radpsi{\ii}$ norms in lemma \ref{lem:equiavelenceOfradphiAndraddegphi}, it follows that there is a constant $\rCutOffInrpWave$ such that, for all $\alpha\in[\delta,2-\delta]$,  $\rCutOff\geq\rCutOffInrpWave$, and $\timefunc\leq \timefunc_0$,
\begin{align}
\sum_{i=0}^4&\left(
  \norm{r\VOp\radpsi{\ii}}_{W^{\ireg}_{\alpha-2}(\StautR)}^2
  +\norm{\radpsi{\ii}}_{W^{\ireg+1}_{-2}(\StautR)}^2
  +\norm{\radpsi{\ii}}_{W^{\ireg+1}_{\alpha-3}(\DtautearlyR)}^2
\right)\nonumber\\
\lesssim{}&
\sum_{i=0}^4
\norm{\radpsi{\ii}}_{H^{\ireg+1}_{\alpha-1}(\Stauini)}^2 \nonumber\\
&+ \sum_{i=0}^4 \norm{\radpsi{\ii}}_{W^{\ireg+1}_{0}(\DtautearlyRcR)}^2
+ \sum_{i=0}^4 \norm{\radpsi{\ii}}_{W^{\ireg+1}_{-\delta}(\StautRcR)}^2 \nonumber\\
&+\sum_{i=0}^4 \norm{M r^{-1}\radpsi{\ii}}_{W^{\ireg+1}_{\alpha-3}(\DtautearlyRc)}^2.
\label{eq:rpInTheEarlyRegion}
\end{align}
 To treat the last term on the right-hand side, note that we can take $\rCutOff$ sufficiently large such that  $\norm{Mr^{-1}\radpsi{\ii}}_{W^{\ireg+1}_{\alpha-3}(\DtautearlyR)}^2$ can be absorbed into the $\norm{\radpsi{\ii}}_{W^{\ireg+1}_{\alpha-3}(\DtautearlyR)}^2$ terms on the left,  leaving $\norm{Mr^{-1}\radpsi{\ii}}_{W^{\ireg+1}_{\alpha-3}(\DtautearlyRcR)}^2$.
For all $\timefunc\leq\timefunc_0$, since $\rCutOff$ is bounded, the terms
\begin{align}
\norm{\radpsi{\ii}}_{W^{\ireg+1}_{\alpha-2}(\StautRcR)}^2, \quad  \norm{\radpsi{\ii}}_{W^{\ireg+1}_{\alpha-3}(\DtautearlyRcR)}^2
\end{align}
can be bounded by a multiple of the initial norm
\begin{align}
\norm{\radpsi{\ii}}_{H^{\ireg+1}_{\alpha-1}(\Stauini)}^2,
\end{align}
by standard exponential growth estimates for wave-like equations. Similarly, since $\Dtautearly\cap\{r\leq\rCutOff\}$ is bounded in spacetime, standard exponential growth estimates can be used to bound the energy on the upper boundary. Thus,
\begin{align}
\sum_{i=0}^4&\left(
  \norm{r\VOp\radpsi{\ii}}_{W^{\ireg}_{\alpha-2}(\Staut)}^2
  +\norm{\radpsi{\ii}}_{W^{\ireg+1}_{-2}(\Staut)}^2
  +\norm{\radpsi{\ii}}_{W^{\ireg+1}_{\alpha-3}(\Dtautearly)}^2
\right)\nonumber\\
\lesssim{}&
 \sum_{i=0}^4
\norm{\radpsi{\ii}}_{H^{\ireg+1}_{\alpha-1}(\Stauini)}^2 .
\label{eq:(0.9)}
\end{align}
In particular, with $\alpha=2-\delta$, and recalling the $\radpsi{\ii}$ are related via derivatives with an $r^2$ weight, but the
norms $\norm{\varphi}_{H^\ireg_\alpha(\Stauini)}^2$
are based on an $r^1$ weight for each derivative, \eqref{eq:(0.9)} for $\timefunc=\timefunc_0$ yields \begin{align}
\NEWinienergyminustwo{\ireg+1}
+\sum_{i=0}^4\norm{\radpsi{\ii}}_{W^{\ireg+1}_{-1-\delta}(\Dtauearly)}^2
\lesssim{}& \sum_{i=0}^4
\norm{\radpsi{\ii}}_{H^{\ireg+1}_{1-\delta}(\Stauini)}^2
\lesssim{}
\norm{\psibase[-2]}_{H^{\ireg+5}_{9-\delta}(\Stauini)}^2 . \label{eq:intermediatepsi-2}
\end{align}
Reindexing and using the notation introduced in definition \ref{def:generalSpinorInitialDataNorm}
gives \eqref{eq:psi-2:pastIntegrabilityNew}.

For $\timefunc$ negative and large in absolute value, it is possible to prove stronger estimates by combining the ideas in the proofs of decay for the spin-weight $-2$ Teukolsky scalar in section \ref{sec:spin-2TeukolskyEstimates} and of decay for the spin-weight $+2$ Teukolsky scalar as $\timefunc\rightarrow-\infty$ in the proof of theorem \ref{thm:DecayEstimatesSpin+2}. In particular, for $\ii\in\{2,3,4\}$, one follows the proof of lemma \ref{lem:EnerDecPsi234}, and, for $\ii\in\{0,1\}$, one follows that of theorem \ref{thm:ImproEstipsi01} in the case of the exterior region.

Let $\norm{\radpsi{i}}_{F^\ireg(\Scritini)}^2$ be as in definition \ref{def:generalSpinorScriNorm}. As in the proof of theorem \ref{thm:DecayEstimatesSpin+2}, let $r(\timefunc)$ denote the value of $r$ corresponding to the intersection of $\Stauini$ and $\Staut$, and recall that, for $\rCutOff$ fixed and $\timefunc$ sufficiently negative, we have that $r(\timefunc) > \rCutOff$ and $r(\timefunc) \sim -\timefunc$. Recall that the proof of \eqref{eq:rpInTheEarlyRegion} and \eqref{eq:(0.9)} is based on lemma \ref{lem:GeneralrpSpinWaveFar} and in particular an application of Stokes' theorem, and hence we may add a term of the form
\begin{align}
\sum_{i=0}^4 \norm{\radpsi{i}}_{F^\ireg(\Scritini)}^2
\end{align}
 on the left of \eqref{eq:(0.9)} and replace $\Stauini$ by $\Stauini^{r(\timefunc)} = \Stauini \cap \{r > r(\timefunc)\}$. Further, the resulting inequality holds with
the summation over $\ii\in\{0,\ldots,4\}$ replaced by summation over $\iip\in\{0,\ldots,\ii\}$ for $\ii\in\{2,3,4\}$.
We now have, for $\alpha\in[\delta,2-\delta]$,
\begin{align}
\sum_{\iip=0}^\ii \norm{\radpsi{\iip}}_{F^\ireg(\Scritini)}^2
\lesssim{}& \sum_{\iip=0}^\ii
\norm{\radpsi{\iip}}_{H^{\ireg+1}_{\alpha-1}(\Stauini^{r(\timefunc)})}^2 \nonumber \\
\lesssim{}&
\norm{\psibase[-2]}_{H^{\ireg+1+\ii}_{\alpha+2\ii-1}(\Stauini^{r(\timefunc)})}^2 \label{eq:(0.14)} .
\end{align}
With $\alpha = 2-\delta$ and $\ii=4$, we get the weight $9-\delta$ as in \eqref{eq:psi-2:pastIntegrabilityNew}.
On the other hand, taking $\alpha=\delta$, one finds, for $i \in \{2,3,4\}$,
\begin{align}
\sum_{\iip=0}^\ii \norm{\radpsi{\iip}}_{F^\ireg(\Scritini)}^2
\lesssim{}&
\norm{\psibase[-2]}_{H^{\ireg+1+\ii}_{\delta+2\ii-1}(\Stauini\cap\{r>r(\timefunc)\})}^2
\nonumber\\
\lesssim{}& r(\timefunc)^{-10+2\ii+2\delta}
\norm{\psibase[-2]}_{H^{\ireg+1+\ii}_{9-\delta}(\Stauini \cap\{r > r(t)\})}^2 \nonumber \\
\lesssim{}&|\timefunc|^{-10+2\ii+2\delta} \norm{\psibase[-2]}_{H^{\ireg+1+\ii}_{9-\delta}(\Stauini)}^2. \label{eq:(0.15)}
\end{align}
The case $\ii = 2$ of \eqref{eq:(0.15)} gives, after renaming $\iip$ to $\ii$,  the estimate
\begin{align}
\norm{\radpsi{\ii}}_{F^\ireg(\Scritini)}^2
\lesssim{}&|\timefunc|^{-6+2\delta} \norm{\psibase[-2]}_{H^{\ireg+3}_{9-\delta}(\Stauini)}^2, \quad \text{for $\ii\in\{0,1,2\}$}.
\end{align}
Since $\LxiOp$ commutes with the Teukolsky equation, but each derivative is weighted with $r^{-1}$ in $L^2$ in the definition of $\norm{\psibase[-2]}_{H^{\ireg+1+\ii}_{9-\delta}(\Stauini)}$,
one obtains an improved decay rate for $\LxiOp^\ij\radpsi{\ii}$,
\begin{align} \label{eq:xi^j-est}
\sum_{i=0}^2\norm{\LxiOp^\ij \radpsi{\ii}}_{F^\ireg(\Scritini)}^2
\lesssim{}& r(\timefunc)^{-6+2\delta}
\norm{\LxiOp^\ij\psibase[-2]}_{H^{\ireg+3}_{9-\delta}(\Stauini \cap\{r > r(t)\})}^2 \nonumber \\
\lesssim{}&|\timefunc|^{-6-2\ij +2\delta} \norm{\psibase[-2]}_{H^{\ireg+\ij + 3}_{9-\delta}(\Stauini)}^2, \quad \text{for $j \in \Naturals$.}
\end{align}
Restricting the system \eqref{eq:Psi01EllipExt-explicit} to $\Scri^+$, using \eqref{eq:Psi01EllipExt}, and taking $\rInv=0$, one finds a system of the form
\begin{align}\label{eq:LAB}
L \begin{pmatrix} \radpsi{0} \\ \radpsi{1} \end{pmatrix} = A \LxiOp \begin{pmatrix} \radpsi{0} \\ \radpsi{1} \end{pmatrix} + B \begin{pmatrix} 0 \\ \LxiOp \radpsi{2} \end{pmatrix} ,
\end{align}
with
\begin{align}
L  = \begin{pmatrix} L^{(0)} & 0 \\
6  + 6 \tfrac{a}{M} \LetaOp  & L^{(1)} \end{pmatrix}
\end{align}
and
\begin{align}
L^{(i)} = 2 \hedt \hedtp -2(2+i), \quad i=0,1
\end{align}
and where $A$ is a matrix of operators of maximal order $1$ involving $\LxiOp, \LetaOp$, and constants, and with bounded, $\timefunc$-independent coefficients, and where $B \in \Reals$. In particular, the first row of system \eqref{eq:LAB} is of the form
\begin{align} \label{eq:LAB0}
L^{(0)} \radpsi{0} = A^{(00)} \LxiOp \radpsi{0} + A^{(01)} \LxiOp \radpsi{1}
\end{align}
where $A^{(00)}, A^{(01)}$ are operators of maximal order $1$ involving $\LxiOp$, $\LetaOp$ and constants, with bounded $\timefunc$-independent coefficients.

The operators $L^{(i)}$, $i=0,1$ are invertible on Sobolev spaces on the cross-sections
$\Sphere_\tau = \Scri^+\cap \{\timefunc = \tau\}$. Thus, since $L$ is lower triangular and the off-diagonal term has maximal order $1$ with the first order part involving only derivatives tangent to $S_\timefunc$, we find that $L$ is invertible on Sobolev spaces on $\Sphere_\timefunc$. In particular, we have
\begin{align}
\left \Vert \begin{pmatrix} \radpsi{0} \\ \radpsi{1} \end{pmatrix} \right \Vert_{L^2(\Sphere_\timefunc)} \lesssim
\left \Vert L \begin{pmatrix} \radpsi{0} \\ \radpsi{1} \end{pmatrix} \right \Vert_{L^2(\Sphere_\timefunc)} .
\end{align}
This estimate yields corresponding estimates for the semi-norms $F^k(\Scritini)$. Using \eqref{eq:xi^j-est} and \eqref{eq:LAB}, this gives
\begin{align} \label{eq:xi^j-est01}
\norm{\LxiOp^\ij \radpsi{\ii}}_{F^{\ireg-\ij-K}(\Scritini)}^2
\lesssim{}&|\timefunc|^{-8-2\ij +2\delta} \norm{\psibase[-2]}_{H^{\ireg}_{9-\delta}(\Stauini)}^2, \quad \text{for $i \in \{0,1\}$, $j\in \Naturals$.}
\end{align}
Finally, using \eqref{eq:LAB0} and \eqref{eq:xi^j-est01} gives
\begin{align} \label{eq:xi^j-est0}
\norm{\LxiOp^\ij \radpsi{0}}_{F^{\ireg-\ij-K}(\Scritini)}^2
\lesssim{}&|\timefunc|^{-10-2\ij +2\delta} \norm{\psibase[-2]}_{H^{\ireg}_{9-\delta}(\Stauini)}^2, \quad \text{for $j\in \Naturals$.}
\end{align}
In particular, with $j=0$, we have
\begin{align} \label{eq:xi^j-est0,j=0}
\norm{\radpsi{0}}_{F^{\ireg-K}(\Scritini)}^2
\lesssim{}&|\timefunc|^{-10 +2\delta} \norm{\psibase[-2]}_{H^{\ireg}_{9-\delta}(\Stauini)}^2 .  \end{align}
Reindexing and using $\radpsi{0} = \psibase[-2]$ gives
\begin{align} \label{eq:xi^j-est0,final}
\norm{\LxiOp^j \psibase[-2]}_{F^{\ireg-\ij-K}(\Scritini)}^2
\lesssim{}&|\timefunc|^{-10 -2j+2\delta} \norm{\psibase[-2]}_{H^{\ireg}_{9-\delta}(\Stauini)}^2 .
\end{align}

From the fundamental theorem of calculus, the Cauchy-Schwarz inequality, and the definition of the $F^k$ norm, one finds
\begin{align}
\norm{\LxiOp^{j+1} \psibase[-2][i]}_{W^{\ireg-j-K}(\Sphere)}^2
\lesssim{}& \int_{\Scritini} \absHighOrder{\LxiOp^{j+1} \psibase[-2][i]}{\ireg-j-K}{\sphereOps} \absHighOrder{\LxiOp^{j+2} \psibase[-2][i]}{\ireg-j-K}{\sphereOps}\diThreeVol \nonumber\\
\lesssim{}&\norm{\LxiOp^j \psibase[-2][i]}_{F^{\ireg-j-K}(\Scritini)}
\norm{\LxiOp^{j+1} \psibase[-2][i]}_{F^{\ireg-j-K}(\Scritini)}\nonumber\\
\lesssim{}&|\timefunc|^{-11 -2j+2i+2\delta} \norm{\psibase[-2][i]}_{H^{\ireg}_{9-\delta}(\Stauini)}^2 .
\end{align}
Taking the square root and applying the fundamental theorem of calculus again, one finds
\begin{align}
\norm{\LxiOp^{j} \psibase[-2][i]}_{W^{\ireg-j-K}(\Sphere)}
\lesssim{}& |\timefunc|^{-9/2 -j+i+\delta} \norm{\psibase[-2][i]}_{H^{\ireg}_{9-\delta}(\Stauini)} ,
\end{align}
which gives \eqref{eq:psi-2:pastSphereDecay-NEW}.
\end{proof}



\begin{lemma}[The Teukolsky scalar $\protect{\psibase[-2]}$ has an expansion]
\label{lem:TMEHasAnExpansion}
\standardMinusTwoHypothesisNotBEAM.
\standardPlusTwoHypothesisNotBEAM There is a regularity constant $K$ such that the following holds.
\standardMinusTwoHypothesisBEAM.
Assume
\standardPlusTwoHypothesisPointwiseCondition.
Let
$\ireg \in \Naturals$ such that $\ireg-K$ is sufficiently large,
and let
$\Ealpha{\psibase[-2]}=10-$.

Then $\psibase[-2]$ has a $(\ireg-K,3,0,\Ealpha{\psibase[-2]},\EDNoArg^2)$  expansion where $\EDNoArg^2=\iniEnergyMinusTwoForDescent{\ireg}$.
\end{lemma}
\begin{proof}
Throughout the proof, $K$ denotes a regularity constant that may vary from line to line.
Before considering $\psibase[-2]$, for a general spin-weighted scalar $\varphi$, consider Taylor expansions in the $\rInv$ variable. Recall Taylor's expansion lemma \ref{thm:Taylor-est}. In particular, consider $n\in\Naturals$, $A>0$, $f=f(R)\in C^{n+1}([-\epsilon, A])$ for some $\epsilon>0$, and $P_n$ the order $n$ Taylor polynomial in $R$ about $R=0$. Observe that $\di R=-r^{-2}\di r$,  so if the $L^2$ norms are defined in terms of $\di r$, we get from Taylor's expansion lemma \ref{thm:Taylor-est} that for $-1 < \beta < 1$, the following Taylor remainder estimate: \begin{align}
\norm{r^{n+\beta/2} (f-P_n)}_{L^2((1/A,\infty))}
\lesC{n,\beta}{}&\norm{ r^{\beta/2-1} f^{(n+1)}}_{L^2((1/A,\infty))}.
\label{eq:expansionremainderL2controlledbyhighorderderiL2}
\end{align}

Consider now a spin-weighted scalar $\varphi$ defined in the Kerr exterior and $j\in\Naturals$. The Taylor remainder estimate implies that if there is the expansion
\begin{equation}
\varphi = \sum_{i=0}^j \frac{R^i}{i!} \varphi_i(\timefunc,\omega) + \varphi_{\rem;j},
\end{equation}
we get, for $-1 < \beta < 1$,
\begin{align}
\norm{ \varphi_{\rem;j}}_{W^{0}_{2j+\beta}(\Dtautext)}
&{}={}\norm{ r ^{j+\beta/2}\varphi_{\rem;j}}_{W^{0}_0(\Dtautext)}\nonumber\\
&{}\lesC{j,\beta}{}
\norm{ r^{\beta/2-1}(\partial_R)^{j+1}\varphi}_{W^{0}_0(\Dtautext)}\nonumber\\
&{}\lesC{j,\beta}{}\norm{(\partial_R)^{j+1}\varphi}_{W^{0}_{\beta-2}(\Dtautext)}.
\end{align}
Substituting $\beta=\alpha-1$, one finds, for $0 < \alpha < 2$,
\begin{align}
\norm{ \varphi_{\rem;j}}_{W^{0}_{2j-1+\alpha}(\Dtautext)}
\lesC{j,\alpha}{}&
\norm{(\partial_R)^{j+1}\varphi}_{W^{0}_{\alpha-3}(\Dtautext)}.
\label{eq:generalremainderbulkestimate}
\end{align}
Commuting with the $\rescaledOps$ operators only introduces lower-order terms, so that, for $0 < \alpha < 2$ and $k\in\Naturals$,
\begin{align}
\norm{ \varphi_{\rem;j}}_{W^{\ireg}_{2j-1+\alpha}(\Dtautext)}
\lesC{j,\alpha}{}&
\norm{(\partial_R)^{j+1}\varphi}_{W^{\ireg}_{\alpha-3}(\Dtautext)}.
\label{eq:generalremainderbulkestimateWithRegularity}
\end{align}
Arguing in a similar way, we have for $0 < \alpha < 2$ and $k\in\Naturals$,
\begin{align}
\norm{ \varphi_{\rem;j}}_{W^{\ireg}_{2j-1+\alpha}(\Dtauearly)}
\lesC{j,\alpha}{}&
\norm{(\partial_R)^{j+1}\varphi}_{W^{\ireg}_{\alpha-3}(\Dtauearly)}.
\label{eq:generalremainderbulkestimateEarlyregion}
\end{align}

Now, consider the existence of an expansion for $\psibase[-2]$. From the expansion of $\partial_R$ in \eqref{eq:ddRasVOpstructure} and the expansion of $\LxiOp$ in \eqref{eq:LxiToVOpYOpLeta}, for $r\geq10M$ and $i\in\{0,\ldots,4\}$, one has, with $\mathbb{X}=\{MY,\LetaOp\}$,
\begin{align}
\label{eq:partialRiradpsiExpansion}
\abs{\partial_R^i \psibase[-2]}
\lesssim{}&\sum_{l=0}^i \sum_{|\mathbf{a}|\leq i-l}M^{i}\bigOAnalytic(1)
\abs{\mathbb{X}^{\mathbf{a}}\radpsi{l}} .
\end{align}
Since for bounded $r$, in particular for $r\in [r_+,10M]$, one has $\partial_R$ is in the span of $\YOp$, $\VOp$, $\LetaOp$, one finds that, for all $r$, equation \eqref{eq:partialRiradpsiExpansion} remains valid.
Since the $\partial_R^i\psibase[-2]$ exist for $i\in\{0,\ldots,4\}$, the following Taylor expansion exists
\begin{align}
\psibase[-2]=\sum_{i=0}^3\frac{R^i}{i!}(\psibase[-2])_i+(\psibase[-2])_{\rem;3}.
\end{align}
This proves condition \eqref{eq:expansion:expansionExists} in the definition of the expansion.

Using estimate \eqref{eq:ImproEnerDecPsiiExt} in theorem \ref{thm:ImproEstipsi01} and estimate \eqref{eq:psi-2:pastIntegrabilityNew} in lemma \ref{lem:idata}, and taking $\alpha=2-\delta \in (0,2)$, one finds for any $\timefunc\geq \timefunc_0$,
\begin{align}
\norm{ (\partial_R)^4\psibase[-2]}^2_{W^{\ireg-K}_{-1-\delta}(\Dtautext)}
\lesssim{}&\sum_{i=0}^4\norm{  \radpsi{i}}^2_{W^{\ireg-K}_{-1-\delta}(\Dtautext)}
\lesssim{}\NEWinienergyminustwo{\ireg-4}
\lesssim{} \iniEnergyMinusTwoForDescent{\ireg}.
\end{align}
Therefore, from the Taylor remainder bound \eqref{eq:generalremainderbulkestimate}, we conclude, for any $\timefunc\geq \timefunc_0$,
\begin{align}
\norm{ (\psibase[-2]){}_{\rem;3}}^2_{W^{\ireg-K}_{7-\delta}(\Dtautext)}
\lesssim{}&
\norm{(\partial_R)^4\psibase[-2]}^2_{W^{\ireg-K}_{-1-\delta}(\Dtautext)}
\lesssim  \iniEnergyMinusTwoForDescent{\ireg},
\end{align}
from which it follows that, letting
\begin{align}
\Ealpha{\psibase[-2]}=10-11\delta
\label{eq:alpha_1-2}
\end{align}
and noting that $(10-11\delta)-3 <7-\delta$, we have, for any $\timefunc\geq \timefunc_0$,
\begin{align}
\norm{ (\psibase[-2]){}_{\rem;3}}^2_{W^{\ireg-K}_{\Ealpha{\psibase[-2]}-3}(\Dtautext)}
\lesssim{}&
 \iniEnergyMinusTwoForDescent{\ireg},
\label{eq:radphiremest-v2}
\end{align}
which proves the remainder condition \eqref{eq:expansion:remainderBound} in the definition of an expansion.

In the region $\Dtauearly$, using the bound \eqref{eq:generalremainderbulkestimateEarlyregion}, using the estimates \eqref{eq:generalremainderbulkestimateEarlyregion} and \eqref{eq:psi-2:pastIntegrabilityNew}, and taking $\alpha=2-\delta$, we conclude
\begin{align}
\norm{ (\psibase[-2]){}_{\rem;3}}^2_{W^{\ireg-K}_{7-\delta}(\Dtauearly)}
\lesssim{}&
\norm{(\partial_R)^4\psibase[-2]}^2_{W^{\ireg-K}_{-1-\delta}(\Dtauearly)}\nonumber\\
\lesssim{}&\sum_{i=0}^4\norm{\radpsi{\ii}}_{W^{\ireg-K}_{-1-\delta}(\Dtauearly)}^2\nonumber\\
\lesssim{}&  \iniEnergyMinusTwoForDescent{\ireg}.
\end{align}
Hence, by letting $\Ealpha{\psibase[-2]}$ be as in equation \eqref{eq:alpha_1-2} and noting $(10-11\delta)-3<7-\delta$, it follows that
\begin{align}
\norm{ (\psibase[-2]){}_{\rem;3}}^2_{W^{\ireg-K}_{\Ealpha{\psibase[-2]}-3}(\Dtauearly)}
\lesssim{}&
 \iniEnergyMinusTwoForDescent{\ireg},
\label{eq:radphiremest-earlyregion}
\end{align}
 and this proves  condition \eqref{eq:expansion:remainderBoundInThePast} in the definition of an expansion.

Consider the Taylor expansion terms $(\psibase[-2])_i$ for $i\in\{0,\ldots,3\}$. From the pointwise bound on the $\psibase[-2][i]$ in inequality \eqref{eq:pointwisePsiiExt}, one finds that there are the decay estimates, for $i\in\{0,\ldots,3\}$ and $\timefunc\geq \timefunc_0$,
\begin{align}
\int_{\Sphere}\absScri{
(\psibase[-2])_i(\timefunc,\omega) }{\ireg-K}^2 \diTwoVol
\lesssim{}& \timefunc^{-9+2i+11\delta}\iniEnergyMinusTwoForDescent{\ireg} .
\label{eq:psi-2Expansion:radpsiToExpansionRelation-NEW}
\end{align}
For $\timefunc\geq\timefunc_0$, this proves that there are the decay estimates for the expansion terms, which is condition \eqref{eq:expansion:expansionTermsDecay}. For $\timefunc\leq\timefunc_0$, the same argument applies using the decay estimate in equation \eqref{eq:psi-2:pastSphereDecay-NEW}. Thus, condition \eqref{eq:expansion:expansionTermsDecay} holds for all $\timefunc\in\Reals$.

From the assumptions on $\psibase[+2]$ and $\psibase[-2]$, lemmas \ref{lem:intA=0} and \ref{lem:Ailim+infty} can be applied, and equation \eqref{eq:Ailim+infty} states the vanishing of the integral along $\Scri^+$ of the $\psibase[-2][i]$. From the relation between the $\psibase[-2][i]$ and the $(\psibase[-2])_i$ in equation \eqref{eq:partialRiradpsiExpansion},  it follows that the integral of the $(\psibase[-2])_i$ along $\Scri^+$ vanishes. Thus, condition \eqref{eq:expansion:integrability} holds, and $\psibase[-2]$ has the desired expansion.
\end{proof} 


\subsection{Estimates in the exterior region} 
\label{sec:exteriorest}
This section proves decay estimates for $\hat\sigma'$, $\widehat{G}_{2}$, $\hat\tau'$, $\widehat{G}_{1}$, $\hat\beta'$, and $\widehat{G}_{0}$ in the exterior region $r\geq t$. The proof consists of three major components. First, $\psibase[-2]$ has an expansion by lemma \ref{lem:TMEHasAnExpansion}. Second, the scalars $\hat\sigma',\widehat{G}_{2},\hat\tau',\widehat{G}_{1},\hat\beta',\widehat{G}_{0}$ are related to each other by the transport equations in the first-order formulation of the Einstein equation in lemma \ref{lem:TransportSystem}. Third, lemma \ref{lem:YintExpansionEst} states that if the source for a transport equation has an expansion, then the solution has an expansion and satisfies decay estimates. Thus, iterating through $\hat\sigma'$, $\widehat{G}_{2}$, $\hat\tau'$, $\widehat{G}_{1}$, $\hat\beta'$, and $\widehat{G}_{0}$ one finds each of these has an expansion and decays. 
The exterior region and the geodesics along which we integrate are illustrated in figure \ref{fig:Estimates:NPnExt}; see also figure \ref{fig:FoliationHyperbolic} for related regions. 

The following indices are useful in this iteration process. These indices are such that, for $\varphi$, $\Espin{\varphi}$ is the spin weight of $\varphi$, and $\El{\varphi}$ and $\Em{\varphi}$ will be the $\ElNoArg$ and $\EmNoArg$ arguments in the expansion of $\varphi$.

\begin{definition}
\label{def:alpha1Values}
Define
\index{S1svarphi@$\Espin{\cdot}$}
\index{L1lvarphi@$\El{\cdot}$}
\index{M1mvarphi@$\Em{\cdot}$}
\begin{subequations}
\begin{align}
\Espin{\psibase[-2]} ={}& -2, &
\El{\psibase[-2]} ={}& 3, &
\Em{\psibase[-2]} ={}& 0, 
\\
\Espin{\hat\sigma'} ={}& -2, &
\El{\hat\sigma'} ={}& 2, &
\Em{\hat\sigma'} ={}& 0, 
\\
\Espin{\widehat{G}_{2}} ={}& -2, &
\El{\widehat{G}_{2}} ={}& 1, &
\Em{\widehat{G}_{2}} ={}& 0, 
\\
\Espin{\hat\tau'} ={}& -1, &
\El{\hat\tau'} ={}& 3, &
\Em{\hat\tau'} ={}& 2, 
\\
\Espin{\widehat{G}_{1}} ={}& -1, &
\El{\widehat{G}_{1}} ={}& 0, &
\Em{\widehat{G}_{1}} ={}& 0, 
\\
\Espin{\hat\beta'} ={}& -1, &
\El{\hat\beta'} ={}& 2, &
\Em{\hat\beta'} ={}& 3, 
\\
\Espin{\widehat{G}_{0}} ={}& 0, &
\El{\widehat{G}_{0}} ={}& 1, &
\Em{\widehat{G}_{0}} ={}& 2 .
\end{align}
\end{subequations}
For all $\beta\in \mathbb{R}$ we also define $\Espin{M^\beta \varphi}=\Espin{\varphi}, \El{M^\beta \varphi}=\El{\varphi}, \Em{M^\beta \varphi}=\Em{\varphi}$.
\end{definition}

\begin{definition}[Initial data norms]
\label{def:setOfFieldsAndInitialDataNorm}
Define the following set of dimensionless fields
\index{P4Phi@$\setOfFields$}
\begin{align}
\setOfFields={}& \{M\psibase[-2],\hat\sigma',M^{-1}\widehat{G}_{2},M\hat\tau',M^{-2}\widehat{G}_{1},\hat\beta',M^{-1}\widehat{G}_{0}\} .
\end{align}
 For any $\ireg\in\Naturals$, define
\index{I2Ikinit@$\inienergynoalpha{\ireg}$}
\begin{align}
\inienergynoalpha{\ireg}[\setOfFields]
={}&\sum_{\varphi\in \setOfFields}\inienergy{\ireg}{2\El{\varphi}+3}(\varphi) .
\end{align}
\end{definition}

\begin{lemma}[Exterior estimates]
\label{lem:allquantitiesextestimates}
Consider an outgoing BEAM solution of the linearized Einstein equation satisfying as in definition \ref{def:BEAMLinearisedEinsteinSolution}. There is a regularity constant $K$ such that the following hold. Let $\ireg\in\Naturals$ such that $\ireg- K$ is sufficiently large. For $\varphi\in\setOfFields$, 
\begin{enumerate}
\item $\varphi$ has a $(\ireg-K,\El{\varphi},\Em{\varphi},2\El{\varphi}+4-,\inienergynoalpha{\ireg}[\setOfFields])$ expansion, 
\item for $\Lxitimes\in \{0,1\}$ and $\timefunc\geq \timefunc_0$, 
\begin{align}
\norm{\LxiOp^{\Lxitimes}\varphi}_{W^{\ireg-K-\Lxitimes}_{2\El{\varphi}+2-} (\Boundtext)}^2
\lesssim{}&  \timefunc^{-2\Lxitimes}
\inienergynoalpha{\ireg}[\setOfFields] ,
\label{eq:allquantitiestransitionregionesti}
\end{align}
\item and, for $\timefunc\geq\timefunc_0$ and $r\geq \timefunc$, 
\begin{align}
\absRescaled{\varphi}{\ireg-K}^2 
\lesssim& 
r^{-2\Em{\varphi}}
\timefunc^{2\Em{\varphi}-3-2\El{\varphi}+}
\inienergynoalpha{\ireg}[\setOfFields] .
\label{eq:pointwiseLinearGravityExterior}
\end{align}
\end{enumerate}
\end{lemma}

\begin{proof}
For ease of presentation we will here use mass normalization as in definition~\ref{def:massnormalization}, and throughout this proof, $K$ denotes a regularity constant that may vary from line to line.
The overall strategy of this proof is to use the expansion for $\psibase[-2]$ from lemma \ref{lem:TMEHasAnExpansion}, the hierarchy of transport equations \eqref{eq:TransEqsKerr}, lemma \ref{lem:YintExpansionEst} to conclude that solutions of the transport equation have expansions if the source does, and lemma \ref{lem:transformationofexpansion} when a transformation of the expansions for the source is required. The details now follow.  

From lemma \ref{lem:TMEHasAnExpansion}, $\psibase[-2]$ has a $(\ireg-K,3,0,10-,\ED{\psibase[-2]}^2)$ expansion where 
\begin{align}
\ED{\psibase[-2]}^2
=\iniEnergyMinusTwoForDescent{\ireg}
\leq \inienergynoalpha{\ireg}[\setOfFields].
\end{align} 
The decay estimates of the transition flux for $\varphi=\psibase[-2]$ follow from integrating the pointwise estimates \eqref{eq:pointwisePsiiExt} on $\Boundtext$ and making use of \eqref{eq:psi-2:pastIntegrabilityNew}, while the pointwise decay estimates \eqref{eq:pointwiseLinearGravityExterior} for $\varphi=\psibase[-2]$ follow directly from the estimates \eqref{eq:pointwisePsiiExt} and \eqref{eq:psi-2:pastIntegrabilityNew}. 

Consider $\hat\sigma'$. The transport equation \eqref{eq:TransEqTildeSigma} states that 
\begin{align}
\YOp(\hat{\sigma}')=- \frac{12 \bar{\kappa}_{1'}{} \psibase[-2]}{\sqrt{r^2 + a^2}}.
\end{align} 
The factor $\frac{\bar{\kappa}_{1'}}{\sqrt{r^2+a^2}}$ is a rational function in $\bar{\kappa}_{1'}/\sqrt{r^2+a^2}$ of homogeneous degree $0$. Thus, lemma \ref{lem:transformationofexpansion} implies that $- \frac{12 \bar{\kappa}_{1'}{} \psibase[-2]}{\sqrt{r^2 + a^2}}$ also has a $(\ireg-K,3,0,10-,\inienergynoalpha{\ireg}[\setOfFields])$ expansion. Thus, lemma \ref{lem:YintExpansionEst} implies that $\hat\sigma'$ has a $(\ireg-K,3-1,0,(10-2)-,\inienergynoalpha{\ireg}[\setOfFields])$ expansion, that is a $(\ireg-K,2,0,8-,\inienergynoalpha{\ireg}[\setOfFields])$ expansion. 

Consider $\widehat{G}_{2}$. The transport equation \eqref{eq:TransEqG20} states \begin{align}
\YOp(\widehat{G}_2)=- \tfrac{2}{3} \hat{\sigma}'.
\end{align} 
From this, lemma \ref{lem:YintExpansionEst} implies that $\widehat{G}_{2}$ has a $(\ireg-K,2-1,0,(8-2)-,\inienergynoalpha{\ireg}[\setOfFields])$ expansion, that is a $(\ireg-K,1,0,6-,\inienergynoalpha{\ireg}[\setOfFields])$ expansion.

Consider $\hat\tau'$. The transport equation \eqref{eq:TransEqHatTau} states 
\begin{align}
\YOp(\hat{\tau}')
={}& - \frac{\kappa_{1}{} (\edt - 2 \tau + 2 \bar{\tau}')\hat{\sigma}'}{6 \overline{\kappa}_{1'}{}^2} \nonumber\\
={}&\frac{1}{6\bar{\kappa}_{1'}^2}\left(
  \kappa_{1}\edt\hat\sigma'
  -2\tau\kappa_1\hat\sigma'
  +2\bar{\tau}'\kappa_1\hat\sigma'
\right) .
\label{eq:TransEqHatTauAgain}
\end{align}
The operator $\kappa_{1}\edt\hat\sigma'$ can be expanded in terms of $\rescaledOps$ with rational coefficients of order at most $0$. Thus, $\kappa_{1}\edt\hat\sigma'$ has an expansion with indices $(\ireg-K,2,0,8-,\inienergynoalpha{\ireg}[\setOfFields])$ $=(\ireg-K,2,0,8-,\inienergynoalpha{\ireg}[\setOfFields])$. The terms $-2\tau\kappa_1\hat\sigma'$ and $2\bar{\tau}'\kappa_1\hat\sigma'$ have similar expansions, where the regularity index can be trivially lowered to match that for $\kappa_{1}\edt\hat\sigma'$. The coefficient $\bar{\kappa}_{1'}^{-2}$ has homogeneous degree $-2$. Thus, from the expansion for $\hat\sigma'$, one finds that the right-hand side of equation \eqref{eq:TransEqHatTauAgain} has an expansion with indices $(\ireg-K,2+2,0+2,8+4-,\inienergynoalpha{\ireg}[\setOfFields])$ $=(\ireg-K,4,2,12+4-,\inienergynoalpha{\ireg}[\setOfFields])$. Thus, $\hat\tau'$ has an expansion with indices $(\ireg-K,4-1,2,(12-2)-,\inienergynoalpha{\ireg}[\setOfFields])$ $=(\ireg-K,3,2,10-,\inienergynoalpha{\ireg}[\setOfFields])$. 

Consider $\widehat{G}_{1}$. The transport equation \eqref{eq:TransEqG10} states
\begin{align}
\YOp(\widehat{G}_1)={}&\frac{2 \kappa_{1}{}^2 \overline{\kappa}_{1'}{}^2 \hat{\tau}'}{r^2}
 + \frac{\kappa_{1}{}^2 \overline{\kappa}_{1'}{} (\edt -  \tau + \bar{\tau}')\widehat{G}_2}{2 r^2} .
\end{align}
The term involving $\hat\tau'$ has an expansion with indices $(\ireg-K,3-2,2-2,(10-4)-,\inienergynoalpha{\ireg}[\setOfFields])$ $=(\ireg-K,1,0,6-,\inienergynoalpha{\ireg}[\setOfFields])$. The term with $\widehat{G}_{2}$ has an expansion with indices $(\ireg-K,1,0,6-,\inienergynoalpha{\ireg}[\setOfFields])$ $=(\ireg-K,1,0,6-,\inienergynoalpha{\ireg}[\setOfFields])$. Thus, the right-hand side has an expansion with indices $(\ireg-K,1,0,6-,\inienergynoalpha{\ireg}[\setOfFields])$, and $\widehat{G}_{1}$ has an expansion with indices $(\ireg-K,1-1,0,(6-2)-,\inienergynoalpha{\ireg}[\setOfFields])$ $=(\ireg-K,0,0,4-,\inienergynoalpha{\ireg}[\setOfFields])$.

Consider $\hat\beta'$. The transport equation \eqref{eq:TransEqHatBeta'} states 
\begin{align}
\YOp(\hat{\beta}')
={}&\frac{r \widehat{G}_1}{6 \kappa_{1}{}^2 \overline{\kappa}_{1'}{}^2}
 + \frac{\kappa_{1}{} \tau \widehat{G}_2}{6 \overline{\kappa}_{1'}{}^2} . 
\end{align}
The term involving $\widehat{G}_{1}$ has an expansion with indices $(\ireg-K,0+3,0+3,4+6-,\inienergynoalpha{\ireg}[\setOfFields])$ $=(\ireg-K,3,3,10-,\inienergynoalpha{\ireg}[\setOfFields])$. The term involving $\widehat{G}_{2}$ has an expansion with indices $(\ireg-K,1+3,0+3,6+6-,\inienergynoalpha{\ireg}[\setOfFields])$ $=(\ireg-K,4,3,12-,\inienergynoalpha{\ireg}[\setOfFields])$. The first of these is more restrictive. Thus, the right-hand side has a $(\ireg-K,3,3,10-,\inienergynoalpha{\ireg}[\setOfFields])$ expansion, and $\hat\beta'$ has an expansion with indices $(\ireg-K,3-1,3,(10-2)-,\inienergynoalpha{\ireg}[\setOfFields])$ $=(\ireg-K,2,3,8-,\inienergynoalpha{\ireg}[\setOfFields])$. 

Finally, consider $\widehat{G}_{0}$. The transport equation \eqref{eq:TransEqG00} states 
\begin{align}
\YOp(\widehat{G}_0)={}&-  \frac{(\edt -  \tau)\widehat{G}_1}{3 \kappa_{1}{}}
 - \frac{\tau \widehat{G}_1}{r}
 -  \frac{\bar{\tau} \overline{\widehat{G}_1}}{r}
 + \frac{2 \kappa_{1}{}^2 \overline{\kappa}_{1'}{} (\edt -  \bar{\tau}')\hat{\beta}'}{r^2}
 -  \frac{(\edtp -  \bar{\tau})\overline{\widehat{G}_1}}{3 \overline{\kappa}_{1'}{}}
 + \frac{2 \kappa_{1}{} \overline{\kappa}_{1'}{}^2 (\edtp -  \tau')\overline{\hat{\beta}'}}{r^2} . 
\end{align}
Complex conjugation does not change the indices in an expansion. The terms involving $\widehat{G}_{1}$ have an additional level of regularity and coefficients with homogeneous degree $-2$, so that the expansion has indices $(\ireg-K,0+2,0+2,4+4-,\inienergynoalpha{\ireg}[\setOfFields])$ $=(\ireg-K,2,2,8-,\inienergynoalpha{\ireg}[\setOfFields])$. The terms involving $\hat\beta'$ also have one derivative but coefficients with homogeneous degree $0$, so that the expansion has indices $(\ireg-K,2,3,8-,\inienergynoalpha{\ireg}[\setOfFields])$ $=(\ireg-K,2,3,8-,\inienergynoalpha{\ireg}[\setOfFields])$. Thus, the right-hand side has an expansion with indices $(\ireg-K,2,2,8-,\inienergynoalpha{\ireg}[\setOfFields])$, and $\widehat{G}_{0}$ has an expansion with indices $(\ireg-K,2-1,2,(8-2)-,\inienergynoalpha{\ireg}[\setOfFields])$ $=(\ireg-K,1,2,6-,\inienergynoalpha{\ireg}[\setOfFields])$. 

For each of $\hat\sigma'$, $\hat{G}_2$, $\hat\tau'$, $\hat{G}_1$, $\hat\beta'$, $\hat{G}_0$, lemma \ref{lem:YintExpansionEst} was applied to obtain the expansion. This lemma also gives estimates for the integral on $\Boundtext$ and for pointwise norms. The pointwise bound \eqref{eq:pointwiseDescentSpecialCase} is stronger than the bound \eqref{eq:pointwiseDescentGeneralCase}, so in all cases, one can apply the bound \eqref{eq:pointwiseDescentGeneralCase}. (Because of this observation, it is not necessary to track which of the two bounds holds, although a carefully tracking of this would reveal that the bound \eqref{eq:pointwiseDescentSpecialCase} never holds in this argument.)
\end{proof}

\subsection{Estimates in the interior region}
\label{sec:interiorEst}
In this section, we prove decay in the interior region $r<t$. To do so, we cannot use the expansion at infinity. To obtain decay estimates here, there are several key ideas. First, we integrate the hierarchy of transport equations along ingoing null geodesics again. Second, for the variable itself, we use the value at the transition hypersurface $r=t$ as the initial value. Third, we use the estimates for the previous variable in the hierarchy to control the source term. The interior region and the geodesics along which we integrate are illustrated in figure \ref{fig:Estimates:NPnInt}, and the bulk region above the surface swept out by the geodesics is illustrated in figure \ref{fig:FoliationHyperbolic}. 

\begin{lemma}
\label{lem:YintInteriorEsti}
Let $\varphi$ and $\varrho$ be spin-weighted scalars which solve
\begin{align}
\YOp\varphi
={}&\varrho ,
\label{eq:ingoingtransport}
\end{align}
and let $0 \leq \alphalo < \alphahi$ be given. Let $k\in\Integers^+$. Let $D\geq0$. 

Assume that for all $\timefunc \geq \timefunc_0 + h(\timefunc_0)$, $\alpha\in[\alphalo, \alphahi]$, and $q \in \{0,1\}$, $\varphi$ and $\varrho$ satisfy
\begin{subequations} \label{eq:varrhotransitionregionesti}
\begin{align}
\norm{\LxiOp^{\Lxitimes}\varphi}_{W^{k-\Lxitimes}_{\alpha} (\Boundtext)}^2
\lesssim{}& D^2 \timefunc^{\alpha-\alphahi-2\Lxitimes} , \label{eq:varrhotransitionregionesti-varphi} \\
\norm{\LxiOp^{\Lxitimes}\varrho}_{W^{k-\Lxitimes}_{\alpha+1}(\Dtautimenear)}^2
\lesssim{}& D^2 \timefunc^{\alpha-\alphahi-2\Lxitimes} , \label{eq:varrhotransitionregionesti-varrho}
\end{align}
\end{subequations}
then, for all  $\alphalo \leq \alpha \leq \alphahi$,
the following holds. For all $q \in \{0,1\}$ and $\timefunc\geq\timefunc_0 + h(\timefunc_0)$,
\begin{subequations}
\begin{align}
\norm{\LxiOp^{\Lxitimes}\varphi}_{W^{k-\Lxitimes}_{\alpha}(\Stautimeint)}^2
+ \norm{\LxiOp^{\Lxitimes}\varphi}_{W^{k-\Lxitimes}_{\alpha-1}(\Dtautimenear)}^2
\lesssim{}& D^2
\timefunc^{\alpha-\alphahi-2\Lxitimes} ,
\label{eq:varphiinteriorregionesti} 
\end{align}
and, if $k\geq 4$, then for all $\timefunc\geq \timefunc_0 + h(\timefunc_0)$ and $(\timefunc,r,\omega)\in \Omega^{\near}_{\timefunc,\infty}$,
\begin{align}
\absRescaled{\varphi(\timefunc,r,\omega)}{k-4} \lesssim{}& D^2 r^{-\frac{\alpha}{2}}\timefunc^{-\frac{\alphahi+1-\alpha}{2}}.
\label{eq:pointwisevarphigeneralinterior}
\end{align}
\end{subequations}
\end{lemma}
\begin{remark}
For $\timefunc \geq \timefunc_0 + h(\timefunc_0)$, $\varphi$ is determined in $\Omega^{\near}_{\timefunc,\infty}$ by $\varrho$ and $\varphi \big{|}_{\Xi_{\timefunc_{\ingoingcone}(\timefunc),\infty}}$.
\end{remark}
\begin{proof}
For ease of presentation we will here use mass normalization as in definition~\ref{def:massnormalization}.
For $\timefunc \geq \timefunc_0 + h(\timefunc_0)$, inequality \eqref{eq:TransIngoingMoranearregion} gives
\begin{align}
\hspace{4ex}&\hspace{-4ex}
\norm{\LxiOp^{\Lxitimes}\varphi}_{W^{k-\Lxitimes}_{\alpha}(\Sigma^{\interior}_{\timefunc})}^2
+ \norm{\LxiOp^{\Lxitimes}\varphi}_{W^{k-\Lxitimes}_{\alpha-1}(\Dtautnear{\timefunc})}^2 \nonumber\\
\lesssim{}&\norm{\LxiOp^{\Lxitimes}\varphi}_{W^{k-\Lxitimes}_{\alpha} (\Xi_{\timefunc_{\ingoingcone}(\timefunc),\infty})}^2
+ \norm{\LxiOp^{\Lxitimes}\varrho}_{W^{k-\Lxitimes}_{\alpha+1}(\Dtautnear{\timefunc})}^2 \nonumber\\
\lesssim{}& D^2 \timefunc^{\alpha-\alphahi-2\Lxitimes} ,
\label{eq:varphiinteriorgeneral}
\end{align}
which proves \eqref{eq:varphiinteriorregionesti}. Here we have used assumption
\eqref{eq:varrhotransitionregionesti} and $\timefunc_{\ingoingcone}(\timefunc) \sim \timefunc$.

The estimate \eqref{eq:SobolevOnStautWithgammazero} gives for $\timefunc > \timefunc_0 + h(\timefunc_0)$,
\begin{align}
\absRescaled{r^{\frac{\alpha}{2}}\varphi}{k-3}^2\lesssim{}& \norm{r^{\frac{\alpha}{2}}\varphi}_{W^{k}_{-1}(\Sigma_{\timefunc}^{\interior})}^2
\lesssim \norm{\varphi}_{W^{k}_{-1+\alpha}(\Sigma_{\timefunc}^{\interior})}^2 .
\end{align}
Since $-1+\alpha < \alpha$, we find, in view of \eqref{eq:varphiinteriorregionesti} that $r^{\frac{\alpha}{2}} \varphi$ tends to zero as $\timefunc \nearrow \infty$.
We can therefore apply lemma \ref{lem:anisotropicSpaceTimeSobolev} which gives
\begin{align}
\absRescaled{r^{\frac{\alpha}{2}}\varphi}{k-4}^2\lesssim {}& \norm{r^{\frac{\alpha}{2}}\varphi}_{W^{k-1}_{-1}(\Omega_{\timefunc,\infty}^{\interior})}
\norm{r^{\frac{\alpha}{2}}\LxiOp\varphi}_{W^{k-1}_{-1}(\Omega_{\timefunc,\infty}^{\interior})}
\nonumber\\
\leq {}&
\norm{r^{\frac{\alpha}{2}}\varphi}_{W^{k-1}_{-1}(\Omega_{\timefunc,\infty}^{\near})}
\norm{r^{\frac{\alpha}{2}}\LxiOp\varphi}_{W^{k-1}_{-1}(\Omega_{\timefunc,\infty}^{\near})}
\nonumber\\
\lesssim{}&\norm{\varphi}_{W^{k-1}_{\alpha-1}(\Dtautnear{\timefunc})}
\norm{\LxiOp\varphi}_{W^{k-1}_{\alpha-1}(\Dtautnear{\timefunc})} ,
\end{align}
which using \eqref{eq:varphiinteriorregionesti} proves \eqref{eq:pointwisevarphigeneralinterior}.
\end{proof}

\begin{lemma}
\label{lem:EnerMoraAllquantitiesestiint}
There is a regularity constant $K$ such that the following holds.
Consider an outgoing BEAM solution of the linearized Einstein equation as in definition \ref{def:BEAMLinearisedEinsteinSolution}, with regularity $\ireg\in \Naturals$ and $\ireg-K $ sufficiently large. For $\varphi\in\setOfFields\backslash\{M\psibase[-2]\}$, let $\Em{\varphi}$, $\El{\varphi}$ be as in definition \ref{def:alpha1Values}, and set
\begin{align}
\alphalow{\varphi} ={}& 2\max\{\Em{\varphi}-1,0\}{}+, & \alphahigh{\varphi} ={}& 2\El{\varphi}+{2-} . 
\label{eq:defAlphalowAlphahigh}
\end{align}
The following hold for $\varphi\in\setOfFields\backslash\{M\psibase[-2]\}$ and $\timefunc \geq \timefunc_0$.
\begin{enumerate}
\item \label{point:EnergMora-int}
For $\Lxitimes\in \{0,1\}$ and $\alpha\in[\alphalow{\varphi},\alphahigh{\varphi}]$,
there are energy and Morawetz estimates in the interior region
\begin{subequations}\label{eq:EnerMoraAllquantitiesestiint}
\begin{align}
\norm{\LxiOp^{\Lxitimes}\varphi}_{W^{\ireg-K-\Lxitimes}_{\alpha}(\Stautimeint)}^2
\lesssim{}&
\inienergynoalpha{\ireg}[\setOfFields]
\timefunc^{\alpha-\alphahigh{\varphi}-2\Lxitimes} ,  \label{eq:EnerMora-int-varphi-Sigma}\\
\norm{\LxiOp^{\Lxitimes}\varphi}_{W^{\ireg-K-\Lxitimes}_{\alpha-1}(\Dtautimenear)}^2
\lesssim{}&
\inienergynoalpha{\ireg}[\setOfFields]
\timefunc^{\alpha-\alphahigh{\varphi}-2\Lxitimes} . \label{eq:EnerMora-int-varphi-Omega}
\end{align}
\end{subequations}
\item \label{point:pointwise-int}
There are pointwise-in-time estimates in the interior region, for $(\timefunc,r,\omega)\in\Dtauitnear$, 
\begin{align}
\absRescaled{\varphi(\timefunc,r,\omega)}{\ireg-K}^2
\lesssim{}&r^{-\alphalow{\varphi}}
 \timefunc^{-(\alphahigh{\varphi}+1) + \alphalow{\varphi}}
\inienergynoalpha{\ireg}[\setOfFields].
\label{eq:pointwiseAllquantitiesestiint}
\end{align}
\end{enumerate}
\end{lemma}

\begin{proof}
For ease of presentation, we use mass normalization as in definition~\ref{def:massnormalization}. Furthermore, the regularity constant $K$ can vary from term to term, not merely, line to line. 
We put all the equations of the system \eqref{eq:TransEqsKerr} into the form of \eqref{eq:ingoingtransport}, and denote the corresponding right-hand side of each equation of $\varphi\in \{\hat{\sigma}', \widehat{G}_2, \hat{\tau}', \widehat{G}_1, \hat{\beta}', \widehat{G}_0\}$ by $\varrho[\varphi]$. The general strategy is to use estimate \eqref{eq:Psi0InteriorMoraDecay} for $\LxiOp^\ij \psibase[-2]$, estimates \eqref{eq:allquantitiestransitionregionesti} for the transition flux, and lemma \ref{lem:YintInteriorEsti} applied to each transport equation in the system \eqref{eq:TransEqsKerr} to iteratively conclude that estimates \eqref{eq:EnerMoraAllquantitiesestiint} and \eqref{eq:pointwiseAllquantitiesestiint} are valid. Since the part of the interior region $\{r \leq \timefunc\}$ to the future of $\Stauini$ and to the past of $\Sigma_{\timefunc_0 + h(\timefunc_0)}$ is compact, we can without loss of generality state our estimates in terms of $\inienergynoalpha{\ireg}[\setOfFields]$. We will now discuss the proof of the energy and Morawetz estimate \eqref{eq:EnerMoraAllquantitiesestiint} for each of the fields and comment on the proof of the pointwise estimate \eqref{eq:pointwiseAllquantitiesestiint} at the end of the proof. 

For $\psibase[-2]$, define
\begin{align}
\alphalow{\psibase[-2]} ={}& 2+, & \alphahigh{\psibase[-2]} ={}& 2\El{\psibase}+{2-} .
\end{align}
Observe that $\alphalow{\psibase[-2]}$ does not conform to the formula for $\alphalow{\varphi}$ given in equation \eqref{eq:defAlphalowAlphahigh}. 
For ease of reference, for $\varphi \in \setOfFields$, the values of $\alphalow{\varphi}$ and $\alphahigh{\varphi}$ are given in the following table: 
\begin{center}
\begin{tabular}{c | c  c}
$\varphi$ & $\alphalow{\varphi}$ & $\alphahigh{\varphi}$ \\
\hline
$\psibase[-2]$ & $2+$ & $8-$ \\
$\hat{\sigma}'$ & $0+$ & $6-$ \\
$\widehat{G}_2$ & $0+$ & $4-$  \\
$\hat{\tau}'$ & $2+$ & $8-$ \\
$\widehat{G}_1$ & $0+$ & $2-$ \\
$\hat{\beta}'$ & $4+$ & $6-$ \\
$\widehat{G}_0$ & $2+$ & $4-$
\end{tabular} .
\end{center}

In applying lemma \ref{lem:YintInteriorEsti}, we shall freely make use of the fact that since $r=\timefunc$ on the transition surface $\Xi$, \eqref{eq:allquantitiestransitionregionesti} can be restated in the form with explicit time decay, that is as estimate \eqref{eq:varrhotransitionregionesti-varphi} with the range of weights $\alphalow{\varphi} \leq \alpha \leq \alphahigh{\varphi}$, for $\varphi \in \setOfFields\backslash\{\psibase[-2]\}$.

\subsubsection*{The case $\psibase[-2]$}
From \eqref{eq:Psi0InteriorMoraDecay}, we get after a straightforward change of parameters, using $6- = \alphahigh{\psibase[-2]}-2$,
 for $\alphalow{\psibase[-2]} \leq  \alpha \leq \alphahigh{\psibase[-2]}$,
\begin{align}
\label{eq:psibase-2InteriorDecay}
\norm{\LxiOp^{\Lxitimes}  \psibase[-2]}^2_{W_{\alpha-1}^{ \ireg-K-\Lxitimes}(\Dtautimeint)}
\lesssim {}&\timefunc^{\alpha-\alphahigh{\psibase[-2]}-2\Lxitimes}
\inienergynoalpha{\ireg}[\setOfFields] . 
\end{align}
which is \eqref{eq:EnerMora-int-varphi-Sigma}.

\subsubsection*{The case $\hat\sigma'$}
From estimate \eqref{eq:allquantitiestransitionregionesti},
we get hypothesis \eqref{eq:varrhotransitionregionesti-varphi} for $\alphalow{\hat{\sigma}'} \leq \alpha \leq  \alphahigh{\hat{\sigma}'}$. From \eqref{eq:psibase-2InteriorDecay}, for $ \varrho[\hat{\sigma}']=f\psibase[-2] $ with $f=\bigOAnalytic(1)$,
 we get, after a reparametrization, hypothesis \eqref{eq:varrhotransitionregionesti-varrho} for the same range of weights. An application of lemma \ref{lem:YintInteriorEsti} proves point \ref{point:EnergMora-int} for $\hat\sigma'$.

\subsubsection*{The case $\widehat{G}_2$} The argument for the $\widehat{G}_2$ follows exactly the same pattern, which establishes point \ref{point:EnergMora-int} for $\widehat{G}_2$.

\subsubsection*{The case $\hat{\tau}'$} From the transport equation \eqref{eq:TransEqHatTau}, we have 
\begin{align}
\Vert \LxiOp^{\Lxitimes} \varrho[\hat{\tau}'] \Vert_{W^{k-K-\Lxitimes}_{\alpha+1}(\Dtautimeint)}^2 \lesssim
\Vert \LxiOp^{\Lxitimes} \hat{\sigma}' \Vert_{W^{k-K-\Lxitimes}_{\alpha-3}(\Dtautimeint)}^2 .
\end{align}
Making the substitution $\alpha-3 = \alpha + 1 - 2\Em{\hat{\tau}'}= \beta -1$, or $\beta = \alpha - 2(\Em{\hat{\tau}'}-1)$, we find using estimate \eqref{eq:EnerMora-int-varphi-Omega} for $\hat{\sigma}'$, after a reparametrization, that $\varrho[\hat{\tau}']$ satisfies hypothesis \eqref{eq:varrhotransitionregionesti-varrho} for the range of weights $\alphalow{\hat{\tau}'} \leq \alpha \leq \alphahigh{\hat{\tau}'}$, where
\begin{subequations} \label{eq:range-int-tau}
\begin{align}
\alphalow{\hat{\tau}'} = \alphalow{\hat{\sigma}'} + 2(\Em{\hat{\tau}'}-1) = 2+ ,\\
\alphahigh{\hat{\tau}'} = \alphahigh{\hat{\sigma}'} + 2(\Em{\hat{\tau}'}-1)  = 8- .
\end{align}
\end{subequations}
On the other hand, we have that estimate \eqref{eq:varrhotransitionregionesti-varphi} holds for the range $0+ \leq \alpha \leq \alphahigh{\hat{\tau}'}$. We may thus apply lemma \ref{lem:YintInteriorEsti} for the intersection of these ranges, $\alphalow{\hat{\tau}'} \leq \alpha \leq \alphahigh{\hat{\tau}'}$ to prove point \ref{point:EnergMora-int} for $\hat{\tau}'$.

\subsubsection*{The case $\widehat{G}_1$}
We have that
\begin{align}
\Vert \LxiOp^{\Lxitimes} \varrho[\widehat{G}_1] \Vert_{W^{k-K-\Lxitimes}_{\alpha+1}(\Dtautimeint)}^2
\lesssim
\Vert \LxiOp^{\Lxitimes} \hat{\tau}' \Vert_{W^{k-K-\Lxitimes}_{\alpha+5}(\Dtautimeint)}^2
+ \Vert \LxiOp^{\Lxitimes} \widehat{G}_2 \Vert_{W^{k-K-\Lxitimes}_{\alpha+1}(\Dtautimeint)}^2 .
\end{align}
We consider the second term on the right-hand side first. Writing $\alpha+1 = \beta-1$ and using estimate \eqref{eq:EnerMora-int-varphi-Omega} for $\widehat{G}_2$, we have, after a reparametrization,
\begin{align}
\Vert \LxiOp^{\Lxitimes} \widehat{G}_2 \Vert_{W^{k-K-\Lxitimes}_{\alpha+1}(\Dtautimeint)}^2 \lesssim
\timefunc^{\alpha-(\alphahigh{\widehat{G}_2} -2)-2\Lxitimes}
\inienergynoalpha{\ireg}[\setOfFields]
\end{align}
for $\alphalow{\widehat{G}_2}-2 \leq \alpha \leq \alphahigh{\widehat{G}_2}-2$. Here we must restrict the lower limit to zero, which yields the range $0+\leq \alpha \leq 2-$. For the first term, from the estimates for $\hat{\tau}'$, we get after the substitution $\alpha\mapsto\alpha+6$, 
\begin{align}
\Vert \LxiOp^{\Lxitimes} \hat{\tau}' \Vert_{W^{k-K-\Lxitimes}_{\alpha+5}(\Dtautimeint)}^2
\lesssim \timefunc^{\alpha-(\alphahigh{\hat{\tau}'}-6)-2\Lxitimes}
\inienergynoalpha{\ireg}[\setOfFields]
\end{align}
for the range $\alphalow{\hat{\tau}'}-6 \leq \alpha \leq \alphahigh{\hat{\tau}'}-6$, which is $-4+\leq\alpha\leq2-$, which is less restrictive than the one arising from $\widehat{G}_2$. Thus, we find that estimate \eqref{eq:varrhotransitionregionesti-varrho} holds for $\varrho[\widehat{G}_1]$ for the range of weights $\alphalow{\widehat{G}_1} \leq \alpha \leq \alphahigh{\widehat{G}_1}$ with $\alphalow{\widehat{G}_1} = 0+$, $\alphahigh{\widehat{G}_1} = 2-$. This proves point \ref{point:EnergMora-int} for $\widehat{G}_1$.

\subsubsection*{The case $\hat{\beta}'$} We have
\begin{align}
\Vert \varrho[\hat{\beta}'] \Vert_{W^{k-K-\Lxitimes}_{\alpha+1}(\Dtautimeint)}^2 \lesssim
\Vert \LxiOp^{\Lxitimes} \widehat{G}_1 \Vert_{W^{k-K-\Lxitimes}_{\alpha-5}(\Dtautimeint)}^2
+ \Vert \LxiOp^{\Lxitimes} \widehat{G}_2 \Vert_{W^{k-K-\Lxitimes}_{\alpha-5}(\Dtautimeint)}^2 .
\end{align}
Making the substitution $\alpha-5 = \beta-1$, we get estimates for the ranges $\alphalow{\widehat{G}_1}+4 < \alpha < \alphahigh{\widehat{G}_1}+4$, and $\alphalow{\widehat{G}_2}+4 \leq \alpha \leq \alphahigh{\widehat{G}_2}+4$, respectively. Here the case $\widehat{G}_1$ gives the more restrictive range, and we find that estimate \eqref{eq:varrhotransitionregionesti-varrho} holds for $\varrho[\hat{\beta}']$ for the range $\alphalow{\hat{\beta}'} \leq \alpha \leq \alphahigh{\hat{\beta}'}$ with $\alphalow{\hat{\beta}'} = 2(\Em{\hat{\beta}'}-1)+ = 4+$, $\alphahigh{\hat{\beta}'} = 2\El{\hat{\beta}'} +2- = 6-$. This proves point \ref{point:EnergMora-int} for $\hat{\beta}'$.

\subsubsection*{The case $\widehat{G}_0$} We have
\begin{align}
\Vert \varrho[\widehat{G}_0] \Vert_{W^{k-K-\Lxitimes-1}_{\alpha+1}(\Dtautimeint)}^2 \lesssim
\Vert \LxiOp^{\Lxitimes} \widehat{G}_1 \Vert_{W^{k-K-\Lxitimes}_{\alpha-3}(\Dtautimeint)}^2
+ \Vert \LxiOp^{\Lxitimes} \hat{\beta}' \Vert_{W^{k-K-\Lxitimes}_{\alpha+1}(\Dtautimeint)}^2 .
\end{align}
Proceeding as above yields for the first term
\begin{align}
\Vert \LxiOp^{\Lxitimes} \widehat{G}_1 \Vert_{W^{k-K-\Lxitimes}_{\alpha-3}(\Dtautimeint)}^2 \lesssim
\timefunc^{\alpha-(\alphahigh{\widehat{G}_1} +2)-2\Lxitimes}
\inienergynoalpha{\ireg}[\setOfFields]
\end{align}
for the range $\alphalow{\widehat{G}_1}+2 \leq \alpha \leq \alphahigh{\widehat{G}_1}+2$, i.e. $2+ \leq \alpha \leq 4-$.
Analogously, for the second term we get for the range $\alphalow{\hat{\beta}'}-2 \leq \alpha \leq \alphahigh{\hat{\beta}'}-2$, i.e. $2+ \leq \alpha \leq 4-$,
\begin{align}
\Vert \LxiOp^{\Lxitimes} \hat{\beta}' \Vert_{W^{k-K-\Lxitimes}_{\alpha+1}(\Dtautimeint)}^2 \lesssim
\timefunc^{\alpha-(\alphahigh{\hat{\beta}'} -2)-2\Lxitimes}
\inienergynoalpha{\ireg}[\setOfFields]
\end{align}
This proves \eqref{eq:varrhotransitionregionesti-varrho} for the range $\alphalow{\widehat{G}_0} \leq \alpha \leq \alphahigh{\widehat{G}_0}$, with $\alphalow{\widehat{G}_0} = 2(\Em{\widehat{G}_0}-1) += 2+$, $\alphahigh{\widehat{G}_0} = 2\El{\widehat{G}_0} + 2- = 4-$, and hence completes the proof of point \ref{point:EnergMora-int} for $\varphi \in \setOfFields\backslash\{\psibase[-2]\}$.

It remains to consider point \ref{point:pointwise-int}. For $\varphi \in \setOfFields\backslash\{\psibase[-2]\}$, this follows from estimate \eqref{eq:pointwisevarphigeneralinterior} with $\alpha=\alphalow{\varphi}$.
\end{proof}


\subsection{Proof of the main theorems \ref{thm:mainintro} and \ref{thm:BEAMintro}}
\label{sec:ProofMainThms}
This section completes the proofs of the theorems from the introduction.

\begin{proof}[Proof of theorem \ref{thm:BEAMintro}]
If $\delta g$ satisfies the linearized Einstein equation in the outgoing radiation gauge and satisfies the basic decay condition of definition \ref{def:BEAM-intro}, then it corresponds to an outgoing BEAM solution of the linearized Einstein equation as in definition \ref{def:BEAMLinearisedEinsteinSolution}. Thus, lemmas \ref{lem:allquantitiesextestimates} and \ref{lem:EnerMoraAllquantitiesestiint} can be applied. These yield that, for $\varphi\in\{\widehat{G}_i\}_{i=0}^2$, and $\ireg \in \Naturals$ sufficiently large,
\begin{align}
\abs{\varphi}^2
\lesssim{}&\begin{cases}
r^{-2\Em{\varphi}}\timefunc^{2\Em{\varphi}-3-2\El{\varphi}+}
\inienergynoalpha{\ireg-2}[\setOfFields]
&\text{if $r\geq t$},\\
r^{-2\max\{\Em{\varphi}-1,0\}-}\timefunc^{-(2\El{\varphi}+3)+2\max\{\Em{\varphi}-1,0\}+}\inienergynoalpha{\ireg-2}[\setOfFields]
&\text{if $r\leq t$} .
\end{cases}
\label{eq:decayRateOfDeboostedMetricComponents}
\end{align}
Equation \eqref{eq:BoostWeightZeroQuantities} relates the $\widehat{G}_i$ to the $G_{i0'}$ by a rescaling by some rational factor that grows as a particular power in $r$, which will be denoted by $p[\varphi]$ in this paragraph. From definition \ref{def:alpha1Values} and equation \eqref{eq:BoostWeightZeroQuantities}, the relevant parameters are given in the following table:
\begin{center}
\begin{tabular}{l|ccc}
$\varphi$           &$m$&$\ell$&$p$\\
\hline
$\widehat{G}_2$&$0$&$1$     &$1$\\
$\widehat{G}_1$&$0$&$0$     &$2$\\
$\widehat{G}_0$&$2$&$1$     &$1$
\end{tabular} .
\end{center}
Thus, one finds, in the exterior region,
\begin{subequations}
\label{eq:metricDecayExt}
\begin{align}
\abs{G_{20'}}^2
\lesssim{}& r^{-2}\timefunc^{-5+}\inienergynoalpha{\ireg-2}[\setOfFields] ,\\
\abs{\uplambda^{-1}G_{10'}}^2
\lesssim{}& r^{-4}\timefunc^{-3+}\inienergynoalpha{\ireg-2}[\setOfFields] ,\\
\abs{\uplambda^{-2}G_{00'}}^2
\lesssim{}& r^{-6}\timefunc^{-1+}\inienergynoalpha{\ireg-2}[\setOfFields] ,
\end{align}
\end{subequations}
and, in the interior region,
\begin{subequations}
\label{eq:metricDecayInt}
\begin{align}
\abs{G_{20'}}^2
\lesssim{}& r^{-2-} \timefunc^{-5+}\inienergynoalpha{\ireg-2}[\setOfFields] ,\\
\abs{\uplambda^{-1}G_{10'}}^2
\lesssim{}& r^{-4-} \timefunc^{-3+}\inienergynoalpha{\ireg-2}[\setOfFields] ,\\
\abs{\uplambda^{-2}G_{00'}}^2
\lesssim{}& r^{-4-}\timefunc^{-3+}\inienergynoalpha{\ireg-2}[\setOfFields] .
\end{align}
\end{subequations}

Recall that the fields $\varphi \in \setOfFields$ are defined in definition \ref{def:setOfFieldsAndInitialDataNorm} in terms of the linearized metric $\delta g_{ab}$ and its derivatives up to second order as specified in section \ref{sec:connectcomp}. From these definitions, the definition of the initial data norm $\inienergynoalpha{\ireg-2}[\setOfFields]$ in definition \ref{def:setOfFieldsAndInitialDataNorm}, and the definition of $\Vert \delta g \Vert_{H^{\ireg}_{7}(\Sigma_{\initial})}^2$ in equation \eqref{eq:Hkbeta-def}, it is straightforward to verify that
\begin{align}
\inienergynoalpha{\ireg-2}[\setOfFields] \lesssim \Vert \delta g \Vert_{H^{\ireg}_{7}(\Sigma_{\initial})}^2 .
\end{align}
This completes the proof of theorem \ref{thm:BEAMintro}.
\end{proof}

\begin{proof}[Proof of theorem \ref{thm:mainintro}]
From \cite{2017arXiv170807385M}, it is known that, for $|a|/M$ sufficiently small and $\ireg\in \Naturals$ sufficiently large, the BEAM condition for $\psibase[-2]$ from definition \ref{def:BEAM} holds, and, also, the BEAM condition \ref{point:BEAMspin+2withDelta} for $\psibase[+2]$ from definition \ref{ass:BEAMspin+2} holds. Moreover, there is a bound $\inienergyplustwoForDescent{\ireg-2}\lesssim \Vert \delta g \Vert_{H^{\ireg}_{7}(\Sigma_{\initial})}^2$, which is finite by assumption. 
Thus, theorem \ref{thm:DecayEstimatesSpin+2} implies, for $|a|/M$ sufficiently small, the pointwise condition \ref{condition:spin+2goestozero} for $\psibase[+2]$ from definition \ref{ass:BEAMspin+2} holds. The BEAM condition from definition \ref{def:BEAM} for $\psibase[-2]$ and the pointwise condition \ref{condition:spin+2goestozero} from definition \ref{ass:BEAMspin+2} for $\psibase[+2]$ imply the basic decay conditions of definition \ref{def:BEAM-intro}. 
We now define 
\begin{equation} \label{eq:lamdg}
|\delta g|^2 = |\uplambda^{-2}G_{00'}|^2 + |\uplambda^{-1}G_{10'}|^2 + |G_{20'}|^2 . 
\end{equation}
Since the basic decay conditions of definition \ref{def:BEAM-intro} holds for $|a|/M$ sufficiently small, theorem \ref{thm:BEAMintro} immediately implies 
\begin{align}
\label{eq:mainGHP}
|\delta g|
\lesssim{}& r^{-1} \timefunc^{-3/2+\epsilon} \Vert \delta g \Vert_{H^{\ireg}_{7}(\Sigma_{\initial})}^2 ,
\end{align}
which completes the proof. 
\end{proof}


\subsection*{Acknowledgements}
We are grateful to Steffen Aksteiner, Jinhua Wang, and Bernard F. Whiting for helpful discussions. Lars Andersson thanks the Royal Institute of Technology (KTH), Stockholm, where a significant portion of the work was done,  for hospitality and support. Pieter Blue thanks the Albert Einstein Institute (AEI) Potsdam for hospitality and support. 
Siyuan Ma is supported by the National Natural Science Foundation of China under Grant No. 12288201 and the ERC grant ERC-2016 CoG 4399 725589 EPGR.

\appendix

\section{Field equations} \label{sec:fieldeq} 
\subsection{Linearized Einstein vacuum equations}
\label{sec:fieldeqGeneral} 
In this appendix, we give the component form in GHP notation of the linearized Einstein field equations which are used in this paper.
The structure equations \eqref{eq:QoppaToGCov} in general take the form
\begin{subequations}\label{eq:Aspincoeff}
\begin{align}
\tilde{\beta}={}&\tfrac{1}{4} (\tho + 2 \rho -  \bar{\rho})G_{12'}
 -  \tfrac{1}{4} (\thop + \rho' - 2 \bar{\rho}')G_{01'}
 -  \tfrac{1}{4} (\edt + 2 \tau - 2 \bar{\tau}')G_{11'}
 + \tfrac{1}{4} (\edtp -  \bar{\tau} + \tau')G_{02'}
 + \tfrac{1}{16} \edt \slashed{G},\\
\tilde{\beta}'={}&- \tfrac{1}{4} (\tho + \rho - 2 \bar{\rho})G_{21'}
 + \tfrac{1}{4} (\thop + 2 \rho' -  \bar{\rho}')G_{10'}
 + \tfrac{1}{4} (\edt + \tau -  \bar{\tau}')G_{20'}
 -  \tfrac{1}{4} (\edtp - 2 \bar{\tau} + 2 \tau')G_{11'}\nonumber\\
& + \tfrac{1}{16} \edtp \slashed{G},\label{eq:Abetap}\\
\tilde{\epsilon}={}&\tfrac{1}{4} (\tho + 2 \rho - 2 \bar{\rho})G_{11'}
 -  \tfrac{1}{4} (\thop + \rho' -  \bar{\rho}')G_{00'}
 -  \tfrac{1}{4} (\edt + 2 \tau -  \bar{\tau}')G_{10'}
 + \tfrac{1}{4} (\edtp - 2 \bar{\tau} + \tau')G_{01'}
 + \tfrac{1}{16} \tho \slashed{G},\label{eq:Aepsilon}\\
\tilde{\epsilon}'={}&- \tfrac{1}{4} (\tho + \rho -  \bar{\rho})G_{22'}
 + \tfrac{1}{4} (\thop + 2 \rho' - 2 \bar{\rho}')G_{11'}
 + \tfrac{1}{4} (\edt + \tau - 2 \bar{\tau}')G_{21'}
 -  \tfrac{1}{4} (\edtp -  \bar{\tau} + 2 \tau')G_{12'}\nonumber\\
& + \tfrac{1}{16} \thop \slashed{G},\label{eq:Aepsilonp}\\
\tilde{\kappa}={}&\tfrac{1}{2} (\tho - 2 \bar{\rho})G_{01'}
 -  \tfrac{1}{2} (\edt -  \bar{\tau}')G_{00'},\\
\tilde{\kappa}'={}&\tfrac{1}{2} (\thop - 2 \bar{\rho}')G_{21'}
 -  \tfrac{1}{2} (\edtp -  \bar{\tau})G_{22'},\label{eq:Akappap}\\
\tilde{\rho}={}&- \tfrac{1}{2} G_{00'} \rho'
 + \tfrac{1}{2} G_{01'} \tau'
 + \tfrac{1}{2} (\tho - 2 \bar{\rho})G_{11'}
 -  \tfrac{1}{2} (\edt -  \bar{\tau}')G_{10'}
 -  \tfrac{1}{8} \tho \slashed{G},\label{eq:Arho}\\
\tilde{\rho}'={}&- \tfrac{1}{2}\rho G_{22'} 
 + \tfrac{1}{2}\tau G_{21'} 
 + \tfrac{1}{2} (\thop - 2 \bar{\rho}')G_{11'}
 -  \tfrac{1}{2} (\edtp -  \bar{\tau})G_{12'}
 -  \tfrac{1}{8} \thop \slashed{G},\label{eq:Arhop}\\
\tilde{\sigma}={}&\tfrac{1}{2} (\tho -  \bar{\rho})G_{02'}
 -  \tfrac{1}{2} (\edt - 2 \bar{\tau}')G_{01'},\\
\tilde{\sigma}'={}&\tfrac{1}{2} (\thop -  \bar{\rho}')G_{20'}
 -  \tfrac{1}{2} (\edtp - 2 \bar{\tau})G_{21'},\label{eq:Asigmap}\\
\tilde{\tau}={}&- \tfrac{1}{2}\rho' G_{01'} 
 + \tfrac{1}{2}\tau' G_{02'} 
 + \tfrac{1}{2} (\tho -  \bar{\rho})G_{12'}
 -  \tfrac{1}{2} (\edt - 2 \bar{\tau}')G_{11'}
 -  \tfrac{1}{8} \edt \slashed{G},\label{eq:Atau}\\
\tilde{\tau}'={}&- \tfrac{1}{2}\rho G_{21'} 
 + \tfrac{1}{2}\tau G_{20'} 
 + \tfrac{1}{2} (\thop -  \bar{\rho}')G_{10'}
 -  \tfrac{1}{2} (\edtp - 2 \bar{\tau})G_{11'}
 -  \tfrac{1}{8} \edtp \slashed{G}.\label{eq:Ataup}
\end{align}
\end{subequations}
The linearized vacuum Einstein equations \eqref{eq:DivQop11Cov} and \eqref{eq:CurlDgQop31Cov}  are
\begin{subequations}
\begin{align}
0={}&- (\tho -  \rho -  \bar{\rho})\tilde{\epsilon}'
 + (\tho -  \rho -  \bar{\rho})\tilde{\rho}'
 -  (\thop -  \rho' -  \bar{\rho}')\tilde{\epsilon}
 + (\thop -  \rho' -  \bar{\rho}')\tilde{\rho}
 + (\edt -  \tau -  \bar{\tau}')\tilde{\beta}'\nonumber\\
& -  (\edt -  \tau -  \bar{\tau}')\tilde{\tau}'
 + (\edtp -  \bar{\tau} -  \tau')\tilde{\beta}
 -  (\edtp -  \bar{\tau} -  \tau')\tilde{\tau},\label{eq:AEinsteinA}\\
0={}&(\thop -  \rho')\tilde{\sigma}
 -  (\edt -  \tau)\tilde{\tau}
 -  \tfrac{1}{2} G_{02'} \Psi_{2}
 + 2 \tau \tilde{\beta},\\
0={}&(\thop -  \rho')\tilde{\beta}
 + (\edt -  \tau)\tilde{\epsilon}'
 + G_{12'} \Psi_{2}
 -  \tau \tilde{\rho}' 
 -  \rho' \tilde{\tau},\label{eq:AEinsteinC}\\
0={}&- (\thop -  \rho')\tilde{\rho}'
 + (\edt -  \tau)\tilde{\kappa}'
 -  \tfrac{1}{2} \Psi_{2} G_{22'} 
 + 2 \tilde{\epsilon}' \rho',\\
0={}&- (\tho -  \rho)\tilde{\rho}
 + (\edtp -  \tau')\tilde{\kappa}
 -  \tfrac{1}{2} \Psi_{2} G_{00'} 
 + 2 \tilde{\epsilon} \rho,\\
0={}&- \tfrac{1}{2} (\tho -  \rho + \bar{\rho})\tilde{\tau}
 + \tfrac{1}{2} (\thop -  \rho' + \bar{\rho}')\tilde{\kappa}
 -  \tfrac{1}{2} (\edt -  \tau + \bar{\tau}')\tilde{\rho}
 + \tfrac{1}{2} (\edtp + \bar{\tau} -  \tau')\tilde{\sigma}
 -  \tfrac{1}{2} \Psi_{2} G_{01'}
 + \tilde{\beta} \rho
 + \tilde{\epsilon} \tau,\\
0={}&(\tho -  \rho)\tilde{\beta}'
 + (\edtp -  \tau')\tilde{\epsilon}
 + \Psi_{2} G_{10'} 
 -  \tau' \tilde{\rho} 
 -  \rho \tilde{\tau}',\\
0={}&\tfrac{1}{2} (\tho -  \rho + \bar{\rho})\tilde{\epsilon}'
 + \tfrac{1}{2} (\thop -  \rho' + \bar{\rho}')\tilde{\epsilon}
 + \tfrac{1}{2} (\edt -  \tau + \bar{\tau}')\tilde{\beta}'
 + \tfrac{1}{2} (\edtp + \bar{\tau} -  \tau')\tilde{\beta}
 + \Psi_{2} G_{11'} 
 -  \tfrac{1}{2} \rho \tilde{\rho}'
 -  \tfrac{1}{2} \rho' \tilde{\rho}\nonumber\\
& -  \tfrac{1}{2} \tau \tilde{\tau}'
 -  \tfrac{1}{2} \tau' \tilde{\tau},\label{eq:AEinsteinH}\\
0={}&(\tho -  \rho)\tilde{\sigma}'
 -  (\edtp -  \tau')\tilde{\tau}'
 -  \tfrac{1}{2}\Psi_{2} G_{20'} 
 + 2\tau' \tilde{\beta}',\label{eq:AEinsteinI}\\
0={}&\tfrac{1}{2} (\tho -  \rho + \bar{\rho})\tilde{\kappa}'
 -  \tfrac{1}{2} (\thop -  \rho' + \bar{\rho}')\tilde{\tau}'
 + \tfrac{1}{2} (\edt -  \tau + \bar{\tau}')\tilde{\sigma}'
 -  \tfrac{1}{2} (\edtp + \bar{\tau} -  \tau')\tilde{\rho}'
 -  \tfrac{1}{2} \Psi_{2} G_{21'} 
 + \rho' \tilde{\beta}' 
 + \tau' \tilde{\epsilon}'.\label{eq:AEinsteinJ}
\end{align}
\end{subequations}
The remaining Ricci relations \eqref{eq:CurlQop31Cov} are
\begin{subequations}
\begin{align}
\vartheta \Psi_{0}={}&(\tho -  \bar{\rho})\tilde{\sigma}
 -  (\edt -  \bar{\tau}')\tilde{\kappa},\\
\vartheta \Psi_{1}={}&\tfrac{1}{2} (\tho + \rho -  \bar{\rho})\tilde{\beta}
 + \tfrac{1}{4} (\tho + \rho -  \bar{\rho})\tilde{\tau}
 -  \tfrac{1}{4} (\thop + 3 \rho' -  \bar{\rho}')\tilde{\kappa}
 -  \tfrac{1}{2} (\edt + \tau -  \bar{\tau}')\tilde{\epsilon}
 -  \tfrac{1}{4} (\edt + \tau -  \bar{\tau}')\tilde{\rho}\nonumber\\
& + \tfrac{1}{4} (\edtp -  \bar{\tau} + 3 \tau')\tilde{\sigma},\\
\vartheta \Psi_{2}={}&- \tfrac{1}{4} \Psi_{2} \slashed{G}
 -  \tfrac{1}{3} (\tho + 2 \rho -  \bar{\rho})\tilde{\epsilon}'
 -  \tfrac{1}{6} (\tho + 2 \rho -  \bar{\rho})\tilde{\rho}'
 -  \tfrac{1}{3} (\thop + 2 \rho' -  \bar{\rho}')\tilde{\epsilon}
 -  \tfrac{1}{6} (\thop + 2 \rho' -  \bar{\rho}')\tilde{\rho}\nonumber\\
& + \tfrac{1}{3} (\edt + 2 \tau -  \bar{\tau}')\tilde{\beta}'
 + \tfrac{1}{6} (\edt + 2 \tau -  \bar{\tau}')\tilde{\tau}'
 + \tfrac{1}{3} (\edtp -  \bar{\tau} + 2 \tau')\tilde{\beta}
 + \tfrac{1}{6} (\edtp -  \bar{\tau} + 2 \tau')\tilde{\tau},\\
\vartheta \Psi_{3}={}&- \tfrac{1}{4} (\tho + 3 \rho -  \bar{\rho})\tilde{\kappa}'
 + \tfrac{1}{2} (\thop + \rho' -  \bar{\rho}')\tilde{\beta}'
 + \tfrac{1}{4} (\thop + \rho' -  \bar{\rho}')\tilde{\tau}'
 + \tfrac{1}{4} (\edt + 3 \tau -  \bar{\tau}')\tilde{\sigma}'\nonumber\\
& -  \tfrac{1}{2} (\edtp -  \bar{\tau} + \tau')\tilde{\epsilon}'
 -  \tfrac{1}{4} (\edtp -  \bar{\tau} + \tau')\tilde{\rho}',\\
\vartheta \Psi_{4}={}&(\thop -  \bar{\rho}')\tilde{\sigma}'
 -  (\edtp -  \bar{\tau})\tilde{\kappa}'. \label{eq:APsi4}
\end{align}
\end{subequations}
We also have the commutator relations \eqref{eq:DivQop31Cov}
\begin{subequations}
\begin{align}
0={}&2 (\tho -  \rho -  \bar{\rho})\tilde{\beta}
 -  (\tho + \rho -  \bar{\rho})\tilde{\tau}
 + (\thop -  \rho' -  \bar{\rho}')\tilde{\kappa}
 - 2 (\edt -  \tau -  \bar{\tau}')\tilde{\epsilon}
 + (\edt + \tau -  \bar{\tau}')\tilde{\rho}
 -  (\edtp -  \bar{\tau} -  \tau')\tilde{\sigma},\\
0={}&- (\tho -  \bar{\rho})\tilde{\rho}'
 + (\thop -  \bar{\rho}')\tilde{\rho}
 + (\edt -  \bar{\tau}')\tilde{\tau}'
 -  (\edtp -  \bar{\tau})\tilde{\tau}
 + 2 \rho \tilde{\epsilon}' 
 - 2 \rho' \tilde{\epsilon} 
 - 2 \tau \tilde{\beta}' 
 + 2 \tau' \tilde{\beta},\label{eq:ACommutatorB}\\
0={}&- (\tho -  \rho -  \bar{\rho})\tilde{\kappa}'
 - 2 (\thop -  \rho' -  \bar{\rho}')\tilde{\beta}'
 + (\thop + \rho' -  \bar{\rho}')\tilde{\tau}'
 + (\edt -  \tau -  \bar{\tau}')\tilde{\sigma}'
 + 2 (\edtp -  \bar{\tau} -  \tau')\tilde{\epsilon}'\nonumber\\
& -  (\edtp -  \bar{\tau} + \tau')\tilde{\rho}',\label{eq:ACommutatorC}
\end{align}
\end{subequations}
and reality conditions $\bar{\slashed{\Qop}}{}_{A'A}=\slashed{\Qop}{}_{AA'}$
\begin{align} \label{eq:Areality}
\overline{\tilde{\epsilon}} -  \overline{\tilde{\rho}}={}&\tilde{\epsilon}
 -  \tilde{\rho},&
\overline{\tilde{\beta}} -  \overline{\tilde{\tau}}={}&\tilde{\beta}'
 -  \tilde{\tau}',&
\overline{\tilde{\beta}'} -  \overline{\tilde{\tau}'}={}&\tilde{\beta}
 -  \tilde{\tau},&
\overline{\tilde{\epsilon}'} -  \overline{\tilde{\rho}'}={}&\tilde{\epsilon}'
 -  \tilde{\rho}'.
\end{align}
Furthermore, the linearized vacuum Bianchi equations \eqref{eq:VacuumLinBianchiCov} take the form 
\begin{subequations}
\begin{align}
0={}&(\thop -  \rho')\vartheta \Psi_{0}
 -  (\edt - 4 \tau)\vartheta \Psi_{1}
 -  \tfrac{3}{2} \Psi_{2} \rho G_{02'}
 - 3 \Psi_{2} \tilde{\sigma}
 + \tfrac{3}{2} \Psi_{2} \tau G_{01'},\\
0={}&(\thop - 2 \rho')\vartheta \Psi_{1}
 -  (\edt - 3 \tau)\vartheta \Psi_{2}
 + 3 G_{12'} \Psi_{2} \rho
 + \tfrac{3}{2} \Psi_{2} \rho' G_{01'}
 - 3 G_{11'} \Psi_{2} \tau
 -  \tfrac{3}{2} \Psi_{2} \tau' G_{02'}
 + 3 \Psi_{2} \tilde{\tau},\\
0={}&(\thop - 3 \rho')\vartheta \Psi_{2}
 -  (\edt - 2 \tau)\vartheta \Psi_{3}
 -  \tfrac{3}{2} \Psi_{2} \rho G_{22'}
 - 3 \Psi_{2} \rho' G_{11'}
 - 3 \Psi_{2} \tilde{\rho}'
 + \tfrac{3}{2} \Psi_{2} \tau G_{21'}
 + 3 \Psi_{2} \tau' G_{12'},\\
0={}&(\thop - 4 \rho')\vartheta \Psi_{3}
 -  (\edt -  \tau)\vartheta \Psi_{4}
 + 3 \Psi_{2} \tilde{\kappa}'
 + \tfrac{3}{2} \Psi_{2} \rho' G_{21'}
 -  \tfrac{3}{2} \Psi_{2} \tau' G_{22'},\\
0={}&(\tho - 4 \rho)\vartheta \Psi_{1}
 -  (\edtp -  \tau')\vartheta \Psi_{0}
 + 3 \Psi_{2} \tilde{\kappa}
 + \tfrac{3}{2} \Psi_{2} \rho G_{01'}
 -  \tfrac{3}{2} \Psi_{2} \tau G_{00'},\\
0={}&(\tho - 3 \rho)\vartheta \Psi_{2}
 -  (\edtp - 2 \tau')\vartheta \Psi_{1}
 - 3 \Psi_{2} \rho G_{11'}
 -  \tfrac{3}{2} \Psi_{2} \rho' G_{00'}
 - 3 \Psi_{2} \tilde{\rho}
 + 3 \Psi_{2} \tau G_{10'}
 + \tfrac{3}{2} \Psi_{2} \tau' G_{01'},\\
0={}&(\tho - 2 \rho)\vartheta \Psi_{3}
 -  (\edtp - 3 \tau')\vartheta \Psi_{2}
 + \tfrac{3}{2} \Psi_{2} \rho G_{21'}
 + 3 \Psi_{2} \rho' G_{10'}
 -  \tfrac{3}{2} \Psi_{2} \tau G_{20'}
 - 3 \Psi_{2} \tau' G_{11'}
 + 3 \Psi_{2} \tilde{\tau}',\\
0={}&(\tho -  \rho)\vartheta \Psi_{4}
 -  (\edtp - 4 \tau')\vartheta \Psi_{3}
 -  \tfrac{3}{2} \Psi_{2} \rho' G_{20'}
 - 3 \Psi_{2} \tilde{\sigma}'
 + \tfrac{3}{2} \Psi_{2} \tau' G_{21'}.
\end{align}
\end{subequations}

\subsection{Linearized Einstein field equations in ORG}
\label{sec:LinGraSysteminORG}
A calculation using the relations \eqref{eq:ORGcondGHP} and \eqref{eq:ORGlinconnectioncoeffConseq} yields the following lemma, which we state for completeness. Observe however that the proof of lemma~\ref{lem:TransportSystem} is directly referring to the equations in appendix~\ref{sec:fieldeqGeneral}.

\begin{lemma}
Under the ORG condition the vacuum linearized Einstein equations can be organized as the transport equations 
\begin{subequations}
\label{eq:linearizedEinsteinEquationInORG}
\begin{align}
\thop G_{00'}={}&-4 \tilde{\epsilon}
 + 2 \tilde{\rho}
 - 2 \overline{\tilde{\rho}}
 - 2 G_{10'} \tau
 - 2 G_{01'} \bar{\tau}
 + G_{01'} \tau'
 + G_{10'} \bar{\tau}',\\
(\thop -  \rho')G_{01'}={}&- G_{02'} \bar{\tau}
 + 2 \overline{\tilde{\tau}'},\\
(\thop -  \rho')G_{02'}={}&2 \overline{\tilde{\sigma}'},\\
(\thop -  \bar{\rho}')G_{10'}={}&- G_{20'} \tau
 + 2 \tilde{\tau}',\\
(\thop -  \bar{\rho}')G_{20'}={}&2 \tilde{\sigma}',\\
(\thop -  \rho' + \bar{\rho}')\tilde{\tau}'={}&2 \tilde{\beta}' \rho'
 + (\edt -  \tau + \bar{\tau}')\tilde{\sigma}',\label{eq:linearizedEinsteinEquationInORGtauprime}\\
(\thop - 2 \rho' -  \bar{\rho}')\tilde{\beta}'={}&\rho' \tilde{\tau}'
 -  \bar{\rho}' \tilde{\tau}'
 + (\edt -  \tau)\tilde{\sigma}',\label{eq:linearizedEinsteinEquationInORGbetaprime}\\
(\thop -  \bar{\rho}')\tilde{\sigma}'={}&\vartheta \Psi_{4},\\
(\thop -  \rho' -  \bar{\rho}')\tilde{\rho}={}&\tilde{\epsilon} \rho'
 + \overline{\tilde{\epsilon}} \rho'
 + \tfrac{1}{2} G_{00'} \rho' \bar{\rho}'
 + 2 \tilde{\beta}' \tau
 + \tfrac{1}{2} G_{10'} \bar{\rho}' \tau
 -  G_{01'} \bar{\rho}' \tau'
 -  \tfrac{1}{2} G_{02'} \bar{\tau} \tau'
 + \tau' \overline{\tilde{\tau}'}\nonumber\\
& -  (\edt -  \bar{\tau}')\tilde{\tau}',\\
(\thop + \bar{\rho}')\tilde{\kappa}={}&\tfrac{5}{4} G_{01'} \Psi_{2}
 + \frac{G_{01'} \bar\Psi_{2} \overline{\kappa}_{1'}{}}{4 \kappa_{1}{}}
 - 2 \tilde{\beta} \rho
 -  \tfrac{3}{2} G_{01'} \bar{\rho} \rho'
 - 2 \tilde{\epsilon} \tau
 -  \tfrac{1}{2} G_{00'} \bar{\rho}' \tau
 + \tfrac{1}{2} G_{02'} \rho \tau'\nonumber\\
& + G_{02'} \bar{\rho} \tau'
 + \tfrac{1}{2} \tau' (\edt -  \tau -  \bar{\tau}')G_{01'}
 + (\edt -  \tau + \bar{\tau}')\tilde{\rho}
 -  (\edtp + \bar{\tau} - 2 \tau')\tilde{\sigma}
 -  \tfrac{1}{2} \rho' \edt G_{00'},\\
(\thop - 2 \rho')\tilde{\epsilon}={}&\tilde{\beta}' \tau
 -  \tilde{\beta} \bar{\tau}
 -  \tilde{\beta} \tau'
 -  \tfrac{1}{2} G_{01'} \rho' \tau'
 + \tfrac{1}{2} G_{02'} \tau'^2
 -  \tilde{\beta}' \bar{\tau}'
 -  (\edt -  \tau -  \bar{\tau}')\tilde{\tau}',\\
(\thop -  \rho')\tilde{\beta}={}&- \tfrac{1}{2} G_{01'} \rho'^2
 + \tfrac{1}{2} G_{02'} \rho' \tau',\\
(\thop -  \rho')\tilde{\sigma}={}&\tfrac{3}{4} G_{02'} \Psi_{2}
 -  \frac{G_{02'} \bar\Psi_{2} \overline{\kappa}_{1'}{}}{4 \kappa_{1}{}}
 + \tfrac{1}{2} G_{02'} \rho \rho'
 -  \tfrac{1}{2} G_{02'} \bar{\rho} \rho'
 - 2 \tilde{\beta} \tau
 -  \tfrac{1}{2} \rho' (\edt + \tau)G_{01'}\nonumber\\
& + \tfrac{1}{2} \tau' \edt G_{02'},\\
(\thop - 4 \rho')\vartheta \Psi_{3}={}&(\edt -  \tau)\vartheta \Psi_{4},\\
(\thop - 3 \rho')\vartheta \Psi_{2}={}&(\edt - 2 \tau)\vartheta \Psi_{3},\\
(\thop - 2 \rho')\vartheta \Psi_{1}={}&(\edt - 3 \tau)\vartheta \Psi_{2},\\
(\thop -  \rho')\vartheta \Psi_{0}={}&\tfrac{3}{2} G_{02'} \Psi_{2} \rho
 + 3 \Psi_{2} \tilde{\sigma}
 -  \tfrac{3}{2} G_{01'} \Psi_{2} \tau
 + (\edt - 4 \tau)\vartheta \Psi_{1},
\end{align}
\end{subequations}
together with the set
\begin{subequations}
\begin{align}
\tilde{\beta}={}&- \tfrac{1}{2} G_{01'} \rho'
 + \tfrac{1}{2} G_{01'} \bar{\rho}'
 -  \tfrac{1}{2} \overline{\tilde{\tau}'}
 + \tfrac{1}{4} (\edtp + \tau')G_{02'},\\
\tilde{\beta}'={}&\tfrac{1}{2} G_{10'} \rho'
 + \tfrac{1}{2} \tilde{\tau}'
 + \tfrac{1}{4} (\edt -  \bar{\tau}')G_{20'},\\
\tilde{\kappa}={}&\tfrac{1}{2} (\tho - 2 \bar{\rho})G_{01'}
 -  \tfrac{1}{2} (\edt -  \bar{\tau}')G_{00'},\\
\tilde{\rho}={}&- \tfrac{1}{2} G_{00'} \rho'
 + \tfrac{1}{2} G_{01'} \tau'
 -  \tfrac{1}{2} (\edt -  \bar{\tau}')G_{10'},\\
\tilde{\sigma}={}&\tfrac{1}{2} (\tho -  \bar{\rho})G_{02'}
 -  \tfrac{1}{2} (\edt - 2 \bar{\tau}')G_{01'},\\
\tilde{\tau}={}&- \tfrac{1}{2} G_{01'} \rho'
 + \tfrac{1}{2} G_{02'} \tau',\\
(\tho -  \rho)\tilde{\rho}={}&- \tfrac{1}{2} G_{00'} \Psi_{2}
 + 2 \tilde{\epsilon} \rho
 + (\edtp -  \tau')\tilde{\kappa},\\
(\tho - 2 \rho -  \bar{\rho})\tilde{\beta}={}&- \tfrac{1}{2} G_{01'} (\Psi_{2} + \rho \rho' -  \bar{\rho} \rho')
 + \tilde{\kappa} \bar{\rho}'
 + \tfrac{1}{2} G_{02'} (\rho -  \bar{\rho}) \tau'
 + (\edt -  \bar{\tau}')\tilde{\epsilon}
 + (\edtp -  \tau')\tilde{\sigma}
 -  \edt \tilde{\rho},\\
(\tho -  \rho)\tilde{\beta}'={}&- G_{10'} \Psi_{2}
 + \tilde{\rho} \tau'
 + \rho \tilde{\tau}'
 -  (\edtp -  \tau')\tilde{\epsilon},\\
(\tho -  \rho)\tilde{\sigma}'={}&\tfrac{1}{2} G_{20'} \Psi_{2}
 - 2 \tilde{\beta}' \tau'
 + (\edtp -  \tau')\tilde{\tau}', \label{eq:Athosigma'} \\
0={}&\tilde{\epsilon} (\rho' + \bar{\rho}')
 -  \rho' \tilde{\rho}
 -  (\edt -  \bar{\tau}')\tilde{\tau}'
 + (\edtp - 2 \tau')\tilde{\beta}
 + \edt \tilde{\beta}',\\
\vartheta \Psi_{0}={}&(\tho -  \bar{\rho})\tilde{\sigma}
 -  (\edt -  \bar{\tau}')\tilde{\kappa},\\
\vartheta \Psi_{1}={}&- \tilde{\kappa} \rho'
 + \tilde{\sigma} \tau'
 + (\tho -  \bar{\rho})\tilde{\beta}
 -  (\edt -  \bar{\tau}')\tilde{\epsilon},\\
\vartheta \Psi_{2}={}&-2 \tilde{\epsilon} \rho'
 + 2 \tilde{\beta} \tau'
 + (\edt -  \bar{\tau}')\tilde{\tau}',\\
\vartheta \Psi_{3}={}&2 \tilde{\beta}' \rho'
 + (\rho' -  \bar{\rho}') \tilde{\tau}'
 + \edt \tilde{\sigma}', \label{eq:APsi3-ORG}\\
0={}&3 \Psi_{2} \tilde{\kappa}
 + \tfrac{3}{2} G_{01'} \Psi_{2} \rho
 -  \tfrac{3}{2} G_{00'} \Psi_{2} \tau
 + (\tho - 4 \rho)\vartheta \Psi_{1}
 -  (\edtp -  \tau')\vartheta \Psi_{0},\\
0={}&- \tfrac{3}{2} G_{00'} \Psi_{2} \rho'
 - 3 \Psi_{2} \tilde{\rho}
 + 3 G_{10'} \Psi_{2} \tau
 + \tfrac{3}{2} G_{01'} \Psi_{2} \tau'
 + (\tho - 3 \rho)\vartheta \Psi_{2}
 -  (\edtp - 2 \tau')\vartheta \Psi_{1},\\
0={}&3 G_{10'} \Psi_{2} \rho'
 -  \tfrac{3}{2} G_{20'} \Psi_{2} \tau
 + 3 \Psi_{2} \tilde{\tau}'
 + (\tho - 2 \rho)\vartheta \Psi_{3}
 -  (\edtp - 3 \tau')\vartheta \Psi_{2},\\
0={}&- \tfrac{3}{2} G_{20'} \Psi_{2} \rho'
 - 3 \Psi_{2} \tilde{\sigma}'
 + (\tho -  \rho)\vartheta \Psi_{4}
 -  (\edtp - 4 \tau')\vartheta \Psi_{3}, \label{eq:AG2ethPsi3}
\end{align}
\end{subequations}
and the reality conditions
\begin{align} 
\overline{\tilde{\epsilon}} -  \overline{\tilde{\rho}}={}&\tilde{\epsilon}
 -  \tilde{\rho},&
\overline{\tilde{\beta}} -  \overline{\tilde{\tau}}={}&\tilde{\beta}'
 -  \tilde{\tau}',&
\overline{\tilde{\beta}'} -  \overline{\tilde{\tau}'}={}&\tilde{\beta}
 -  \tilde{\tau}.
\end{align}
\end{lemma}

\section{Linearized parameters in ORG} \label{sec:linpara}
In this appendix, we present examples of linearized metrics corresponding to varying the parameters $M$ and $a$ in ORG. 
The fact that we obtain decay estimates in theorems \ref{thm:mainintro} and \ref{thm:BEAMintro} imply that the class of initial data we consider must exclude such solutions. As noted in remark \ref{rem:1.7}, the formulas for the linear $M$ and $a$ perturbations illustrate the fact that these perturbations fall off too slowly for the initial data norm to be finite. 

\subsection{Linearized mass} \label{sec:LinMass}
Performing a variation $\delta M$ of the mass parameter of the Kerr metric in Eddington-Finkelstein coordinates yields 
\begin{align}
\delta g_{ab}={}&- \frac{4 n_{a} n_{b} r}{\Sigma}\delta M,
\end{align}
which satisfies the ORG condition. 
Thus, we have in the Znajek tetrad, the only non-vanishing metric component is $G_{00'}=-4 r \Sigma^{-1}\delta M$.
The only non-vanishing components of the linearized connection, as in equation \eqref{eq:tildespincoeff}, and linearized curvature are
\begin{align}
\tilde{\epsilon}={}&\frac{1}{9 \sqrt{2} \kappa_{1}{}^2}\delta M,&
\tilde{\kappa}={}&\frac{i \sqrt{2} a r \sin\theta}{9 \kappa_{1}{}^2 \Sigma}\delta M,&
\tilde{\rho}={}&- \frac{\sqrt{2} r}{3 \kappa_{1}{} \Sigma}\delta M,&
\vartheta \Psi_{2}={}&\frac{\delta M}{27 \kappa_{1}{}^3}.
\end{align}
The rescaled metric components are
\begin{align}
\widehat{G}_2={}&0,&
\widehat{G}_1={}&0,&
\widehat{G}_0={}&- \tfrac{2}{81}\delta M. \label{eq:G0hatlinmass} 
\end{align}

\subsection{Linearized angular momentum} 
Performing a variation $\delta a$ of the angular momentum parameter per unit mass $a$ of the Kerr metric in Eddington-Finkelstein coordinates and transforming to ORG gauge\footnote{The generator for the transformation is $\nu_a=- \frac{\sqrt{2} a  r \sin^2\theta}{\Sigma}n_{a}
- \tfrac{i}{\sqrt{2}} \sin\theta m_{a}
 +  \tfrac{i}{\sqrt{2}} \sin\theta\overline{m}_ {a}
 = - a  \cos\theta \sin\theta(d\theta)_ {a} 
 -   r \sin^2\theta(d\phi)_{a}.$} yields in the Znajek tetrad 
 the non-vanishing components
\begin{align}
G_{00'}={}&\frac{4 M a (1 + \cos^2\theta) r}{\Sigma^2}\delta a,&
G_{01'}={}&- \frac{2i M r \sin\theta}{3 \overline{\kappa}_{1'}{} \Sigma}\delta a,&
G_{10'}={}&\frac{2i M r \sin\theta}{3 \kappa_{1}{} \Sigma}\delta a.
\label{eq:LinAngMomentumG}
\end{align}
The non-vanishing components of the linearized connection and curvature are
\begin{subequations}
\begin{align}
\tilde{\beta}={}&\frac{i M  \sin\theta}{6 \sqrt{2} \kappa_{1}{} \Sigma}\delta a,&
\tilde{\beta}'={}&\frac{i M \sin\theta (\kappa_{1}{} + 2 \overline{\kappa}_{1'}{})}{6 \sqrt{2} \kappa_{1}{}^2 \Sigma}\delta a,\\
\tilde{\tau}={}&- \frac{i M  r \sin\theta}{\sqrt{2} \Sigma^2}\delta a,&
\tilde{\tau}'={}&\frac{i M  \sin\theta}{27 \sqrt{2} \kappa_{1}{}^3}\delta a,\\
\tilde{\sigma}={}&- \frac{M a  r \sin^2\theta}{3 \sqrt{2} \kappa_{1}{} \Sigma^2}\delta a,&
\vartheta \Psi_{1}={}&\frac{i M  (a^2 + r^2) \sin\theta}{486 \kappa_{1}{}^6}\delta a,\\
\vartheta \Psi_{2}={}&\frac{M  (a + i r\cos\theta)}{81 \kappa_{1}{}^5}\delta a,&
\vartheta \Psi_{3}={}&- \frac{i M  \sin\theta}{54 \kappa_{1}{}^4}\delta a.&
\end{align}
\end{subequations}
The rescaled metric components are
\begin{align}
\widehat{G}_2={}&0,&
\widehat{G}_1={}&- \tfrac{1}{81}i \sqrt{2} M  \sin\theta\delta a,&
\widehat{G}_0={}&\frac{2 M a  (1 + \cos^2\theta)}{81 \Sigma}\delta a.
\label{eq:LinAngMomentumGhat}
\end{align}

\section{An alternative form of the Teukolsky equations} \label{sec:DeBoostingTME}
In this appendix, we derive a deboosted, rescaled version of the Teukolsky equations we are using in the paper. Concretely, we prove in this appendix the Teukolsky equations \eqref{eq:TeukolskyRegular-2} and \eqref{eq:TeukolskyRegular+2} in Lemma \ref{lem:TeukolskyRegular}.

We begin by defining a class of Teukolsky operators.

\begin{definition}
\index{S5squareSpqrs@$\squareS_{p,q,j,k}$}
For a field $\phi$ with weight $\{p,q\}$ and a set of parameters $(j,k)\in\Integers^2$, we define the Teukolsky operator
\begin{align}
\squareS_{p,q,j,k}(\phi)={}&18 \kappa_{1}{} \bar{\kappa}_{1'}{} \bigl(\tho -  (p -  j) \rho -  (1 + q -  k) \bar{\rho}\bigr)\bigl(\thop + (j-1) \rho ' + k \bar{\rho}'\bigr)\phi\nonumber\\
& - 18 \kappa_{1}{} \bar{\kappa}_{1'}{} \bigl(\edt -  (p -  j) \tau + (k - 1) \bar{\tau}'\bigr)\bigl(\edtp -  (q -  k) \bar{\tau} + (j-1) \tau '\bigr)\phi\nonumber\\
& - 9 (p-2) q \Psi_{2} \kappa_{1}{}^2 \phi
 - 9 (p-2) (p -1) \Psi_{2} \kappa_{1}{} \bar{\kappa}_{1'}{} \phi
 - 9 q^2 \bar\Psi_{2} \kappa_{1}{} \bar{\kappa}_{1'}{} \phi\nonumber\\
& - 9 (p-1) q \bar\Psi_{2} \bar{\kappa}_{1'}{}^2 \phi
 + 18 q \kappa_{1}{} \bar{\kappa}_{1'}{} (\bar{\rho} \rho ' -  \bar{\rho} \bar{\rho}' -  \bar{\tau} \bar{\tau}' + \tau ' \bar{\tau}') \phi .
\end{align}
\end{definition}
The equations \eqref{eq:TME:spin-2} can be written in terms of $\squareS_{p,q,j,k}$ by moving $\kappa_{1}$ out and using GHP commutators in the second equation. This yields
\begin{subequations}
\label{eq:TMEDeBoostedStep0}
\begin{align}
\squareS_{4,0,0,0}(\vartheta \Psi_{0})={}&0,\\
\squareS_{-4,0,-4,0}(\vartheta \Psi_{4})={}&0.
\end{align}
\end{subequations}

We next rescale and deboost these equations.
\begin{lemma}
The Teukolsky equations \eqref{eq:TME:spin-2} can be written in terms of the boost-weight zero quantities $\uplambda^{-2}\kappa_{1}{}^4\vartheta \Psi_{0} $ and $\uplambda^2 \vartheta \Psi_{4}$ as
\begin{subequations}
\label{eq:TMEDeBoostedStep1}
\begin{align}
\squareS_{2,-2,2,-2}(\uplambda^{-2}\kappa_{1}{}^4\vartheta \Psi_{0})={}&-8 r\mathcal{L}_{\xi}{}(\uplambda^{-2}\kappa_{1}{}^4\vartheta \Psi_{0})
 - 4 (M -  r)\YOp(\uplambda^{-2}\kappa_{1}{}^4\vartheta \Psi_{0}),\\
 \label{eq:TMEDeBoostedStep1:Psi4}
\squareS_{-2,2,-2,2}(\uplambda^2 \vartheta \Psi_{4})={}&8 r\mathcal{L}_{\xi}{}(\uplambda^2 \vartheta \Psi_{4})
 + 4 (M -  r)\YOp(\uplambda^2 \vartheta \Psi_{4}).
\end{align}
\end{subequations}
\end{lemma}
\begin{proof}
 Within this proof, we take $(p,q,j,k)\in\Integers^4$, $n\in\Integers$, and a GHP scalar with type $\{p,q\}$. 
Observing the scaling relation 
\begin{align}
\squareS_{p,q,j,k}(\kappa_{1}{}^n \varphi)={}&\kappa_{1}{}^n\squareS_{p,q, j - n,k}(\varphi),
\end{align}
we deduce the following alternative form of the Teukolsky equations \eqref{eq:TMEDeBoostedStep0} 
\begin{subequations}
\label{eq:TMEJustRescaled}
\begin{align}
\squareS_{4,0,4,0}(\kappa_{1}{}^4 \vartheta \Psi_{0})={}&0,\\
\squareS_{-4,0,-4,0}(\vartheta \Psi_{4})={}&0.
\end{align}
\end{subequations}

Since we only work with boost-weight zero quantities, we shall deboost the above equations. To do this it is convenient to extend the $\YOp$ and $\mathcal{L}_{\xi}$ operators to work on a $\{p,q\}$-weighted scalar $\phi$ as
\begin{subequations}
\begin{align}
\YOp\varphi={}&\sqrt{2} \uplambda \thop \varphi,\\
\mathcal{L}_{\xi}{}\varphi={}&-3 \kappa_{1}{} \rho ' \tho \varphi
 + 3 \kappa_{1}{} \rho \thop \varphi
 + 3 \kappa_{1}{} \tau ' \edt \varphi
 - 3 \kappa_{1}{} \tau \edtp \varphi
 + \tfrac{3}{2} p \Psi_2 \kappa_{1}{} \varphi
 + \tfrac{3}{2} q \bar\Psi_2 \bar{\kappa}_{1'}{} \varphi.
\end{align}
\end{subequations}
The general Teukolsky operator interacts with a boost scaling $\uplambda^n$ as follows
\begin{align}
\squareS_{n + p,n + q,j,k}(\uplambda^n \varphi)={}&4 n \uplambda^n r\mathcal{L}_{\xi}{}\varphi
 - 2 n \uplambda^n (r - M)\YOp\varphi
 -  \tfrac{2}{3} n \uplambda^n (r- M) (\frac{n + p -  j}{\kappa_{1}{}} + \frac{n + q -  k}{\bar{\kappa}_{1'}{}}) \varphi\nonumber\\
& + \uplambda^n\squareS_{p,q,j-n,k-n}(\varphi).
\end{align}
With this formula,  the Teukolsky equations \eqref{eq:TMEJustRescaled} are deboosted to the equations \eqref{eq:TMEDeBoostedStep1}.
\end{proof}

Finally, we provide a proof for Lemma \ref{lem:TeukolskyRegular}.

\begin{proof}[Proof of Lemma \ref{lem:TeukolskyRegular}]
For a spin-weighted scalar $\varphi$ with spin weight $s$, its $\{p,q\}$ weight is $\{s,-s\}$. We use the commutator \eqref{eq:CommutatorofYandMathcalV} to 
express $\squareS_{s,-s,s,-s}$ in terms of the $\YOp$, $\VOp$, $\mathcal{L}_{\eta}$ and $\TMESOp_s$ operators when acting on such a scalar $\varphi$
\begin{align}\label{eq:TMEIntermediate1}
\squareS_{s,- s,s,- s}\varphi={}&- \frac{2 a r}{a^2 + r^2}\mathcal{L}_{\eta}\varphi
 - 2 r\VOp\varphi
 + 2 (a^2 + r^2)\YOp\VOp\varphi
 + \frac{r \Delta}{a^2 + r^2}\YOp\varphi
 -  \TMESOp_s\varphi.
\end{align}
By Definition \ref{def:TMEOperators}, a straightforward and simple calculation yields the following relation between the operators $\widehat\squareS_{s}$ and $\squareS_{s,- s,s,- s}$:
\begin{align}
\widehat\squareS_{s}(\sqrt{a^2 + r^2}\varphi)={}&\sqrt{a^2 + r^2}\squareS_{s,- s,s,- s}\varphi.
\end{align}
As a consequence, for $\psibase[-2]= \tfrac{1}{2}\sqrt{a^2 + r^2}\uplambda^{2} \vartheta \Psi_{4}$, we obtain $\widehat\squareS_{s}\psibase[-2]=\tfrac{1}{2}\sqrt{a^2 + r^2}\squareS_{s,- s,s,- s}(\uplambda^{2} \vartheta \Psi_{4})$. Substituting this back into the Teukolsky equation \eqref{eq:TMEDeBoostedStep1:Psi4}, we then derive the alternative form \eqref{eq:TeukolskyRegular-2} of Teukolsky equation. The other Teukolsky equation \eqref{eq:TeukolskyRegular+2} for $\psibase[+2]$ is deduced in a similar manner.
\end{proof}

\begin{proof}[Proof of Lemma \ref{lem:TSIRegular}]
The definition \ref{def:VY-cov} together with the relations \eqref{eq:GHPkappa} gives the relations for any $\{p,q\}$-weighted scalar $\varphi$
\begin{subequations}
\begin{align}
\thop \bigl(\uplambda (a^2 + r^2)^{-1/2}\varphi\bigr)={}&\tfrac{1}{\sqrt{2}}(a^2 + r^2)^{-1/2}\Bigl(\YOp+ \frac{r}{a^2 + r^2}\Bigr)\varphi,\\
\edt (\kappa_{1}{}^k\uplambda^{-2}\varphi)={}&\tfrac{1}{3}\uplambda^{-2}\kappa_{1}{}^{k-1} \bigl(3 \kappa_{1}{}\edt - 3 (k-2) \kappa_{1}{} \tau \bigr)\varphi.
\end{align}
\end{subequations}
Using the definition \ref{def:psibase} we get
\begin{align}
0={}&  -  (3 \kappa_{1}{}\edt + 3 \kappa_{1}{} \tau)(3\kappa_{1}{}\edt)(3 \kappa_{1}{}\edt - 3 \kappa_{1}{} \tau)(3 \kappa_{1}{}\edt - 6 \kappa_{1}{} \tau)\psibase[-2]\nonumber\\
&-3 M\mathcal{L}_{\xi}{}\overline{\psibase[-2]}
 + \frac{1}{4} \Bigl(\frac{r}{a^2 + r^2} + \YOp\Bigr)^4\psibase[+2].
\end{align}
Translating to the $\hedt$ operator yields
\begin{align}
0={}&  -  \bigl(\hedt + \ringtau\mathcal{L}_{\xi}{}\bigr)^4\psibase[-2]-3 M\mathcal{L}_{\xi}{}\overline{\psibase[-2]}
 + \frac{1}{4} \Bigl(\frac{r}{a^2 + r^2} + \YOp\Bigr)^4\psibase[+2].
\end{align}
A binomial expansion gives equation \eqref{eq:TSIpsihat}.
\end{proof}




\input{\jobname.ind}


\newcommand{\apj}{ApJ}
\newcommand{\mnras}{Mon. Not. R. Astron Soc.}
\newcommand{\aap}{Astron. Astrophys.}


\begin{thebibliography}{73}%
\makeatletter
\providecommand \@ifxundefined [1]{%
 \@ifx{#1\undefined}
}%
\providecommand \@ifnum [1]{%
 \ifnum #1\expandafter \@firstoftwo
 \else \expandafter \@secondoftwo
 \fi
}%
\providecommand \@ifx [1]{%
 \ifx #1\expandafter \@firstoftwo
 \else \expandafter \@secondoftwo
 \fi
}%
\providecommand \natexlab [1]{#1}%
\providecommand \enquote  [1]{``#1''}%
\providecommand \bibnamefont  [1]{#1}%
\providecommand \bibfnamefont [1]{#1}%
\providecommand \citenamefont [1]{#1}%
\providecommand \href@noop [0]{\@secondoftwo}%
\providecommand \href [0]{\begingroup \@sanitize@url \@href}%
\providecommand \@href[1]{\@@startlink{#1}\@@href}%
\providecommand \@@href[1]{\endgroup#1\@@endlink}%
\providecommand \@sanitize@url [0]{\catcode `\\12\catcode `\$12\catcode
  `\&12\catcode `\#12\catcode `\^12\catcode `\_12\catcode `\%12\relax}%
\providecommand \@@startlink[1]{}%
\providecommand \@@endlink[0]{}%
\providecommand \url  [0]{\begingroup\@sanitize@url \@url }%
\providecommand \@url [1]{\endgroup\@href {#1}{\urlprefix }}%
\providecommand \urlprefix  [0]{URL }%
\providecommand \Eprint [0]{\href }%
\providecommand \doibase [0]{http://dx.doi.org/}%
\providecommand \selectlanguage [0]{\@gobble}%
\providecommand \bibinfo  [0]{\@secondoftwo}%
\providecommand \bibfield  [0]{\@secondoftwo}%
\providecommand \translation [1]{[#1]}%
\providecommand \BibitemOpen [0]{}%
\providecommand \bibitemStop [0]{}%
\providecommand \bibitemNoStop [0]{.\EOS\space}%
\providecommand \EOS [0]{\spacefactor3000\relax}%
\providecommand \BibitemShut  [1]{\csname bibitem#1\endcsname}%
\let\auto@bib@innerbib\@empty
\bibitem [{\citenamefont {{Aksteiner}}\ and\ \citenamefont
  {{Andersson}}(2013)}]{2013CQGra..30o5016A}%
  \BibitemOpen
  \bibfield  {author} {\bibinfo {author} {\bibfnamefont {S.}~\bibnamefont
  {{Aksteiner}}}\ and\ \bibinfo {author} {\bibfnamefont {L.}~\bibnamefont
  {{Andersson}}},\ }\bibfield  {title} {\enquote {\bibinfo {title} {{Charges
  for linearized gravity}},}\ }\href {\doibase 10.1088/0264-9381/30/15/155016}
  {\bibfield  {journal} {\bibinfo  {journal} {Class. Quant. Grav.}\ }\textbf
  {\bibinfo {volume} {30}},\ \bibinfo {eid} {155016} (\bibinfo {year}
  {2013})},\ \Eprint {http://arxiv.org/abs/1301.2674}{arXiv:1301.2674
  [gr-qc]}\BibitemShut {NoStop}%
\bibitem [{\citenamefont {Aksteiner}, \citenamefont {Andersson},\ and\
  \citenamefont {B\"ackdahl}(2019)}]{2016arXiv160106084A}%
  \BibitemOpen
  \bibfield  {author} {\bibinfo {author} {\bibfnamefont {S.}~\bibnamefont
  {Aksteiner}}, \bibinfo {author} {\bibfnamefont {L.}~\bibnamefont
  {Andersson}}, \ and\ \bibinfo {author} {\bibfnamefont {T.}~\bibnamefont
  {B\"ackdahl}},\ }\bibfield  {title} {\enquote {\bibinfo {title} {New
  identities for linearized gravity on the kerr spacetime},}\ }\href {\doibase
  10.1103/PhysRevD.99.044043} {\bibfield  {journal} {\bibinfo  {journal} {Phys.
  Rev. D}\ }\textbf {\bibinfo {volume} {99}},\ \bibinfo {pages} {044043}
  (\bibinfo {year} {2019})},\ \Eprint
  {http://arxiv.org/abs/1601.06084}{arXiv:1601.06084 [gr-qc]}\BibitemShut
  {NoStop}%
\bibitem [{\citenamefont {{Andersson}}, \citenamefont {{B{\"a}ckdahl}},\ and\
  \citenamefont {{Blue}}(2015)}]{2015arXiv150402069A}%
  \BibitemOpen
  \bibfield  {author} {\bibinfo {author} {\bibfnamefont {L.}~\bibnamefont
  {{Andersson}}}, \bibinfo {author} {\bibfnamefont {T.}~\bibnamefont
  {{B{\"a}ckdahl}}}, \ and\ \bibinfo {author} {\bibfnamefont {P.}~\bibnamefont
  {{Blue}}},\ }\bibfield  {title} {\enquote {\bibinfo {title} {{Spin geometry
  and conservation laws in the {K}err spacetime}},}\ }in\ \href {\doibase
  10.4310/SDG.2015.v20.n1.a8} {\emph {\bibinfo {booktitle} {One hundred years
  of general relativity}}},\ \bibinfo {editor} {edited by\ \bibinfo {editor}
  {\bibfnamefont {L.}~\bibnamefont {Bieri}}\ and\ \bibinfo {editor}
  {\bibfnamefont {S.-T.}\ \bibnamefont {Yau}}}\ (\bibinfo  {publisher}
  {International Press},\ \bibinfo {address} {Boston},\ \bibinfo {year}
  {2015})\ pp.\ \bibinfo {pages} {183--226},\ \Eprint
  {http://arxiv.org/abs/1504.02069}{arXiv:1504.02069 [gr-qc]}\BibitemShut
  {NoStop}%
\bibitem [{\citenamefont {Andersson}\ \emph {et~al.}(2022)\citenamefont
  {Andersson}, \citenamefont {B\"ackdahl}, \citenamefont {Blue},\ and\
  \citenamefont {Ma}}]{Andersson:2021eqc}%
  \BibitemOpen
  \bibfield  {author} {\bibinfo {author} {\bibfnamefont {L.}~\bibnamefont
  {Andersson}}, \bibinfo {author} {\bibfnamefont {T.}~\bibnamefont
  {B\"ackdahl}}, \bibinfo {author} {\bibfnamefont {P.}~\bibnamefont {Blue}}, \
  and\ \bibinfo {author} {\bibfnamefont {S.}~\bibnamefont {Ma}},\ }\bibfield
  {title} {\enquote {\bibinfo {title} {{Nonlinear Radiation Gauge for Near
  {K}err Spacetimes}},}\ }\href {\doibase 10.1007/s00220-022-04461-3}
  {\bibfield  {journal} {\bibinfo  {journal} {Commun. Math. Phys.}\ }\textbf
  {\bibinfo {volume} {396}},\ \bibinfo {pages} {45--90} (\bibinfo {year}
  {2022})},\ \Eprint {http://arxiv.org/abs/2108.03148}{arXiv:2108.03148
  [gr-qc]}\BibitemShut {NoStop}%
\bibitem [{\citenamefont {Andersson}\ and\ \citenamefont
  {Blue}(2015{\natexlab{a}})}]{AnderssonBlue:KerrWave}%
  \BibitemOpen
  \bibfield  {author} {\bibinfo {author} {\bibfnamefont {L.}~\bibnamefont
  {Andersson}}\ and\ \bibinfo {author} {\bibfnamefont {P.}~\bibnamefont
  {Blue}},\ }\bibfield  {title} {\enquote {\bibinfo {title} {Hidden symmetries
  and decay for the wave equation on the {K}err spacetime},}\ }\href {\doibase
  10.4007/annals.2015.182.3.1} {\bibfield  {journal} {\bibinfo  {journal} {Ann.
  of Math. (2)}\ }\textbf {\bibinfo {volume} {182}},\ \bibinfo {pages}
  {787--853} (\bibinfo {year} {2015}{\natexlab{a}})},\ \Eprint
  {http://arxiv.org/abs/0908.2265}{arXiv:0908.2265 [math.AP]}\BibitemShut
  {NoStop}%
\bibitem [{\citenamefont {Andersson}\ and\ \citenamefont
  {Blue}(2015{\natexlab{b}})}]{MR3450059}%
  \BibitemOpen
  \bibfield  {author} {\bibinfo {author} {\bibfnamefont {L.}~\bibnamefont
  {Andersson}}\ and\ \bibinfo {author} {\bibfnamefont {P.}~\bibnamefont
  {Blue}},\ }\bibfield  {title} {\enquote {\bibinfo {title} {Uniform energy
  bound and asymptotics for the {M}axwell field on a slowly rotating {K}err
  black hole exterior},}\ }\href {\doibase 10.1142/S0219891615500204}
  {\bibfield  {journal} {\bibinfo  {journal} {J. Hyperbolic Differ. Equ.}\
  }\textbf {\bibinfo {volume} {12}},\ \bibinfo {pages} {689--743} (\bibinfo
  {year} {2015}{\natexlab{b}})},\ \Eprint
  {http://arxiv.org/abs/1310.2664}{arXiv:1310.2664 [math.AP]}\BibitemShut
  {NoStop}%
\bibitem [{\citenamefont {Andersson}, \citenamefont {Blue},\ and\ \citenamefont
  {Wang}(2020)}]{2017arXiv170806943A}%
  \BibitemOpen
  \bibfield  {author} {\bibinfo {author} {\bibfnamefont {L.}~\bibnamefont
  {Andersson}}, \bibinfo {author} {\bibfnamefont {P.}~\bibnamefont {Blue}}, \
  and\ \bibinfo {author} {\bibfnamefont {J.}~\bibnamefont {Wang}},\ }\bibfield
  {title} {\enquote {\bibinfo {title} {Morawetz estimate for linearized gravity
  in {S}chwarzschild},}\ }\href {\doibase 10.1007/s00023-020-00886-5}
  {\bibfield  {journal} {\bibinfo  {journal} {Ann. Henri Poincar\'{e}}\
  }\textbf {\bibinfo {volume} {21}},\ \bibinfo {pages} {761--813} (\bibinfo
  {year} {2020})},\ \Eprint {http://arxiv.org/abs/1708.06943}{arXiv:1708.06943
  [math.AP]}\BibitemShut {NoStop}%
\bibitem [{\citenamefont {{Andersson}}\ \emph {et~al.}(2017)\citenamefont
  {{Andersson}}, \citenamefont {{Ma}}, \citenamefont {{Paganini}},\ and\
  \citenamefont {{Whiting}}}]{2017JMP....58g2501A}%
  \BibitemOpen
  \bibfield  {author} {\bibinfo {author} {\bibfnamefont {L.}~\bibnamefont
  {{Andersson}}}, \bibinfo {author} {\bibfnamefont {S.}~\bibnamefont {{Ma}}},
  \bibinfo {author} {\bibfnamefont {C.}~\bibnamefont {{Paganini}}}, \ and\
  \bibinfo {author} {\bibfnamefont {B.~F.}\ \bibnamefont {{Whiting}}},\
  }\bibfield  {title} {\enquote {\bibinfo {title} {{Mode stability on the real
  axis}},}\ }\href {\doibase 10.1063/1.4991656} {\bibfield  {journal} {\bibinfo
   {journal} {J. Math. Phys.}\ }\textbf {\bibinfo {volume} {58}},\ \bibinfo
  {eid} {072501} (\bibinfo {year} {2017})},\ \Eprint
  {http://arxiv.org/abs/1607.02759}{arXiv:1607.02759 [gr-qc]}\BibitemShut
  {NoStop}%
\bibitem [{\citenamefont {Angelopoulos}, \citenamefont {Aretakis},\ and\
  \citenamefont {Gajic}(2018)}]{MR3859608}%
  \BibitemOpen
  \bibfield  {author} {\bibinfo {author} {\bibfnamefont {Y.}~\bibnamefont
  {Angelopoulos}}, \bibinfo {author} {\bibfnamefont {S.}~\bibnamefont
  {Aretakis}}, \ and\ \bibinfo {author} {\bibfnamefont {D.}~\bibnamefont
  {Gajic}},\ }\bibfield  {title} {\enquote {\bibinfo {title} {A vector field
  approach to almost-sharp decay for the wave equation on spherically
  symmetric, stationary spacetimes},}\ }\href {\doibase
  10.1007/s40818-018-0051-2} {\bibfield  {journal} {\bibinfo  {journal} {Ann.
  PDE}\ }\textbf {\bibinfo {volume} {4}},\ \bibinfo {pages} {Art. 15, 120}
  (\bibinfo {year} {2018})},\ \Eprint
  {http://arxiv.org/abs/1612.01565}{arXiv:1612.01565 [math.AP]}\BibitemShut
  {NoStop}%
\bibitem [{\citenamefont {{B\"{a}ckdahl}}(2011-2021)}]{Bae11a}%
  \BibitemOpen
  \bibfield  {author} {\bibinfo {author} {\bibfnamefont {T.}~\bibnamefont
  {{B\"{a}ckdahl}}},\ }\href@noop {} {\enquote {\bibinfo {title}
  {Sym{M}anipulator},}\ } (\bibinfo {year} {2011-2021}),\ \bibinfo {note}
  {\href{http://www.xact.es/SymManipulator}{http://www.xact.es/SymManipulator}}\BibitemShut
  {NoStop}%
\bibitem [{\citenamefont {{B\"{a}ckdahl}}\ and\ \citenamefont
  {{Aksteiner}}(2015-2021)}]{SpinframesPackage}%
  \BibitemOpen
  \bibfield  {author} {\bibinfo {author} {\bibfnamefont {T.}~\bibnamefont
  {{B\"{a}ckdahl}}}\ and\ \bibinfo {author} {\bibfnamefont {S.}~\bibnamefont
  {{Aksteiner}}},\ }\href@noop {} {\enquote {\bibinfo {title} {Spin{F}rames},}\
  } (\bibinfo {year} {2015-2021}),\ \bibinfo {note}
  {\href{http://xact.es/SpinFrames}{http://xact.es/SpinFrames}}\BibitemShut
  {NoStop}%
\bibitem [{\citenamefont {B{\"a}ckdahl}\ and\ \citenamefont
  {Valiente~Kroon}(2010)}]{backdahl:etal:2010:MR2753388}%
  \BibitemOpen
  \bibfield  {author} {\bibinfo {author} {\bibfnamefont {T.}~\bibnamefont
  {B{\"a}ckdahl}}\ and\ \bibinfo {author} {\bibfnamefont {J.~A.}\ \bibnamefont
  {Valiente~Kroon}},\ }\bibfield  {title} {\enquote {\bibinfo {title} {On the
  construction of a geometric invariant measuring the deviation from {K}err
  data},}\ }\href {\doibase 10.1007/s00023-010-0063-2} {\bibfield  {journal}
  {\bibinfo  {journal} {Ann. Henri Poincar\'e}\ }\textbf {\bibinfo {volume}
  {11}},\ \bibinfo {pages} {1225--1271} (\bibinfo {year} {2010})},\ \Eprint
  {http://arxiv.org/abs/1005.0743}{arXiv:1005.0743 [gr-qc]}\BibitemShut
  {NoStop}%
\bibitem [{\citenamefont {B\"ackdahl}\ and\ \citenamefont
  {Valiente~Kroon}(2016)}]{BaeVal15}%
  \BibitemOpen
  \bibfield  {author} {\bibinfo {author} {\bibfnamefont {T.}~\bibnamefont
  {B\"ackdahl}}\ and\ \bibinfo {author} {\bibfnamefont {J.~A.}\ \bibnamefont
  {Valiente~Kroon}},\ }\bibfield  {title} {\enquote {\bibinfo {title} {A
  formalism for the calculus of variations with spinors},}\ }\href {\doibase
  10.1063/1.4939562} {\bibfield  {journal} {\bibinfo  {journal} {J. Math.
  Phys.}\ }\textbf {\bibinfo {volume} {57}},\ \bibinfo {pages} {022502}
  (\bibinfo {year} {2016})},\ \Eprint
  {http://arxiv.org/abs/1505.03770}{arXiv:1505.03770 [gr-qc]}\BibitemShut
  {NoStop}%
\bibitem [{\citenamefont {Blue}\ and\ \citenamefont
  {Soffer}(2003)}]{MR1972492}%
  \BibitemOpen
  \bibfield  {author} {\bibinfo {author} {\bibfnamefont {P.}~\bibnamefont
  {Blue}}\ and\ \bibinfo {author} {\bibfnamefont {A.}~\bibnamefont {Soffer}},\
  }\bibfield  {title} {\enquote {\bibinfo {title} {Semilinear wave equations on
  the {S}chwarzschild manifold. {I}. {L}ocal decay estimates},}\ }\href
  {\doibase 10.57262/ade/1355926842} {\bibfield  {journal} {\bibinfo  {journal}
  {Adv. Differential Equations}\ }\textbf {\bibinfo {volume} {8}},\ \bibinfo
  {pages} {595--614} (\bibinfo {year} {2003})},\ \Eprint
  {http://arxiv.org/abs/gr-qc/0310091}{arXiv:gr-qc/0310091}\BibitemShut
  {NoStop}%
\bibitem [{\citenamefont {Blue}\ and\ \citenamefont
  {Soffer}(2006)}]{blue2006errata}%
  \BibitemOpen
  \bibfield  {author} {\bibinfo {author} {\bibfnamefont {P.}~\bibnamefont
  {Blue}}\ and\ \bibinfo {author} {\bibfnamefont {A.}~\bibnamefont {Soffer}},\
  }\bibfield  {title} {\enquote {\bibinfo {title} {Errata for``{G}lobal
  existence and scattering for the nonlinear {S}chrodinger equation on
  {S}chwarzschild manifolds'',``{S}emilinear wave equations on the
  {S}chwarzschild manifold {I}: {L}ocal {D}ecay {E}stimates'', and``{T}he wave
  equation on the {S}chwarzschild metric {II}: {L}ocal {D}ecay for the spin 2
  {R}egge {W}heeler equation''},}\ }\href@noop {} {\  (\bibinfo {year}
  {2006})},\ \Eprint {http://arxiv.org/abs/gr-qc/0608073}{arXiv:gr-qc/0608073
  [gr-qc]}\BibitemShut {NoStop}%
\bibitem [{\citenamefont {Blue}\ and\ \citenamefont
  {Sterbenz}(2006)}]{blue2006uniform}%
  \BibitemOpen
  \bibfield  {author} {\bibinfo {author} {\bibfnamefont {P.}~\bibnamefont
  {Blue}}\ and\ \bibinfo {author} {\bibfnamefont {J.}~\bibnamefont
  {Sterbenz}},\ }\bibfield  {title} {\enquote {\bibinfo {title} {Uniform decay
  of local energy and the semi-linear wave equation on schwarzschild space},}\
  }\href {\doibase 10.1007/s00220-006-0101-6} {\bibfield  {journal} {\bibinfo
  {journal} {Comm. Math. Phys.}\ }\textbf {\bibinfo {volume} {268}},\ \bibinfo
  {pages} {481--504} (\bibinfo {year} {2006})},\ \Eprint
  {http://arxiv.org/abs/math/0510315}{arXiv:math/0510315 [math.AP]}\BibitemShut
  {NoStop}%
\bibitem [{\citenamefont {Chandrasekhar}(1983)}]{Chandrasekhar}%
  \BibitemOpen
  \bibfield  {author} {\bibinfo {author} {\bibfnamefont {S.}~\bibnamefont
  {Chandrasekhar}},\ }\href@noop {} {\emph {\bibinfo {title} {The mathematical
  theory of black holes}}},\ \bibinfo {series} {International Series of
  Monographs on Physics}, Vol.~\bibinfo {volume} {69}\ (\bibinfo  {publisher}
  {The Clarendon Press, Oxford University Press, New York},\ \bibinfo {year}
  {1983})\ pp.\ \bibinfo {pages} {xxi+646},\ \bibinfo {note} {oxford Science
  Publications}\BibitemShut {NoStop}%
\bibitem [{\citenamefont {Christodoulou}\ and\ \citenamefont
  {Klainerman}(1993)}]{MR1316662}%
  \BibitemOpen
  \bibfield  {author} {\bibinfo {author} {\bibfnamefont {D.}~\bibnamefont
  {Christodoulou}}\ and\ \bibinfo {author} {\bibfnamefont {S.}~\bibnamefont
  {Klainerman}},\ }\href {\doibase doi:10.1515/9781400863174} {\emph {\bibinfo
  {title} {The global nonlinear stability of the {M}inkowski space}}},\
  \bibinfo {series} {Princeton Mathematical Series}, Vol.~\bibinfo {volume}
  {41}\ (\bibinfo  {publisher} {Princeton University Press, Princeton, NJ},\
  \bibinfo {year} {1993})\ pp.\ \bibinfo {pages} {x+514}\BibitemShut {NoStop}%
\bibitem [{\citenamefont {{Chrzanowski}}(1975)}]{Chrzanowski}%
  \BibitemOpen
  \bibfield  {author} {\bibinfo {author} {\bibfnamefont {P.~L.}\ \bibnamefont
  {{Chrzanowski}}},\ }\bibfield  {title} {\enquote {\bibinfo {title} {{Vector
  potential and metric perturbations of a rotating black hole}},}\ }\href
  {\doibase 10.1103/PhysRevD.11.2042} {\bibfield  {journal} {\bibinfo
  {journal} {Phys. Rev. D}\ }\textbf {\bibinfo {volume} {11}},\ \bibinfo
  {pages} {2042--2062} (\bibinfo {year} {1975})}\BibitemShut {NoStop}%
\bibitem [{\citenamefont {Dafermos}, \citenamefont {Holzegel},\ and\
  \citenamefont {Rodnianski}(2019{\natexlab{a}})}]{2017arXiv171107944D}%
  \BibitemOpen
  \bibfield  {author} {\bibinfo {author} {\bibfnamefont {M.}~\bibnamefont
  {Dafermos}}, \bibinfo {author} {\bibfnamefont {G.}~\bibnamefont {Holzegel}},
  \ and\ \bibinfo {author} {\bibfnamefont {I.}~\bibnamefont {Rodnianski}},\
  }\bibfield  {title} {\enquote {\bibinfo {title} {Boundedness and decay for
  the {T}eukolsky equation on {K}err spacetimes {I}: {T}he case {$|a|\ll
  M$}},}\ }\href {\doibase 10.1007/s40818-018-0058-8} {\bibfield  {journal}
  {\bibinfo  {journal} {Ann. PDE}\ }\textbf {\bibinfo {volume} {5}},\ \bibinfo
  {pages} {Paper No. 2, 118} (\bibinfo {year} {2019}{\natexlab{a}})},\ \Eprint
  {http://arxiv.org/abs/1711.07944}{arXiv:1711.07944 [math.AP]}\BibitemShut
  {NoStop}%
\bibitem [{\citenamefont {Dafermos}, \citenamefont {Holzegel},\ and\
  \citenamefont {Rodnianski}(2019{\natexlab{b}})}]{2016arXiv160106467D}%
  \BibitemOpen
  \bibfield  {author} {\bibinfo {author} {\bibfnamefont {M.}~\bibnamefont
  {Dafermos}}, \bibinfo {author} {\bibfnamefont {G.}~\bibnamefont {Holzegel}},
  \ and\ \bibinfo {author} {\bibfnamefont {I.}~\bibnamefont {Rodnianski}},\
  }\bibfield  {title} {\enquote {\bibinfo {title} {The linear stability of the
  {S}chwarzschild solution to gravitational perturbations},}\ }\href {\doibase
  10.4310/ACTA.2019.v222.n1.a1} {\bibfield  {journal} {\bibinfo  {journal}
  {Acta Math.}\ }\textbf {\bibinfo {volume} {222}},\ \bibinfo {pages} {1--214}
  (\bibinfo {year} {2019}{\natexlab{b}})},\ \Eprint
  {http://arxiv.org/abs/1601.06467}{arXiv:1601.06467 [gr-qc]}\BibitemShut
  {NoStop}%
\bibitem [{\citenamefont {Dafermos}\ and\ \citenamefont
  {Rodnianski}(2009)}]{MR2527808}%
  \BibitemOpen
  \bibfield  {author} {\bibinfo {author} {\bibfnamefont {M.}~\bibnamefont
  {Dafermos}}\ and\ \bibinfo {author} {\bibfnamefont {I.}~\bibnamefont
  {Rodnianski}},\ }\bibfield  {title} {\enquote {\bibinfo {title} {The
  red-shift effect and radiation decay on black hole spacetimes},}\ }\href
  {\doibase 10.1002/cpa.20281} {\bibfield  {journal} {\bibinfo  {journal}
  {Comm. Pure Appl. Math.}\ }\textbf {\bibinfo {volume} {62}},\ \bibinfo
  {pages} {859--919} (\bibinfo {year} {2009})},\ \Eprint
  {http://arxiv.org/abs/gr-qc/0512119}{arXiv:gr-qc/0512119 [gr-qc]}\BibitemShut
  {NoStop}%
\bibitem [{\citenamefont {Dafermos}\ and\ \citenamefont
  {Rodnianski}(2010)}]{DafermosRodnianski:rp}%
  \BibitemOpen
  \bibfield  {author} {\bibinfo {author} {\bibfnamefont {M.}~\bibnamefont
  {Dafermos}}\ and\ \bibinfo {author} {\bibfnamefont {I.}~\bibnamefont
  {Rodnianski}},\ }\bibfield  {title} {\enquote {\bibinfo {title} {A new
  physical-space approach to decay for the wave equation with applications to
  black hole spacetimes},}\ }in\ \href {\doibase 10.1142/9789814304634_0032}
  {\emph {\bibinfo {booktitle} {X{VI}th {I}nternational {C}ongress on
  {M}athematical {P}hysics}}}\ (\bibinfo  {publisher} {World Sci. Publ.,
  Hackensack, NJ},\ \bibinfo {year} {2010})\ pp.\ \bibinfo {pages} {421--432},\
  \Eprint {http://arxiv.org/abs/0910.4957}{arXiv:0910.4957
  [math.AP]}\BibitemShut {NoStop}%
\bibitem [{\citenamefont {Dafermos}, \citenamefont {Rodnianski},\ and\
  \citenamefont {Shlapentokh-Rothman}(2016)}]{MR3488738}%
  \BibitemOpen
  \bibfield  {author} {\bibinfo {author} {\bibfnamefont {M.}~\bibnamefont
  {Dafermos}}, \bibinfo {author} {\bibfnamefont {I.}~\bibnamefont
  {Rodnianski}}, \ and\ \bibinfo {author} {\bibfnamefont {Y.}~\bibnamefont
  {Shlapentokh-Rothman}},\ }\bibfield  {title} {\enquote {\bibinfo {title}
  {Decay for solutions of the wave equation on {K}err exterior spacetimes
  {III}: {T}he full subextremal case {$|a|<M$}},}\ }\href {\doibase
  10.4007/annals.2016.183.3.2} {\bibfield  {journal} {\bibinfo  {journal} {Ann.
  of Math. (2)}\ }\textbf {\bibinfo {volume} {183}},\ \bibinfo {pages}
  {787--913} (\bibinfo {year} {2016})}\BibitemShut {NoStop}%
\bibitem [{\citenamefont {Donninger}, \citenamefont {Schlag},\ and\
  \citenamefont {Soffer}(2012)}]{MR2864787}%
  \BibitemOpen
  \bibfield  {author} {\bibinfo {author} {\bibfnamefont {R.}~\bibnamefont
  {Donninger}}, \bibinfo {author} {\bibfnamefont {W.}~\bibnamefont {Schlag}}, \
  and\ \bibinfo {author} {\bibfnamefont {A.}~\bibnamefont {Soffer}},\
  }\bibfield  {title} {\enquote {\bibinfo {title} {On pointwise decay of linear
  waves on a {S}chwarzschild black hole background},}\ }\href {\doibase
  10.1007/s00220-011-1393-8} {\bibfield  {journal} {\bibinfo  {journal} {Comm.
  Math. Phys.}\ }\textbf {\bibinfo {volume} {309}},\ \bibinfo {pages} {51--86}
  (\bibinfo {year} {2012})},\ \Eprint
  {http://arxiv.org/abs/0911.3179}{arXiv:0911.3179 [math.AP]}\BibitemShut
  {NoStop}%
\bibitem [{\citenamefont {{Eastwood}}\ and\ \citenamefont
  {{Tod}}(1982)}]{1982MPCPS..92..317E}%
  \BibitemOpen
  \bibfield  {author} {\bibinfo {author} {\bibfnamefont {M.}~\bibnamefont
  {{Eastwood}}}\ and\ \bibinfo {author} {\bibfnamefont {P.}~\bibnamefont
  {{Tod}}},\ }\bibfield  {title} {\enquote {\bibinfo {title} {{Edth--a
  differential operator on the sphere}},}\ }\href {\doibase
  10.1017/S0305004100059971} {\bibfield  {journal} {\bibinfo  {journal} {Math.
  Proc. Cambridge Philos. Soc.}\ }\textbf {\bibinfo {volume} {92}},\ \bibinfo
  {pages} {317} (\bibinfo {year} {1982})}\BibitemShut {NoStop}%
\bibitem [{\citenamefont {Finster}\ \emph {et~al.}(2006)\citenamefont
  {Finster}, \citenamefont {Kamran}, \citenamefont {Smoller},\ and\
  \citenamefont {Yau}}]{MR2215614}%
  \BibitemOpen
  \bibfield  {author} {\bibinfo {author} {\bibfnamefont {F.}~\bibnamefont
  {Finster}}, \bibinfo {author} {\bibfnamefont {N.}~\bibnamefont {Kamran}},
  \bibinfo {author} {\bibfnamefont {J.}~\bibnamefont {Smoller}}, \ and\
  \bibinfo {author} {\bibfnamefont {S.-T.}\ \bibnamefont {Yau}},\ }\bibfield
  {title} {\enquote {\bibinfo {title} {Decay of solutions of the wave equation
  in the {K}err geometry},}\ }\href {\doibase 10.1007/s00220-006-1525-8}
  {\bibfield  {journal} {\bibinfo  {journal} {Comm. Math. Phys.}\ }\textbf
  {\bibinfo {volume} {264}},\ \bibinfo {pages} {465--503} (\bibinfo {year}
  {2006})},\ \Eprint
  {http://arxiv.org/abs/gr-qc/0504047}{arXiv:gr-qc/0504047}\BibitemShut
  {NoStop}%
\bibitem [{\citenamefont {Finster}\ and\ \citenamefont
  {Smoller}(2017)}]{MR3783838}%
  \BibitemOpen
  \bibfield  {author} {\bibinfo {author} {\bibfnamefont {F.}~\bibnamefont
  {Finster}}\ and\ \bibinfo {author} {\bibfnamefont {J.}~\bibnamefont
  {Smoller}},\ }\bibfield  {title} {\enquote {\bibinfo {title} {Linear
  stability of the non-extreme {K}err black hole},}\ }\href {\doibase
  10.4310/ATMP.2017.v21.n8.a4} {\bibfield  {journal} {\bibinfo  {journal} {Adv.
  Theor. Math. Phys.}\ }\textbf {\bibinfo {volume} {21}},\ \bibinfo {pages}
  {1991--2085} (\bibinfo {year} {2017})},\ \Eprint
  {http://arxiv.org/abs/1606.08005}{arXiv:1606.08005 [math-ph]}\BibitemShut
  {NoStop}%
\bibitem [{\citenamefont {Friedlander}(1975)}]{MR0460898}%
  \BibitemOpen
  \bibfield  {author} {\bibinfo {author} {\bibfnamefont {F.~G.}\ \bibnamefont
  {Friedlander}},\ }\href@noop {} {\emph {\bibinfo {title} {The wave equation
  on a curved space-time}}}\ (\bibinfo  {publisher} {Cambridge University
  Press, Cambridge-New York-Melbourne},\ \bibinfo {year} {1975})\ pp.\ \bibinfo
  {pages} {x+282},\ \bibinfo {note} {{C}ambridge Monographs on Mathematical
  Physics, No. 2}\BibitemShut {NoStop}%
\bibitem [{\citenamefont {{Friedrich}}(1986)}]{1986CMaPh.107..587F}%
  \BibitemOpen
  \bibfield  {author} {\bibinfo {author} {\bibfnamefont {H.}~\bibnamefont
  {{Friedrich}}},\ }\bibfield  {title} {\enquote {\bibinfo {title} {{On the
  existence of n-geodesically complete or future complete solutions of
  Einstein's field equations with smooth asymptotic structure}},}\ }\href
  {\doibase 10.1007/BF01205488} {\bibfield  {journal} {\bibinfo  {journal}
  {Comm. Math. Phys.}\ }\textbf {\bibinfo {volume} {107}},\ \bibinfo {pages}
  {587--609} (\bibinfo {year} {1986})}\BibitemShut {NoStop}%
\bibitem [{\citenamefont {{Geroch}}, \citenamefont {{Held}},\ and\
  \citenamefont {{Penrose}}(1973)}]{GHP}%
  \BibitemOpen
  \bibfield  {author} {\bibinfo {author} {\bibfnamefont {R.}~\bibnamefont
  {{Geroch}}}, \bibinfo {author} {\bibfnamefont {A.}~\bibnamefont {{Held}}}, \
  and\ \bibinfo {author} {\bibfnamefont {R.}~\bibnamefont {{Penrose}}},\
  }\bibfield  {title} {\enquote {\bibinfo {title} {{A space-time calculus based
  on pairs of null directions}},}\ }\href {\doibase 10.1063/1.1666410}
  {\bibfield  {journal} {\bibinfo  {journal} {J. Math. Phys.}\ }\textbf
  {\bibinfo {volume} {14}},\ \bibinfo {pages} {874--881} (\bibinfo {year}
  {1973})}\BibitemShut {NoStop}%
\bibitem [{\citenamefont {Giorgi}, \citenamefont {Klainerman},\ and\
  \citenamefont {Szeftel}(2024)}]{GioriKlainermanSzeftel:WaveEstimatesForKerr}%
  \BibitemOpen
  \bibfield  {author} {\bibinfo {author} {\bibfnamefont {E.}~\bibnamefont
  {Giorgi}}, \bibinfo {author} {\bibfnamefont {S.}~\bibnamefont {Klainerman}},
  \ and\ \bibinfo {author} {\bibfnamefont {J.}~\bibnamefont {Szeftel}},\
  }\bibfield  {title} {\enquote {\bibinfo {title} {Wave equations estimates and
  the nonlinear stability of slowly rotating {K}err black holes},}\ }\href
  {\doibase 10.4310/PAMQ.241128023033} {\bibfield  {journal} {\bibinfo
  {journal} {{Pure Appl. Math. Q.}}\ }\textbf {\bibinfo {volume} {20}},\
  \bibinfo {pages} {2865--3849} (\bibinfo {year} {2024})},\ \Eprint
  {http://arxiv.org/abs/arXiv:2205.14808}{arXiv:2205.14808}\BibitemShut
  {NoStop}%
\bibitem [{\citenamefont {{H{\"a}fner}}, \citenamefont {{Hintz}},\ and\
  \citenamefont {{Vasy}}(2021)}]{Hafner:2019kov}%
  \BibitemOpen
  \bibfield  {author} {\bibinfo {author} {\bibfnamefont {D.}~\bibnamefont
  {{H{\"a}fner}}}, \bibinfo {author} {\bibfnamefont {P.}~\bibnamefont
  {{Hintz}}}, \ and\ \bibinfo {author} {\bibfnamefont {A.}~\bibnamefont
  {{Vasy}}},\ }\bibfield  {title} {\enquote {\bibinfo {title} {{Linear
  stability of slowly rotating {K}err black holes}},}\ }\href {\doibase
  10.1007/s00222-020-01002-4} {\bibfield  {journal} {\bibinfo  {journal}
  {Invent. math.}\ }\textbf {\bibinfo {volume} {223}},\ \bibinfo {pages}
  {1227--1406} (\bibinfo {year} {2021})},\ \Eprint
  {http://arxiv.org/abs/1906.00860}{arXiv:1906.00860 [math.AP]}\BibitemShut
  {NoStop}%
\bibitem [{\citenamefont {{Harnett}}(1990)}]{1990CQGra...7.1681H}%
  \BibitemOpen
  \bibfield  {author} {\bibinfo {author} {\bibfnamefont {G.}~\bibnamefont
  {{Harnett}}},\ }\bibfield  {title} {\enquote {\bibinfo {title} {{The GHP
  connection: a metric connection with torsion determined by a pair of null
  directions}},}\ }\href {\doibase 10.1088/0264-9381/7/10/004} {\bibfield
  {journal} {\bibinfo  {journal} {Class. Quant. Grav.}\ }\textbf {\bibinfo
  {volume} {7}},\ \bibinfo {pages} {1681--1705} (\bibinfo {year}
  {1990})}\BibitemShut {NoStop}%
\bibitem [{\citenamefont {Hawking}\ and\ \citenamefont
  {Ellis}(1973)}]{hawking_ellis_1973}%
  \BibitemOpen
  \bibfield  {author} {\bibinfo {author} {\bibfnamefont {S.~W.}\ \bibnamefont
  {Hawking}}\ and\ \bibinfo {author} {\bibfnamefont {G.~F.~R.}\ \bibnamefont
  {Ellis}},\ }\href {\doibase 10.1017/CBO9780511524646} {\emph {\bibinfo
  {title} {The Large Scale Structure of Space-Time}}},\ Cambridge Monographs on
  Mathematical Physics\ (\bibinfo  {publisher} {Cambridge University Press},\
  \bibinfo {year} {1973})\BibitemShut {NoStop}%
\bibitem [{\citenamefont {{Hintz}}\ and\ \citenamefont
  {{Vasy}}(2018)}]{2016arXiv160604014H}%
  \BibitemOpen
  \bibfield  {author} {\bibinfo {author} {\bibfnamefont {P.}~\bibnamefont
  {{Hintz}}}\ and\ \bibinfo {author} {\bibfnamefont {A.}~\bibnamefont
  {{Vasy}}},\ }\bibfield  {title} {\enquote {\bibinfo {title} {The global
  non-linear stability of the {K}err--de {S}itter family of black holes},}\
  }\href {\doibase 10.4310/ACTA.2018.v220.n1.a1} {\bibfield  {journal}
  {\bibinfo  {journal} {Acta Math.}\ }\textbf {\bibinfo {volume} {220}},\
  \bibinfo {pages} {1--206} (\bibinfo {year} {2018})},\ \Eprint
  {http://arxiv.org/abs/1606.04014}{arXiv:1606.04014 [math.DG]}\BibitemShut
  {NoStop}%
\bibitem [{\citenamefont {Hung}, \citenamefont {Keller},\ and\ \citenamefont
  {Wang}(2020)}]{2017arXiv170202843H}%
  \BibitemOpen
  \bibfield  {author} {\bibinfo {author} {\bibfnamefont {P.-K.}\ \bibnamefont
  {Hung}}, \bibinfo {author} {\bibfnamefont {J.}~\bibnamefont {Keller}}, \ and\
  \bibinfo {author} {\bibfnamefont {M.-T.}\ \bibnamefont {Wang}},\ }\bibfield
  {title} {\enquote {\bibinfo {title} {Linear stability of {S}chwarzschild
  spacetime: decay of metric coefficients},}\ }\href {\doibase
  10.4310/jdg/1606964416} {\bibfield  {journal} {\bibinfo  {journal} {J.
  Differential Geom.}\ }\textbf {\bibinfo {volume} {116}},\ \bibinfo {pages}
  {481--541} (\bibinfo {year} {2020})},\ \Eprint
  {http://arxiv.org/abs/1702.02843}{arXiv:1702.02843 [gr-qc]}\BibitemShut
  {NoStop}%
\bibitem [{\citenamefont {Klainerman}\ and\ \citenamefont
  {Szeftel}(2020)}]{KlainermanSzeftel:polarized}%
  \BibitemOpen
  \bibfield  {author} {\bibinfo {author} {\bibfnamefont {S.}~\bibnamefont
  {Klainerman}}\ and\ \bibinfo {author} {\bibfnamefont {J.}~\bibnamefont
  {Szeftel}},\ }\href {\doibase 10.2307/j.ctv15r57cw} {\emph {\bibinfo {title}
  {Global nonlinear stability of {S}chwarzschild spacetime under polarized
  perturbations}}},\ \bibinfo {series} {Annals of Mathematics Studies}, Vol.\
  \bibinfo {volume} {210}\ (\bibinfo  {publisher} {Princeton University Press,
  Princeton, NJ},\ \bibinfo {year} {2020})\ pp.\ \bibinfo {pages}
  {xi+840}\BibitemShut {NoStop}%
\bibitem [{\citenamefont {Klainerman}\ and\ \citenamefont
  {Szeftel}(2022{\natexlab{a}})}]{KlainermanSzeftel:GCMSphereConstruction}%
  \BibitemOpen
  \bibfield  {author} {\bibinfo {author} {\bibfnamefont {S.}~\bibnamefont
  {Klainerman}}\ and\ \bibinfo {author} {\bibfnamefont {J.}~\bibnamefont
  {Szeftel}},\ }\bibfield  {title} {\enquote {\bibinfo {title} {Construction of
  {GCM} spheres in perturbations of {K}err},}\ }\href {\doibase
  10.1007/s40818-022-00131-8} {\bibfield  {journal} {\bibinfo  {journal} {Ann.
  PDE}\ }\textbf {\bibinfo {volume} {8}},\ \bibinfo {pages} {17} (\bibinfo
  {year} {2022}{\natexlab{a}})},\ \Eprint
  {http://arxiv.org/abs/1911.00697}{arXiv:1911.00697 [math.AP]}\BibitemShut
  {NoStop}%
\bibitem [{\citenamefont {Klainerman}\ and\ \citenamefont
  {Szeftel}(2022{\natexlab{b}})}]{KlainermanSzeftel:GCMSphereEffectiveResults}%
  \BibitemOpen
  \bibfield  {author} {\bibinfo {author} {\bibfnamefont {S.}~\bibnamefont
  {Klainerman}}\ and\ \bibinfo {author} {\bibfnamefont {J.}~\bibnamefont
  {Szeftel}},\ }\bibfield  {title} {\enquote {\bibinfo {title} {Effective
  results on uniformization and intrinsic {GCM} spheres in perturbations of
  {K}err},}\ }\href {\doibase 10.1007/s40818-022-00132-7} {\bibfield  {journal}
  {\bibinfo  {journal} {Ann. PDE}\ }\textbf {\bibinfo {volume} {8}},\ \bibinfo
  {pages} {18} (\bibinfo {year} {2022}{\natexlab{b}})},\ \Eprint
  {http://arxiv.org/abs/1912.12195}{arXiv:1912.12195 [math.AP]}\BibitemShut
  {NoStop}%
\bibitem [{\citenamefont {Klainerman}\ and\ \citenamefont
  {Szeftel}(2023)}]{Klainerman:2021qzy}%
  \BibitemOpen
  \bibfield  {author} {\bibinfo {author} {\bibfnamefont {S.}~\bibnamefont
  {Klainerman}}\ and\ \bibinfo {author} {\bibfnamefont {J.}~\bibnamefont
  {Szeftel}},\ }\bibfield  {title} {\enquote {\bibinfo {title} {{Kerr stability
  for small angular momentum}},}\ }\href {\doibase 10.4310/PAMQ.2023.v19.n3.a1}
  {\bibfield  {journal} {\bibinfo  {journal} {{Pure Appl. Math. Q.}}\ }\textbf
  {\bibinfo {volume} {19}},\ \bibinfo {pages} {791--1678} (\bibinfo {year}
  {2023})},\ \Eprint {http://arxiv.org/abs/2104.11857}{arXiv:2104.11857
  [math.AP]}\BibitemShut {NoStop}%
\bibitem [{\citenamefont {Lawson}\ and\ \citenamefont
  {Michelsohn}(1989)}]{MR1031992}%
  \BibitemOpen
  \bibfield  {author} {\bibinfo {author} {\bibfnamefont {H.~B.}\ \bibnamefont
  {Lawson}, \bibfnamefont {Jr.}}\ and\ \bibinfo {author} {\bibfnamefont
  {M.-L.}\ \bibnamefont {Michelsohn}},\ }\href@noop {} {\emph {\bibinfo {title}
  {Spin geometry}}},\ \bibinfo {series} {Princeton Mathematical Series},
  Vol.~\bibinfo {volume} {38}\ (\bibinfo  {publisher} {Princeton University
  Press, Princeton, NJ},\ \bibinfo {year} {1989})\ pp.\ \bibinfo {pages}
  {xii+427}\BibitemShut {NoStop}%
\bibitem [{\citenamefont {{Leonard}}\ and\ \citenamefont
  {{Poisson}}(1997)}]{leonard:poisson:1997}%
  \BibitemOpen
  \bibfield  {author} {\bibinfo {author} {\bibfnamefont {S.~W.}\ \bibnamefont
  {{Leonard}}}\ and\ \bibinfo {author} {\bibfnamefont {E.}~\bibnamefont
  {{Poisson}}},\ }\bibfield  {title} {\enquote {\bibinfo {title} {{Radiative
  multipole moments of integer-spin fields in curved spacetime}},}\ }\href
  {\doibase 10.1103/PhysRevD.56.4789} {\bibfield  {journal} {\bibinfo
  {journal} {Phys. Rev. D}\ }\textbf {\bibinfo {volume} {56}},\ \bibinfo
  {pages} {4789--4814} (\bibinfo {year} {1997})},\ \Eprint
  {http://arxiv.org/abs/arXiv:gr-qc/9705014}{arXiv:gr-qc/9705014}\BibitemShut
  {NoStop}%
\bibitem [{\citenamefont {{Ma}}(2018)}]{SiyuanMathesis}%
  \BibitemOpen
  \bibfield  {author} {\bibinfo {author} {\bibfnamefont {S.}~\bibnamefont
  {{Ma}}},\ }\bibfield  {title} {\enquote {\bibinfo {title} {{Analysis of
  Teukolsky equations on slowly rotating Kerr spacetimes}},}\ }\href
  {http://nbn-resolving.de/urn:nbn:de:kobv:517-opus4-414781} {\bibfield
  {journal} {\bibinfo  {journal} {Ph.D. Thesis, Potsdam University}\ }
  (\bibinfo {year} {2018})}\BibitemShut {NoStop}%
\bibitem [{\citenamefont {{Ma}}(2020)}]{ma2020uniform}%
  \BibitemOpen
  \bibfield  {author} {\bibinfo {author} {\bibfnamefont {S.}~\bibnamefont
  {{Ma}}},\ }\bibfield  {title} {\enquote {\bibinfo {title} {{Uniform Energy
  Bound and Morawetz Estimate for Extreme Components of Spin Fields in the
  Exterior of a Slowly Rotating Kerr Black Hole I: Maxwell Field}},}\ }\href
  {\doibase 10.1007/s00023-020-00884-7} {\bibfield  {journal} {\bibinfo
  {journal} {Annales Henri Poincar\'e}\ }\textbf {\bibinfo {volume} {21}},\
  \bibinfo {pages} {815--863} (\bibinfo {year} {2020})},\ \Eprint
  {http://arxiv.org/abs/1705.06621}{arXiv:1705.06621 [gr-qc]}\BibitemShut
  {NoStop}%
\bibitem [{\citenamefont {Ma}(2020)}]{2017arXiv170807385M}%
  \BibitemOpen
  \bibfield  {author} {\bibinfo {author} {\bibfnamefont {S.}~\bibnamefont
  {Ma}},\ }\bibfield  {title} {\enquote {\bibinfo {title} {Uniform energy bound
  and {M}orawetz estimate for extreme components of spin fields in the exterior
  of a slowly rotating {K}err black hole {II}: {L}inearized gravity},}\ }\href
  {\doibase 10.1007/s00220-020-03777-2} {\bibfield  {journal} {\bibinfo
  {journal} {Comm. Math. Phys.}\ }\textbf {\bibinfo {volume} {377}},\ \bibinfo
  {pages} {2489--2551} (\bibinfo {year} {2020})},\ \Eprint
  {http://arxiv.org/abs/1708.07385}{arXiv:1708.07385 [gr-qc]}\BibitemShut
  {NoStop}%
\bibitem [{\citenamefont {{Ma}}\ and\ \citenamefont
  {{Zhang}}(2023)}]{MZ21Kerr}%
  \BibitemOpen
  \bibfield  {author} {\bibinfo {author} {\bibfnamefont {S.}~\bibnamefont
  {{Ma}}}\ and\ \bibinfo {author} {\bibfnamefont {L.}~\bibnamefont {{Zhang}}},\
  }\bibfield  {title} {\enquote {\bibinfo {title} {{Sharp Decay for Teukolsky
  Equation in Kerr Spacetimes}},}\ }\href {\doibase 10.1007/s00220-023-04640-w}
  {\bibfield  {journal} {\bibinfo  {journal} {Comm. Math. Phys.}\ }\textbf
  {\bibinfo {volume} {401}},\ \bibinfo {pages} {333--434} (\bibinfo {year}
  {2023})},\ \Eprint {http://arxiv.org/abs/2111.04489}{arXiv:2111.04489
  [gr-qc]}\BibitemShut {NoStop}%
\bibitem [{\citenamefont {Metcalfe}, \citenamefont {Tataru},\ and\
  \citenamefont {Tohaneanu}(2017)}]{MR3672902}%
  \BibitemOpen
  \bibfield  {author} {\bibinfo {author} {\bibfnamefont {J.}~\bibnamefont
  {Metcalfe}}, \bibinfo {author} {\bibfnamefont {D.}~\bibnamefont {Tataru}}, \
  and\ \bibinfo {author} {\bibfnamefont {M.}~\bibnamefont {Tohaneanu}},\
  }\bibfield  {title} {\enquote {\bibinfo {title} {Pointwise decay for the
  {M}axwell field on black hole space-times},}\ }\href {\doibase
  10.1016/j.aim.2017.05.024} {\bibfield  {journal} {\bibinfo  {journal} {Adv.
  Math.}\ }\textbf {\bibinfo {volume} {316}},\ \bibinfo {pages} {53--93}
  (\bibinfo {year} {2017})},\ \Eprint
  {http://arxiv.org/abs/1411.3693}{arXiv:1411.3693 [math.AP]}\BibitemShut
  {NoStop}%
\bibitem [{\citenamefont {{Millet}}(2023)}]{2023arXiv230206946M}%
  \BibitemOpen
  \bibfield  {author} {\bibinfo {author} {\bibfnamefont {P.}~\bibnamefont
  {{Millet}}},\ }\bibfield  {title} {\enquote {\bibinfo {title} {{Optimal decay
  for solutions of the Teukolsky equation on the Kerr metric for the full
  subextremal range $|a|<M$}},}\ }\href@noop {} {\  (\bibinfo {year} {2023})},\
  \Eprint {http://arxiv.org/abs/arXiv:2302.06946}{arXiv:2302.06946
  [math.AP]}\BibitemShut {NoStop}%
\bibitem [{\citenamefont {Misner}, \citenamefont {Thorne},\ and\ \citenamefont
  {Wheeler}(1973)}]{MTW}%
  \BibitemOpen
  \bibfield  {author} {\bibinfo {author} {\bibfnamefont {C.~W.}\ \bibnamefont
  {Misner}}, \bibinfo {author} {\bibfnamefont {K.~S.}\ \bibnamefont {Thorne}},
  \ and\ \bibinfo {author} {\bibfnamefont {J.~A.}\ \bibnamefont {Wheeler}},\
  }\href@noop {} {\emph {\bibinfo {title} {Gravitation}}}\ (\bibinfo
  {publisher} {W. H. Freeman and Co.},\ \bibinfo {address} {San Francisco,
  Calif.},\ \bibinfo {year} {1973})\ pp.\ \bibinfo {pages}
  {ii+xxvi+1279+iipp}\BibitemShut {NoStop}%
\bibitem [{\citenamefont {{Newman}}\ and\ \citenamefont
  {{Penrose}}(1962)}]{1962JMP.....3..566N}%
  \BibitemOpen
  \bibfield  {author} {\bibinfo {author} {\bibfnamefont {E.}~\bibnamefont
  {{Newman}}}\ and\ \bibinfo {author} {\bibfnamefont {R.}~\bibnamefont
  {{Penrose}}},\ }\bibfield  {title} {\enquote {\bibinfo {title} {{An Approach
  to Gravitational Radiation by a Method of Spin Coefficients}},}\ }\href
  {\doibase 10.1063/1.1724257} {\bibfield  {journal} {\bibinfo  {journal} {J.
  Math. Phys.}\ }\textbf {\bibinfo {volume} {3}},\ \bibinfo {pages} {566--578}
  (\bibinfo {year} {1962})}\BibitemShut {NoStop}%
\bibitem [{\citenamefont {Penrose}\ and\ \citenamefont
  {Rindler}(1987)}]{MR917488}%
  \BibitemOpen
  \bibfield  {author} {\bibinfo {author} {\bibfnamefont {R.}~\bibnamefont
  {Penrose}}\ and\ \bibinfo {author} {\bibfnamefont {W.}~\bibnamefont
  {Rindler}},\ }\href {\doibase 10.1017/CBO9780511564048} {\emph {\bibinfo
  {title} {Spinors and space-time. {V}ol. 1}}},\ Cambridge Monographs on
  Mathematical Physics\ (\bibinfo  {publisher} {Cambridge University Press,
  Cambridge},\ \bibinfo {year} {1987})\ pp.\ \bibinfo {pages} {x+458},\
  \bibinfo {note} {two-spinor calculus and relativistic fields}\BibitemShut
  {NoStop}%
\bibitem [{\citenamefont {Penrose}\ and\ \citenamefont
  {Rindler}(1988)}]{MR944085}%
  \BibitemOpen
  \bibfield  {author} {\bibinfo {author} {\bibfnamefont {R.}~\bibnamefont
  {Penrose}}\ and\ \bibinfo {author} {\bibfnamefont {W.}~\bibnamefont
  {Rindler}},\ }\href {\doibase 10.1017/CBO9780511524486} {\emph {\bibinfo
  {title} {Spinors and space-time. {V}ol. 2}}},\ \bibinfo {edition} {2nd}\
  ed.,\ Cambridge Monographs on Mathematical Physics\ (\bibinfo  {publisher}
  {Cambridge University Press, Cambridge},\ \bibinfo {year} {1988})\ pp.\
  \bibinfo {pages} {x+501},\ \bibinfo {note} {spinor and twistor methods in
  space-time geometry}\BibitemShut {NoStop}%
\bibitem [{\citenamefont {{Press}}\ and\ \citenamefont
  {{Teukolsky}}(1973)}]{1973ApJ...185..649P}%
  \BibitemOpen
  \bibfield  {author} {\bibinfo {author} {\bibfnamefont {W.~H.}\ \bibnamefont
  {{Press}}}\ and\ \bibinfo {author} {\bibfnamefont {S.~A.}\ \bibnamefont
  {{Teukolsky}}},\ }\bibfield  {title} {\enquote {\bibinfo {title}
  {{Perturbations of a Rotating Black Hole. II. Dynamical Stability of the Kerr
  Metric}},}\ }\href {\doibase 10.1086/152445} {\bibfield  {journal} {\bibinfo
  {journal} {Astrophys. J.}\ }\textbf {\bibinfo {volume} {185}},\ \bibinfo
  {pages} {649--674} (\bibinfo {year} {1973})}\BibitemShut {NoStop}%
\bibitem [{\citenamefont {{Price}}, \citenamefont {{Shankar}},\ and\
  \citenamefont {{Whiting}}(2007)}]{2007CQGra..24.2367P}%
  \BibitemOpen
  \bibfield  {author} {\bibinfo {author} {\bibfnamefont {L.~R.}\ \bibnamefont
  {{Price}}}, \bibinfo {author} {\bibfnamefont {K.}~\bibnamefont {{Shankar}}},
  \ and\ \bibinfo {author} {\bibfnamefont {B.~F.}\ \bibnamefont {{Whiting}}},\
  }\bibfield  {title} {\enquote {\bibinfo {title} {{On the existence of
  radiation gauges in Petrov type II spacetimes}},}\ }\href {\doibase
  10.1088/0264-9381/24/9/014} {\bibfield  {journal} {\bibinfo  {journal}
  {Class. Quant. Grav.}\ }\textbf {\bibinfo {volume} {24}},\ \bibinfo {pages}
  {2367--2388} (\bibinfo {year} {2007})},\ \Eprint
  {http://arxiv.org/abs/gr-qc/0611070}{arXiv:gr-qc/0611070}\BibitemShut
  {NoStop}%
\bibitem [{\citenamefont {{Regge}}\ and\ \citenamefont
  {{Teitelboim}}(1974)}]{1974AnPhy..88..286R}%
  \BibitemOpen
  \bibfield  {author} {\bibinfo {author} {\bibfnamefont {T.}~\bibnamefont
  {{Regge}}}\ and\ \bibinfo {author} {\bibfnamefont {C.}~\bibnamefont
  {{Teitelboim}}},\ }\bibfield  {title} {\enquote {\bibinfo {title} {{Role of
  surface integrals in the Hamiltonian formulation of general relativity}},}\
  }\href {\doibase 10.1016/0003-4916(74)90404-7} {\bibfield  {journal}
  {\bibinfo  {journal} {Annals of Physics}\ }\textbf {\bibinfo {volume} {88}},\
  \bibinfo {pages} {286--318} (\bibinfo {year} {1974})}\BibitemShut {NoStop}%
\bibitem [{\citenamefont {{Shen}}(2023)}]{Shen2022GCMhypersurface}%
  \BibitemOpen
  \bibfield  {author} {\bibinfo {author} {\bibfnamefont {D.}~\bibnamefont
  {{Shen}}},\ }\bibfield  {title} {\enquote {\bibinfo {title} {{Construction of
  GCM hypersurfaces in perturbations of {K}err}},}\ }\href {\doibase
  10.1007/s40818-023-00152-x} {\bibfield  {journal} {\bibinfo  {journal} {Ann.
  PDE}\ }\textbf {\bibinfo {volume} {9}},\ \bibinfo {pages} {11} (\bibinfo
  {year} {2023})},\ \Eprint {http://arxiv.org/abs/2205.12336}{arXiv:2205.12336
  [math.AP]}\BibitemShut {NoStop}%
\bibitem [{\citenamefont {{Shlapentokh-Rothman}}(2015)}]{2015AnHP...16..289S}%
  \BibitemOpen
  \bibfield  {author} {\bibinfo {author} {\bibfnamefont {Y.}~\bibnamefont
  {{Shlapentokh-Rothman}}},\ }\bibfield  {title} {\enquote {\bibinfo {title}
  {{Quantitative Mode Stability for the Wave Equation on the Kerr
  Spacetime}},}\ }\href {\doibase 10.1007/s00023-014-0315-7} {\bibfield
  {journal} {\bibinfo  {journal} {Ann. Henri Poincar{\'e}}\ }\textbf {\bibinfo
  {volume} {16}},\ \bibinfo {pages} {289--345} (\bibinfo {year} {2015})},\
  \Eprint {http://arxiv.org/abs/1302.6902}{arXiv:1302.6902 [gr-qc]}\BibitemShut
  {NoStop}%
\bibitem [{\citenamefont {{Shlapentokh-Rothman}}\ and\ \citenamefont {{Teixeira
  da Costa}}(2020)}]{2020arXiv200707211S}%
  \BibitemOpen
  \bibfield  {author} {\bibinfo {author} {\bibfnamefont {Y.}~\bibnamefont
  {{Shlapentokh-Rothman}}}\ and\ \bibinfo {author} {\bibfnamefont
  {R.}~\bibnamefont {{Teixeira da Costa}}},\ }\bibfield  {title} {\enquote
  {\bibinfo {title} {{Boundedness and decay for the Teukolsky equation on Kerr
  in the full subextremal range $|a|<M$: frequency space analysis}},}\
  }\href@noop {} {\  (\bibinfo {year} {2020})},\ \Eprint
  {http://arxiv.org/abs/arXiv:2007.07211}{arXiv:2007.07211 [gr-qc]}\BibitemShut
  {NoStop}%
\bibitem [{\citenamefont {{Shlapentokh-Rothman}}\ and\ \citenamefont {{Teixeira
  da Costa}}(2023)}]{2023arXiv230208916S}%
  \BibitemOpen
  \bibfield  {author} {\bibinfo {author} {\bibfnamefont {Y.}~\bibnamefont
  {{Shlapentokh-Rothman}}}\ and\ \bibinfo {author} {\bibfnamefont
  {R.}~\bibnamefont {{Teixeira da Costa}}},\ }\bibfield  {title} {\enquote
  {\bibinfo {title} {{Boundedness and decay for the Teukolsky equation on Kerr
  in the full subextremal range $|a|<M$: physical space analysis}},}\
  }\href@noop {} {\  (\bibinfo {year} {2023})},\ \Eprint
  {http://arxiv.org/abs/arXiv:2302.08916}{arXiv:2302.08916 [gr-qc]}\BibitemShut
  {NoStop}%
\bibitem [{\citenamefont {{Starobinski{\v i}}}\ and\ \citenamefont
  {{Churilov}}(1974)}]{1974JETP...38....1S}%
  \BibitemOpen
  \bibfield  {author} {\bibinfo {author} {\bibfnamefont {A.~A.}\ \bibnamefont
  {{Starobinski{\v i}}}}\ and\ \bibinfo {author} {\bibfnamefont {S.~M.}\
  \bibnamefont {{Churilov}}},\ }\bibfield  {title} {\enquote {\bibinfo {title}
  {{Amplification of electromagnetic and gravitational waves scattered by a
  rotating ``black hole''}},}\ }\href
  {http://www.jetp.ac.ru/cgi-bin/dn/e\_038\_01\_0001.pdf} {\bibfield  {journal}
  {\bibinfo  {journal} {Sov. Phys.--JETP}\ }\textbf {\bibinfo {volume} {38}},\
  \bibinfo {pages} {1} (\bibinfo {year} {1974})}\BibitemShut {NoStop}%
\bibitem [{\citenamefont {Stephani}\ \emph {et~al.}(2003)\citenamefont
  {Stephani}, \citenamefont {Kramer}, \citenamefont {MacCallum}, \citenamefont
  {Hoenselaers},\ and\ \citenamefont {Herlt}}]{2003esef.book.....S}%
  \BibitemOpen
  \bibfield  {author} {\bibinfo {author} {\bibfnamefont {H.}~\bibnamefont
  {Stephani}}, \bibinfo {author} {\bibfnamefont {D.}~\bibnamefont {Kramer}},
  \bibinfo {author} {\bibfnamefont {M.}~\bibnamefont {MacCallum}}, \bibinfo
  {author} {\bibfnamefont {C.}~\bibnamefont {Hoenselaers}}, \ and\ \bibinfo
  {author} {\bibfnamefont {E.}~\bibnamefont {Herlt}},\ }\href {\doibase
  10.1017/CBO9780511535185} {\emph {\bibinfo {title} {Exact Solutions of
  Einstein's Field Equations}}},\ \bibinfo {edition} {2nd}\ ed.,\ Cambridge
  Monographs on Mathematical Physics\ (\bibinfo  {publisher} {Cambridge
  University Press},\ \bibinfo {year} {2003})\BibitemShut {NoStop}%
\bibitem [{\citenamefont {Sterbenz}\ and\ \citenamefont
  {Tataru}(2015)}]{MR3373052}%
  \BibitemOpen
  \bibfield  {author} {\bibinfo {author} {\bibfnamefont {J.}~\bibnamefont
  {Sterbenz}}\ and\ \bibinfo {author} {\bibfnamefont {D.}~\bibnamefont
  {Tataru}},\ }\bibfield  {title} {\enquote {\bibinfo {title} {Local energy
  decay for {M}axwell fields {P}art {I}: {S}pherically symmetric black-hole
  backgrounds},}\ }\href {\doibase 10.1093/imrn/rnu034} {\bibfield  {journal}
  {\bibinfo  {journal} {Int. Math. Res. Not. IMRN}\ ,\ \bibinfo {pages}
  {3298--3342}} (\bibinfo {year} {2015})},\ \Eprint
  {http://arxiv.org/abs/1305.5261}{arXiv:1305.5261 [math.AP]}\BibitemShut
  {NoStop}%
\bibitem [{\citenamefont {{Stewart}}\ and\ \citenamefont
  {{Walker}}(1974)}]{1974RSPSA.341...49S}%
  \BibitemOpen
  \bibfield  {author} {\bibinfo {author} {\bibfnamefont {J.~M.}\ \bibnamefont
  {{Stewart}}}\ and\ \bibinfo {author} {\bibfnamefont {M.}~\bibnamefont
  {{Walker}}},\ }\bibfield  {title} {\enquote {\bibinfo {title} {{Perturbations
  of space-times in general relativity}},}\ }\href {\doibase
  10.1098/rspa.1974.0172} {\bibfield  {journal} {\bibinfo  {journal} {Proc. R.
  Soc. Lond. A}\ }\textbf {\bibinfo {volume} {341}},\ \bibinfo {pages} {49--74}
  (\bibinfo {year} {1974})}\BibitemShut {NoStop}%
\bibitem [{\citenamefont {Szeftel}\ and\ \citenamefont
  {Klainerman}(2020)}]{2017arXiv171107597K}%
  \BibitemOpen
  \bibfield  {author} {\bibinfo {author} {\bibfnamefont {J.}~\bibnamefont
  {Szeftel}}\ and\ \bibinfo {author} {\bibfnamefont {S.}~\bibnamefont
  {Klainerman}},\ }\href {\doibase 10.2307/J.CTV15R57CW} {\emph {\bibinfo
  {title} {Global nonlinear stability of {S}chwarzschild spacetime under 
  polarized peturbations}}},\ \bibinfo {series} {Annals of Mathematics
  Studies}, Vol.\ \bibinfo {volume} {210}\ (\bibinfo  {publisher} {Princeton
  University Press, Princeton, NJ},\ \bibinfo {year} {2020})\ p.\ \bibinfo
  {pages} {840},\ \Eprint {http://arxiv.org/abs/1711.07597}{arXiv:1711.07597
  [gr-qc]}\BibitemShut {NoStop}%
\bibitem [{\citenamefont {Tataru}\ and\ \citenamefont
  {Tohaneanu}(2011)}]{MR2764864}%
  \BibitemOpen
  \bibfield  {author} {\bibinfo {author} {\bibfnamefont {D.}~\bibnamefont
  {Tataru}}\ and\ \bibinfo {author} {\bibfnamefont {M.}~\bibnamefont
  {Tohaneanu}},\ }\bibfield  {title} {\enquote {\bibinfo {title} {A local
  energy estimate on {K}err black hole backgrounds},}\ }\href {\doibase
  10.1093/imrn/rnq069} {\bibfield  {journal} {\bibinfo  {journal} {Int. Math.
  Res. Not. IMRN}\ ,\ \bibinfo {pages} {248--292}} (\bibinfo {year} {2011})},\
  \Eprint {http://arxiv.org/abs/0810.5766}{arXiv:0810.5766
  [math.AP]}\BibitemShut {NoStop}%
\bibitem [{\citenamefont {{Teukolsky}}(1972)}]{1972PhRvL..29.1114T}%
  \BibitemOpen
  \bibfield  {author} {\bibinfo {author} {\bibfnamefont {S.~A.}\ \bibnamefont
  {{Teukolsky}}},\ }\bibfield  {title} {\enquote {\bibinfo {title} {{Rotating
  Black Holes: Separable Wave Equations for Gravitational and Electromagnetic
  Perturbations}},}\ }\href {\doibase 10.1103/PhysRevLett.29.1114} {\bibfield
  {journal} {\bibinfo  {journal} {Phys. Rev. Lett.}\ }\textbf {\bibinfo
  {volume} {29}},\ \bibinfo {pages} {1114--1118} (\bibinfo {year}
  {1972})}\BibitemShut {NoStop}%
\bibitem [{\citenamefont {{Teukolsky}}(1973)}]{teukolsky:1973}%
  \BibitemOpen
  \bibfield  {author} {\bibinfo {author} {\bibfnamefont {S.~A.}\ \bibnamefont
  {{Teukolsky}}},\ }\bibfield  {title} {\enquote {\bibinfo {title}
  {{Perturbations of a Rotating Black Hole. I. Fundamental Equations for
  Gravitational, Electromagnetic, and Neutrino-Field Perturbations}},}\ }\href
  {\doibase 10.1086/152444} {\bibfield  {journal} {\bibinfo  {journal}
  {Astrophys. J.}\ }\textbf {\bibinfo {volume} {185}},\ \bibinfo {pages}
  {635--648} (\bibinfo {year} {1973})}\BibitemShut {NoStop}%
\bibitem [{\citenamefont {{Teukolsky}}(2015)}]{2015CQGra..32l4006T}%
  \BibitemOpen
  \bibfield  {author} {\bibinfo {author} {\bibfnamefont {S.~A.}\ \bibnamefont
  {{Teukolsky}}},\ }\bibfield  {title} {\enquote {\bibinfo {title} {{The Kerr
  metric}},}\ }\href {\doibase 10.1088/0264-9381/32/12/124006} {\bibfield
  {journal} {\bibinfo  {journal} {Class. Quant. Grav.}\ }\textbf {\bibinfo
  {volume} {32}},\ \bibinfo {eid} {124006} (\bibinfo {year} {2015})},\ \Eprint
  {http://arxiv.org/abs/1410.2130}{arXiv:1410.2130 [gr-qc]}\BibitemShut
  {NoStop}%
\bibitem [{\citenamefont {{Teukolsky}}\ and\ \citenamefont
  {{Press}}(1974)}]{1974ApJ...193..443T}%
  \BibitemOpen
  \bibfield  {author} {\bibinfo {author} {\bibfnamefont {S.~A.}\ \bibnamefont
  {{Teukolsky}}}\ and\ \bibinfo {author} {\bibfnamefont {W.~H.}\ \bibnamefont
  {{Press}}},\ }\bibfield  {title} {\enquote {\bibinfo {title} {{Perturbations
  of a rotating black hole. III - Interaction of the hole with gravitational
  and electromagnetic radiation}},}\ }\href {\doibase 10.1086/153180}
  {\bibfield  {journal} {\bibinfo  {journal} {Astrophys. J.}\ }\textbf
  {\bibinfo {volume} {193}},\ \bibinfo {pages} {443--461} (\bibinfo {year}
  {1974})}\BibitemShut {NoStop}%
\bibitem [{\citenamefont {{Walker}}\ and\ \citenamefont
  {{Penrose}}(1970)}]{walker:penrose:1970CMaPh..18..265W}%
  \BibitemOpen
  \bibfield  {author} {\bibinfo {author} {\bibfnamefont {M.}~\bibnamefont
  {{Walker}}}\ and\ \bibinfo {author} {\bibfnamefont {R.}~\bibnamefont
  {{Penrose}}},\ }\bibfield  {title} {\enquote {\bibinfo {title} {{On quadratic
  first integrals of the geodesic equations for type $\{$2,2$\}$
  spacetimes}},}\ }\href {\doibase 10.1007/BF01649445} {\bibfield  {journal}
  {\bibinfo  {journal} {Comm. Math. Phys.}\ }\textbf {\bibinfo {volume} {18}},\
  \bibinfo {pages} {265--274} (\bibinfo {year} {1970})}\BibitemShut {NoStop}%
\bibitem [{\citenamefont {{Whiting}}(1989)}]{whiting:1989}%
  \BibitemOpen
  \bibfield  {author} {\bibinfo {author} {\bibfnamefont {B.~F.}\ \bibnamefont
  {{Whiting}}},\ }\bibfield  {title} {\enquote {\bibinfo {title} {{Mode
  stability of the Kerr black hole}},}\ }\href {\doibase 10.1063/1.528308}
  {\bibfield  {journal} {\bibinfo  {journal} {J. Math. Phys.}\ }\textbf
  {\bibinfo {volume} {30}},\ \bibinfo {pages} {1301--1305} (\bibinfo {year}
  {1989})}\BibitemShut {NoStop}%
\bibitem [{\citenamefont {{Znajek}}(1977)}]{1977MNRAS.179..457Z}%
  \BibitemOpen
  \bibfield  {author} {\bibinfo {author} {\bibfnamefont {R.~L.}\ \bibnamefont
  {{Znajek}}},\ }\bibfield  {title} {\enquote {\bibinfo {title} {{Black hole
  electrodynamics and the Carter tetrad}},}\ }\href {\doibase
  10.1093/mnras/179.3.457} {\bibfield  {journal} {\bibinfo  {journal} {\mnras}\
  }\textbf {\bibinfo {volume} {179}},\ \bibinfo {pages} {457--472} (\bibinfo
  {year} {1977})}\BibitemShut {NoStop}%
\end{thebibliography}

%

\bigskip

\end{document}